\pdfoutput = 1
\documentclass[a4paper,11pt, doublespacing, book]{ut-thesis}
\usepackage{appendix}
\usepackage[dvipdfmx]{graphicx}
\usepackage[dvipsnames]{xcolor}
\usepackage[T1]{fontenc}
\usepackage[utf8]{inputenc}
\usepackage{lmodern}
\usepackage{hyperref}
\usepackage[english]{babel}
\usepackage{pst-node}
\usepackage{dsfont}
\usepackage{tikz-cd} 
\usepackage{mathtools}
\usepackage{scalerel}[2014/03/10]
\usepackage[usestackEOL]{stackengine}
\usepackage{stmaryrd}
\usepackage{cancel}
\usepackage{color}
\usepackage{amsthm}
\usepackage{amsmath}
\usepackage{amsfonts}
\usepackage{enumerate}
\usepackage{amssymb}
\usepackage{epstopdf}
\usepackage{hyperref}
\usepackage{mathrsfs}
\usepackage{ulem}
\usepackage{comment}
\usepackage{blindtext} 

\newtheorem{theorem}{Theorem}[section]
\newtheorem{proposition}[theorem]{Proposition}
\newtheorem{notation}[theorem]{Notation}
\newtheorem{fact}[theorem]{Fact}
\newtheorem{remark}[theorem]{Remark}
\newtheorem{example}[theorem]{Example}

\newtheorem{lemma}[theorem]{Lemma}
\newtheorem{corollary}[theorem]{Corollary}

\newtheorem{openquestion}[theorem]{Open Question}

\newtheorem{property}[theorem]{Property}

\newtheorem{claim}[theorem]{Claim}
\newtheorem{definition}[theorem]{Definition}
\newtheorem{exercise}[theorem]{Exercise}

\newtheorem{note}[theorem]{Notes}
\newtheorem{conj}[theorem]{Conjecture}

\newcommand{\cut}{\mathbf{Cut}}

\newcommand{\dga}{\dot{\gamma}}

\newcommand{\R}{\mathbb{R}}

\newcommand{\Z}{\mathbb{Z}}
\newcommand{\Q}{\mathbb{Q}}
\newcommand{\N}{\mathbb{N}}

\newcommand{\haus}{\mathcal{H}}
\DeclareMathOperator{\diam}{\textbf{diam}}
\DeclareMathOperator{\rad}{\textbf{rad}}

\newcommand{\supp}{\mathbf{supp}\,}

\newcommand{\br}[1]{\left\{#1\right\}}
\newcommand{\inp}[1]{\langle #1 \rangle}
\newcommand{\brac}[1]{\left(#1\right)}
\newcommand{\rms}[1]{\mathcal{M}^{k}} 

\newcommand{\length}{\mathbf{Length}}
\newcommand{\grad}{\mathbf{grad}}
\newcommand{\id}{\mathbf{Id}}

\newcommand{\hess}{\mathbf{Hess}}
\newcommand{\ric}{\mathbf{Ric}}
\newcommand{\sect}{\mathbf{Sec}}

\newcommand{\tr}{\mathbf{Trace}}
\renewcommand{\det}[1]{\mathbf{det}{(#1)}}
\newcommand{\vol}{\mathbf{Vol}}
\newcommand{\dist}{\mathbf{dist}}

\newcommand{\pik}{\varpi^\kappa}

\newcommand{\eps}{\varepsilon}
\newcommand{\modsp}[1]{\mathbb{M}_{#1}^n}
\newcommand{\mdk}{\mathbf{md}_\kappa}
\newcommand{\snk}{\mathbf{sn}_\kappa}
\newcommand{\csk}{\mathbf{cs}_\kappa}
\newcommand{\cs}{\mathbf{cs}}
\newcommand{\ctk}{\mathbf{ct}_\kappa}

\newcommand{\tnk}{\mathbf{tan}_\kappa}
\newcommand{\an}{An}

\providecommand{\abs}[1]{\lvert #1\rvert}
\providecommand{\norm}[1]{\lvert\lvert #1\rvert\rvert}

\newcommand{\qst}[1]{\textcolor{red}{(\textbf{Question: }{#1})}}

\newcommand{\cmtv}[1]{\textcolor{magenta}{(\textbf{Vitali: }{#1})}}

\newcommand{\corrv}[1]{\textcolor{ForestGreen}{{#1}}}

\newcommand{\C}{\mathbb{C}}

\newcommand{\sh}{\mathfrak{Sh}}
\newcommand{\inj}{\mathbf{Injrad}}

\newcommand{\ghto}{\xrightarrow{\mathbf{G-H}}}

\newcommand{\pghto}{\xrightarrow{\text{pointed $\mathbf{G-H}$}}}

\newcommand{\modangle}{\Tilde{\measuredangle}}
\newcommand{\modtriangle}{\Tilde{\triangle}}
\newcommand{\mtr}[3]{[\Tilde{#1}\Tilde{#2}\Tilde{#3}]}
\newcommand{\modcvee}{\Tilde{\curlyvee}}
\newcommand{\mangle}{\measuredangle}
\newcommand{\peri}{\mathbf{Per}}
\newcommand{\acts}{\curvearrowright}
\newcommand{\rank}{\mathbf{rank}}
\newcommand{\h}{\mathbb{H}}
\newcommand{\alex}{\mathfrak{Alex}}

\newcommand{\sn}{\mathbf{sn}}
\newcommand{\md}{\mathbf{md}}
\newcommand{\modspace}[1]{\mathbb{M}^2\brac{#1}}
\def\dashint{\,\ThisStyle{\ensurestackMath{%
  \stackinset{c}{.2\LMpt}{c}{.5\LMpt}{\SavedStyle-}{\SavedStyle\phantom{\int}}}%
  \setbox0=\hbox{$\SavedStyle\int\,$}\kern-\wd0}\int}
\newcommand{\bigzero}{\mbox{\normalfont\Large\bfseries 0}}
\newcommand{\rvline}{\hspace*{-\arraycolsep}\vline\hspace*{-\arraycolsep}}
\newcommand{\isom}{\textbf{Isom}}

\newcommand{\CR}{\mathfrak{CompRad}}
\newcommand{\curv}{curv}

\newcommand{\aut}{\mathbf{Aut}}

\author{Xinze Li}
\gradyear{2024}
\title{Lecture Notes on Comparison Geometry}
\begin{document}
\frontmatter
\maketitle
\begin{abstract}
This note is based on Professor Vitali Kapovitch's comparison geometry course at the University of Toronto. It delves into various comparison theorems, including those by Rauch and Toponogov, focusing on their applications, such as Bishop-Gromov volume comparison, critical point theory of distance functions, diameter sphere theorem, and negative and nonnegative curvature. Additionally, it covers the soul theorem, splitting theorem, and covering theorem by Cheeger-Gromoll, as well as Perelman's proof of the soul conjecture. Finally, the note introduces Gromov-Hausdorff convergence, Alexandrov Spaces, and the Finite Homotopy type theorem by Grove-Peterson.
\end{abstract}

\acknowledgements
I want to express my deep gratitude to Professor Vitali Kapovitch for his insightful lectures and profound mathematical knowledge, which have significantly shaped the content of this note. His enduring patience, devoted guidance, and careful proofreading were crucial in completing this project. The finalization of this note stands as a testament to his invaluable contributions. Without his unwavering support, this note would not have been possible, and I am genuinely thankful for that.

I'd also like to thank Professor Yevgeny Liokumovitch for his encouragement throughout the writing process.

A significant portion of this note was crafted during my visit to Peking University BIMCR in the summer of 2023. I want to thank Professor Gang Tian for his hospitality and support. I'm also grateful for the wonderful friends I met during my travels. Engaging in discussions, sharing ideas, and exploring mathematics with them enriched this experience.

I would also like to thank Yueheng Bao, Wenkui Du, Shengxuan Zhou, and Xingyu Zhu for our enriching discussions while writing the notes. 

\quad 
\tableofcontents

\mainmatter


\chapter{Solutions of Jacobi Equations}

	\section{Jacobi Fields and Exponential Maps}
	\begin{definition}[Variation of Geodesics]
		Suppose $I, K \subseteq \R$ are intervals, $\gamma: I \to M$ is a geodesic. Then a variation $\Gamma: K \times I$ of $\gamma$ is called a \textbf{variation through geodesics} if each of the  curves $\Gamma_s(t) = \Gamma(s, t)$ is also a geodesic.
	\end{definition}
	\begin{theorem}[See Theorem 10.1 and Proposition 10.4 in \cite{Lee19}]\label{thm: variational field jacobi}
		Let $(M, g)$ be a Riemannian manifold, and let $\gamma$ be a geodesic in $M$. If $J$ is a variation field of a variation through geodesics, then $J$ satisfies the Jacobi equation
		\begin{equation}\label{eq: Jacobi}
			D_t^2J + R(J, \dga)\dga = 0.
		\end{equation}
		The converse of the theorem is true if $M$ is complete or $I$ is a compact interval.
	\end{theorem}
	\begin{proof}
		Denote $T(s, t) = \partial_t\Gamma(s, t)$ and $S(s, t) = \partial_s\Gamma(s, t)$. Because $\Gamma$ is a variation through geodesics, then by the geodesic equation, we have for all $(s, t) \in K \times I$, we have $D_tT \equiv 0$. Then $D_sD_tT \equiv 0$. By proposition 7.5 in \cite{Lee19} the commutativity of the covariant derivative over a smooth vector field $V$ along any smooth one-parameter family of curves $\Gamma: J \times I$ in $M$, i.e.
		\begin{equation*}
			D_sD_tV - D_tD_sV = R(\partial_s\Gamma, \partial_t\Gamma)V, 
		\end{equation*}
		we have 
		\begin{align*}
			0 =& D_sD_tT = D_tD_sT + R(S, T)T.
		\end{align*}
		Then by the symmetry lemma (Lemma 6.2 in \cite{Lee19}) of any admissible family of curves in a Riemannian manifold, i.e $D_t\partial_s\Gamma = D_s\partial_t\Gamma$, we have
		\begin{align*}
			&0 = D_tD_tS + R(S, T)T\\
			\implies &0 = D_t^2J + R(J, \dga)\dga\quad\text{evaluating at $s = 0$.}
		\end{align*}
		Conversely, let $J$ be a Jacobi field. After applying a translation in $t$, we can assume $I$ is the interval contains $0$, and write $p = \gamma(0)$ and $v = \dga(0)$. Note that this implies $\gamma(t) = \exp_p(tv)$ for all $t \in I$ by uniqueness of ODE. Next, we are going to construct the variation through geodesics. Choose a smooth curve $\sigma: (-\varepsilon, \varepsilon) \to M$ and a smooth vector field $V$ along $\sigma$ satisfying
		\begin{align*}
			\sigma(0) = p, &\quad V(0) = v,\\
			\Dot{sigma}(0) = J(0), &\quad D_sV(0) = D_tJ(0)
		\end{align*}
		where $D_s$ and $D_t$ are covariant differentiation along $\sigma$ and $\gamma$. We define a variation of $\gamma$ by setting
		\begin{equation}\label{eq: Jacobi variation}
			\Gamma(s, t) = \exp_{\sigma(s)}(tV(s)).
		\end{equation}
		\begin{itemize}
			\item If $M$ is geodesically complete, this is defined for all $(s, t) \in (-\eps, \eps) \times I$. 
			\item If $I$ is compact. Then the fact that the domain of the exponential map is an open subset of $TM$ that contains the compact set $\br{(p, tv): t \in I}$ guarantees that there is some $\delta > 0$ such that $\Gamma(s, t)$ is defined for all $(s, t) \in (-\delta, \delta) \times I$.
		\end{itemize}
		Note that 
		\begin{equation}\label{eq: exponential variation 1}
			\Gamma(0, t) = \exp_{\sigma(0)}(tV(0)) = \exp_p(tv) = \gamma(t),
		\end{equation}
		\begin{equation}\label{eq: exponential variation 2}
			\Gamma(s, 0) = \exp_{\sigma(s)}(0) = \sigma(s).
		\end{equation}
		In particular, \ref{eq: exponential variation 1} shows that $\Gamma$ is a variation of $\gamma$. By the properties of the exponential map, $\Gamma$ is a variation through geodesics, and therefore its variation field $W(t) = \partial_s\Gamma(0, t)$ is a Jacobi field along $\gamma$. 
		
		Now we want to show $W \equiv J$. Notice that we can write the Jacobi equation as a system of second-order linear ordinary differential equations using the orthonormal frame. So, given initial values for $J$ and $D_tJ$ there is a unique Jacobi field that solves the equation \ref{eq: Jacobi} by the existence and uniqueness theorem in ODE theory (See Proposition 10.2 in \cite{Lee19}). Therefore, to show $W \equiv J$, we only need to show
		\begin{equation*}
			J(0) = W(0)\quad\&\quad D_tJ(0) = D_tW(0)
		\end{equation*}
		By the equation \ref{eq: exponential variation 1}, we know that
		\begin{equation*}
			W(0) = \partial_s \Gamma_s(0)|_{s = 0} = \dot{\sigma}(0) = J(0).
		\end{equation*}
		Because each $\Gamma_s(t)$ is a geodesic with the initial velocity $V(s)$, 
		\begin{align*}
			\partial_t\Gamma(s, 0) = \partial_t\Gamma_s(t) = 0 = V(s)
		\end{align*}
		Then the symmetry lemma $D_t\partial_s\Gamma = D_s\partial_t\Gamma$ implies $D_tJ(0) = D_tW(0)$,
		\begin{equation*}
			D_tW(0) = D_t\partial_s\Gamma(0, 0) = D_s\partial_t\Gamma(0, 0) = D_sV(0) = D_sJ(0).
		\end{equation*} 
	\end{proof}
	\begin{notation}
		We denote $\mathfrak{X}(\gamma)$ the space of all smooth vector fields along $\gamma$. 
	\end{notation}
	\begin{definition}
		A smooth vector field along a geodesic that satisfies the Jacobi equation \ref{eq: Jacobi} is called a \textbf{Jacobi field}.
	\end{definition}
	If we think $\mathfrak{X}(\gamma)$ as a linear space, then as the corollary (see corollary 10.3 in \cite{Lee19}), $\mathfrak{J}(\gamma) \subseteq \mathfrak{X}(\gamma)$ is a $2n$-dimensional linear subspace of $\mathfrak{X}(\gamma)$ where $\mathfrak{J}(\gamma)$ denotes the set of Jacobi fields along $\gamma$. The Jacobi field is also invariant under local isometry by proposition 10.5 in \cite{Lee19}.
	
	Jacobi fields can be used to determine whether the exponential map is a local diffeomorphism. To discuss that, we need to introduce what are conjugate points. For a more detailed discussion on the motivation of conjugate points see \cite{Lee19}. 
	\begin{definition}[See \cite{Lee19}]
		Let $(M, g)$ be a Riemannian manifold, $\gamma: I \to M$ a geodesic, and $p = \gamma(a), q = \gamma(b)$ for some $a, b \in I$. We say that $p$ and $q$ are conjugate along $\gamma$ if there is a Jacobi field vanishing at $t = a$ and $t = b$ but not identically zero along $\gamma$.
	\end{definition}
	It is also important to consider the Jacobi fields vanish at a point for this purpose. Let $(M, g)$ be a Riemannian manifold, $I \subseteq \R$ an interval containing $0$, and $\gamma: I \to M$ a geodesic. Assume $M$ is complete or $I$ is compact, then by the theorem \ref{thm: variational field jacobi}, the Jacobi fields are given by the variational fields of the variation \ref{eq: Jacobi variation}. If moreover, we assume $ J(0)= 0$, then $J$ is the variation field of 
	 \begin{equation*}
	 	\Gamma(s, t) = \exp_p(t(v + sw))
	 \end{equation*}
	 where $p = \gamma(0)$, $v = \dga(0)$, and $w = D_tJ(0)$ (see lemma 10.9 in \cite{Lee19}). This result allows us to write the explicit formula for all Jacobi fields vanishing at a point. 
	\begin{theorem}[See proposition 10.10 in \cite{Lee19}]\label{thm: Ja vanishes at a point}
		Let $(M, g)$ be a Riemannian manifold and $p \in M$. Suppose $\gamma: I \to M$ is a geodesic such that $0 \in I$ and $\gamma(0) = p$. For every $w \in T_pM$, the Jacobi field $J$ along $\gamma$ such that $J(0) = 0$ and $D_tJ(0) = w$ is given by
		\begin{equation*}
			J(t) = d(\exp_p)_{tv}(tw)
		\end{equation*}
		where $v = \dga(0)$, and we regard $tw$ as an element of $T_{tv}(T_pM)$ by means of the canonical identification $T_{tv}(T_pM) \cong T_pM$.
	\end{theorem}
	\begin{proof}
		Since every $t \in I$ is contained in some compact interval, then by translating $t$, we can show that $J$ is the variational field of $\Gamma(s, t) = \exp_p(t(v + sw))$ for all $t$. Then by chain rule
		\begin{align*}
			J(t) 
			&= \partial_s\Gamma(0, t) = \partial_s\Gamma(s, t)|_{s = 0}\\
			&= d(\exp_{p})_{t(v + sw)}(tw)|_{s = 0}\\
			&= d(\exp_p)_{tv}(tw)
		\end{align*}
	\end{proof}
	\begin{proposition}[See proposition 10.20 in \cite{Lee19}]
		Suppose $p \in M$, $v \in \mathcal{E}_p \subseteq T_pM$. Let $\gamma := \gamma_v: [0, 1] \to M$ be the geodesic segment $\gamma(t) = \exp_p(tv)$. Take $q = \gamma(1) = \exp_p(v)$. Then $\exp_p$ is a local diffeomorphism in a neighborhood of $v$ if and only if $q$ is not conjugate to $p$ along the geodesic $\gamma$.
	\end{proposition}
	\begin{proof}
		We know that 
		\begin{align*}
			\text{$\exp_p$ is a local diffeomorphism at $v$}
			\iff \text{$v$ is not a critical point of $\exp_p$}
		\end{align*}
		Therefore, suppose first that $v$ is a critical point of $\exp_p$. Then there is a nonzero vector $w \in T_v(T_pM)$ such that $d(\exp_p)_v(w) = 0$. Since $T_v(T_pM) \cong T_pM$,
		\begin{align*}
			0 = & d(\exp_p)_v(w) = \frac{\partial}{\partial s}\Big|_{s = 0}\exp_p(v + sw)\\
			=& \partial_s\Gamma(0, 1) = J(1) 
		\end{align*}
		Thus, the Jacobi field $J$ vanishes at $t = 1$. Therefore, $q$ is a conjugate point. Thus, we have proved that if $q$ is not conjugate to $p$ then $\exp_p$ is a local diffeomorphism. 
		
		Conversely, if $q$ is conjugate to $p$ along $\gamma$, then there is some nontrivial Jacobi field $J$ along $\gamma$ such that $J(0) = J(1) = 0$. On can show that $J$ is the variation field of the following variation of $\gamma$ through geodesics $\Gamma(s, t) = \exp_p(t(v + sw))$ with $w = D_tJ(0) \in T_pM$. Thus, by the computation in the preceding paragraph, we know that $d(\exp_p)_v(w) = J(1) = 0$. Thus $v$ is the critical point for $\exp_p$.
	\end{proof}
	
\section{Scalar Riccati Equation for $n = 2$}
Let $\gamma$ be a geodesic. Recall that a vector field $J$ along $\gamma$ is a Jacobi field if the following Jacobi equation holds
\begin{equation*}
    D_t^2 J + R(J, \Dot{\gamma})\Dot{\gamma} = 0.
\end{equation*}
As an introduction, we want to show that we can derive the Riccati equation from the Jacobi equation in the case when $n = 2$. Consider $(M^2, g)$ and $\gamma: [0, l] \to M^2$ a unit speed geodesic. Let $J$ be a Jacobi field along $\gamma$ such that $J\perp \dga$. Then $J(t) = y(t)X(t)$ for each $t$ where $X$ is a unit length vector field parallel along $\gamma$.
Thus,
\begin{equation*}
	D_tJ = y'(t)X(t) \quad\&\quad D_t^2 J = y''(t)X(t).
\end{equation*}
We can substitute these into the Jacobi equation, then
\begin{align*}
	& 0 = D_t^2J + R(J, \dot{\gamma})\dot{\gamma} = y''(t)X(t) + y(t)R(X, \dot{\gamma})\dot{\gamma}\\
	\implies & 0 = \inp{y''(t)X(t) + y(t)R(X, \dot{\gamma})\dot{\gamma}, X}\\
	\implies & 0 = y''(t) + \kappa(t)y(t)
\end{align*}
where $\kappa(t) = Rm(X, \dot{\gamma}, \dot{\gamma}, X)$ is the sectional curvature at $\gamma(t)$, which is the Gauss curvature in this case. 
Assume $y(t)\ne 0$ and put $a = \frac{y'}{y} \iff y' = ay$. Then, 
\begin{align*}
	& y'' = a'y + ay' \\
	\iff & -\kappa(t)y = a'y + a^2y\\
	\iff & (a' + a^2 + \kappa(t))y = 0\\
	\iff & a' + a^2 + \kappa(t) = 0 \quad\text{since $y \neq 0$}
\end{align*}
Thus, we can split the second-order equation equation $y'' + \kappa(t)y = 0$ into two first-order equations:
\begin{equation}\label{eq: scalar riccati}
    y'' + \kappa(t)y = 0 \iff 
    \begin{cases}
        a' + a^2 + \kappa(t) = 0 \quad\text{(Scalar Riccati Equation)}\\
        y' = ay
    \end{cases}
\end{equation}
\section{Matrix Riccati Equation}
Let $(M, g)$ be a Riemannian manifold, and Let $\gamma: [0, l] \to M$ be a unit speed geodesic starting at a point $p \in M$. Then any Jacobi field $J$ such that $J(0) \neq 0$ can be split into $J_1$ the tangential component and $J_2$ the perpendicular component, i.e. $J = J_1 + J_2$ such that for each $t \in [0, l]$, $J_2(t) \perp \dga(t)$ and $J_1(t) \parallel \dga(t)$. 
\begin{lemma}
	$J_1(t) = \lambda_1\dga(t) + \lambda_2 t\dga(t)$ for some constants $\lambda_1, \lambda_2 \in \R$. 
\end{lemma}
\begin{proof}
	To compute $J_1$, we notice that 
\begin{align*}
	& g(J, \dga)'' = g(D_t^2J, \dga) = g(-R(J, \dga)\dga, \dga) = 0\\
    \implies & g(J, \dga) = \lambda_2 t + \lambda_1\\
    \implies & J_1 = \lambda_1\dga + \lambda_2 t \dga
\end{align*}
for some constant $\lambda_1 , \lambda_2 \in \R$.
\end{proof}
\begin{remark}
	The key point is that when $n = 2$, then any  $J\perp\dga$ is as a scalar multiple of a unit perpendicular vector field $X$ parallel along $\gamma$. Indeed. Firstly we take $X(0) \in T_{\gamma(0)}M^2$ a vector perpendicular to $\gamma$. Then we extend $X$ along $\gamma$ via parallel transport. Since we assume $J$ a vector field perpendicular along $J$, then we can write $J$ as $J = y(t)X(t)$ for some scalar function $y(t)$.
\end{remark}

\begin{lemma}
	$J_1$ is still a Jacobi field, i.e.
	\begin{equation*}
		D_t^2J_1 + R(J_1, \dga)\dga = 0
	\end{equation*}
\end{lemma}
\begin{proof}
	We only need to check $\dga$ and $t\dga$ are both Jacobi fields. Since $\gamma$ is a geodesic, then $D_t\dga = 0$, so that $D_t^2\gamma = 0$ and $R(\dga, \dga)\dga= 0$, thus $\dga(t)$ is a Jacobi field, i.e. $D_t^2J + R(J, \dga)\dga = 0 + 0 = 0$. Next, $D_t(t\dga(t)) = \dga(t)$ then clearly $D_t^2(t\dga(t)) = 0$ and $R(t\dga(t), \dga(t))\dga(t) = tR(\dga(t), \dga(t))\dga(t) = 0$. And we have the same conclusion.
\end{proof}
Therefore, $J_2 = J - J_1$ is also a Jacobi field. That means, essentially, we only need to solve for perpendicular Jacobi fields $J_2 \perp \dga$. In the rest of this section, we always assume $J(t) \perp \dga(t)$ for all $t$. We will call the vector fields orthogonal to $\dga$ \textbf{normal}. The following theorem shows how a general Jacobi equation can be reduced to two first-order equations which can be solved separately.
\begin{theorem}\label{thm: split Jacobi equation}
	Let $(M, g)$ be a Riemannian manifold, $p \in M$ and $\gamma: [0, l] \to M$ a unit speed geodesic starting at $p$. Let $J(t)$ be a normal Jacobi field along $\gamma$.
	Then for some $0 < T \le l$ there is a symmetric linear operator $S(t): (\dga(t))^\perp\to  (\dga(t))^\perp$ along $\gamma$  on $[0,T]$ such that
	\begin{equation}\label{eq: Matrix Riccati Shape Operator}
    \begin{cases}
        D_tJ = SJ\\
        D_tS + S^2 + R_\nu = 0
    \end{cases}
\end{equation}
Conversely, if a symmetric $S$ is as above then a normal field $J$ satisfying \eqref{eq: Matrix Riccati Shape Operator} is a normal Jacobi field.
\end{theorem}
The operator $S$ plays the role of $a$ in the scalar Riccati equation \ref{eq: scalar riccati}. Hence, we can solve the Jacobi field by reducing the Jacobi equation via the following steps:

\textbf{Step 1:} By the assumption, $J(t) \perp \dga(t)$, then $D_tJ(t) \perp \dga(t)$ as well. This is because
\begin{align*}
	0 = g(J(t), \dga(t))' = g(D_t J(t), \dga(t)) + g(J(t), \underbrace{D_t\dga(t)}_{= 0})
\end{align*}
Suppose the initial conditions $J(0)$ and $D_tJ(0)$ are given and are both perpendicular to $\gamma'(0)$ and $J(0)\ne 0$. Then there always exist a symmetric linear map $A: \dga(0)^\perp \to \dga(0)^\perp$ such that
\begin{equation*}
	D_tJ(0) = A(J(0))
\end{equation*}

\textbf{Step 2:} Solve the following Riccati matrix ODE
\begin{equation}\label{eq-riccati-matrix}
	\begin{cases}
		D_tS + S^2 + R_{\dga(t)} = 0\\
		S(0) = A
	\end{cases}
\end{equation}
and we can obtain $S(t)$ along $\gamma$.

Note that the equation on $S$ is nonlinear and hence the maximal existing time $T$ for the solution can be smaller than $l$.

\textbf{Step 3:} Then we can solve for  $J$ by solving 
\begin{equation*}
		D_tJ(t) = S(t)(J(t))
\end{equation*}
given initial conditions $J(0)$. This equation is linear and hence the solution is always guaranteed to exist on any interval $[0, T]$ on which $S$ is defined. 

The operator $S$ and 
equation \ref{eq: Matrix Riccati Shape Operator} naturally arise via the following geometric construction.
  Let $(M, g)$ be a Riemannian manifold. Suppose $f: M \to \R$ smooth and $\abs{\grad(f)} = 1$. Remember that we denote $\nabla$ the Levi-Civita connection on $M$ and for a smooth function on a Riemannian manifold, we have
\begin{equation*}
	\nabla f(X) = \nabla_X f = df(X) = \langle \grad(f), X\rangle \quad\text{for $X \in \mathfrak{X}(M)$}
\end{equation*} 
where $\nabla f$ is the total differential. 
\begin{notation}
	We will abuse our notation and denote the $\grad(f)$ by $\nabla f$ if there is no confusion.
\end{notation}
Denote $M_t = \br{x \in M: f(x) = t}$ the $t$-level set of $f$. We know that for each $t$, $M_t$ is a regular submanifold of codimension $1$ by the regular value theorem ($\abs{\nabla f} \neq 0$). Since $f$ is smooth, by definition the gradient curve $\gamma(t)$ of $f$ satisfies $\Dot{\gamma}(t) = \nabla f_{\gamma(t)}$. And we can see that along $\gamma$, $f$ increases with unit speed:
\begin{align*}
    f(\gamma(t))' &= \inp{\nabla f_{\gamma(t)}, \underbrace{\dot{\gamma}(t)}_{= \nabla f_{\gamma(t)}}} \quad\text{By the definition of gradient.}\\
    &=  \abs{\nabla f}^2(\gamma(t)) = 1
\end{align*}
Thus $f(\gamma(t)) = t + C$ for some constant $C$. 
\begin{lemma}\label{cla: gradient curve exp}
    The gradient curve $\gamma(t)$ is a unit speed geodesic.
\end{lemma}
\begin{proof}
We will show this lemma by claiming that for a smooth function  \begin{equation}\label{eq: grad < 1 iff 1-lip}
    \abs{\nabla f} \leq 1 \iff \text{$f$ is $1$-Lipschitz}
\end{equation}
Suppose \ref{eq: grad < 1 iff 1-lip} is true. Consider $t_1 < t_2$. If $\gamma(t_1) = p \in M_{t_1}$ then by our construction,$f(\gamma(t_1)) = f(p) = t_1$ and $f(\gamma(t_2)) = t_1 + (t_2 - t_1) = t_2$. Therefore, $f(q) = t_2$ and $f(p) = t_1$.
    
     Since $\abs{\nabla f} \leq 1$, then by \ref{eq: grad < 1 iff 1-lip}, we know that $f$ is $1$-Lipschitz, i.e.
    \begin{equation*}
        t_2 - t_1 = f(p) - f(q) \leq d(p, q) \implies d(p, q) \geq t_2 - t_1.
    \end{equation*}
    On the other hand, $\length(\gamma|_{[t_1, t_2]}) \geq d(p, q)$ by the definition of $d(p,q)$. Therefore, we
    have
    \begin{equation*}
    	d(p,q)\le \length(\gamma|_{[t_1, t_2]}) =t_2-t_1=f(q)-f(p)\le d(p,q).
    \end{equation*}
So all the inequalities are equalities. Hence  $\gamma|_{[t_1, t_2]}$ is a unit curve whose length is equal to the distance between its endpoints, and hence it's a geodesic.

This means that the gradient curve $\gamma$ is unit speed geodesic.  Now we only need to show \ref{eq: grad < 1 iff 1-lip}. Suppose that $\abs{\nabla f} \leq 1$, take a unit speed geodesic $c(t)$ from $x$ to $y$. Denote $d = d(x, y)$. Then $c(0) = x$ and $c(d) = y$. Then 
        \begin{align*}
            f(y) - f(x) = & f(c(0)) - f(c(d))\\
            = & \int_0^d (f\circ c)'(t)dt\\
            = & \int_0^d \inp{\nabla f, \Dot{c}(t)}dt\\
            \leq & \int_0^d \underbrace{\abs{\nabla f}}_{ \leq 1}\cdot\underbrace{\abs{\Dot{c}(t)}}_{= 1} dt \quad\text{By Cauchy Schwartz}\\
            \leq & \int_0^d 1dt = d
        \end{align*}
On the other hand, if $\abs{\nabla f} > 1$. Consider $v = \frac{\nabla f(p)}{\abs{\nabla f(p)}}$, then
\begin{align*}
	\abs{df_p(v)} 
	&= \inp{\nabla f(p), \frac{\nabla f(p)}{\abs{\nabla f(p)}}}\\
	&= \frac{\abs{\nabla f(p)}^2}{\abs{\nabla f(p)}} = \abs{\nabla f(p)} > 1. 
\end{align*}
By the definition of directional derivative, taking $\gamma(t) = \exp_p(tv)$, then 
\begin{equation*}
	\abs{df_p(v)} = \lim_{t \to 0^+}\frac{\abs{f(\gamma(t)) - f(p)}}{t} > 1
\end{equation*}
which turns out that $f$ is not $1$-Lipschitz. 
\end{proof}

Now we conclude that if $\abs{\nabla f} = 1$, then for any $p \in M_t$, the gradient curve starting at $p$ are unit speed geodesics normal to the level sets, i.e. $\nabla f \perp M_t$. Moreover, the level sets are equidistant with respect to $f$. Namely, for any point $p \in M_{t_1}$, the distance $d_{M_{t_2}}(p) = \abs{t_2 - t_1}$. \begin{figure}[htbp]
    \centering
        \includegraphics[width=0.6\textwidth]{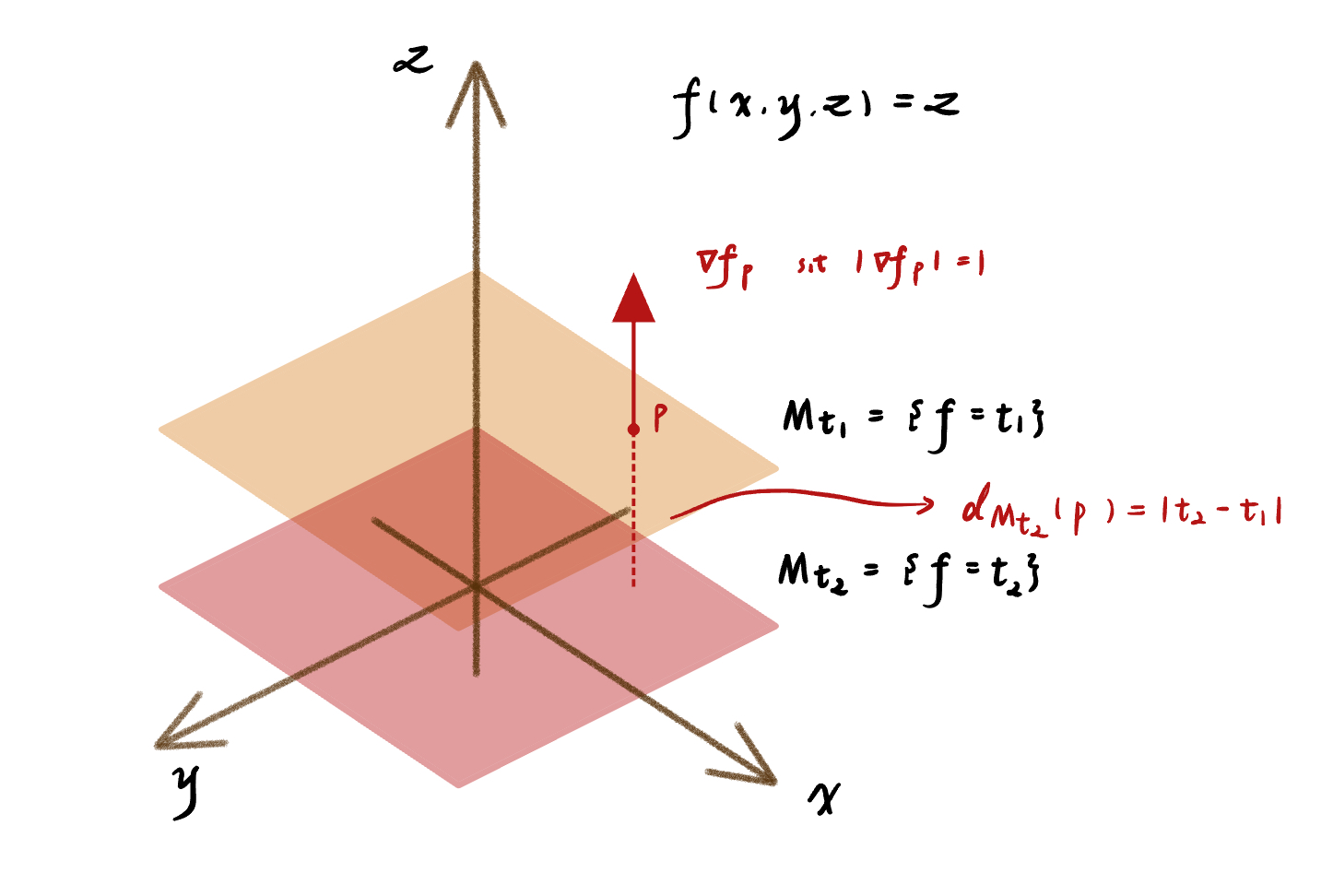}\caption{Gradient and level sets}
\end{figure}

Remember in the theorem \ref{thm: variational field jacobi}, we produce Jacobi fields by computing the variational field of the variation \ref{eq: Jacobi variation}. Now we are particularly interested in the situation that the geodesic is the gradient curve of $f$. Fix $t_1 \in \R$ and $p \in M_{t_1}$, Let $c: (-\eps, \eps) \to M_{t_1}$ be a smooth curve in $M_{t_1}$ such that $c(0) = p \in M_{t_1}$.
\begin{figure}[htbp]
    \centering
        \includegraphics[width=0.6\textwidth]{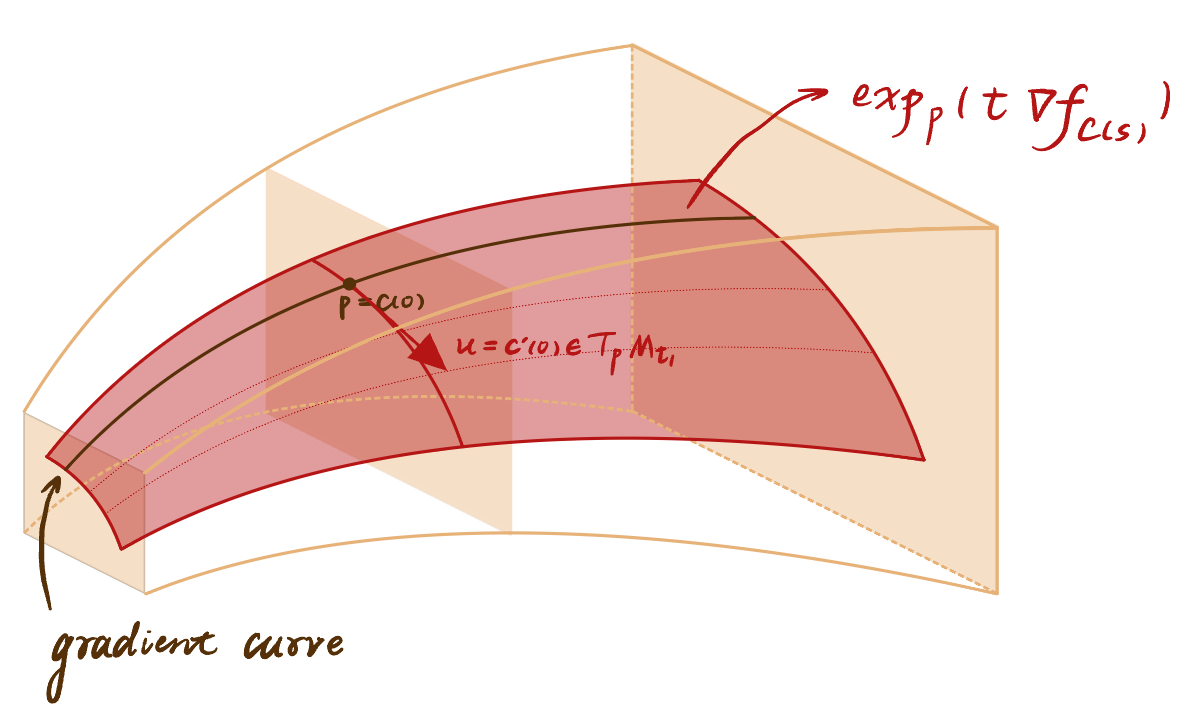}\caption{Variation}
\end{figure}
 Then the variation
\begin{equation}\label{eq: variation of gradient curve}
	\Gamma(s, t) = \exp_p(t\nabla f_{c(s)})
\end{equation}
also produce Jacobi fields (see the proof of the theorem \ref{thm: variational field jacobi}).
\begin{remark}
	Remember that the variation \ref{eq: variation of gradient curve} is defined Jacobi field when the level set $M_{t_1}$ is complete or the interval of the geodesic is compact. And since we already define our geodesic $\gamma$ over the interval $[0, l]$ we don't need to worry about whether $M_{t_1}$ is complete or not. 
\end{remark}
\begin{remark}
	It immediately follows from   Lemma~\ref{cla: gradient curve exp}  that $\Gamma(t, s) = \phi_t(c(s))$ where $\phi_t$ is the gradient flow of $f$.
\end{remark}
Because $\abs{\nabla f} = 1$, $\nabla f \perp M_t$ along $M_t$, $\nu = \nabla f$ is the unit normal to $M_{t_1}$. Therefore by the Weingarten equation for a hypersurface (Theorem~8.13 in \cite{Lee19})
\begin{equation*}
	D_tJ(t) = \nabla_{J(t)}\nabla f = \nabla_{J(t)}\nu = S(J(t))
\end{equation*}
where $S = S^\nu$ is the shape operator of $M_{t_1}$. This tells us that the Hessian operator of $f$ is the shape operator of the level sets of $f$. 

We can differentiate $D_tJ = S(J)$ along $\gamma(t)$. Then, on the one hand,
\begin{equation}\label{eq: shape for grad curve 1}
    D_t^2 J  = D_t(S(J)) 
    = (D_tS)(J) + S(D_tJ)
    = (D_tS)(J) + S^2(J)
\end{equation}
On the other hand, by the Jacobi equation and $\nu = \nabla f = \dga$,
\begin{equation}\label{eq: shape for grad curve 2}
	    D_t^2 J = -R(J, \dga)\dga = -R(J, \nu)\nu.
\end{equation}
Therefore, combining \ref{eq: shape for grad curve 1} and \ref{eq: shape for grad curve 2},
\begin{align*}
    D_tS(J) + S^2(J) + R_{\nu}(J) = 0 
\end{align*}
where $R_{\nu}(\cdot)$ is a symmetric and bilinear form $u \mapsto R(u, \nu)\nu = R_{\nu}(u)$. Since we can choose an arbitrary initial vector $J(0)$ perpendicular to the geodesic, we can conclude that 
\begin{equation*}
	 D_tS + S^2 + R_{\nu} = 0 
\end{equation*}
Therefore we always split the Jacobi equations of the Jacobi fields arising from variations of gradient curves of $f$ into two equations.
\begin{equation*}
D_t^2 J + R_{\nu}(J)= 0 \iff
    \begin{cases}
        D_tJ = SJ\\
        D_tS + S^2 + R_{\nu} = 0
    \end{cases}
\end{equation*}
So the shape operator $S$ plays the role of $a$ as in the scalar Riccati equation \ref{eq: scalar riccati}.

\begin{remark}
	 Let $f = d(\cdot, \Sigma)$, then outside of $\Sigma$, $f$ is smooth near $\Sigma$ and gradient curves are unit speed geodesics perpendicular to $\Sigma$. Especially, consider the hypersurface $\Sigma^{n - 1} \subseteq M^n$. If $\Sigma$ is a compact embedded submanifold, then by the tubular neighborhood theorem, $U_\eps(\Sigma)$ is diffeomorphic to $\Sigma \times \R$. And if the normal bundle $\nu\Sigma$ is trivial, $\Sigma$ separates $M$ locally, i.e. $U_\eps(\Sigma)\backslash\Sigma$ has two components. 
	 
	 In general $\Sigma$ does not separate $M$ locally when $\nu\Sigma$ is not trivial. For example, consider the case when $M$ is a Mobius band and $\Sigma$ is a circle. $\Sigma$ does not separate locally. 
\end{remark}

\begin{lemma} Let $M$ be a complete Riemannian manifold and $A$  be a closed subset in $M$. Let $f=d_A$. Suppose $f$ is smooth on some open set $U\subset M\setminus A$. Then $\abs{\nabla f} = 1$ on $U$.
\end{lemma}
\begin{proof}
We prove it in a special case $A=\{p\}$ is a point. The general case can be proved using the first variation formula.
Take $p \in M$ and $\gamma$ the unit speed geodesic started at $p$. In the case when the initial condition $J(0) = 0$ is given, we can also solve $J$ by choosing appropriate $f$. Just take $f = d(\cdot, p)$ so that $\gamma$ is the gradient curve of $f$ such that $f(\gamma(t)) = t$. Then $f$ is smooth inside the injectivity radius, and more generally, outside of $\cut(p)$. We want to show $\abs{\nabla f} = 1$. We first notice that $f$ is $1$-Lipschitz by the triangle inequality.
\begin{equation*}
    f(x) - f(y) = d(x, p) - d(y, p) \leq d(x, y) \implies \abs{\nabla f} \leq 1
\end{equation*}
On the other hand, because $f(\gamma(t)) - f(\gamma(0)) = t$, then 
\begin{align*}
    f(\gamma(t)) - f(\gamma(0)) = & \int_0^t (f\circ \gamma)'(\tau)d\tau\\
    = & \int_0^t \inp{\nabla f, \Dot{\gamma}}(\tau)d\tau\\
    \leq & \int_0^t \abs{\nabla f}\cdot\abs{\Dot{\gamma}}(\tau)d\tau\\
    = t
\end{align*}
implies all inequalities are equalities, thus $\inp{\nabla f, \Dot{\gamma}} = \abs{\nabla f}\cdot\abs{\Dot{\gamma}}$. Since $\abs{\dga} = 1$, and we have already known that $\abs{\nabla f} \leq 1$, we must have $\abs{\nabla f} = 1$. 

\end{proof}

Given a unit speed geodesic $\gamma$  with $\gamma(0)=p$
 using  function $f=d(,\cdot p)$ the above construction can produce any normal Jacobi field $J$ along $\gamma$ with $J(0)=0$ on $[0,l]$ where $l$ is smaller than the injectivity radius at $p$.
 Indeed, let $v=\dga(0)$ and let $w\in T_pM$ be a vector orthogonal to $v$. Let $c(s)$ be a curve in the unit sphere in $T_pM$ with $c'(0)=w)$.  Let $\hat \Gamma(t,s)=t\cdot c(s)$ and $\Gamma(t,s)=\exp_p\hat \Gamma(t,s)$. then by construction $\Gamma(t,0)\gamma(t)$ and $J=\frac{\partial \Gamma}{\partial s}(t,0)$ is a Jacobi field along $\gamma$ with $J(0)=0, J'(0)=w$.
 Also by construction, we have that $d(\Gamma(t,s), p)=t$ for any $t<l$ and for any fixed $s$ the curve $t\mapsto \Gamma(t,s)$ is a gradient curve of $f=d(\cdot, p)$. Note that here $S(t)$ blows up as $t\to 0$, more precisely  $S(t) \sim\frac {1} {t} \id$ as $t\to 0$.

\begin{remark}\label{exist-S-I}
  More generally the above works inside the conjugate locus rather than the cut locus of $p$. That is if $\gamma :[0, l]\to M$ is a (not necessarily shortest) unit speed geodesic with $\gamma(0)=p$ and there are no conjugate points along $\gamma$ then such $f$ can be defined locally near $\gamma$ and hence there exists $S(t)$ on $[0,l]$  satisfying the Riccati equation \eqref{eq-riccati-matrix}  with initial condition $S(t) \sim\frac {1} {t} \id$.
\end{remark}

\begin{remark}\label{exist-S-II}
For later applications, we will also want to be able to guarantee the existence interval of $S$ with initial conditions $S(0)=0$. This can be accomplished as follows.
Given a unit speed geodesic $\gamma :[0, l]\to M$ with $\gamma(0)=p$ and $\dga(0)=v\in T_pM$  let $N^{n-1} = \exp_p(v^\perp)\cap B_\eps(p)$  where $\eps\ll \inj(p)$.

We will call such $N$ a \textbf{geodesic submanifold defined by $v$}. Note that by construction the second fundamental form of $N$ at $p$ is 0.

We will say that there are no \emph{focal points} of $N$ along $\gamma$ if the normal exponential map to $N$ is a local diffeomorphism near $[0,lv]$. This condition guarantees the function $f=d_N$ defined locally near $\gamma$ is smooth.

Then $S(t)$ given by the second fundamental form of $\{f=t\}$ at $\gamma(t)$ solves the Riccati equation \eqref{eq-riccati-matrix} along $\gamma$ wit initial condition $S(0)=0$.

\end{remark}


\chapter{Model Spaces, Hessian, Cosine Law}
The key in the last lecture is to split the Jacobi equation
\begin{equation*}
    D_t^2J + R_{\nu}(J) = 0 \iff
    \begin{cases}
        D_tJ = SJ\\
        D_tS + S^2 + R_{\nu} = 0 \quad\text{Riccati Equation}
    \end{cases}
\end{equation*}
Next, we want to do comparison theory for these equations.

Let $E_1, E_2, \dots, E_{n - 1}$ be a  unit orthonormal basis at $T_p{M_{t_1}}$. We can extend $E_1, \dots, E_{n - 1}$ to parallel vector fields along $\gamma$. Since we know the parallel transport preserves the orthogonality, we obtain an orthonormal frame along $\gamma$. This allows us to write the shape operator $S(t)$ in the basis $E_1(t), \dots, E_{n - 1}(t)$ as a symmetric $(n - 1) \times (n - 1)$ matrix $A$. Then, the equation \ref{eq: Matrix Riccati Shape Operator} can be rewritten as
\begin{equation*}
    J'' + R(J, \Dot{\gamma})\Dot{\gamma} = 0 \iff
    \begin{cases}
        D_tJ = AJ\\
        D_tA + A^2 + R_{\nu}(\cdot) = 0
    \end{cases}
\end{equation*}
which is a first-order matrix ODE system. Moreover, the second equation can be solved first, independently of the first equation.

This simplifies our problem. However, it is very difficult to solve this ODE system explicitly for the general Riemannian manifold. In fact, we will see soon that we can solve this matrix ODE explicitly for the spaces of constant sectional curvature. And we can develop our comparison theory based on that.  When $\sect_M \geq \kappa$, we compare things to the equations in the model space $\mathbb{M}^2_\kappa$, the simply connected space of $\sect \equiv \kappa$.
This section is the preparation of the comparison theory of the constant sectional curvature. We will first develop the trigonometry of the model space. This helps us to write the explicit expression of the Jacobi fields in the model space. Next, using these trigonometry functions, we can define the modified distance function so that we can have a scalar function version of the Hessian matrix. 
\section{Angles and Triangles in the Models Spaces}
\begin{definition}
	Given a number $\kappa$, the define the \textbf{$n$-dimensional model $\kappa$-space} to be a complete simply connected $n$-dimensional Riemannian manifold of constant curvature $\kappa$. And the $n$-dimensional model $\kappa$-space will be denoted by $\modsp{\kappa}$. 
	\begin{itemize}
		\item If $\kappa > 0$, $\modsp{\kappa}$ is isometric to a $n$-sphere of radius $\frac{1}{\sqrt{\kappa}}$. In particular, $\modsp{1}$ is the unit $n$-sphere which is also denoted by $S^n$.
		\item If $\kappa = 0$, $\modsp{0}$ is just the $n$-dimensional Euclidean space which is also denoted by $\R^n$. 
		\item If $\kappa < 0$, $\modsp{\kappa}$ is the hyperbolic space rescaled by $\frac{1}{\sqrt{|\kappa|}}$. In particular, $\modsp{-1}$ is the standard $n$-dimensional hyperbolic space denoted by $\h^n$.
	\end{itemize}
\end{definition}
\begin{notation}[See \cite{AKP22}]
We denote $\varpi^\kappa = \diam(\modsp{\kappa})$ so that 
\begin{equation*}
\varpi^\kappa = 
	\begin{cases}
		\infty \quad\text{if $\kappa \leq 0$},\\
		\frac{\pi}{\sqrt{\kappa}} \quad\text{if $\kappa > 0$}
	\end{cases}
\end{equation*}
\end{notation}	
\begin{notation}[See \cite{AKP22}]
	The distance between points $x, y \in \modsp{\kappa}$ will be denoted by $\abs{x - y}$. 
\end{notation}
\begin{notation}[See \cite{AKP22}]
	Let $x , y \in \modsp{\kappa}$, a shortest geodesic segment connecting $x$ and $y$ will be denoted by $[xy]$.
\end{notation}
The segment $[xy]$ is uniquely defined for $\kappa \leq 0$.  For $\kappa > 0$  it is defined uniquely if $\abs{x - y} < \varpi^\kappa$. This is a strict inequality because when $\abs{x - y} = \varpi^\kappa$, we can think of $x$ and $y$ as the north and the south poles and there are infinitely many choices of the geodesic segments connecting these two points. 

\begin{definition}[Triangle (See \cite{AKP22})]\label{def: triangle}
We have two ways of defining triangles, one is by vertices and another one is by side lengths. 
\begin{itemize}
	\item \textbf{Triangle by Vertices}: A triangle in $\modsp{\kappa}$ with vertices $x, y, z \in \modsp{\kappa}$ is the ordered set of the three sides $([yz], [zx], [xy])$. In short, it will be denoted by $[xyz]$. 
	\item \textbf{Triangle by Side Lengths}: A triangle in $\modsp{\kappa}$ with side lengths $a, b, c$ will be denoted by $\modtriangle^\kappa\br{a, b, c}$.
\end{itemize}
\end{definition}
Given $x, y, z \in \modsp{\kappa}$, the triangle $[xyz]$ is uniquely defined as long as the segments $[yz], [zx], [xy]$ are uniquely defined which happens iff $|x-y|<\pik, |x-z|<\pik, |z-y|<\pik$. So $[xyz] = \modtriangle^\kappa\br{a, b, c}$ means that $x, y, z \in \modsp{\kappa}$ are such that 
\begin{equation*}
	\abs{x - y} = c, \quad\abs{y - z} = a,\quad \abs{z - x} = b
\end{equation*}

For  $\modtriangle^\kappa\br{a, b, c}$ to be defined, the sides $a, b, c$ must satisfy the triangle inequalities 
\begin{equation}\label{eq: tri eqs}
	a + b \geq c, \quad a + c \geq b, \quad b + c \geq a.
\end{equation}
Conversely, given $a, b, c$ satisfying the triangle inequality, when can we find a triangle in $\modsp{\kappa}$ with sides $a, b, c$? 
\begin{proposition}\label{prop: model triangle defined}
Let $a, b, c > 0$ satisfies the triangle inequalities \ref{eq: tri eqs}. Then for $\kappa \leq 0$, we can always find such a triangle and it is unique up to isometry. However if $\kappa > 0$, the triangle $\modtriangle^\kappa\br{a, b, c}$ exists if and only if 
\begin{equation*}
	a + b + c \leq 2 \varpi^\kappa
\end{equation*}
	Moreover, if 
\begin{equation*}
	a + b + c < 2 \varpi^\kappa
\end{equation*}
the triangle is unique up to an isometry of $\modsp{\kappa}$. 
\end{proposition}
\begin{proof}
	The case when $\kappa \leq 0$ is trivial. Suppose $\kappa > 0$. If $\modtriangle^\kappa\br{a, b, c}$ is  defined, we can consider $x, y, z \in \modsp{\kappa}$ such that $[xyz] = \modtriangle^\kappa\br{a, b, c}$. 
	By moving $z$ to the north pole by an isometry and extending $[zx], [zy]$ to the south pole, we observe that
	\begin{align*}
		& a + b \geq c\\
		\iff & (\varpi^\kappa - a) + (\varpi^\kappa - b) \geq c\\
		\iff & a + b + c \leq 2\varpi^\kappa
	\end{align*}
	
	\begin{figure}[htbp]
    \centering
        \includegraphics[width=0.5\textwidth]{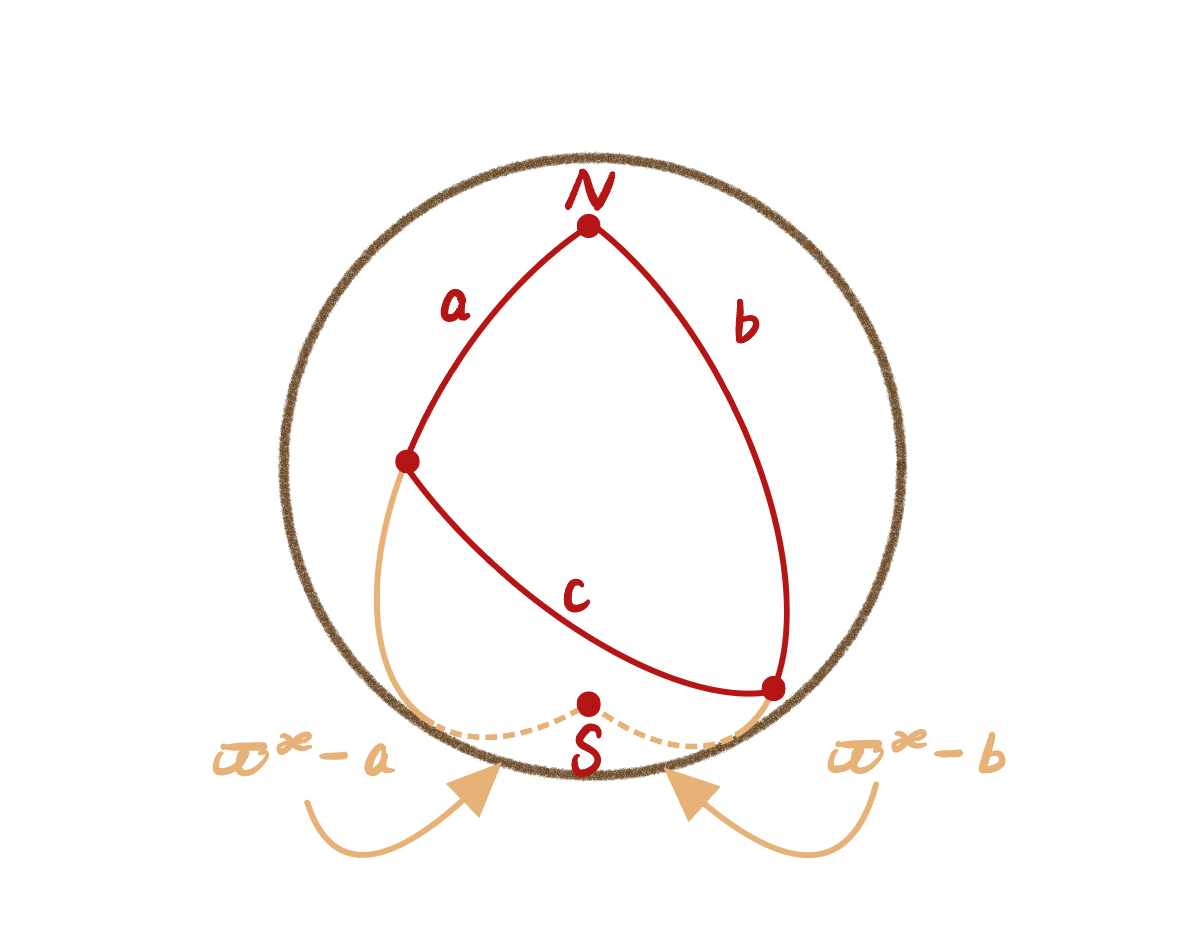}\caption{$a + b + c \leq 2\varpi^\kappa$.}
	\end{figure}
	
	Similarly, we can also show 
	\begin{align*}
		a + b + c \leq 2\varpi^\kappa \iff a + c \geq b\\
		a + b + c \leq 2\varpi^\kappa \iff b + c \geq a
	\end{align*}
	This means, if $a + b + c \leq 2\pik$, then $a, b, c \leq \pik$ since otherwise, for example, if $c > \pik$, we have $a + b \geq c > \pik$ thus, $a + b + c > 2\pik$. Contradiction
	Uniqueness can fail when equality holds. For example, there are infinitely many geodesics of length $\varpi^\kappa$ connecting the north and the south poles. 
	\begin{figure}[htbp]
    \centering
        \includegraphics[width=0.6\textwidth]{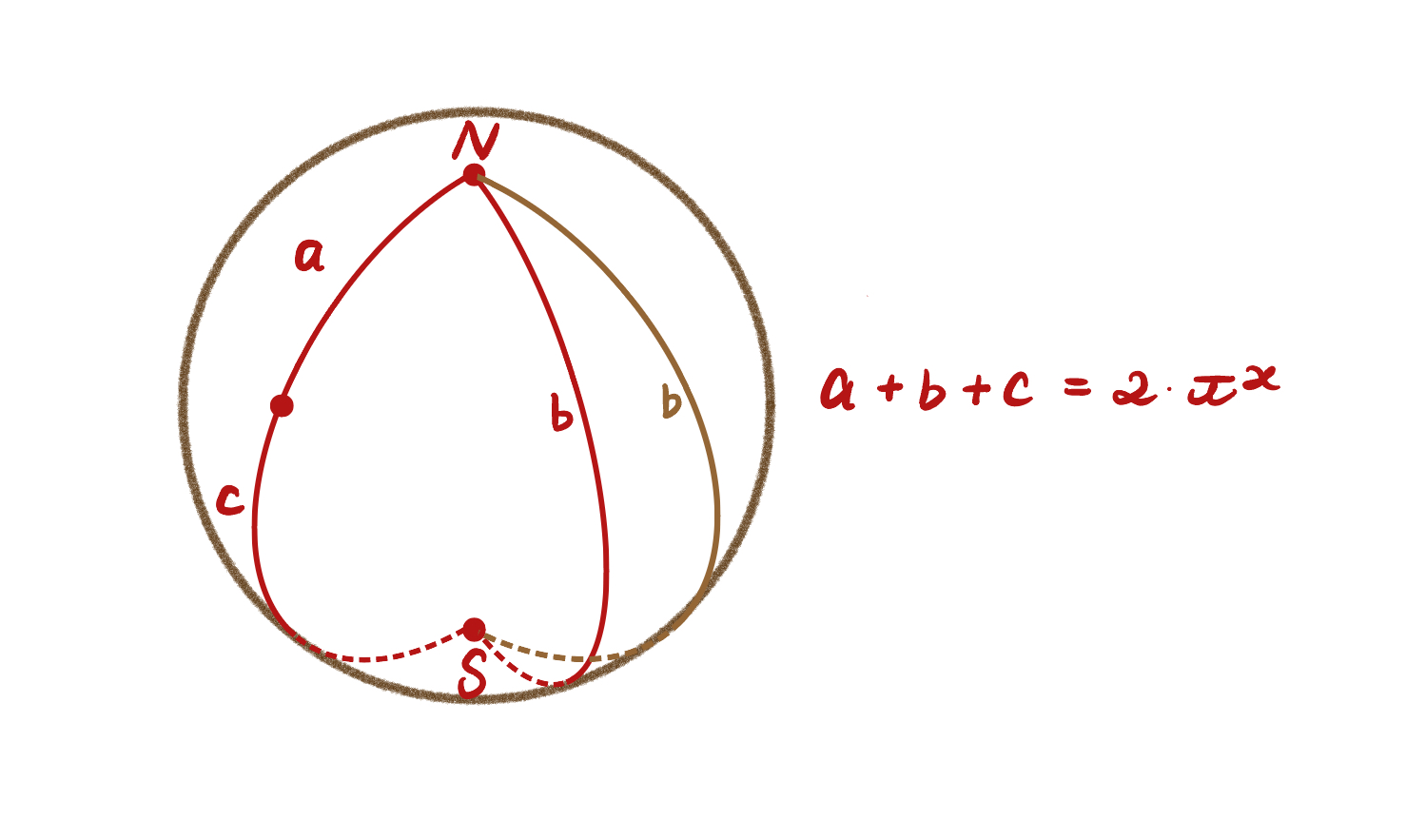}\caption{Uniqueness fails in the equality case. We have infinitely many options for the sides of length $b$.}
	\end{figure}
\end{proof}
\begin{definition}[Hinge (See \cite{AKP22})]
	Let $p, x, y \in \modsp{\kappa}$ be a triple of points such that $p$ is distinct from $x$ and $y$. A pair of geodesics $([px], [py])$ will be called a \textbf{hinge} and will be denoted by
	\begin{equation*}
		[p_y^x] = ([px], [py]).
	\end{equation*}
\end{definition}
\begin{notation}[See \cite{AKP22}]\label{not: angle}
	Let $[xyz] = \modtriangle^\kappa\br{a, b, c}$ be a triangle in $\modsp{\kappa}$. Then the angle $\varphi$ at $x$ will be denoted by $\mangle[x_z^y]$. Notice that $\varphi$ is also the angle of $\modtriangle^\kappa\br{a, b, c}$ opposite to $a$. In this case, we will write 
	\begin{equation*}
		a = \modcvee^\kappa\br{\varphi; b, c} \quad\text{or}\quad \varphi = \modangle^\kappa\br{a; b, c}
	\end{equation*}
	where the functions $\modcvee^\kappa$ and $\modangle^\kappa$ will be called respectively the \textbf{model side} and the \textbf{model angle}. 
\end{notation}
\section{Trigonometry of the Model Space}

Consider the model space $\modsp{\kappa}$. Let $\gamma: [0, l] \to \modsp{\kappa}$ be a unit speed geodesic, i.e. $\abs{\dga}(t) = 1$ for each $t \in [0, l]$. We want to solve for the explicit expression for the Jacobi fields perpendicular to the geodesic, i.e. $\inp{J(t), \dga(t)} \equiv 0$. Because the model space $\modsp{\kappa}$ is the space of constant curvature $\kappa$, we can choose $X_0 \in T_{\gamma(t_0)}M$ a unit vector perpendicular to $\gamma$, and extend $X_0$ to $X = X(t)$ a vector field along $\gamma$ obtained by parallel transport. In particular, $X(0) = X_0$. Obviously, $\abs{X} = 1$ and $\inp{X(t), \dga}\equiv 0$ for all $t$.

We can set $J(t) = \sum_{i = 1}^n\eta_i(t)X_i(t) =: \sum_{i = 1}^n J_i(t)$ so that we only need to solve for $\eta_i(t)$ to find the explicit expression of $J(t)$ in $\modsp{\kappa}$. Since space $\modsp{\kappa}$ is the space of constant sectional curvature $\kappa$. Then by Proposition 8.36 in~\cite{Lee19}, 
\begin{align*}
	Rm(X, Y, Z, T) = \kappa(\inp{X, Z}\inp{Y, T} - \inp{X, T}\inp{Y, Z})
	\quad\forall X, Y, Z, T \in \mathfrak{X}(\modsp{{\kappa}}).
\end{align*}
Thus, for any vector fields $T$ along $\gamma$,
\begin{align*}
	\inp{R(\dga, J_i)\dga, T} 
	&= \kappa(\inp{\dga, \dga}\inp{J_i, T} - \inp{\dga, T}\inp{\dga, J_i})\\
	&= \kappa\inp{J_i, T}
\end{align*}
Therefore, $R(\dga, J_i)\dga = \kappa J_i$ and we can rewrite the Jacobi equation as
\begin{align*}
	D_t^2 J_i + \kappa J_i = 0.
\end{align*}
Remember that $X(t)$ is parallel, therefore $D_tX(t) \equiv 0$, and we have
\begin{equation*}
	D_t^2 J_i(t) = D_t^2 (\eta(t)X(t)) = D_t(\eta'(t)X(t)) = \eta''(t)X(t)
\end{equation*}
Thus, we conclude that $J$ is a Jacobi field if and only if
\begin{equation}\label{eq: eta solution}
	\eta_i''(t) + \kappa \eta_i(t) = 0
\end{equation}
for each $i = 1, \dots, n$. By solving $\eta(t)$, we can find the explicit expression of the Jacobi fields in the model spaces. 
\begin{remark}
	Remember from the last lecture, that we can derive the scalar Riccati equation from \ref{eq: eta solution}. When $n = 2$, we can derive \ref{eq: eta solution} from any Jacobi field of any $2$-dimensional manifolds. For $n > 2$, we can derive the scalar Riccati equation only for model spaces. For general manifolds in higher dimensions, we can only have matrix Riccati equations.
\end{remark}

The general solution, as a second-order constant-coefficient linear ODE equation \ref{eq: eta solution}, is
\begin{align*}
	&\eta(t) = C_1\sinh{(t\sqrt{\abs{\kappa}})}
	 + C_2\cosh{(t\sqrt{\abs{\kappa}}))} \quad \text{if $\kappa < 0$}\\
	&\eta(t) = C_1\sin{(t\sqrt{\abs{\kappa}})}
	 + C_2\cos{(t\sqrt{\abs{\kappa}})} \quad \text{if $\kappa > 0$} \\
	&\eta(t) = C_1 t + C_2 \quad \text{if $\kappa = 0$}
\end{align*}
Under the initial condition $\eta(0) = 1$ and $\eta'(0) = 0$, The explicit solution of \ref{eq: eta solution} is denoted by $\sn_{(\cdot)}$ such that for any $\kappa \geq 0$,
\begin{equation}\label{eq: equation system for x}
    \sn_{\pm \kappa} = \frac{1}{\sqrt{\kappa}}\sn_{\pm 1}(t\sqrt{\kappa});
\end{equation}
where
\begin{equation*}
	\sn_{1}(t) = \sin(t), \quad \sn_{-1}(t) = \sinh(t), \quad \sn_0(x) = x;
\end{equation*}
On the other hand, under the initial condition $\eta(0) = 0$ and $\eta'(0) = 1$, the explicit solution of \ref{eq: eta solution} is denoted by $\cs_\kappa{(\cdot)}$ such that for any $\kappa \geq 0$, 
\begin{equation}\label{eq: equation system for y}
    \cs_{\pm \kappa} = \cs_{\pm 1}(t\sqrt{\kappa})
\end{equation}
where
\begin{equation*}
	\cs_{1}(t) = \cos(t), \quad \cs_{-1}(t) = \cosh(t), \quad \cs_0(x) = 1.
\end{equation*}

In terms of the $\csk$ and $\snk$, we can rewrite the general solution of the equation \ref{eq: eta solution} as 
    \begin{equation*}
    	\eta(t) = C_1\snk(t) + C_2\csk(t)
    \end{equation*}
    so that the Jacobi field $J$ in $\modsp{\kappa}$ can be written as
    \begin{equation*}
    	J(t) = (C_1\snk(t) + C_2\csk(t)) X(t)
    \end{equation*}
We will call $\csk$ and $\snk$ the \textbf{cosine} and \textbf{sine} functions of the model space $\modsp{\kappa}$. 

A straightforward computation gives us the laws of the trigonometry functions in the model spaces, 
\begin{equation*}
	(\snk)' = \csk ,\quad (\csk)' = -\kappa\snk.
\end{equation*}
Remember that we used the replacement $a = \frac{\eta'}{\eta} \iff \eta' = a\eta$ to derive the scalar Riccati equation $a' + a^2 + \kappa = 0$. When $\eta = \snk$, then $a = \frac{\eta'}{\eta}$ is the model cotangent function which is denoted by $ \ctk := \frac{y'}{y} = \frac{\csk}{\snk}$. Then the Riccati equation can be written as
\begin{equation*}
	(\ctk)' + (\ctk)^2 + \kappa = 0.
\end{equation*}
This can also be easily checked by a direct computation. Similarly, if we take $\eta=\csk(t)$ which also solves \ref{eq: eta solution} then  $a = \frac{\eta'}{\eta}=-k\tnk(t):=-k \frac{\snk}{\csk} $ also solves the Riccati equation $a'+a^2+k=0$.
\section{Hessians Operator}
In order to compare the Hessian operator of a general Riemannian manifold with those of the model spaces, we need the following explicit expression of the Hessian operator in the case of model space. 
\begin{proposition}\label{prop: projection ctk}
	Let $p \in \modsp{\kappa}$, $v\in T_p\modsp{\kappa}$. Take $\gamma(t)$ a unit speed geodesic starting at $p$ with $\dga(0) = v$. Denote $d(\cdot, p)$ by $d_p$, then
		\begin{equation*}
    		\hess_{d_p} = \ctk(t)\pi_{d_p}
    	\end{equation*}
   	is defined for any $q \in \modsp{\kappa}$ with $d_p(q) < \varpi^\kappa$ and $\pi_f: T_q\modsp{\kappa} \to T_q\modsp{\kappa}$ is the orthogonal projection onto the tangent space of the sphere centered at $p$ of radius $d_p(q)$.
\end{proposition}
\begin{proof}
Let $J$ be a normal Jacobi field along $\gamma$ such that $J(0) = 0$. Remember that we have already know that we can split the Jacobi equation into 
\begin{equation*}
	D_tJ(t) = S(t)(J(t))
\end{equation*}
and the Riccati equation. Because the shape operator along $\gamma$ is just the Hessian matrix.
This is 
\begin{equation}\label{eq: Hessian Jacobi}
	D_tJ(t) = \hess_f(J(t)).
\end{equation}
Let $(E_0(t), \dots, E_{n - 1}(t))$ be a parallel orthonormal frame along $\gamma$ and in particular we take $E_0(t) = \dga(t)$ for $t \in [0, \pik]$. It is easy to check that for each $i = 1, \dots, n - 1$, the vector fields 
\begin{equation*}
	J_i(t) = C_1\snk(t)E_i(t)
\end{equation*}
are also normal Jacobi fields along $\gamma$ for some constant $C_1$. Then
    \begin{equation}\label{eq: Hessian 1}
    	D_tJ_i(t) = C_1\csk(t)E_i(t)
    \end{equation}
    On the other hand, by the above equality \ref{eq: Hessian Jacobi}, by the linearity of the Hessian operator, we have 
    \begin{equation}\label{eq: Hessian 2}
    	D_tJ_i(t) = \hess_f(J_i(t)) = C_1\snk(t)\hess_f(E_i(t))
    \end{equation}
    As the right hand sides of the equation \ref{eq: Hessian 1} and \ref{eq: Hessian 2} are equal, we have
    \begin{align*}
    	\hess_f(E_i(t)) 
    	 &= \frac{C_1\csk(t)}{C_1\snk(t)}E_i(t)\\
    	 &= \ctk(t)\pi_f(E_i(t))
    \end{align*}
    On the other hand, since $E_0 = \nabla f_{\gamma(t)} = \dga(t)$. Then, as $f(\gamma(t)) = t$ is linear, we have
    \begin{equation}\label{eq: f''}
    	0 = f(\gamma(t))'' = (\nabla^2 f) (\dga(t), \dga(t)) = g(\hess_f(E_0), E_0)
    \end{equation}
    Thus, in terms of the basis $\br{E_0, \dots, E_{n - 1}}$, the explicit expression of $\hess_f$ is
    \begin{equation*}
    	\hess_f = \ctk(t)\pi_f
    \end{equation*}
    \end{proof}
    
    Since $t=d(\gamma(t),p)$  $t \in [0, \pik]$  this can be rewritten as  
    
     \begin{equation*}
    	\hess_f = \ctk(d_p)\pi_f
    \end{equation*}

    Therefore, under the basis $\br{E_0, \dots, E_{n - 1}}$ we can write $\hess_{d_p}$ as the following matrix. 
    \begin{equation*}
        \hess_{d_p} = 
        \begin{bmatrix}
            0 & 0 & 0 & \cdots & 0\\
            0 & \ctk(t) & 0 & \cdots & 0 \\
            0 & 0 & \ctk(t) & \dots & 0\\
            \vdots & \vdots & \vdots & \ddots & \vdots\\
            0 & 0 & 0 & \cdots & \ctk(t)
        \end{bmatrix}=
        \begin{bmatrix}
            0 & 0 & 0 & \cdots & 0\\
            0 & \ctk(d_p) & 0 & \cdots & 0 \\
            0 & 0 & \ctk(d_p) & \dots & 0\\
            \vdots & \vdots & \vdots & \ddots & \vdots\\
            0 & 0 & 0 & \cdots & \ctk(d_p)
        \end{bmatrix}
    \end{equation*}

\section{The Cosine Law for Model Spaces}
In this section, we want to derive the cosine law for the model space. The key point here is to obtain a scalar Hessian at every point of the model space (instead of just on the equidistant sphere of $d_p$) by changing $d_p$ to $\phi\circ d_p$ via some smooth function $\phi: \R \to \R$. Namely, we want
\begin{equation}\label{eq: scalar hessian}
	\hess_{\phi\circ d_p}|_x = \lambda(x)\id
\end{equation}
for each $x \in \modsp{\kappa}$ with $0 < d_p(x) < \pik$. And $\lambda$ is some scalar-valued function over $M$. Denote $f = \phi\circ d_p$, if we have \ref{eq: scalar hessian}, then for any $p \in \modsp{\kappa}$, and any unit speed geodesics $\gamma(t)$ in $\modsp{\kappa}$ (not necessarily passing through $p$), we can easily compute the second-order derivative along any unit speed geodesic. Indeed, remember the step \ref{eq: f''}
\begin{equation*}
	f(\gamma(t))'' = (\nabla^2 f)(\dga(t), \dga(t)) = g(\hess_f(\dga(t)), \dga(t)) = \lambda(\gamma(t)).
\end{equation*}
\subsection{Modification Function}
We define the \textbf{modification function} by
\begin{equation}\label{eq: mdk}
	\mdk(t) = 
	\begin{cases}
		\int_0^t \snk(s)ds \quad \text{for $0 \leq t \leq \varpi^\kappa$},\\
		\frac{2}{\kappa} \quad \text{for $t > \varpi^\kappa$}.
	\end{cases}
\end{equation}
Then the explicit expression for $\mdk$ is
\begin{equation*}
	\md_{\pm \kappa} = \frac{1}{\kappa}\md_{\pm 1}(t\sqrt{\kappa});
\end{equation*} 
where in particular when $\kappa = 1, 0, -1$, 
\begin{equation*}
	\md_1(t) = 
	\begin{cases}
		1 - \cos(t) \quad\text{for $t \leq \pi$,}\\
		2 \quad\text{for $t > \pi$.}
	\end{cases}
	\md_{-1}(t) = \cosh(t) - 1, \quad \md_{0}(t) = \frac{1}{2}t^2 
\end{equation*}

\begin{corollary}
	By the formula \ref{eq: mdk}, it is not hard to check that $\mdk$ solves the following initial value problem: 
	\begin{equation}\label{eq: equation system for z}
    \begin{cases}
        z'' + \kappa z = 1\\
        z(0) = 0\\
        z'(0) = 0.
    \end{cases}
\end{equation}
\end{corollary}

\subsection{Scalar Hessian Operator}
Let $p \in \modsp{\kappa}$, the distance function from $p$ is $d_p = d(\cdot, p)$. We will call $\mdk\circ d_p$ the \textbf{modified distance function} to $p$. It turns out that $\mdk$ plays the role of the function $\phi$ in \ref{eq: scalar hessian}.

\begin{theorem}\label{thm: hessian csk}
	The Hessian of $\mdk\circ d_p$ is scalar. Namely,
	\begin{equation}\label{eq: Hess cs}
		\hess_{\mdk\circ d_p}|_x = \csk(d_p(x))\id
	\end{equation}
	at each $x \in \modsp{\kappa}$.
\end{theorem}
\begin{proof}
	Recall the Hessian formula \cite{Lee19}, for any $X, Y \in \mathfrak{X}(M)$, 
		\begin{align*}
		g(\hess_{\phi\circ f}(X), Y) 
		&= Y(X(\phi\circ f)) - (\nabla_Y X)(\phi\circ f)\\
		&= Y(\phi'(f)X(f)) - \phi'(f)(\nabla_Y X)f\\
		&= \phi''(f)Y(f)X(f) + \phi'(f)Y(Xf) - \phi'(f)(\nabla_Y X)f\\
		&= \phi''(f)g(\nabla f, Y)g(\nabla f, X) + \phi'(f)g(\hess_f(X), Y)
	\end{align*}
	for any smooth functions $\phi: \R \to \R$ and $f: M \to \R$. Setting $\phi = \mdk$ and $f = d_p$. Along a unit speed geodesic $\gamma(t)$ starting at $p$, since $f(\gamma(t)) = t$, for $i = 1, \dots, n - 1$, we have
	\begin{equation*}
		g(\hess_{\mdk\circ d_p}(E_i), E_i) = \mdk''(t)g(\nabla f_{\gamma(t)}, E_i)^2 + \snk(t)g(\hess_{d_p}(E_i), E_i)
	\end{equation*}
	which followed by definition $\mdk'(t) = \snk(t)$ and the law $\snk'(t) = \csk(t)$. Then since $E_i(t) \perp \dga(t)$ along $\gamma$ and $\nabla f_{\gamma(t)} = \dga(t)$, we have
	\begin{equation*}
		g(\hess_{\mdk\circ d_p}(E_i), E_i) = \snk(t)g(\hess_{d_p}(E_i), E_i) = \snk(t)\ctk(t) = \csk(t)
	\end{equation*}
	because $g(\nabla f_{\dga(t)}, E_i) = 0$ for $i = 1, \dots, n - 1$. 	On the other hand, since $g(\hess_{d_p}(E_n), E_n) = 0$,
	\begin{equation*}
		g(\hess_{\mdk\circ d_p}(E_n), E_n) = \mdk''(t) = \csk(t)
	\end{equation*}
	The second term vanishes due to $E_n = \dga(t) = -N$ so that $g(\nabla f_{\gamma(t)}, E_n) = 1$. Therefore, we can conclude that using the modified distance function $\mdk$, we have
\begin{equation*}
	\hess_{\mdk\circ d_p} = \csk(d_p)\cdot \id 
\end{equation*}    
\end{proof}
\begin{corollary}\label{cor: key of scalar hessian}
Let $p \in \modsp{\kappa}$ and $\sigma$ be any unit-speed geodesic in $\modsp{\kappa}$ (not necessarily passing through $p$), then $z(s) = \mdk(d_p(\sigma(s)))$ solves \ref{eq: equation system for z}. Thus, for any $x \in \modsp{\kappa}$, we can alternatively express the scalar Hessian (equation \ref{eq: Hess cs}) as
\begin{equation}\label{model-hess-ode}
	\hess_{\mdk\circ d_p}|_x + \kappa\mdk(d_p(x)) \id = \id.
\end{equation}
for any $x \in \modsp{\kappa}$. 
\end{corollary}
\begin{remark}
	Unlike $d_p$, the modified distance function $\mdk\circ d_p$ is smooth at $p$.
\end{remark}

\subsection{The Cosine Law in the Model Space}
The above corollary \ref{cor: key of scalar hessian} is essential to derive the \textbf{cosine law} in the model space $\modsp{\kappa}$. To do that, we introduce the follow notations for metric spaces.

\begin{definition}[Model Angle]
    Let $X$ be a metric space, $p, x, y \in X$ and $\mtr{p}{x}{y} = \modtriangle^\kappa(pxy)$ is defined and $\abs{p - x}, \abs{p - y} > 0$, then the angle measure of $\mtr{p}{x}{y}$ at $\Tilde{p}$ will be called the \textbf{model angle of the triple $p, x, y$ at $p$} and will be denoted by
    \begin{equation*}
        \modangle^\kappa(p_x^y) := \modangle^\kappa\br{\abs{x - y}, \abs{p - x}, \abs{p - y}}
    \end{equation*}
\end{definition}
\begin{definition}[Hinge Angle \& Model Side of Hinge]\label{def-angle}
    Let $X$ be a metric space, $p, x, y \in X$. Let $[p_y^x]$ be a hinge, then we can define the \textbf{hinge angle} of the hinge $[p_y^x]$ as
    \begin{equation*}
        \mangle[p_y^x] = \lim_{\overline{x}, \overline{y} \to p}\modangle^\kappa(p_{\overline{y}}^{\overline{x}}),
    \end{equation*}
    for $\overline{x} \in ]px]$ and $\overline{y} = ]py]$ if this limit exists. And we can define the \textbf{model side} of the hinge $[p_y^x]$ as
    \begin{equation*}
    	\modcvee^\kappa[p_x^y] = \modcvee^\kappa\br{\mangle[p_y^x]; \abs{p - x}, \abs{p - y}},
    \end{equation*}
\end{definition}

Fix $a, b, c  > 0$ such that $\modtriangle^\kappa\br{a, b, c}$ is a triangle uniquely defined in the model space $\modsp{\kappa}$. Then we can write $\varphi = \modangle^\kappa\br{a; b, c}$ and $a = \modcvee^\kappa\br{\varphi; b, c}$. 
\begin{corollary}[See \cite{AKP22}]\label{eq: modcvee eq}
	Fix $a, \varphi$, the function $z(t) = \mdk(\modcvee^\kappa\br{\varphi; a, t})$ solves the IVP \ref{eq: equation system for z}. Namely, $z'' + \kappa z = 1$.
\end{corollary}
\begin{proof}
		We define $f = \mdk\circ d_p$ and it is easy to see from Figure~\ref{fig: starting at q such that d = t} that $f(\gamma(t)) = z(t)$.
	\begin{figure}[htbp]
    \centering
        \includegraphics[width=0.5\textwidth]{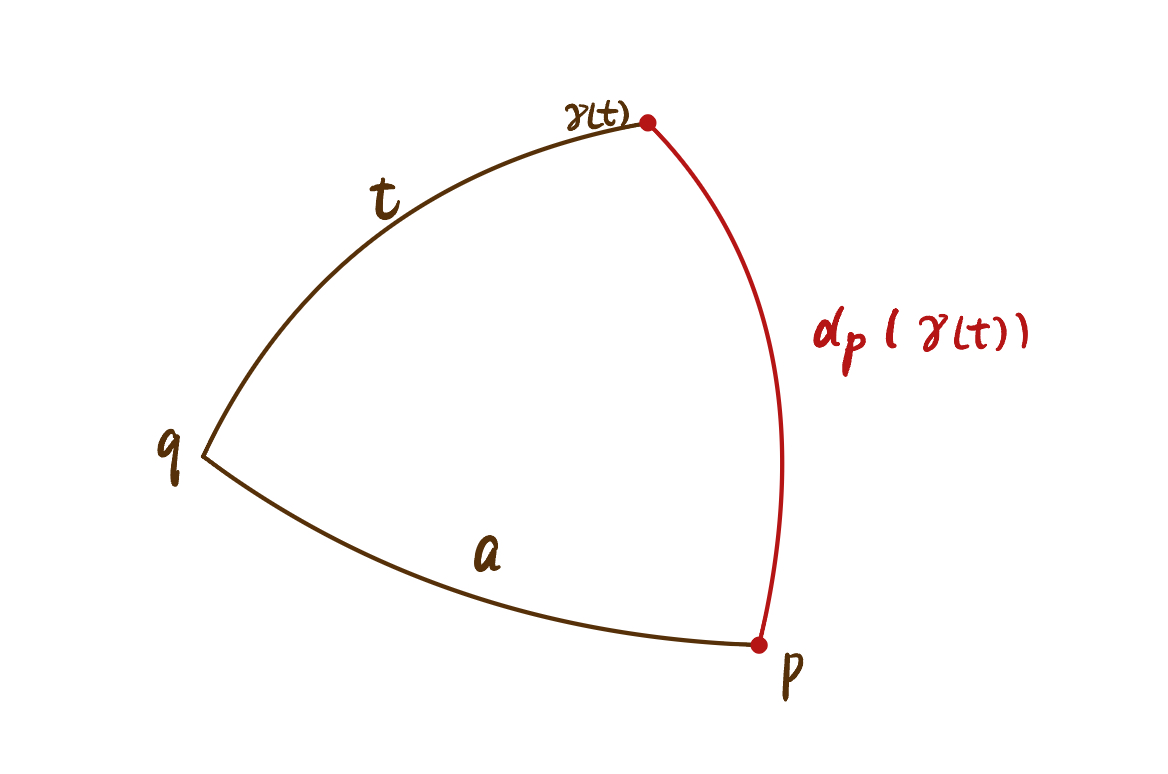}\caption{$\gamma$ is the geodesics starting at $q$ such that $d(q, \gamma(t)) = t$.}
        	\label{fig: starting at q such that d = t}
	\end{figure}
	
	Then by \ref{cor: key of scalar hessian}, we can know that along $\gamma(t)$, 
	\begin{equation*}
		f(\gamma(t))'' + \kappa f(\gamma(t)) = 1
	\end{equation*}
	Therefore, $z'' + \kappa z = 1$. 
\end{proof}
Moreover, at $t = 0$, we have
\begin{align*}
	d_p'(0) = g((\nabla d_p)_{q}, \dga(0)) = \cos{(\pi - \alpha)} = -\cos{(\alpha)}.
\end{align*}
Therefore, 
\begin{align*}
	z(0) &= \mdk(a);\\
	z'(0) &= -\snk(a)\cos{(\alpha)}.
\end{align*}
Knowing  $z(0)$ and $z'(0)$  we can easily find $z(t)$ by solving the IVP for the equation $z'' + \kappa z = 1$.  This allows us to find $\mdk ( d(\gamma(t),p))$ and hence also $d(\gamma(t),p)$ which yields the cosine law in $\modsp{\kappa}$.

\begin{theorem}
	In $\modsp{\kappa}$, the formula of $a = \modcvee^\kappa\br{\varphi; b, c}$ and $\varphi = \modtriangle^\kappa\br{a; b, c}$ can be written as the cosine law in $\modsp{\kappa}$:
	\begin{equation*}
	\cos{(\varphi)} = 
		\begin{cases}
			\frac{b^2 + c^2 - a^2}{2bc} \quad\text{if $\kappa = 0$}\\
			\frac{\csk(a) - \csk(b)\csk(c)}{\kappa\snk(b)\snk(c)} \quad\text{if $\kappa \neq 0$}
		\end{cases}
	\end{equation*}
\end{theorem}
\begin{proof}
Here we only compute the case when $\kappa = 0$. Given $p, q, \varphi$ and $\gamma(t)$ a unit speed geodesic starting at $q$ as in Figure~\ref{fig: cosine law}
\begin{figure}[htbp]
    \centering
        \includegraphics[width=0.5\textwidth]{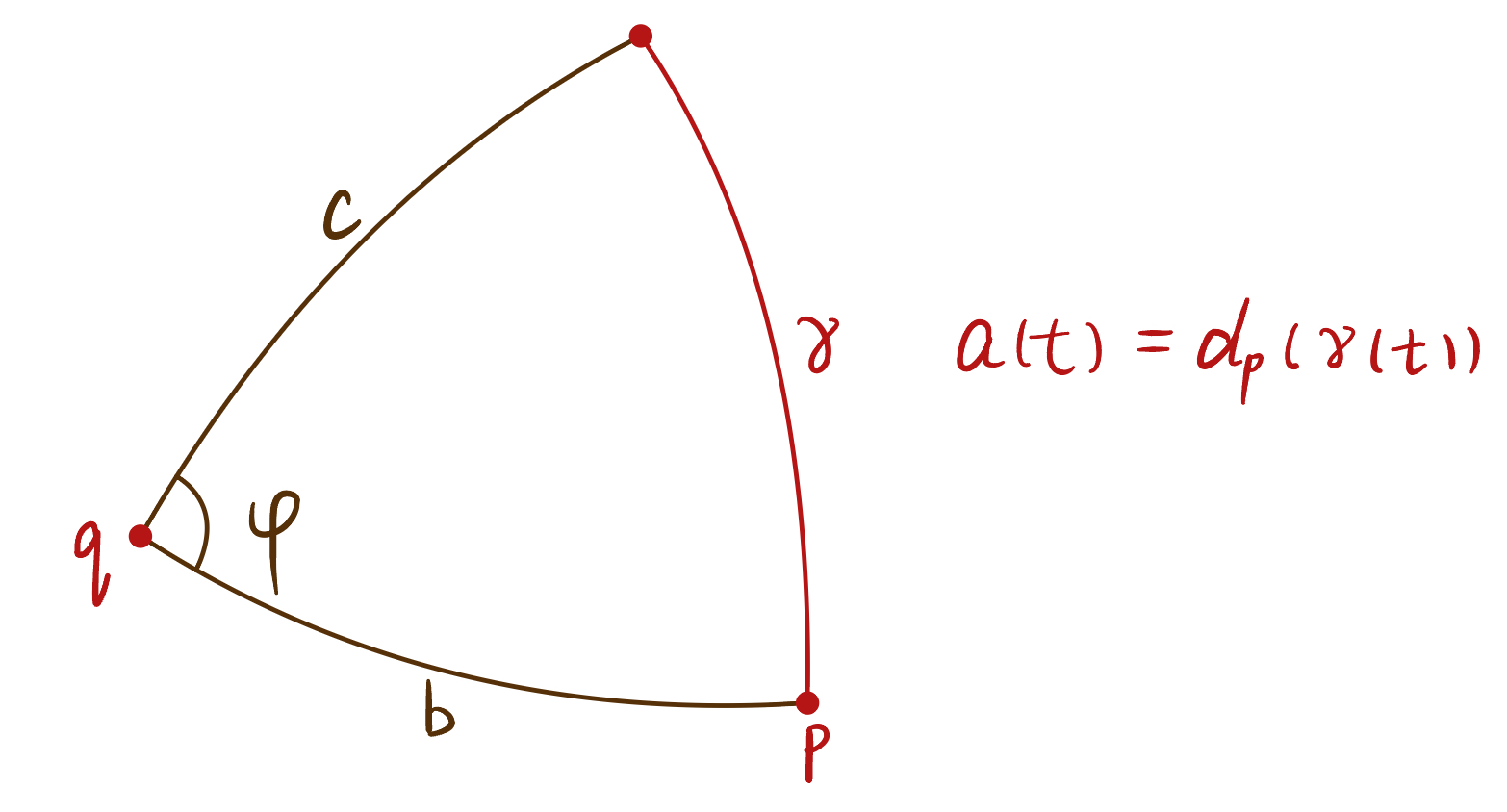}\caption{$\gamma$ is the geodesics starting at $p$ and we denote $a(t) = d_p(\gamma(t))$.}
        	\label{fig: cosine law}
	\end{figure}
We define $a(t) = d_p(\gamma(t))$. The question is to find $\varphi$. It is easy to see that $\hess_{a^2} = 2\cdot \id$. Then $a(0)^2 = b^2$ and $(a^2)'(0) = 2a(0)a'(0) = -2b\cos{\varphi}$. Therefore, denote $y(t) = a^2(t)$, by solving the IVP
\begin{equation*}
	\begin{cases}
		& y''(t) = 2\\
		& y(0) = b^2\\
		& y'(0) = -2b\cos{(\varphi)}
	\end{cases}
\end{equation*}
we can derive the cosine law for $\kappa = 0$, i,e.
\begin{equation*}
	a^2(t) = y(t) = b^2 + t^2 - 2tb\cos{(\varphi)}
\end{equation*}
\end{proof}

\section{Hessian Comparison Theorem}

The cosine law is essential to the theory of comparison geometry. And it is very important for proving Toponogov's comparison theorems. The key to proving the cosine law in the model spaces is the identity \ref{model-hess-ode} for  the Hessian operator $\hess_{\mdk\circ d_{\bar{p}}}$ for $\bar{p} \in \modsp{\kappa}$. In this section, we are going to compare this operator to the Hessian operator of the modified distance function for Riemannian manifolds of sectional curvature bounded below via the following theorem. 

\begin{theorem}[Hessian Comparison Theorem]\label{cor: Hess ineq}
    Let $(M^n, g)$ be a Riemannian manifold such that $\sect_M \geq \kappa$. Fix $p \in M$, let $f = \mdk \circ d_p$. Then for $r < \inj(p)$ or more generally outside of the cut-locus of $p$, we have
    \begin{equation}\label{eq: hessian comparison theorem}
        \hess_f + \kappa f\id \leq \id.
    \end{equation}
\end{theorem}

Let $(M, g)$ be a Riemannian manifold, $p \in M$ and $\gamma: [0, l] \to M$ a unit speed geodesic starting at $p$. By the theorem \ref{thm: split Jacobi equation}, we can use a symmetric operator $S(t)$ along the geodesic $\gamma$ to split the Jacobi equation: 
\begin{equation*}
D_t^2 J + R_{\nu}(J)= 0 \iff
    \begin{cases}
        D_tJ = SJ\\
        D_tS + S^2 + R_{\nu} = 0 \quad(\text{Riccati Equation})
    \end{cases}
\end{equation*}
Let $Y$ be a normal unit vector field parallel along $\gamma$. We use $\inp{\cdot, \cdot}$ to denote $g(\cdot, \cdot)$. Let $y(t) = \inp{SY, Y}(t)$. Then by the Riccati equation and the fact that $S$ is symmetric, i.e. $\inp{SX, Y} = \inp{X, SY}$, 
\begin{align*}
	y'(t) =& \inp{(D_t S)Y, Y} = -\inp{S^2 Y, Y} - \inp{R_{\nu}(Y), Y}\\
	=& -\abs{SY}^2 - \kappa(t)
\end{align*}
where $\kappa(t) = \sect(\sigma(t)) = \inp{R_{\dga(t)}(Y), Y}$ and $\sigma(t) = \mathbf{Span}(\dga(t), Y(t))$.

By Cauchy Schwartz: $\inp{SY, Y}^2 \leq \abs{SY}^2\cdot \abs{Y}^2 = \abs{SY}^2$, we can conclude that
\begin{equation*}
	y' = -\abs{SY}^2 - \kappa(t) \leq -\inp{SY, Y}^2- \kappa(t) =-y^2- \kappa(t)
\end{equation*}
That is 
\begin{equation}\label{eq: y' + y^2 + k < 0}
	y'+y^2+\kappa(t) \leq 0
\end{equation}
\begin{lemma}\label{lem-riccati}
	if we have
    \begin{equation*}
        \begin{cases}
            g'(t) + g^2(t) + \kappa_1(t) = 0\\
            G'(t) + G^2(t) + \kappa_2(t) = 0
        \end{cases}
    \end{equation*}
    and $\kappa_1 \geq \kappa_2$. Then,
    \begin{itemize}
    	\item if $g(0) \leq G(0) \implies g(t)\leq G(t)$ for $t \geq 0$;
    	\item if $g(0) \geq G(0) \implies g(t)\geq G(t)$ for $t \geq 0$;
    \end{itemize}
\end{lemma}
\begin{proof}
	Take the subtraction, then 
    \begin{equation*}
       g' - G' + \underbrace{ g^2 - G^2}_{(g - G)(g + G)} = \kappa_2 - \kappa_1 \leq 0
    \end{equation*}
    Take $u = g - G$, then we have
    \begin{equation*}
        u' + u(g + G) \leq 0
    \end{equation*}
    Multiplying both sides by $e^{\int g + G}$ then 
    \begin{equation*}
        \underbrace{u'e^{\int g + G} + u(g + G)e^{\int g + G}}_{(ue^{\int g + G})'} \leq 0
    \end{equation*}
    Therefore, $ue^{\int g + G}$ is non-increasing. If $g(0) \leq G(0)$, then $u(0) \leq 0 $ and hence $u(t) \leq 0$ for $t \ge 0$, or equivalently $ g(t)\le G(t)$ for $t \geq 0$. Similarly if $g(0) \geq G(0)$, then $u(0) \geq 0 $ and hence $u(t)\geq 0$ for $t\leq 0$ i.e. $ g(t)\ge G(t)$ for $t\leq 0$. 
\end{proof}
\begin{corollary}
    Let $\kappa \in \R$, consider
    \begin{equation*}
        \begin{cases}
            g'(t) + g^2(t) + \kappa \leq 0\\
            G'(t) + G^2(t) + \kappa \geq 0
        \end{cases}
    \end{equation*}
    Then 
    \begin{equation*}
        \begin{cases}
            g'(t) + g^2(t) + \kappa_1(t) = 0\\
            G'(t) + G^2(t) + \kappa_2(t) = 0
        \end{cases}
    \end{equation*}
 	For some $\kappa_1(t) \geq \kappa \geq \kappa_2(t)$. Apply what we just discussed, then  
    \begin{itemize}
        \item If $g(0) \leq G(0) \implies g(t) \leq G(t)$ for $t > 0$;
        \item If $g(0) \geq G(0) \implies g(t) \geq G(t)$ for $t < 0$;
    \end{itemize}
\end{corollary}

Therefore, for Riemannian manifold $(M^n, g)$ of sectional curvature bounded below, i.e. $\sect_M \geq \kappa$, $\kappa(t) \geq \kappa$ for all $t \in [0, l]$. Hence for $y = \inp{SY, Y}$, by \ref{eq: y' + y^2 + k < 0}, we have
\begin{equation*}
	y' + y^2 + \kappa \leq 0
\end{equation*}
we can now compare it to what we obtained in $\modsp{\kappa}$, i.e. $\bar{y}' + \bar{y}^2 + \kappa = 0$. 
Especially, we need to be careful with the blowing up case when $y(t) \to \infty$ as $t \to 0^+$, which happens when $J(0) = 0$ and $D_tJ(0) \neq 0$. This is because for $f = d(\cdot, p)$ on $S_t(p)$, the shape operator $S_t \sim \frac{1}{t}$ as $t \to 0^+$. In order to address this case, we need the following lemma. 
\begin{lemma}\label{lem: g(t) < ct_k}
    Let $g:(0, T) \to \R$ satisfies 
    \begin{itemize}
        \item $g' + g^2 + \kappa \leq 0$;
        \item $\lim_{t \to 0^+}g(t) = \infty$
    \end{itemize}
    Then $g(t) \leq \ctk(t) = \frac{\csk(t)}{\snk(t)}$ for all $t$  for which  $g(t)$ is defined
\end{lemma}
\begin{proof}
    Suppose not. that is $\exists t_0 > 0$ such that $g(t_0) > \ctk(t_0)$. This implies $\exists \eps > 0$ such that $g(t_0) > \ctk(t_0 - \eps)$ by the continuity of $\ctk$. Define $G(t) = \ctk(t - \epsilon)$. It satisfies
    \begin{equation*}
        G' + G^2 + \kappa = 0
    \end{equation*}
    where $G(t_0) = \ctk(t_0 - \epsilon) < g(t_0)$. Thus by the above discussion, we have 
    \begin{equation*}
        G(t) < g(t) \quad\forall t \leq t_0
    \end{equation*}
    Now we have the contradiction: $g$ is finite over $(0, T)$. 
    \begin{equation*}
        \infty > g(\eps) > \lim_{t \to \eps^+}G(t) = \lim_{(t - \eps) \to 0^+}g(t - \eps) = \infty \quad\lightning
    \end{equation*}
    contradiction.
\end{proof}

Now we can prove the Hessian comparison theorem. 

\begin{proof}[Proof of the Hessian comparison theorem~\ref{cor: Hess ineq}]
    Let $\eta(t)$ be a unit speed radial geodesic starting at $p$. Let $S(t)$ be the shape operator of the sphere $S_t(p)$ at $\eta(t)$ . Let $Y$ be a unit normal vector field parallel along $\eta$. If we take $y(t) = \inp{SY, Y}(t)$, then $a$ satisfies $y' + y^2 + \kappa \leq 0$ and $\lim_{t \to 0^+}y(t) = \infty$. By the lemma \ref{lem: g(t) < ct_k}, we have $y(t) \leq \ctk(t)$. Since $y(t) = \inp{SY, Y}(t)$, then 
    \begin{equation*}
    	y(t) \leq \ctk(t) \implies \underbrace{\max_{\abs{Y} = 1}\inp{S(t)Y, Y}}_{\text{$\lambda_{\max}(S(t))$ is the largest eigenvalue}} \leq \ctk(t)
    \end{equation*}
    Therefore, $S(t) \leq \ctk(t)\pi_{d_p}$ where $\pi_d$ is the projection map to $T_{\eta(t)}S_t(p)$. Recall that $S(t) = \hess_{d_p}$ on $T_{\eta(t)}S_t(p) = (\Dot{\eta}(t))^\perp$.
   And in the direction parallel to $\dot{\eta}(t)$ we have that both $ \hess_{d_p}=0$ and $\pi_{d_p}=0$.

    Thus we got the inequality version of the proposition \ref{prop: projection ctk}:
    \begin{equation}\label{eq: Hessian inequality for d_p}
    	\hess_{d_p} \leq \ctk(d_p)\pi_{d_p}
    \end{equation}
    If we switch $d_p$ to $\mdk\circ d_p$, then by our computation in \ref{thm: hessian csk}, $\hess_{\mdk\circ d_p} \leq \csk(d_p)\id$. Therefore, we can conclude \ref{eq: hessian comparison theorem} the inequality version of the corollary \ref{cor: key of scalar hessian}
\end{proof}
As we explained before, along any other unit speed geodesics $\gamma(t)$ in the $M^n$ with $\sect_M \geq \kappa$,
because
    \begin{align*}
        & (f(\gamma(t)))'' = (\nabla^2 f)(\dga(t), \dga(t)) = \inp{\hess_f(\dga(t)), \dga(t)}.
    \end{align*}
it holds that 
    \begin{equation*}
        (f(\gamma(t)))''  + \kappa f(\gamma(t)) \leq 1
    \end{equation*}


\chapter{Local Comparison Theorems}

\section{Rauch Comparison Theorem}
\begin{theorem}[Rauch ($I \& II$) Comparison Theorem]\label{thm: Rauch}
    Let $(M^n, g)$ be a Riemannian manifold, $\sect_M \geq \kappa$. Let $\gamma: [0, l] \to M$ be a unit speed geodesic, i.e. $\gamma(t) = \exp_p(tv)$. Let $J$ be a normal Jacobi field along $\gamma$. 
    We will also denote by $\Tilde{\gamma}$ a unit speed geodesic in $\modspace{\kappa}$ and by $\Tilde{J}$ a normal Jacobi field along $\Tilde{\gamma}$. Suppose one of the following holds
    \begin{itemize}
        \item \textbf{Rauch $I$:} $J(0) = \Tilde{J}(0) = 0$ and $\abs{D_tJ(0)} = \abs{D_t\Tilde{J}(0)} \neq 0$ and there are no conjugate points to $p$ along $\gamma$ on $[0,l)$ ; or
        \item \textbf{Rauch $II$:}  $|J(0)| = |\Tilde{J}(0)| =1$ and $D_tJ(0) = D_t\Tilde{J}(0)= 0$ and there are no focal points along $\gamma$ on $[0,l)$ for the geodesic submanifold defined $N$  by $\dga(0)$ ;
    \end{itemize}
    Then $\frac{|J(t)|}{\abs{\tilde{J}(t)}}$ is non-increasing. And $\abs{J(t)} \leq \abs{\Tilde{J}(t)}$ on $[0,l]$.
    
\end{theorem}

\begin{remark}
The conclusion of the theorem implies that in both cases of Rauch I and Rauch II, the first zero of $J$ (if it exists) must occur before the first zero of $\tilde J$. In our applications Rauch I will only be used for shortest geodesics in which case the no conjugate points assumption is always satisfied.
\end{remark}

\begin{proof}[Proof of Rauch Comparison Theorem $I \& II$]
Let us treat Rauch I first.

As before, we denote by $S(t)$ the shape operator of $S_t(p)$. By Remark~\ref{exist-S-I}  the assumption that there are no conjugate points along $\gamma$ guarantees that $S(t)$ is smooth on $(0,l]$ and $S(t)\sim \frac{1}{t}\id$  as $t\to 0$.

Recall that we can convert the Jacobi equation of $J$ and $\Tilde{J}$ to 
    \begin{equation}\label{eq-ricc}
        \begin{cases}
            D_tJ = SJ\\
            D_tS + S^2 + R_J = 0
        \end{cases}
        \quad\&\quad
        \begin{cases}
            D_t\Tilde{J} = S\Tilde{J}\\
            D_tS + S^2 + R_{\Tilde{J}} = 0
        \end{cases}
    \end{equation}
with initial conditions $J(0)=\tilde J(0)=0$ and $S(t)\sim \frac{1}{t}\id$ and $\tilde S(t)\sim \frac{1}{t}\id$ as $t\to 0$.

    Look at
    \begin{align*}
        \frac{\abs{J}'}{\abs{J}} = & \frac{(\sqrt{\inp{J, J}})'}{\abs{J}}\\
        = & \frac{\frac{1}{2\sqrt{\inp{J, J}}}\cdot (\inp{J, J})'}{\abs{J}} \\
        = & \frac{\cancel{2}\inp{D_tJ, J}}{\cancel{2}\abs{J}^2} = \frac{\inp{D_tJ, J}}{\abs{J}^2}\\
        = & \underbrace{\frac{\inp{SJ, J}}{\inp{J}^2} 
        \leq \ctk(t) = \frac{\inp{\Tilde{S}\Tilde{J}, \Tilde{J}}}{\abs{\Tilde{J}}^2}}_{\text{$S = S(t) \leq \ctk(t)\id$ and the equality holds in  $\modsp{\kappa}$}}\\
        = & \frac{\abs{\Tilde{J}}'}{\abs{\Tilde{J}^2}}
    \end{align*}

    Therefore, we proved
    \begin{align*}
        & (\ln{\abs{J}})' = \frac{\abs{J}'}{\abs{J}} \leq \frac{\abs{\Tilde{J}}'}{\abs{\Tilde{J}}} = (\ln{\abs{\Tilde{J}}})'\\
        \implies & \ln{\brac{\frac{\abs{J}}{\abs{\Tilde{J}}}}}' \leq 0\\
        \implies & \text{$\ln{\brac{\frac{\abs{J}}{\abs{\Tilde{J}}}}}$ is non-increasing}\\
        \implies & \text{$\frac{\abs{J}}{\abs{\Tilde{J}}}$ is non-increasing by the monotonicity of $\ln{\empty}$}
    \end{align*}
    By our assumption, (using L'Hopital's rule), then $\lim_{t \to 0^+}\frac{\abs{J(t)}}{\abs{\Tilde{J}(t)}} = 1$. Because it goes down, the first zero of $J$ occurs before the first zero of $\Tilde{J}$. 
    This proves Rauch $I$.
    
    The proof of Rauch $II$ is very similar. We only indicate the differences.
    
    Take a hypersurface $N^{n-1}$ containing $p$ and such that $T_pN=\gamma(0)^\perp$ and  its second fundamental form is zero at $p$. For example one can take $N$ to be the image under $\exp_p$ of a small ball in $B_\eps(0)\cap\gamma(0)^\perp\subset T_pM$. Then set $S(t)$ to be the shape operator of of the $t$-sphere around $N$ at $\gamma(t)$. Then the comparison argument is the same except we get $S = S(t) \leq -k\tnk(t)\id$ and the equality holds in the $\modsp{\kappa}$. Then the initial conditions for system \eqref{eq-ricc} become $|J(0)|=|\tilde J(0)|=1$ and $S(0)=\tilde S(0) =0$.  The rest of the proof is the same except we don't need to use L'Hopital.    
\end{proof}
\begin{remark}
The above proof only works so long as there are no conjugate points (Rauh I) or focal points (Rauch II) along $\gamma$. In particular, it means that in either case  $J(t)$ is not zero on $(0,l)$. Nothing can be said after the first zero of $J$.
\end{remark}

\begin{proposition}[Rigidity Case of Rauch Comparison Theorem]\label{prop: rigidity rauch}
	Under the assumptions of Rauch comparison if there is a positive $t_0\le l$ such that $|J(t_0)|=|\bar J(t_0)|$ then $|J(t)|=|\bar J(t)|$ on $[0,t_0]$ and moreover there is a parallel normal vector field $Y$ along $\gamma$ such that on  $[0,t_0]$ it holds that 
    \begin{itemize}
        \item \textbf{Rauch $I$:} $J(t)=\snk(t) Y(t)$;
        \item \textbf{Rauch $II$:}  $J(t)=\csk(t) Y(t)$.
    \end{itemize}
\end{proposition}
\begin{proof}
We will only give proof for Rauch II, the proof for Rauch I is similar. Without loss of generality $|J(0)|=1$.

    Suppose $|J(t_0)|=|\tilde J(t_0)|\ne 0$  and $J(t)\ne 0$ for $0<t\le t_0$.  Then by monotonicity of $\frac{\abs{J}}{\abs{\Tilde{J}}}$ we have that $|J(t)|=|\tilde J(t)|=\csk(t)$ for all $0\le t\le t_0$.

      The proof of Rauch comparison gives that
$ \frac{\inp{SJ, J}}{\inp{J}^2} 
        \leq-k\tnk(t)= \frac{\inp{\Tilde{S}\Tilde{J}, \Tilde{J}}}{\abs{\Tilde{J}}^2}$ and $S(t)\le -k\tnk(t)\id$ before the first zero of $J$.
        
        If the inequality is strict at any point  $t\in [0,t_0]$ then the proof gives that $|J(t_0)|< |\tilde J(t_0)|$ which we know is false.        
        
        Recall that if a symmetric matrix $A$  satisfies $A\le \lambda \id$ and $\langle AV,V\rangle =\lambda |V|^2$ for some nonzero vector $V$ then $V$ is a $\lambda$-eigenvector of $A$, i.e. $AV=\lambda V$.
        
        Hence $S(J(t))=-k\tnk(t)J(t) $ for any $t\in [0,t_0]$. But $J'=SJ$ which means that $J'=-k\tnk(t)J(t) $ on  $ [0,t_0]$.   Let $Y$ be parallel  along $\gamma$ with  $Y(0)=J(0)$ and Let $V=\csk(t) Y(t)$. Then $V$ also satisfies $V'=-
k\tnk(t)V$ on $ [0,t_0]$. Hence both $V$ and $J$ satisfy the same first-order IVP and hence $J=V$ on  $ [0,t_0]$.

\end{proof}

\section{Berger Comparison Theorem}
We are going to introduce the Berger comparison theorem \cite{Be62} in this section. This is useful in proving the concavity of the distance function. 
The following Berger comparison theorem is an important implication of the Rauch $II$ comparison theorem.

\begin{theorem}[Berger Comparsion Theorem (Figure~\ref{fig: Berger's comparison})]\label{lem: Berger comparison}
    Let $(M^n, g)$ be a manifold with $\sect_M \geq \kappa$. Let $\gamma: [0, l] \to M$ be a (not necessarily shortest) unit speed geodesic in $M$. Let $V$ be a unit parallel vector field along $\gamma$ such that for all $t \in [0, l]$, $V(t) \perp \Dot{\gamma}(t)$. Let
    \begin{equation*}
        \gamma(t, s) = \exp_{\gamma(t)}(sV(t)).
    \end{equation*}
    
    Now, consider $\tilde{\gamma}$, $\tilde{V}$, $\tilde{\gamma}(t, s)$ in the corresponding picture in the model space $\modsp{\kappa}$. Namely,
    \begin{equation*}
        \tilde{\gamma}(t, s) = \exp_{\tilde{\gamma}(t)}(s\tilde{V}(t)).
    \end{equation*}

    Then for all small $s>0$ it holds that
    \begin{equation*}
        \length(\gamma_s) \leq \length(\tilde{\gamma}_s)
    \end{equation*}
    where
    \begin{equation*}
        \gamma_s(t) = \gamma(t, s),\quad \tilde{\gamma}_s(t) = \tilde{\gamma}(t, s)
    \end{equation*}
    In the special case $\kappa=0$, the above estimate becomes 
     \begin{equation*}
        \length(\gamma_s) \leq \length(\tilde{\gamma}_0)
    \end{equation*}
    for all small $s$.
    \begin{figure}[htbp]
    \centering
        \includegraphics[width=0.7\textwidth]{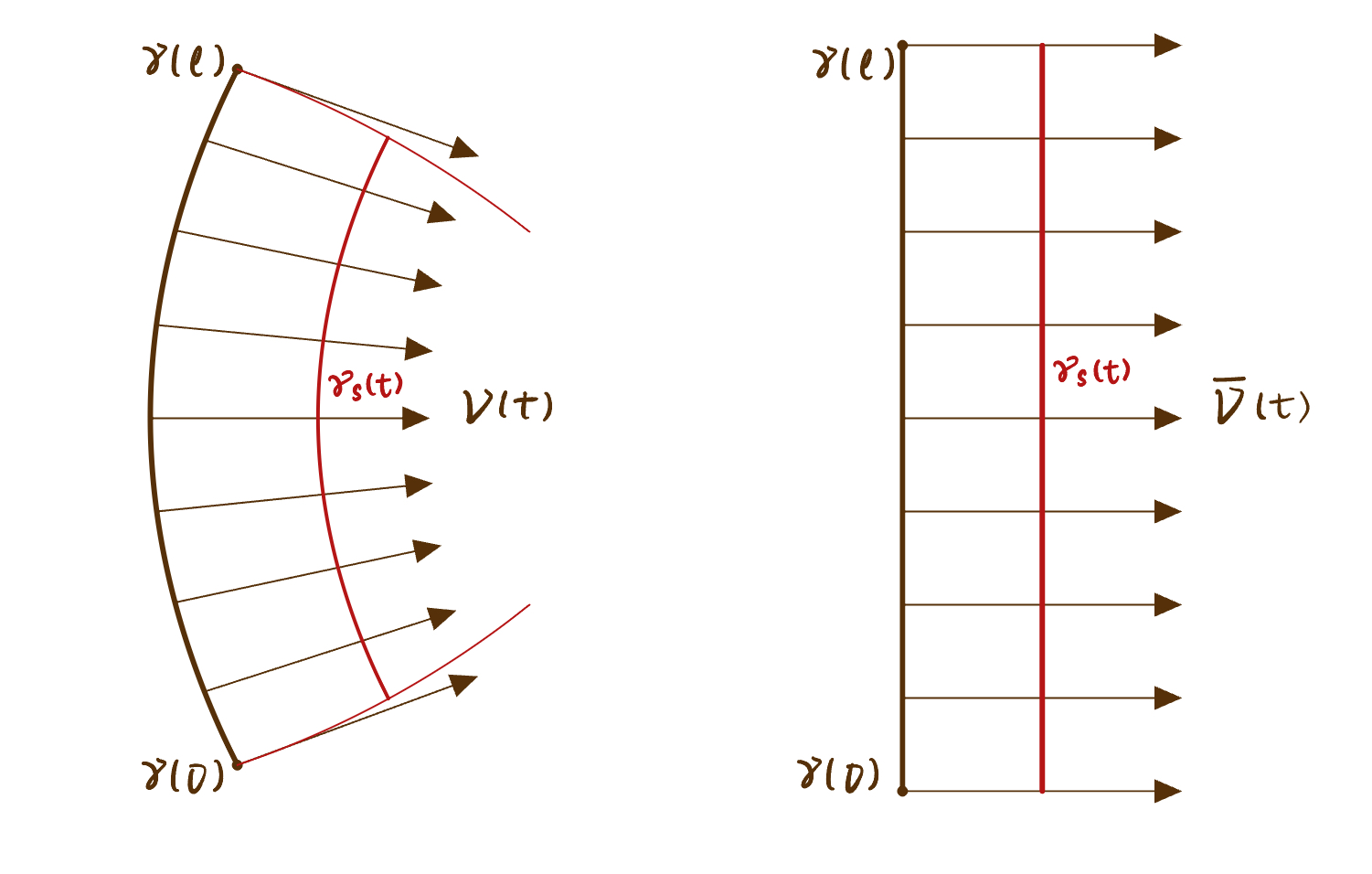}
        \caption{Berger's comparison theorem.}
        \label{fig: Berger's comparison}
    \end{figure}
\end{theorem}
\begin{proof}
    Use Rauch $II$ in \ref{thm: Rauch}. Since for any fixed $t$ the curve $s \mapsto \gamma(t, s)$ is a geodesic, $Y=\frac{\partial \gamma}{\partial t}$ is Jacobi along $s \mapsto \gamma(t, s)$. Also, at $s = 0$, $Y(0) = \dga(0)$ ad $\abs{Y(0)} = 1$.
    \begin{equation*}
        Y'(0) = \frac{D}{ds}\brac{\frac{\partial \gamma}{\partial t}} = \frac{D}{dt}\brac{\frac{\partial \gamma}{\partial s}}\Big|_{t,s = 0} = \frac{D}{dt}(V) = 0
    \end{equation*}
    Notice that the last equality holds because $V$ is a parallel vector field. Thus, we know that $Y'(0) = 0$ and $\abs{Y(0)} = 1$. Similarly, the above also works for $\tilde{Y} = \frac{\partial \tilde{\gamma}}{\partial t}$. Then by Rauch $II$, we have that
    \begin{align*}
        & \abs{Y(s)} \leq \abs{\tilde{Y}(s)} \quad\text{for small $s$}\\
        \implies & \abs{\dga_s(t)} \leq \abs{\dot{\tilde{\gamma}}_s(t)} \quad\text{for any $t$ (for small $s$)}\\
        \implies & \length(\gamma_s) \leq \length(\tilde{\gamma}_s) \quad\text{for small $s$}
    \end{align*}

In particular, if $\sect_M \geq 0$, then the model space $M^n_\kappa$ is just $\R^n$. 
Then $\tilde{\gamma}$ is a straight line and $\tilde{Y}$ a constant vector field along $\tilde{\gamma}$. 
Notice that we can write the straight line $\tilde{\gamma}$ as a straight passing through $0$, i.e. $\tilde{\gamma} = tV$ (WLOG we assume $\bar p=0$) for some vector $V$. 
Then 
    \begin{equation*}
        \tilde{\gamma}(t, s) = tV + sY.
    \end{equation*}
In this case, $\length(\tilde{\gamma}_s)$ is constant in $s$. 
In Rauch $II$, $\frac{\partial \tilde \gamma}{\partial t}$ is parallel along $s \mapsto \tilde{\gamma}(t, s)$.
By Berger's comparison theorem \ref{lem: Berger comparison}, 
    \begin{equation*}
        \length(\gamma_s) \leq \length(\tilde{\gamma}_s) \quad\text{for small $s$ and $\length(\tilde{\gamma}_s)$ is a constant}. 
    \end{equation*}
    
    \end{proof}
    This immediately gives.
    \begin{corollary}\label{berger-k=0}
    Suppose under the assumptions of Berger's comparison $\gamma_0$ is shortest and $\kappa=0$.
    Then $d(\gamma(0,s), \gamma(l,s))\le d(\gamma(0,0), \gamma(l,0))$ for all small $s$.
    
    \end{corollary}
    
Moreover, for later applications we will need to understand the rigidity case in the above corollary.


\begin{proposition}[Rigidity Case of Berger Comparison Theorem]\label{lem: rigidity case}
Suppose that in the assumptions of the Corollary \ref{berger-k=0} we have that 
\[
d(\gamma(0,s_0), \gamma(l,s_0))=d(\gamma(0,0), \gamma(l,0))
\]
 for some small $s_0 > 0$.

Then 
    \begin{equation*}
        \br{\gamma(t, s): 0 \leq t \leq l, 0\leq s\leq s_0}
    \end{equation*}
    is a totally geodesic flat isometrically immersed rectangle in $M$. 
\end{proposition}
\begin{proof}

Recall that a submanifold $N \subseteq (M, g)$ is called totally geodesic if for each $p \in N$, there is a neighborhood of $U_\delta(0) \subseteq T_pN \subseteq T_pM$ is mapped  into $N$ via the exponential map of $M$, i.e. $\exp^M_p(U_\delta(0)) \subseteq N$. 
This is well known to be equivalent to the second fundamental form of $N$ vanishing.

And a Riemannian manifold is called flat if it is locally isometric to Euclidean space. This is well known to be equivalent to the curvature of this manifold to be identically zero,


By Berger's comparison, we have that 
\begin{equation*}
	d(\gamma(0,s_0), \gamma(l,s_0))=\length(\gamma_{s_0})\ge \length(\gamma_0)=d(\gamma(0,0), \gamma(l,0))
\end{equation*}
However, we are given that $d(\gamma(0,s_0), \gamma(l,s_0))=d(\gamma(0,0), \gamma(l,0))$ which means that all of the above inequalities are equalities. In other words

\begin{equation}
d(\gamma(0,s_0), \gamma(l,s_0))=\length(\gamma_{s_0})=\length(\gamma_0)=d(\gamma(0,0), \gamma(l,0))
\end{equation}

In particular $\length(\gamma_0)= \length(\gamma_{s_0})$

In the proof of Berger's comparison we have that for any small fixed $s$ we have $|Y(t,s)|\le |\tilde Y(t,s)|$ and $\length (\gamma_s)=\int_0^l|Y(t,s)| dt$ and similarly $\length (\tilde \gamma_s)=\int_0^l \tilde Y(t,s)| dt$.

Hence the equality $\length(\gamma_{s_0}) = \length(\tilde{\gamma}_{s_0})$ implies that   $|Y(t,s_0)|= |\tilde Y(t,s_0)|$  for any $t$. By the rigidity case of Rauch II with $\kappa=0$ this implies  that for any fixed $t$ the field  $Y(t,s)=\frac{\partial \gamma}{\partial t}$ is parallel along $s \to \gamma(t, s), 0\le s\le s_0$.

Moreover, since $\frac{\partial \gamma}{\partial s}$ is also parallel along this curve we wee see that  $\frac{\partial \gamma}{\partial s}$ and  $\frac{\partial \gamma}{\partial t}$ are orthonormal for any $t\in[0,l,s\in [0,s_0] $.
This immediately gives that the map $\gamma$ is an isometric immersion.

It remains to be observed that this immersion is totally geodesic. We already know  that $ \frac{D}{ds}\brac{\frac{\partial \gamma}{\partial t}} = \frac{D}{dt}\brac{\frac{\partial \gamma}{\partial s}} = \frac{D}{ds}\brac{\frac{\partial \gamma}{\partial s}} =0$. We will show that  $\frac{D}{dt}\brac{\frac{\partial \gamma}{\partial t}} =0$ too.

 Recall that we proved that 
 \begin{equation*}
 	\length(\gamma_{s_0})=d(\gamma(0,0), \gamma(l,0))=d(\gamma(0,s), \gamma(l,s))
 \end{equation*}
 
 for any $s\in [0,s_0]$. Since $|Y|\equiv 1$ for $s\le s_0$ we also have that $\gamma_s$ is unit speed for $s\le s_0$. Hence it's a geodesic and therefore $\frac{D}{dt}\brac{\frac{\partial \gamma}{\partial t}} =0$ for any $t\in[0,l,s\in [0,s_0] $. This proves that the map $\gamma$ restricted to this rectangle has trivial second fundamental form and hence is totally geodesic.

	\begin{figure}[htbp]
    \centering
        \includegraphics[width=0.7\textwidth]{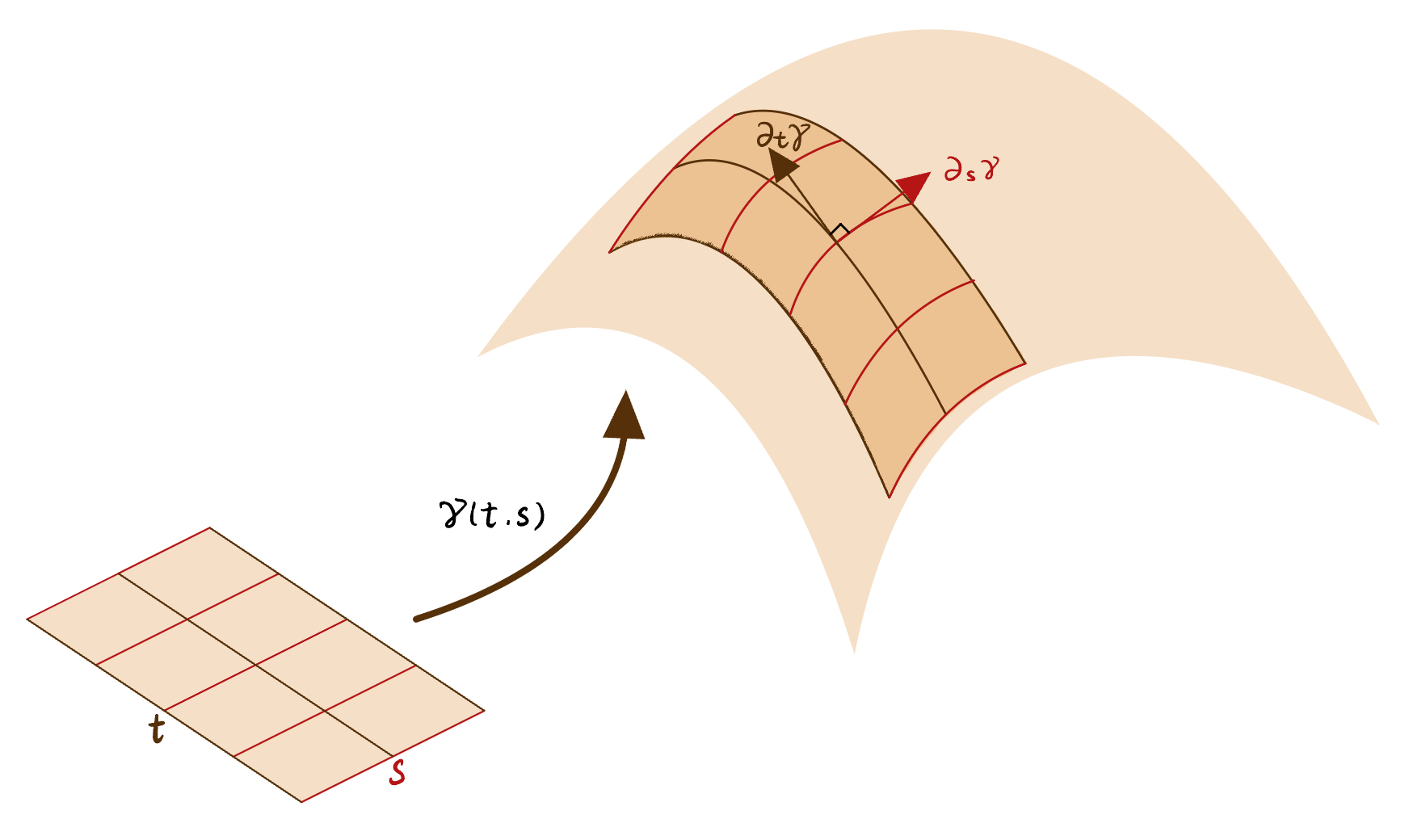}
        \caption{The totally geodesic flat rectangle is an isometric embedding.}
    \end{figure}
\end{proof}

\section{Local Toponogov Comparison Theorem}
The most important corollary that we are introducing now is the hinge version of the Toponogov comparison theorem, which is further equivalent to the angle version of the Toponogov comparison theorem. 
\begin{theorem}[Local Toponogov Hinge Comparison]\label{thm: hinge}
    Let $(M^n, g)$ be a Riemannian manifold of $\sect_M \geq \kappa$.
    Fix $z \in M$. There is a small  $\eps\in (0, \frac{\pik}{2})$ such that the following holds.
For any  $p, x, y \in B_\eps(z)$ we have
    \begin{equation*}
        \abs{x - y} \leq \modcvee^\kappa(\mangle[p_x^y]; \abs{x - p}, \abs{y - p}).
    \end{equation*}
\end{theorem}
\begin{figure}[htbp]
    \centering
        \includegraphics[width=0.6\textwidth]{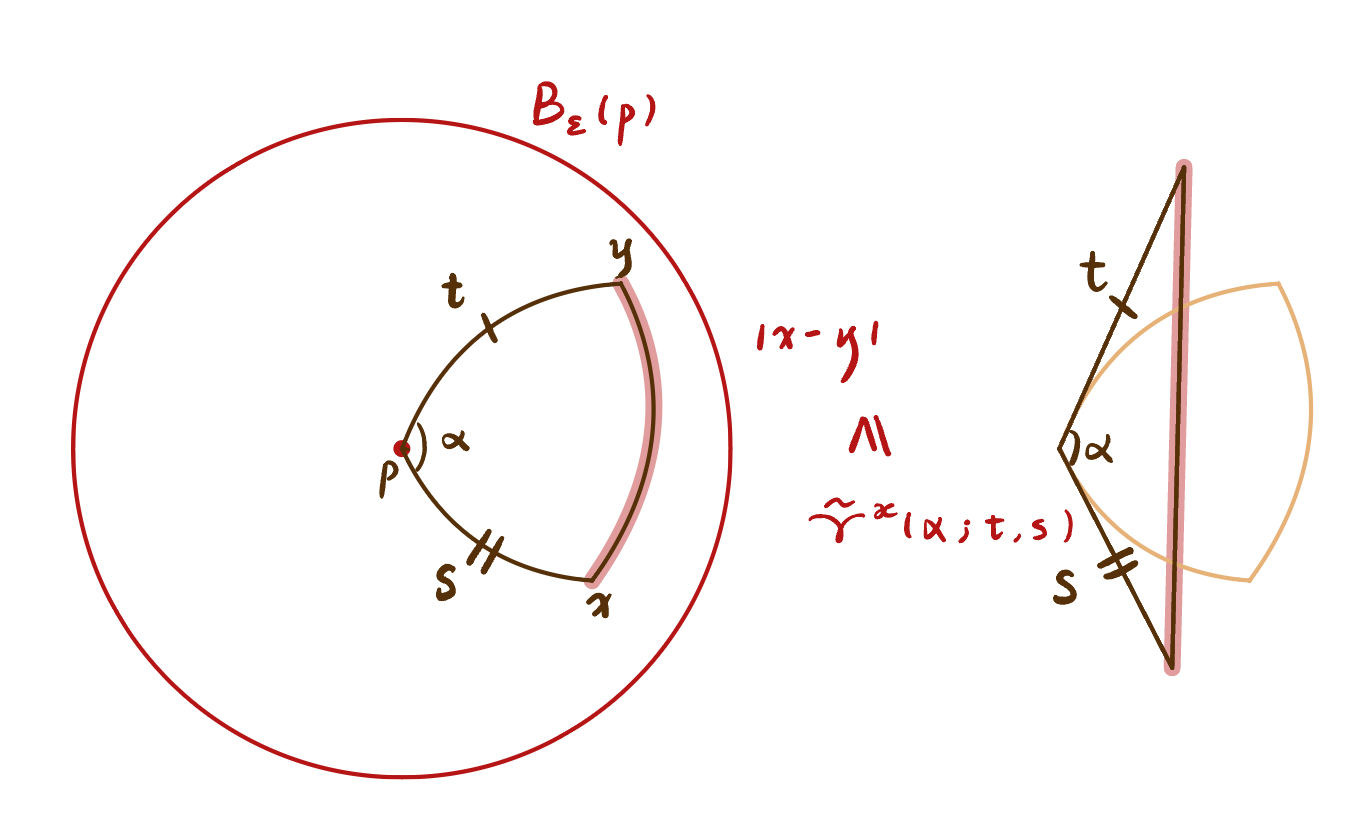}\caption{Local Hinge Comparison theorem \ref{thm: hinge} where $t = \abs{p - y}$, $s = \abs{p - x}$ and $\alpha = \mangle[p_x^y]$.}
\end{figure}
\begin{theorem}[Local Toponogov Angle Comparison]\label{thm: angle}
    Let $(M^n, g)$ be a Riemannian manifold of $\sect_M \geq \kappa$. 
   	Fix $z \in M$. There is a small $\eps\in (0, \frac{\pik}{2})$ such that the following holds.
      For any $p, x, y \in B_\eps(z)$  and the comparison triangle $[\Tilde{p}\Tilde{x}\Tilde{y}] := \modtriangle^\kappa(p, x, y) := \modtriangle^\kappa(\abs{p - x}, \abs{p - y}, \abs{x - y})$, we have
    \begin{align*}
        & \mangle[p_x^y] \geq \mangle[{{\Tilde{p}_{\Tilde{x}}}^{\Tilde{y}}}]\\
        & \mangle[x_p^y] \geq \mangle[{{\Tilde{x}_{\Tilde{p}}}^{\Tilde{y}}}]\\
        & \mangle[y_x^p] \geq \mangle[{{\Tilde{y}_{\Tilde{x}}}^{\Tilde{p}}}]
    \end{align*}
\end{theorem}
\begin{figure}[htbp]
    \centering
        \includegraphics[width=0.6\textwidth]{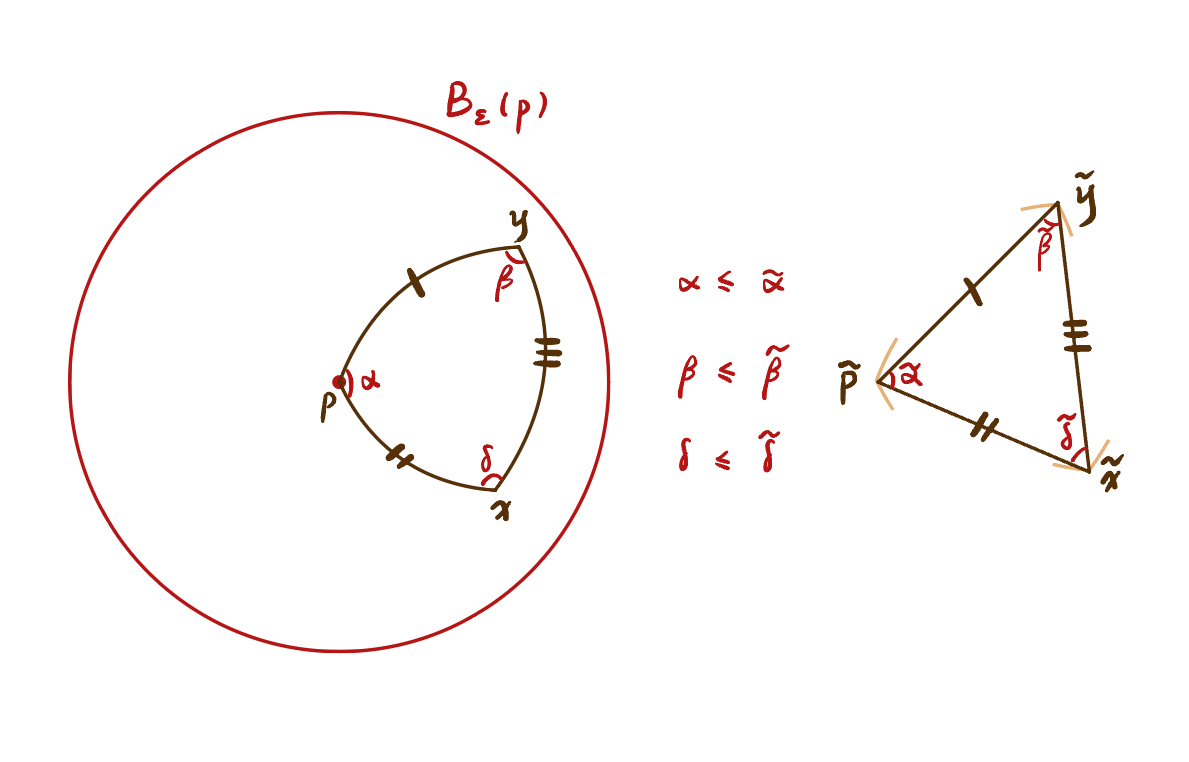}\caption{Local Angle Comparison theorem \ref{thm: angle}.}
\end{figure}
We are going to argue that the two versions of local Toponogov comparison theorems are equivalent. 
\begin{proof}[Proof (Local Hinge Comparison to Local  Angle Comparison)]
We only need to prove $\alpha \leq \tilde{\alpha}$ and the same argument works for the other two angles. 
The key observation is that $\alpha\mapsto-\cos\alpha$ is monotone in creasing on $(0,\pi)$ and hence by the cosine law in the model plane the map $\alpha\mapsto  \modcvee^\kappa(\alpha; t, s)$ is monotone increasing for any fixed $0<t,s<\pi_k$.

Suppose the hinge comparison holds. Consider the triangle $\triangle(p, x, y)$ in $M$ and denote $t = \abs{p - y}$, $s = \abs{p - x}$ and $\alpha = \mangle[p_x^y]$. We also define $d(t, s) = \abs{x - y}$. Similarly, we define $\tilde{d}(t, s) = \modcvee^\kappa(\alpha; t, s)$. On the other hand, consider the triangle $[\Tilde{p}\Tilde{x}\Tilde{y}]$.
Then by the hinge comparison, 
\begin{equation*}
	\tilde{d}(t, s) \geq d(t, s).
\end{equation*}
 By its definition, we know that $t = \abs{\tilde{p} - \tilde{y}}$, $s = \abs{\tilde{p} - \tilde{x}}$ and $d(t, s) = \abs{\tilde{x} - \tilde{y}}$. Now by the cosine law, we know that
 \begin{equation*}
 	\tilde{\alpha} := \modangle^\kappa\br{d(t, s); t, s} \leq \modangle^\kappa\{\tilde{d}(t, s); t, s\} = \alpha
 \end{equation*}

	\begin{figure}[htbp]
    \centering
        \includegraphics[width=0.6\textwidth]{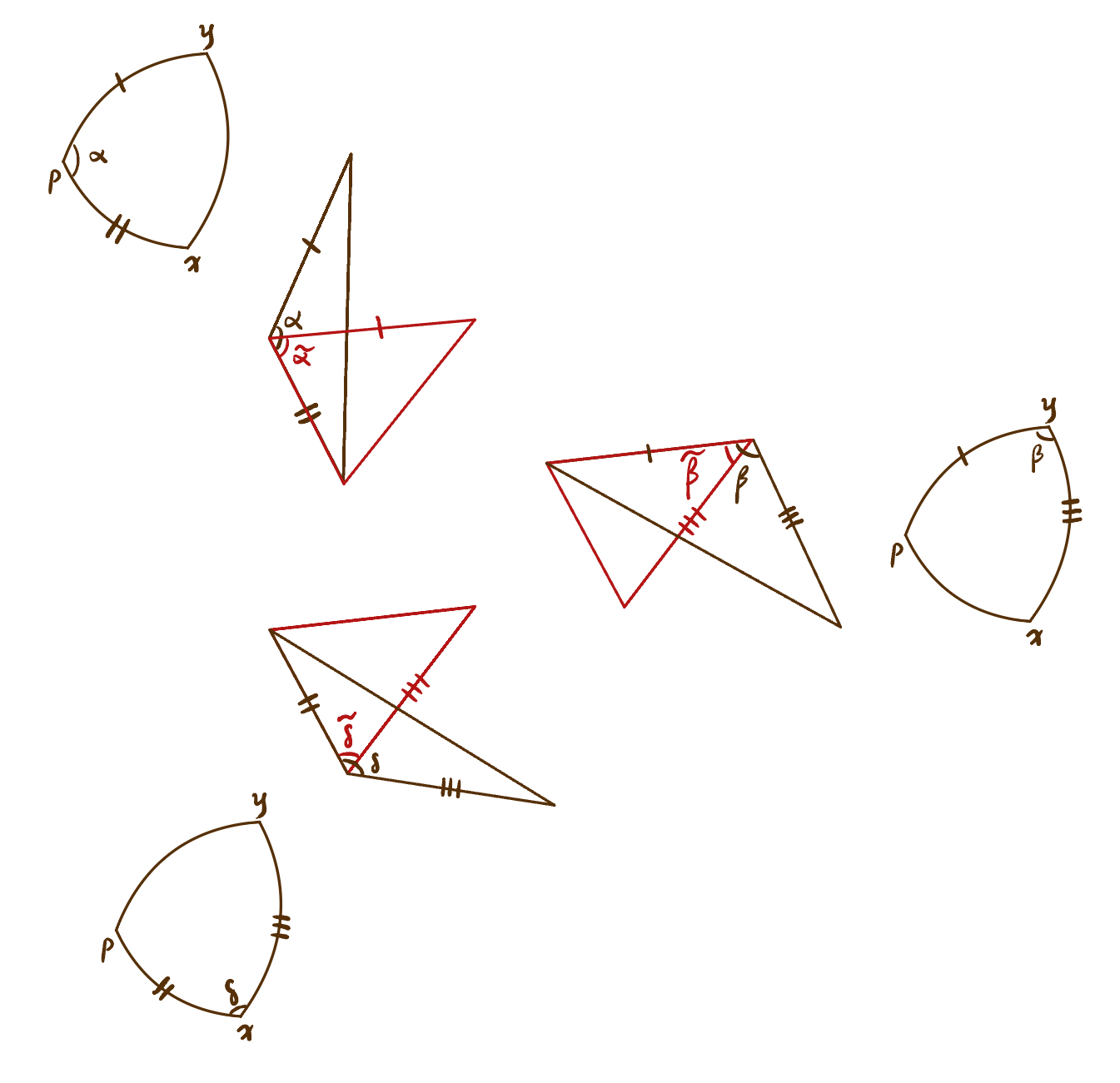}\caption{The same argument works for all the three angles}
\end{figure}
It is also easy to see that the local angle comparison implies the local hinge comparison. The arguments are the same and the key point is still the cosine law. 
\end{proof}
Now we relate the local Toponogov theorem to Rauch comparison theorem
\begin{theorem}
	Rauch $I$ comparison implies the hinge comparison
\end{theorem}
\begin{proof}
    To show how Rauch $I$ comparison implies local hinge comparison. Take $p \in M^n$ and $\tilde{p} \in \modsp{\kappa}$. Notice that both $\exp_p$ and $\exp_{\tilde{p}}$ are diffeomorphism near $p$ and $\tilde{p}$ respectively. 
    
    Especially in $\modsp{\kappa}$, we know that $\exp_{\tilde{p}}: B_{\varpi^\kappa}(\bar{0}) \subseteq T_{\tilde{p}}\modsp{\kappa} \to \modsp{\kappa}$ is a diffeomorphism near $0$, we can pull back the Riemannian metric from $\modsp{\kappa}$ to $B_{\varpi^\kappa}(\bar{0}) \subseteq T_{\tilde{p}}\modsp{\kappa}$, i.e $\tilde{h} = \exp_{\tilde{p}}^*(\tilde{g})$ where $\tilde{g}$ is the Riemannian metric of $\modsp{\kappa}$. Take some isometry $I: T_p\modsp{\kappa} \to T_pM$. Consider the map 
    \begin{equation*}
    	\Phi = \exp_p\circ I: (T_{\tilde{p}}\modsp{\kappa}, \tilde{h}) \to M
    \end{equation*}
    and the map
    \begin{equation*}
    	\tilde{\Phi} = \exp_p\circ I\circ \exp_{\tilde{p}}^{-1}: \modsp{\kappa} \to M.
    \end{equation*}
    It is easy to see that 
    \begin{equation*}
    	\text{$\Phi$ is $1$-Lipschitz} \iff \text{$\tilde{\Phi}$ is $1$-Lipschitz}.
    \end{equation*}
    because by construction $\exp_{\tilde{p}}$ is an isometry from $\tilde{h}$ to $\tilde{g}$. Notice that it is easy to see $\tilde{\Phi}$ being $1$-Lipschitz is equivalent to the hinge comparison because for any geodesics $\gamma_1, \gamma_2$ starting at $p \in M$, their length, and the angle between them are preserved by $\tilde{\Phi}$. 
    
   Therefore, we only need to show $\Phi$ is $1$-Lipschitz. Here we only consider the case when $\kappa \equiv 0$ and in this case, $\modsp{\kappa} = \R^n$ and $\exp_{\tilde{p}} = \id$. And in this case, we just need to show $\exp_p: T_pM \to M$ is $1$-Lipschitz on a small neighborhood of $0$. We need to show that for $v \in B_\epsilon(0) \subseteq T_pM$ and $w \in T_v(T_pM) \simeq T_pM$, 
   \begin{equation*}
   	\abs{d(\exp_p)_v(w)} \leq \abs{w}
   \end{equation*}
   This is obvious for radial vectors, i.e. $w \parallel v$. Because the exponential map is an isometry, in this case, for general $w \in T_v(T_pM)$, we can decompose $w = w^\perp + w^\parallel$. where $w^\parallel = \lambda v$ for some $\lambda \in \R$ is pointing to the radius direction. Then we have
   \begin{equation*}
   	d(\exp_p)_v(w) = d(\exp_p)_v(w^\perp) + \underbrace{d(\exp_p)_v(w^\parallel)}_{\text{has the same length as $w^\parallel$}}.
   \end{equation*}
   Recall that by the Gauss lemma, a geodesic starting at $p$ is perpendicular to a small sphere $S_t(p)$ near $p$ for some small $t$. Then by the Gauss lemma, we notice that 
   \begin{equation*}
   	d(\exp_p)_v(w^\perp) \perp d(\exp_p)_v(w^\parallel)
   \end{equation*}
   Therefore, by the Pythagorean theorem, we have
   \begin{align*}
   	& \abs{w}^2 = \abs{w^\perp}^2 + \abs{w^\parallel} \\
   	\implies & \abs{d(\exp_p)_v(w)}^2 = \abs{d(\exp_p)_v(w^\perp)}^2 + \underbrace{\abs{d(\exp_p)_v(w^\parallel)}^2}_{= \abs{w^\parallel}^2}.
   \end{align*}
   It is sufficient to show 
   \begin{equation*}
   		\abs{d(\exp_p)_v(w^\perp)} \leq \abs{w^\perp}
   \end{equation*}
   Recall the explicit construction of the Jacobi field vanishing at the initial point (see theorem \ref{thm: Ja vanishes at a point}). For $J(0) = 0$ and $D_tJ(0) = w^\perp$ along $\gamma(t)$, we have 
   \begin{equation*}
   	(d\exp_p)_v(w^\perp) = J(1)
   \end{equation*}
   We denote $\bar{J}$ the corresponding Jacobi field in the model space. In our case, $\bar{J}(t) = t\cdot w$ thus $\bar{J}(1) = w$. By the Rauch $I$, we get that
   \begin{equation*}\abs{d(\exp_p)_v(w^\perp)} =  \abs{J(1)} \leq \abs{\bar{J}(1)} = \abs{w^\perp}
   \end{equation*}  
   \end{proof}
\begin{exercise}
	Show that a Riemannian manifold $(M, g)$ satisfies $\sect_M \geq \kappa$ if the local angle comparison holds. \textit{hint: This follows from the Taylor expansion formula in Meyer's Notes (see \cite{Meyer04})} 
\end{exercise}

What if we reverse the inequality in our angle comparison theorem? i.e. Is that correct that if $\sect_M \leq \kappa$ the angle comparison holds (At least locally) with appropriate inequality? The answer is YES! However, our proof DOES NOT work in this setting. The recall what we did for lower curvature bound

Consider $D_tS + S^2 + R_{\nu} = 0$ along $\gamma$ where $S$ is the shape operator of the equidistant hypersurface and $\nu$ the unit normal vector field of the hypersurface. Let $Y$ be unit parallel vector fields along $\gamma$. If $R_{\nu} \geq \kappa \id$, then $D_tS + S^2 + \kappa\id \leq 0$. Look at $a(t) = \inp{SY, Y}$. Then 
\begin{align*}
	a'(t) 
	&= \inp{D_tS Y, Y} = \inp{(-S^2 - R_{\nu})Y, Y}\\
	&\leq -\inp{SY, Y}^2  - \inp{R_{\nu}(Y), Y}
\end{align*}
by the Cauchy-Schwartz inequality, i.e $\inp{S^2 Y, Y} = \inp{SY, SY} \geq \inp{SY, Y}^2$. Also, because $\inp{R_{\nu}(Y) , Y} \leq \kappa\abs{Y}^2$, we can conclude that
\begin{equation*}
	a' + a^2 + \kappa \leq 0
\end{equation*}
This allows us to reduce the study of the matrix Riccati  $D_tS + S^2 + \kappa\id \leq 0$ to the scalar Riccati inequality $a' + a^2 + \kappa \leq 0$.
However, this trick does not work for $\sect_M \leq \kappa$. In that case, we need to deal with the Ricatti inequality  $D_tS + S^2 + \kappa\id \geq 0$. Here the above trick using the Cauchy-Schwartz inequality doesn't work and one has to study the matrix Riccati inequality directly.
This can be done with the help of the following result. We omit the proof as in this notes we are only interested in applications to lower curvature bounds.
\begin{theorem}[See \cite{EH90}]
	Let $R_1, R_2: \R \to S(E)$ be smooth with $R_1 \geq R_2$. For $i = \br{1, 2}$ let $A_i: [t_0, t_1) \to S(E)$ be a solution of 
	\begin{equation*}
		D_tA_i + A_i^2 + R_i = 0
	\end{equation*}
	with maximal $t_i \in (t_0, \infty]$. Assume that $A_1(t_0) \leq A_2(t_0)$. Then $t_1 \leq t_2$ and $A_1(t) \leq A_2(t)$ on $(t_0, t_1)$
\end{theorem}
Then the angle comparison for the upper curvature bound can be explained by the following general Rauch comparison theorem.
\begin{theorem}[General Rauch $I$ and Rauch $II$ Comparison]
	Let $(M_1, g_1)$ and $(M_2, g_2)$ be two Riemannian manifold along their geodesics $\gamma_1: [0,T]\to M_1$ and $\gamma_2: [0,T]\to  M_2$. Assume that 
	\begin{equation*}
		\sect(\dga_1, Y_1) \geq \sect(\dga_2, Y_2)
	\end{equation*}
	for any normal unit vector fields $Y_1, Y_2$ along $\gamma_1, \gamma_2$. Let $J_i$ be Jacobi fields along $\gamma_i$ for $i = 1, 2$. Suppose either
	\begin{itemize}
		\item \textbf{Rauch $I$:}  Assume that $\gamma_1$ has no conjugate points on $[0, T)$,  $J_i(0) = 0$ and $\abs{D_tJ_1(0)} = \abs{D_tJ_2(0)}$; or
		\item \textbf{Rauch $I$:} Assume that $\gamma_1$ has no focal points on $[0,T)$, $D_tJ_1(0) = D_tJ_2(0) = 0$ and $\abs{J_1(0)} = \abs{J_2(0)}$.
	\end{itemize}
	Then $\frac{\abs{J_1}}{\abs{J_2}}$ is non-increasing up to the first zero of $J_1$. 
\end{theorem}


\section{Other Local Comparison Theorems}
	We studied the two versions of local Toponogov's comparison theorem (local angle comparison and local hinge comparison) holding in a small ball $B_\eps(p) \subseteq M$ where $\sect_M \geq \kappa$. In this lecture, we are going to introduce a few more comparison theorems that are equivalent to the angle comparison and the local hinge comparison theorem. 
\subsection{Monotonicity of Angle Comparison}
\begin{theorem}[Local Monotonicity of Angle Comparison]\label{thm: monotonicity comparison}
    Let $(M^n, g)$ be a Riemannian bounded from below.  
    For any $z \in M$  there is a small $0<\eps<\pik/2$ such that the following holds.
    Let $p,x,y\in B_\eps(z)$.
    Let $\gamma: [0,|px|]\to M$, $\sigma: [0,|py|]\to M$ be unit speed parameterizations of $[px], [p,y]$ respectively. Let $d_{s, t};=d(\gamma(t,\sigma(s))$.
   Then the model angle $\alpha_{s, t} = \modangle^\kappa(d_{s, t}; s, t)$ is monotonically non-increasing with respect to both $t$ and $s$. 
\end{theorem}
\begin{figure}[htbp]
    \centering
        \includegraphics[width=0.6\textwidth]{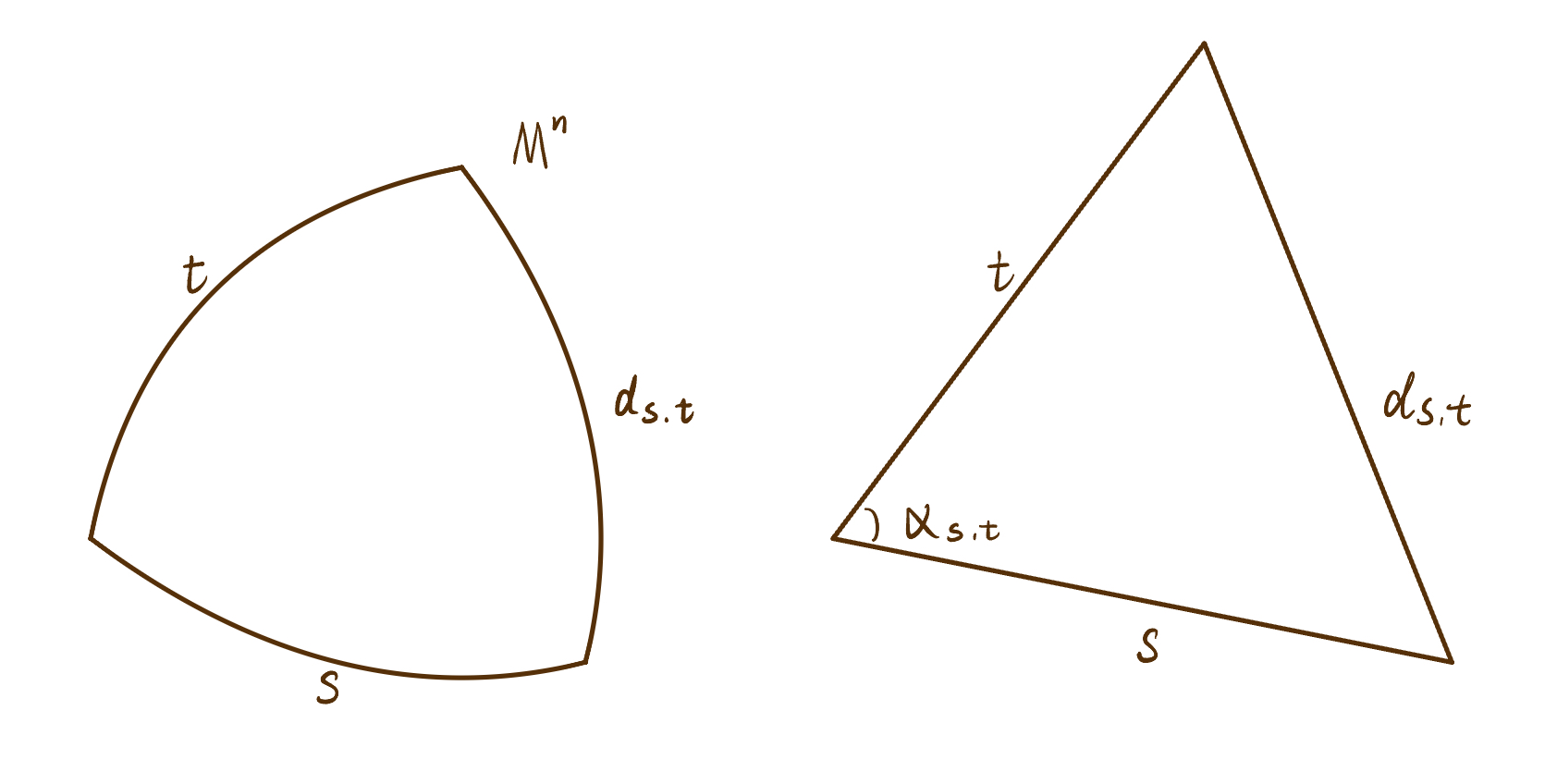}\caption{Monotonicity of angle comparison}
    \end{figure}

In this section, we mainly want to show the local monotonicity of angle comparison is equivalent to the local angle comparison. The key lemma used in proving the local angle comparison implies the local monotonicity of angle comparison is  Alexandrov's lemma. 
\begin{lemma}[Alexandrov's Lemma]\label{lem: Alexandrov's lemma}
In $\mathbb{M}_\kappa^n$, let $x, y, z, d > 0$ and consider the following picture where the triangles $\modtriangle^\kappa\br{x, y, z}$ and $\modtriangle^\kappa\br{y, z, d}$ is uniquely defined. We also marked the angles $\alpha, \beta, \gamma_1$ in the picture. Assume 
\begin{itemize}
	\item $\alpha + \beta \leq \pi$; 
	\item And if $\kappa > 0$, assume $x + y + z + d < 2\varpi^\kappa$. \footnote{This assumption is necessary since we want the triangle $\modtriangle^\kappa\br{x + y, z, d}$ to be uniquely defined up to isometry.}
\end{itemize}
Now consider the triangle $\modtriangle^\kappa\br{x + y, z, d}$ and the angle $\gamma_2$ in the following picture, we claim that $\gamma_2 \leq \gamma_1$. 
\end{lemma}
\begin{figure}[htbp]
    \centering
        \includegraphics[width=0.6\textwidth]{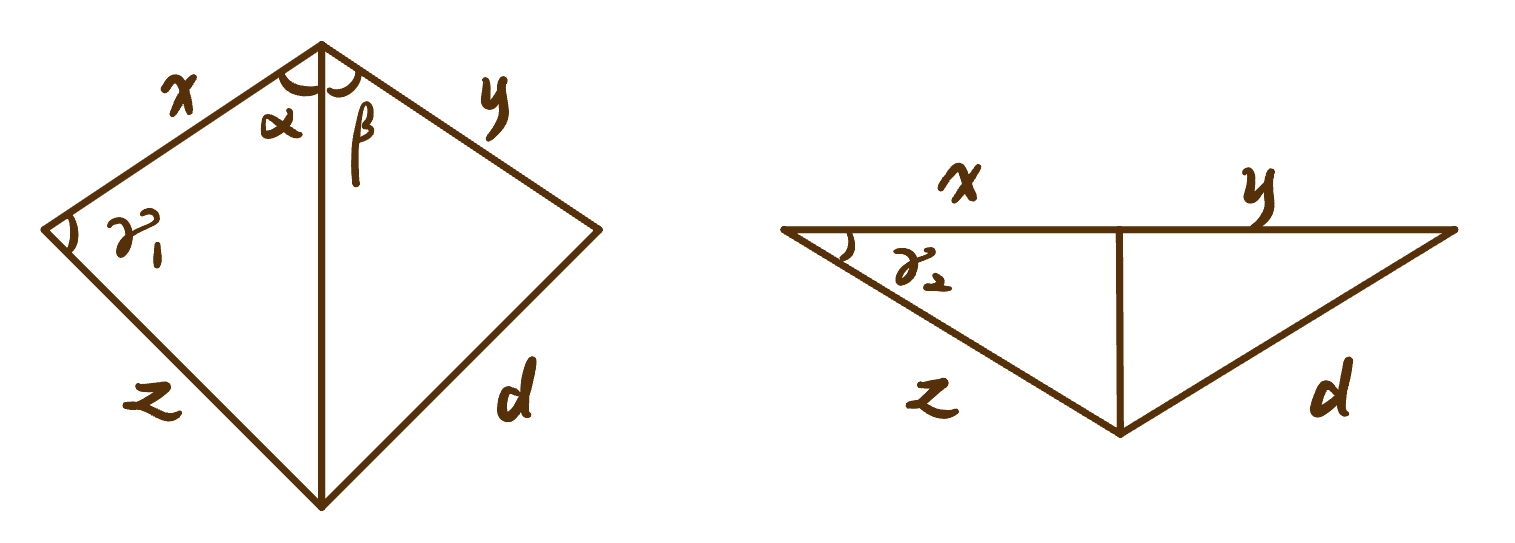}\caption{The Alexandrov's lemma \ref{lem: Alexandrov's lemma}}
    \end{figure}
\begin{proof}
	We extend the side of length $x$ by length $y$ to point   $q$. We also marked the point $p$ in the picture below. 
	
	\begin{figure}[htbp]
    \centering
        \includegraphics[width=0.5\textwidth]{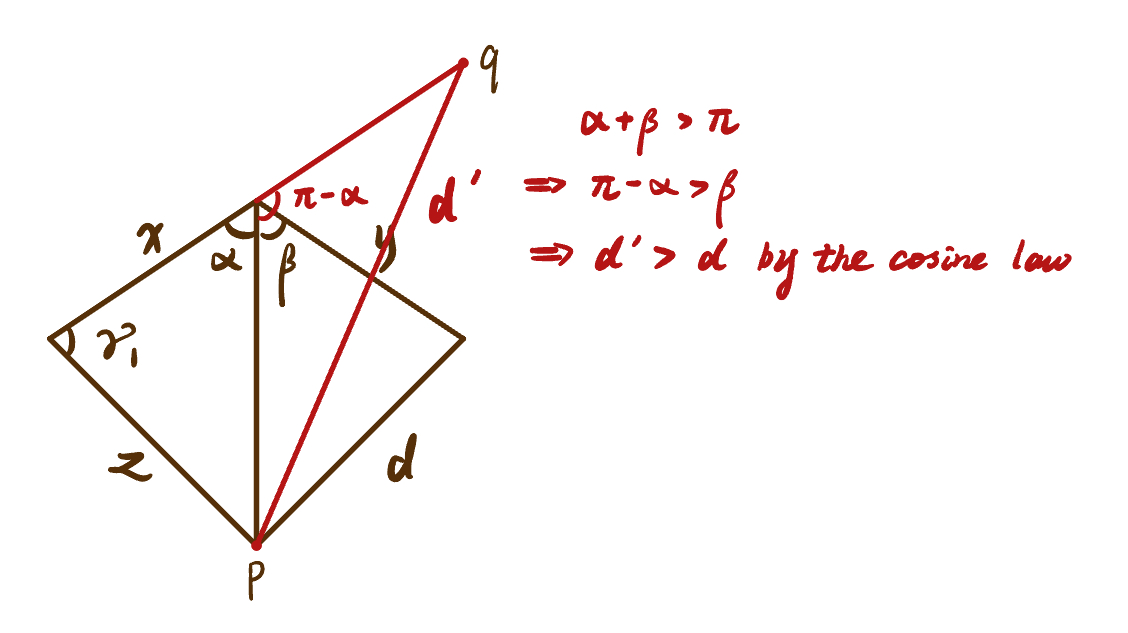}
    \end{figure}
    
	Denote $d' = \abs{p - q}$. Now it is easy to see that $d' \geq d$ by the cosine law. 
		
	\begin{figure}[htbp]
    \centering
        \includegraphics[width=0.25\textwidth]{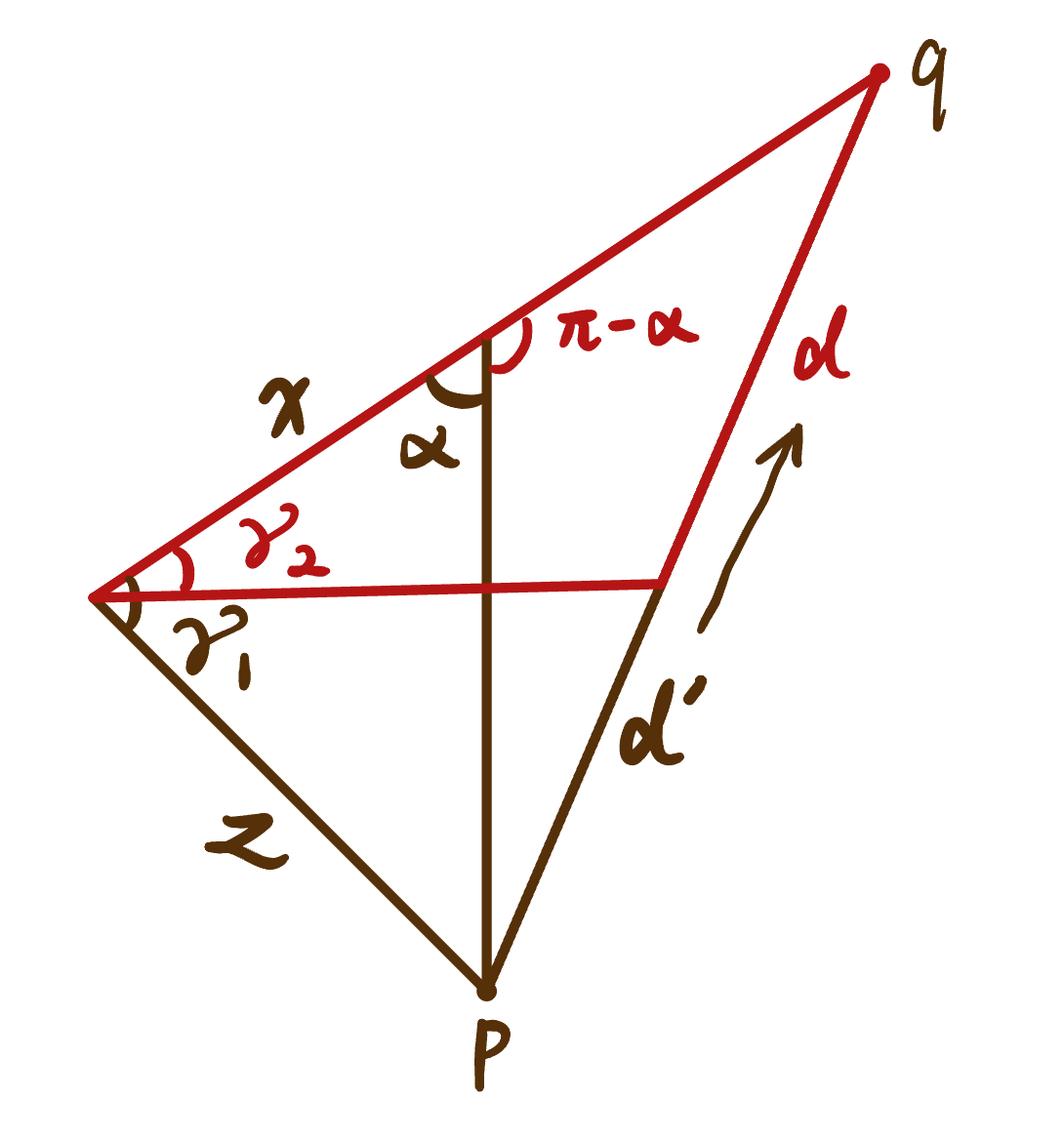}
    \end{figure}
    
	Next, still by the cosine law, by reducing $d'$ to $d$ in the picture above, we got $\gamma_1 \geq \gamma_2$. 
\end{proof}
\begin{theorem}
	The local monotonicity comparison theorem \ref{thm: monotonicity comparison} is equivalent to the local angle comparison theorem \ref{thm: angle}. 
\end{theorem}
\begin{proof}
	Suppose the angle $\alpha_{s, t}$ is monotone with respect to $s$ and $t$, then (see \cite{AKP22})
	\begin{equation}\label{eq: limit model angle is actual angle}
		\lim_{t \to 0^+}\alpha_{t, t} = \alpha
	\end{equation}
	since $M^n$ is infinitesimally close to $\R^n$, and we can show this using Taylor expansion formula for $d_{s, t}$ in Meyer's notes \cite{Meyer04}. 
	
	Therefore $\alpha\ge \alpha_{s, t}$ for any $s, t>0$. 
	
	Conversely, suppose the local angle comparison theorem holds. Consider the following picture, 
	\begin{figure}[htbp]
    \centering
        \includegraphics[width=0.6\textwidth]{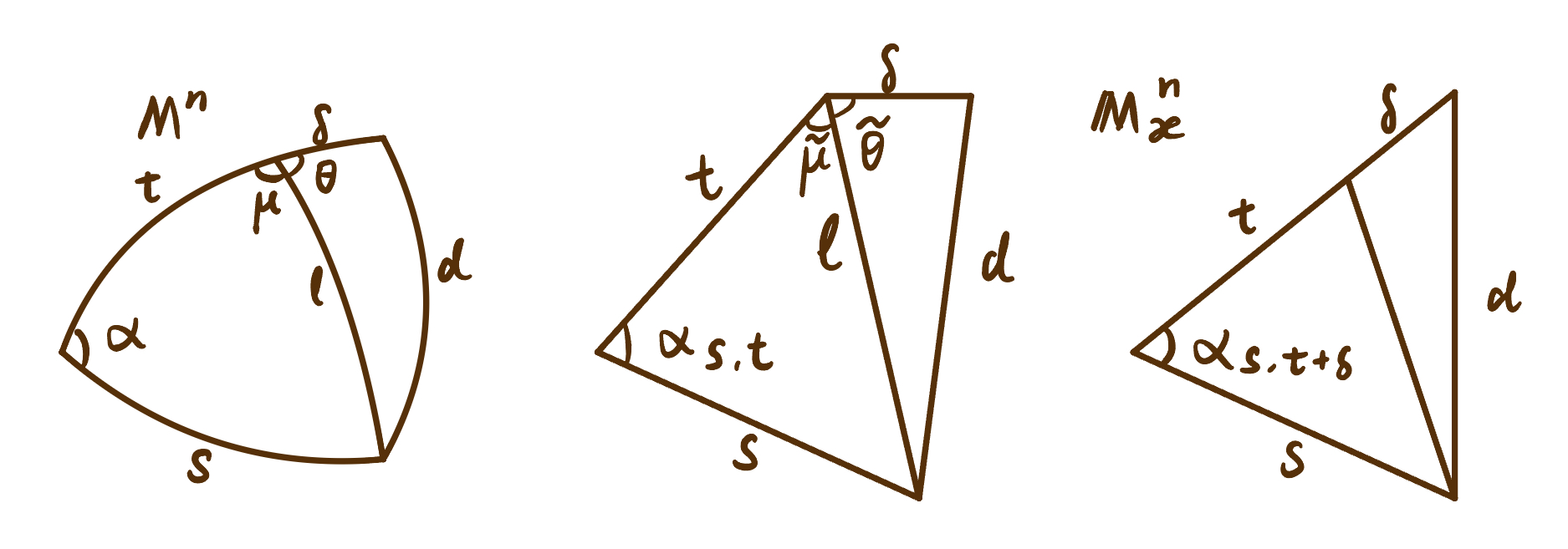}
    \end{figure}
    
	we want to show that $\alpha_{s, t} \geq \alpha_{s, t + \delta}$ for $\delta > 0$. By the local angle comparison \ref{thm: angle} we have $\tilde{\mu} \leq \mu$ and $\tilde{\theta} \leq \theta$ so that $\tilde{\mu} + \tilde{\theta} \leq \mu + \theta = \pi$. Next, by the Alexandrov lemma \ref{lem: Alexandrov's lemma}, we can conclude that $\alpha_{s, t} \geq \alpha_{s, t + \delta}$.
\end{proof}
\subsection{Point-on-a-side Comparison}
\begin{theorem}[Point-on-a-side Comparison]\label{thm: point-on-a-side}
Let $(M^n, g)$ be a Riemannian manifold with $\sect_M \geq \kappa$. Let $t, s, x, y > 0$ be positive numbers such that $t + s, x, y$ satisfies the triangle inequalities such that there exists a triangle in the $B_\eps(p)$ for some small $\eps > 0$ and $p \in M$ (See Figure~\ref{fig: point-on-a-side}). We claim that $d \geq \tilde{d}$, where $d$ is the length in $M^n$ and $\tilde{d}$ is the length in $\modsp{\kappa}$.
\end{theorem}
\begin{figure}[htbp]
    \centering
        \includegraphics[width=0.5\textwidth]{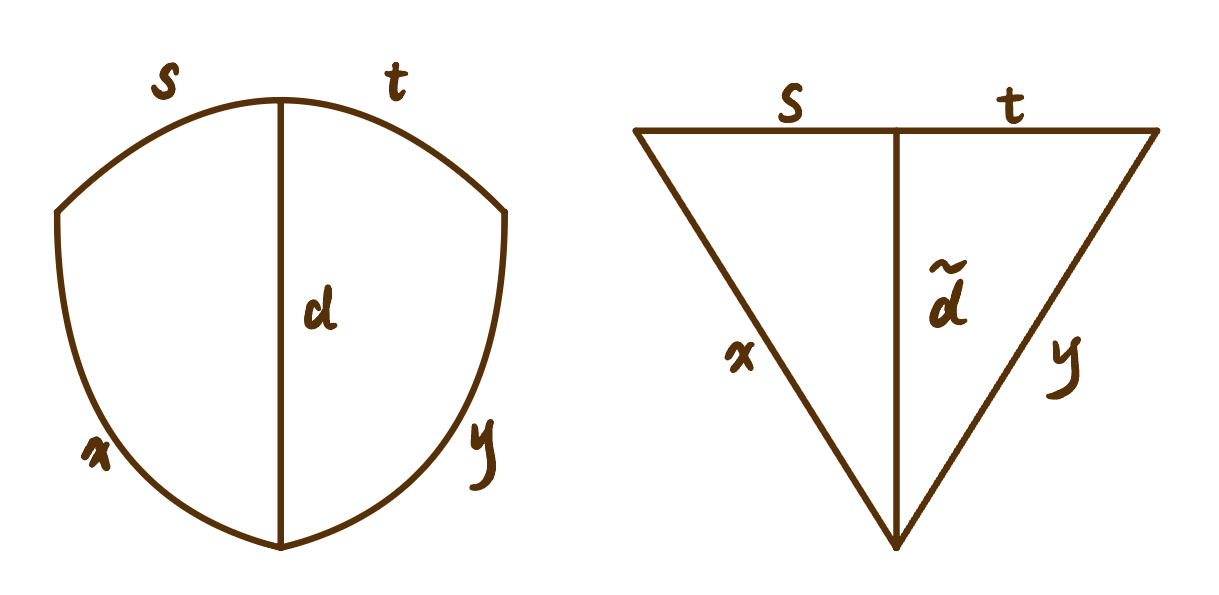}\caption{Point-on-a-side Comparison [Theorem \ref{thm: point-on-a-side}]}
        \label{fig: point-on-a-side}
    \end{figure}
\begin{theorem}
	The point-on-a-side comparison is equivalent to the monotonicity of angle comparison. 
\end{theorem}
\begin{proof}
	Suppose the monotonicity of angle comparison holds, we can extend side $t$ as in the following picture.
	
	\begin{figure}[htbp]
    \centering
        \includegraphics[width=0.5\textwidth]{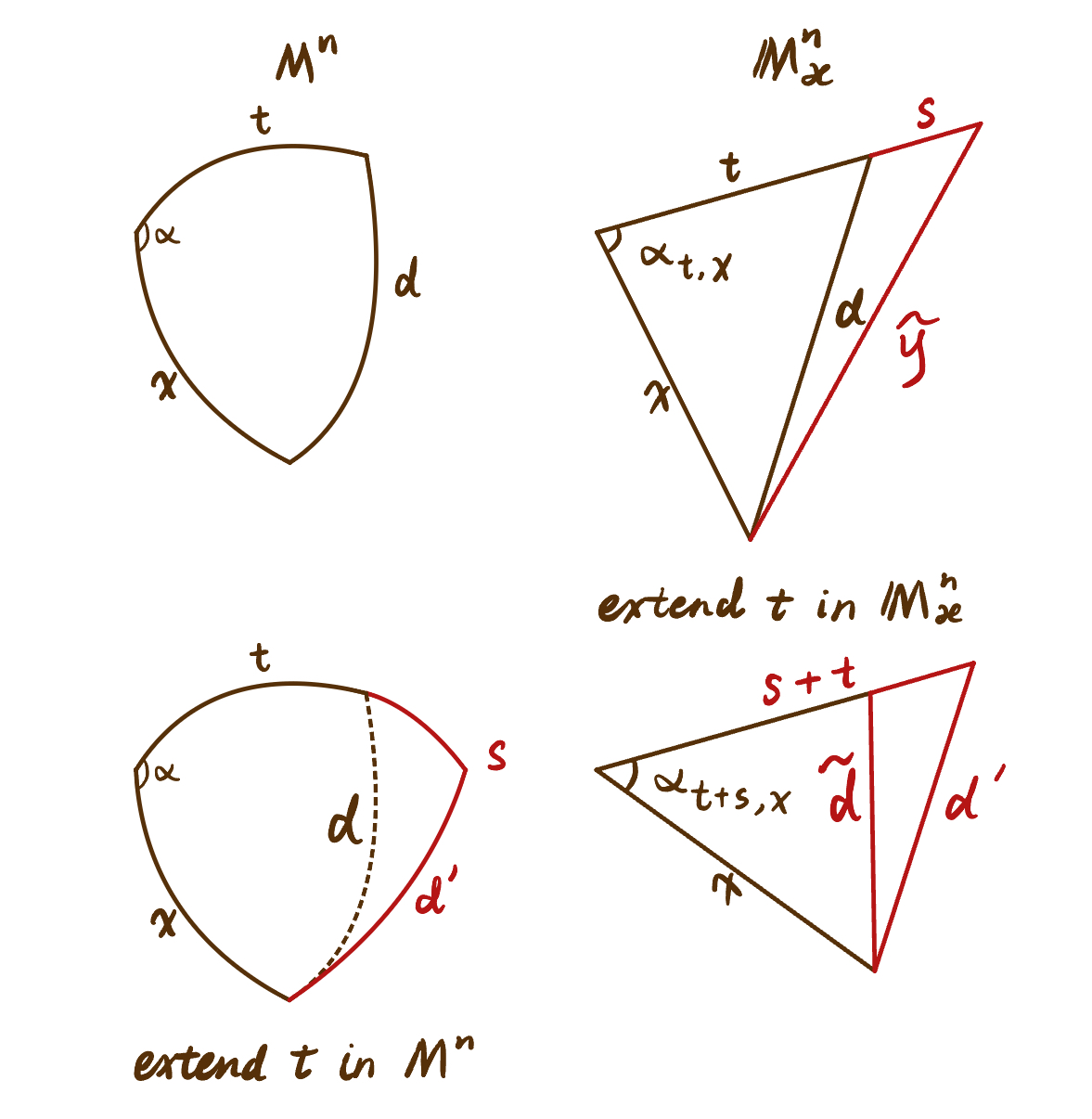}
    \end{figure}
    
    Draw the triangle $\modtriangle^\kappa\br{t, x, d}$ and extend the $t$ by $s$ in $\modsp{\kappa}$ so that we can obtain a new triangle $\modtriangle^\kappa\br{x, t + s, \tilde{y}}$. We denote $\alpha_{t, x} = \modangle^\kappa\br{d; t, x}$. On the other hand, we extend the side $t$ of the triangle $\triangle(x, t, d)$ by $s$ directly in $M$ so that we can have the triangle $\triangle(x, t + s, d')$. Compare $\modtriangle^\kappa\br{x, t + s, \tilde{y}}$ with $\modtriangle^\kappa\br{x, t + s, d'}$, by the monotonicity of angle comparison, we have 
    \begin{equation*}
    	\alpha_{t, x} \geq \alpha_{s + t, x} := \modangle^\kappa\br{d'; x, s + t}.
    \end{equation*}
    By the cosine law, this implies that $\tilde{y} \geq d'$. Sending $s \to 0$, so that $\tilde{y} \to d$ and $d' \to \tilde{d}$, we can conclude that $d \geq \tilde{d}$ which is the point-on-a-side comparison. The converse direction is similar.  
\end{proof}
\subsection{Four-points Comparison}
\begin{theorem}[Four-points Comparison]\label{thm: four-point}
    Let $(M^n, g)$ be a Riemannian manifold such that $\sect_M \geq \kappa$. Fix $q \in M$ and consider quadruple of points $P, A, B, C \in B_\eps(q)$ for some small $\eps$. We draw the following picture.

\begin{figure}[htbp]
    \centering
        \includegraphics[width=0.7\textwidth]{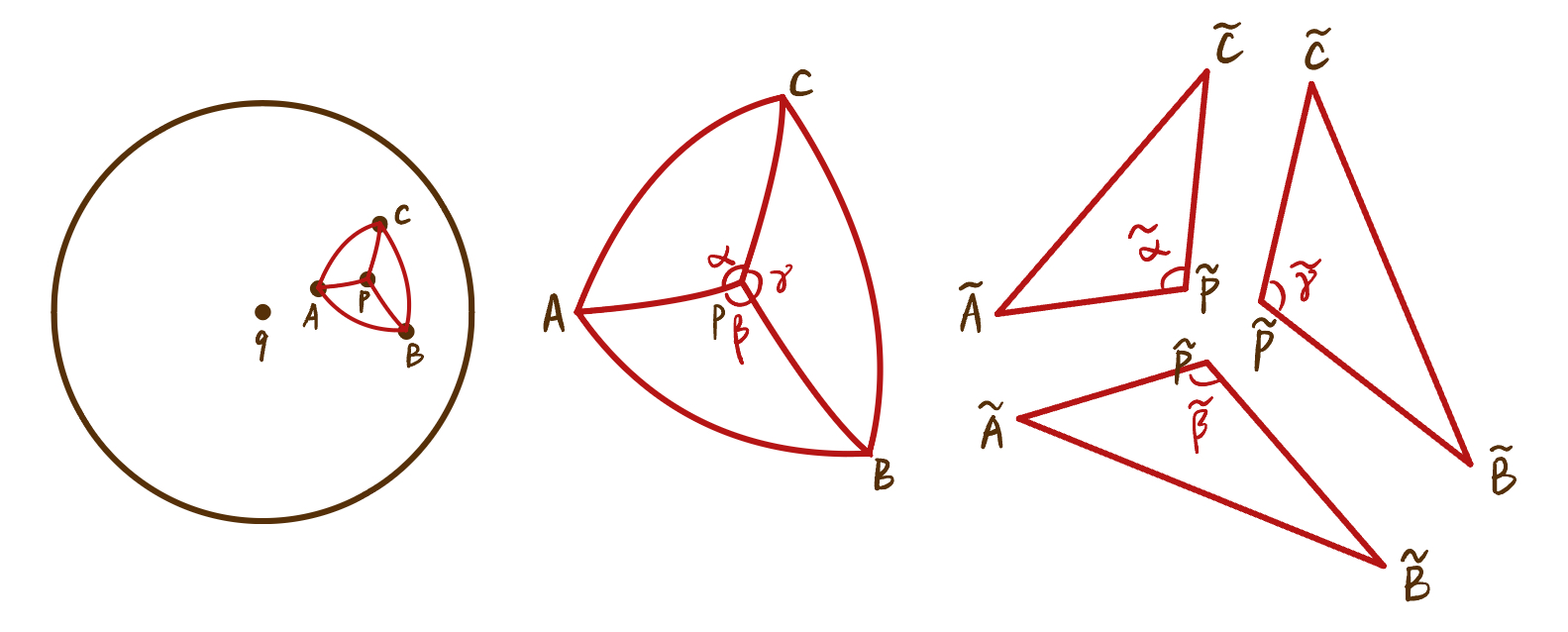}\caption{Four-Point Comparison \ref{thm: four-point}}
\end{figure}
And we assume all comparison triangles exist and are unique. This is automatic if $\eps$ is sufficiently small. Then for $\tilde{\alpha}, \tilde{\beta}, \tilde{\gamma}$ in their corresponding comparison triangle, i.e. 
	\begin{equation*}
		[\tilde{A}\tilde{C}\tilde{P}] = \modtriangle^\kappa\br{\abs{A - C}, \abs{C - P}, \abs{A - P}}
	\end{equation*}
Then, 
    \begin{equation*}
    	\tilde{\alpha} + \tilde{\beta} + \tilde{\gamma} \leq 2\pi
    \end{equation*}
\end{theorem}
\begin{theorem}
	We want to show that the four-points comparison \ref{thm: four-point} is equivalent to the angle comparison. 
\end{theorem}
\begin{proof}
	We take $v_1, v_2, v_3$ three initial vectors of $[PA], [PB], [PC]$ starting at $P$. They belong to the unit sphere $S^{n - 1} \subseteq T_PM$. Because $S^{n - 1} = \mathbb{M}_\kappa^{n - 1}$, then we have 
	\begin{equation*}
		\alpha + \beta + \gamma \leq 2\varpi^1 = 2\pi
	\end{equation*}
	
	\begin{figure}[htbp]
    \centering
        \includegraphics[width=0.5\textwidth]{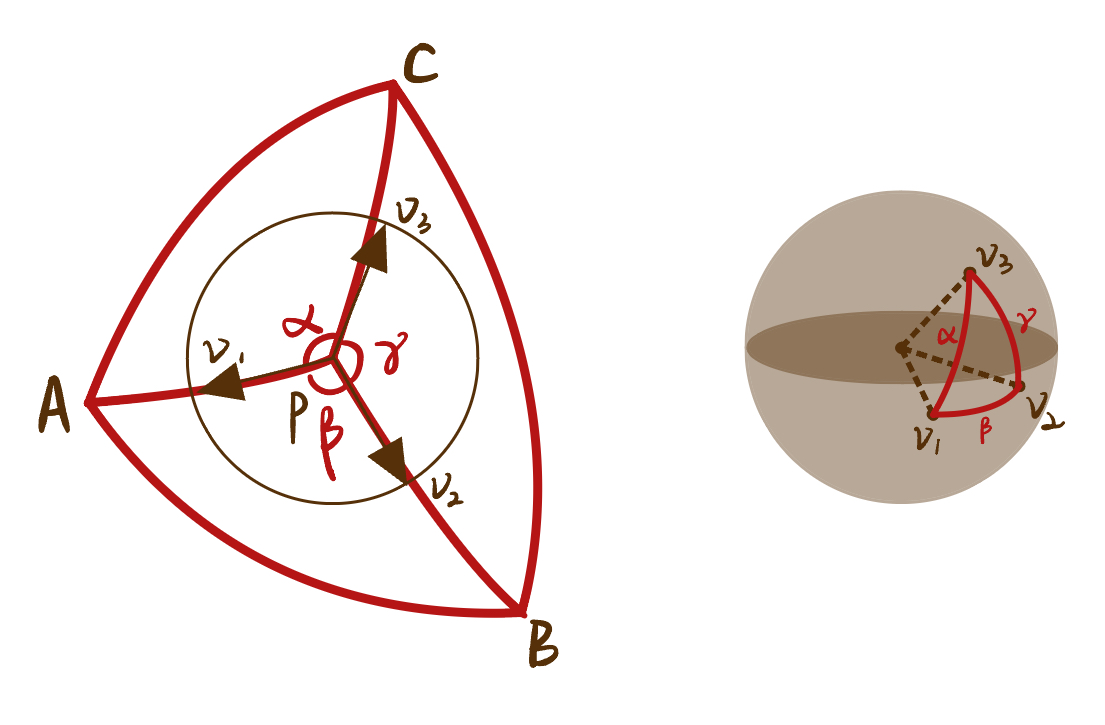}
    \end{figure}
    Then by the angle comparison, since $\alpha \geq \tilde{\alpha}, \beta \geq \tilde{\beta}, \gamma \geq \tilde{\gamma}$, we derive that 
    \begin{equation*}
    	\tilde{\alpha} + \tilde{\beta} + \tilde{\gamma} \leq 2\pi
    \end{equation*}
  
  	Conversely, we instead prove the four-point comparison implies the monotonicity of the comparison angle. We draw the following picture, $P \in [AC]$ so that $\tilde{P}, \tilde{A}, \tilde{C}$ is also in a straight line, thus $\gamma = \tilde{\gamma} = \pi$. 
  	
  	\begin{figure}[htbp]
    \centering
        \includegraphics[width=0.7\textwidth]{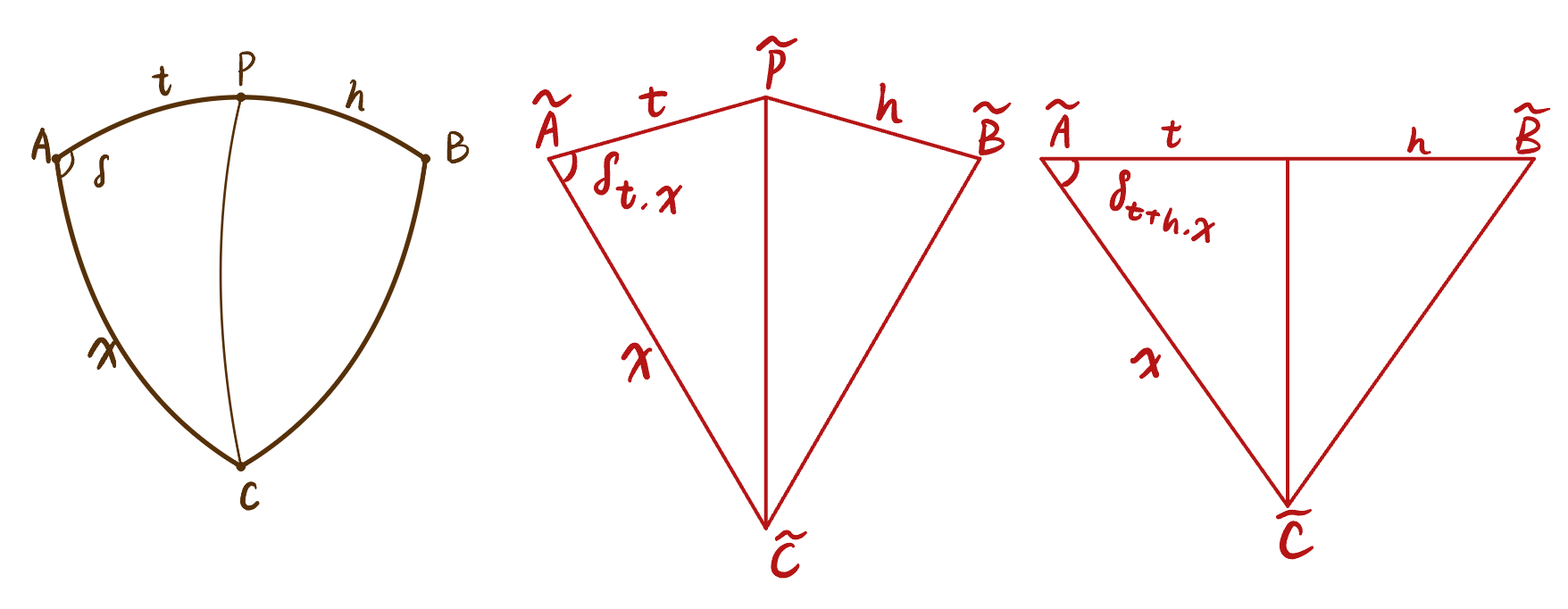}
    \end{figure}
  	Therefore, by the four-points comparison, $\tilde{\alpha} + \tilde{\beta} \leq \pi$. Then by Alexandrov's lemma, we can conclude that the comparison angle satisfies $\delta_{t, x} \geq \delta_{t + h, x}$. 
\end{proof}
\begin{remark}
	The four-points comparison involves only distances in $M$.  This allows it to be used to give a metric definition of lower sectional curvature bound. It is also clear that it behaves well with respect to most notions of convergence of metric spaces. We will see later that it is stable under Gromov-Hausdorff convergence. 	
	\end{remark}
\begin{remark}
In this section, we have established the equivalence of various local versions of angle comparison including the four-points comparison. Equivalence of global versions of these comparisons also holds but for $\kappa>0$ some implications require the globalization theorem. For $\kappa\le 0$  the local proofs directly generalize to global ones.
\end{remark}

\subsection{Point-on-a-side Comparison by Jensen's Inequality}\label{sect: function comparison}
In this section, we want to study Jensen's inequality of $f = \mdk\circ d_p$ and give a new proof of local point-on-a-side comparison~\ref{thm: point-on-a-side}. 
Remember we already proved that if $f = \mdk\circ d_p$ in $(M, g)$ where $\sect_M \geq \kappa$. Then near $p$ (or more generally, outside of cut locus of $p$)  $f$ satisfies 
\begin{equation*}
	\hess_f + \kappa f\id \leq \id
\end{equation*}
as a matrix. Equivalently, for any unit speed geodesics $\gamma(t)$ in a small ball around $p$ (or more generally, outside of cut locus of $p$) it holds that  
\begin{equation}\label{eq: functional comparison ineq}
	f(\gamma(t))'' + \kappa f(\gamma(t)) \leq 1
\end{equation}
{In the model space the inequalities above are equalities, in particular $ f(\gamma(t))'' + \kappa f(\gamma(t)) = 1$.}

\begin{example}
For example, when $\kappa = 0$, if $\sect_M \geq 0$, then $\mdk(t) = \frac{t^2}{2}$ satisfies \ref{eq: equation system for z}, therefore $f = \frac{d_p^2}{2}$. Then along any geodesic $\gamma$, we have $f(\gamma(t))'' \leq 1$. And for $p = 0 \in \R^n$, $x \in \R^n$, we have
	\begin{align*}
		& d_p(x) = \sqrt{\sum_i x_i^2}\\
		\implies & f(x) = \frac{\sum_i x_i^2}{2}\\
		\implies & \hess_f \equiv \id\\
		\implies & f(\gamma(t))'' = 1
	\end{align*}
\end{example}
In this section, we are going to justify that the above result of the modified distance function gives a different proof of local  point-on-a-side comparison \ref{thm: point-on-a-side}. To show this we need to understand the inequality \ref{eq: functional comparison ineq} geometrically. Let $f: \R \to \R$ be a smooth function, fix $\kappa, \lambda \in \R$, and suppose it satisfies
\begin{equation*}
	f'' + \kappa f \leq \lambda
\end{equation*}
We want to understand this condition geometrically.
Firstly, we consider the simple case when $\kappa = 0$, then $f'' \le  \lambda$ is easy to understand:
\begin{align*}
	& f'' \leq \lambda\\
	\iff & (f(t) - \frac{\lambda t^2}{2})'' \leq 0\\
	\iff & \text{$f(t) - \frac{\lambda t^2}{2}$ is concave.}
\end{align*}
Denote $g(t) = f(t) - \frac{\lambda t^2}{2}$, since $g$ is concave, then for $t \in [0, 1]$, by definition, we have
\begin{equation*}
	g((1 - t)x + t y) \geq (1 - t)g(x) + t g(y)
\end{equation*}
In terms of $f$, this is 
\begin{equation}\label{eq: concavity ineq}
	f((1 - t)x + t y) \geq (1 - t)f(x) + t f(y) - \frac{\lambda}{2}t(1 - t)(x - y)^2
\end{equation}
the concavity inequality of $f$. Notice that the concavity inequality \ref{eq: concavity ineq} is equivalent to Jensen's inequality, which says that if $f$ is a function that satisfies $f'' \leq \lambda$, then for any function $\bar{f}$ satisfying $\bar{f}(x) = f(x)$ and $\bar{f}(y) = f(y)$ and $\bar{f}'' = \lambda$, we have 
\begin{equation*}
	f((1 - t)x + ty) \geq \bar{f}((1 - t)x + ty) \quad \forall t \in [0, 1]
\end{equation*} 
That means the graph of the comparison function $\bar{f}$ is below the function $f$. 
\begin{figure}[htbp]
    \centering
        \includegraphics[width=0.3\textwidth]{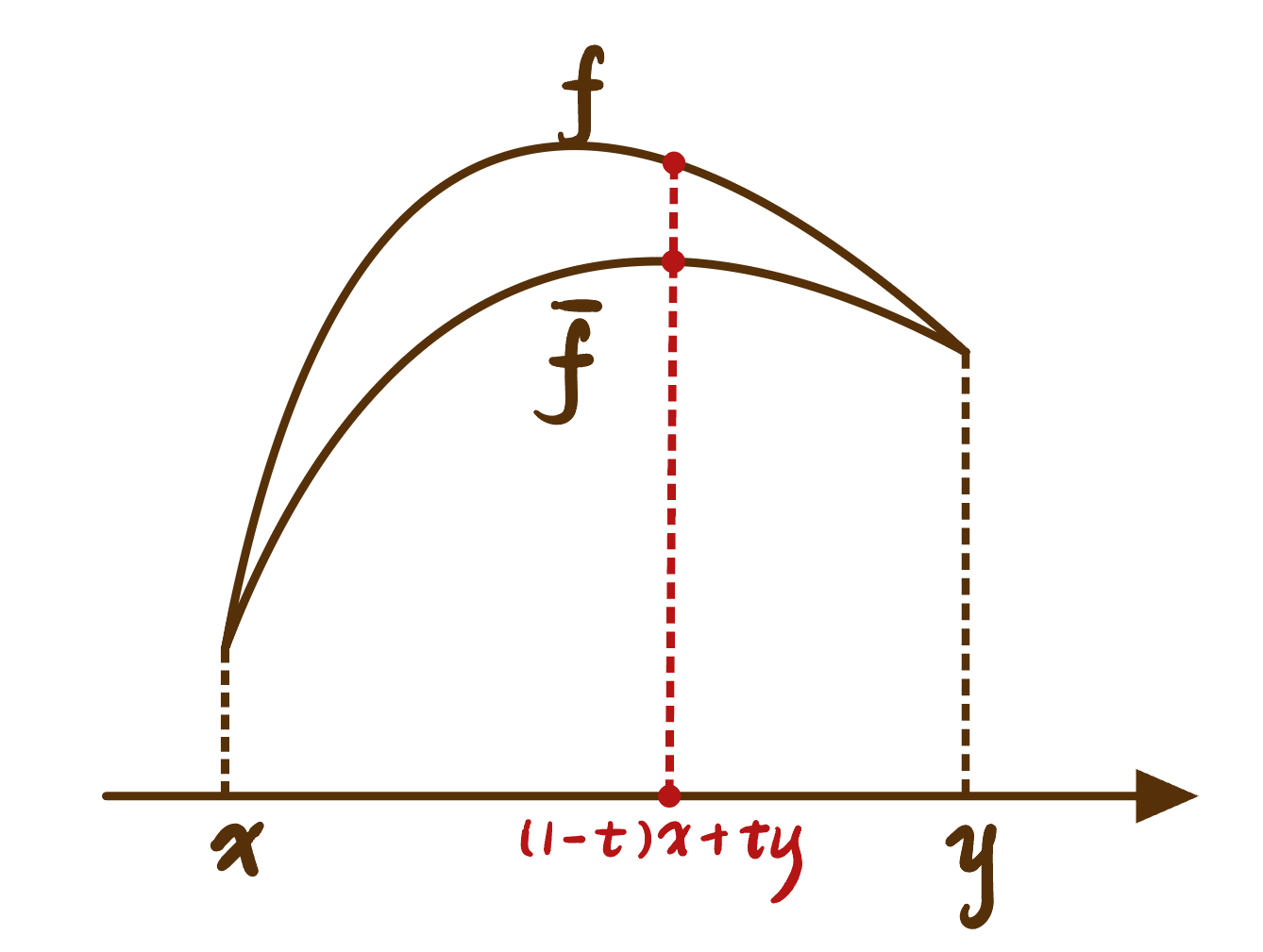}\caption{Jensen's inequality}
    \end{figure}
\begin{remark}
	The $\bar{f}(t)$ above can be $\frac{\lambda t^2}{2}$ plus some affine term $at + b$ for some constant $a, b \in \R$. 
\end{remark}
Now let's consider the case when $\kappa \neq 0$ so that 
\begin{equation}\label{eq: kappa neq 0}
	f'' + \kappa f \leq \lambda
\end{equation}
Then we want to derive a similar Jensen's inequality as above, 
\begin{lemma}[Jensen's inequality]\label{thm: Jensen's inequality}
	Suppose inequality \ref{eq: kappa neq 0} holds for some $\kappa \neq 0$.  Let $x, y \in \R$ such that $\abs{x - y} < \varpi^\kappa$  (note that this is only a restriction if $\kappa > 0$). Fix $\bar{f}$ satisfying $\bar{f}(x) = f(x)$ and $\bar{f}(y) = f(y)$ and $\bar{f}'' + \kappa \bar{f} = \lambda$. Then $f \ge \bar{f}$ on $[x, y]$, i.e. for any $t \in [0, 1]$, we have the Jensen's inequality
	\begin{equation*}
		f((1 - t)x + ty) \geq \bar{f}((1 - t)x + ty).
	\end{equation*}
\end{lemma}
\begin{proof}
Let us first prove the lemma under a slightly stronger assumption $f'' + \kappa f < \lambda$.

Suppose the conclusion of the lemma is false. Namely, there exists $t_0 \in (x, y)$ such that 
	\begin{equation}\label{eq: f(t_0)}
		 f(t_0) < \bar{f}(t_0)
	\end{equation}
	Because $f'' + \kappa f < \lambda$ and $\bar{f}'' + \kappa \bar{f} = \lambda$,  
	\begin{equation*}
		(f - \bar{f})'' + \kappa(f - \bar{f}) \leq 0. 
	\end{equation*}
	Denote $u = f - \bar{f}$. We have $u(x) = u(y) = 0$ and $u'' + \kappa u < 0$. Therefore, our assumption \ref{eq: f(t_0)} in terms of $u$ becomes 
	\begin{equation*}
		\min_{(x, y)} u < 0
	\end{equation*}
	
	Suppose $\kappa <0$.

	Let $t_{\min} \in [x, y]$ be value minimizing $u$, then $u'(t_{\min}) = 0$ and $u''(t_{\min}) \geq 0$. Notice that since $u(x) = u(y) = 0$, then $t_{\min} \in (x, y)$. When $\kappa < 0$, 
	\begin{equation*}
		u''(t_{\min}) + \kappa u(t_{\min}) \leq 0 \implies u''(t_{\min}) \leq -u(t_{\min})\kappa < 0, 
	\end{equation*}
	which contradicts to  $u''(t_{\min}) \geq 0$. 
	
	Let us now consider the case $\kappa>0$. By rescaling it is enough to consider the case $\kappa=1$.
	
	Thus we are assuming that $f$ satisfies $f''+f<\lambda$ and  $\abs{x - y} < \pi$.  We will use the following \textbf{$\sin$-trick}: Since $|x-y|<\pi$ we can find an $\eps \geq 0$ such that $v(t) = \sin(t - \eps) > 0$ on $[x, y]$ (translation if necessary). Note that $v$ solves $v'' + v = 0$ and $v > 0$ on $[x, y]$. Still, we denote $u = f - \bar{f}$.
	Then $u''+u<0$ on $[x, y]$.

	 Denote $\phi = \frac{u}{v}$ on $[x, y]$ then there exists the minimizing point $t_{\min} \in (x, y)$ such that 
	\begin{equation*}
		\phi(t_{\min}) = \min_{[x, y]}\phi < 0.
	\end{equation*}
	Then by the second order derivative test, we have $\phi'(t_{\min}) = 0$ and $\phi''(t_{\min}) \geq 0$. 
	Notice that 
	\begin{equation*}
		\brac{\frac{u}{v}}' = \frac{u'v - uv'}{v^2} \implies (u'v - uv')(t_{\min}) = 0. 
	\end{equation*} 
	Then, since
	\begin{align*}
		\brac{\frac{u}{v}}'' 
		& = \brac{\frac{u'v - uv'}{v^2}}'\\
		& = \frac{(u'v - uv')'v^2 - (u'v - uv')(v^2)'}{v^4}
	\end{align*}
	by substituting $(u'v - uv')(t_{\min}) = 0$ and using that $ v + v'' = 0$, we have
	\begin{align*}
		& \phi''(t_{\min}) = \frac{u''v - uv''}{v^2}(t_{\min}) = \frac{(u'' + u)}{v}(t_{\min}) \\
		\implies & u'' (t_{\min}) + u(t_{\min}) \geq 0
	\end{align*}
	This contradicts the assumption that $u'' + u < 0$.

	The case of weaker assumption $f'' + \kappa f \leq \lambda$ can be proved via approximation. We take $\lambda + \epsilon$ for arbitrary small $\eps > 0$, then 
	\begin{equation*}
		f'' + \kappa f < \lambda + \eps
	\end{equation*}
	Then we obtain Jensen's inequality 
	\begin{equation*}
		f((1 - t)x + ty) \geq \bar{f}((1 - t)x + ty).
	\end{equation*}
	for some $\bar{f}'' + \kappa \bar{f} = \lambda + \eps$. And since $\eps$ is arbitrary, we can prove Jensen's inequality for $\eps = 0$ as well. 
\end{proof}

We are going to show the above Jensen's inequality gives a different proof of the point-on-a-side comparison in $B_\eps(p)$. 
\begin{proof}
	This is a new proof to the point-on-a-side comparison \ref{thm: point-on-a-side}. Consider the following picture, fix $p \in M$ and $\bar{p} \in \modsp{\kappa}$ and $\gamma$ and $\bar{\gamma}$ the geodesics parameterize the opposite sides of $p$ and $\bar{p}$.

	we denote $f = \mdk\circ d_p$ and $\bar{f} = \mdk\circ d_{\bar{p}}$. Then $f'' + \kappa f \leq 1$ along $\gamma$ and $\bar{f}'' + \kappa \bar{f} = 1$ along $\bar{\gamma}$. Since $f(x) = \bar{f}(x)$ and $f(y) = \bar{f}(y)$, by Jensen's inequality \ref{thm: Jensen's inequality}, we can conclude that $f(z) \geq \bar{f}(\bar{z})$ for each $z$ on $[xy]$ and its corresponding point $\bar{z}$ in the model space. Then, since $\mdk$ is monotone increasing, we can conclude that $d_p(z) \geq d_{\bar{p}}(\bar{z})$. 
\end{proof}

\chapter{Global Toponogov Comparison Theorems}
In this lecture, we will prove the global Toponogov comparison theorem via the key lemma (Remember that we only proved the Toponogov comparison theorem on $B_\eps(p)$).
\section{The Key Lemma}
\begin{theorem}[The Key Lemma]\label{lem: The Key Lemma}
    Let $(M^n, g)$ be complete and $\sect_M \geq \kappa$, , $p \in M$ and $0<l\le  \varpi^\kappa$ such that $\forall q \in B_l(p)$ if the comparison
        \begin{equation*}
            \mangle[x_q^y] \geq \modangle^\kappa(x_q^y)
        \end{equation*}
        hold for any hinge $[x_q^y]$ satisfying
        \begin{equation}\label{eq: key assumption 1}
            \abs{x - p} + \abs{x - q} < \frac{2}{3}l
        \end{equation}
        Then the comparison 
        \begin{equation*}
            \mangle[x_p^q] \geq \modangle^\kappa(x_p^q)
        \end{equation*}
        hold for any hinge $[x_p^q]$ satisfying
        \begin{equation}\label{eq: key assumption 2}
            \abs{x - p} + \abs{x - q} < l
        \end{equation}
\end{theorem}
This key lemma allows us to extend the hinge comparison from  small hinges to bigger hinges. We want to prove this key lemma via contradiction. Firstly, we assume the local hinge comparison fails \ref{thm: hinge} for a hinge $[x_q^p]$ satisfying \ref{eq: key assumption 2}, i.e.
\begin{equation}\label{eq: key lemma contra assumption}
	\abs{x - p} + \abs{x - q} < l \quad\&\quad \abs{p - q} > \modcvee^\kappa[x_p^q]
\end{equation}
Notice that $\mangle[x_q^p] < \modangle^\kappa(x_q^p) \iff \abs{p - q} > \modcvee^\kappa[x_p^q]$. 

Our idea is to construct a sequence of smaller and smaller hinges satisfying \ref{eq: key lemma contra assumption}. And at the limits, the hinge admits the length in \ref{eq: key assumption 1}, which is impossible because we cannot have both $\abs{p - q} > \modcvee^\kappa[x_p^q]$ and $\abs{p - q} \leq \modcvee^\kappa[x_p^q]$. 

\subsection{Step 1: The Construction of a Smaller Hinge}
Equivalently, we can modify the assumption \ref{eq: key lemma contra assumption} a little bit. Suppose there exists a hinge $[x_q^p]$ such that 
\begin{equation}\label{eq: key lemma contra assumption}
	\frac{2}{3}l \leq \abs{x - p} + \abs{x - q} < l \quad\&\quad \abs{p - q} > \modcvee^\kappa[x_p^q]
\end{equation}
We want to show that we can find a smaller hinge such that the comparison still fails. WLOG Suppose $\abs{x - p} \leq \abs{x - q}$. Our first step is pick $x' \in [xq]$ such that 
    \begin{equation*}
        \abs{p-x} + 3\abs{x-x'} = \frac{2}{3}l
    \end{equation*}
    It is easy to see such $x'$ exists. Denote $\gamma$ the geodesic such that $\gamma(0) = x$ and $\gamma(1) = q$, we want to take $x' = \gamma(t)$ for some $t$. Look at 
    \begin{equation*}
    	f(t) = \abs{p - x} + 3\abs{x - \gamma(t)}
    \end{equation*}
    a continuous on $[0, 1]$. At $t = 0, 1$, we have 
    \begin{align*}
    	f(0) &= \abs{p - x} \leq \frac{l}{2} \leq \frac{2}{3}l\\
    	f(1) &= \abs{p - x} + 3\abs{x - q} \geq \frac{2}{3}l
    \end{align*}
    Therefore, by the intermediate value theorem, there exists $t$ such that $f(t) = \frac{2}{3}l$ and we denote $\gamma(t)$ as $x'$. Note that $x' \in B_l(p)\cap B_l(q)$. 
    
    Now we draw the comparison triangle $[\tilde{x}\tilde{x}'\tilde{p}] = \modtriangle^\kappa(\abs{x - p}, \abs{p - x'}, \abs{x' - x})$. Then we can extend $[\tilde{x}\tilde{x}']$ beyond $\tilde{x}'$ to $\tilde{q}$ such that $\abs{\tilde{x} - \tilde{q}} = \abs{x - q}$. denote 
    \begin{align*}
       & \alpha = \mangle[x_p^{x'}], \quad \Tilde{\alpha} = \modangle^\kappa(x_p^{x'})\\
       & \alpha' = \mangle[{x'}_p^{x}], \quad \Tilde{\alpha}' = \modangle^\kappa({x'}_p^{x})
    \end{align*}
    
    \begin{figure}[htbp]
    \centering
        \includegraphics[width=0.6\textwidth]{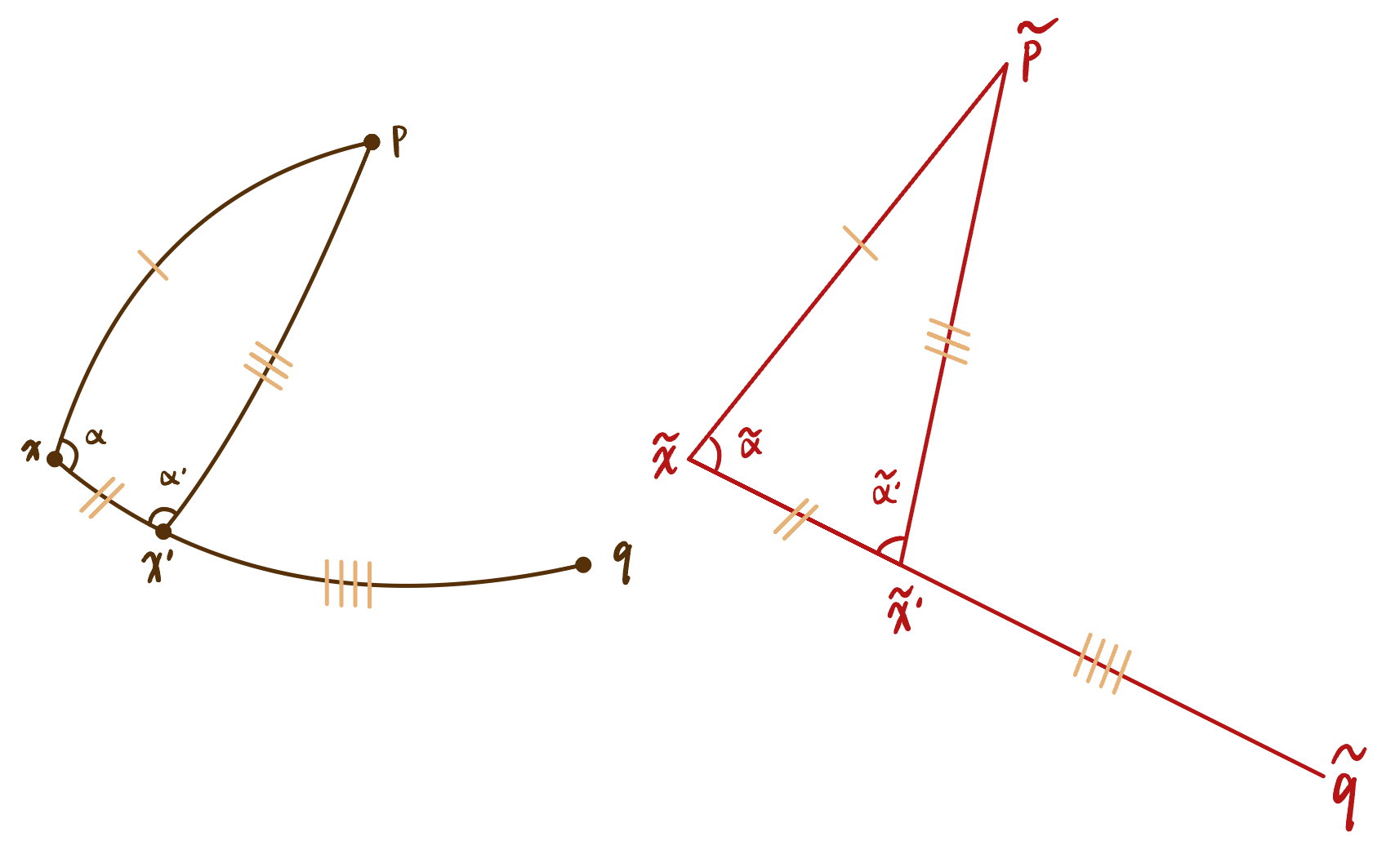}
    \end{figure}
    
    \begin{claim}
    	The hinge comparison fails on the smaller hinge $[{x'}_p^q]$, i.e. $\modcvee^\kappa[{x'}_p^q] < \abs{p - q}$
    \end{claim}
    \begin{proof}
    	For $[x_p^{x'}]$ and $[{x'}_p^x]$, we can check that
    	\begin{align*}
    		\abs{p - x} + \abs{x - x'} < \abs{p - x} + 3\abs{x - x'} \leq \frac{2}{3} l\\
    		\abs{x - x'} + \abs{x' - p} \leq \abs{p - x} + 2\abs{x - x'} < \frac{2}{3} l
    	\end{align*}
    	Thus, the the hinge comparison holds for both $[x_p^{x'}]$ and $[{x'}_p^x]$, i.e. 
    	\begin{align*}
    		& \alpha \geq \tilde{\alpha}\\
    		& \alpha' \geq \tilde{\alpha}' \iff \pi - \alpha' \leq \pi - \tilde{\alpha}'
    	\end{align*}
    	We draw the following picture and annotate the length of each side.
    \begin{equation*}
    	t = \abs{x - p}, \quad w = \abs{x' - p}, \quad s_1 = \abs{x - x'}, \quad s_2 = \abs{x' - q}, \quad s = s_1 + s_2 = \abs{x - q}
    \end{equation*}
    Also, for opposite sides, we denote
    \begin{align*}
    	& d = \modcvee^\kappa\br{\alpha; t, s}, \quad \tilde{d} = \modcvee^\kappa\br{\tilde{\alpha}; t, s}\\
    	& d' = \modcvee^\kappa\br{\pi - \alpha'; w, s_2}, \quad \tilde{d}' = \modcvee^\kappa\br{\pi - \tilde{\alpha}'; w, s_2}
    \end{align*}
    \begin{figure}[htbp]
    \centering
        \includegraphics[width=0.9\textwidth]{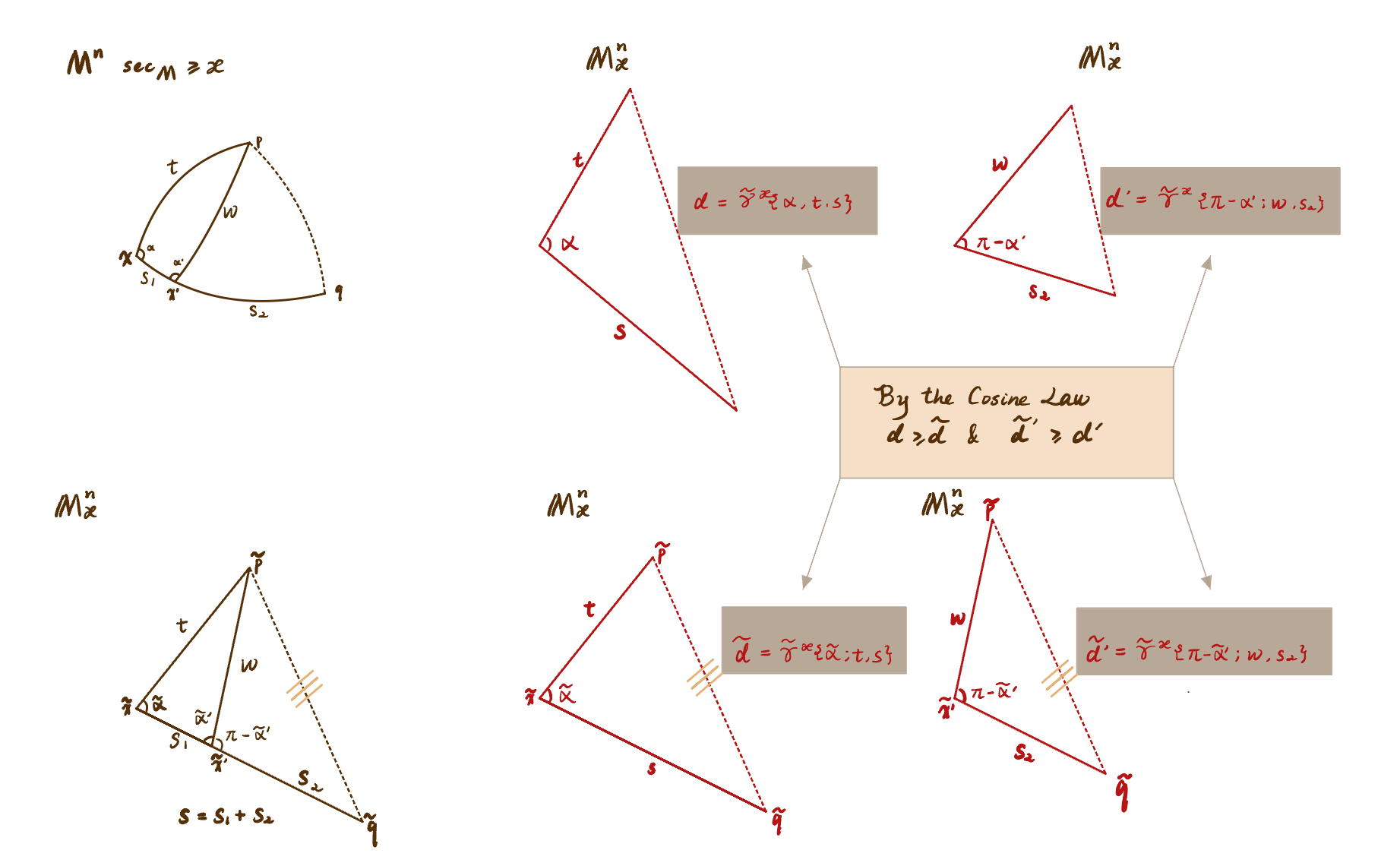}
    \end{figure}
    It is easy to observe that  both $\tilde{d}$ and $\tilde{d}'$ indicate the length of the side $[\tilde{x}\tilde{q}]$, thus $\tilde{d} = \tilde{d}'$. By the cosine law, we have both $\tilde{d} \leq d$ and $d' \leq \tilde{d}'$. Therefore, we have
    \begin{equation*}
    	d' \leq \tilde{d}' = \tilde{d} \leq d
    \end{equation*}
    And by the assumption \ref{eq: key lemma contra assumption}, we can conclude that $d' \leq d < \abs{p - q}$. That means we can construct a ``smaller hinge'' $[{x'}_p^q]$ where the hinge comparison fails.
    \end{proof}
    
\subsection{Step 2: Construction of a Sequence of Smaller Hinges}

In this step, we repeat our construction. We denote $x$ by $x_0$, $x'$ by $x_1$, $d$ by $d_0$, $d'$ by $d_1$. Then we can construct a smaller hinge $[{x_2}_p^q]$ from $[{x_1}_p^q]$. 
\begin{itemize}
    \item If $\abs{x_1 - p} \leq \abs{x_1 - q}$, then again, we take $x_2 \in [x_1 q]$ such that 
    \begin{equation*}
        \abs{p - x_1} + 3\abs{x_1-x_2} = \frac{2}{3}l;
    \end{equation*}
    \item If $\abs{x_1 - p} \geq \abs{x_1 - q}$, then we take $x_2 \in [x_1 p]$ such that 
    \begin{equation*}
        \abs{q - x_1} + 3\abs{x_1-x_2} = \frac{2}{3}l. 
    \end{equation*}
\end{itemize}
By continuing this process, we will construct a sequence of hinges $\br{[{x_i}_p^q]}_{i = 0}^\infty$ and a corresponding sequence $\br{d_i=\modcvee^\kappa[{x_i}_p^q]}_{i = 0}^\infty$ such that 
\begin{equation*}
	\abs{p - q} > d_0 > d_1 > d_2 > \dots
\end{equation*}

\subsection{Step 3: How Far This Process Can Go?}

The question now is how far this process can go. We are going to discuss this in cases: 
\begin{itemize}
    \item \textbf{Case 1: }If there exists some $N$ such that the hinge $[{x_N}_p^q]$ satisfies
    \begin{equation*}
    \abs{p - x_N} + \abs{x_N - q} \leq \frac{2}{3}l.
\end{equation*}
    Then by assumption, the hinge comparison holds, i.e.
    \begin{equation}\label{eq: key x_N}
    \mangle[{x_N}_p^q] \geq \modangle^\kappa({x_N}_p^q)
\end{equation}
    However, the contradiction happens because
    \begin{equation*}
    \abs{p - q} > d_0 > d_1 > \cdots > d_N > \cdots
\end{equation*}
But \ref{eq: key x_N} is equivalent to $\abs{p - q} \leq d_N$. 
    
    \item  \textbf{Case 2: } If the \textbf{Case 1} doesn't happen, consider the sequence $r_n:= \abs{p - x_n} + \abs{x_n - q}$. Notice that $r_n \in (\frac{2}{3}l, l]$ is non-increasing and the limits $r_\infty \in [\frac{2}{3}l, l]$. Moreover, 
    \begin{equation*}
        r_n - r_{n + 1} = \abs{x_n - x_{n + 1}} + \abs{x_n - p} - \abs{x_{n + 1} - p} \to 0\quad \text{as $n \to \infty$}
    \end{equation*} 
    It is not hard to check that all the sides of the triangle $\triangle(x_n, x_{n + 1}, p)$ remain bounded away from $0$ for large $n$, i.e. we can find $N$ such that
    \begin{equation}\label{eq: key lemma sides not too small}
        \abs{x_N - x_{N + 1}} \geq \frac{l}{100}, \quad \abs{x_{N + 1} - p} \geq \frac{l}{100}, \quad \abs{x_N - p} \geq \frac{l}{100}
    \end{equation}
    Let's explain why \ref{eq: key lemma sides not too small} is true for some large $n$. For example, if we have $[{x_n}_p^q]$ admits $\abs{x_n - p} < \frac{l}{100}$. Since $r_n \geq \frac{2}{3}l$ and 
    \begin{equation*}
    	\abs{x_n - p} + 3\abs{x_n - x_{n + 1}} = \frac{2}{3}l
    \end{equation*}
    Then $\abs{x_n - x_{n + 1}} > \frac{197l}{900} > \frac{l}{9}$, and by triangular inequality $\abs{p - x_{n + 1}} \geq \abs{x_n - x_{n + 1}} - \abs{x_n - p} \geq \frac{l}{9} - \frac{l}{100} > \frac{l}{100}$. Still by triangular inequality, $\abs{p - x_{n + 1}} \leq \abs{x_n - x_{n + 1}} + \abs{x_n - p} < \frac{297l}{900}$. Then since 
    \begin{equation*}
    	\abs{x_{n + 1} - p} + 3\abs{x_{n + 1} - x_{n + 2}} = \frac{2}{3}l,
    \end{equation*}
    it is easy to compute that $\abs{x_{n + 1} - x_{n + 2}} > \frac{l}{100}$ and $\abs{p - x_{n + 2}} > \frac{l}{100}$. Therefore we have \ref{eq: key lemma sides not too small} holds for $N = n + 1$. 
    \begin{figure}[htbp]
    \centering
        \includegraphics[width=0.6\textwidth]{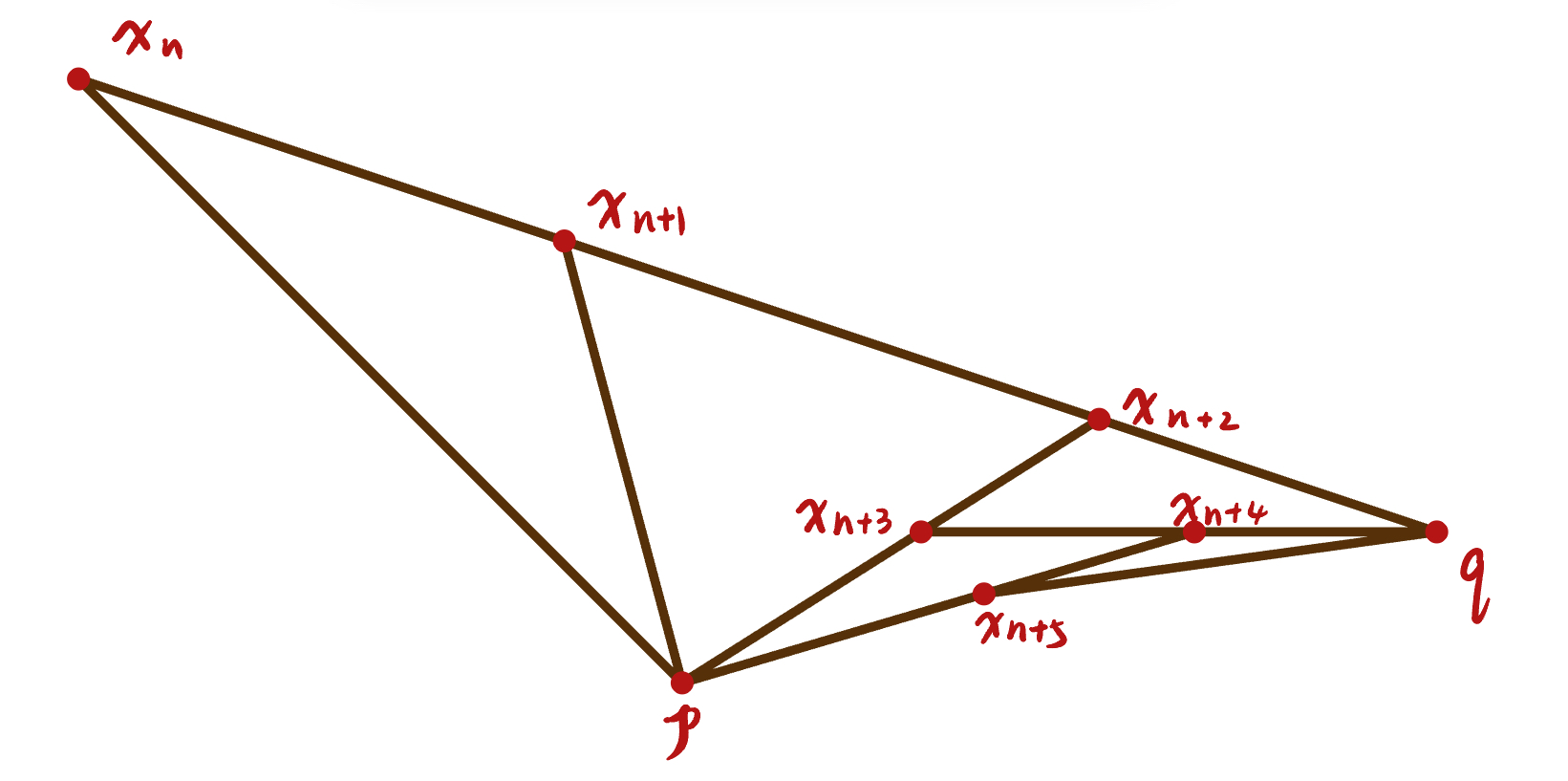}\caption{For example, if $\abs{x_{n} - p} < \frac{l}{100}$, we can pick the point $x_{n + 2}$ such that \ref{eq: key lemma sides not too small} holds.}
    \end{figure}
    We can also check that the sides $[x_np]$ and $[x_nq]$ are approximately the same, i.e. $\abs{\abs{x_n - p} - \abs{x_n - q}} \leq \frac{2l}{9}$. Use that $|p-x_n|\le |q-x_n|$ and $|p-x_n|+|q-x_n|\le l $. This implies that $|p-x_n|\le \frac{l}{2}$.   If $|x_n-x_{n+1}|$ is very small then  $\abs{p - x_n} + 3\abs{x_n-x_{n+1}} = \frac{2}{3}l$ is impossible because  $|p-x_n|\le \frac{l}{2}$.
    
     Consider the triangle $[x_n x_{n + 1}p]$ of all the sides $\geq \frac{l}{100}$ and $\abs{p - x_n} + \abs{x_n - x_{n + 1}} - \abs{p - x_{n + 1}} = r_n - r_{n + 1} \to 0$. In the model triangle $\Tilde{\alpha}_n \to \pi$. 
    Here $\alpha_n=\mangle[{x_n}_p^{x_{n+1}}]=\mangle[x_p^{q}]$ (recall that $x_{n+1}\in [x_n,q]$,  else we switch the roles of $p$ and $q$)
     But we also know that $\alpha_n \geq \Tilde{\alpha}_n=\Tilde\mangle[{x_n}_p^{x_{n+1}}]$ because 
    \begin{equation*}
        \abs{p - x_n} + \abs{x_n - x_{n + 1}} < \abs{p - x_n} + 3\abs{x_n - x_{n + 1}} = \frac{2}{3}l
    \end{equation*}
    Thus $\alpha_n \geq \Tilde{\alpha}_n \to \pi$ as well. Therefore, $d_n - (\abs{p - x_n} - \abs{x_n - q}) \to 0$. That means $d_n = \underbrace{(\abs{p - x_n} - \abs{x_n - q})}_{\text{$\geq \abs{p - q}$ by triangular inequality}} \pm \epsilon_n \to 0$. Therefore, 
    \begin{equation*}
        \lim_{n \to \infty}d_n \geq \abs{p- q}
    \end{equation*}
    But recall that \begin{equation*}
    \abs{p - q} > d_0 \geq d_1 > \cdots \geq d_n \geq \cdots
    \end{equation*}
    So we have the contradiction. 
\end{itemize}
Now we finished the proof of the key lemma. 

\begin{remark}
    The techniques in this lemma can be generalized to Alexandrov spaces because they do not depend on Riemannian geometry, see \cite{AKP22}.
\end{remark}

\section{The Global Toponogov's Comparison}
In this section, we want to generalize the various local Toponogov comparison theorems to the entire manifolds.  We will present an argument in a simplified situation of compact $M$ and $\kappa\le 0$. For the proof in general cases see \cite{AKP22}.

\begin{definition}[Comparison Radius]
	Let $M^n$ be a compact Riemannian manifold with $\sect_M \geq \kappa$. For $p \in M$, we define the \textbf{comparison radius at $p$} as 
	\begin{equation*}
		\CR(p) = \sup{\br{l: \text{Comparison holds for $[x_p^q]$ with $\abs{x - p} + \abs{p - q} < l$}}}
	\end{equation*}
\end{definition}
We know that $\CR(p) > 0$ for any $p \in M$ since the Toponogov comparison theorem holds locally. In fact, if $M^n$ is compact, then 
\begin{equation}\label{eq: inf comp rad}
    \inf_{p \in M}\br{\CR(p)} = r > 0
\end{equation}
This is because for any $\br{p_i}$ such that $\CR(p_i)$ converges to $r$, there exists a $p_\infty \in M$ the subsequential limits of $\br{p_i}$. Since near $p_\infty$, the comparison holds. Then $\CR(p_i) > \eps > 0$ for all large $p_i$ in the subsequence. 
\begin{remark}
	If $M$ is not compact, then \ref{eq: inf comp rad} may not hold. To deal with the case of noncompact $M$ one needs to use the Lemma on Almost Minimum  \cite[8.38]{AKP22}.
\end{remark}
Suppose $\CR(p_i) \to r$ and $p_i \to p_\infty$. Suppose for simplicity that $\CR(p_\infty) = r$, For $p_\infty$, $\CR(p_\infty)$ is the smallest possible. Then by applying the key lemma \ref{lem: The Key Lemma}, the comparison will hold for all hinge $[x_p^q]$ of size $\leq \frac{3}{2}r$ because $\frac{2}{3}(\frac{3}{2}r) = r$. This implies $\CR(p) \geq \frac{3}{2}\CR(p)$ a contradiction. In general, take large $i$ so that $\CR(p_i)$ is very close to $r$, if we take $l = \frac{3}{2}(r -\eps)$ for some fixed small $\eps$, then $\CR(p_i) > r - \frac{3}{2}l$ for all large $i$. 

To finish the proof for $\kappa > 0$ one also needs the Short Hinge Lemma~\cite{AKP22} to deal with hinges longer than $ \varpi^\kappa$.

Also, for $\kappa > 0$ the perimeter of any triangle in the model space does not exceed $2\varpi^\kappa$.
Hence the global Toponogov theorem proves the comparison for triangles of perimeter $< 2\varpi^\kappa$. Is this a serious restriction? We will see that, in Lemma~\ref{lem: perimeter bound}, the answer is No!

\section{The Global Functional Toponogov Comparison Theorem}
\subsection{Myer's Theorem}
Let's start this chapter with the following famous result due to Myer \cite{My35}.
\begin{theorem}[Myer]\label{thm: meyer's theorem}
    Let $(M^n, g)$ be a Riemannian manifold with $\sect_M \geq \kappa > 0$, then $\diam(M) \leq \varpi^\kappa = \frac{\pi}{\sqrt{\kappa}}$
\end{theorem}
\begin{proof}
    Suppose this theorem fails for some $M$ with $\sect_M \geq \kappa > 0$, by rescaling, we can assume $\kappa = 1$, so that $\sect_M \geq 1$ and $\diam > \pi$. Then there exist $p, q \in M$ such that $\abs{p - q} = \pi + \eps$ for some small $\eps$. Let $x \in [pq]$ be the midpoint, i.e. $\abs{x - p} = \abs{q - x} = \frac{\pi}{2} + \frac{\eps}{2}$. Let $v \in T_xM$ be a unit vector perpendicular to the geodesics $[pq]$, Denote $c(s) = \exp_x(sv)$. 
    \begin{claim}
        We claim that 
        \begin{equation*}
            \abs{p - c(s)} + \abs{c(s) - q} < \abs{p - q}
        \end{equation*}
        for small $s$. 
    \end{claim}
    By this claim, $[pq]$ is no longer the shortest, so we have the contradiction. 
    
    To prove this claim, we want to apply the hinge comparison to the $[x_{c(s)}^p]$, so that
    \begin{equation}\label{eq: Meyer claim 1}
    	\abs{p - c(s)} \leq \modcvee^1(\mangle[x_p^{c(s)}]; \abs{x - p}, \abs{x - c(s)}) =: \abs{\Tilde{p} - \Tilde{c}(s)}
    \end{equation}
    To check the perimeter condition holds, we first notice that $\abs{p - x} = \frac{\pi}{2} + \frac{\eps}{2}$ and $\abs{x - c(s)} = s$, by triangle inequality,
    \begin{equation*}
        \abs{p - c(s)} \leq \frac{\pi}{2} + \frac{\eps}{2} + s.
    \end{equation*}
    Thus the perimeter $\peri([xpc(s)]) \leq \frac{\pi}{2} + \frac{\eps}{2} + s + \frac{\pi}{2} + \frac{\eps}{2} + s = \pi + \eps + 2s  < 2\pi$ when $s$ and $\eps$ are small.
    
  	Consider $f(s) = d_{\tilde{p}}(\tilde{c}(s))$, it satisfies
  	\begin{equation*}
  		f'(0) = 0
  	\end{equation*}
  	and
  	\begin{equation*}
  		f''(0) = \hess_{d(\cdot, \Tilde{p})}(\Tilde{v}, \Tilde{v}) = \cot({\underbrace{d(\Tilde{p}, \Tilde{x})}_{\in (\frac{\pi}{2} , \pi)}}) < 0.
  	\end{equation*}
  	Therefore, 
  	\begin{equation}\label{eq: Meyer claim 2}
  		f(s) < f(0) \quad\text{for small $s$}
  	\end{equation}
  	By \ref{eq: Meyer claim 1} and \ref{eq: Meyer claim 2}, we have for $s > 0$ small
  	\begin{equation*}
  		\abs{p - c(s)} \leq \abs{\tilde{p} - \tilde{c}(s)} < \abs{\tilde{p} - \tilde{x}} = \abs{p - x}
  	\end{equation*}
	Similarly, 
    \begin{equation*}
        \abs{q - c(s)} \leq \abs{\tilde{q} - \tilde{c}(s)} < \abs{\Tilde{q} - \Tilde{x}} = \abs{q - x}
    \end{equation*}
    Then, by adding them together, we have 
    \begin{equation*}
        \underbrace{\abs{p - c(s)} + \abs{q - c(s)}}_{\text{$\geq \abs{p - q}$ by triangular inequality}} < \underbrace{\abs{p-x} + \abs{q - x}}_{= \abs{p - q}},
    \end{equation*}
the contradiction.
\end{proof}

\begin{lemma}\label{lem: perimeter bound}
    If $(M^n, g)$ has $\sect_g \geq \kappa > 0$, then for any triangle $[xyz]$ in $M$, we have
    \begin{equation*}
        \peri([xyz]) \leq \frac{2\pi}{\sqrt{\kappa}}
    \end{equation*}
\end{lemma}
\begin{proof}
By rescaling assume $\sect_M \geq 1$. Suppose this lemma fails for some triangle $[xyz]$ in $M$, i.e 
    \begin{equation*}
        \peri([xyz]) > 2\pi
    \end{equation*}
         Then, there exists $\kappa \in (0, 1)$ such that 
    \begin{equation*}
        \peri([xyz]) = \frac{2\pi}{\sqrt{\kappa}}    \end{equation*}
    So that for $\eps > 0$ sufficient small, we can find 
    \begin{equation*}
    	 \peri([xyz]) < \frac{2\pi}{\sqrt{\kappa - \eps}}
    \end{equation*} 
    Now, we can do comparison for $\sect_M \geq \kappa - \eps$ with the comparison triangle $[\bar{x}\bar{y}\bar{z}] = \modtriangle^{\kappa - \eps}\br{\abs{x - y}, \abs{y - z}, \abs{x - z}}$. It also holds that 
    \begin{equation*}
        \abs{\overline{x} - \overline{z}} + \abs{\overline{y} - \overline{z}} + \abs{\overline{x} - \overline{z}} = \frac{2\pi}{\sqrt{\kappa}} < \frac{2\pi}{\sqrt{\kappa - \eps}},
    \end{equation*}
    which means $\overline{x}, \overline{y}, \overline{z}$ almost lie on a great circle in $\modsp{\kappa - \eps}$. Moreover, this means we can find $\bar{q} \in [\bar{y}\bar{z}]$ almost opposite to $\bar{x}$. By point-on-a-side comparison, $d \geq \bar{d}$, so
    \begin{equation*}
    	d(x, q) \geq \bar{d}(\bar{x}, \bar{q}) \geq \frac{\pi}{\sqrt{\kappa - \eps}} \pm \delta > \pi
    \end{equation*}
    since $\kappa < 1$ and $\eps, \delta$ are very small. Now the contradiction rises since $\diam(M) \leq \pi$ due to Meyer's theorem. 
\end{proof}
\subsection{Functional Global Toponogov Comparison Theorem}
Recall that we proved if $(M^n, g)$ with $\sect_M \geq \kappa$, let $p \in M$, denote $f = \mdk\circ d_p$.  Near $p$ it satisfies
 \begin{equation*}
     \text{$f'' + \kappa f \leq 1$ along any geodesic} \iff \hess_f + \kappa f\id \leq \id. 
 \end{equation*}
We want to say that this holds globally so that we can globalize Jensen's inequality \ref{thm: Jensen's inequality} and prove the global version of the point-on-a-side comparison. However, $d(\cdot, p)$ need not be smooth outside of the ball of the $\inj(M)$. Therefore, we need to interpret the above inequality appropriately in the region where $f \in L^1([x, y])$ for $x < y \in \R$

Let $f \in L^1([x, y])$ be a function that we want to understand. Remember the derivative in general sense is also denoted by $f'$ such that for all $\phi \in C_c^\infty ([x, y])$, 
\begin{equation*}
	\int_x^y f(t)\phi'(t)dt = -\int_x^y f'(t)\phi(t)dt
\end{equation*}
Let $\lambda, \kappa \in \R$. We want to understand the inequality
\begin{equation*}
    f'' + \kappa f \leq \lambda
\end{equation*} 
The simplest case is when $\kappa = 0$ then $f'' \leq \lambda$ in the general sense.

\begin{definition}[$\lambda$-Concave(Convex)]
$f: [x, y] \to \R$ is \textbf{$\lambda$-concave(convex)} or  $f'' \leq \lambda$ ($f'' \geq \lambda$) in the general if $g(t) = f(t) - \frac{\lambda t^2}{2}$ is concave (convex). In other words, $g$ satisfies
 \[
   g((1 - t)x + ty) \geq(\leq) (1 - t)g(x) + t g(y)
 \]
 for any $ t \in [0, 1]$.
\end{definition}
\begin{corollary}[Jensen's inequality]
	Let $f: [x, y] \to \R$ be $\lambda$-concave. Then for any $\bar{f}: [x, y] \in \R$ such that $\bar{f}'' = \lambda$ in the general sense, $\bar{f}(x) = f(x)$ and $\bar{f}(y) = f(y)$, it holds that
	\begin{equation*}
		\bar{f}\leq f 
	\end{equation*}
	on $[x, y]$. 
\end{corollary}

\begin{definition}[Semi-concave(convex)]\label{def: semi-concave}
    $f$ is called \textbf{semi-concave(convex)} if $\forall t \in [x, y]$, $\exists \eps > 0$, $\lambda > 0$ such that $f|_{[t - \eps, t + \eps]\cap I}$ is $\lambda$-concave(convex). 
\end{definition}
\begin{lemma}\label{lem: semi-concave implies local lipschitz}
	If $f: (a, b) \to \R$ is $\lambda$-concave(convex), then $f$ is locally Lipschitz on $(a, b)$, i.e. for any $t_0 \in (a, b)$, there is a constant $L > 0$ and an closed interval $[c, d] \subseteq (a, b)$ contains $t_0$ such that 
	\begin{equation*}
		\abs{f(t) - f(t_0)} \leq L\abs{t - t_0}
	\end{equation*}
	for all $t \in [c, d]$. 
\end{lemma}

\begin{proof}
	We can prove this for convex functions, then the same applies to concave functions by adding the negative sign. Fix $[c, d]$ and let $x < y$ be arbitrary points in $[c, d]$. Take $\delta > 0$ such that
	\begin{equation*}
		a + \delta < c < d < b - \delta.
	\end{equation*}
	Firstly, we claim that for a function $g$ that is convex on $(a, b)$, it is also locally Lipschitz.  It is not hard to show that that a convex function is bounded on any closed subinterval of $(a,b)$. Hence, there exists $m, M$ such that 
	\begin{equation*}
		f(x) \in [m, M]
	\end{equation*} 
	for all $x \in [a + \delta, b - \delta]$. Take a fixed $z \in (d, b - \delta]$ and define $\kappa = \frac{y - x}{z - x}$. It follows that $0 < \kappa < 1$ and $y = \kappa z + (1 - \kappa)x$, and by the convexity
	\begin{equation*}
		g(y) \leq \kappa g(z) + (1 - \kappa)g(x) = g(x) + \kappa(g(z) - g(x))
	\end{equation*}
	Hence, 
	\begin{align*}
		g(y) - g(x) 
		&\leq \kappa(g(z) - g(x))\\
		&\leq \kappa(M - m)\\
		&= \frac{y - x}{z - x}(M - m)\\
		&< \frac{y - x}{z - d}(M - m)\\
		&< \frac{M - m}{z - d}\abs{y - x}
	\end{align*}
	Switching the variable we get
	\begin{equation*}
		-[g(y) - g(x)] = \frac{M - m}{z - d}\abs{y - x},
	\end{equation*}
	This implies,
	\begin{equation*}
		\abs{g(y) - g(x)} \leq L\abs{y - x}
	\end{equation*}
	where $L = \frac{M - m}{z - d}$. 
	
	And for an $\lambda$-convex function $f$  we have that $f(t)-\frac{\lambda}{2} t^2$ is convex and hence  by above it is locally Lipschitz on $(a,b)$.
	
	Hence $f$ is Locally Lipschitz as well.

\end{proof}

\begin{theorem}\label{thm: lambda-concave equivalent}
    The following are equivalent for $f: (a, b) \to \R$. 
    \begin{enumerate}
        \item $f$ is $\lambda$-concave;
        \item \label{lambda-generalized} $f'' \leq \lambda$ in generalized sense, i.e. $\forall \phi \in C_c^\infty((a, b))$, $\phi \geq 0$ we have
        \begin{equation*}
            \int_a^b (f'' - \lambda)\phi := \int_a^b f\phi'' - \lambda\phi \leq 0
        \end{equation*}
        if $f$ were smooth, then by integration by parts
        \begin{equation*}
            \int_a^b f''\phi = \int_a^b f \phi'',
        \end{equation*}
        In particular, when $\lambda = 0$, 
        \begin{equation*}
            \text{$f$ is concave} \iff \int_a^b f\phi'' \leq 0 \quad\forall \phi \in C_c^\infty((a, b)), \phi \geq 0
        \end{equation*}
        \item \textbf{Barrier Inequality:} $\forall x_0 \in (a, b)$, there exists open interval $I \ni x_0$ and $\overline{f}: I \to \R$ such that $\overline{f}'' = \lambda$, $f(x_0) = \overline{f}(x_0)$ and $f \leq \overline{f}$ on $(a, b)$. If $\lambda = 0$, then $\overline{f}$ is linear. 
    \end{enumerate}
\end{theorem}
\begin{proof}
    $1. \iff 2.$ comes from smoothing by convolutions. We observe that if a sequence of $\lambda_i$-concave functions $f_i$ point-wise convergent to $f$ and $\lambda_i$ convergence to $\lambda$. Then $f$ is $\lambda$-concave function immediately from the definition. We can use the following: If $f$ is concave, then $f$ can be approximated locally uniformly by a sequence of smooth functions, i.e. $\delta_i$ be a sequence of smooth bump functions converging to delta function $\delta$ so that $f$ is approximated by $f_i:= f * \delta_i$. Since $f_i'' \leq 0$ in the usual sense, $f'' \leq 0$ in the generalized sense. 
    \end{proof}
\begin{remark}
	Note the equivalence  $\text{1.} \iff \text{2.}$ in the theorem \ref{thm: lambda-concave equivalent} implies that $\lambda$-concavity in the sense of Jensen's inequality can be checked locally. 
	That is, to check that $f$ is $\lambda$-concave on $(a,b)$ it is enough to show that for any $x \in (a, b)$ there exists $\eps > 0$ such that Jensen's inequality holds on $(x - \eps, x + \eps)$. Then Jensen's inequality holds globally.  This follows because \eqref{lambda-generalized} can be checked locally using partition of unity. Indeed.
Let  $\phi: (a, b) \to \R$ be smooth with compact support. Then by POU, we can write $\phi = \sum \phi_i$ where $\phi_i\ge 0$. Moreover, we can choose the partition of unity so that  $\supp \phi_i$ is small enough so that \eqref{lambda-generalized}  holds on  $\supp \phi_i$ for every $i$. Then 
    \begin{equation*}
        \int(f'' - \lambda)\phi_i \leq 0 \implies \int(f'' - \lambda)\phi \leq 0
    \end{equation*}
    also. 
\end{remark}

This verifies that  $f''-\lambda\le 0$ holds on $(a, b)$.

It is immediate from definition that if we have a family of $\lambda$-concave functions $f_\alpha(x)$ $\alpha \in A$ such that 
\begin{equation*}
   f_\alpha|_{(a, b)} \ge C \quad\text{for some } C\in \R
\end{equation*}
Then $f(x) = \inf\br{f_\alpha(x): \alpha \in A}$ is also $\lambda$-concave on $(a,b)$. 
\begin{proposition}
    If $f: (a, b) \to \R$ is semi-concave and $\phi: \R \to \R$ is smooth and $\phi' \geq 0$. Then $\phi\circ f$ is also semi-concave.
\end{proposition}
\begin{proof}
	It is sufficient for us to prove $\phi\circ f$ is locally $\lambda$-concave for some constant.
	By directly taking the derivatives, we have \begin{equation}\label{eq-second-der}
			(\phi\circ f)'' = \phi''(f)(f')^2 + \phi'(f)f''
	\end{equation} 
	By the smoothness of $\phi$, we know that there is a local bound for  $\phi, \phi'$ and $\phi''$. By Lemma~\ref{lem: semi-concave implies local lipschitz}, we know that $f$ is locally bounded and locally Lipschitz. Thus, first term $\phi''(f)(f')^2$ is locally bounded.  On the other hand, we know that $f''$ also has a local bound above since $f$ is semi-concave. Since $\phi'\ge 0$ and $\phi'$ is locally bounded this gives that $\phi'(f)f''$ is locally bounded above. This argument works if $f$ is smooth. For general semiconcave $f$ the result then follows by approximation.
\end{proof}
\begin{remark}
	The above proposition also holds if we replace the assumption with $\phi$ being semi-concave and $\phi' \geq 0$. 
\end{remark}
\begin{definition}
    Now if $\kappa \neq 0$, \textbf{$f: (a, b) \to \R$ is $(\kappa, \lambda)$-concave} (satisfies $f'' + \kappa f \leq \lambda$) if Jensen's inequality holds, i.e. $\forall x, y \in (a, b)$ such that $\abs{x - y} < \frac{2\pi}{\sqrt{\kappa}}$ if $\kappa > 0$ take $\overline{f}$ satisfies 
    \begin{equation*}
        \overline{f}'' + \kappa \overline{f} = \lambda
    \end{equation*}
    $\overline{f}(x) = f(x)$ and $\overline{f}(y) = f(y)$. Then $f \geq \overline{f}$ on $[xy]$. (We are not assuming that $f$ is smooth) Indeed, we assume that $f$ is semi-continuous.
\end{definition}
Recall that we proved that if $f$ is smooth, then $f'' + \kappa f \le  \lambda \iff \text{Jensen's inequality holds for $f$}$.  This justifies the above definition.
\begin{theorem}
    Let $f: (a, b) \to \R$ be LSC(lower semi-continuous). Then the following are equivalent.
    \begin{itemize}
        \item $f'' + \kappa f \leq \lambda$ in the above sense of Jensen's inequality.
        \item  $f'' + \kappa f \leq \lambda$ in generalized sense, i.e. $\forall \phi \in C_c^\infty((a, b))$, $\phi > 0$, 
        \begin{equation*}
            \int_a^b (f'' + \kappa f - \lambda)\cdot \phi := \int f\cdot \phi' + \kappa f \phi - \lambda \phi \leq 0. 
        \end{equation*}
        \item \textbf{Barrier Inequality: } $\forall x_0 \in (a, b)$, then $\exists \overline{f}: I \to \R$ where $I \ni x_0$ is an open interval and $\length(I) < \frac{\pi}{\sqrt{\kappa}}$ if $\kappa > 0$ and $\overline{f}'' + \kappa \overline{f} = \lambda$, $\overline{f}(x_0) = f(x_0)$, $f \leq \overline{f}$ on $I$. And $\overline{f}$ solves exact equation $\overline{f}'' + \kappa \overline{f} = \lambda$. 
    \end{itemize}
\end{theorem}
\begin{remark}
    We have some comments. If $\lambda - \kappa f \leq c$ on $(a, b)$ and
\begin{align*}
    f'' + \kappa f \leq \lambda \implies &  f'' \leq \lambda - \kappa f \leq c\\
    \implies & f'' \leq c\\
    \implies & \text{$f$ is $c$-concave functions. $(\lambda, \kappa)$-concave functions. }
\end{align*} 
$f$ is semi-concave in particular locally Lipschitz on $(a, b)$. 
\begin{itemize}
    \item As before if $f_\alpha(x)$, $\alpha \in A$ are $(\lambda, \kappa)$ concave functions on $(a, b)$, 
    \begin{equation*}
       (f_\alpha(x)|_{(a, b)}) \geq c\quad \text{ for some } c\in \R
    \end{equation*}
    then $\inf\br{f_\alpha(x): \alpha \in A}$ is again $(\lambda, \kappa)$ concave follows from Jensen's inequality.
    \item  If $f_i \to f$ pointwise, and $f_i$ is $(\lambda_i, \kappa_i)$ concave and $\kappa_i \to \kappa$ and $\lambda_i \to \lambda$. Then $f$ is $(\lambda, \kappa)$-concave.
    \item Let $f:(a, b) \to \R$, then $(\lambda, \kappa)$-concavity can be checked locally because $f'' + \kappa f \leq \lambda$n in generalized sense if local. 
\end{itemize}
\end{remark}
\begin{definition}
    Let $M$ be a complete Riemannian manifold, then \textbf{$f: M \to \R$ is $\lambda$-concave} if for any geodesics $\gamma(t)$, $f(\gamma(t))$ is $\lambda$-concave. Moreover, \textbf{$f$ is $(\lambda, \kappa)$-concave} if for any unit speed geodesic $\gamma(t)$, $f(\gamma(t))$ is $(\lambda, \kappa)$ concave. 
\end{definition}
\begin{theorem}\label{thm: l-k theorem}
    Let $(M^n, g)$ be complete and $\sect_M \geq \kappa$, $p \in M$, $f = \md_\kappa(d(\cdot, p))$. Then $f$ is $(\lambda, \kappa)$-concave if $f'' + \kappa f \leq \lambda$ on $M$. ($f$ need not be globally smooth.)
\end{theorem}
\begin{proof}[Proof of Theorem~\ref{thm: l-k theorem}]
    We only need to check the Jensen's inequality. 
    By Toponogov, 
    \begin{equation*}
        d_p(\gamma(t)) \geq d_{\overline{p}}(\overline{\gamma}(t))
    \end{equation*}
    for $t \in [t_0, t_1]$ and since $\md_\kappa$ is monotone, then
    \begin{equation*}
        f(t) = \md_\kappa(d_p(\gamma(t))) \geq \md_\kappa(d_p(\overline{\gamma}(t))) = \overline{f}(t). 
    \end{equation*}
    This implies that $d_p$ is semi-concave on $M \backslash\br{p}$ since
    \begin{equation*}
        d_p = (\md_\kappa^{-1})\circ (\md_\kappa(d_p))
    \end{equation*}
    Therefore, $\md_\kappa(d_p)$ is $(\lambda, \kappa)$-concave. Thus it is semi-concave. Therefore $d_p$ is also semi-concave on $M\setminus \{p\}$. 
\end{proof}
\begin{remark}
 The last result holds on any complete Riemannian manifold.  That is if $(M, g)$ is complete and $p\in M$ then $d_p$ is semi-concave on $M\setminus \{p\}$. 
\end{remark}

\chapter{Appliction of Toponogov's Comparison Theorem}
\section{Curvature bounds and Topological Complexity}
There are many results when the curvature bound implies a topological bound. In this section, we introduce the work of Gromov on bounding topological complexity (growth of the size of the fundamental group) using the short basis method.
\subsection{Gromov's Estimate Theorem}
\begin{theorem}[Gromov's Estimates Theorem]\label{thm: gromov's theorem}
   Let $(M^n, g)$ be a complete, connected Riemannian manifold. 
    Then
    \begin{enumerate}
\item If $\sect_M \geq \kappa$ for any $\kappa \in \R$ and $\diam(M) \leq D$, then $\pi_1(M)$ can be generated by at most $C(n, \kappa, D)$ elements. 
\item If $(M^n, g)$ has $\sect_M \geq 0$ then $\pi_1(M)$ can be generated by at most $C(n)$ elements. 
\end{enumerate}
\end{theorem}
\begin{remark}
Note that in case 2 $M$ is not assumed to be compact.	
\end{remark}
\begin{remark}
 Constants $C(n)$ and $C(n, \kappa, D)$ can be made explicit (See \ref{eq: explicit C(n)} and \ref{eq: explicit C(D, k, n)}). 
\end{remark}
We are going to prove this theorem in the next section. 
\begin{lemma}\label{lem: loop length bound.}
    Let $(M,g)$ be compact. Then for any  $p \in M$,  $\pi_1(M, p)$ is generated by loops of length $\leq 2\diam(M)$. 
\end{lemma}
\begin{proof}
   	Let $\gamma$ be any loop based at $p$ and $\gamma: [0, 1] \to M$ such that $\gamma(0) = \gamma(1) = p$ and $\gamma$ is continuous. Fix $\eps > 0$ and we can subdivide $[0, 1]$ into small intervals such that 
    \begin{equation*}
        0 = t_0 < t_1 < \cdots < t_k = 1
    \end{equation*}
    such that for each $i = 0, \dots, k$,
    \begin{equation*}
        \length(\gamma|_{[t_i, t_{i + 1}]}) \leq \eps.
    \end{equation*}

    
    For each $i$, we connect $p$ with $\gamma(t_i)$ with geodesic $c_i$, then 
    \begin{equation*}
    	\pi_1(M, p) \ni [\gamma] = \prod_i(\underbrace{c_i\gamma|_{[t_i, t_{i + 1}]}c_{i + 1}^{-1}}_{\alpha_i})
    \end{equation*}
    Notice that each $c_i$ is a geodesic, $\length(c_i) \leq \diam(M)$, and thus
    \begin{equation*}
        \length(\alpha_i) \leq \eps + 2\diam(M)
    \end{equation*}
    Therefore, we can conclude that $\pi_1(M)$ is generated by loops shorter than $\eps + 2\diam(M)$. Remember that $\eps > 0$ is arbitrary, as a consequence, $\pi_1(M)$ is generated by loops of length $\leq 2\diam(M)$. 
\end{proof}
\subsection{Short Basis}
In this section, we are going to introduce a construction called short basis due to Gromov \cite{Gr82}. This construction is important in the proof of Theorem \ref{thm: gromov's theorem}.

By assumption, since $M$ is a connected Riemannian manifold, we know the universal cover $\tilde{M}$ and the projection map $\pi: \tilde{M} \to M$ exist. As a universal cover, $\tilde{M}$ must be simply connected, which means we have the following isomorphism between the deck transformation group (automorphism group) and the fundamental group
\begin{equation*}
	\aut_{\pi}(\tilde{M}) \cong \pi_1(M)
\end{equation*}
This isomorphism identifies each $[\gamma] \in \pi_1(M)$ with an unique deck transformation $\varphi_\gamma$, i.e. $\varphi_\gamma$ is an homeomorphism acting on the universal cover $\tilde{M}$ such that $\pi\circ\varphi_\gamma = \pi$. Thus, we can say the fundamental group $\pi_1(M)$ acting on the universal cover $\tilde{M}$ and for each $p \in M$, we denote $\varphi_{\gamma}(p)$ by $\gamma(p)$. 

We would like to recall that the fundamental group action on the universal cover admits many good properties. Denote $\Gamma = \pi_1(M)$, 
\begin{itemize}
	\item $\Gamma$ acts freely on $\tilde{M}$, i.e. for any $[\gamma] \in \Gamma$, if $\gamma(a) = a$ for some $a \in \tilde{M}$, then $[\gamma]$ must be the identity element in $\Gamma$; This just the property of Deck transformation on any covering space.
	\item $\Gamma$ acts transitively on each fiber $\pi^{-1}(x)$ with $x \in M$, i.e. for any $a, b \in \pi^{-1}(x)$, there exists a $[\gamma] \in \Gamma$ such that $\gamma(a) = b$; This because a universal cover must be a normal cover, which acts transitively on each fiber.
	\item $\Gamma$ acts properly discontinuously on $\tilde{M}$; Namely, for all compact subset $K \subseteq M$, the set
	\begin{equation*}
		\br{[\gamma] \in \pi_1(M): \gamma K \cap K \neq \varnothing}.
	\end{equation*}
	is finite. Indeed, if this is not true. Then there is a compact set $K$ and an infinite sequence of $[\gamma_i] \in \Gamma$ such that $\gamma_iK \cap K \neq \varnothing$. Then for each $[\gamma_i]$, we can find $k_i \in \gamma_iK \cap K$ for each $i$, which produces a sequence $\br{k_i} \subseteq K$. Since $K$ is compact, we know that there is a sub-convergent limit $k \in K$. Taking a small neighborhood $U$ of $k$, then $\pi(U)$ is evenly covered by $\coprod_{[\gamma] \in \pi_1(M)}\gamma(U)$. But by the construction of $k$, there are infinitely many $k_i \in U$ so that $\gamma_ik_i \in \gamma_iU\cap K$ for infinitely many $\gamma_i$. Then the set $\br{\gamma_ik_i}\cap K$ is an infinite discrete subset in $K$. Moreover, the set $\br{\gamma_ik_i}\cap K$ must also be compact since it is a closed subset of the compact set $K$. Then the contradiction arises as we know there is no discrete space with an infinite number of points that is compact.
	
	\item $\Gamma$ acts isometrically on $\tilde{M}$ if we endowed $\tilde{M}$ with the induced metric, say $\tilde{g}$ from $M$. i.e. for any $a, b \in \tilde{M}$, 
		\begin{equation*}
			d_{\tilde{g}}(a, b) = d_{\tilde{g}}(\gamma(a), \gamma(b))
		\end{equation*}
		for any $[\gamma] \in \pi_1(M)$. Indeed, if we denote $\Gamma(a, b)$ and $\Gamma(\gamma(a), \gamma(b))$ the family of paths from $a$ to $b$ and the family of paths from $\gamma(a)$ to $\gamma(b)$ respectively, by the definition of deck transformation ($\pi\circ\varphi_\gamma = \pi$), we have
		\begin{equation*}
			\pi(\Gamma(a, b)) = \pi(\Gamma(\gamma(a), \gamma(b)))
		\end{equation*}
		therefore, 
		\begin{align*}
			d_{\tilde{g}}(a, b) 
			&= \inf_{\sigma\in \pi(\Gamma(a, b))}\length(\sigma)\\
			& = \inf_{\sigma \in \pi(\Gamma(\gamma(a), \gamma(b)))}\length(\sigma)\\
			& =  d_{\tilde{g}}(\gamma(a), \gamma(b))
		\end{align*}
\end{itemize}

 One can construct a \textbf{short basis of $\Gamma$ at $p$} as the following. Fix a base point $\tilde{p}$ in $\tilde{M}$ such that $p = \pi(\tilde{p})$ via the projection $\pi: \tilde{M} \to M$. For each $[\mathfrak{g}] \in \Gamma$, we can define 
\begin{equation*}
	\abs{\mathfrak{g}} := d_{\tilde{g}}(\tilde{p}, \mathfrak{g}(\tilde{p}))
\end{equation*}
the length of $[\mathfrak{g}]$. 
\begin{remark}
	the length of $[\mathfrak{g}]$ is dependent on the choice of $\tilde{p}$, i.e. in general, we CAN NOT have
	\begin{equation*}
		d_{\tilde{g}}(\tilde{p}, \mathfrak{g}(\tilde{p})) = d_{\tilde{g}}(\mathfrak{h}(\tilde{p}), \mathfrak{g}(\mathfrak{h}(\tilde{p})))
	\end{equation*}
	for any other $[\mathfrak{h}] \in \Gamma$
	This is because 
	\begin{align*}
		d_{\tilde{g}}(\tilde{p}, \mathfrak{g}(\tilde{p})) = d_{\tilde{g}}(\mathfrak{h}(\tilde{p}), \mathfrak{h}(\mathfrak{g}(\tilde{p})))
	\end{align*}
	but in general, the fundamental group is not abelian. Thus it is not necessarily true that $[\mathfrak{h}\cdot\mathfrak{g}] = [\mathfrak{g}\cdot\mathfrak{h}]$ and therefore
	\begin{equation*}
		d_{\tilde{g}}(\mathfrak{h}(\tilde{p}), \mathfrak{g}(\mathfrak{h}(\tilde{p}))) \neq d_{\tilde{g}}(\mathfrak{h}(\tilde{p}), \mathfrak{h}(\mathfrak{g}(\tilde{p})))
	\end{equation*}
\end{remark} 
Take $\tilde{\gamma}$ the geodesics connecting $\tilde{p}$ and $\mathfrak{g}(\tilde{p})$. Since $\pi(\tilde{p}) = \pi(\mathfrak{g}(\tilde{p}))$. Then $\gamma = \pi(\tilde{\gamma})$ is the geodesic loop at $p$ representing $\mathfrak{g}$. Notice that $ \abs{\mathfrak{g}} = \abs{\gamma}_g$ (the length of the geodesic loop in metric $g$). 

\begin{figure}[htbp]
    \centering
        \includegraphics[width=0.5\textwidth]{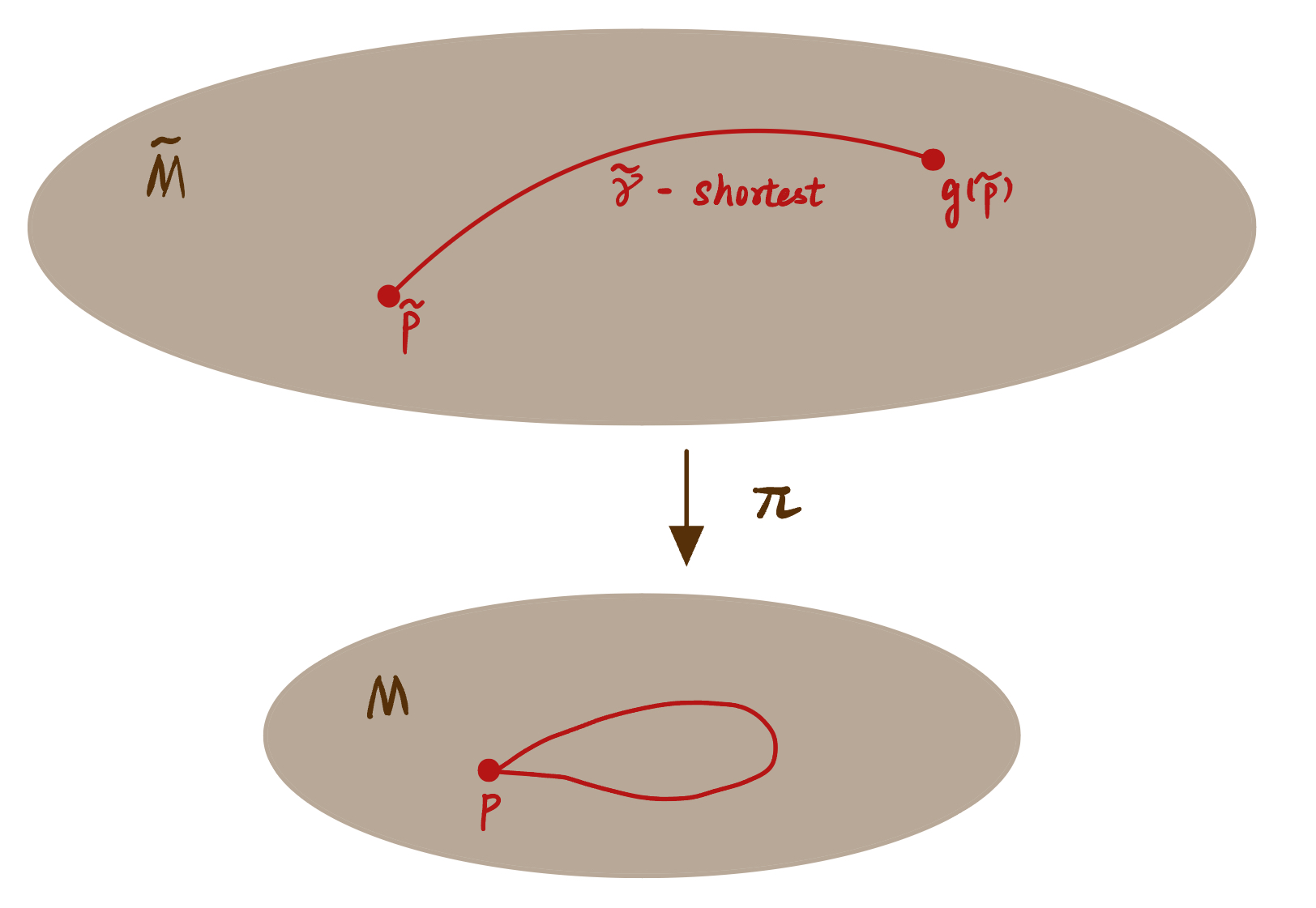}
    \end{figure}

Now we want to construct a sequence of geodesic loops at $p$ via the following procedure: 
    \begin{align*}
    	&\text{Take $\gamma_1 \neq e$ to be a  shortest in $\Gamma$. Denote $\Gamma_1 = \inp{\gamma_1}$}\\
    	&\text{Take $\gamma_2 \neq e$ to be a  shortest in $\Gamma \backslash \Gamma_1$. Denote $\Gamma_2 = \inp{\gamma_1, \gamma_2}$}\\
    	&\dots \\
    	&\text{Take $\gamma_i \neq e$ to be a  shortest in $\Gamma \backslash \Gamma_{i - 1}$ where $\Gamma_i = \inp{\gamma_1, \gamma_2, \dots, \gamma_i}$}\\
    	&\dots 
    \end{align*}
    The sequence $\br{\gamma_1, \gamma_2, \dots}$ is called a \textbf{short basis} of $\Gamma$ at $p$. 

\begin{lemma}\label{short-basis-diam}
      Let $(M,g)$ be compact. Then each of its short basis $\gamma_i$ has the length $|\gamma_i| \leq 2\diam(M)$.
    \end{lemma}
    \begin{proof}
    	We should prove this lemma by induction. The base case is to show $\abs{\gamma_{1}} \leq 2\diam(M)$. Suppose not, i.e. $\abs{\gamma_{1}} > 2\diam(M)$, after writing $\gamma_1$ in terms of the generators of the fundamental group
    	\begin{equation*}
    		\gamma_1 = \sigma_1^1\dots\sigma_{l_1}^1
    	\end{equation*}
    given by 	 Lemma~\ref{lem: loop length bound.}, $\abs{\sigma^1_{i_1}} \leq 2\diam(M)$ for each $i_1 = 1, \dots, l_1$. This implies $\abs{\gamma_1} > \abs{\sigma_{i_1}^!}$, which is impossible since $\gamma_1$ is  shortest in $\Gamma$. 
    
    	Now, assume $\gamma_1, \dots, \gamma_{k - 1}$ all have length no bigger than $2\diam(M)$. And we want to show $\abs{\gamma_{k + 1}} \leq 2\diam(M)$. 
    	Suppose not, i.e. $\abs{\gamma_{k}} > 2\diam(M)$. We can again write $\gamma_{k}$ in terms of its generators
    	\begin{equation*}
    		\gamma_{k} = \sigma_1^{k}\dots\sigma_{l_{k}}^{k}
    	\end{equation*}
    	We know one of its generators $\sigma_{i_k}^k$ belongs to $\Gamma\backslash\Gamma_{k - 1}$. By Lemma~\ref{lem: loop length bound.}, the length of the generator must be less than $2\diam(M)$. Thus we have
    	\begin{equation*}
    		\abs{\sigma_{1_k}^k} \leq 2\diam(M) < \abs{\gamma_{k}}
    	\end{equation*}
    	This contradicts the fact that $\gamma_{k}$ is the shortest in $\Gamma\backslash\Gamma_k$.
    \end{proof}
    \begin{figure}[htbp]
    \centering
        \includegraphics[width=0.8\textwidth]{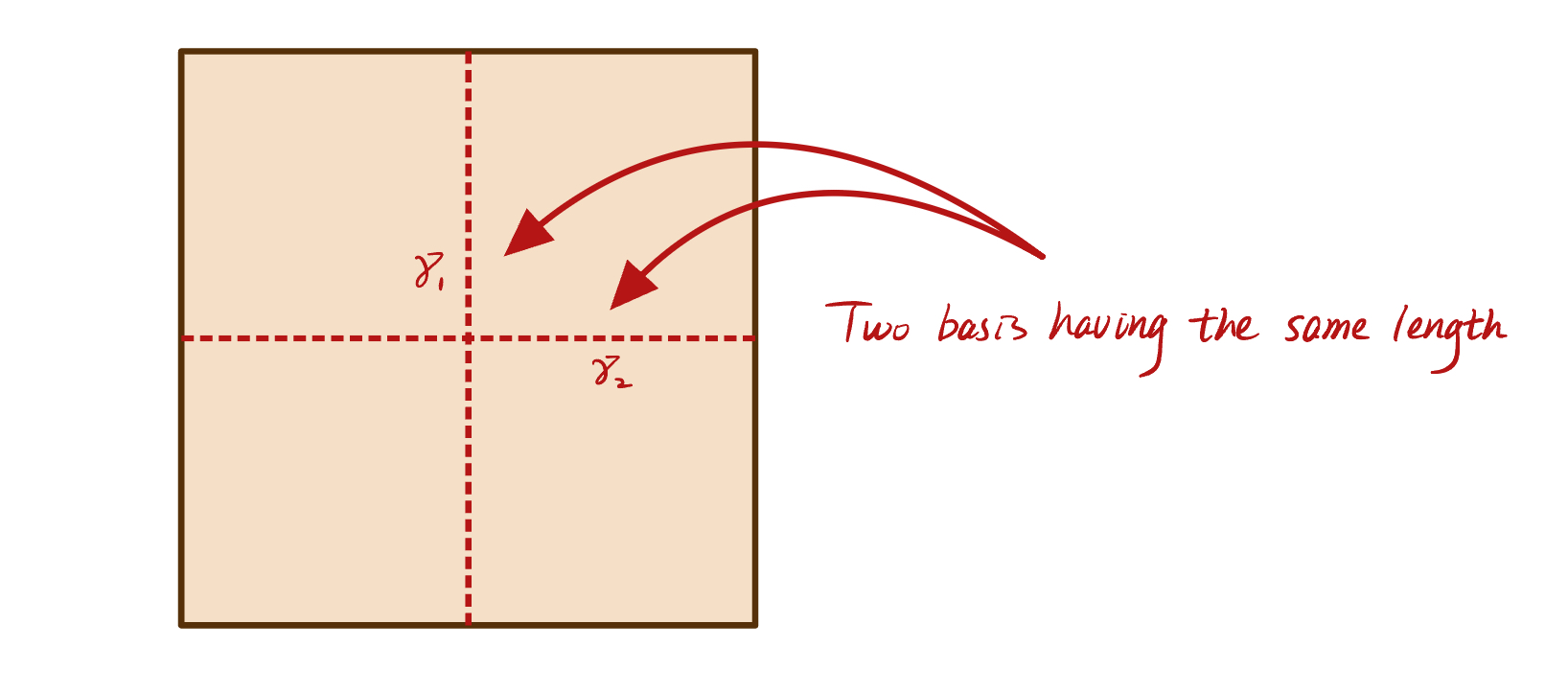}
        \caption{By Remark~\ref{rmk: equal length bases}, two elements of a short  basis of a flat torus might have the same length}
    \end{figure}
    \begin{remark}\label{rmk: equal length bases}
   		By construction
   		\begin{equation*}
   			|\gamma_1|\le |\gamma_2|\le\ldots
   		\end{equation*}
   		but in general, these inequalities need not be strict, i.e. it is possible that $|\gamma_i| =|\gamma_{i+1}|$ for some $i$. For example, the two generators of the fundamental group of a square torus might have the same length.
   \end{remark}

   	We would like to mention that while a short basis at $p$ is not unique, its length spectrum $\br{\abs{\gamma_1}_g, \abs{\gamma_2}_g, \dots} \subseteq \R$ \emph{is} unique. Denote $\Gamma(r)$ the subgroup of $\Gamma$ generated by all loops of length at most $r$. And define 
   	\begin{equation*}
   		G(r) := \inp{\text{$\gamma_i$ is in the  short basis}: \abs{\gamma_i}_g \leq r}
   	\end{equation*}
   	the subgroup generated by the short basis of length at most $r$. Since the subgroups $\Gamma(r)$ are invariantly defined and do not depend on  the choice of a short basis,  the uniqueness on the length spectrum  easily follows the following proposition
   	\begin{proposition}\label{prop: G(r) = Gamma(r)}
   		$G(r) = \Gamma(r)$ for any $r>0$. 
   	\end{proposition}
   	\begin{proof}
   		It is obvious that $G(r) \subseteq \Gamma(r)$ since each $\gamma_i \in \Gamma(r)$.
   		
   		Now we want to show the converse inclusion. Pick the largest $k$ such that $\gamma_k\in G(r)$. Then   $\gamma_1, \gamma_2, \dots, \gamma_k$ belong to $G(r)$ but  $\gamma_{k + 1}$ does not.		By construction  $\gamma_{k + 1}$ is a shortest element of $\Gamma$ outside of $\inp{\gamma_1, \dots, \gamma_k} = G(r)$. Moreover, consider any other $\mathfrak{g} \in \Gamma\backslash\inp{\gamma_1, \dots, \gamma_k}$, we find
   		\begin{equation*}
   			\abs{\gamma_1}_g \leq \abs{\gamma_2}_g \leq  \dots \leq \abs{\gamma_k}_g \leq r < \abs{\gamma_{k + 1}}_g \leq \abs{\mathfrak{g}}
   		\end{equation*}
   		Let $h \in \Gamma(r)$, then by definition
   		\begin{equation*}
   			h = \sigma_1\cdots\sigma_m
   		\end{equation*}
   		and $\abs{\sigma_i}_g \leq r$ for each $i \in \br{1, \dots, m}$. Suppose for contradiction that $h \notin G(r)$. This means at least one of $\sigma_i \notin G(r)$. Thus our previous discussion, 
   		\begin{equation*}
   			\abs{\sigma_i}_g \geq \abs{\gamma_{k + 1}}_g > r.
   		\end{equation*}
   	On the other hand by above $\abs{\sigma_i}_g \leq r$. This is a contradiction and hence $G(r)=\Gamma(r)$.
   	\end{proof}

\begin{proof}[Proof of Theorem \ref{thm: gromov's theorem}.]    
    Since $M$ is complete, we have a possibly infinite sequence of loops 
    \begin{equation*}
        \gamma_1, \gamma_2, \dots, 
        \gamma_i, \dots, \gamma_j, \dots
    \end{equation*}
   which forms a short basis of $M$ at $p$.
    Suppose $i < j$, consider the triangle $[\tilde{p}\gamma_i(\tilde{p})\gamma_j(\tilde{p})]$, we denote 
    \begin{equation*}
        l_i = \abs{\gamma_i} = d(\Tilde{p}, \gamma_i(\Tilde{p})), \quad l_j = \abs{\gamma_i} = d(\Tilde{p}, \gamma_j(\Tilde{p})), \quad l_{ij} = d(\gamma_i(\Tilde{p}), \gamma_j(\Tilde{p}))
    \end{equation*}    
    We claim that $l_{ij} \geq l_j \geq l_i$. The second inequality is trivial. Let us  show  that $l_{ij} \geq l_j$.
    It is also easy to check that $l_{ij} < l_j$. To see this, we only need to show $l_{ij} = d_{\tilde{g}}(\tilde{p}, \gamma_i^{-1}\gamma_j(\tilde{p}))$.
    Because $\Gamma$ acts on $M$ by isometries, i.e. or any $x, y \in \tilde{M}$, $\mathfrak{g} \in \Gamma$, $d_{\tilde{g}}(x, y) = d_{\tilde{g}}(\mathfrak{g}(x), \mathfrak{g}(y))$, we have
    \begin{align*}
    	l_{ij} &= d_{\tilde{g}}(\gamma_i(\tilde{p}), \gamma_j(\tilde{p})) \\
    	&= d_{\tilde{g}}(\gamma_i^{-1}\gamma_i(\tilde{p}), \gamma_i^{-1}\gamma_j(\tilde{p}))\\
    	&= d_{\tilde{g}}(\tilde{p}, \gamma_i^{-1}\gamma_j(\tilde{p})).
    \end{align*} 
    Then assume $l_{ij} < l_j$, that is $d_{\tilde{g}}(\tilde{p}, \gamma_i^{-1}\gamma_j(\tilde{p})) < \abs{\gamma_i}$.
    Since $\gamma_{j} \notin \Gamma_{j - 1}$, then $\gamma_i^{-1}\gamma_j \notin \Gamma_{j - 1}$, 
    When choosing the $j$'s elements of the shortest basis from $\Gamma\backslash\Gamma_{j - 1}$, we need to take $\gamma_i^{-1}\gamma_j$ instead of $\gamma_j$, which gives us the contradiction. 

	\begin{figure}[htbp]
    \centering
        \includegraphics[width=0.5\textwidth]{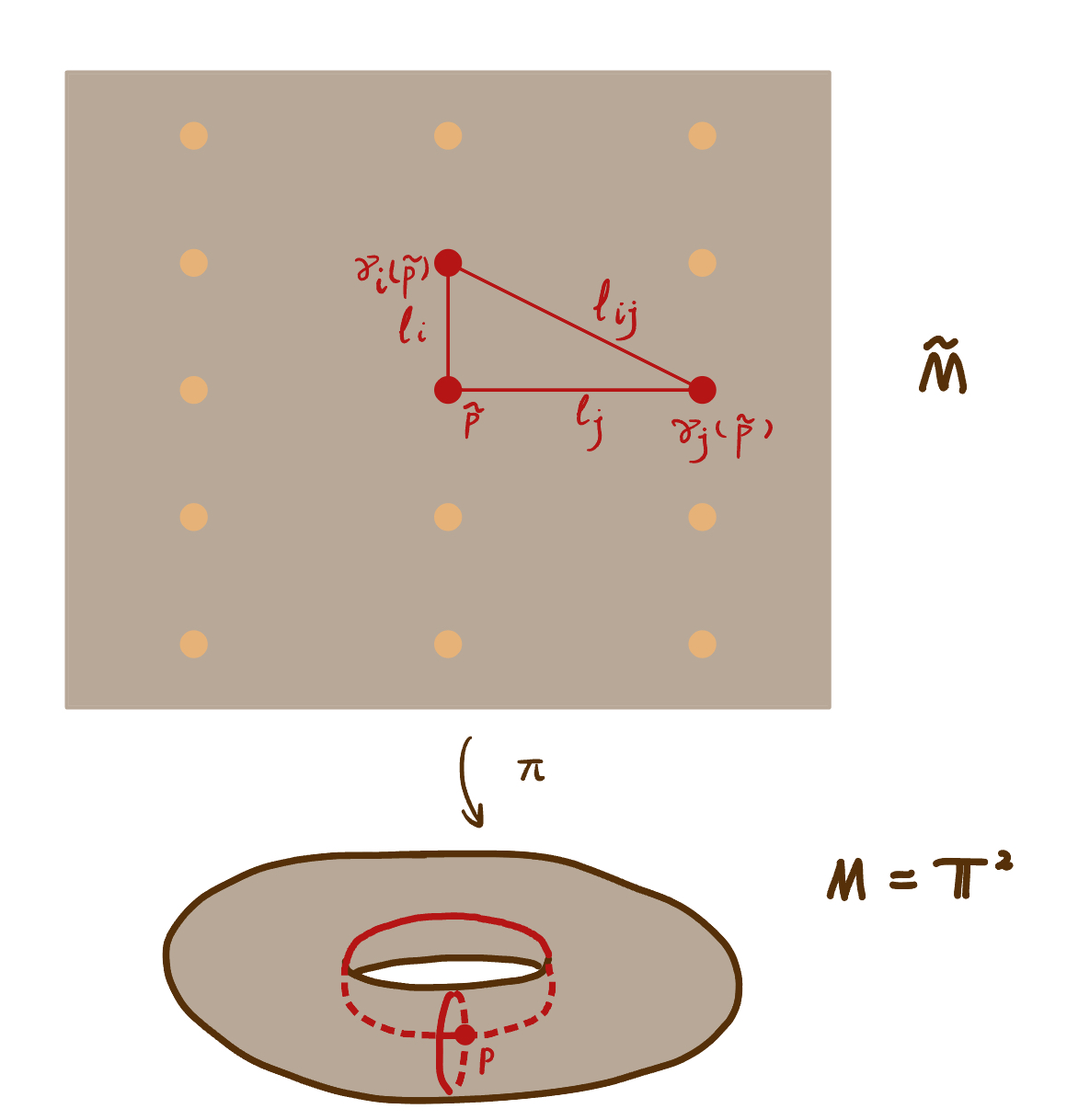}
        \caption{Example of $l_{ij} \geq l_j \geq l_i$, we take $M = \mathbb{T}^2$}
    \end{figure}

    Consider the case when $\kappa = 0$. By Toponogov's angle comparison~\ref{thm: angle}, since $l_{ij}$ is the longest side of this triangle, we have $\alpha_{ij} \geq \Tilde{\alpha}_{ij} \geq \frac{\pi}{3}$ by the cosine law. To see why, we firstly notice that    
    \begin{align*}
        \cos{\Tilde{\alpha}_{ij}} = & \frac{l_i^2 + l_j^2 - l_{ij}^2}{2l_il_j}\\
        \leq & \frac{l_j^2 + l_j^2 - l_{ij}^2}{2l_j^2}\\
        \leq & \frac{2l_j^2 - l_i^2}{2l_j^2} = \frac{1}{2} \implies \tilde{\alpha}_{ij} \geq \frac{\pi}{3}
    \end{align*}
    And by Toponogov's angle comparison~\ref{thm: angle}, we have $\alpha_{ij} \geq \Tilde{\alpha}_{ij} \geq \frac{\pi}{3}$. We draw the initial vectors $v_i$ for each $\gamma_i$. This tells us the angles between any of these vectors $\geq \frac{\pi}{3}$. (See the following picture)
    \begin{figure}[htbp]
    \centering
        \includegraphics[width=0.4\textwidth]{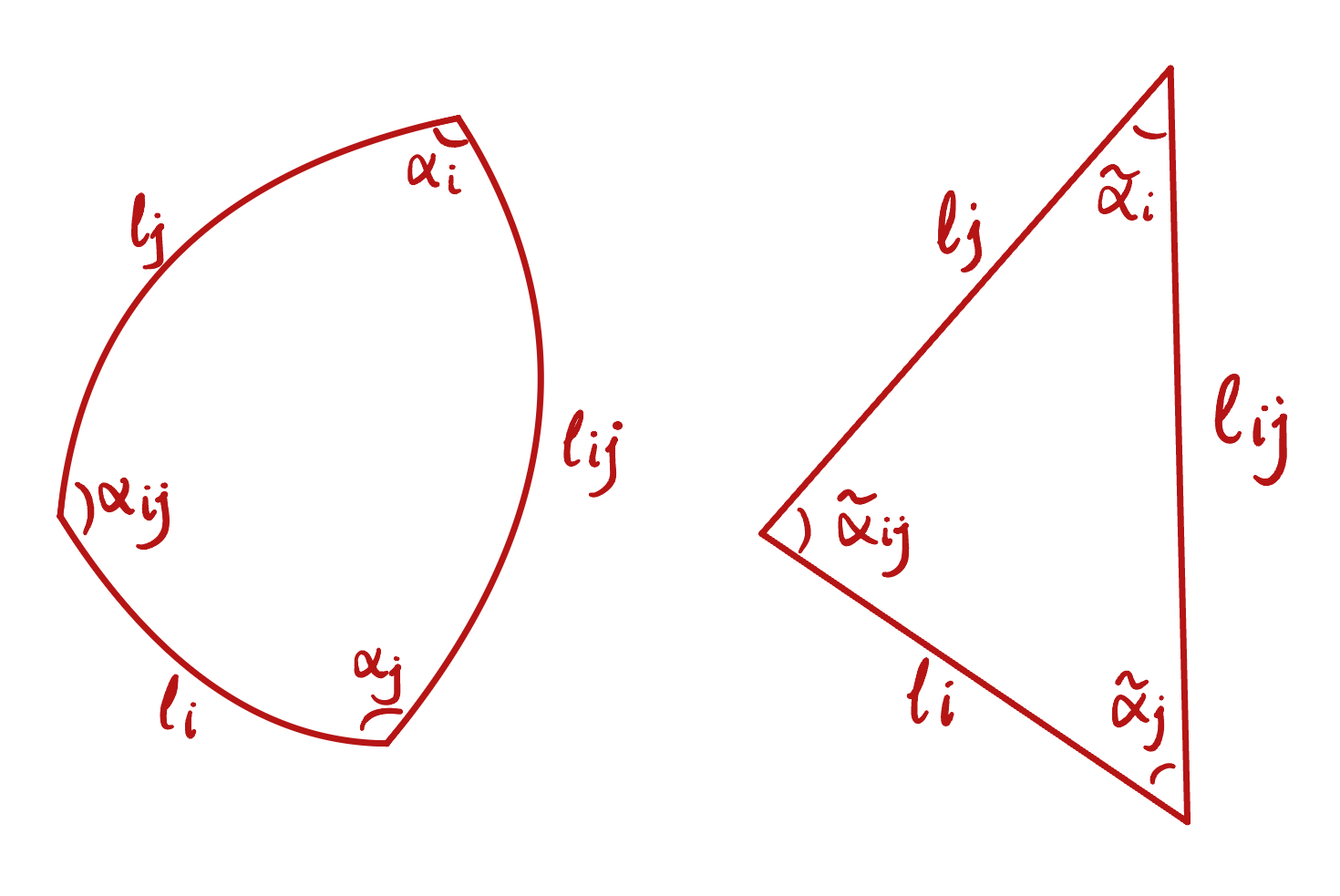}
    \end{figure}
 	
 	Look at balls of radius $\frac{\pi}{6}$ in $S^{n - 1}$ around $v_i$, they are disjoint, i.e. $B_{\frac{\pi}{6}}(v_i) \cap B_{\frac{\pi}{6}}(v_j) = \varnothing$. 
    \begin{figure}[htbp]
    \centering
        \includegraphics[width=0.5\textwidth]{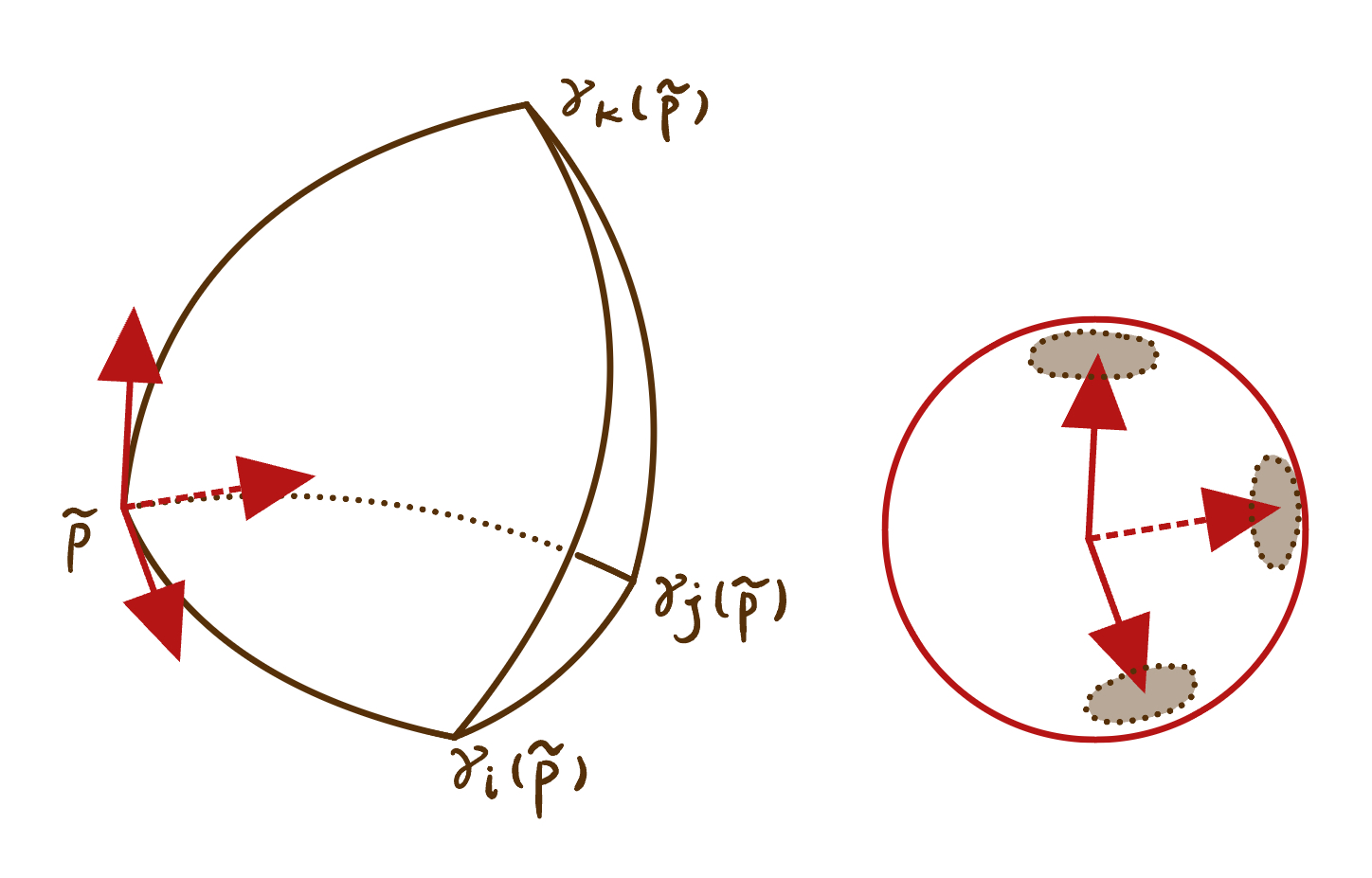}
    \end{figure}
    
    Since $\vol(S^{n - 1}) \geq m\vol(B_{\frac{\pi}{6}}(v_i)$ where $m$ is the number of the balls. We know that 
    \begin{equation}\label{eq: explicit C(n)}
    	m \leq \frac{\vol(S^{n - 1})}{\vol(B_{\frac{\pi}{6}}(v_i))} = C(n), 
    \end{equation}
    which is the second part of Gromov's theorem. Notice that $m$ is the number of the elements in the short basis of $\Gamma$ in the special case when $\sect_M \geq 0$. 

    Now, we consider the general case $M^n$ of $\sect_M \geq \kappa$, $\diam(M) \leq D$, but $\kappa < 0$ (The case $\kappa \geq 0$ has been done in the above.).
    By rescaling, we only need to consider the case when $\sect_M \geq -1$.
    \begin{align*}
        & \cosh{l_{ij}} = \cosh{l_i}\cosh{l_j} - \sinh{l_i}\sinh{l_j}\cos{\Tilde{\alpha}_{ij}}\\
        \implies & \cos{\Tilde{\alpha}_{ij}} = \frac{\cosh{l_i}\cosh{l_j} - \cosh{l_{ij}}}{\sinh{l_i}\sinh{l_j}} \leq \frac{\cosh{l_j}^2 - \cosh{l_{ij}}}{\sinh{l_j}^2} \leq \frac{\cosh{l_j}^2 - \cosh{l_j}}{\sinh{l_j}^2} \\
        & \leq \frac{\cosh{2D}^2 - \cosh{2D}}{\sinh{2D}} = \eps(D)\\
        \implies & \cos{\Tilde{\alpha}_{ij}} \leq \eps(D)\\
        \implies & \Tilde{\alpha}_{ij} \geq \arccos{\eps(D)} =: \eps'(D) \\
        \implies & \alpha_{ij} \geq \Tilde{\alpha}_{ij} \geq \eps(D) 
    \end{align*}
    For general lower curvature bound $\kappa < 0$, we conclude that $\alpha_{ij} \geq \tilde{\alpha}_{ij} \geq \eps(D, \kappa)$. Following the same procedure, the balls $B_{\frac{\eps(D, \lambda)}{2}}(v_i)$ are disjoint in $S^{n - 1}$. These $m$ vectors $v_1, \dots, v_m$ with pairwise angles $\geq \eps(D, \kappa)$. Then, we can conclude that 
    \begin{equation}\label{eq: explicit C(D, k, n)}
        m \leq \frac{\vol(S^{n - 1})}{\vol(B_{\frac{\eps(D, \kappa)}{2}}(v))} = C(D, \kappa, n)
    \end{equation}
\end{proof}
\begin{remark}
Notice  $\Gamma$ is finitely generated if there is  $m$ such that $\Gamma_m = \Gamma$. This is always the case for compact manifolds as Gromov's theorem implies since every compact Riemannian manifold has a finite diameter and satisfies some lower curvature bound (If $(M^n, g)$ is compact then this is a $\kappa \in \R$ such that $\sec_M \geq K$ \footnote{This is not a geometric argument at all. The infimum of a continuous function over a compact set can always be attained.}). It can also be seen in a more elementary way as follows.  Recall that if $M$ is compact and $\diam(M) \leq D$, then $\pi_1(M, p)$ is generated by loops of length $\leq 2D$. And hence $\Gamma = \Gamma(2D)$, i.e. all $\gamma_i$ has length at most $2D$ by Lemma~\ref{lem: loop length bound.}. In other words, in the construction of $\gamma_1, \gamma_2, \dots, \gamma_n$, we only consider the loop of length $\leq 2D$. Since $\overline{B_{2D}(\Tilde{p})}$ is compact, then there are only finitely many elements of $\overline{\Gamma}$ such that $\gamma(p) \in B_{2D}(p)$. Thus we have a finite short basis. However, if $M$ is not compact then a short basis need not be finite.
\end{remark}
\begin{corollary}
    $\sect_M \geq 0 \implies \beta_1 = \rank(H_1) \leq C(n)$. If $M^n$ admits $\sect_M \geq \kappa$, $\diam(M) \leq D \implies \beta_1 \leq C(n, \kappa, D)$. This is because
    \begin{align*}
        & H_1(M) = \pi_1(M)/[\pi_1(M), \pi_1(M)]\text{(Abelinazation)} \\ \implies & \text{$\rank(H_1) \leq $ number of generators of $\pi_1$}
    \end{align*} 
\end{corollary}
Later, Gromov proved that the higher Betti numbers are also bounded by $C(n, \kappa, D)$.

\section{Grove-Shiohama Sphere theorem}
\subsection{Topological Rigidity Problem of Curvature Bound}
The study of controlling the topology of a manifold via the Riemannian structure (curvature) of a Riemannian manifold has a long history. 

Here is a trivial example followed by the Gauss-Bonnet theorem. A 2-dimensional orientable compact and simply connected Riemannian manifold $M^2$ with positive sectional curvature $K = \sect_{M^2} > 0$ (positive Gauss curvature) must be homeomorphic to a $2$-sphere. Indeed,
\begin{equation*}
	2 - 2g = \chi(M^2) = \int_{M^2}K dA > 0 \implies g < 1.
\end{equation*}
Since $M^2$ is simply connected, it must be orientable as well. Therefore, $M^2$ must be homeomorphic to $S^2$ by the classification theorem. 
\begin{remark}
	Hamilton \cite{H82} proved that any closed $3$-dimensional simply connected Riemannian manifold $M^3$ with positive sectional curvature must also be diffeomorphic to $S^3$. The Ricci flow method was first introduced in this paper. In fact, what is proved is a stronger argument: Any closed $3$-manifold that admits strictly positive Ricci curvature and also admits a metric of constant curvature. As a corollary: Any simply connected closed $3$-manifold that admits a metric of strictly positive Ricci curvature is diffeomorphic to the $3$-sphere. 
\end{remark}

As for higher dimension, one result that we have studied is Myer's theorem \ref{thm: meyer's theorem}: For Riemannian manifold $(M^n, g)$, if $\sect_{M^n} \geq \kappa > 0$ implies the Riemannian manifold compact, in particular, $\diam(M^n) \leq \frac{\pi}{\sqrt{\kappa}}$. By rescaling the Riemannian metric, we have 
\begin{equation*}
	\sect_{M^n} \geq 1 \implies \diam(M^n) \leq \pi
\end{equation*}

Since this also applies to the universal cover $\tilde M$ it follows that $\tilde M$ is compact and hence $\pi_1(M)$ is finite.

Next, the following sphere theorem, which characterized the topology of the bounded sectional curvature manifold more precisely, was posed by Rauch \cite{Ra51} and later resolved by Berger \cite{Be60} and Klingberg \cite{Kl61}.
\begin{theorem}[The Quarter Pinched Sphere Theorem]\label{thm: sphere}
	If $(M^n, g)$ is a complete, simply connected, Riemannian manifold with $ \sect_M \in (1, 4]$, then $M^n \stackrel{\text{homeo}}{\simeq} S^n$.
\end{theorem}
\begin{remark}
	The proof of the classical sphere theorem \ref{thm: sphere} in the compact setting can also be found in do Carmo's book \cite{dCa}.  
\end{remark}
\begin{remark}
	If we replace the assumption by $\sect_M \in [1, 4]$, the classical sphere theorem \ref{thm: sphere} fails. Here is a counter-example (See Exercise 12, Chap. 8 in \cite{dCa}): The sectional curvature of complex projective space $\C P^n$, $n > 1$ lies in the interval $[1, 4]$ but is not homeomorphic to a sphere. In fact, it was proved by Berger \cite{Be60}, that if $\sect_M \in [1, 4]$, then either 
	\begin{itemize}
		\item $\diam(M^n) > \pi$ and $M^n \stackrel{\text{homeo}}{\simeq} S^n$, or
		\item $\diam(M^n) = \pi$ and $M^n$ is isometric to a rank $1$ symmetric space (See \cite{CE75})
	\end{itemize}
\end{remark}

In this section, we are going to study, in detail, the sphere theorem due to Grove and Shiohama \cite{GS77}, in which they replacing the upper bound of the sectional curvature with a lower diameter bound.

\begin{theorem}[Grove-Shiohama Sphere Theorem]\label{thm: Grove-Shiohama sphere}
    If $(M^n, g)$ is a complete Riemannian manifold with $\sect_M \in [1, \infty)$ and $\diam(M) > \frac{\pi}{2}$, then $M^n \stackrel{\text{homeo}}{\simeq} S^n$.
\end{theorem}
\begin{remark}
	In 1993, Grove and Peterson \cite{GP93} generalized Theorem \ref{thm: Grove-Shiohama sphere} to Alexandrov spaces, which is going to be introduced in Chapter \ref{Ch13} of this lecture notes. In this generality to get a sphere theorem one needs to replace the diameter lower bound with a stronger assumption of a lower radius bound. Given a metric space $(X, d)$, the radius is defined by 
	\begin{equation*}
		\textbf{rad}(X) = \inf{\br{\text{R > 0: $B_R(x)=X$ for some $x \in X$}}}
	\end{equation*}
	It's easy to see that $\diam(X)/2\le \rad(X)\le \diam (X)$.
	Therefore, the assumption on $\textbf{rad}(X) > \frac{\pi}{2}$ is stronger than $\diam(M) > \frac{\pi}{2}$. Grove and Petersen proved that if $X^n$ is an n-dimensional Alexandrov space of $\curv\ge 1$ and $\rad>\pi/2$, then $X$ is homeomorphic to $S^n$.
\end{remark}
\begin{remark}
	For further developments on topological sphere theorem and differential sphere theorem, see surveys by Wilking \cite{Wi07B}, Brendle and Schoen \cite{BS09}.
\end{remark}
\begin{remark}
	The lower diameter bound $\diam(M^n) > \frac{\pi}{2}$, given in the assumption of Grove-Shiohama theorem \ref{thm: Grove-Shiohama sphere} is sharp. For examples, $\R P^n$, $\C P^n$ and $\h P^n$ has $\sect \geq 1$ and $\diam = \frac{\pi}{2}$ but they are not homeomorphic to $S^n$. The work of Gromoll-Grove-Wilking \cite{GG87, Wi07A} fully classify the rigidity of the manifolds when $\diam(M) = \frac{\pi}{2}$. That is, if $(M^n, g)$ is a compact manifold with $\sect_M \geq 1$ and $\diam(M) = \frac{\pi}{2}$, then either $M^n \stackrel{\text{homeo}}{\simeq} S^n$ or it is isometric to one of $\R P^n$, $\C P^n$, $\h P^n$ and $\mathbb{C}aP^2$. 
\end{remark}

\subsection{The Critical Points Theory of the Distance Function}
Our goal in this section is to prove the Grove-Shiohama sphere theorem \ref{thm: Grove-Shiohama sphere}. The key is the critical point theory of the distance function \cite{GS77, Gr93}. 
\begin{remark}
	The critical points theory is also called the Morse theory of distance functions. In the Morse theory, one relates the topology of a Riemannian manifold $M$ to the critical points of a Morse function on $M$. 
\end{remark}

Let $(M^n, g)$ be a complete manifold, $A \subseteq M$ a closed subset. Let $f = d(\cdot, A)$ the distance function to $A$. Given a $p \notin A$, $f(p) > 0$, We denote the family of unit initial vectors of the distance minimizing geodesics from $p$ to $A$ by $\Uparrow_p^A$. To clarify  $\Uparrow_p^A$ consists of directions of shortest geodesics $[p,a]$ such that $a\in A$ and $d(p,a)=d(p,A)$.

\begin{figure}[htbp]
    \centering
        \includegraphics[width=0.3\textwidth]{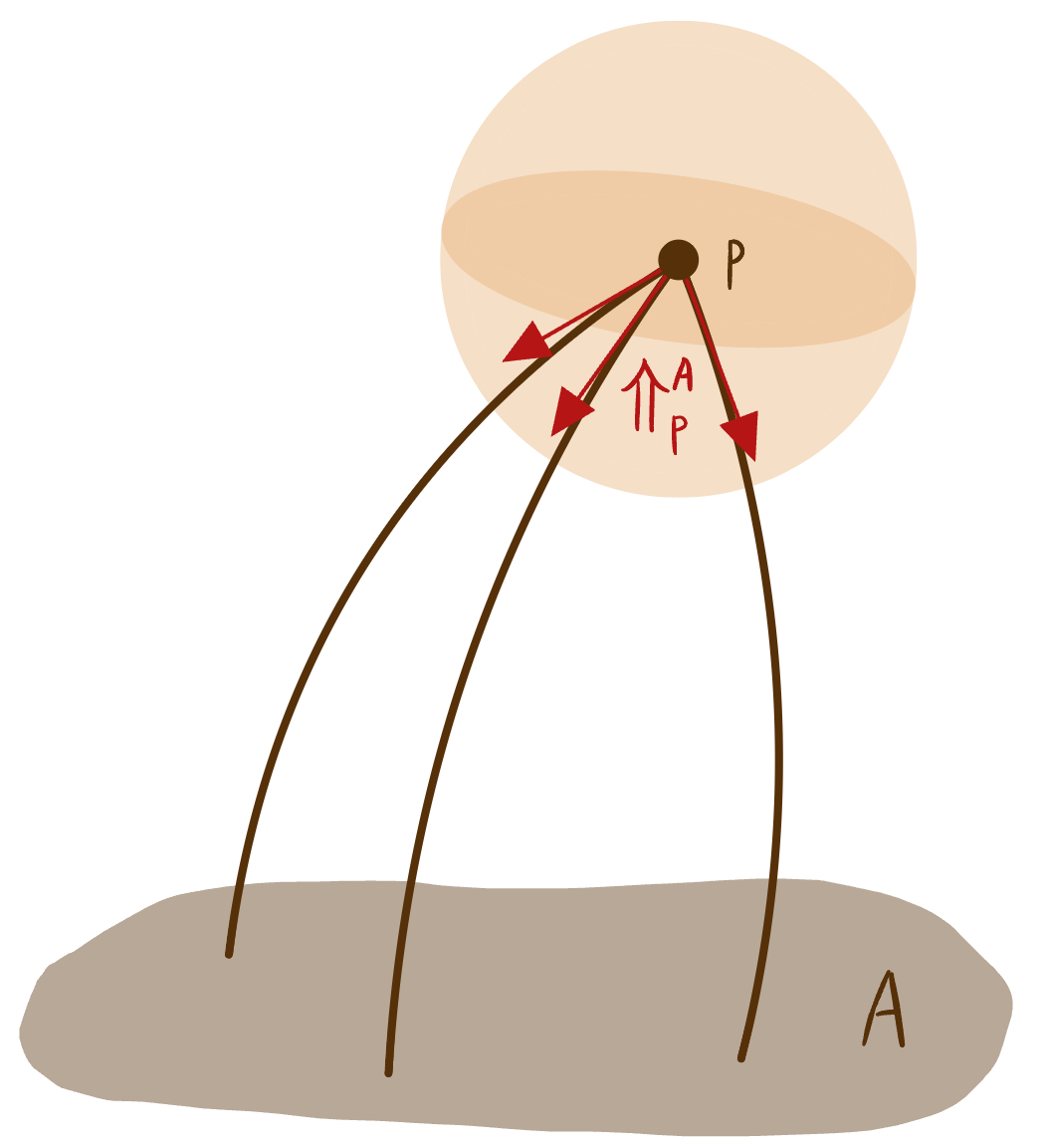}
        \caption{For each $w \in \Uparrow_p^A$, $w = \dga(0)$ for some distance minimizing geodesic $\gamma$ starting from $p$ to the closed set $A$.}
    \end{figure}
    
Firstly, we are going to introduce the first variation formula.
 \begin{theorem}[First Variation Formula]\label{thm: the first variation formula}
 	Given $(M^n, g)$ with $\sect_M \geq \kappa$, $A \subseteq M$ a closed subset, $f = d(\cdot, A)$, $p \in M \backslash A$ as above. 
 	Let $v \in T_pM$ be a unit vector at $p$. 
 	Then $f'_+(v) := \frac{df(\exp(tv))}{dt}|_{t = 0^+}$ exists and 
 	\begin{equation}\label{eq: the first variation formula}
 		df_p(v): = f_+'(v) =  -\cos{(\alpha)} 		 = \inf\br{\inp{v, u}: w \in \Uparrow_p^A}.
 		 \footnote{Taking $-\cos{(\alpha)}$ because it is monotone increasing with respect to $\alpha$ in $[0, \pi]$.}
 		 \footnote{In the first variation formula, we write $df_p(v) = f'_+(v)$. 
 		 It doesn't mean the function is differentiable. 
 		 They both represent the directional derivative in the direction of $v$.}
 	\end{equation}
 	where $\alpha$ is the smallest angle between $v$ and the shortest geodesic from $p$ to $A$, i.e. $\alpha = \inf\br{\mangle(v, w): w \in \Uparrow_p^A}$. 
 	The equation \ref{eq: the first variation formula} is called the \textbf{first variation formula of $f$ along $v$}.
 \end{theorem}
 
\begin{proof}[Proof of the First Variation Formula \ref{thm: the first variation formula}]
 The distance function $d(\cdot, a)$ is semi-concave on $M\setminus \{a\}$. Hence, the function $d(\cdot, A)$ as the infimum of the distance functions over $A$, must also be semi-concave. Hence the directional derivative $f'_+(v) $ exists. 
 
  For simplicity, assume $\sect_M \geq 0$. Denote the unit speed geodesic starting at $p$ in the direction $v$ by $\gamma(t)$. $\alpha$ is the minimum angle between $\gamma$ and the initial vector of the geodesics starting at $p$ to $A$. We denote $d = d(p, A)$ and $d(t)=d(\gamma(t), A)$. (see the following picture).

\begin{figure}[htbp]
    \centering
        \includegraphics[width=0.5\textwidth]{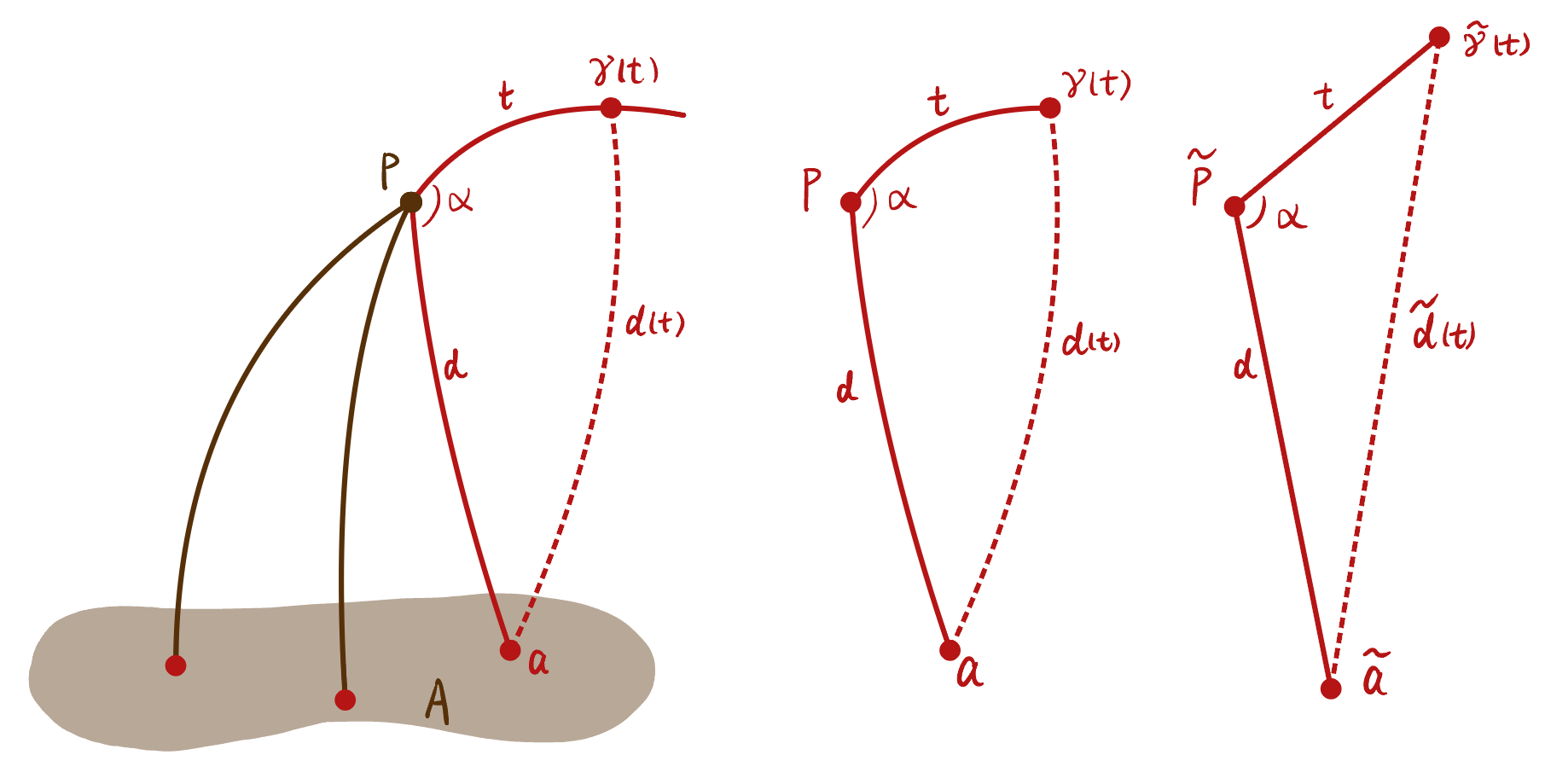}
    \end{figure}

And by the hinge comparison
\begin{equation*}
    d(t) = d(\gamma(t), A) \leq d(\gamma(t), a) \leq \tilde{d}(\tilde{\gamma}(t), \tilde{a}) =: \tilde{d}(t)
\end{equation*}
Then by the cosine law 
\begin{equation*}
	\tilde{d}(t)^2 = t^2 + d^2 - 2t\cdot d\cos{\alpha} \implies \tilde{d}(t) = \sqrt{d^2 + t^2 - 2t\cdot d\cos{\alpha}}
\end{equation*}
which can be differentiated at $t = 0$. That is
\begin{equation*}
    \tilde{d}'(t)|_{t = 0} = \frac{2t - 2d\cos{\alpha}}{2\sqrt{d^2}}\Big|_{t = 0} = -\cos{\alpha}
\end{equation*}
Therefore, we can conclude that $d(t) \leq d - t\cos{\alpha} + ct^2$, where $c = c(d)$ a constant depends on $d$. Thus
\begin{equation*}
 	d(t)=  d(\gamma(t), A) \leq d(\gamma(t), a) \leq d - t\cos{\alpha} + ct^2
\end{equation*}
for $t$ near $0$, more precisely $t \leq \eps$ for some $\eps = \eps(d)$. We can make $\eps(d)$ more precise. Recall that $d = d(p, A) = d(\gamma(t), A)$. Let $f(t) = d(\gamma(t), a)$, then 
\begin{equation*}
    f_+'(0) \leq -\cos{\alpha}\end{equation*}
This estimate follows the hinge version of Toponogov which gives that $d(t)\le \tilde{d}(t)$ for all small $t$.

\begin{figure}[htbp]
    \centering
        \includegraphics[width=0.5\textwidth]{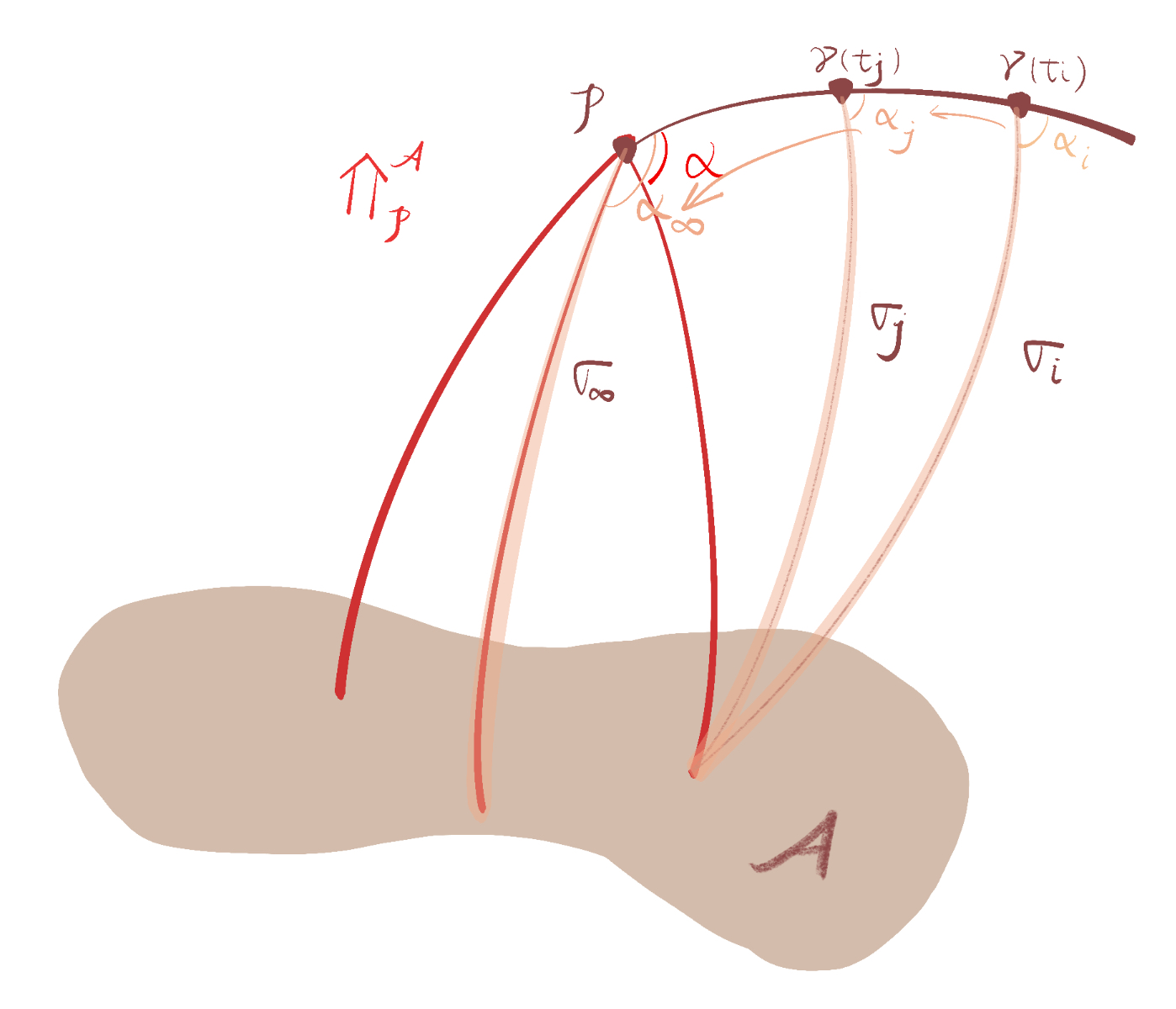}
    \end{figure}

Next, we are going to show $f_+'(0) \geq -\cos{\alpha}$ so that we can conclude thet $f_+'(0) = -\cos{\alpha}$. Let $t_i \to 0^+$ and look at $d(\gamma(t_i), A)$ where $\alpha_i$ is the angle between $\gamma(t)$ and the shortest geodesic $\sigma_i$ from $\gamma(t_i)$ to $A$. Then $\sigma_i \to \sigma$ where $\sigma$ is some shortest geodesic from $p$ to $A$. And $\alpha_i \to \alpha_\infty$. Note that by the definition of $\alpha$, we have that $\alpha\le \alpha_\infty$.
By the same argument as above, 
\begin{align*}
    d &\leq d_i - \cos(\pi - \alpha_i)t_i + ct_i^2\\
    & = d_i + \cos(\alpha_i)t_i + ct_i^2\\
    \implies & d_i \geq d - \cos{\alpha_i}t + ct^2 \quad\text{for large $i$}
\end{align*}
Therefore $\frac{d_i - d}{t_i} \geq -\cos{\alpha_i} + ct_i$. Then taking  $t_i \to 0$ we conclude that $f'(0)_+ \geq -\cos{\alpha_\infty}$. Since $\alpha_\infty\ge \alpha$ this implies that  $f'(0)_+ \geq -\cos{\alpha}$. 
 \end{proof}
\begin{definition}[Regular/Singular Points]
Let $(M^n, g)$ be complete. Let $A \subseteq M$ be a closed subset. Let $f: M \to \R$ be given by $f= d(\cdot, A)$. 

Now we want to give the definition of the critical points of the distance function. A point $p \in M\backslash A$ is called a \textbf{regular point of $f$} if there exists $v \in T_pM$ such that $df_p(v) > 0$. If no such point vector exists, then $p$ is called a \textbf{singular or critical point of $f$}.

$R \geq 0$ is a \textbf{regular value of $f = d(\cdot, A)$} if the level set $\br{f = R}$ is a set of regular points. Otherwise, $R$ is a \textbf{singular/critical value}.
\end{definition}

\begin{remark}
    This definition for singular and regular points can also be applied to general semi-concave functions. 
\end{remark}
\begin{remark}
	If a point $p \in M\backslash A$ is a regular point of $f$, by the first variation formula, we know that there exists a unit $v \in T_pM$,  $df_p(v) > 0 \iff -\cos{(\alpha)} > 0 \implies \alpha > \frac{\pi}{2}$. That means $\mangle{(v, u)} > \frac{\pi}{2}$ for any $u \in \Uparrow_p^A$. If $p$ is singular, then for any unit $v \in T_pM$, we can always find a vector $u \in \Uparrow_p^A$ such that the angle between $u$ and $v$ is less than or equal to $\frac{\pi}{2}$. 
	
	\begin{figure}[htbp]
    \centering
        \includegraphics[width=0.5\textwidth]{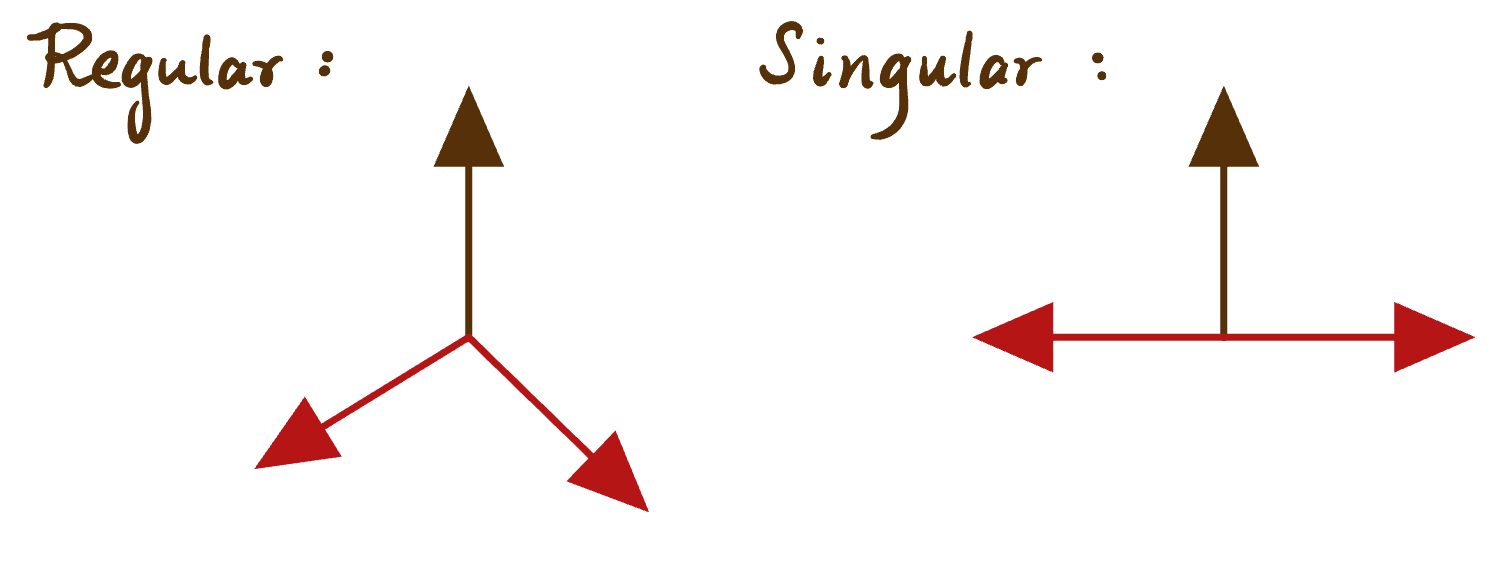}
    \end{figure}
\end{remark}
\begin{example}
    Consider $M = S \times \R$. Let $A = \br{(a, 0)}$. Then all points of $L = \cut{(\br{A})}$ except for $(-a, 0)$ are regular.  
    \begin{figure}[htbp]
    \centering
        \includegraphics[width=0.4\textwidth]{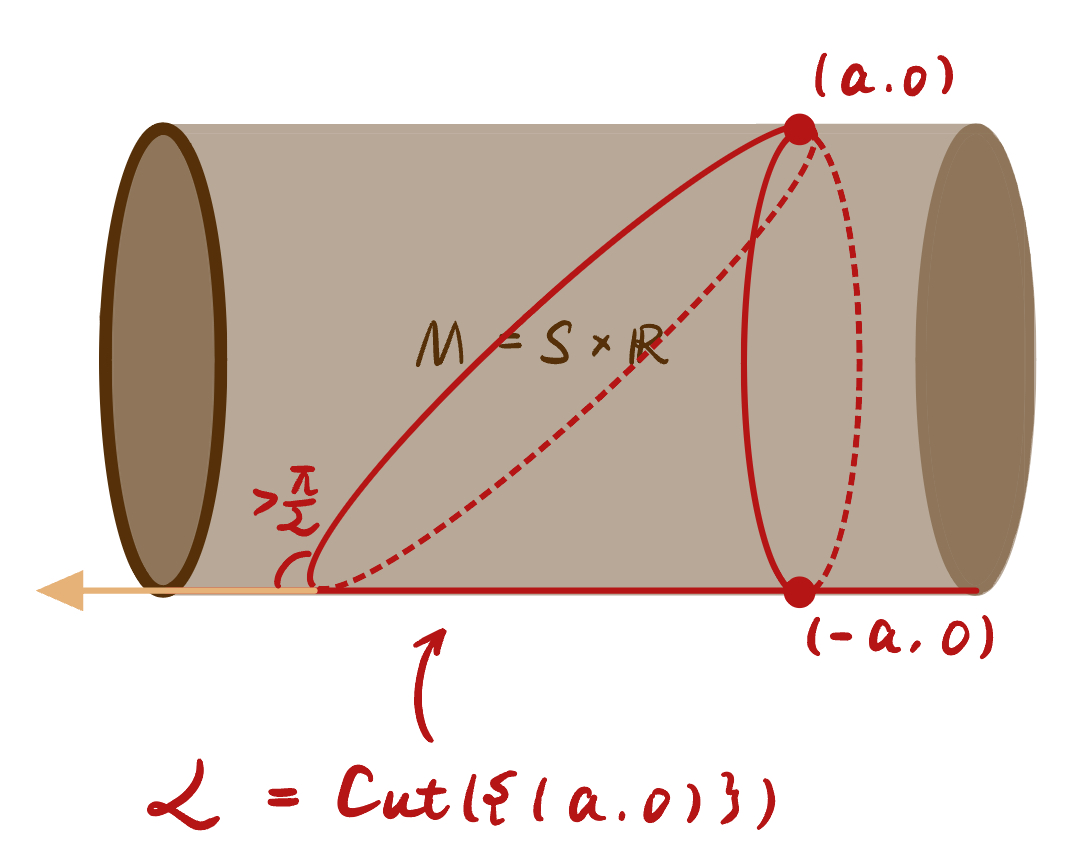}
    \end{figure}
    Note that $d(\cdot, p)$ must be regular outside of the cut locus of a point $p$. But this example shows that it can also be regular at some points of cut locus.
\end{example}

\begin{example}
	If $\Uparrow_p^A$ consists of only a single vector corresponding to the unique shortest geodesic from $p$ to $A$ then $p$ is regular. This is because we can take $v = -u$ so that by the first variation formula, we have $df_p(v) > 0$.
	
	\begin{figure}[htbp]
    \centering
        \includegraphics[width=0.4\textwidth]{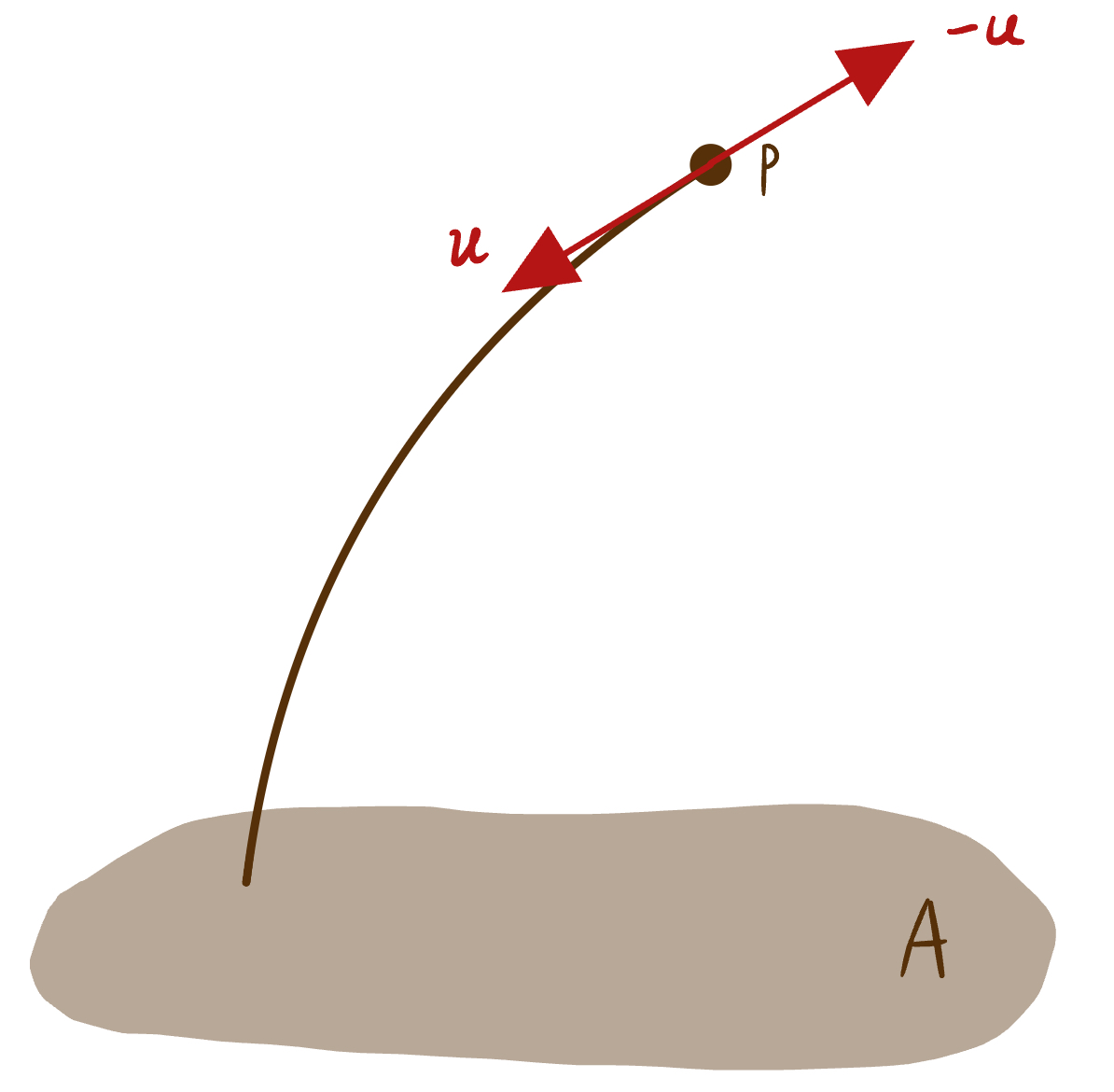}
    \end{figure}
    
\end{example}

\begin{example}
    Consider $T^2 = \R^2/\Z^2$. 
    
    \begin{figure}[htbp]
    \centering
        \includegraphics[width=0.4\textwidth]{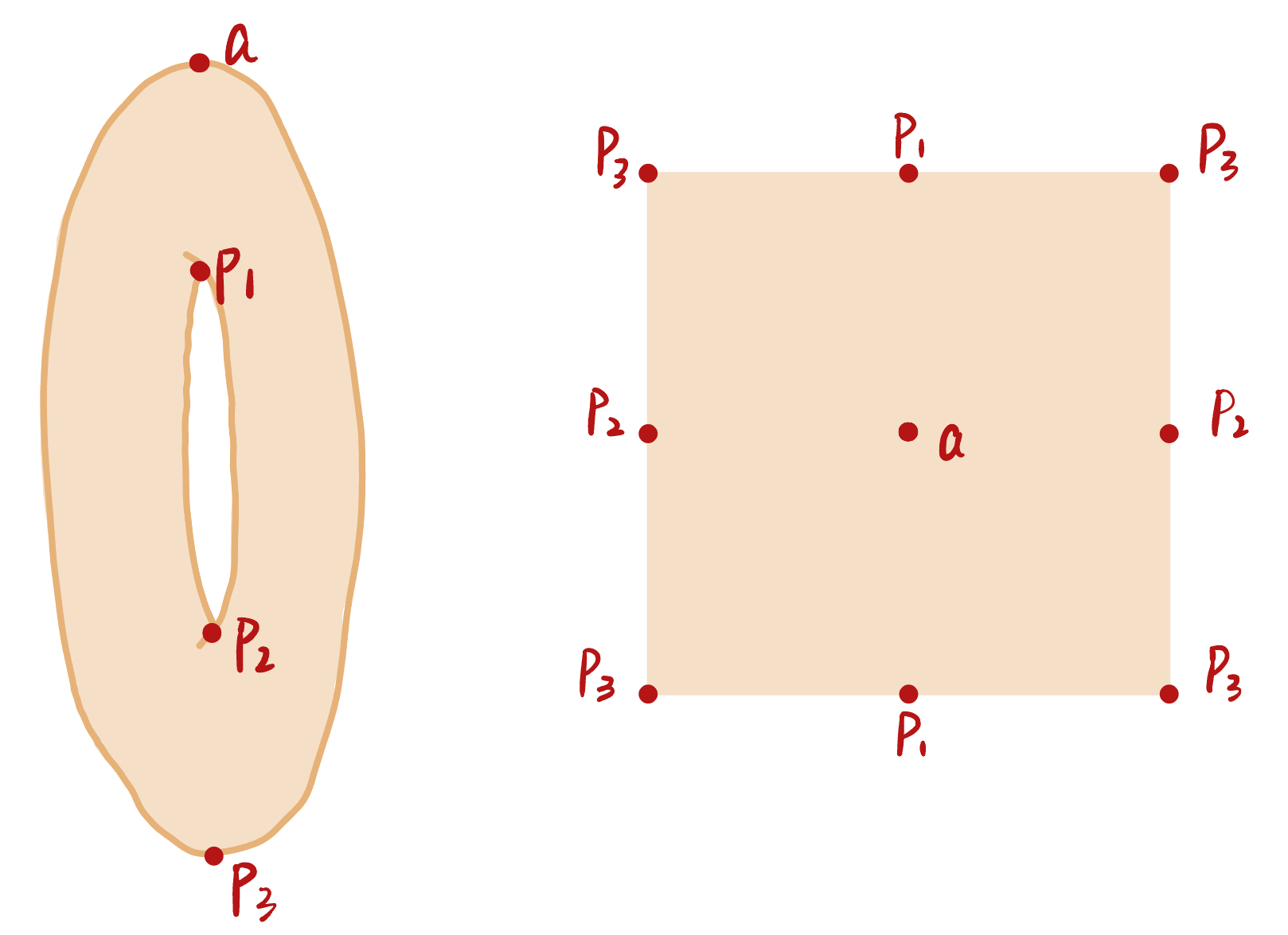}
    \end{figure}
     Let $a$ is the center of the square, $f = d(\cdot, \br{a})$ on $T^2$. 
     Then, $q_1, q_2, q_3$ are critical points. 
     The other points are regular. 
     This is easy to verify by drawing the following pictures.
    \begin{figure}[htbp]
    \centering
        \includegraphics[width=0.8\textwidth]{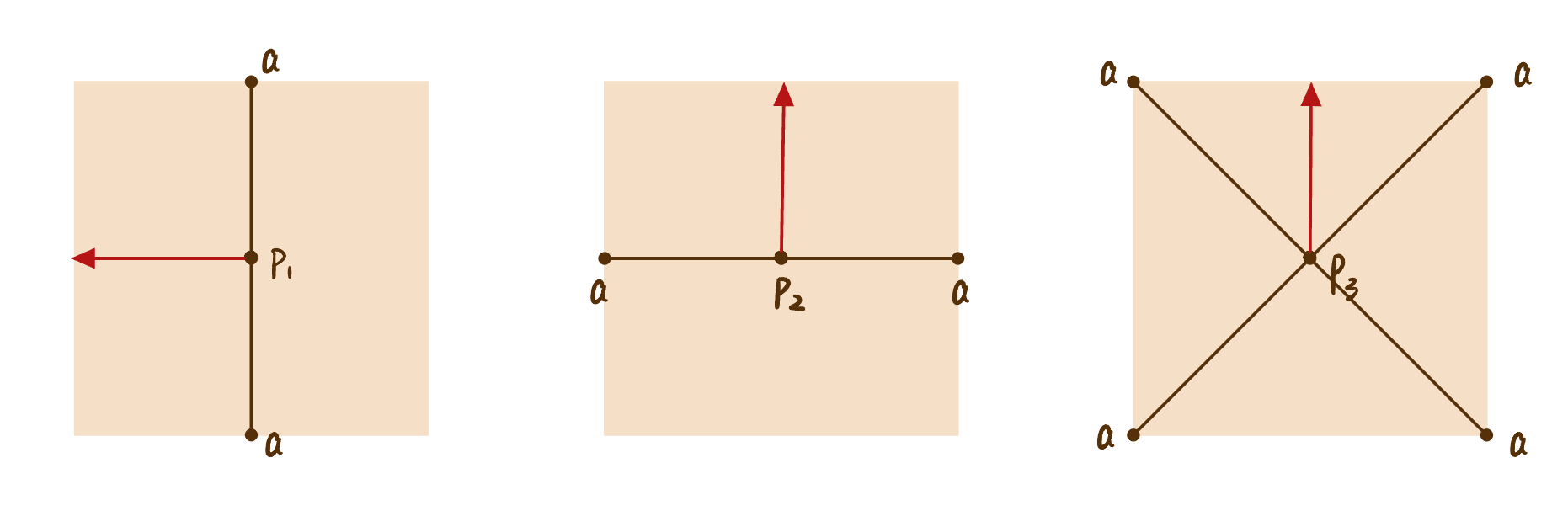}
    \end{figure}
\end{example}

\begin{definition}
Let $(M^n, g)$ be complete, $A \subseteq M$ a closed subset. 
Denote $f = d(\cdot, A)$ and consider an open subset $U \subseteq M\backslash A$ such that $f$ is regular on $U$.
A $C^\infty$-vector field $V$ on $U$ is called \textbf{gradient-like for $f$} if $df_p(V_p) > 0$ for any $p \in U$.
\end{definition}

\begin{proposition}\label{prop: gradient-like extension}
    Let $f$ be regular on an open subset $U \subseteq M$. 
    Then there exists a smooth gradient-like vector field $V$ for $f$ on $U$.
\end{proposition}
\begin{proof}
    Firstly, we want to prove the gradient-like vector field exists locally. 
    From the classical theory of differentiable manifold, we know that given a vector $v \in T_pM$ for some $p \in M$, we can always extend $v$ to a smooth global vector field $V$ on $M$ such that $V_p = v$. Based on this fact, we claim that when $p$ is a regular point so that $df_p(v) > 0$, we can find a sufficiently small neighborhood $U_p$ of $p$ such that the restriction of $V|_{U_p}$ is a gradient-like vector field. 
    
    Suppose not, then we can find a sequence of points $p_i \to p$ such that $df_{p_i}(V_{p_i}) \leq 0$. 
    \begin{figure}[htbp]
    \centering
        \includegraphics[width=0.4\textwidth]{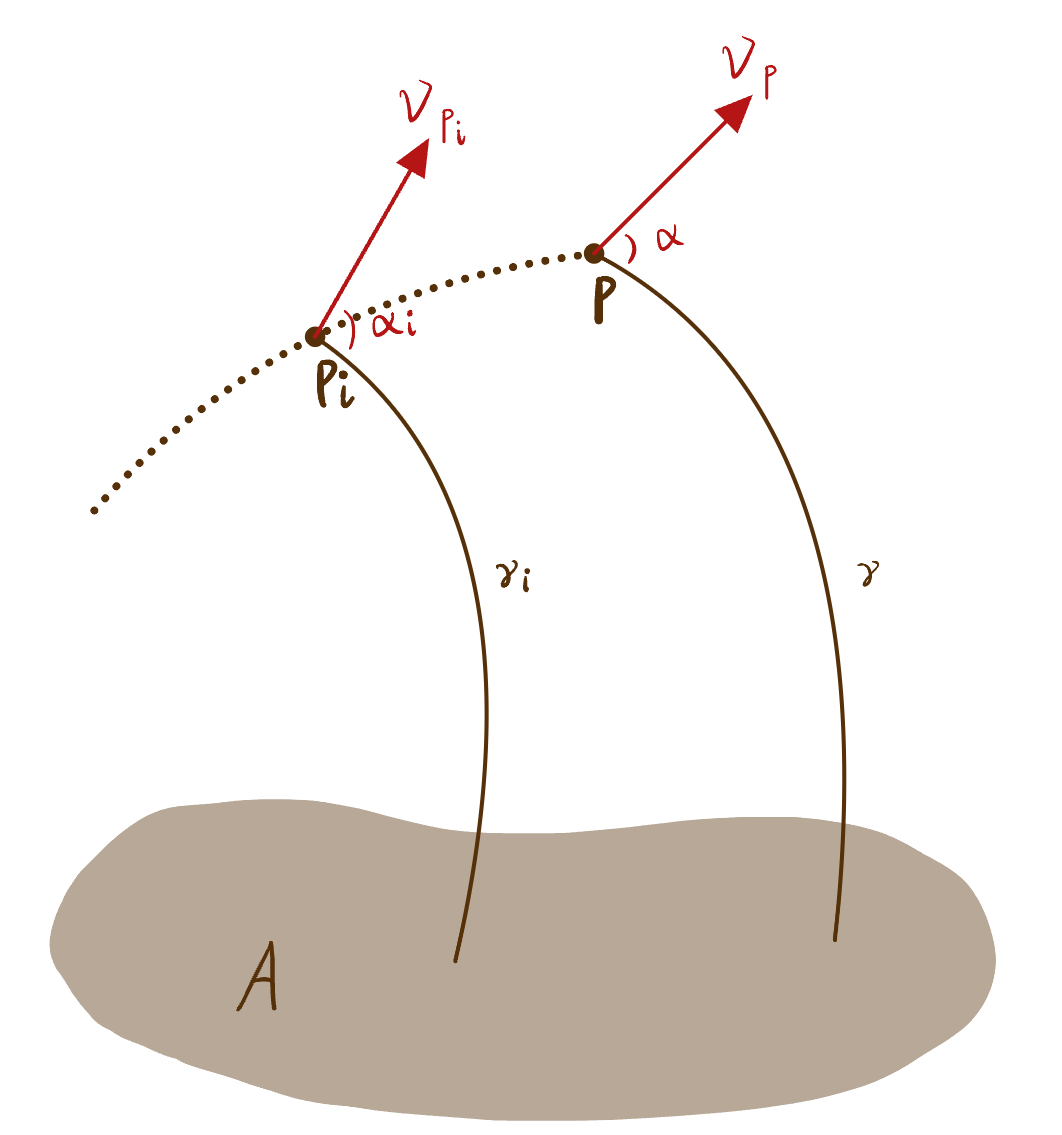}
    \end{figure} 
    And for each index $i$, there exists shortest geodesic $\gamma_i$ from $p_i$ to $A$ such that $\alpha_i = \mangle{(V_{p_i}, \dga_{i}(0))} \leq \frac{\pi}{2}$. Since the initial vector $\dga_i(0)$ all lies in a unit sphere, which is certainly compact, then we can extract sub-sequences from $\br{V_{p_i}}$ and $\br{\dga_i(0)}$ such that 
    \begin{align*}
    	& V_{p_{ij}} \to V_p\\
    	& \dga_{ij}(0) \to \dga(0)
    \end{align*}
    And this implies the sub-convergence of the angles:
    \begin{equation*}
    	\frac{\pi}{2}\geq \alpha_{ij} \to \alpha \leq \frac{\pi}{2}
    \end{equation*}
  Also as a limit of shortest geodesic from $p_i$ to$A$  $\gamma$ is the shortest geodesic from $p$ to $A$.
    
    Therefore, $p$ is also critical. And this is a contradiction. Therefore, we know that
    \begin{equation*}
        df_q(V_q) > 0\quad\text{for $q$ near $p$}.
    \end{equation*}
    Since $f$ is regular on $U$ we can do this near any point in $U$. Hence we can cover $U$ by open sets $ U_\alpha$ such that each $U_\alpha$ $f$ admits a gradient-like vector field $V_\alpha$.

    Then we can glue them into a global vector field $V$ on $U$ using partition of unity $\phi_\alpha$ subordinate to this cover $\br{U_\alpha}$ of $U$. 
    And take $V = \sum_{\alpha}\phi_\alpha V_\alpha$ with $\phi_\alpha \geq 0$ and $\sum_{\alpha}\phi_\alpha \equiv 1$ on $U$. 
    And we know that $\supp \phi_\alpha\subset U_\alpha$ and locally there are only finitely many of $\phi_\alpha \neq 0$. We claim that $V$ is gradient-like for $f$. 
    Recall that $df_p$ is concave on $T_pM$  for any  $p \in U$, 
    \begin{align*}
        df_p(V) = & df_p(\sum_{\alpha}\phi_\alpha V_\alpha)\\
        \geq & \sum_{\alpha}\phi_\alpha(p) df_p(V_\alpha) \\
        \geq & 0 
    \end{align*}
 	Moreover, the last inequality is strict since there exists $\alpha$ such that $\phi_\alpha(p)>0$.
	Note that if $V$ is gradient like for $f$ on $U$, then $\frac{V}{\abs{V}}$ is also gradient like. WLOG, we can conclude that there exists a unit gradient-like vector field. 
\end{proof}
 
\begin{corollary}\label{cor: regular neighborhood}
	From the first part of the proof of Proposition~\ref{prop: gradient-like extension}, we know that if $p \in M\backslash A$ is a regular point, then we can find a sufficiently small neighborhood $U_p$ of $p$ such that every point $q \in U_p$ is regular.
\end{corollary}
From now on, we will assume $V$ is gradient-like and $\abs{V} = 1$.
\begin{lemma}\label{lem: integral curve}
    For $f = d(\cdot, A)$ and $f$ being regular on open subset $U \subseteq M \backslash A$. Let $K \subseteq U$ be compact. Then there exists $C(K) > 0$ such that $f$ increases with speed bigger than $C(K)$ along the integral curves of $V$ contained in $K$, i.e. if $\sigma(t)$ is an integral curve of $V$ and $\sigma(t) \subseteq K$, we have for $t_1 > t_2$, 
    \begin{equation*}
        f(c(t_1)) - f(c(t_2)) \geq C(K)\cdot(t_1 - t_2).
    \end{equation*}
\end{lemma}

\begin{proof}
    Recall the Rademacher theorem, since $h: I\subseteq \R \to \R$ is locally Lipschitz, $h$ is differentiable almost everywhere and 
    \begin{equation*}
        h(t_2) - h(t_1) = \int_{t_1}^{t_2}h'(t)dt
    \end{equation*}
    consider $t_1 < t_2$, $f(\sigma(t))$ Lipschitz, then 
    \begin{equation*}
        f(\sigma(t)) - f(\sigma(t)) = \int_{t_1}^{t_2}\frac{d}{dt} f(\sigma(t))dt = \int_{t_1}^{t_2}df_{\sigma(t)}(\underbrace{\sigma'(t)}_{V_{\sigma(t)}})dt 
    \end{equation*}
    We claim that $\inf{\br{df_x(V_x): x \in K}} > 0$. Suppose not, that is $I = \inf{\br{df_x(V_x): x \in K}} = 0$ (it is obvious that $I\ge 0$).
    We take $x_i \in K$ a minimizing sequence for $df_{x_i}(V_{x_i}) \to 0$. 
    By compactness of $K$, we can say that there exists $x \in K$ such that $x_i \to x \in K$ up to subsequence. 
    \begin{figure}[htbp]
    \centering
        \includegraphics[width=0.6\textwidth]{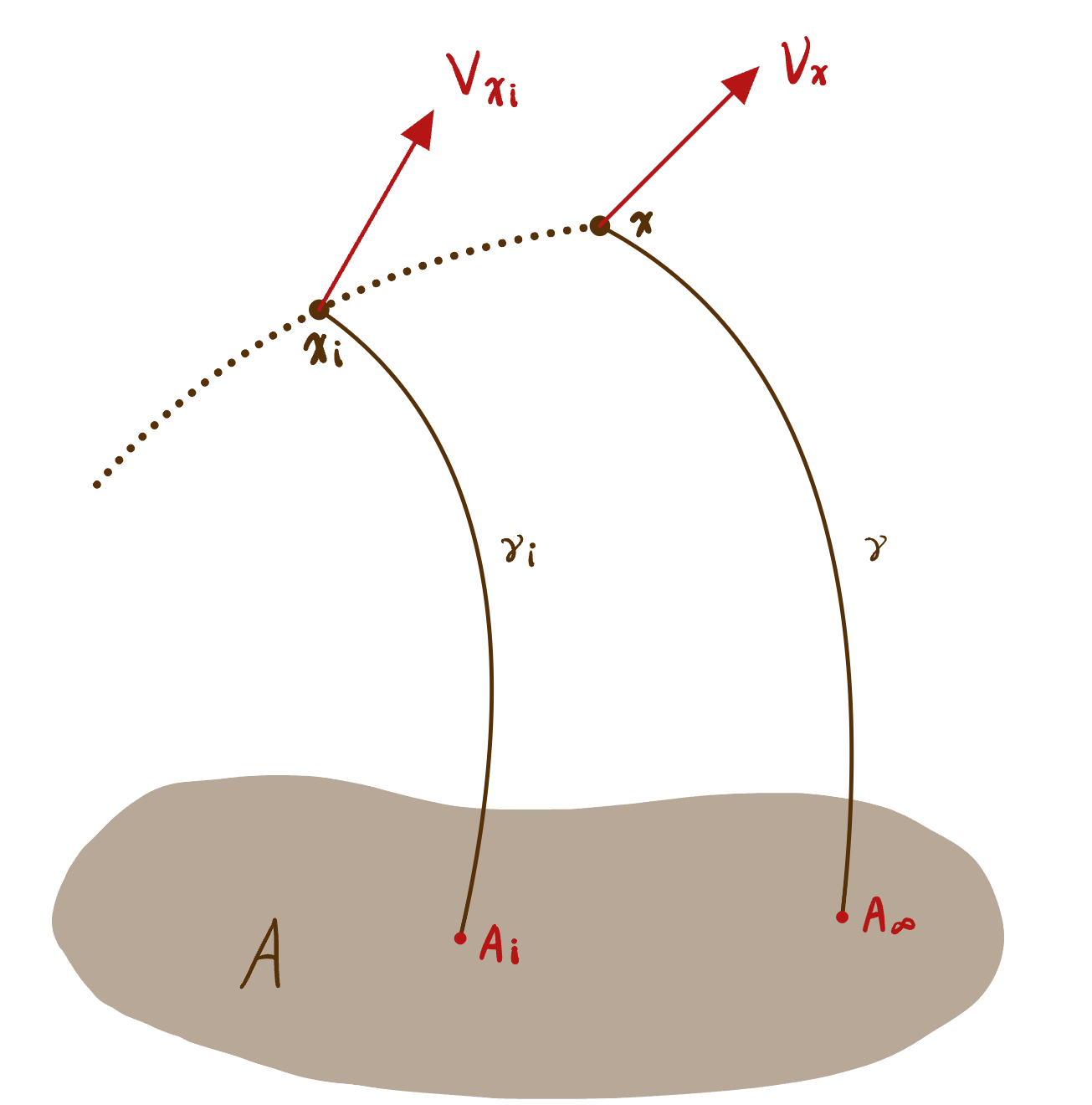}
    \end{figure}
    Since $df_{x_i}(V_{x_i}) = \inp{V_{x_i}, u_i}$, where $\gamma_i$ the shortest geodesic from $x_i$ to $A$ which touches $A$ at $A_i$ and $u_i = \dga_i(0)$. 
    Then $\inp{V_{x_i}, u_i} \to \inp{V_x, u}$. Hence $ \inp{V_x, u}=0$. This contradicts that $V$ is gradient-like for $f$ on $K$.
    
\end{proof}
\begin{lemma}\label{lem: regular level set theorem}
    Let $f = d(\cdot, A): M^n \to \R$. Let $C$ be a regular value of $f$ thus $\br{f = C}$ has no critical points. Then $\br{f = C}$ is an $(n - 1)$-dimensional topological manifold. 
\end{lemma}
\begin{proof}
By definition, we need to show that for any $p \in \br{f = C}$, we can find a neighborhood $U_p$ homeomorphic, say via $h$, to an open subset of $\R^{n-1}$.


By Proposition~\ref{prop: gradient-like extension}, we can construct a unit gradient-like vector field $V$  near $p$. Take $L^{n - 1} \subseteq M$ be a smooth submanifold transversal to $V$ near $p$, i.e. $L^{n - 1} \subseteq M$ such that $V_p \perp T_pL$. 
    \begin{figure}[htbp]
    \centering
        \includegraphics[width=0.6\textwidth]{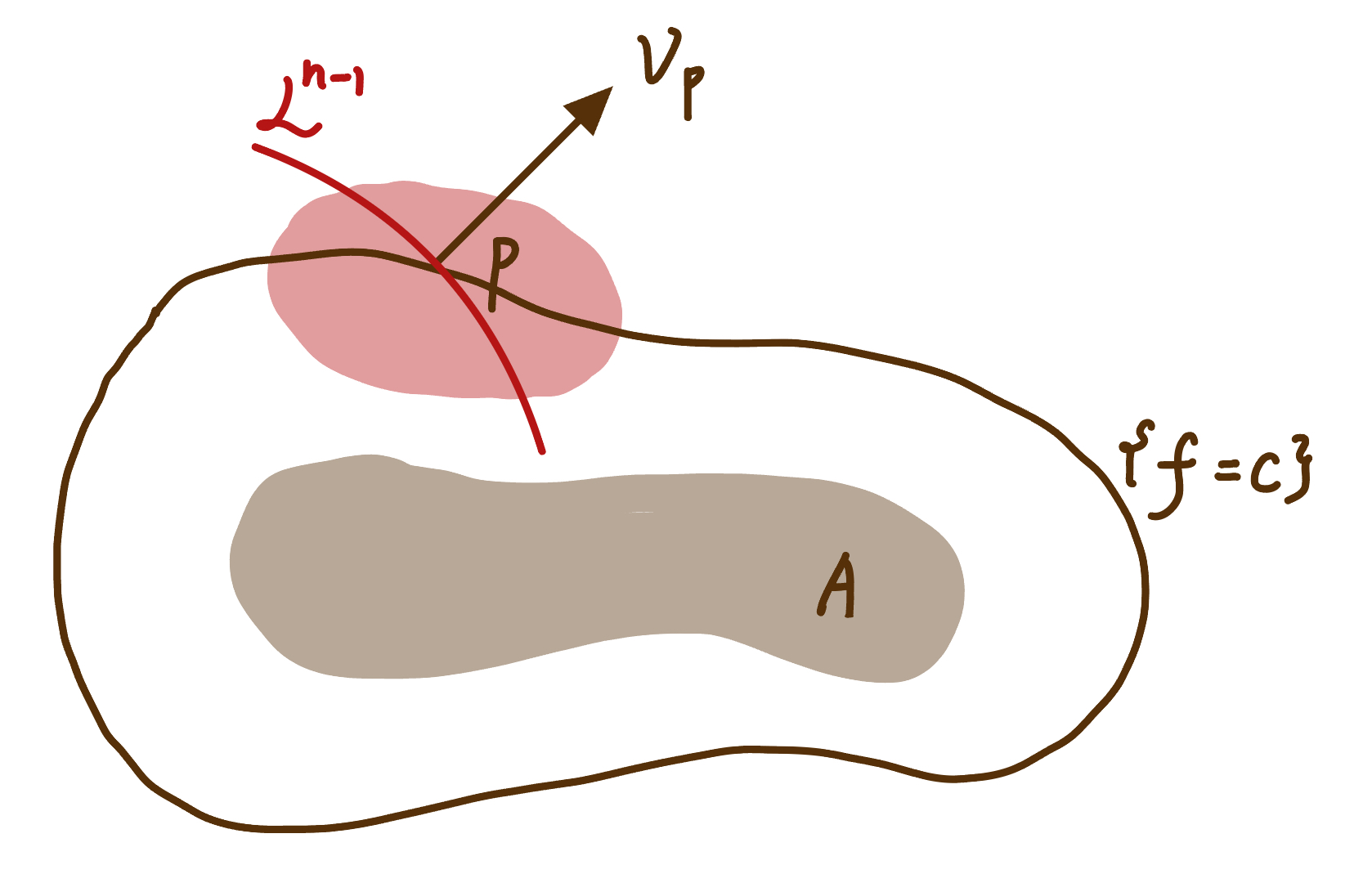}
    \end{figure}
    
    Look at a small neighborhood $\overline{B_R(p)}$ of $p$. on which $f$ is regular.
    Notice that $\overline{B_R(p)}$ is compact. 
    Then by the previous lemma \ref{lem: integral curve}, $f$ increases with speed $\geq C =: C(\overline{B_R(p)})$ along the integral curves of $V$ in $\overline{B_R(p)}$. 
   Thus, there exists $\delta > 0, 0<r\ll R$ such that 
     \begin{equation}
    f(\Phi(x, +\delta)) > C = f(p) = C > f(\Phi(x, -\delta))\, \text{ for any } x \in B_r(p)\cap L
    \end{equation}
    where $\Phi_t$ is the integral flow of $V$.

    \begin{figure}[htbp]
    \centering
        \includegraphics[width=0.6\textwidth]{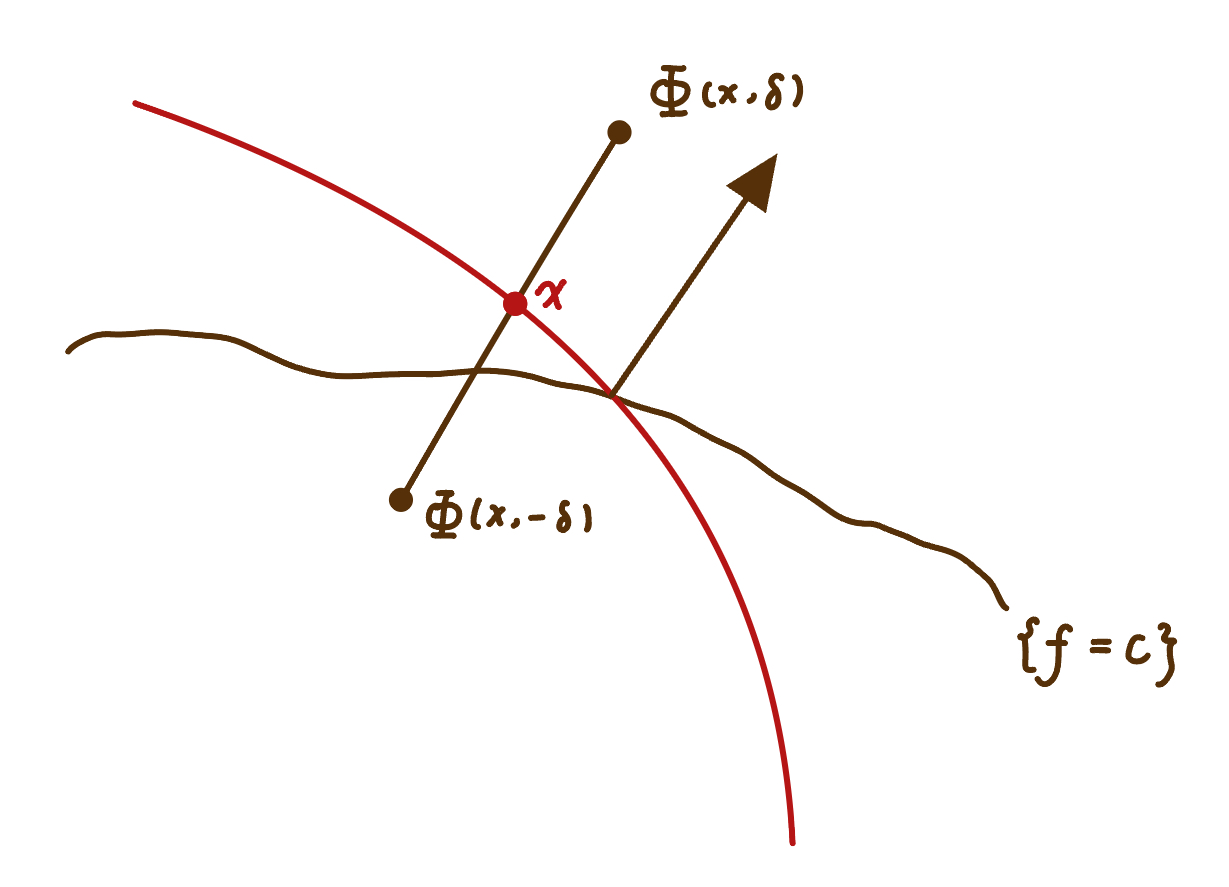}
    \end{figure}
    
    Therefore, $t \mapsto \Phi(x, t)$ has unique intersection with $\br{f = C}$, which means there exists a unique $\tau(x)$ such that $\Phi(x, \tau(x)) \in \br{f = C}$. 
    
    It is easy to see that $x \mapsto \tau(x)$ is continuous since $\Phi(x, t)$ is. Here is why: If it is not continuous at some $x$. Then  
    \begin{align*}
        &\text{$\exists x_i \to x$ such that $\tau(x_i) \to \tau_\infty \neq \tau(x)$}\\
        \implies & C = \Phi(x_i, \tau(x_i)) \to \Phi(x, \tau_\infty(x))
    \end{align*}
    However, by the uniqueness of $\tau$, we know that $\tau_\infty(x) = \tau(x)$. 
    Therefore, $\tau$ is continuous. 
    This gives us a map $L \cap B_\eps(p) \to \br{f = C}$ such that $x \mapsto \Phi(x, \tau(x))$ is a local homeomorphism near $p$. 
    Therefore, the level set $\br{f = C}$ is an $(n - 1)$-dimensional topological manifold. 
\end{proof}
\begin{theorem}[Morse Type Theorem 1]\label{thm: morse 1}
    Let $f: M\setminus A \to \R$ be given by $f = d(\cdot, A)$ or more generally, $f$ is semi-concave and Lipschitz. 
    Let $c_1 < c_2$, assume $K = \br{c_1 \leq f \leq c_2}$ is compact and $f$ is regular in $K$.
    Then 
    \begin{equation*}
    	K \stackrel{\text{homeo}}{\simeq} \br{f = c_1} \times [0, 1]
    \end{equation*}
    and
    \begin{align*}
    	\br{f = c_1}\times \br{0} \to \br{f = c_1}\\
    	\br{f = c_2}\times \br{1} \to \br{f = c_2}
    \end{align*}
    In particular,
    \begin{equation*}
        \br{f = c_1}\stackrel{\text{homeo}}{\simeq}\br{f = c_2}
    \end{equation*}
\end{theorem}
\begin{proof}
Take $V$ as a smooth unit speed gradient-like vector field on $K$. Since $f$ is regular in $K$, then there exists a constant $C(K) > 0$ such that $df_p(V_p) > C(K)$ for any $p \in K$. This means our function along the gradient curves in $K$ increases with the speed of at least $C(K)$. 
We claim that the for any $x \in K$, there exists a continuous function $\tau$ such that $\phi_{\tau(x)}(x) \in \br{f = c_2}$

\begin{figure}[htbp]
    \centering
        \includegraphics[width=0.4\textwidth]{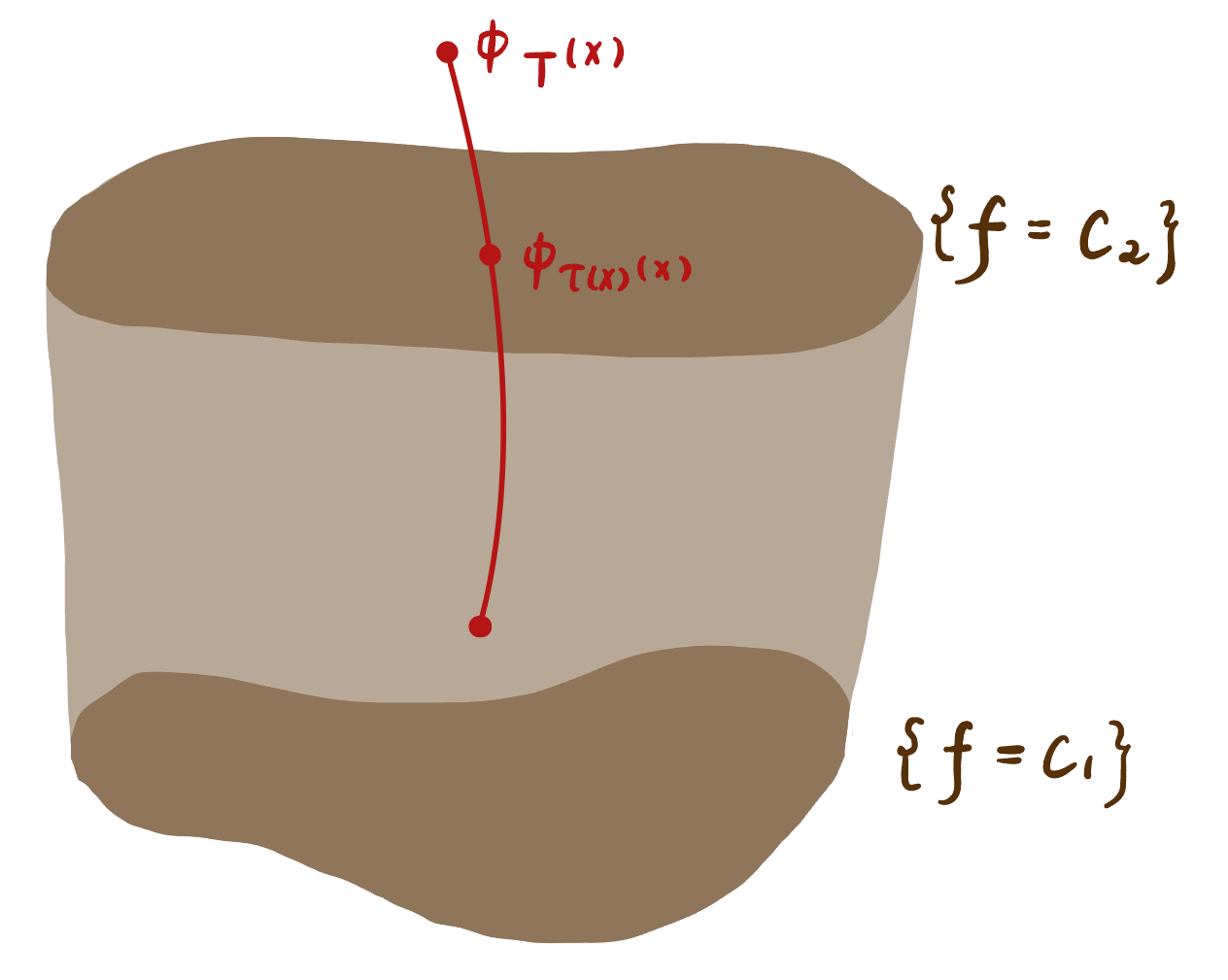}
    \end{figure} 

Firstly we can find a $T > 0$ such that we can flow any $x \in K$ by the gradient flow $\phi_T(x)$ to somewhere beyond the level set $\br{f = c_2}$, i.e $f(\phi_T(x)) \geq c_2$. 
This is possible since $f(\phi_t(x)) \geq f(x) + C(K)t$ for $\phi_t(x) \in K$ by the lemma \ref{lem: integral curve}. 
So we can take $T = \frac{c_2 - c_1}{C(K)}$. 
Suppose $f(\phi_T(x)) < c_2$, then $f(\phi_T(x)) - f(x) \geq T\cdot C(K) = c_2 - c_1$. 
As $f(x) \geq c_1$, we can conclude that $f(\phi_T(x)) \geq c_2$. 
Hence, there exists $\tau(x) \leq T$ such that $f(\phi_{\tau(x)}(x)) = c_2$ for any $x \in K$.
And this $\tau(x)$ is unique since $f(\phi_{\tau(x)}(x))$ is monotone increasing for $t \in [0, T]$ and $\tau(x)$ is continuous in $x$ by the same arguments as in the lemma \ref{lem: regular level set theorem}.
Now we want to construct a homeomorphism $F: \br{f = c_1}\times [0, 1] \to K$ to finish our proof. The homeomorphism can be constructed as follows, 
    \begin{equation*}
        F(x, t) = \phi_{t \cdot \tau(x)}(x)
    \end{equation*}
    In this map, $F(x, 0) = x \in \br{f = c_1}$ and $F(x, 1) = \phi_{\tau(x)}(x) \in \br{f = c_2}$. 
    
    We claim that $F$ is a homeomorphism. 
    To show that $F$ is onto, for $y \in K$, we can flow $y$ back to some point $x \in \br{f = c_1}$. 
    Thus we can find $x$. Next, we can construct the integral curve via the flow passing through $y$ at some time $t<\tau(x)$. 
    
    \begin{figure}[htbp]
    \centering
        \includegraphics[width=0.4\textwidth]{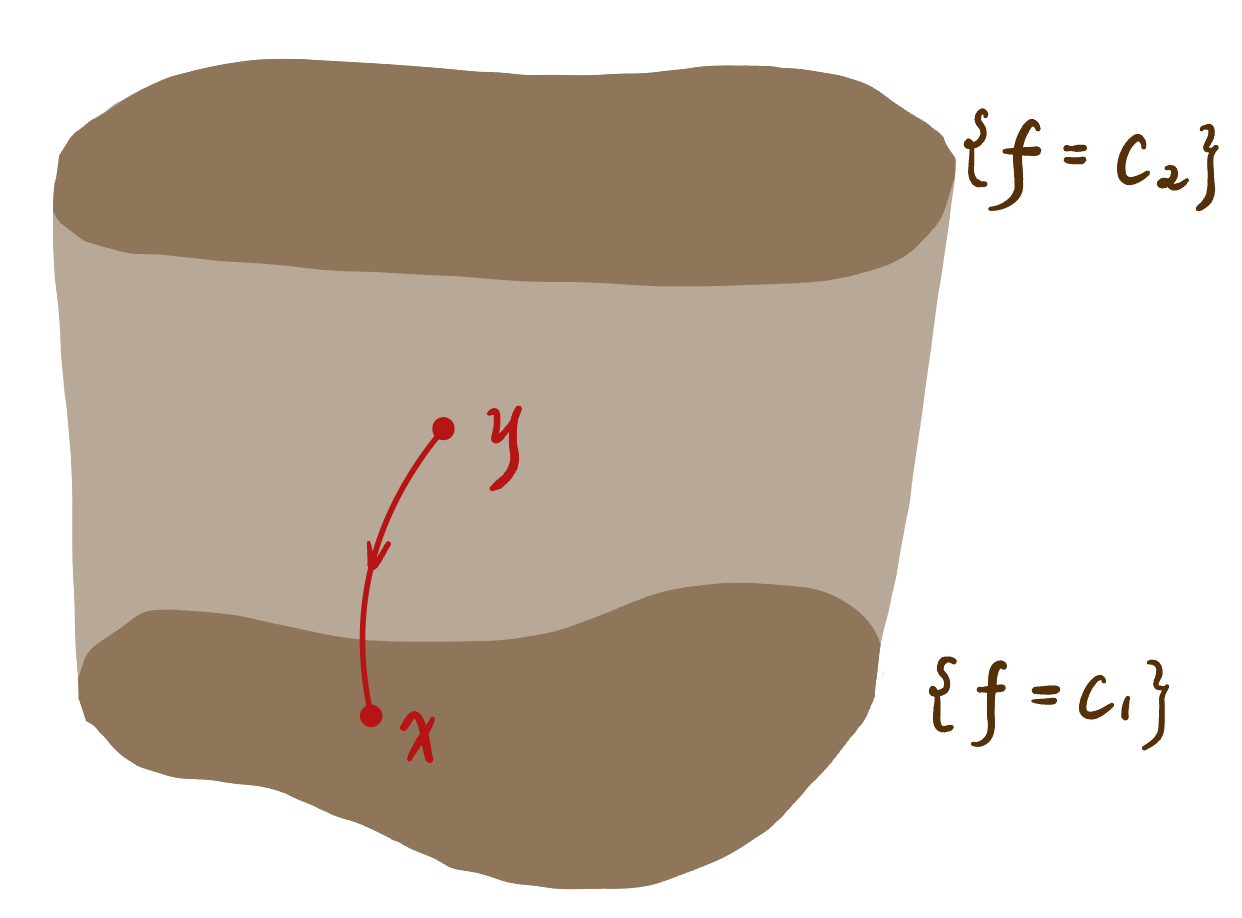}
    \end{figure}
    these spaces are Hausdorff and compact so it's enough to prove that $F$ is a continuous bijection so only injectivity needs to be checked. Indeed, suppose $F(x_1, t_1) = F(x_2, t_2)$, then it must be $x_1 = x_2$. Otherwise, the flow lines would not intersect all the way along $[t_1, t_2]$ and thus we will have $F(x_1, t_1) \neq F(x_2, t_2)$. It is also impossible that $t_1 \neq t_2$ due to the fact that $f$ is strictly monotone with respect to $t$ along the flow lines.
\end{proof}
\begin{theorem}[Morse Type Theorem 2]\label{thm: morse 2}
	Using the same proof as in Theorem~\ref{thm: morse 1}, we can also show that if $f$ is regular over $\br{f \geq R}$, then
	\begin{equation*}
		\br{f \geq R} \stackrel{\text{homeo}}{\simeq} \br{f = R} \times [0, \infty)
	\end{equation*}
	and
	\begin{equation*}
		\br{f = R} \times \br{0} \to \br{f = R}
	\end{equation*}
	If in addition $\br{f = R}$ is smooth, then $\br{f \geq R} \stackrel{\text{diffeo}}{\simeq} \br{f = R} \times [0, \infty)$.
\end{theorem}
\begin{proof}
Let us first treat the case when $ \br{f = R}$ is smooth. Let $V$ be a smooth gradient-like vector field of $f$ on  $\br{f \geq R}$. Then $ \br{f = R}$ must be transverse to $V$. Let $\phi_t$ be the flow of $V$. Then we claim that  $\phi: \br{f = R}\times [0, \infty) \to \br{f \geq R}$ is a diffeomorphism. Since $ \br{f = R}$ is transverse to $V$ the map $\phi$ is a local diffeomorphism. Hence we only need to claim that this map is bijective. The proof for $\phi$ being a bijection would be the same as the proof for Morse Theorem 1~\ref{thm: morse 1}. To show $\phi$ is surjective, that is for $y$ such that $d(A, y) =: \overline{R} > R$, $K = \br{R \leq f \leq \overline{R}}$ compact. On $K$, since $f'(v) > \text{const} > 0$, in the finite time flowing backward, we hit $\br{f = R}$ at some $x$. Therefore $\phi$ is surjective. Moreover, $\phi$ is injective by the uniqueness of $ODE$.
	
	To show the homeomorphism case, we can write
	\begin{equation*}
		\br{f \geq R} = \bigcup_{i = 0}^\infty \br{R + i \leq f \leq R + i + 1}.
	\end{equation*}
	By Morse Theorem 1~\ref{thm: morse 1}, each $\br{R + i \leq f < R + i + 1}$ can be written as the produce $\br{f = R + i} \times [0, 1)$ which is homeomorphic to $\br{f = R + i} \times [i, i + 1)$. By gluing these together, we have
	\begin{align*}
		\br{f \geq R} 
		&= \bigcup_{i = 0}^\infty \br{R + i \leq f < R + i + 1}\\
		&\stackrel{\text{homeo}}{\simeq}\bigcup_{i = 0}^\infty \br{f = R + i} \times [i, i + 1)\\
		&= \br{f = R} \times [0, \infty)
	\end{align*}
	\end{proof}

\begin{corollary}\label{cor: homeo to disk}
    If $f = d(\cdot, p)$ has no critical point in $\overline{B_R(p)}\backslash\br{p}$. Then 
    \begin{equation*}
        \overline{B_R(p)}\stackrel{\text{homeo}}{\simeq} \overline{D}^n 
    \end{equation*}
    where $\overline{D}^n$ is the closed unit disk in $\R^n$.
\end{corollary}
\begin{proof}
    The first step is to take a sufficient small $0<\eps <\inj(p)$.  Then
    \begin{equation*}
        \overline{B_\eps(p)}\stackrel{\text{diff}}{\cong} \overline{D}^n \qquad \text{ and }  \qquad   S_\eps(p) \stackrel{\text{diff}}{\cong} S^{n-1}.
    \end{equation*}
    Next, by Morse Theorem 1~\ref{thm: morse 1}, since $f = d(\cdot, p)$ has no critical point in $\overline{B_R(p)}\backslash\br{p}$. 
    We can write the annulus $\br{\eps \leq f \leq R}$ as 
    \begin{equation*}
        \br{\eps \leq f\leq R} \stackrel{\text{homeo}}{\simeq}\br{f = \eps}\times [0, 1] \stackrel{\text{homeo}}{\simeq} S^{n - 1}\times [0, 1]
    \end{equation*}
    \begin{figure}[htbp]
    \centering
        \includegraphics[width=0.7\textwidth]{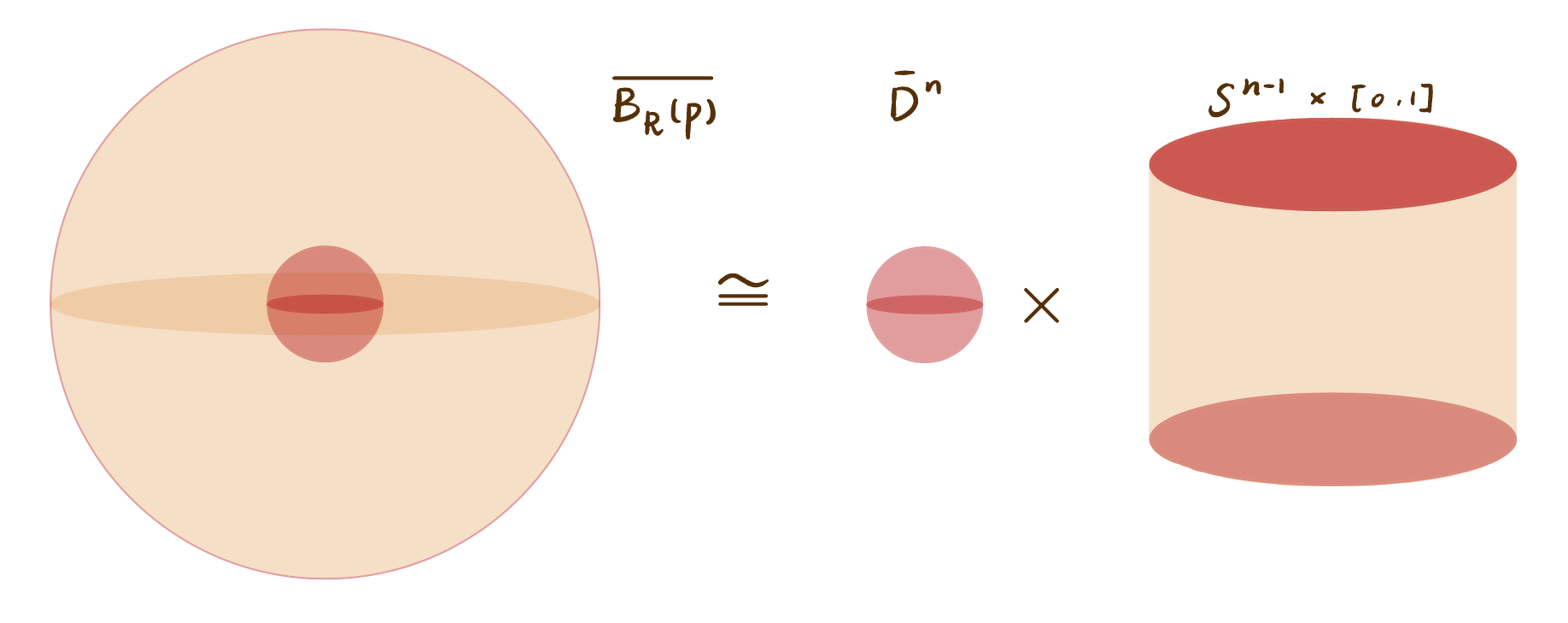}
    \end{figure} 
    
    Therefore, $\overline{B_R(p)} \simeq \overline{D}^n \cup (S^{n - 1}\times[0, 1]) \simeq \overline{D}^n$. 
    
    \end{proof}

\begin{lemma}\label{lem: R - eps}
	Let $M^n$ be a complete manifold, $A \subseteq M$ a compact subset, let $f = d(\cdot, A)$. If $R$ is a regular value of $f$, i.e. $\br{f = R}$ is a set of regular points. Then $R - \eps$ is a regular value for all sufficiently small $\eps > 0$. 
	\begin{proof}
		Suppose not, then we can construct a sequence $\br{q_i}$ from the compact set $\overline{B_R(A)}$ such that $d(q_i, \br{f = R}) \to 0$. This sequence sub-converges in $\overline{B_R(A)}$ and thus the limit point $q_{\infty}$ lies in $\br{f = R}$. Here we have a contradiction, since $q_{\infty}$ must be a regular point and the set of regular points is open by Corollary ~\ref{cor: regular neighborhood}. 
	\end{proof}
\end{lemma}
\begin{remark}
	Lemma~\ref{lem: R - eps} might not hold if $A$ is not compact. One can construct a counterexample when the radius of the regular neighborhoods for each point of $\br{f = R}$ asymptotically vanishes. 
\end{remark}

\begin{definition}[Compact Type]\label{def: compact type}
	A manifold $M$ is said to be of \textbf{compact type} if it is homeomorphic to the interior of a compact manifold with a boundary. 
\end{definition}
\begin{proposition}\label{prop: compact type}
Let $M^n$ be an open complete manifold. Let $f = d(\cdot, A)$ where $A$ is a compact subset in $M$. Let $R > 0$ be a regular value of $f$. Suppose that $f$ has no critical point outside of $\overline{B_R(A)}$. 
Then $M$ has a compact type. In fact $M \stackrel{\text{homeo}}{\simeq} (B_R(A))^\circ$.
Notice that $B_R(A)$ is a compact manifold with a boundary in this case.
\end{proposition}
\begin{proof}
   	Consider the embedded topological submanifold $N^{n - 1}\subseteq M^n$ where $N^{n - 1} := S_R(A) = \partial B_R(A) = \br{f = R}$. 
    By Lemma~\ref{lem: R - eps}, for some small $\eps>0$ $f$ has no critical point in $\br{R - \eps \leq f \leq R}$, which is compact since it's a closed subset of $\br{f \leq R}$ which is compact.
    
    Then by Morse Theorem 1~\ref{thm: morse 1}, 
    \begin{align*}
        &\br{R - \eps \leq f \leq R} \stackrel{\text{homeo}}{\simeq} N^{n - 1}\times [0, 1].\\
        &\br{R - \eps \leq f < R} \stackrel{\text{homeo}}{\simeq} N^{n - 1}\times [0, 1).
    \end{align*}
    And by Morse Theorem 2~\ref{thm: morse 2}, 
    \begin{equation*}
    	\br{f \geq R - \eps} \stackrel{\text{homeo}}{\simeq} \br{f = R - \eps} \times [0, \infty)
    \end{equation*}
        Thus, we proved that
    \begin{align*}
    	\br{f < R} 
    	&= \br{f \leq R - \eps}\cup\br{R - \eps \leq f < R}\\
    	&\stackrel{\text{homeo}}{\simeq} \br{f \leq R - \eps}\cup\br{f = R - \eps}\times[0, 1)
    \end{align*}
    and on the other hand, 
    \begin{align*}
    	M &= \br{f \leq R - \eps}\cup\br{f \geq R - \eps}\\
    	&\stackrel{\text{homeo}}{\simeq} \br{f \leq R - \eps}\cup\br{f = R - \eps}\times[0, \infty).
    \end{align*}
    And thus, $M \stackrel{\text{homeo}}{\simeq} \br{f < R}$.
\end{proof}
\begin{remark}
	Moreover, it follows from the above proof that $M$ is homotopically equivalent to $\br{f \leq R - \eps}$ and thus homotopically equivalent to $\br{f \leq R}$, which is a compact manifold with boundary. 
\end{remark}
\begin{remark}\label{rem: crit-diffeomophism}
    If in the above proposition $\br{f = R} = N^{n - 1}\subseteq M^n$ is a smooth submanifold, then 
    \begin{equation*}
        \Phi: N \times [0, \infty) \to \br{f \geq R}
    \end{equation*}
    is a diffeomorphism. In this case, $M \stackrel{\text{diffeo}}{\simeq} \br{f \leq R}\cup  N\times [0, \infty)$ where $\br{f \leq R}$ is a compact manifold with boundary. In this case, we can conclude that $M \stackrel{\text{diff}}{\cong}\br{f < R}$ by a similar argument as in  theorem~\ref{thm: morse 1} and \ref{thm: morse 2}. 
\end{remark}

\subsection{Proof of Grove-Shiohama Sphere Theorem}

    Denote $d = \diam(M) > \frac{\pi}{2}$. Let $p, q \in M$ be the points such that $d(p, q) = \diam(M)$. 
    Denote $f = d(\cdot, p)$.
    Since we assume $\sect_M \geq 1$, we consider the model space $\modsp{1} = S^n$. For $\tilde{p} \in S^n$, the function $\tilde{f} = d(\cdot, \tilde{p})$ satisfies
    \begin{equation*}
        \hess_{\tilde{f} = d(\cdot, \tilde{p})} = 
        \begin{bmatrix}
            0 & 0\\
            0 & \cot{\tilde{f}}I
        \end{bmatrix}
    \end{equation*}
    For $\frac{\pi}{2} < \alpha \leq \pi$, $\cot{\alpha} = \frac{\cos{\alpha}}{\sin{\alpha}} < 0$. 
    As a consequence, we know that $\tilde{f}$ is concave in $S^n$. 
    Then by the point-on-a-side comparison \ref{thm: point-on-a-side} and the Jensen's inequalities \ref{thm: Jensen's inequality}, $f = d(\cdot, p)$ is concave outside of $B_{\frac{\pi}{2}}(p)$.
    
    \textbf{Step 1:}

    We claim that $f = d(\cdot, p)$ has a unique maximum at $q$. 
    
    When $d = \pi$, then the uniqueness of the maximum follows since the perimeter of any triangle in $M$ is at most $2\pi$. For instance, consider $q_1, q_2$ such that $d(q_1, p) = d(q_2, p) = \pi$, then the length of the side $[q_1, q_2]$ must be zero, which implies $q_1$ and $q_2$ must be equal. 
    
    Suppose $d<\pi$ and there are two distinct maximal points $q_1, q_2 \in M$ such that $d(q_1, p) = d(q_2, p) = d$. We draw the mid-point $m$ of the segment $[q_1q_2]$. From the following picture, we can conclude that $l \geq \overline{l} > d$ since $\overline{f}$ is strictly concave along non-radial geodesics.  Therefore, $l > d$  which contradicts to the fact that $\diam(M) = d$.
    \begin{figure}[htbp]
    \centering
        \includegraphics[width=0.5\textwidth]{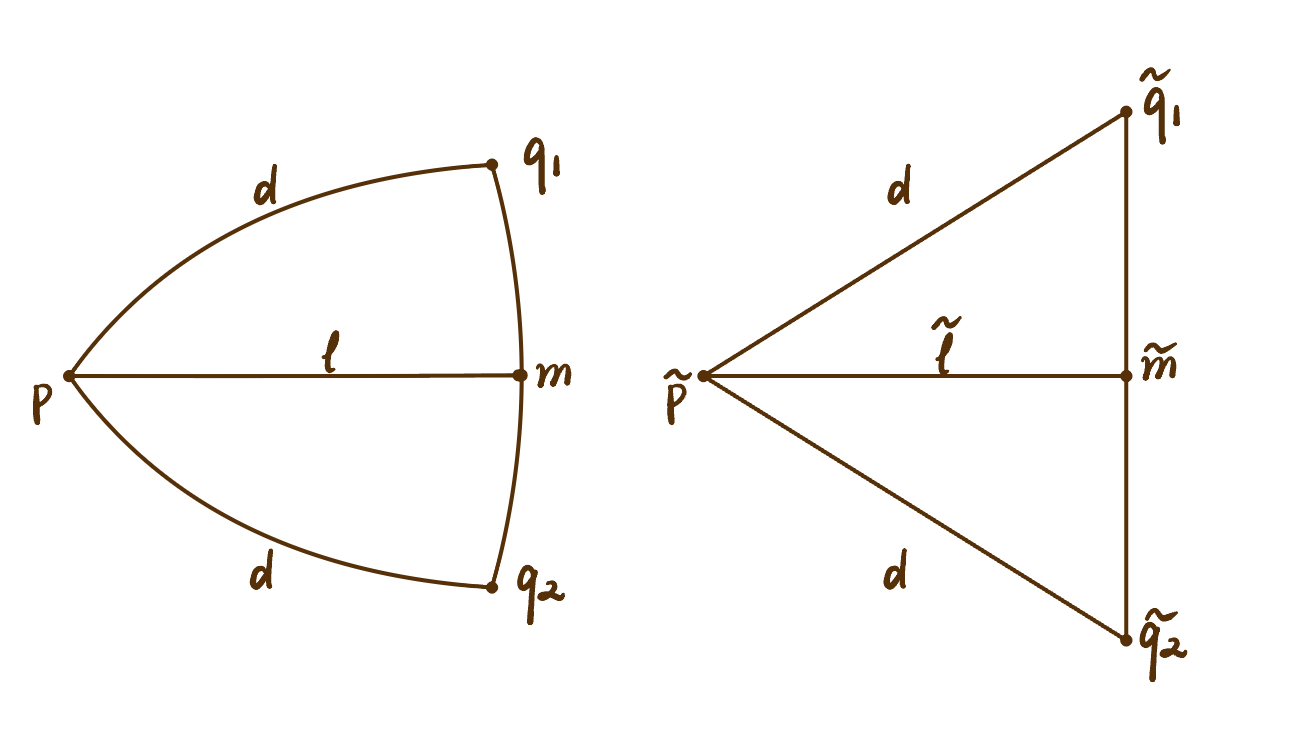}
    \end{figure}
    
    \textbf{Step 2:}
    Next, we claim that $f = d(\cdot, p)$ has no critical point outside $p$ and $q$. Suppose not, then there exists $x \in M \backslash\br{p, q}$ that is critical for $f$. Denote $\gamma_1$ be any shortest geodesic from $x$ to $p$ and $\gamma_2$ be any shortest geodesic from $x$ to $q$. 
    \begin{figure}[htbp]
    \centering
        \includegraphics[width=0.5\textwidth]{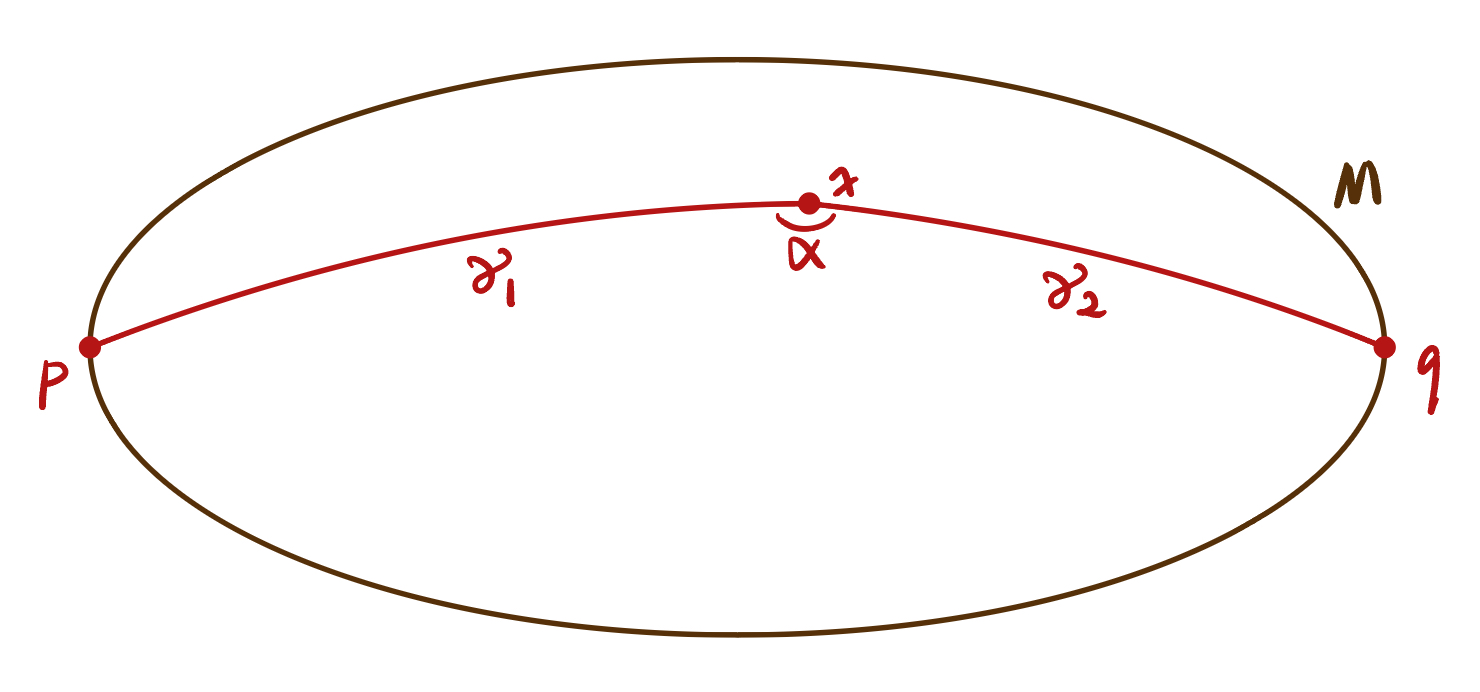}
    \end{figure}
    We claim that the angle $\alpha$, in the picture, is bigger that $\frac{\pi}{2}$. 
    We want to show $\alpha > \frac{\pi}{2}$ by cases. 
    The first case is that when $d(p, x) > \frac{\pi}{2}$, where $\gamma_2$ is a geodesic outside of $B_{\frac{\pi}{2}}(p)$. 
    \begin{figure}[htbp]
    \centering
        \includegraphics[width=0.5\textwidth]{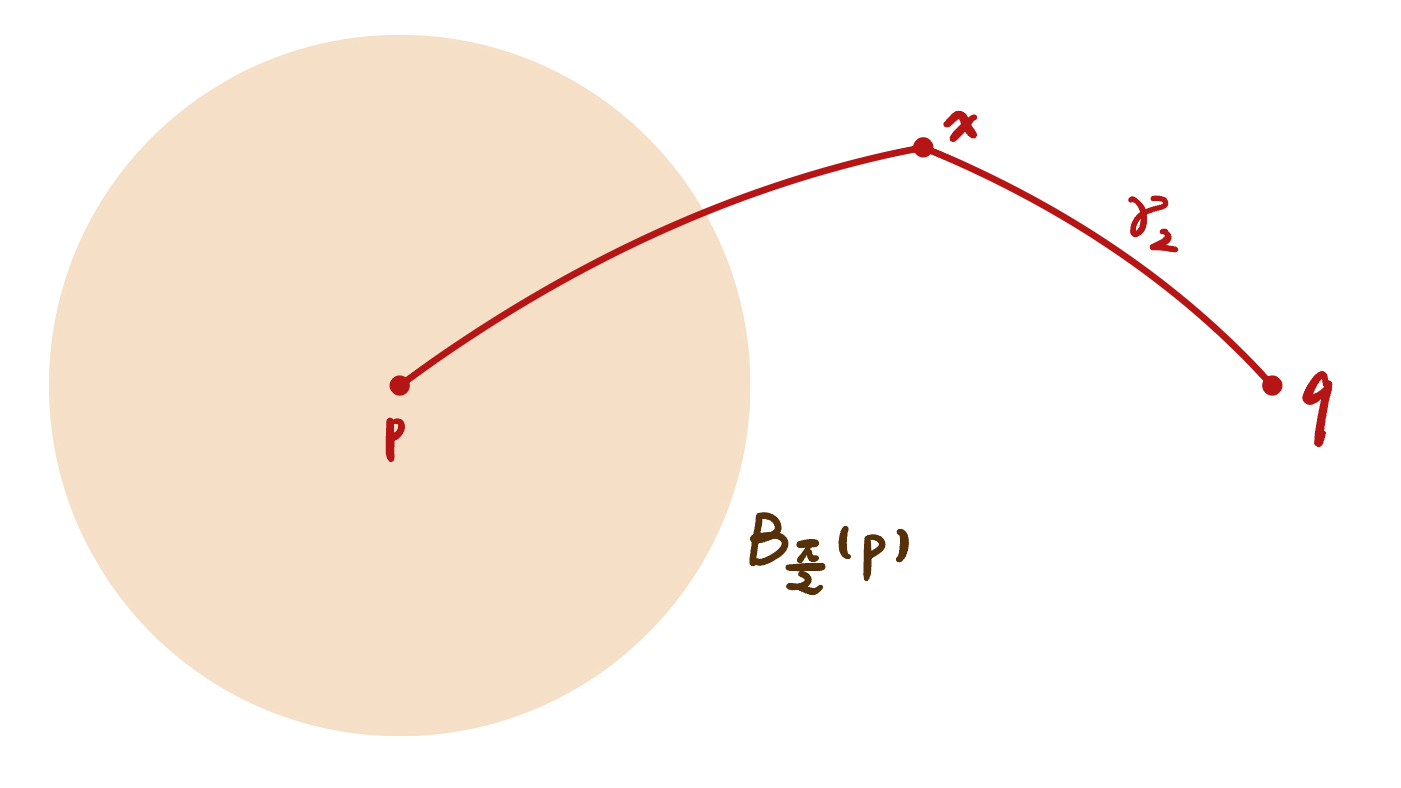}
    \end{figure}
    In this case, $f(\gamma_2(t))$ is concave. And because $f(x) < f(q)$, we have $\frac{d}{dt}\big|_{t = 0^+}f(\gamma_2(t)) > 0$.\footnote{Otherwise, if $\frac{d}{dt}\big|_{t = 0^+}f(\gamma_2(t)) \leq 0$, because concave function is below the tangent line, $f(\gamma_2(t)) \leq f(\gamma_2(0))$ for all $t \geq 0$, thus $f(q) \leq f(x)$. Contradiction.}
    By the first variation formula, $\alpha > \frac{\pi}{2}$. 
    On the other hand, we apply the same argument to $d(\cdot, q)$ to prove the case when $d(p, x) < \frac{\pi}{2}$ and conclude $\alpha > \frac{\pi}{2}$.
    In the last case where $d(x, p) \leq \frac{\pi}{2}, d(x, q) \leq \frac{\pi}{2}$ and $\alpha \leq \frac{\pi}{2}$, apply the hinge comparison, we have $d(p, q) \leq d(\tilde{p}, \tilde{q})$. 
   	However, in the model space $S^n$, since both $l_1, l_2 \leq \frac{\pi}{2}$, we can conclude that $d(p, q) \leq d(\tilde{p}, \tilde{q}) \leq \frac{\pi}{2}$. 
   	By the cosine law, since $l_1, l_2, \alpha \leq \frac{\pi}{2}$, we have
    \begin{equation*}
    	\cos{(d(\tilde{p}, \tilde{q}))} = \cos{(l_1)}\cos{(l_2)} + \sin{(l_1)}\sin{(l_2)}\cos{(\alpha)} \geq 0 \implies y \leq \frac{\pi}{2}
    \end{equation*}
    Then the contradiction arises since $d(p, q) = \diam{M} > \frac{\pi}{2}$ but $d(p, q) \leq d(\tilde{p}, \tilde{q}) \leq \frac{\pi}{2}$. 
    Therefore, in any cases, we have $\alpha > \frac{\pi}{2}$, which implies $x$ is not critical for $d(\cdot, p)$ (and also for $d(\cdot, q)$).

    \textbf{Step 3:} Consider small $\eps$-neighborhoods of $p$ and $q$.
    Construct gradient-like unit vector field for $d(\cdot, p)$ on $M \backslash\br{p, q}$, which is radial for $d(\cdot, p)$ on $B_\eps(p)$. 
    On $B_\eps(q)\backslash \br{q}$. 
    by Step 3 we can take $V = -\nabla d(\cdot, q)$. By Lemma~\ref{lem: integral curve}  $f'(V) \geq c$ on $M\backslash (B_\eps(q)\cup B_\eps(p))$. 
    On $B_\eps(p)\backslash\br{p}$, $f'(V) = 1$ and $M\backslash \underbrace{(B_\eps(p)\cup B_\eps(q))}_{=: W}$ is compact. 
    We claim that $W$ is homeomorphic to $\br{f = \eps} \times [0, 1]$. 
    For any $x \in W$, we can extend the integral curve through $x$ backward and forward, it hits both $\br{f = \eps} = S_\eps(p)$ and $S_\eps(q)$ exactly once. 
    This is because it will eventually hit both $p$ and $q$. 
    Thus we will have to intersect both $S_\eps(p)$ and $S_\eps(q)$. 
    It is impossible for us to intersect $S_\eps(p)$ and $S_\eps(q)$ more than once since the vector field is transversal to both of the surfaces which separate the manifold into two components. Thus there exists a unique $\tau(x)$ such that $\phi_{\tau(x)}(x) \subseteq S_\eps(q)$. As before, $\tau$ is continuous. Now, we construct $F: S_\eps(p) \times [0, 1] \to W$
    \begin{equation*}
        F(x, t) = \phi_{t \cdot\tau(x)}(x) \in S_\eps(q)
    \end{equation*}
    We can easily see that $F$ is onto. And $F$ is continuous since $\phi$ is continuous and $\tau$ is continuous. 
    To show $F$ is one-to-one, i.e. $F(x, t) = F(y, s) \iff x = y$ and $t = s$.
  	Suppose $x \neq y$, then flow lines $\phi$ through $x, y$ do not intersect, which is impossible. 
  	On the other hand, since $x = y$, then $F(x, t) = F(x. s) \iff t = s$
    Therefore, $F$ is a homeomorphism. 
    
    \textbf{Step 4:} Now we are going to finish our proof. 
    Since $F$ is a homeomorphism, 
    \begin{equation*}
    	U := W \cup B_\eps(p) = M \backslash B_\eps(q) \simeq \underbrace{\overline{D}^n \cup (S^{n - 1} \times [0, 1])}_{\simeq \overline{D}_1^n}
    \end{equation*}
  	Because $M^n = U \cup B_\eps(q)$ so that $M^n$ is a twisted sphere, i.e. $M^n \simeq \overline{D}_1^n \stackrel{h}{\cup}\overline{D}_2^n$ where $h : S^{n - 1} \to S^{n-1}$ is an homeomorphism as well.
  	We need to introduce the following lemma
  	\begin{lemma}[Alexander's Trick]
		Any homeomorphism $h: S^{n - 1} \to S^{n - 1}$ can be extended to a homeomorphism between disks $\overline{h}: \overline{D}^n \to \overline{D}^n$ by the formula
		\begin{equation*}
    		\overline{h}(t, x) = th(x)
		\end{equation*}
	\end{lemma}
	\begin{corollary}
		Any twisted sphere is homeomorphic to a sphere. 
	\end{corollary}
	By Alexander's trick, we can extend $h$ and get the homeomorphism from $S^n = \overline{D}^n\stackrel{\id}{\cup}\overline{D}^n$ to $M^n \simeq \overline{D}_1^n \stackrel{h}{\cup}\overline{D}_2^n$. 
	By the following picture, we denote $\overline{D}_{\pm}^{n}$ the upper and lower halves of $S^n$.
	\begin{figure}[htbp]
    \centering
        \includegraphics[width=0.5\textwidth]{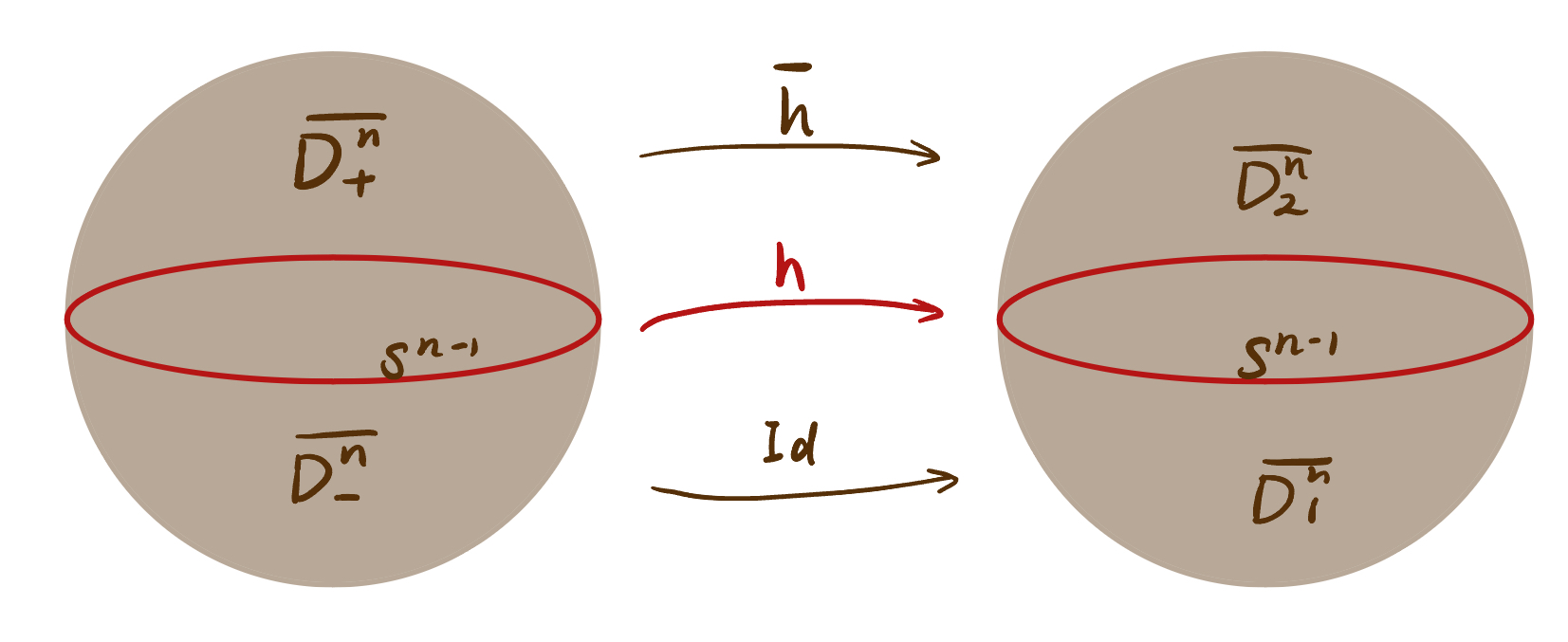}
    \end{figure}
	
	 And we map $\overline{D}_{-}^{n}$ to $\overline{D}_2^n$. 
	 And we glue the image via the homeomorphism $h$. 
	 Finally, we apply the Alexandrov trick to extend the upper sphere so that we construct the homeomorphism for $\overline{D}_+^n$ to $\overline{D}_1^n$.

\begin{remark}
    What if $M^n$ is a smooth twisted sphere and $h: S^{n - 1} \to S^{n - 1}$ is a diffeomorphism. Is $M$ still diffeomorphic to $S^n$? Notice that our proof using the extension of $h$ no longer works. Since the extension is not smooth at the center of the ball. However, in 1956. Milnor proved the following theorem.
\begin{theorem}[Milnor \cite{Mil56}]
   There exists a exotic twisted 7-sphere, that is there exists  exists $h: S^6 \to S^6$ diffeomorphism such that $M \simeq \overline{D}^7 \stackrel{h}{\cup} \overline{D}^7\stackrel{\text{homeo}}{\simeq} S^7$. However, $M$ is not diffeomorphic to $S^7$. Therefore, $M$ is an exotic sphere.
\end{theorem}
\end{remark}
\begin{openquestion}
	If $M^n$ admits $\sect_M \geq 1$ and $\diam(M) > \frac{\pi}{2}$, then $M$ is diffeomorphic to $S^n$.
\end{openquestion}
By \ref{Gromoll-Meyer-sphere}  there exists an exotic sphere which admits a metric of $\sec\ge 0$.
\begin{openquestion}
	Does there exist an exotic sphere that admits a metric of  $\sec > 0$?

\end{openquestion}

\chapter{Manifolds of Nonnegative Curvature}

\begin{example}
    For $\sect \equiv 0$. we have $\R^n$ and $T^n$. For $\sect > 0$. we have $S^n$, $\h P^n$, $\C P^n$.
    
     Let $G$ be a compact Lie group with a bi-invariant metric
    \begin{definition}[Left(Right) Invariant Metric]
        Let $(G, g)$ be a compact Lie group. Then
        \begin{itemize}
            \item $g$ is called \textbf{left invariant metric} if $\inp{u, v}_b = \inp{(dL_a)_bu, (dL_a)_b v}$;
            \item $g$ is called a \textbf{right invariant metric} if $\inp{u, v}_b = \inp{(dR_a)_bu, (dR_a)_b v}$
        \end{itemize}
        for all $a, b \in G$ and $u, v \in T_bG$. $g$ is called a \textbf{bi-invariant metric} if it is both left and right invariant.
    \end{definition}
    \begin{fact}
        If $(G, g)$ is a Lie group with a bi-invariant metric, then $\sect_g \geq 0$.
        Every compact Lie group has a bi-invariant metric.
        
        For example, $SO(n), U(n), SU(n), Sp(n) = \br{A: \text{$n\times n$ unitary quaternion matrix}}$. 
    \end{fact}
    We denote $Sp(n)$ the set of $n\times n$  unitary quaternion matrices. In particular, $Sp(1) = \br{q \in \mathds{H}: \abs{q} = 1} = S^3$ is a compact Lie group. Here $\mathds{H}$  denotes the division ring of quaternions. 
    \begin{equation*}
        \mathds{H} = \br{a + bi + cj + dk: a. b, c, d \in \R}
    \end{equation*}
    And
    \begin{equation*}
        Sp(2) = \br{A = 
        \begin{bmatrix}
            q_{11} & q_{12}\\
            q_{21} & q_{22}
        \end{bmatrix}
        : AA^* = \id, q_{ij} \in \mathds{H}}
    \end{equation*}
    We know that $\dim(Sp(2)) = 10$ and $Sp(2)$ have nonnegative curvature with respect to the canonical bi-invariant metric (See Remark~\ref{rmk: canonical metric}).
\end{example}
\begin{remark}\label{rmk: canonical metric}
	We would like to explain what we mean by canonical metric here. 
	We denote $M^{n \times n}(R)$ where $R = \R, \C, \mathds{H}$ the space of $n \times n$ matrixes with $\R, \C$ and $\mathds{H}$ entries. The metric for these spaces are defined as 
	\begin{equation*}
		\inp{A, B} = \mathfrak{R}(\tr{(AB^*)})\quad(\text{$\mathfrak{R}$ means the real part})
	\end{equation*}
	where $B^*$ is the Hermitian transpose (conjugate transpose) so that in the case of $R = \R$, $B^* = B^T$. 
	Notice, that the above metric is just the Euclidean metric if we consider an $n \times n$-matrix as a vector with $n^2$-entries. Thus, one can induce the metric of $O(n), U(n)$, and $Sp(n)$ canonically using the above metric. 
\end{remark}
\section{Review: Riemannian Submersion}
\begin{definition}[Riemannian Submersion]
    Let $\pi: (M^n, g) \to (N^m, h)$ be a smooth submersion ($m \geq n$).
    For each $p \in M$, we can decompose $T_pM$ into two components:
    \begin{itemize}
    	\item The \textbf{vertical tangent space at $p$} is $\nu_pM = \ker{(d\pi_p)}$;
    	\item The \textbf{horizontal tangent space at $p$} is $H_pM = (\nu_pM)^\perp$.
    \end{itemize} 
    
Since $\pi$ is a submersion, $d\pi_p|_{H_pM}: H_pM \to T_{\pi(p)}N$ is an isomorphism.

    Then $\pi$ is a \textbf{Riemannian submersion} if and only if for any $p \in M$,  the isomorphism $d\pi_p|_{H_pM}: H_pM \to T_{\pi(p)}N$ is an isometry, i.e. 
    \begin{equation*}
    	g_p(X, Y) = h_{\pi(p)}(d\pi_p(X), d\pi_p(Y))
    \end{equation*}
    for any $X, Y \in H_pM$
\end{definition}
\begin{remark}
	At each $p \in M$, we  have an orthogonal decomposition for $T_pM$ as a direct sum  $T_pM = \nu_pM \oplus H_pM$.
\end{remark}
\begin{remark}
	By the submersion level set theorem, at each $q \in N$, each fiber $M_q = \pi^{-1}(q)$ is an embedded smooth submanifold of $M$. Thus we can also write $\nu_pM = T_p(M_{\pi(x)})$. 
\end{remark}

\begin{proposition}
	A Riemannian submersion is $1$-Lipschitz.
\end{proposition}
\begin{proof}
	We want to show $d_h(\pi(x), \pi(y)) \leq d_g(x, y)$ for any pairs of $x, y \in M$. Let $\gamma: [0, l] \to M$ be any curve connects $x$ and $y$, then
	\begin{align*}
		\length(\gamma) 
		&= \int_{0}^l\norm{\dga(t)}_g dt\\
		& \geq \int_{0}^l\norm{\dga^{\parallel}(t)}_g dt\\
		&= \int_{0}^l\norm{d\pi_{\pi(\gamma(t))}(\dga^{\parallel}(t))}_h dt\\
		&= \length(\pi \circ\gamma)
	\end{align*}
	By taking the infimum among all such $\gamma$ it follows that $\pi$ is $1$-Lipschitz.
\end{proof}
Here we introduce some important examples of Riemannian submersions. 
\begin{example}\label{homog-riem-submersion}
Let $G$ be a Lie group that acts freely, isometrically and with closed orbits on $M$. Then $M/G$ inherits a natural Riemannian metric $h$ such that $\pi: (M, g) \to (M/G, h)$ is a Riemannian submersion. 
\end{example}
\begin{example}[Warpped products]

    Suppose $(B, g_B)$ and $(F, g_F)$ are Riemannian manifolds. 
    Let $f: B \to \R_+$ be smooth, then the projection $\pi: M = (B \times F, dg_B + f(x)dg_F) \to (B, dg_B)$ is a Riemannian submersion.
\end{example}
\begin{definition}[Horizontal Lift]
	Let $\pi: (M^n, g) \to (N^m, h)$ be a smooth submersion, a vector field $X \in \mathfrak{X}(M)$ is a \textbf{horizontal vector field} (resp. \textbf{vertical vector field}) if for each $p \in M$, $X_p \in H_pM$ (resp. $X_p \in \nu_pM$). 
	
	Given $Y \in \mathfrak{X}(N)$, a vector field $X \in \mathfrak{X}(M)$ is called a \textbf{horizontal lift of $Y$} if for each $p \in M$, 
	\begin{equation*}
		d\pi_p(X_p) = Y_{\pi(p)}.
	\end{equation*}
	In particular, for each $q \in N$, let $\gamma$ be a curve starting at $q$. 
	A curve $\overline{\gamma}$ starting at some $p \in \pi^{-1}(q)$ is called a \textbf{horizontal lift of $\gamma$} if for each $t$, $\pi\circ \overline{\gamma} = \gamma$ and 
	\begin{equation*}
	\dot{\overline{\gamma}}(t) \, \text{is horizontal for all } t
	\end{equation*}
\end{definition}
\begin{notation}
	We denote $\mathfrak{X}^\nu(M)$ the family of smooth vertical vector fields of $M$ and $\mathfrak{X}^H(M)$ the family of smooth horizontal vector fields of $M$.
\end{notation}
The following is the principal property of horizontal vector fields
\begin{proposition}[See Proposition 2.25 in \cite{Lee19}]\label{prop: principle properties of Horiz}
	Let $\pi: (M^n, g) \to (N^m, h)$ be a smooth submersion between two Riemannian manifolds. Then
	\begin{enumerate}
		\item Every $X \in \mathfrak{X}(M)$ can be uniquely expressed in the form $X = X^H + X^\nu$ where $X^H \in \mathfrak{X}^H(M)$ and $X^\nu \in \mathfrak{X}^V(M)$;
		\item Every $Y \in \mathfrak{X}(N)$ has a unique smooth horizontal lift to $M$;
		\item For every $p \in M$ and $v \in H_pM$, there is a vector field $Y \in \mathfrak{X}(N)$ whose horizontal lift $\tilde{Y}$ satisfies $\tilde{Y}_p = v$.
	\end{enumerate}
\end{proposition}

\begin{proposition}\label{curve-hor-lift}
    Let $\pi: (M^n, g) \to (N^m, h)$ be a Riemannian submersion between two Riemannian manifolds. 
    Fix $p \in M$, then for each regular curve $\gamma$ starting at $\pi(p)$, there is a unique regular horizontal lift $\overline{\gamma}$ starting at $p$. 
\end{proposition}
\begin{proof}
Since the regular curve $\gamma$ is a locally embedding. WLOG, we only need to consider the case that $\gamma: [0, 1] \to N$ is a smooth embedding and has no self-intersection, since otherwise we just split $\gamma$ into many small pieces. Consider the vector field $t \mapsto \dga(t)$, we can have a global extension $X$ of this vector field over $N$. And then by 2. in the proposition \ref{prop: principle properties of Horiz}, we have a unique horizontal lift of $X$, say $\overline{X}$ over $M$. Then by taking the integral curve of $\overline{X}$, we can 
we have a unique smooth horizontal lift of $\dga$. 
	And we just take $\overline{\gamma}$ be the integral curve of the smooth horizontal lift of $\dga$ starting at $p$. 
	The uniqueness follows by the uniqueness of  integral curves. 
\end{proof}
By construction, an important observation is that 
    \begin{equation*}
        \length{(\overline{\gamma})} = \length{(\gamma)}
    \end{equation*}
In particular, if $\gamma$ is a shortest geodesic from $\pi(p)$ to $q$, then the horizontal lift $\overline{\gamma}$ is a shortest geodesic as well. Indeed. Let $x\in\pi^{-1}(q)$ be the endpoint of $\bar \gamma$. Then
    \begin{align*}
        d(p, x) \leq & \length{(\overline{\gamma})} = \length{(\gamma)}\\
        = & d(\pi(p), q)\\
        \leq & d(p, x)\quad\text{Because $\pi$ is $1$-Lipschitz}
    \end{align*}
    Thus, all the inequalities above are equalities, and hence $\bar\gamma$ is length minimizing between its endpoints. Therefore, we can conclude the following corollary. 
\begin{corollary}\label{cor: geodesic horizontal lift}
	Every horizontal lift of a shortest geodesic by a Riemannian submersion is also a shortest geodesic.  
\end{corollary}

\section{Comparison Theory  for Riemannian Submersions}
\begin{theorem}\label{thm: Riem-submersion}
    If $\sect_M \geq \kappa$ and $\pi: M \to N$ is a Riemannian submersion, then $\sect_N \geq \kappa$ as well. 
\end{theorem}
\begin{proof}
    We only need to use a four-point condition to check that $N$ satisfies $\sect_N \geq \kappa$. 
    That is, for $p, x, y, z \in N$, we want to show
    \begin{equation*}
        \modangle^\kappa(p_x^y) + \modangle^\kappa(p_y^z) + \modangle^\kappa(p_z^x) \leq 2\pi.
    \end{equation*}
    Since $\pi$ is a Riemannian submersion, by the previous discussions on the horizontal lift of the shortest geodesics, given the initial point $\tilde{p} = \pi^{-1}(p)$ we can find $\tilde{x} \in \pi^{-1}(x), \tilde{y} \in \pi^{-1}(y), \tilde{z} \in \pi^{-1}(z)$ such that
    \begin{align*}
        d^N(p, x) = d^M(\pi^{-1}(p), \tilde{x}),\\
        d^N(p, y) = d^M(\pi^{-1}(p), \tilde{y}),\\
       	d^N(p, z) = d^M(\pi^{-1}(p), \tilde{z}).
    \end{align*}
    By assumption, $\sect_M \geq \kappa$, thus we have 
    \begin{equation*}
        \modangle^\kappa(\tilde{p}_{\overline{x}}^{\tilde{y}}) + \modangle^\kappa(\tilde{p}_{\tilde{y}}^{\tilde{z}}) + \modangle^\kappa(\tilde{p}_{\tilde{z}}^{\tilde{x}}) \leq 2\pi.
    \end{equation*}
We  claim that 
    \begin{equation*}
        \modangle^\kappa(p_x^y) \leq \modangle^\kappa(\overline{p}_{\overline{x}}^{\overline{y}})
    \end{equation*}
    Once we prove that, the four-point comparison in $N$ holds, and the result follows. 
    The key observation now is that $\pi$ is $1$-Lipschitz and we should use hinge comparison, thus
    \begin{align*}
        \abs{x - y} \leq &\abs{\tilde{x} - \tilde{y}}\quad\text{Since $\pi$ is $1$-Lipschitz}\\
        \leq  &\modcvee^\kappa[\tilde{p}_{\tilde{x}}^{\tilde{y}}]\quad \text{Since $\sect_M \geq \kappa$ and by hinge comparison}
    \end{align*}
    Then we immediately conclude that $\modangle^\kappa(p_x^y) \leq \modangle^\kappa(\tilde{p}_{\tilde{x}}^{\tilde{y}})$.
    \end{proof}
\begin{remark}
    There is another method can be found in \cite{Bes07} using O'Neil's formula. 
\end{remark}
We are going to list some applications of theorem \ref{thm: Riem-submersion}. 
\begin{corollary}[See Corollary 2.29 in \cite{Lee19}]\label{cor: quotient group}
	Suppose $G$ is a Lie group acting on $M$ smoothly, freely,  isometrically, and with closed orbits then by \ref{homog-riem-submersion} $\Pi:  M\to M/G$ is a Riemannian submersion. If $\sect_M \geq \kappa$, then by  theorem \ref{thm: Riem-submersion}, $\sect_{M/G} \geq \kappa$ as well.
\end{corollary}
\begin{fact}\label{biinv-curv}
Suppose $(G, g)$ is a Lie group with a bi-invariant Riemannian metric. Then for any $X, Y \in \mathfrak{X}(G)$ it holds that $\langle R(X,Y)Y,X\rangle  = \frac{1}{4}\norm{[X, Y]}^2 $ (See exercise 1 page 103 in \cite{dCa}). In particular,
$\sect_G\ge 0$.
\end{fact}
\begin{example}
Suppose  $G$ is a Lie group with a bi-invariant metric then by \ref{biinv-curv} $\sec_G\ge 0$. Then if $H$ is a closed subgroup of $G$ then it's a closed Lie subgroup of $G$  (theorem 20.21 in \cite{Lee03}) and its left action on $G$ is free, isometric, and has closed orbits. Therefore by above
	  $\sect_{G/H} \geq 0$ as well. 
\end{example}

\begin{example}[Bi-quotients]
    Still let $G$ be a Lie group and let $H < G\times G$ be a closed subgroup. Then $H \acts G$ such that for $h = (h_1, h_2) \in H$, 
    \begin{equation*}
        h(g) = h_1 g h_2^{-1}
    \end{equation*}
    
This action is called a \emph{bi-quotient action}. In a special case when $H$ belongs to either $e\times G$ or $G\times e$ this is a homogeneous space action. Unlike homogeneous space actions, bi-quotient actions need not be free in general.
    
If $H$ is compact then the bi-quotient action is proper.  If the bi-quotient action of $H$ on $G$ is also \emph{free} then the quotient space $G//H$ of the action above is a smooth manifold.
    
    Suppose $G$ admits a bi-invariant metric (note that such a metric always exists if $G$ is compact). Then by above $\sect_G \geq 0$ and by   theorem \ref{thm: Riem-submersion},
    \begin{align*}
& \text{$G \to G//H$ is a Riemannian submersion}\\
        \implies & \text{$G//H$ will have $\sect_{G//H}\geq 0$}
    \end{align*}
\end{example}
\begin{remark}
In the special case of $H = K_1 \times K_2$ where $K_1 < G\times\br{e}$ and $K_2 < \br{e}\times G$  the action of $H$ on $G$ is given by the formula
 \begin{equation*}
        (k_1, k_2)(g) = k_1 g k_2^{-1}.
    \end{equation*}
    where $(k_1, k_2) \in K_1 \times K_2$.
\end{remark}

\begin{example}\label{Gromoll-Meyer-sphere} [Gromoll-Meyer in \cite{GM74}]
Consider 
\begin{equation*}
	G = Sp(2) = \br{A = 
    \begin{bmatrix}
    q_{11} & q_{12}\\
    q_{21} & q_{22}
    \end{bmatrix}: \overline{A^t} =: A^* = A^{-1}}
\end{equation*}
and the unit quaternion $H = Sp(1) = S^3$.

And consider a biquotient action  $H \acts G = S^3 \acts Sp(2)$ given by
    \begin{equation*}
        q(A) = 
        \begin{bmatrix}
          q & 0\\
          0 & q
        \end{bmatrix}A
        \begin{bmatrix}
          1 & 0\\
          0 & \overline{q}
        \end{bmatrix}
    \end{equation*}
    where $\overline{q} = q^{-1}$. It is not hard to check this action is free. 
    Notice that $\dim(Sp(2)) = 10$. 
    Take $M^7 = Sp(2)//Sp(1)$. Then by \cite{GM74}, $M^7$ is an exotic sphere of $\sect_{M^7} \geq 0$. 
    That means $M^7 \stackrel{homeo}{\cong}S^7$ but not diffeomorphic. 
    An open question related to that is whether there exists an exotic sphere of $\sect > 0$. 
\end{example}


\chapter{Cheeger-Gromoll Soul Theorem}
\section{Introduction to the Soul Theorem}

The soul theorem describes the topological structure of open manifolds of non-negative sectional curvature. By open manifolds we mean non-compact complete manifolds without boundary.

\begin{theorem}[Cheeger-Gromoll Soul Theorem]\label{thm: soul theorem}
    Let $(M^n, g)$ be a Riemannian manifold of $\sect_M \geq 0$ and it is complete and non-compact. Then there exists a closed totally convex subset $S \subseteq M^n$ which is also a totally geodesic submanifold such that 
    \begin{equation*}
        M^n \stackrel{\text{diff}}{\cong} \nu(S)
    \end{equation*}
    where $\nu(S)$ is the total space of the normal bundle of $S$. In particular, $M^n$ has compact type (See Definition~\ref{def: compact type}) since
    \begin{equation*}
        M^n \stackrel{\text{diff}}{\cong} \text{interior of unit disk bundle of $\nu(S)$}
    \end{equation*}
    where the unit disk bundle of $\nu(S)$ is also a compact manifold with a boundary.
    \end{theorem}
There are some suggestive examples to explain the theorem. 
Given $(M^n, g)$ a Riemannian manifold with $\sect_M \geq 0$. 
Then $\sect_{M \times \R^n} \geq 0$.
Suppose $G < O(n)$ be a closed subgroup such that $G \acts M^n$ by isometries freely. 
Also $G \acts \R^n$ by matrix multiplication.
We can take the diagonal action of $G$ on $M \times \R^n$ , i.e. $g(x, v) = (gx, gv)$. Then $\sect_{(M \times \R^n)/G} \geq 0$ by ~\ref{cor: quotient group}.
    
    Consider the following vector bundle. 
    \begin{equation*}
    \begin{tikzcd}
        \R^n \arrow{r}{\empty} & (M^n \times \R^n)/G \arrow{d}{\empty} \\%
        & M^n/G
    \end{tikzcd}
\end{equation*}
We can conclude that both the total space and the base of this bundle are manifolds of non-negative sectional curvature and the projection map $(M \times \R^n)/G\to M/G$ is a Riemannian submersion. Note that the total space  $(M \times \R^n)/G$ is necessarily non-compact and the base is compact if and only if $M$ is.

The following two examples are special cases of the above construction. 
\begin{example}
    $M = SO(n + 1)$, $G = SO(n)\acts \R^n$. We equip $ SO(n + 1)$ with the canonical bi-invariant metric.

    We know that $SO(n) < SO(n + 1)$ by
    \begin{equation*}
        SO(n) \ni A \mapsto 
        \begin{bmatrix}
            1 & 0\\
            0 & A
        \end{bmatrix}
    \end{equation*}
    Then $G \acts M$ on the left. This action is obviously free and isometric.  Consider the vector bundle
    \begin{equation*}
    \begin{tikzcd}
        \R^n \arrow{r}{\empty} & (SO(n + 1)\times \R^n)/G \arrow{d}{\empty} \\%
        & SO(n + 1)/SO(n) = S^n
    \end{tikzcd}
\end{equation*}
\begin{fact}
    We want to show this vector bundle is the tangent bundle to $S^n$ in any dimension and $TS^n$ admits a metric of complete Riemannian manifold of $\sect \geq 0$.
\end{fact}
\end{example}
\begin{example}
    Consider $M = S^{n + 1}$, $G = S^1 = \br{z \in \C: \abs{z} = 1}$. $G$ acts on $\C^n \cong \R^{2n}$ by
    \begin{equation*}
        \lambda(z_1, \dots, z_n) = (\lambda z_1, \dots, \lambda z_n).
    \end{equation*}
    Also, $G$  acts on the unit sphere $S^{2n + 1} \subseteq \C^{n + 1}$ by the formula above. We take the bundle
    \begin{equation*}
    \begin{tikzcd}
        \R^2 \arrow{r}{\empty} & S^{2n + 1}\times \C/S^1 \arrow{d}{\empty} \\%
        & S^{2n + 1}/S^1 = \C P^n
    \end{tikzcd}
\end{equation*}
This is the Hopf bundle and by above it admits a metric of nonnegative sectional curvature. 
\end{example}

\section{Busemann Functions}

\begin{definition}[Ray]
    A globally distance minimizing unit speed geodesic $\gamma: [0, \infty) \to M$ is called a \textbf{ray} i.e. 
    \begin{equation*}
        d(\gamma(t), \gamma(s)) = \abs{t - s}\quad\forall t, s \geq 0
    \end{equation*}
\end{definition}

\begin{lemma}\label{lem:ray existence}
    If $M$ is complete and non-compact, then $M$ contains a ray starting at any point $p \in M$.
\end{lemma}
\begin{proof}
    By non-compactness of $M$ there is  a sequence $q_n \in M$ such that $d(p, q_n) \to \infty$. Connect $p, q_n$ using  a shortest geodesic $\gamma_n$, denote $v_n = \gamma_n'(0) \in S^{n - 1}$ the unit vector in $T_pM$. By compactness of $S^{n - 1} \subseteq T_pM$, we can find a $v \in S^{n - 1}$ such that $v_n \to v$ upto subsequence. Thus
    \begin{equation*}
        \exp(tv_n) \to \exp(tv)
    \end{equation*}
    \begin{claim}
        $t \mapsto \exp(tv)$ is a ray. Fix $l > 0$, then $\gamma_n(l) \to \gamma(l)$. So 
        \begin{align*}
            l = d(p, \gamma_n(l)) \to d(p, \gamma(l)) \implies & d(p,  \gamma(l)) = l\quad \forall l;\\
            \implies & \text{$\gamma$ is a ray.}
        \end{align*}
    \end{claim}
\end{proof}
\begin{definition}[Busemann Function]
    Let $\gamma:[0, \infty) \to M$ be a ray, $M \ni p = \gamma(0)$. Then a Busemann function for $\gamma$ is defined as
    \begin{equation*}
        b_\gamma(x) = 
        \lim_{t \to \infty}(d(x, \gamma(t)) - t)
    \end{equation*}
\end{definition}
\begin{claim}\label{cla: Busemann finite}
    We claim that a Busemann function is always well-defined and finite. 
\end{claim}
\begin{proof}
   Since $t=d(\gamma(0), \gamma(t))$,  by the triangular inequality we have 
    \begin{equation*}
        -d(x, \gamma(0)) \leq d(x, \gamma(t)) - t \leq d(x, \gamma(0)),
    \end{equation*}
This means that the quantity $d(x, \gamma(t)) - t$ is bounded.

Let $t_1 < t_2$, then
    \begin{equation*}
        d(x, \gamma(t_2)) \leq d(x, \gamma(t_1)) + t_2 - t_1 \iff d(x, \gamma(t_2)) - t_2 \leq d(x, \gamma(t_1)) - t_1
    \end{equation*}
    Therefore, $t \mapsto d(x, \gamma(t)) - t$ is always non-increasing and bounded. Thus, it has a limit as $t \to \infty$. Therefore, $b_\gamma(x)$ is always defined and finite. 
\end{proof}
\begin{property}\label{property: b 1 lip}
    $b_\gamma$ is $1$-Lipschtiz.   
\end{property}
\begin{proof}
	For each $t$, let $x, y \in M$, by the triangule inequality we have
	\begin{equation*}
		\abs{(d(x, \gamma(t)) - t) - (d(y, \gamma(t)) - t)} \leq d(x, y).
	\end{equation*}
	Therefore we can conclude that the Busemann function is always 1-Lipschitz as a limit of 1-Lipschitz functions.
	\end{proof}
\begin{property}\label{property: b = -s}
    For any $s \in [0, \infty)$, $b_\gamma(\gamma(s)) = -s$.
\end{property}
\begin{proof}
	By direct substitution, for each $s$, we have
	\begin{align*}
		b_\gamma(\gamma(s)) = \lim_{t \to \infty}d(\gamma(t), \gamma(s)) - t =  \lim_{t \to \infty}(t - s) - t = -s
	\end{align*}
\end{proof}
\begin{example}
We are going to give a qualitative example to strengthen our intuition on the sub-level set and the super-level set of the Busemann function. We draw the following picture, Let $P$ be a paraboloid in $\R^3$ and let $x \in P$ be a point. Consider a ray $\gamma$ starting at $x$ and going upwards. 

\begin{figure}[htbp]
    \centering
        \includegraphics[width=0.7\textwidth]{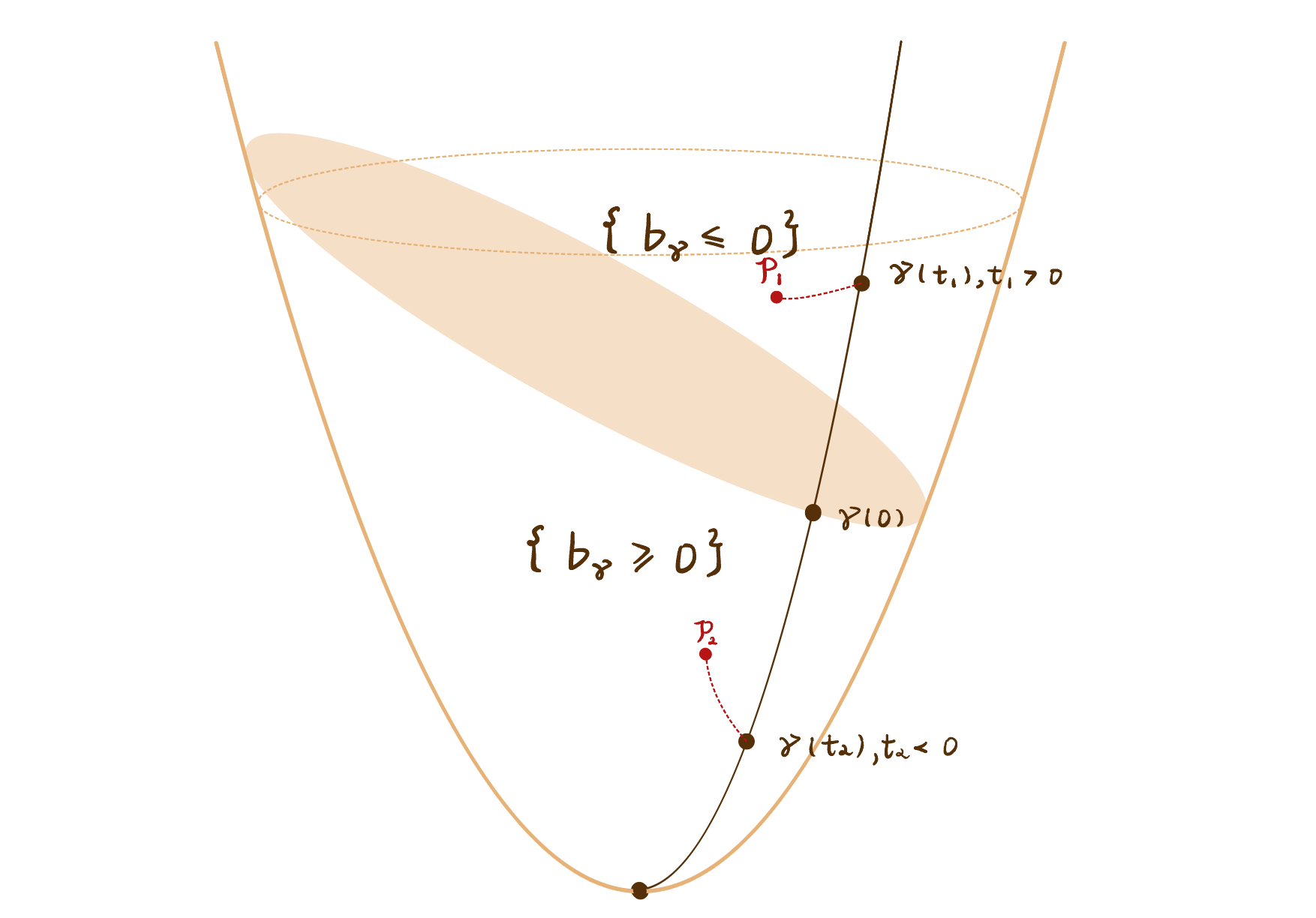}
    \end{figure}

By the previous property \ref{property: b = -s}, we know that $b_\gamma(\gamma(t)) = -t \leq 0$ for any $t \leq 0$.     Because the Busemann function is $1$-Lipschitz, then for any $P \in M$, take the shortest geodesic from $P$ to $\gamma$, which intersect $\gamma$ at some $t$, then we know that 
    \begin{equation*}
    	 b_{\gamma}(\gamma(t)) -  d(\gamma(t), P) \leq b_{\gamma}(P) \leq b_{\gamma}(\gamma(t)) +  d(\gamma(t), P)
    \end{equation*}
    
    If $t_1 >> 0$, any point $P_1$ near $\gamma(t_1)$, $b_\gamma(P_1) \leq 0$. This is because
    \begin{equation*}
    	b_{\gamma}(P_1) \leq b_{\gamma}(\gamma(t_1)) +  d(\gamma(t_1), P_1) = d(\gamma(t_1), P_1) -t_1 \leq 0
    \end{equation*}

     On the other hand, we can extend $\gamma$ in the opposite direction. It will remain globally shortest until it hits the apex of the parabola. Then
	\begin{equation*}
		b_{\gamma}(\gamma(-s)) = \lim_{t \to \infty}d(\gamma(t), \gamma(-s)) - t = t + s - t = s > 0
	\end{equation*}
	If we pick some $t_2 < 0$ before the geodesic hits the apex, we can still take some $P_1$ near $\gamma(t_2)$. In this case, 
	\begin{equation*}
		b_{\gamma}(P_2) \geq b_{\gamma}(\gamma(t_2)) -  d(\gamma(t_2), P_2) = -t_2 - d(\gamma(t_2), P_2) \geq 0	\end{equation*}

\end{example}

\begin{property}\label{asymptotic-ray}
    Given $x \in M$ any point, For a sequence $t_n \to \infty$, denote $v_n$ the initial vector of $[x\gamma(t_n)]$ starting at $x$. 
    Then by compactness $v_n \to v$ sub-converges. 
    We claim that 
    \begin{equation}\label{eq: Busemann decreases with unit speed}
        b_\gamma(\exp_x(tv)) = b_\gamma(x) - t.
    \end{equation}
    That is along $\exp_x(tv)$, $b_\gamma$ decreases with unit speed. 
\end{property}
\begin{proof}
    Fix $l > 0$, $t_n \to \infty$, denote $c_n(t) = \exp_x(tv_n)$ for each $n$. 
    \begin{figure}[htbp]
    \centering
        \includegraphics[width=0.7\textwidth]{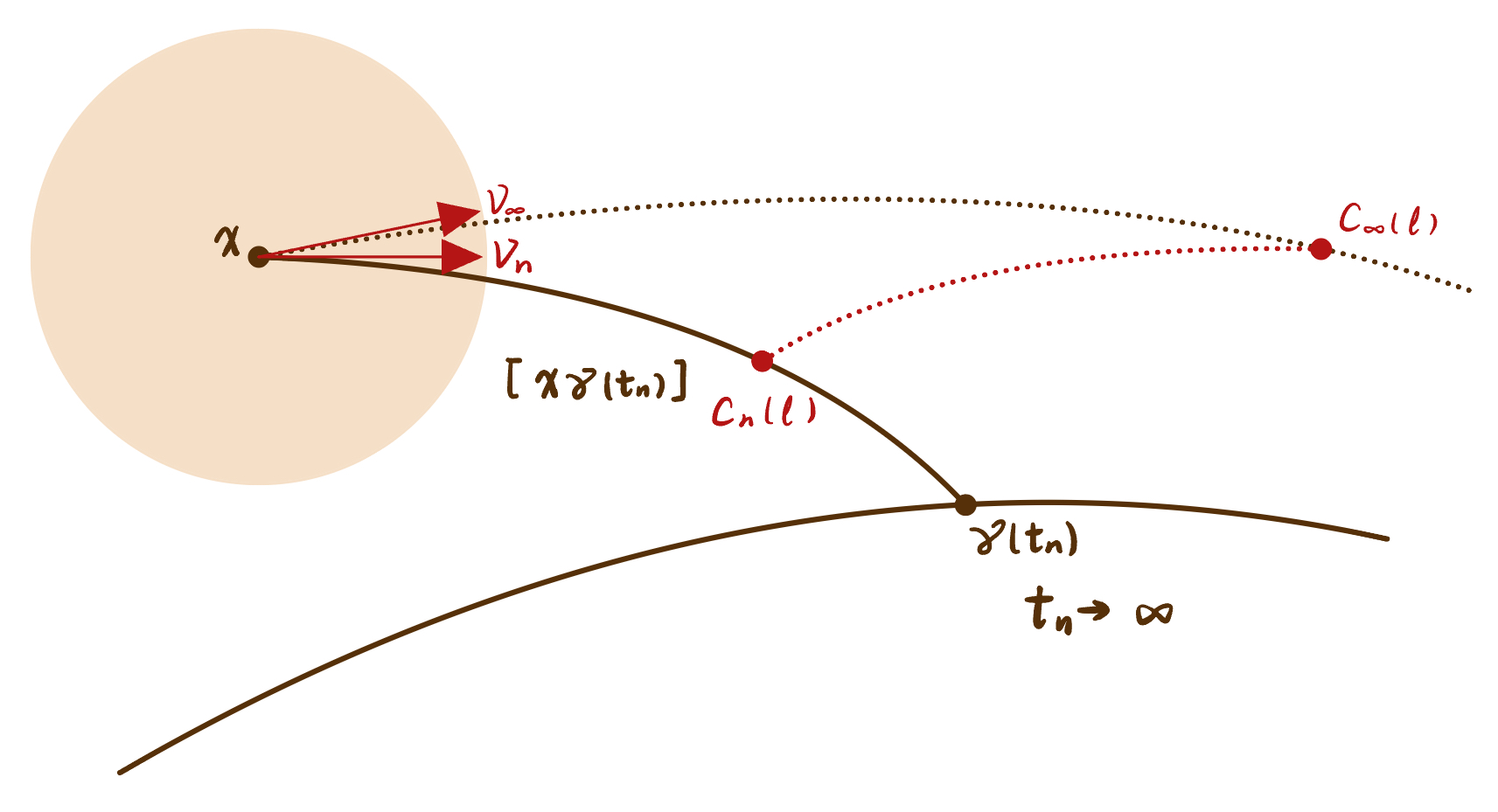}
    \end{figure}
    
    Since the unit initial vectors $v_n$ sub-converge to $v$, and $c(t)  = \exp_x(tv)$, $c_n(l) \to c(l)$ sub-converges as well. 
    Then
    \begin{align*}
    	b_\gamma(c(l)) 
    	& = \lim_{n \to \infty} d(c_n(l), \gamma(t_n)) - t_n\\
    	& = \lim_{n \to \infty} d(x, \gamma(t_n)) - l - t_n\\
    	& = b_\gamma(x) - l
    \end{align*}   	
    \end{proof}
Moreover, the equation \ref{eq: Busemann decreases with unit speed} also implies that $c$ is a ray, i.e. $d(c(0), c(l)) = l$ for any $l > 0$.
   	We can show this by contradiction. Suppose not, $d(c(0), c(l)) < l$.
    Since $b_\gamma$ is $1$-Lipschitz, then $\abs{b_\gamma(c(l)) - b_\gamma(c(0))} \leq d(c(l), c(0)) < l$. 
    However, it is known that $\abs{b_\gamma(c(l)) - b_\gamma(c(0))} = l$.

\begin{lemma}
	For $t_1 \leq t_2$, $B_{t_1}(\gamma(t_1)) \subseteq B_{t_2}(\gamma(t_2))$. 
\end{lemma}
\begin{figure}[htbp]
    \centering
        \includegraphics[width=0.5\textwidth]{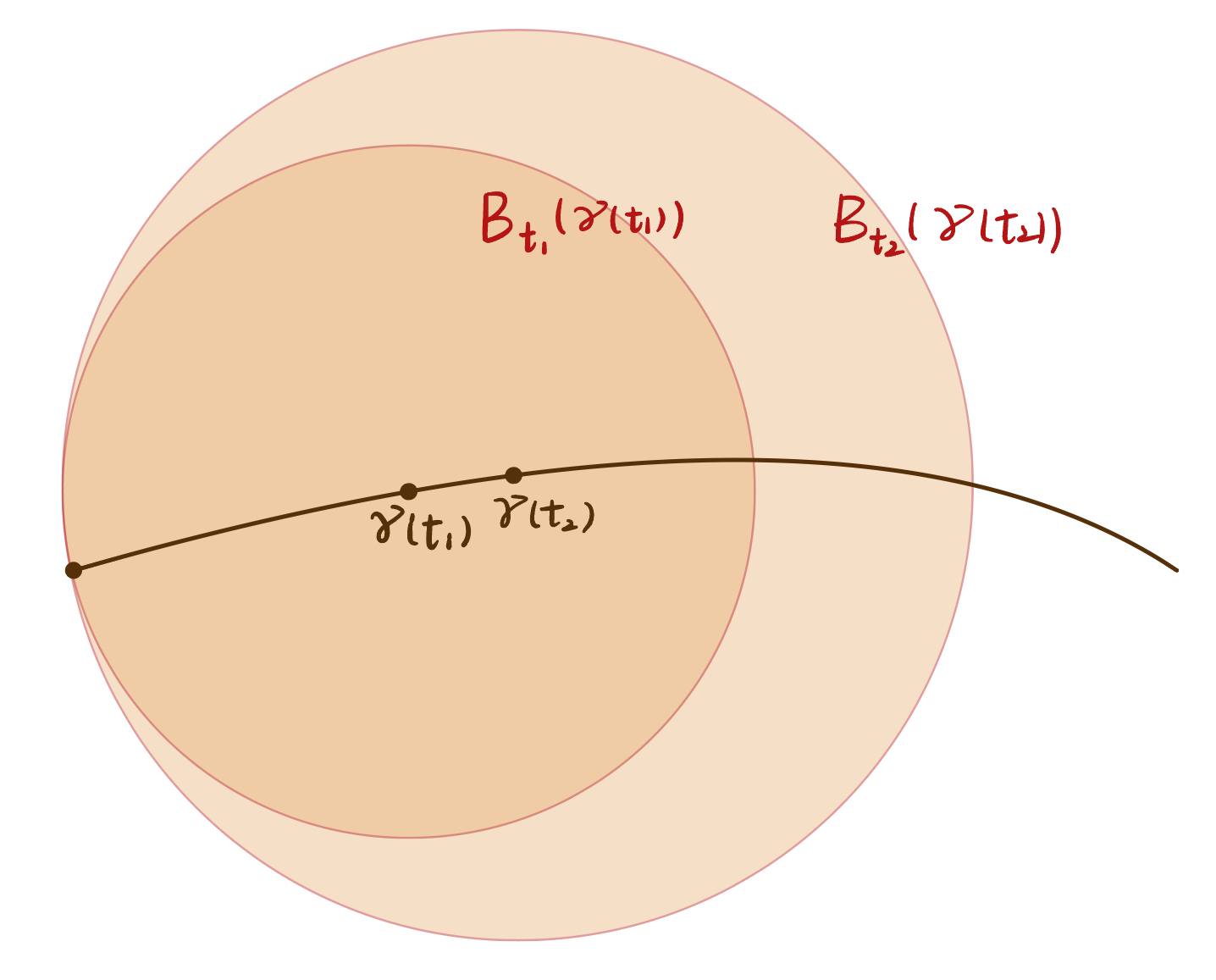}
    \end{figure}
\begin{proof}
	Let $x \in B_{t_1}(\gamma(t_1))$, we want to show $d(x, \gamma(t_2)) \leq t_2$, this follows by the triangular inequality: 
	\begin{align*}
		d(x, \gamma(t_2)) \leq d(x, \gamma(t_1)) + d(\gamma(t_1), \gamma(t_2)) \leq d(\gamma(0), \gamma(t_2)) = t_2
	\end{align*}
\end{proof}

\begin{proposition}
 $\br{b_\gamma < 0} = \bigcup_{t > 0}{B}_{t}(\gamma(t))$
\end{proposition}
\begin{proof}
Let $x \in \bigcup_{t > 0}B_{t}(\gamma(t))$, then $x \in B_{t}(\gamma(t))$ for some $t > 0$, which means $d(x, \gamma(t)) < t$. And since $b_\gamma$ is $1$-Lipschitz, we know that
	\begin{align*}
		& \abs{b_{\gamma}(x) - b_{\gamma}(\gamma(t))} \leq d(x, \gamma(t)) < t\\
		\implies & b_{\gamma}(x) < t + \underbrace{b_{\gamma}(\gamma(t))}_{= -t} = 0
	\end{align*}
	Therefore, $\br{b_\gamma < 0} \supseteq \bigcup_{t > 0}B_{t}(\gamma(t))$,
	
	The opposite direction is trivial, let $x \in \br{b_\gamma < 0}$, then we know that $d(x, \gamma(s)) - s < 0$ for sufficiently large $s$. 
\end{proof}

\begin{remark}
	For any $c > 0$, we have  
	\begin{align*}
		&\br{b_\gamma < c} = \bigcup_{t > 0}B_{t + c}(\gamma(t))\\
		&\br{b_\gamma < -c} = \bigcup_{t > 0}B_{t}(\gamma(t + c))
	\end{align*}
\begin{figure}[htbp]
    \centering
        \includegraphics[width=0.9\textwidth]{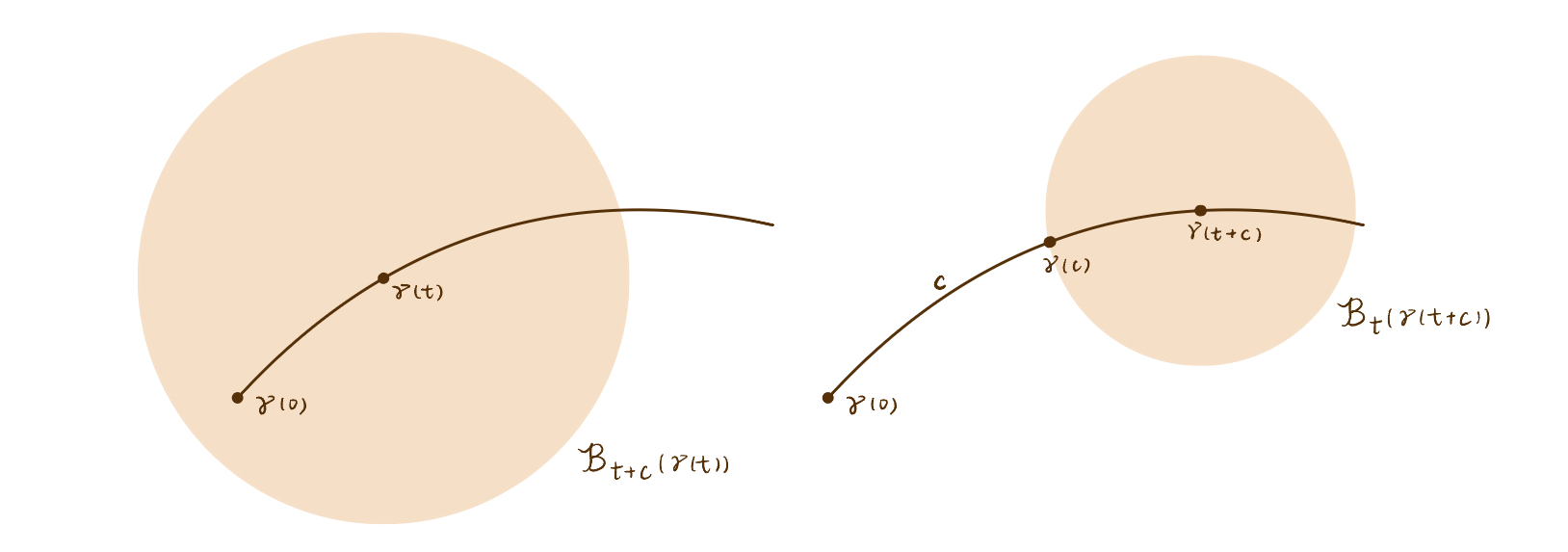}
    \end{figure}
\end{remark}

\begin{property}
    If $\sect_M \geq 0$, then $b_\gamma$ is concave for any ray $\gamma$.
\end{property}
\begin{proof}
	In $\R^n$, take $x = (x_1, x_2, \dots, x_n)$, $\gamma(t) = (t, 0, \dots, 0)$. 
	We can claim that the Busemann function is $b_\gamma(x) = -x_1$ a linear function. 
	Let us verify this for the case $x_1=0$ i.e $x=(0,, x_2, \dots, x_n)$ (the general case is similar).
    Notice that
    \begin{equation*}
        d(x, \gamma(t)) = \sqrt{\abs{x}^2 + t^2}.
    \end{equation*}
    And since
    \begin{align*}
        (d(x, \gamma(t)) - t) 
        & = (d(x, \gamma(t)) - t)\cdot\frac{(d(x, \gamma(t)) + t)}{(d(x, \gamma(t)) + t)}\\
        & = \frac{\abs{x}^2}{d(x, \gamma(t)) + t}.
    \end{align*}
    Taking $t \to \infty$, we have 
    \begin{equation*}
        b_\gamma(x)
        = \lim_{t \to \infty}(d(x, \gamma(t)) - t)
        = \lim_{t \to \infty} \frac{\abs{x}^2}{(d(x, \gamma(t)) + t)}
        = 0
    \end{equation*}
    The general case is by a similar computation. We can conclude that the Busemann function in $\R^n$ is affine, then concave as well.  
    
    Now we want to use the Toponogov theorem to show $b_\gamma$ is concave in $M$ with $\sect_M \geq 0$.
    
As the following picture shows, fix some point $x \in M$, by triangular inequality $d(x, \gamma(t)) > t - d(x, \gamma(0))$. Take $\eps$ sufficiently small, for any $y \in B_\eps(x)$, by the triangular inequality again, after fixing any $R > 0$, we can find a sufficiently large $t$ such that
\begin{align*}
	d(y, \gamma(t)) 
	&\geq d(x, \gamma(t)) - d(x, y)\\
	&> t - d(x, \gamma(0)) - \eps\\
	& > R
\end{align*}

\begin{figure}[htbp]
    \centering
        \includegraphics[width=0.6\textwidth]{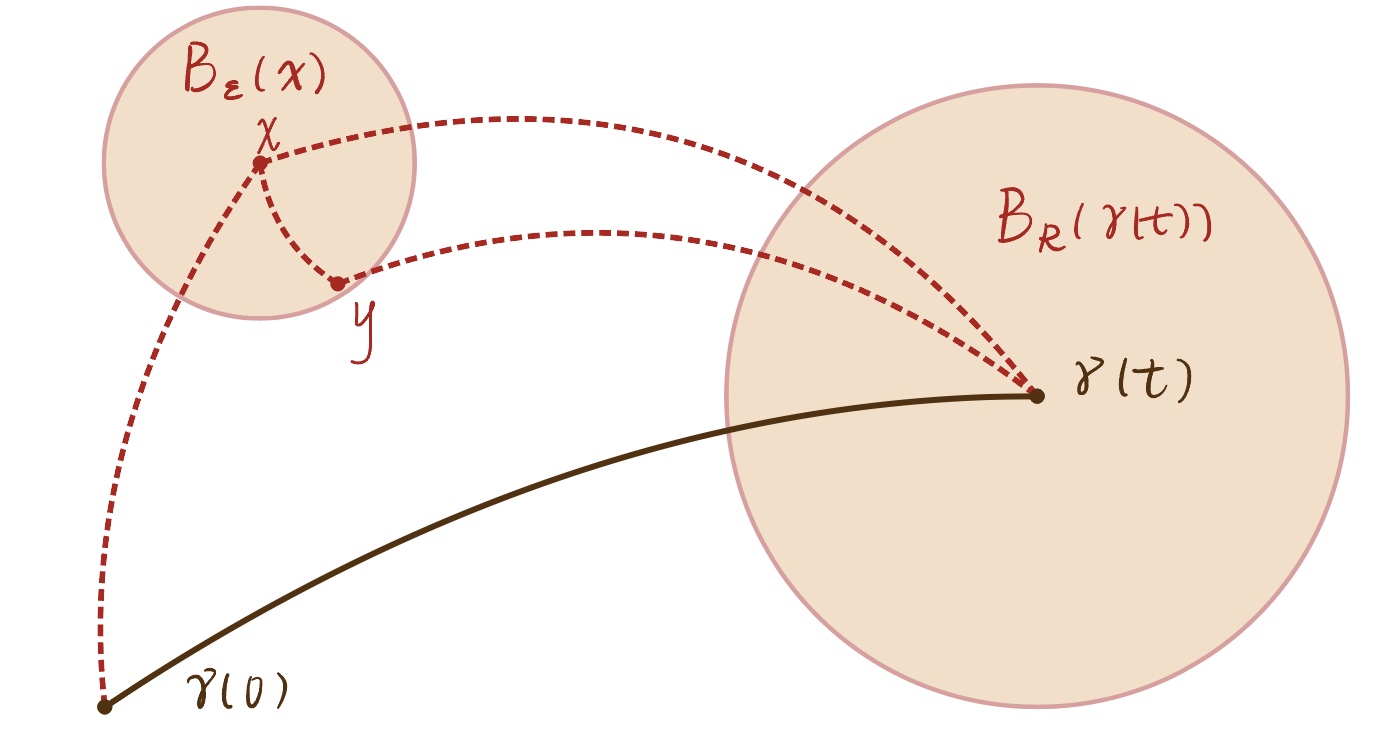}
    \end{figure}

That means $d(y, \gamma(t)) > R$ for large $t$. We claim that this implies that $d(\cdot, \gamma(t))$ is $\frac{1}{R}$-concave outside of $B_R(\gamma(t))$. This is because if we have $p \in M$ and $\tilde{p} \in \R^n$, by the Hessian comparison theorem~\ref{cor: Hess ineq} (in particular inequality~\ref{eq: Hessian inequality for d_p}), we know that 
\begin{equation*}
	\hess_{d_p} \leq \frac{1}{d_p}\pi_{d_p} \leq \frac{1}{R}\id \quad\text{for $d_p \geq R$}
\end{equation*}
where $\pi_p$ is the orthogonal projection onto the tangent space of the sphere centered at $p$ of radius $d_p$. Since $B_\eps(x) \subseteq M\backslash B_R(\gamma(t))$ for all large $t$, we can conclude that $d(\cdot, \gamma(t))$ is $\frac{1}{R}$-concave on $B_\eps(x)$ for all large $t$. Therefore, 
\begin{center}
	$d(\cdot, \gamma(t)) - t$ is also $\frac{1}{R}$-concave on $B_\eps(x)$ for all large $t$. 
\end{center} 
By the definition of the Busemann function, we know that
\begin{equation*}
    b_\gamma(y) = \lim_{t \to \infty}(d(y, \gamma(t)) - t)
\end{equation*}
is also $\frac{1}{R}$-concave on $B_\eps(x)$ for all $R > 0$. Therefore, $b_\gamma(y)$ is concave on $B_\eps(x)$.  Since $b_\gamma$ is concave near $x$ and concavity is a local property, we conclude that $b_\gamma$ is globally concave. \end{proof}
\begin{remark}
There is a corresponding statement about Busemann functions on spaces of nonpositive curvature. Namely, if $M^n$ is complete, simply connected, and has $\sect_M \le 0$ then for any ray $\gamma$ its Busemann function $b_\gamma$ is  \textbf{convex}. 
\end{remark}

\begin{definition}[Total Busemann Function]
    Let $M$ be a noncompact, complete Riemannian manifold without boundary and $\sect_M \geq 0$. Let $p \in M$, we can define the \textbf{total Busemann function} as
    \begin{equation*}
        b(x) = \inf\br{b_\gamma(x): \text{$\gamma$ is a ray starting at $p$}}.
    \end{equation*}
    for each $x \in M$.
\end{definition}

\begin{remark}\label{inf=min}
Since a limit of rays is a ray it easily follows that this infimum is always achieved, i.e. 
\begin{equation*}
        b(x) = \min\br{b_\gamma(x): \text{$\gamma$ is a ray starting at $p$}}
    \end{equation*}
    for any $x\in M$.
\end{remark}

\begin{definition}[Negative Gradient Ray]
    Given a 1-Lipschitz function $f\: M \to \R$ a ray $\sigma$ is called a  \textbf{negative gradient} ray for $f$ if
    \begin{equation*}
        f(\sigma(t)) = f(x) - t.
    \end{equation*}
 \end{definition}
    Note that if $\sigma$ is a negative gradient ray for $f$ then $f$ decreases along $\sigma$ with maximal possible speed (=1). This justifies the term "negative gradient ray".

\begin{proposition}\label{prop:neg-grad-ray}
    Both  $b$ and $b_\gamma$ have the property that for any $x \in M$ each of them admits a negative gradient ray starting at $x$.

\end{proposition}
\begin{proof}
The result for $b_\gamma$ follows by  Property~\ref{asymptotic-ray}.

Let $x\in M$ and let us prove that $b$ has a negative gradient ray starting at $x$.

	By Remark \ref{inf=min}, we know that there exists a ray $\gamma$ starting at $p$ such that $b(x) = b_{\gamma}(x)$. By the Property~\ref{asymptotic-ray}, we can construct a ray $c$ starting at $x$ such that $b_\gamma(c(t))=b_\gamma(x)-t$ for any $\gamma$. Now we want to show this implies $b(c(t)) = b(x) - t$, i.e that $c$ is a negative gradient ray for $b$ starting at $x$. Take $\sigma$ any other ray starting at $p = \gamma(0)$, then
\begin{equation*}
	\abs{b_{\sigma}(c(t)) - b_{\sigma}(x)} \leq t
\end{equation*}
since $b_{\sigma}$ is $1$-Lipschitz. This means 
\begin{align*}
	b_{\sigma}(c(t)) 
	&\geq b_{\sigma}(x) - t\\
	&\geq b(x) - t \\
	&\geq b_{\gamma}(x) - t
\end{align*}
Remember that $b_{\gamma}(c(t)) = b(x) - t$. Therefore, we can conclude that for any ray $\sigma$ starting at $x$, 
\begin{equation*}
	b_{\sigma}(c(t)) \geq b_{\gamma}(c(t))
\end{equation*}
this means the minimum is achieved when $\sigma = \gamma$. $c(t)$ is a negative gradient ray for $b$ and hence $b(c(t))=b_{\gamma}(c(t))=b_{\gamma}(x)-t=b(x)-t$.
\end{proof}

\begin{definition}[Totally Convex Subsets]\label{def: totally convex}
	A subset $C \subset M$ is called  \textbf{totally convex} if for any $x, y \in C$ and any geodesic $\gamma$  (not necessarily shortest ) from $x$ to $y$, it holds that   $\gamma \subset C$. 
\end{definition}
\begin{proposition}\label{prop: total busemann convex and cpt}   
    Suppose $b$ is a total Busemann function and let $c \in \R$, the set 
    \begin{equation*}    
    	C_c = \br{x\in M: b(x) \geq c}
    \end{equation*}
    is compact and totally convex. The total convexity tells us that for any $x, y \in C_c$, the geodesics (not necessarily shortest) connecting $x$ and $y$ lies in $C_c$.
\end{proposition}
\begin{proof}
    Firstly, we claim that $C_c$ is totally convex. 
    Let $x, y \in C_c$, and $\gamma$ any geodesic segment connecting $x$  and $y$ such that $\gamma(0) = x$ and $\gamma(1) = y$, because $b$ is concave, 
    \begin{equation*}
        b(\gamma(t)) \geq (1 - t)b(x) + tb(y) \geq (1 - t)c + tc = c. 
    \end{equation*}
    This means $\gamma \subset C_c$, thus $C_c$ is totally convex. 
    
    Next, we show the set $C_c$ is compact by contradiction. It's enough to prove the statement for $c\le 0$. Note that $b(p)=0$ and hence for any $c\le 0$ the set $C_c=\{b\ge c\}$ is nonempty.
    Suppose the set $C_c$  is non-compact.  Then there exists a sequence $\br{p_i} \in C_c$ such that $d(p, p_i) \to +\infty$. 
    Because $C_c$ is convex, $[pp_i] \subseteq C_c$ for every $i$. 
    By the previous argument, by taking the unit initial vector of $[pp_i]$, which sub-converge in a sphere, the segments $[pp_i]$ also sub-converge to a ray starting at $p$, saying $\gamma_\infty(0) = p$. 
    And $\gamma_\infty \subseteq C_c$. 
    However $b_{\gamma_\infty}(\gamma_\infty(t)) = -t < c$ for large $t$ and $b = \inf\br{b_\gamma: \text{$\gamma$ is a ray starting at $p$}}$. 
    Therefore, $b(\gamma_\infty(t)) = -t < c$, which brings us to a  contradiction.
\end{proof}

\begin{notation}
	By the compactness of Proposition~\ref{prop: total busemann convex and cpt}, we know that $b$ attains the maximum on $C_c$ for each real number. We denote $b_{\max}$ the maximum value and $C_{\max} = \br{b \geq b_{\max}}$ the maximum level set. 
\end{notation}

\begin{proposition}\label{prop-level-sets}
    For $s > t$, $C_s  \subseteq C_t$. Then
    \begin{equation*}
        C_s = \br{x \in C_t: d(x, \partial C_t) \geq s - t},
    \end{equation*}
    Therefore, in particular,
    \begin{align*}
    	\partial C_s 
    	&= \br{x \in M: b(x) = s}\\
    	&= \br{x \in C_t: d(x, \partial C_t) = s - t}.
    \end{align*}
    Thus $b(x) = d(x, \partial C_t) + t$  for $x\in C_t$. 
\end{proposition}
\begin{proof}

	\textbf{Step 1 (Show $C_s \subseteq \br{x \in C_t: d(x, \partial C_t) \geq s - t}$):} Suppose this is false, there exists $x \in C_s$ such that $d(x, \partial C_t) < s - t$, then for some $y \in \partial C_t$, $d(x, y) < s - t$. Since the total Busemann function $b$ is $1$-Lipschitz, we have
\begin{align*}
	b(x) \leq b(y) + d(x, y) < t + s - t = s.
\end{align*}
This means $x \notin C_s$. Contradiction.

\textbf{Step 2 (Show $C_s \supseteq \br{x \in C_t: d(x, \partial C_t) \geq s - t}$):} 
Suppose this is false, then there exists $x \in C_t$ such that $d(x, \partial C_t) \geq s - t$ and $b(x) < s$. 
In fact, we should keep in mind that 
\begin{equation*}
	t \leq b(x) < s
\end{equation*}
Let $\sigma$ be a negative gradient ray of $b$ starting at $x$ such that $b(\sigma(\tau)) = b(x) - \tau$. 
Then, on the other hand, $b(\sigma(s - t)) = b(x) + t - s < t$ (remember $b(x) < s$). 
Therefore, by the intermediate value theorem, there exists $\theta \in [0, s - t)$ such that $b(\sigma(\theta)) = t$. 
This means $\sigma(\theta) \in \partial C_t$. 
By our assumption, we have $d(x, \sigma(\theta)) \geq s - t$. 
However, by the definition of ray, $d(x, \sigma(\theta)) = \theta < s - t$.
Contradiction.
\end{proof}

\begin{proposition}\label{prop: C_max}
$C_{\max}$ has an empty interior.
\end{proposition}
\begin{proof}
If $\br{b \geq b_{\max}}^\circ \neq \varnothing$, then we can take an interior point $x \in \br{b \geq b_{\max}}^\circ$ so that $d(x, \partial\br{b \geq b_{\max}})=\eps > 0$. Take $t=b_{\max}$ and $s=b_{\max}+\eps$. By Proposition \ref{prop-level-sets}, $x\in C_s$ i.e. $b(x)\ge b_{\max}+\eps$. This contradicts the maximality of $b_{\max}$.
Therefore, the set $\br{b \geq b_{\max}}$ has an empty interior.
\end{proof}

\begin{remark}
    The set $\br{b \geq b_{\max}}$ having an empty interior doesn't mean the set must be a point. 
    
    \begin{figure}[htbp]
    \centering
        \includegraphics[width=0.6\textwidth]{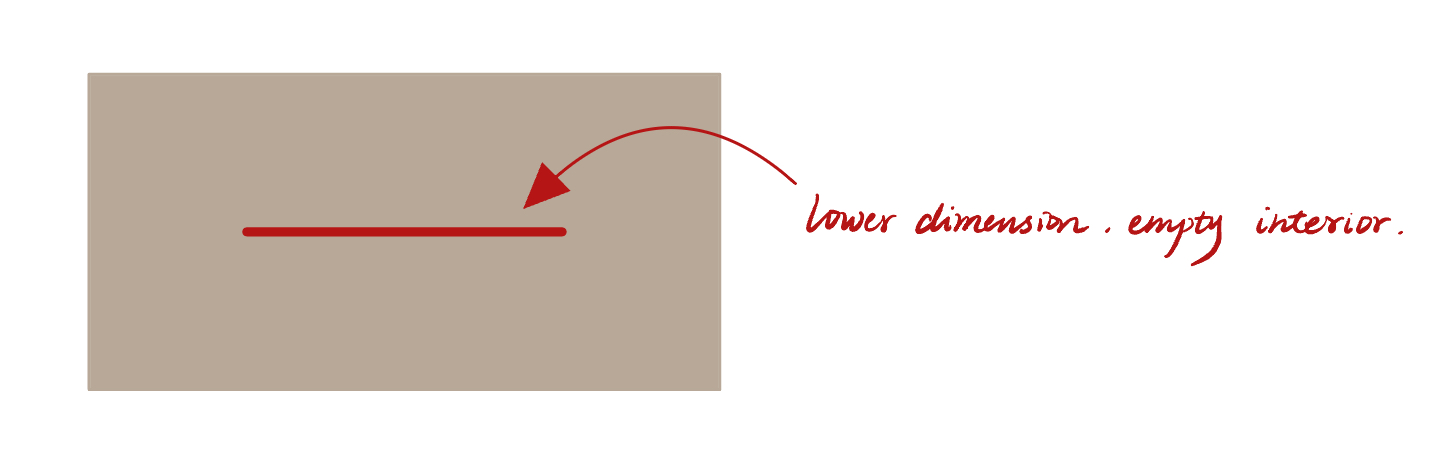}
    \end{figure}

\end{remark}

\subsection{A Brief Plan on the Proof of the Soul Theorem}
The reason why we need to introduce the total Busemann function is that the construction of the set $\br{b \geq c}$ is the key construction in the proof of the Soul theorem. To see this let's briefly discuss our plan for proving the theorem. 

We denote $C_{\max}$ by $C^0$, which is a totally convex subset (By Proposition~\ref{prop: total busemann convex and cpt}). Considering the function $f_0 = d(\cdot, \partial C^0)$, notice that $f_0$ is concave on $C^0$ (We need to check this concavity later) so that $f_0$ attains the maximum. Then we can take $C^1$ to be the maximum level set of $f_0$. Take $f_1 = d(\cdot, \partial C^1)$, ...... We will repeat this construction as long as $C^i$ has a boundary.
    
    Then we obtain a sequence of $\br{C^i}$ such that
    \begin{equation*}
        C^0 \supseteq C^1 \supseteq C^2 \supseteq \cdots
    \end{equation*}
    and
    \begin{equation*}
        \dim(C^0) > \dim(C^1) > \dim(C^2) > \cdots 
    \end{equation*}
    Since the dimensions $C^i$ are strictly decreasing, we eventually will get a $C^i$ such that $\partial C^i = \emptyset$. Which is indeed the Soul of the manifold. Then we will show that $M$ is diffeomorphic to the total space of $\nu(S)$ where $\nu(S)$ means the normal vector bundle of $S$.

    To carry out this process, we need to 
    \begin{itemize}
        \item Understand convex subsets;
        \item  Show that $d(\cdot, \partial C)$ is concave on $S$ is $C$ is a convex subsets of $M$ of $\sect_M \geq 0$.
    \end{itemize}
\section{Convex Subsets in Riemannian Manifolds}
\begin{definition}[Convex Subset]
    $C \subseteq M$ is called \textbf{convex} if $\forall x, y \in C$, there exists a shortest geodesic $[xy]$ in $M$ such that $[xy] \subseteq C$. 
\end{definition}
\begin{definition}[Strongly Convex Subset]
    Let $C \subseteq M$, if $\forall x, y \in C$, any shortest geodesics $[xy]$ connecting $x$ and $y$ will be contained in $C$, then $C$ is called \textbf{strongly convex}.
\end{definition}

\begin{example}
    This is an example of a convex by not strongly convex subset.    Consider $M = S^n$, $C = S_+^n$ an upper hemisphere, then $C$ is convex but not strongly convex.
    \begin{figure}[htbp]
    \centering
        \includegraphics[width=0.3\textwidth]{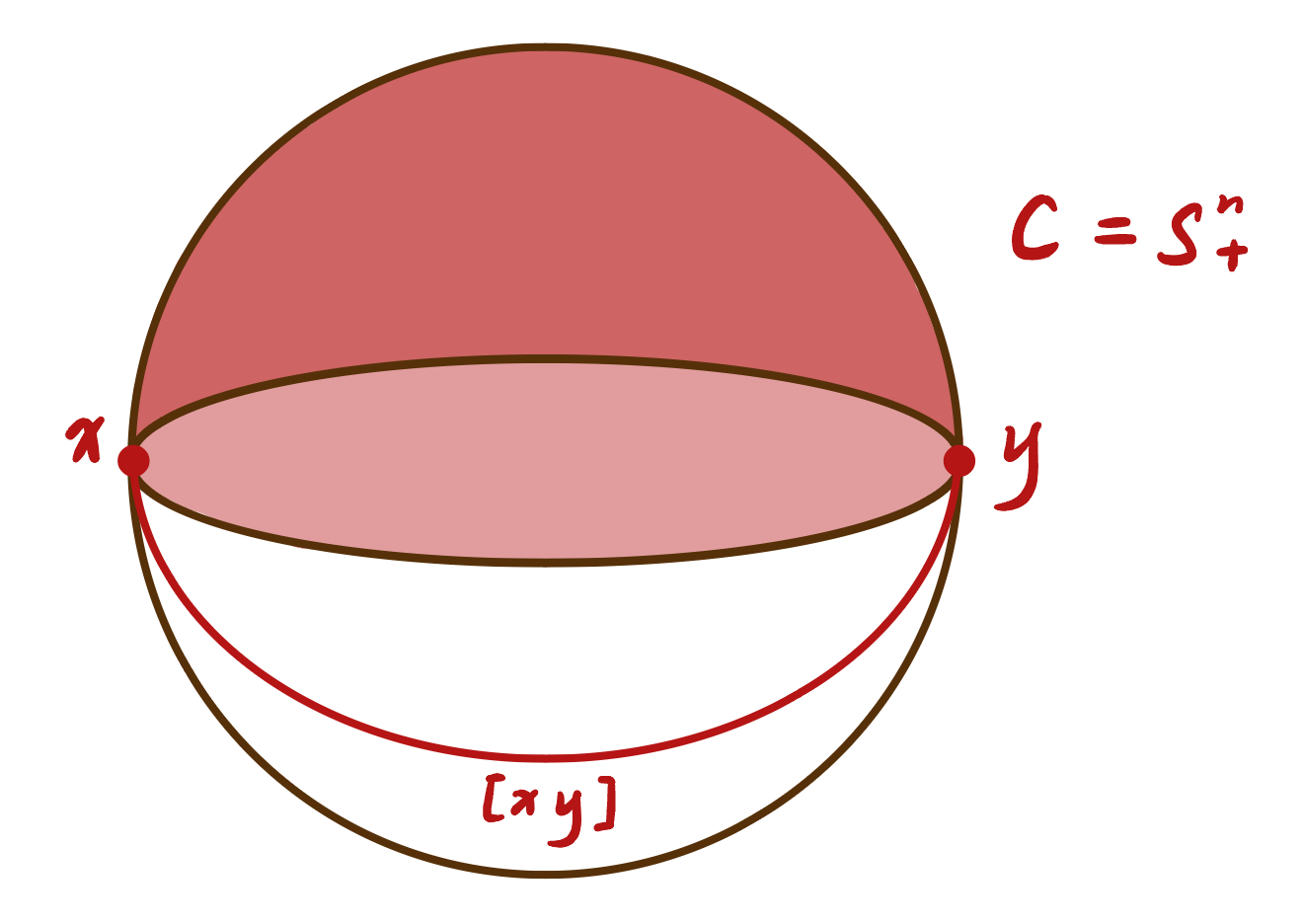}
        \caption{An upper hemisphere is convex but not strongly convex.}
    \end{figure}
\end{example}
\begin{fact}
Any convex subset $C \subseteq M$ is locally strongly convex in the following sense. For any $p \in M$, there exist $\eps > 0$ such that $B_\eps(p)$ is strongly convex. Thus, if $C \subseteq M$ is convex, then for any $p \in C$, there is $\eps>0$ such that $B_\eps(p)\cap C$ is strongly convex. 
   
   \begin{figure}[htbp]
    \centering
        \includegraphics[width=0.3\textwidth]{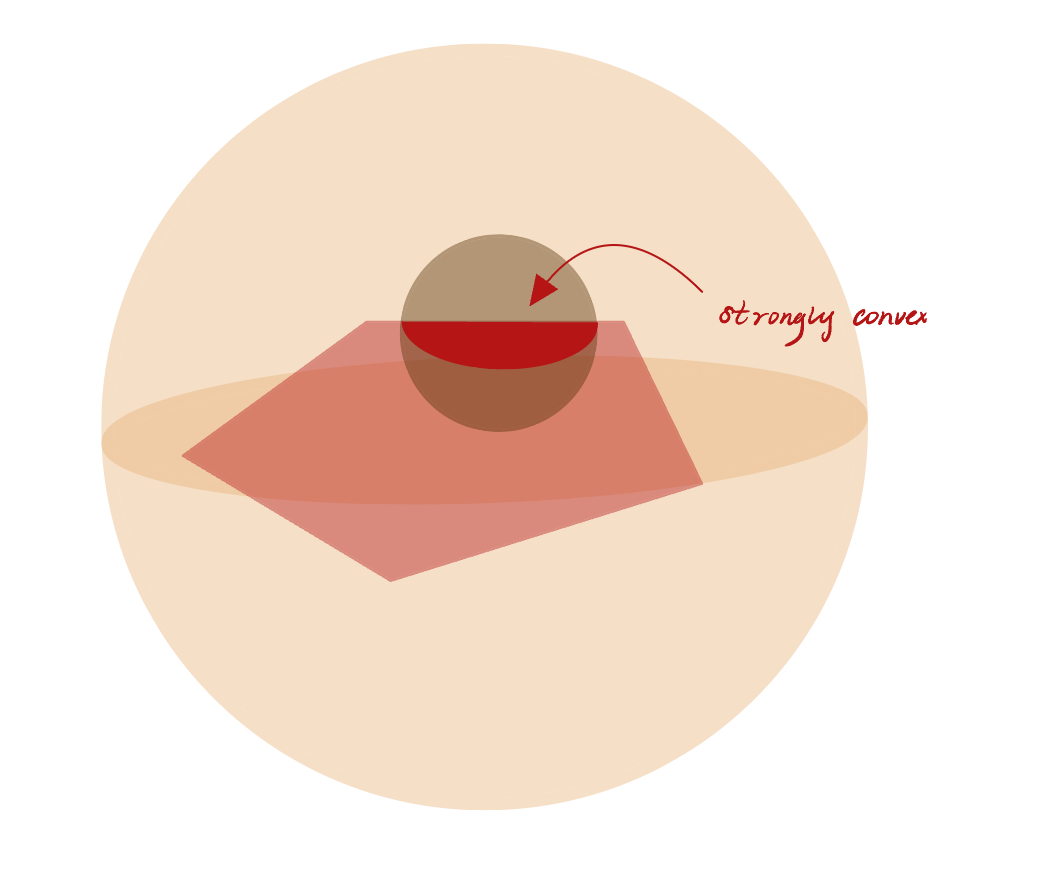}
        \caption{Any convex subset $C \subseteq M$ is locally strongly convex.}
    \end{figure}
    
\end{fact}

\begin{proposition}\label{cla: convex interior}
    If $C \subseteq M$ is convex and closed, then $C$ is a manifold with  boundary, and its manifold interior $C^\circ \subseteq M$ is convex and totally geodesic.
\end{proposition}
\begin{proof}    
     Let $C \subseteq (M^n, g)$ be closed. 
     Take $k$ as the largest possible integer such that there is a $C^\infty$ submanifold $N^k$ of $M$  contained in $C$. Let $p \in N \subseteq C$, we claim that there exists $\eps > 0$ such that $B_\eps(p)\cap N = B_\eps(p)\cap C$. 
     
     \begin{figure}[htbp]
    \centering
        \includegraphics[width=0.6\textwidth]{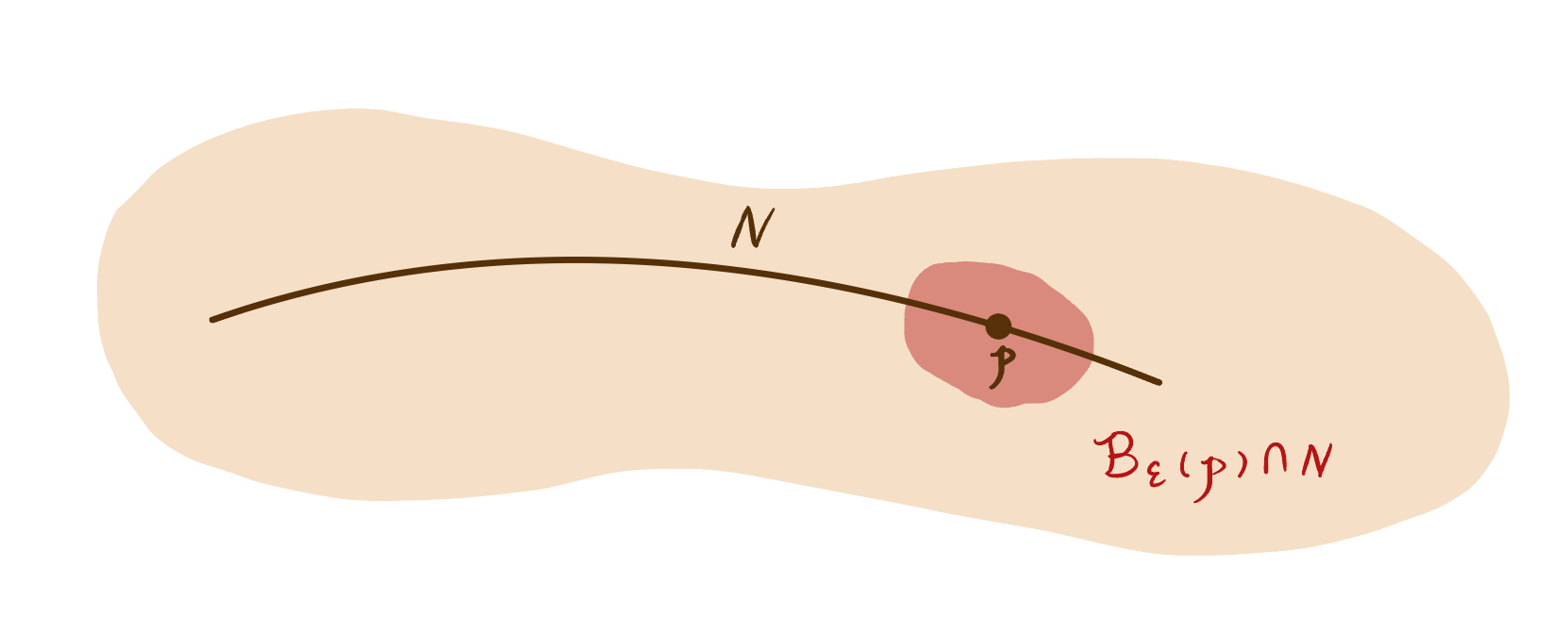}   
     \end{figure}
    
	We prove this claim by contradiction. Suppose not. 
    Then there exists a sequence $p_i \to p$ where $p_i \in C\backslash N$. We can cone off $N$ at $p_i$, take $x \in N$ near $p$. 
    Then $\br{[p_ix]}$ is the collection of shortest geodesics. 
    Take the union (cone)
    \begin{equation*}
        \textbf{Cone}(p, B_\delta(p)\cap N) = \bigcup_{x \in B_\delta(p)\cap N}[p_ix]
    \end{equation*}
    for $\delta$-small. 
    This union will contain $(k + 1)$-dimensional submanifold of $M$. 
    However, we know that $\textbf{Cone}(p, B_\delta(p)\cap N) \subseteq C$ and this contradicts the maximality of $k$.  Therefore, $B_\eps(p) \cap N = B_\eps(p)\cap C$ for small $\eps$. 
    Therefore, $N$ is totally geodesic and locally convex.
    In particular $N$ is $C^\infty$.
    
     \begin{figure}[htbp]
    \centering
        \includegraphics[width=0.8\textwidth]{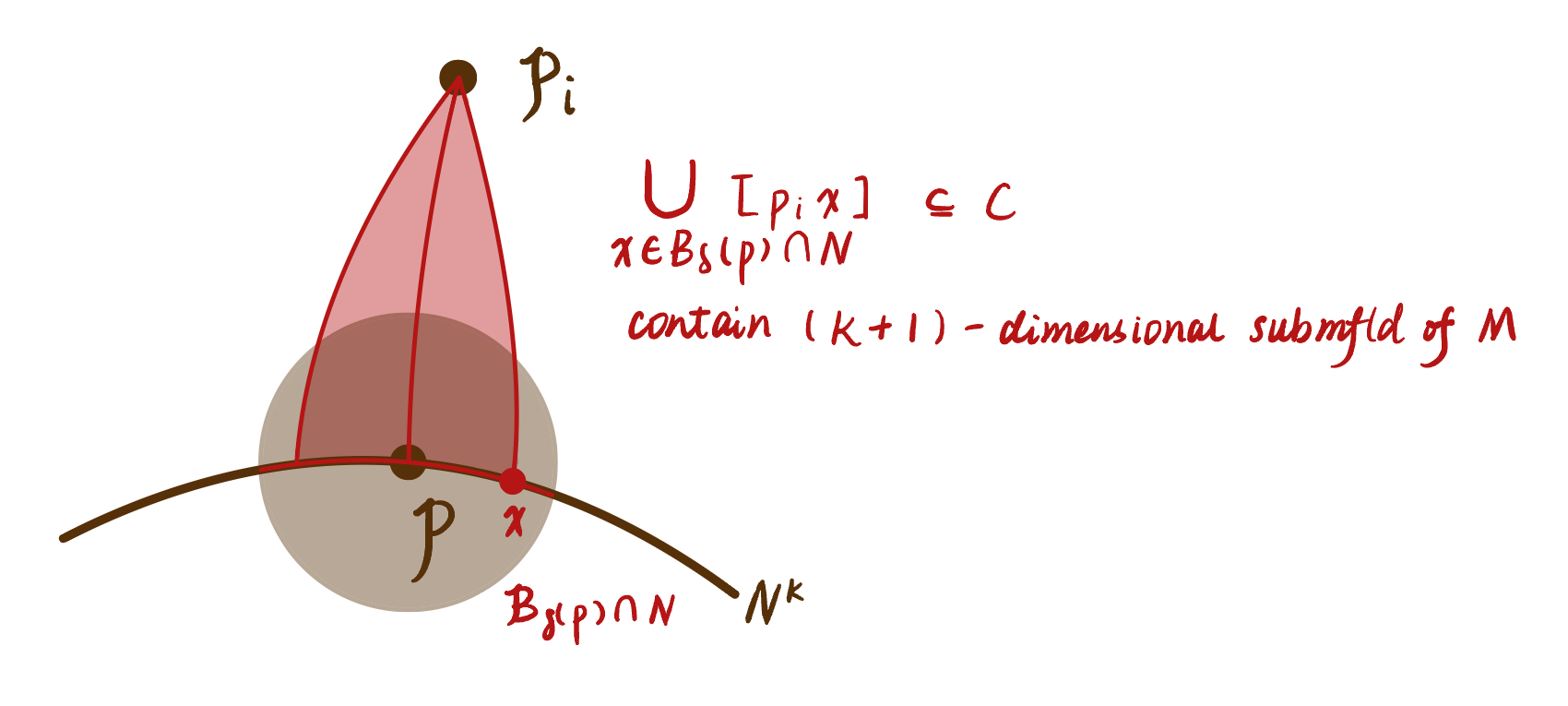}   
     \end{figure}

    We need the following technical Lemma.
    
    \begin{lemma}\label{lem-tech}
    let $C$ be convex and closed, $p\in C\cap \bar N$ and $\delta\ll injrad(p)$ so that for any two points in $B_\delta(p)$ there is a unique shortest geodesic between them and it is contained in $B_\delta(p)$  (such $\delta$ always exists).
    
    Suppose $p', q\in B_\delta(p)$ such that $p'\in C$ and $q\in N$. Let $l=d(p', q)$ and let $\sigma: [0, l]\to M$ be the  unique unit speed shortest geodesic with $\sigma(0)=q, \sigma(l)=p'$.
    
    Then $\sigma([0,l))\subset N$. Moreover, if $p'\notin N$ then $\sigma(l+s)\notin C$ for all small $s$.
    \end{lemma}
    
    \begin{proof}[Proof of Lemma \ref{lem-tech}]
    Let $W^{k-1}\subset C$ be a small  smooth $k-1$-submanifold containing $q$ and transverse to $\sigma$. Such $W$ exists since $q\in N$. Then the punctured cone  $\textbf{Cone}(p', W)\setminus \{p'\}$  is a $k$-manifold, it is contained in $C$ and hence it is contained in $N$. By construction it contains  $\sigma([0,l))$.

    This proves the first part of the lemma. Now suppose $p'\notin N$. If some $p''=\sigma(l+s)\in C$ for some small $s$ then we can repeat the cone construction with $p''$ instead of $p'$ and get that  $p'=\sigma(l)\in N$ which is a contradiction.
    \end{proof}
    
    We are now ready to proceed with the proof of Proposition \ref{cla: convex interior}.
    
    Lastly, let us show that $N$ is convex (and hence is obviously connected).

Let $x,y\in N$. Since $C$ is convex there exists a shortest unit speed geodesic $c: [0, d(x,y)]\to M$ such that $\sigma\subset C$. We claim that $\sigma\subset N$. If not let $t_0$ be the smallest $t$ such that $c(t_0)\in C\setminus N$. Then applying Lemma  \ref{lem-tech} to $p=c(t_0), q=c(t_0-\eps), p'=c(t_0+\eps)$ for a small $\eps$ we conclude that $p\in N$ which is a contradiction. This proves the convexity of $N$.

Next, let us show that $N$ is dense in $C$. Suppose not.
Then there is $p\in C\cap \bar N$ and a small $\delta$ as in Lemma \ref{lem-tech} such that there are $p', q\in B_\delta(p)$ where $q\in N$ and $p'\in C\setminus \bar N$. But then the whole geodesic $ [q ,p')$ is contained in $N$  by  Lemma  \ref{lem-tech} and therefore $p'\in \bar N$. This is a contradiction and hence  $\bar N=C$.

It remains to show that $C$ is a topological manifold with boundary. We will only sketch the argument.

 Look at $\overline{N}\backslash N$, we want to show this is a $(n - 1)$ dimension topological submanifold of $N \subseteq C$ and is relativity open in the topology of $C$. 
    Take $p \in \overline{N}\backslash N$, take $x \in N$ near $p$, take $B_\eps(x)$ for some $\eps$ sufficiently small such that $B_\eps(x) \cap C \subseteq N$. 
    
     \begin{figure}[htbp]
    	\centering
       	\includegraphics[width=0.8\textwidth]{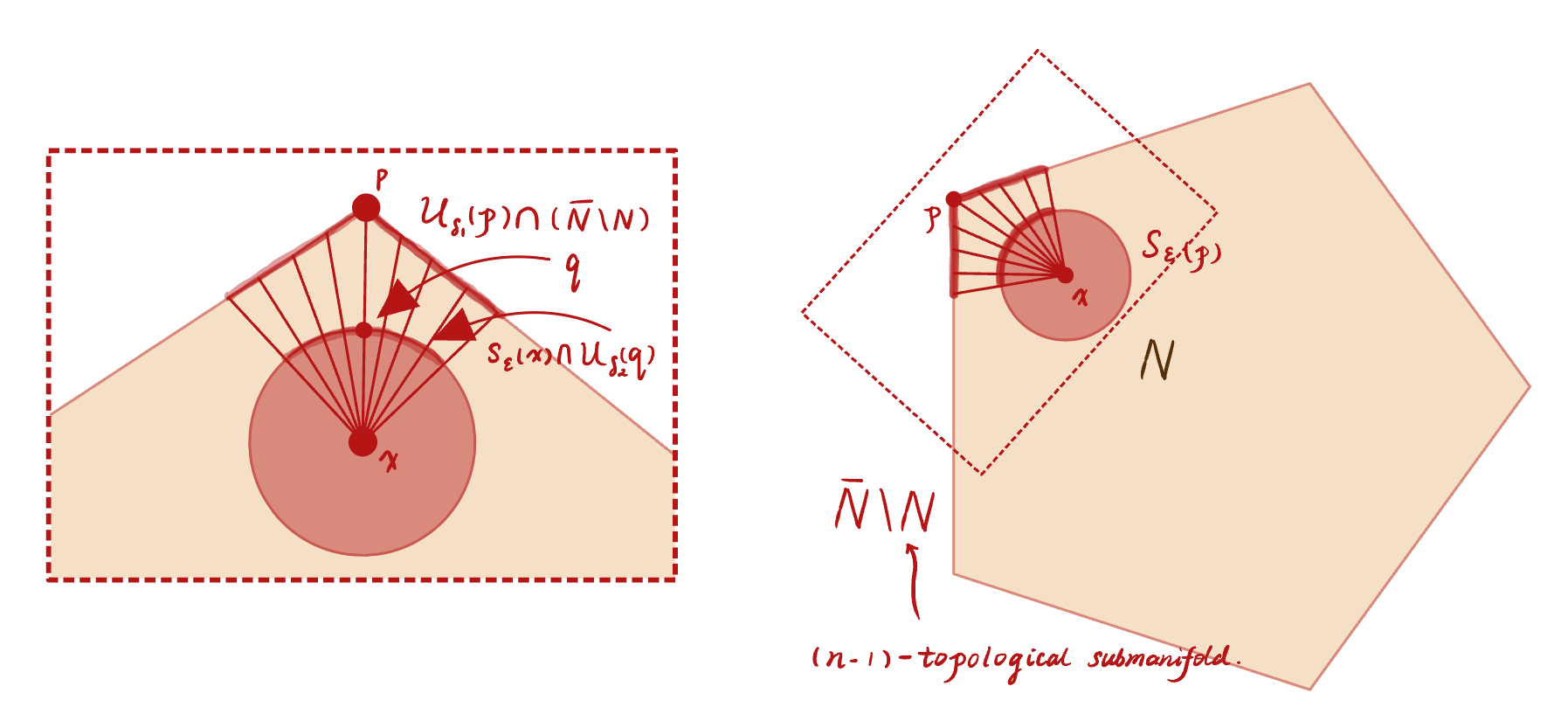}   
     	\end{figure}

    Connect points near $p$ by shortest geodesics with $x$, we can get a homeomorphism 
    \begin{equation*}
        \Psi: U_{\delta_1}(p)\cap(\overline{N}\backslash N) \to S_\eps(x)\cap U_{\delta_2}(q),
    \end{equation*}
    moreover, this map has an inverse map from the neighborhood of $q$ in $S_\eps(x)$ to $U_\eps(p)\cap(\overline{N}\backslash N)$. 
    Then it is not hard to show $\Psi$ is the well-defined and continuous inverse of each other (followed by the uniqueness of points of intersection.). 
    Finally, we can conclude that $C$ is $k$-dimension manifold with boundary. 
    (This is the general result of the convex set of the Riemannian manifold.)
\end{proof}

\begin{example}
	Let $C \subseteq \R^n$ be a closed and convex subset. Then there is an affine subspace $\R^k \subseteq \R^n$ such that $C \subseteq \R^k$ is a $k$-dimensional convex body. It's a manifold with a boundary but the boundary need not be smooth.
\end{example}

\begin{definition}
    Let $C \subseteq M$ be a closed convex subset. Take $p \in C$, and denote
    \begin{equation*}
        T_pC = \br{v \in T_pM: \text{$\exp_p{(tv)} \subseteq C$ for all small $t$}}
    \end{equation*}
    the \textbf{tangent cone of $C$ at $p$}.
\end{definition}
\begin{figure}[htbp]
    \centering
        \includegraphics[width=0.8\textwidth]{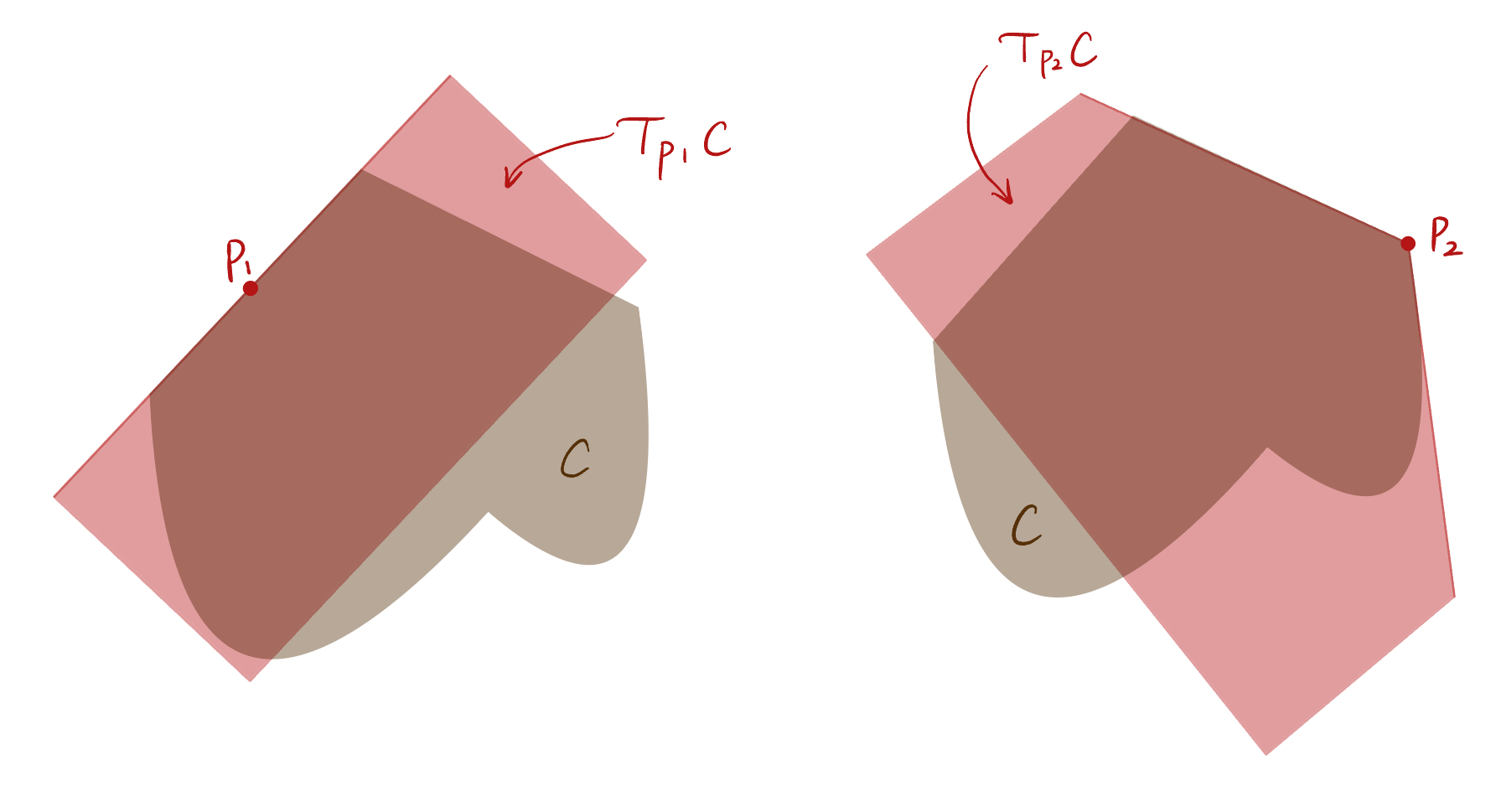}
    \end{figure}

\begin{example}
    Consider the above region $C$, at $p_1$, the tangent cone $T_{p_1}C$ is a half-space. At $p_2$, the tangent cone $T_{p_2}C$ has an angle. 
\end{example}

\section{Concavity of the distance function}
\begin{theorem}\label{thm: distance function to C convex is concave}
    Let $M$ be a manifold with $\sect_M \geq 0$. $C \subseteq M$ be a closed, convex subset with boundary (as a manifold). Denote $f = d(\cdot, \partial C)$ on $C$. Then $f$ is concave on $C$. i.e. for any geodesic $\gamma: [0, l] \to C$, $f(\gamma(t))$ is concave.
\end{theorem}
\begin{proof}
    Fix $t_0 \in (0, l)$, it is enough to prove there exists a supporting linear function $y$ that is always bigger than $f \circ \gamma$ near $t_0$. i.e. $f(\gamma(t_0)) = y(t_0)$ and $f(\gamma(t)) \leq y(t)$ for all $t$ near $t_0$ as in Figure~\ref{fig: supporting concave}. 
    \begin{figure}[htbp]
    \centering
        \includegraphics[width=0.5\textwidth]{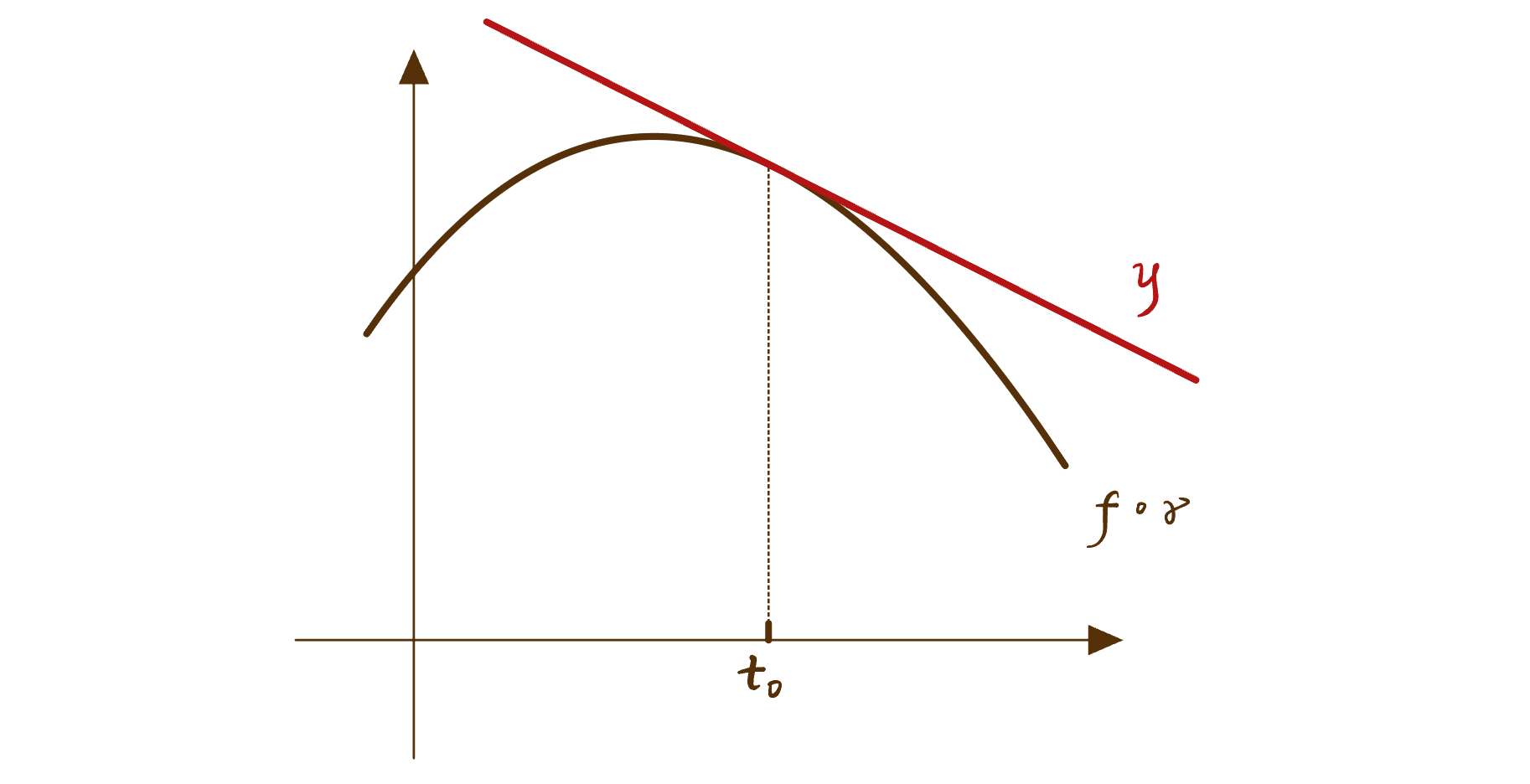}
        \caption{The supporting linear function of the concave function}
        \label{fig: supporting concave}
    \end{figure}

    We are going to start the proof by discussing some geometrical setting. Let $q$ be the closest to $\gamma(t_0)$ on $\partial C$ so that $d(\gamma(t_0), q) = f(\gamma(t_0)) = d(\gamma(t_0), \partial C)$. Let $v \in \partial T_qC$, then we claim that the angle $\alpha$ between $v$ and the initial vector of the geodesic segment $[q\gamma(t_0)]$ starting at $q$, i.e. $\alpha = \mangle(v, \uparrow_q^{\gamma(t_0)})$, must be at least $\frac{\pi}{2}$. 
    
    \begin{figure}[htbp]
    \centering
        \includegraphics[width=0.5\textwidth]{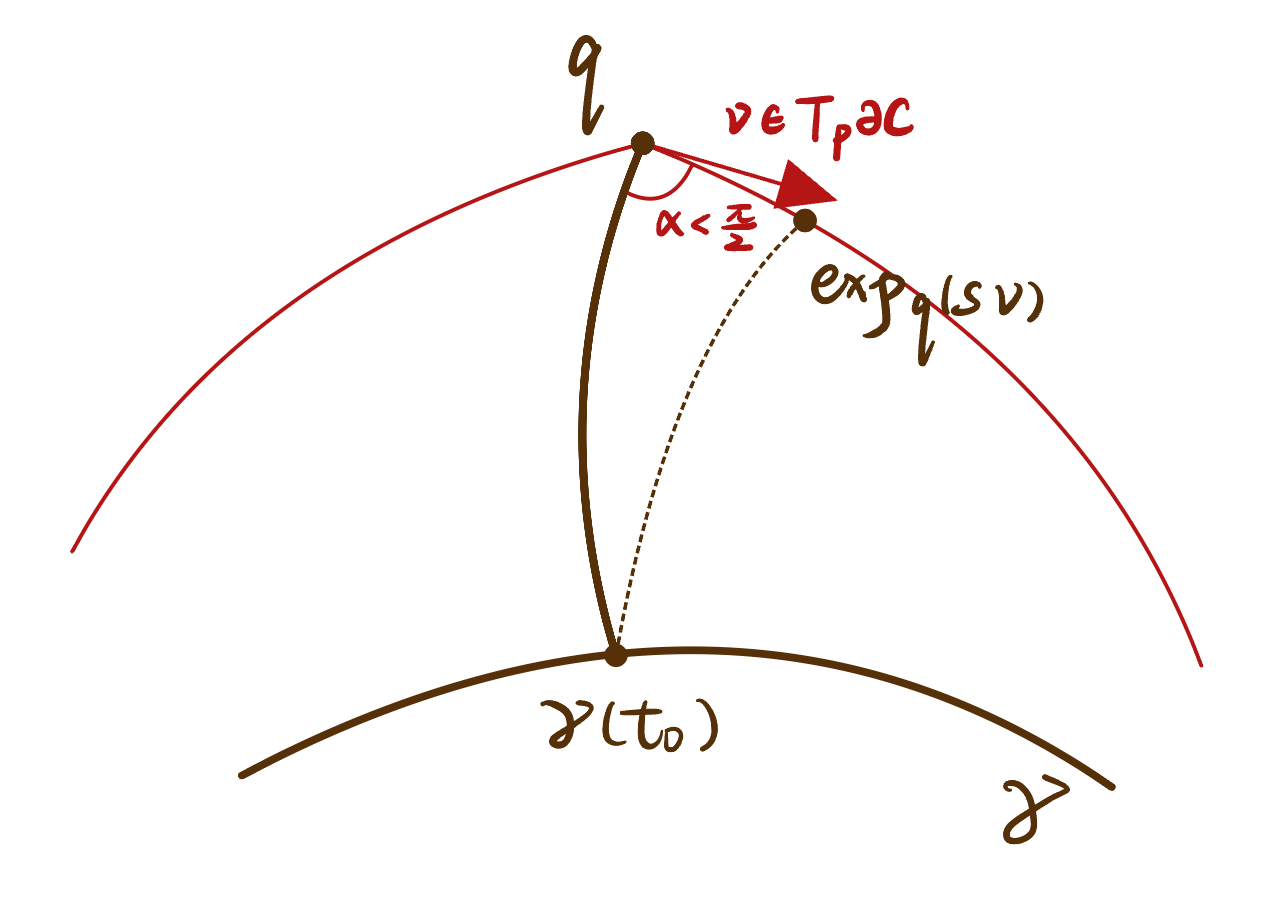}
        \caption{$\alpha < \frac{\pi}{2}$ brings contradiction.}
    \end{figure}

    Suppose not, i.e. $\alpha < \frac{\pi}{2}$. Then by the first variation formula:
    \begin{equation*}
    	\frac{d}{ds}d(\gamma(t_0), \exp_q(sv)) = - \cos{\alpha} < 0.
    \end{equation*}
    This means for sufficiently small $s$, $q$ is no longer the closest point to $\gamma(t_0)$ on $\partial C$, i.e.
    \begin{equation*}
    	d(q, \gamma(t_0)) > d(\exp_q(sv), \gamma(t_0)) \geq d(\partial C, \gamma(t_0))
    \end{equation*}
    Contradiction. 
    
    Remember that $T_qC$ is convex and the subset of some half space $\mathbf{H}_+$. Now we want to claim for $u$ the initial vector of $[q\gamma(t_0)]$, we mush have $u \perp \partial \mathbf{H}_+$. 
    
    \begin{figure}[htbp]
    \centering
        \includegraphics[width=0.6\textwidth]{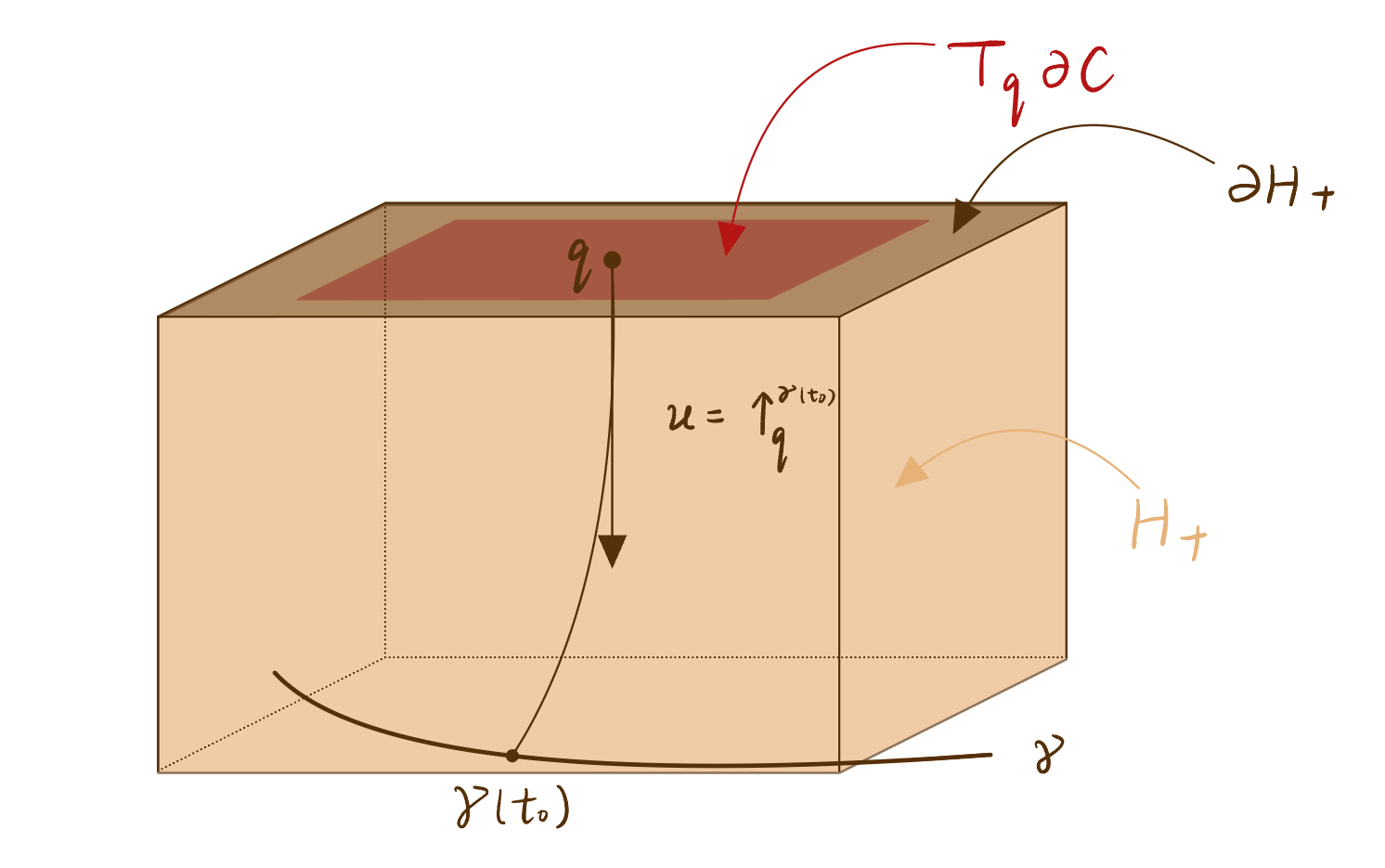}
        \caption{$u \perp \partial \mathbf{H}_+$}
    \end{figure}
    
    Suppose not. Then there will be direction $w\in T_qC$ such that 
    \begin{equation*}
        \alpha = \mangle (u, w) < \frac{\pi}{2}
    \end{equation*}
    which is impossible due to our previous claim. That means $\partial T_pC = \partial \mathbf{H}_+$. Then $[q\gamma(t_0)]$ is unique since it must be perpendicular to $\partial \mathbf{H}_+$. This also shows that the half-space $ \mathbf{H}_+$ containing $T_q(\partial C)$ is unique as well.

    Next, we are going to define the supporting linear function $y(t)$ and show that $f(\gamma(t)) \leq y(t)$ near $t_0$. 
    We parameterize the geodesic segment $[\gamma(t_0)q]$ by a unit speed curve $\sigma: [0, l] \to C$ and consider the angle $\beta = \mangle(\dga(t_0), \uparrow_{\gamma(t_0)}^q)$.
    We define the linear function $y$ as
    \begin{equation}\label{eq: y(t)}
    	y(t) = f(\gamma(t_0)) - (t - t_0)\cos{\beta}
    \end{equation}
    the slope term $\cos{\beta}$ is motivated from the first variation formula \ref{thm: the first variation formula}. Because $f\circ \gamma$ has a derivative at $t_0$, then by the first variation formula
    \begin{equation*}
        f(\gamma(t))' = -\cos{\beta}.
    \end{equation*}
    
We consider 3 cases depending on whether $\beta=\pi/2, >\pi/2$ or $<\pi/2$.
    
 {\bf Case 1. } $\beta = \frac{\pi}{2}$.
    Take $Y(0) = \dga(t_0)$ and extend it to a parallel vector field $Y(t)$ along $\sigma$. Then  $Y(s) \perp \sigma$ for each $s \in [0, l]$. In particular, $Y(l)\perp \sigma(l)$ so that $Y(l) \in T_q(\partial C)$.
    
    \begin{figure}[htbp]
    \centering
        \includegraphics[width=0.6\textwidth]{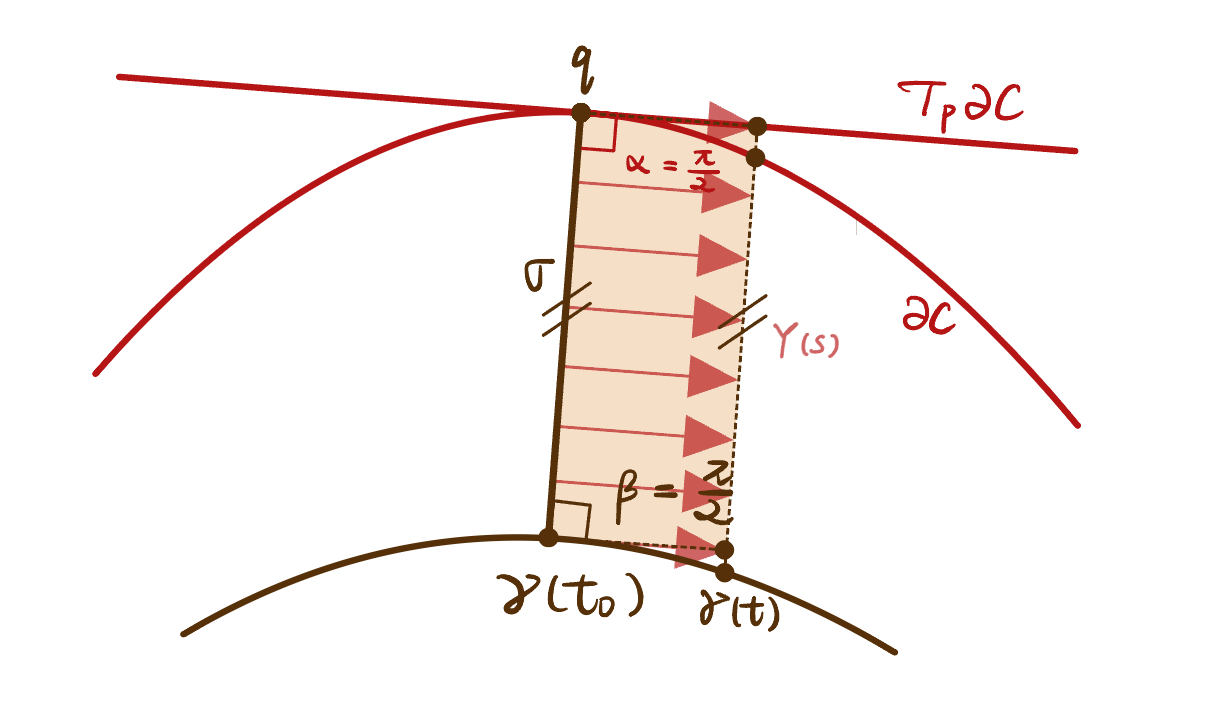}
        \caption{The case when $\beta = \frac{\pi}{2}$}
    \end{figure}
    
     Let $\Sigma(t, s) = \exp_{\sigma(s)}((t - t_0)Y(s))$. We quickly got that 
    \begin{align*}
    	f(\gamma(t)) &\leq \length([Q \gamma(t)])\\
    	&\leq \length([\Sigma(t, l) \gamma(t)]) \quad\text{By Figure~\ref{fig:beta>pi/2}}\\
    	&\leq \length([\tilde{\Sigma}(t, l)\tilde{\gamma}(t)])\quad\text{By Berger's comparison}\\
    	&= \length{(\sigma)} \quad\text{In model space, we have a rectangle when $\beta = \frac{\pi}{2}$}\\
    	& = f(\gamma(t_0)) = y(t) \quad\text{since $\cos{\frac{\pi}{2}} = 0$}
    \end{align*}
    In the Figure~\ref{fig:beta>pi/2} below, because $C$ is concave so that the geodesic connecting $\gamma(t)$ and $\Sigma(t, l)$ intersect $\partial C$ at some point say $Q$. Here we denote $\tilde{\Sigma}(t, l)$ and $\tilde{\gamma}(t)$ the corresponding points of $\Sigma(t, l)$ and $ \gamma(t)$ in the model space.
    
    \begin{figure}[htbp]
    \centering
        \includegraphics[width=0.6\textwidth]{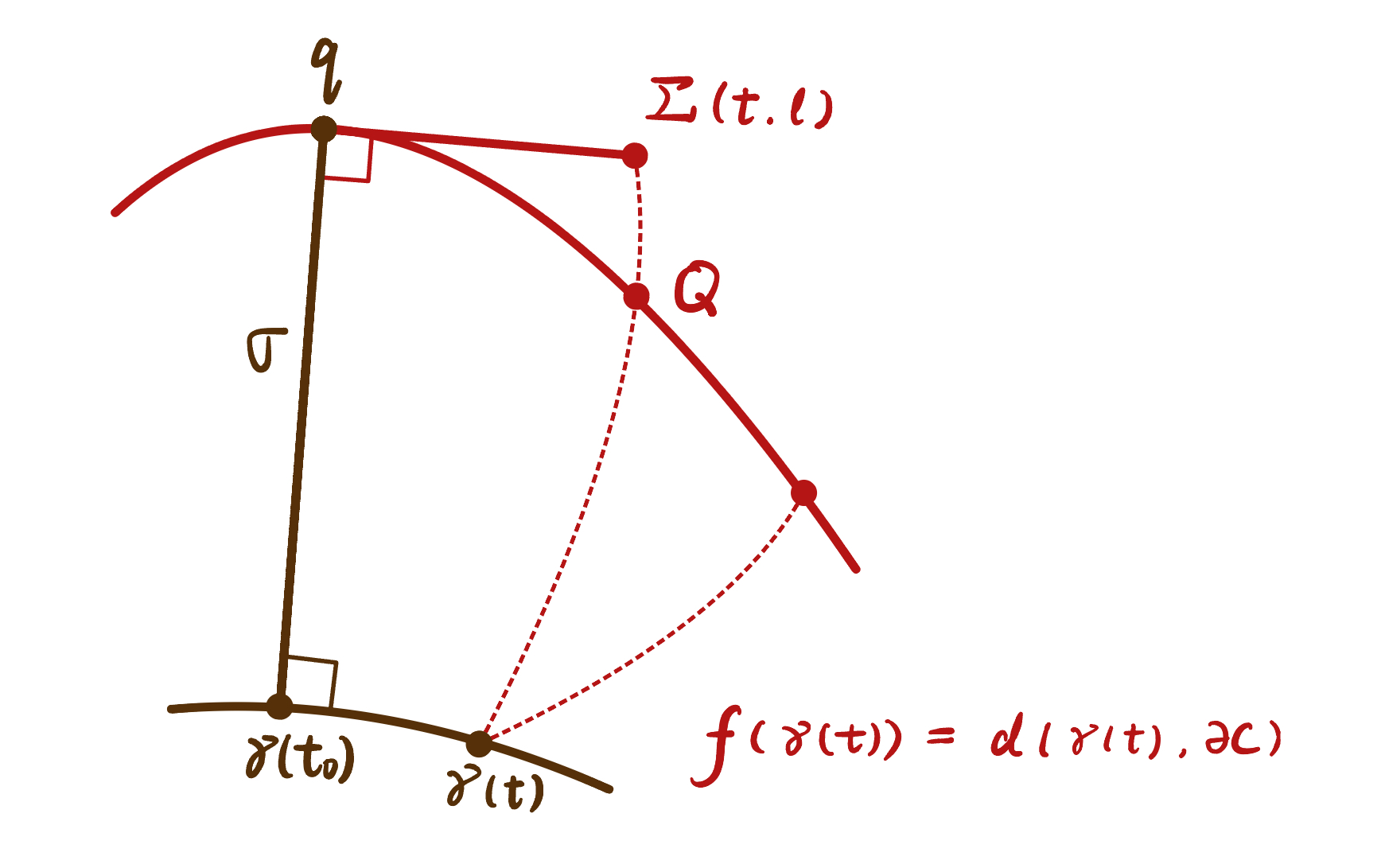}
        \caption{$\length([\gamma(t)Q]) \leq \length([\Sigma(t, l)\gamma(t)])$}
        \label{fig:beta>pi/2}
    \end{figure}

          {\bf Case 2. } $\beta > \frac{\pi}{2}$. In this case, we need to show 
          \begin{equation*}
          	f(\gamma(t)) \leq f(\gamma(t_0))  - (t - t_0)\cos{\beta}
          \end{equation*}
          
          We denote $w = \uparrow_{\gamma(t_0)}^q$ and $w^\perp$ the hyperplane perpendicular to $w$ at $\gamma(t_0)$. We denote $Y(0) = \textbf{Proj}_{w^{\perp}}(\dga(t_0))$ the projection vector of $\dga(t_0)$ to $w^\perp$. Since $\beta > \frac{\pi}{2}$, the vector $\dga(t_0)$ is below the plane $w^\perp$ as we draw in the following picture. 
    
    \begin{figure}[htbp]
    \centering
        \includegraphics[width=1.0\textwidth]{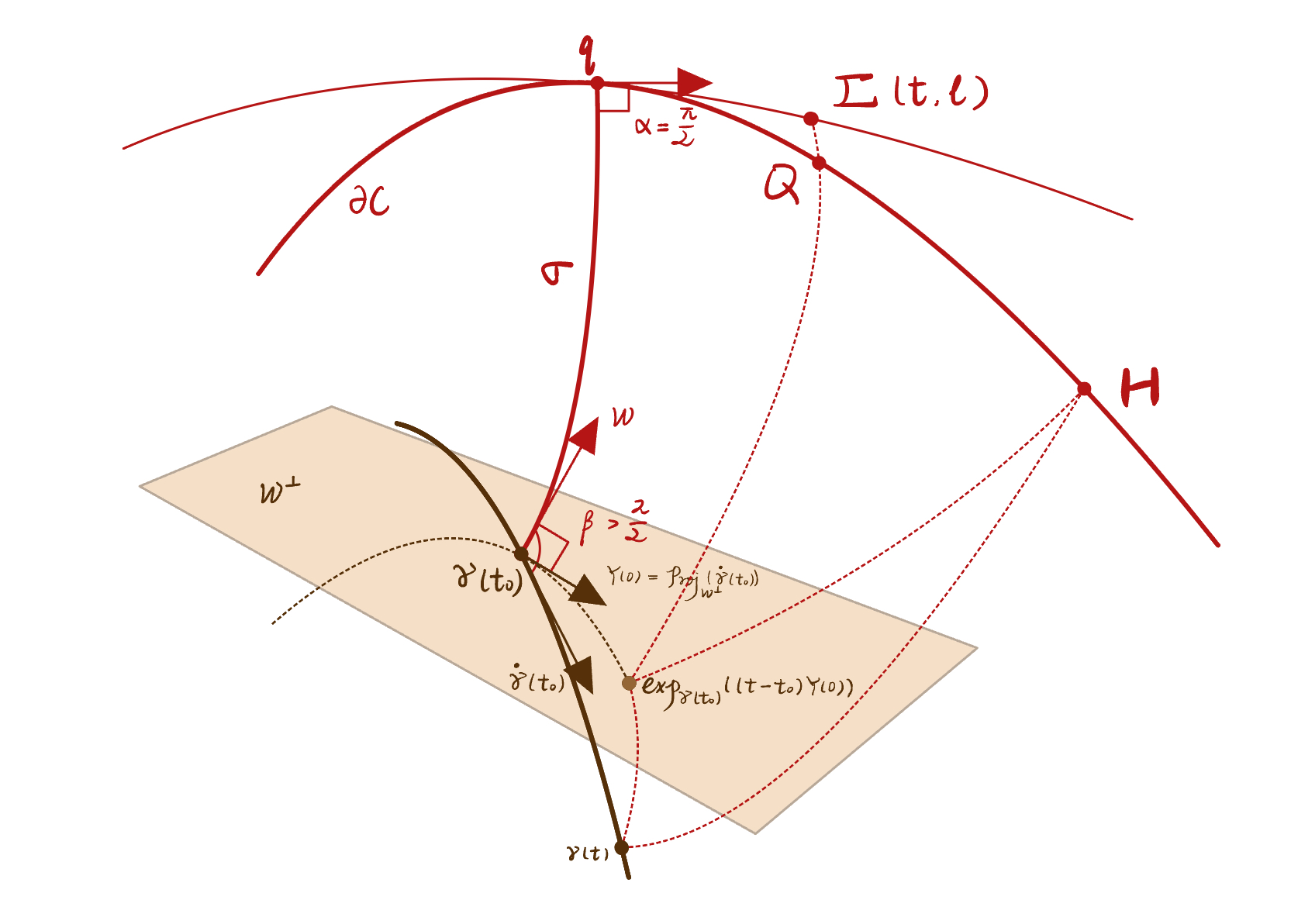}
        \caption{The case when $\beta > \frac{\pi}{2}$}
    \end{figure}

    We also connected $\exp_{\gamma(t_0)}((t - t_0)Y(0))$ with $\gamma(t)$ and with $H \in \partial C$ the point attains the minimum distance from $\gamma(t)$ to $\partial C$. Then by the triangle inequality, we have
    \begin{equation}\label{eq: case beta > pi/2 tri ineq}
    	f(\gamma(t)) \leq d(\exp_{\gamma(t_0)}((t - t_0)Y(0)), H) + d(\gamma(t),\exp_{\gamma(t_0)}((t - t_0)Y(0)))
    \end{equation}
    
    \begin{figure}[htbp]
    \centering
        \includegraphics[width=0.6\textwidth]{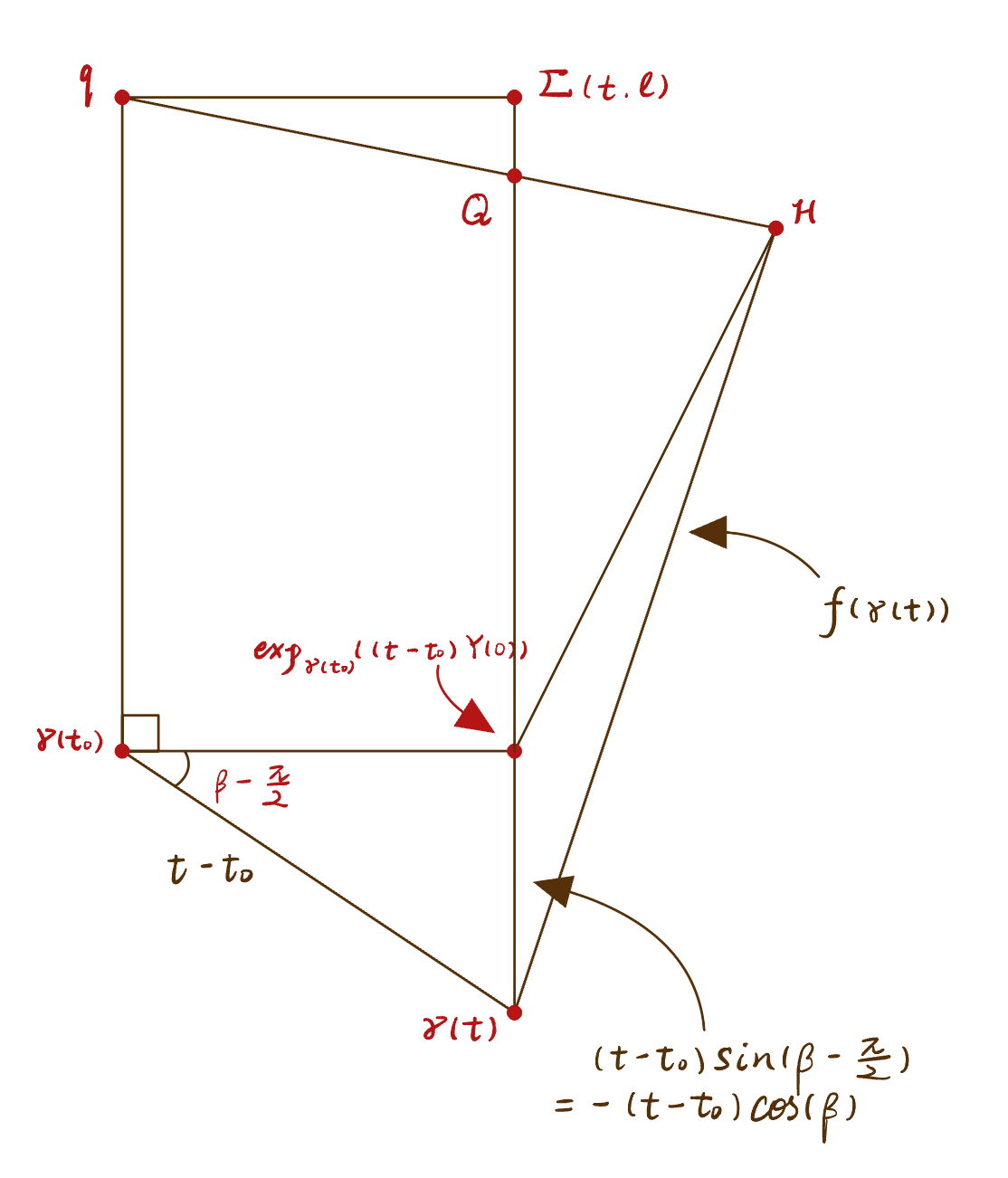}
    \end{figure}
    
    The estimate of the first term in the inequality \ref{eq: case beta > pi/2 tri ineq} is followed straight from the case when $\beta = \frac{\pi}{2}$, we extend $Y(0)$ along $\sigma$ via the parallel transport and construct the map $\Sigma(t, s) = \exp_{\sigma(s)}((t - t_0)Y(s))$. 
    By the case $\beta = \frac{\pi}{2}$, we know that 
    \begin{equation*}
    	d(\exp_{\gamma(t_0)}((t - t_0)Y(0)), H) \leq f(\gamma(t_0)).
    \end{equation*}
    By the Toponogov's hinge comparison \ref{thm: hinge}, we can estimate the second term in the inequality \ref{eq: case beta > pi/2 tri ineq}
    \begin{equation*}
    	d(\gamma(t),\exp_{\gamma(t_0)}((t - t_0)Y(0))) \leq (t - t_0)\sin(\beta - \frac{\pi}{2}) = -(t - t_0)\cos(\beta). 
    \end{equation*}
    Therefore, we can conclude that 
    \begin{equation*}
    	f(\gamma(t)) \leq f(\gamma(t_0)) - (t - t_0)\cos(\beta) = y(t)
    \end{equation*}
    
    {\bf Case 3. } $\beta < \frac{\pi}{2}$. This case is similar to the case when $\beta > \frac{\pi}{2}$. We connect $\gamma(t)$ to $\sigma$ by a geodesic segment $a$ such that $a(0) = \sigma(s_t)$ for some $0 < s_t \leq \length(f(\gamma_0))$. 
    \begin{figure}[htbp]
    \centering
        \includegraphics[width=0.6\textwidth]{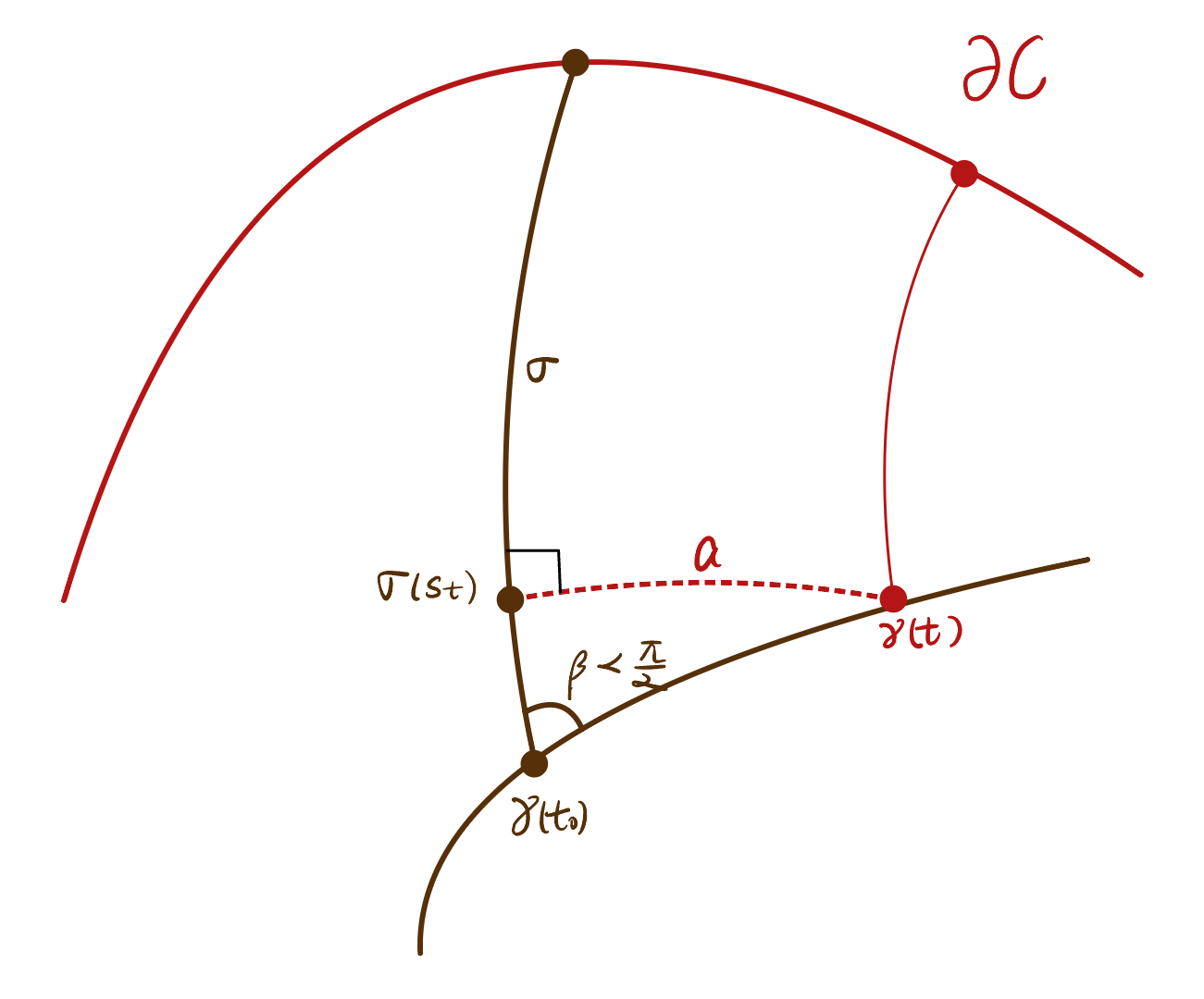}
        \caption{Case 3: $\beta < \frac{\pi}{2}$}
    \end{figure}

    Moreover, we know that $\Dot{a}(0)\perp \Dot{\gamma}(s_t)$ by the first variation formula, thus, from Step 1, we know that
    \begin{equation}\label{eq: concave case 3}
    	f(\gamma(t)) \leq f(\gamma(t_0)) - s_t
    \end{equation}
    On the other hand, by the law of cosine, we have
    \begin{equation*}
    	d(\sigma(s_t), \gamma(t))^2 \leq s_t^2 + (t - t_0)^2 - 2s_t(t - t_0)\cos{(\beta)} 
    \end{equation*}
    Also, since $\Dot{a}(0)\perp \Dot{\gamma}(s_t)$, we also have
    \begin{equation*}
    	(t - t_0)^2 \leq d(\sigma(s_t), \gamma(t))^2 + s_t^2
    \end{equation*}
    Therefore, by adding the above two inequalities, we have
    \begin{align*}
    	& d(\sigma(s_t), \gamma(t))^2 + (t - t_0)^2 \leq s_t^2 + (t - t_0)^2 - 2s_t(t - t_0)\cos{\beta} + d(\sigma(s_t), \gamma(t))^2 + s_t^2\\
    	\implies & s_t(t - t_0)\cos{(\beta)} \leq s_t^2\\
    	\implies & (t - t_0)\cos{(\beta)}  \leq s_t
    \end{align*}
    Therefore, the inequality~\ref{eq: concave case 3} becomes,
    \begin{equation*}
    	f(\gamma(t)) \leq f(\gamma(t_0)) - (t - t_0)\cos{(\beta)} = y(t)
    \end{equation*}

\end{proof}


\section{The Proof of the Soul Theorem}
In the last section, we proved that the distance function to the boundary of a convex subset of a manifold $M$ of $\sect_M \geq 0$ is concave on that subset. That is  theorem \ref{thm: distance function to C convex is concave}. We are going to prove the soul theorem in this section. 

\begin{lemma}\label{lem: flat rectangle}
	Let $(M, g)$ be a Riemannian manifold of $\sect_M \geq 0$, $C \subseteq M$ a closed  convex subset. We denote $f: d(\cdot, \partial C): C \to \R$. Suppose there exists $\gamma: [0, a] \to C$ such that $f\circ \gamma \equiv d$ is a constant. Then there is a flat totally geodesic rectangle $\Gamma: [0, a] \times [0, d] \to C$ such that $\Gamma(t, 0) = \gamma$ and $\Gamma(t, d) \subseteq \partial C$. 
\end{lemma}
\begin{proof}
We connect $q \in \partial C$ and $p = \gamma(0) \in C^\circ$ by the shortest geodesic segment $\sigma: [0, d] \to C$. Take $Y = \frac{\dga(0)}{\norm{\dga(0)}}$ the unit initial vector along $\gamma: [0, a] \to C$. If we extend $Y$ to a  parallel vector field along $\sigma$. The extension is denoted as $Y(t)$. Then we take the exponential map along $\gamma$: 
\begin{equation*}
    \Gamma(t, s) = \exp_{\gamma(t)}(sY(t)).
\end{equation*}
In particular, $\Gamma(0, s) = \sigma(s)$ for $s \in [0, d]$. 

\begin{figure}[htbp]
    \centering
        \includegraphics[width=0.6\textwidth]{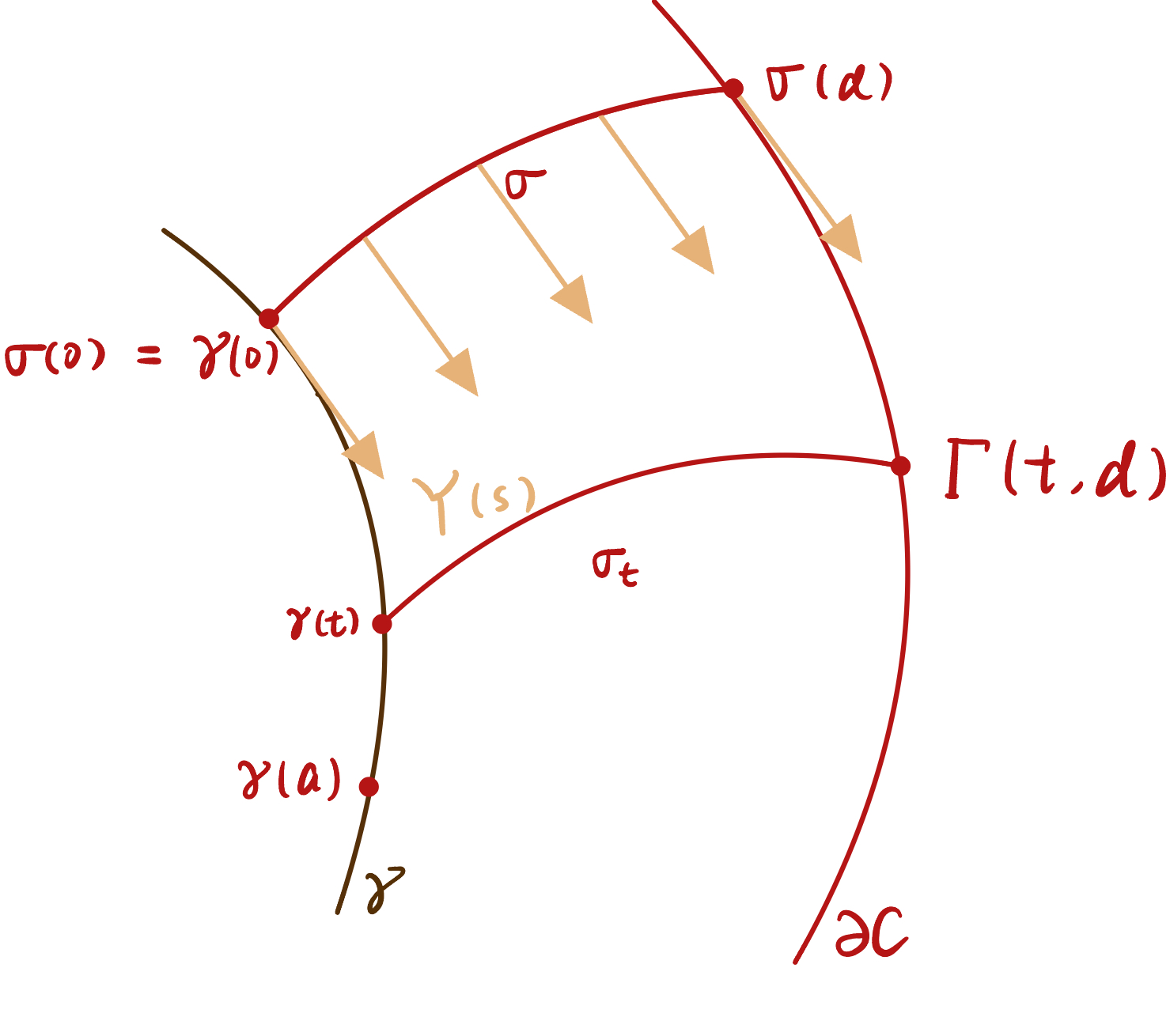}
        \caption{For Lemma~\ref{lem: flat rectangle}}
    \end{figure}

Fix $t \in [0, a]$, and we denote $\sigma_t(s) = \Gamma(t, s)$ to represent the curve connecting $\gamma(t)$ to the boundary. 
, we have, 
\begin{align*}
	f(\gamma(0)) &= f(\gamma(t)) \quad\text{since $f\circ \gamma \equiv d$}\\
	&\leq \length{(\sigma_t)}\\
	&\leq \length{(\sigma)}\quad\text{by Berger's comparison for $\sect_M \geq 0$~\ref{berger-k=0}}\\
	&= f(\gamma(0))
\end{align*}
Thus all the above inequalities are equalities. Then by the Rigidity case of Berger's comparison~\ref{lem: rigidity case}, we can conclude that 
\begin{equation*}
    R = \br{\Gamma(t, s) : t \in [a, b], s \in [0, d]}
\end{equation*}
is a totally geodesic flat rectangle in $C$ and in $M$. 
\end{proof}

Now we are ready to prove the soul theorem by Cheeger and Gromoll. 
Let's first recall the statement of the soul theorem. 
\begin{theorem}[Soul Theorem]\label{thm: soul theorem}
    Let $(M^n, g)$ be a Riemannian manifold of $\sect_M \geq 0$ and it is complete and non-compact. Then there exists a closed totally convex subset $S \subseteq M^n$ which is also a totally geodesic submanifold such that 
    \begin{equation*}
        M^n \stackrel{\text{diff}}{\cong} \nu(S)
    \end{equation*}
    where $\nu(S)$ is the total space of the normal bundle of $S$. In particular, $M^n$ has compact type since 
    \begin{equation*}
    	M^n \stackrel{\text{diff}}{\cong} \nu^1(S)^\circ
    \end{equation*}
    where 
    \begin{equation*}
    	\nu^1(S)^\circ = \br{(x,v)\mid x\in S, v\perp T_xS, \abs{v}<1}
    \end{equation*}
    is the interior of the unit disk bundle
    \begin{equation*}
    	\nu^1(S) = \br{(x,v)\mid x\in S, v\perp T_xS, \abs{v} \le 1}
    \end{equation*} 
    which is a compact manifold with a boundary. 
\end{theorem}

\begin{proof}
    Take $p \in M$ and $b$ the total Busemann function at $p$, i.e.
    \begin{equation*}
        b(x) = \inf{\br{b_\gamma(x): \text{$\gamma$ is a ray starting at $p$}}}
    \end{equation*}
    We know that 
    \begin{itemize}
        \item $b$ is concave;
        \item The superlevel set of $b$ are compact,i.e.
        \begin{equation*}
            C_t = \br{b \geq t}
        \end{equation*}
        is compact for every $t$.
    \end{itemize}
    Thus, we can take the maximum value $b_{\max}$ of $b$.
    By concavity of $b$ the set $C_{\max} =\{b\ge b_{\max}\}=\{b= b_{\max}\}$ is compact and totally convex.
    Moreover, by Proposition~\ref{prop: C_max}  $C_{\max}$ has empty interior.
    Set $C^0: = C_{\max{\empty}}$,  the maximum superlevel set. 
    Take $f^0 = d(\cdot, \partial C^0)$ on $C^0$. And then we take $C^1$ to be the maximum level set of $f^0$ on $C^0$. And we know that $C^1$ is also totally convex in $C^0$. Hence in $M$. We keep going through this process until we get to an $i$ such that $\partial C^i = \varnothing$. We have that
    \begin{equation*}
        \dim{(C^1)} > \dim{(C^2)} > \cdots > \dim{(C^i)}.
    \end{equation*}
    The dimension goes down because $C^{j + 1}$ has an empty interior in $C^j$ for  each $j = 1, 2, \dots$
    Then at some point, we must stop because we have some $i$ such that $\partial C^i = \emptyset$. And we denote $S = C^i$ where $C^i = \emptyset$ is the \textbf{soul} of $M$. We are going to finish our proof by showing the following two claims.
    
        \begin{claim}\label{cla: dist soul has non critical point}
        $d(\cdot, S)$ has no critical points on $M\backslash S$. 
    \end{claim}
    \begin{claim}\label{cla: soul diff V(S)}
        For this $S = C^i$, 
        \begin{equation*}
            M \stackrel{\text{diff}}{\cong} \nu(S)
        \end{equation*}
    \end{claim}

    Our plan is to prove  claim \ref{cla: dist soul has non critical point} first, then use this claim to prove  claim \ref{cla: soul diff V(S)}. 
    \begin{proof}[ Proof of Claim \ref{cla: dist soul has non critical point}]
	Let $x\in M\setminus S$.  By our construction, we can find $C$ to be one of $\br{d(\cdot, \partial C^i)\ge c}$ or $C=\br{b \geq c}$ such that $x$ lies in the manifold boundary of $C$. we can take all shortest geodesic from $x$ to $S$, and $u$ the initial vectors of such shortest geodesics ($u \in \Uparrow_x^S$), then $u \in \R^n_+ = \R^{n - 1}\times [0, \infty)$ and points straightly inside. Then we can conclude that $u \notin \partial \R^n_+ = \R^{n - 1}$ because $u \in T_xC^\circ$ and $\overline{T_xC} \subseteq \R^n_+$. We can argue by contradiction to see why this conclusion is true. If $u$ is not in the interior $T_xC$, then we can take $f = d(\cdot, \partial C)$ and $\gamma(t) = \exp_x(tu)$. If $f$ is concave along $\gamma$ and $l(\gamma) = d\cdot f(\gamma(0)) < f(\gamma(d)) \implies \text{$f(\gamma(t))$ is concave.}$ This implies that $f(\gamma(t))'|_{t = 0} > 0$ by concavity. Then $\gamma$ cannot be tangent to $\partial C$ otherwise $f(\gamma(t))'|_{t = 0} = 0$. Contradiction. So we can conclude that $u$ is not tangent to $\partial C$, then
        \begin{equation*}
            \alpha = \mangle (u, N) > \frac{\pi}{2} \quad \forall u \in \Uparrow_x^S
        \end{equation*}
        where $N$ is the outward normal vector at $x$ to the hyperplane $\partial \R^n_+$. Then $d(\cdot, S)$ is regular at $x$. 
    \end{proof}
    \begin{proof}[Proof of Claim~\ref{cla: soul diff V(S)}]
        Take a small $\eps > 0$ so that 
        \begin{equation*}
            \overline{B_\eps(S)} \stackrel{\text{diff}}{\cong} \overline{B_\eps(\text{$0$-section in $\nu(S)$})}
        \end{equation*}
        via the normal exponential map. Then the boundary of the $\eps$-neighbourhood of $S$ $\partial(\overline{B_\eps(S)}) = \br{f = \eps}$ where $f = d(\cdot, S)$ on $M$ and $    \overline{B_\eps(S)} = \br{f \le \eps}$ is also a manifold with boundary. And outside $B_\eps(S)$, $f$ has no critical points. That implies
        \begin{equation*}
            M^n \stackrel{\text{diff}}{\cong}\overline{B_\eps(S)}\cap (S_{\eps}(S)\times {[0,\infty)})\stackrel{\text{diff}}{\cong} B_\eps(S).
        \end{equation*}
        Here we apply the Remark~\ref{rem: crit-diffeomophism} in the critical point theory that we discussed before. 
    \end{proof}
\end{proof}
\begin{corollary}
    $\pi_i(M^n)$ ad $H_i(M^n)$ are finitely generated for any $i$. because compact manifolds with boundaries have finitely generated homotopy and homology groups. 
\end{corollary}
\begin{corollary}\label{cor: point soul}
    If $S = \br{pt}$, then $M^n \stackrel{\text{diff}}{\cong} \R^n$. 
\end{corollary}
\begin{corollary}\label{cor: positive curvature manifold is Euclidean}
    If $M$ is an open complete manifold of $\sect_M > 0$, then $S = \br{pt}$ and hence $M^n \cong \R^n$. 
\end{corollary}
\begin{proof}[Proof of Corollary~\ref{cor: positive curvature manifold is Euclidean}]
    We claim that $C_{\max}$ of $b$ is a point. If not, then there exists a non-constant geodesic $\gamma: [0, a] \to C_{\max}$. Fix $c < b_{\max}$, so that
    \begin{equation*}
    	C_c = \br{b \geq c} \supseteq \br{b \geq b_{\max}} = C_{\max}. 
    \end{equation*}
    By Propositions \ref{prop: C_max} and \ref{prop-level-sets}, we know that $C_{\max}$ has an empty interior. So that for any $x \in C_{\max}$, 
    \begin{align*}
    	d(x, \partial C_c) = b_{\max} - c
    \end{align*}
    That means $d(\gamma(t), \partial C_c)$ is a constant. Thus by Lemma~\ref{lem: flat rectangle} we know that there exists a flat totally geodesic rectangle $R$ parametrized by $\Gamma: [0, a] \times [0, d] \to C_c$ such that $\Gamma(t, 0) = \gamma$ and $\Gamma(t, d) \subseteq \partial C_c$. Since the $R$ is flat, for any tangent plane $\sigma$ on $R$, we have $\sect^M(\sigma) = 0$. Because this is totally geodesic, then by the Gauss formula the intrinsic sectional curvature equals the extrinsic sectional curvature. To see why, take $\sigma = \mathbf{Span}\br{(X, Y)}$ tangent to $N$.
    \begin{align*}
        X, Y \in T_pN \implies & \sect^N(\sigma) = \sect^M(\sigma) + \inp{II(X, X), II(Y, Y)} - \inp{II(X, Y), II(X, Y)}
    \end{align*}
    Because, 
    \begin{equation*}
        \text{Totally Geodesic} \iff II = 0
    \end{equation*}
    Therefore, 
    \begin{equation*}
        \sect^N(\sigma) = \sect^M(\sigma).
    \end{equation*}
    In our case, $N$ is a flat rectangle, then 
    \begin{equation*}
        \sect^N(\sigma) = \sect^M(\sigma) = 0.
    \end{equation*}
    This is impossible since $\sect_M > 0$. Therefore, on the very first step in the soul construction. $C_{\max} =\br{pt} \implies S = \br{pt}$. Therefore, by the corollary \ref{cor: point soul}, we have $M^n \cong \R^n$.
\end{proof}
\section{Examples and Open Problems}
\begin{example}
    If $N$ is closed and $\sect_N \geq 0$ and $G$ a compact Lie group $G \acts N$ freely by isometries. Suppose further that we have a representation $G \stackrel{P}{\longrightarrow} O(k)$ so that $G \acts \R^k$ by isometries. Consider $M = (N \times \R^k)/G$ by the diagonal action. Then $\sect_M \geq 0$ and its soul is $ S = (N \times \br{0})/G = N/G \subseteq N \times \R^k/G$.
\end{example}
\begin{example}
    Take $N = S^1$, $k = 1$, $S^1 \times R$ admits a soul $S^1 \times \br{t}$ for any $t$. Let $G = \Z_2 = \br{\pm 1}$. $-1$ acts on $S^1 = \br{z \in \C: \abs{z} = 1}$ such that $-1(z) = -z$ and on $x \in \R$, $-1(x) = -1$. So $\Z_2 \acts (S^1 \times \R)$ by
    \begin{equation*}
        -1(z, t) = (-z, -1)
    \end{equation*}
Then $M =(S^1 \times \R)/\Z_2$ is an M\"obius band and by construction, it has $\sec\equiv 0$. The  soul  of $M$ is the circle $S^1 \times \br{0}/\Z_2$.
\end{example}
\begin{remark}
    The difference between $\sect_M \geq 0$ and  $\sect_M > 0$ is indeed large for noncompact manifolds. 
    \begin{itemize}
        \item If $M$ is open then
        \begin{equation*}
            \begin{cases}
                \text{$\sect_M > 0$ is very restrictive $M \stackrel{\text{diff}}{\cong} \R^n$}\\
                \text{$\sect_M \geq 0$, then there are many non-trivial examples.}
            \end{cases}
        \end{equation*}
        \item  If $M^n$ is closed, if $\sect_M > 0$ then since the manifold is closed (compact), we have a universal constant $0 < \delta \leq \sect_M$. Hence the same holds for the universal covering space $\tilde{M}^n$, then by Bonnet-Myer's theorem, the diameter of $\tilde{M}^n$ is finite, so $\tilde{M}^n$ must also be compact and hence $\pi_1(M)$ is finite. 
        In particular, $M=T^n$ admits a metric of $\sect_M \equiv 0$ but does not admit a metric of $\sect_M >0$.

But it's an open question is whether there exists a closed simply connected manifold $M^n$ which admits a metric of $\sect_g \geq 0$ but does not admit a metric $h$ of $\sect_h > 0$.

There exist many examples of closed simply connected  $(M^n, g)$ with  $\sect_g \geq 0$ and $\sect_g > 0$ on some open dense set. In particular, such a metric exists on the Gromoll-Meyer sphere \cite{GM74} and on $S^2\times S^3$ \cite{Wil02}. 
However, attempts to construct metrics of positive sectional curvature on these manifolds have been proved unsuccessful.

Hopf conjectured that  $S^2 \times S^2$ does not admits a metric $g$ of $\sect_g > 0$. A more general version of Hopf's conjecture says that if $M_1, M_2$ are closed manifolds then $M_1\times M_2$ does not admit a metric of positive sectional curvature. If this conjecture is true then $S^2 \times S^2$ would provide an example of a simply connected manifold that admits a metric  $\sect\ge 0$ but does not admit a metric of $\sect>0$. Moreover, this suggests that the above result of $S^2 \times S^3$ by Wilking is sharp.
    \end{itemize}
\end{remark}

Next, let us mention some questions related to the Soul theorem. 

Given a closed manifold $S$ of $\sect_S \geq 0$ and a vector bundle 
\begin{equation*}
    \begin{tikzcd}
        \R^k \arrow{r}{\empty} & E \arrow{d}{\empty} \\%
        & S
    \end{tikzcd}
\end{equation*}
in view of the soul theorem, it is natural to ask whether $E$ admits a metric $g$ of $\sect_g \geq 0$ such that $S$ is a soul. This question is partially solved. For example if $\pi_1(S) \neq 0$, then there are easy counterexamples that follow from the splitting theorem.

 However, the question is still open when $\pi_1(S) = 0$. 

A special case of question is whether all vector bundles over spheres admit $\sect_E \geq 0$. 
By a result of Rigas \cite{Rig78} it is known that this is true stably. That is, give a vector bundle $\R^k\to E\to S^n$  there exists $m$ such that $E\oplus \eps^m$ (trivial bundle) does admit a nonnegatively curved metric.


\chapter{Cheeger-Gromoll Soul Conjecture}

In this lecture, we are going to study the Soul conjecture of Cheeger and Gromoll\cite{CG72}. This is an essential question in Riemannian geometry about the structure of complete manifolds with non-negative sectional curvature. 

\begin{corollary}[Soul Conjecture of Cheeger and Gromoll]\label{cor: Soul Conj}
    Let $(M^n, g)$ be a complete open Riemannian manifold with $\sect_M \geq 0$. Suppose there exists $x \in M$ such that $\sect > 0$ near $x$. Then the soul $S$ is a point and 
    \begin{equation*}
        M^n \stackrel{\text{diffeo}}{\cong}\R^n
    \end{equation*}
\end{corollary}

This conjecture was proved by Perelman \cite{Per94} in 1994 in 4 pages. The proof depends on Berger's comparison theorem \ref{lem: Berger comparison} and the Sharafutdinov Retraction in \cite{Sha77}, which will be introduced soon in this chapter.

\section{Distance Estimates}
\subsection{Gradients of Semi-concave Functions}
Let $f: M^n \to \R$ be semi-concave and locally Lipschitz function. Then for every $p \in M$, the directional derivative map $df_p$ is defined. And
\begin{equation*}
    df_p: T_pM \to \R
\end{equation*}
is also Lipschtiz, concave and positively homogeneous, i.e. $df_p(\lambda v) = \lambda df_p(v)$ for all $\lambda \geq 0$ (See Example~\ref{ex: GH}). 

\begin{example}
Let  $A \subseteq M$ be a closed subset. Let $f = d(\cdot, A)$. Let $p \in M\backslash A$. Then by the first variation formula,
    \begin{equation*}
        df_p(v) = \min{\br{-\inp{u, v}: u \in \Uparrow_p^A \cap T_p^1M}}.
    \end{equation*}
    Notice that this is the minimum of $1$-Lipschitz linear functions of the form $v \mapsto -\inp{u, v}$, which are all concave. Therefore $df_p$ is concave too.  This provides an alternative explanation for the concavity of $df_p$ for distance functions.
\end{example}
Now, we want to define the gradient $\nabla f_p$. We know that when $f$ is smooth then $df_p$ is linear and $\nabla f_p$ is easily defined by the Riesz representation theorem, i.e.
\begin{equation*}
    df_p(v) = \inp{\nabla f_p, v} \quad \forall v
\end{equation*}
What about the general case?
\begin{definition}\label{def: grad of Lip-semi-concave}
    Let $f: M \to \R$ be locally Lipschitz and semi-concave. Let $p \in M$. Then a vector $h \in T_pM$ is called the \textbf{gradient} of $f$ at $p$, denoted by $\nabla f_p$, if 
    \begin{itemize}
        \item $df_p(h) = \abs{h}^2$;
        \item $df_p(v) \leq \inp{h, v}$ for every $v \in T_pM$
    \end{itemize}
\end{definition}
Notice that if $df_p$ is linear, then it is not hard to show the existence and the uniqueness of $h = \nabla f_p$  in the usual sense. In general,  existence and uniqueness still hold. However, this is non-trivial. And we are going to show this here. 
\subsection{Existence and Uniqueness of the Gradient for Semi-concave Functions}
\begin{theorem}[Existence and Uniqueness of Gradients]\label{thm: existence and the uniqueness of grad}
    The gradient $\nabla f_p$ in the definition \ref{def: grad of Lip-semi-concave} exists at every $p \in M$ for locally Lipschitz and semi-concave functions $f: M \to \R$ and it is unique. 
\end{theorem}
\begin{proof}
    We are first going to show uniqueness and then prove its existence. 
    \begin{itemize}
        \item \textbf{Step 1: Uniqueness}\\
      Suppose $h_1$ and $h_2$ are  two vectors satisfying  definition \ref{def: grad of Lip-semi-concave}. Then 
        \begin{align*}
            df_p(h_1) &= \abs{h_1}^2 \leq \inp{h_1, h_2} \quad\text{Since $h_2$ is a gradient};\\
            df_p(h_2) &= \abs{h_2}^2 \leq \inp{h_2, h_1} \quad\text{Since $h_1$ is a gradient}.
        \end{align*}
        \begin{align*}
            0 \leq \abs{h_1 - h_2}^2 = & \inp{h_1 - h_2, h_1 - h_2}\\
            = &\abs{h_1}^2 - 2\inp{h_1, h_2} + \abs{h_2}^2\\
            =& (\abs{h_1}^2 - \inp{h_1, h_2}) + (\abs{h_2}^2 - \inp{h_1, h_2})\\
            \leq & 0.
        \end{align*}
        All the inequalities above are equalities. Therefore
        \begin{equation*}
            \abs{h_1 - h_2}^2 = 0 \implies h_1 = h_2
        \end{equation*}
        \item \textbf{Step 2: Existence}\\
        Denote $S = \sup{\br{df_p(v): v \in T_pM, |v|=1}}$. 
        \begin{itemize}
            \item \textbf{Trivial Case:}\\
            If $S \leq 0$, then $df_p \leq 0$ on $T_pM$, thus we claim that $h = 0 \in T_pM$ is the gradient. And we just write $\nabla f_p = 0$. To check this, take $h = 0$, then
            \begin{equation*}
                df_p(0) = \abs{0}^2 = 0
            \end{equation*}
            and 
            \begin{equation*}
                df_p(v) \leq \inp{v, 0} = 0 \quad\forall v \in T_pM
            \end{equation*}
            Thus we can conclude in this case that $\nabla f_p = 0$.
            \item \textbf{Nontrivial Case:} If $S > 0$, then by the compactness of $S^{n - 1} \subseteq T_pM$, there exists unit $v_{\max} \in T_p^1M$ such that 
            $S = df_p(v_{\max}) > 0$. We claim that
            \begin{claim}
                $h= S\cdot v_{\max}$ is a gradient.
            \end{claim}

            We can check that 
            \begin{align*}
                df_p(S\cdot v_{\max}) = & S\cdot df_p(v_{\max}) = S^2 = S^2\abs{v_{\max}}^2\\
                = & \inp{S\cdot v_{\max}, S\cdot v_{\max}} = \abs{h}^2
            \end{align*}
            Next, we want to check that $\forall v \in T_pM, df_p(v) \leq \inp{h, v}$ for $h = S\cdot v_{\max}$. To check this, we need the following lemma.
            \begin{lemma}\label{lem: inequality of existence}
                $\forall u, v \in T_pM$, $df_p(v) + df_p(u) \leq S\cdot |u + v|$.
              
            \end{lemma}
              We are going to verify this once we proved the existence of the gradient.
            
            Fix $v \in T_pM$, by the lemma \ref{lem: inequality of existence}, we have
            \begin{equation*}
                df_p(h) + df_p(\eps v) = df_p(h) + \eps df_p(v) \leq S\cdot\abs{h + \eps v}\quad \forall \eps \geq 0
            \end{equation*}
            
            \corrv{ Note that $|h|=|S\cdot v_{\max}|=S$. }
            Consider $\abs{h + \epsilon v}$, 
            \begin{align*}
                \abs{h + \eps v} = & \sqrt{\inp{h + \eps v, h + \eps v}} = \sqrt{\abs{h}^2 + 2\eps \inp{h, v} + \eps^2}\\
                = & \abs{h} + \frac{\eps}{2\sqrt{\abs{h^2}}}\cdot 2\inp{h, v} + O(\eps^2)\\
             = &\abs{h} + \frac{\eps\inp{h, v}}{\abs{h}} + O(\eps^2) 
            \end{align*}
            Thus
            \begin{align*}
                df_p(h) + \eps df_p(v)\leq & S\brac{\abs{h} + \frac{\eps\inp{h, v}}{\abs{h}} + O(\eps^2)}\\
                \leq &\abs{h}^2 + \eps\inp{h, v} + O(\eps^2) \quad\forall \eps\geq 0
            \end{align*}
            
    	Since $ df_p(h) =|h|^2$  we can conclude that $df_p(v) \leq \inp{h, v}$.
        \end{itemize}
    \end{itemize}
    
    This concludes the proof of the existence and uniqueness of the  gradient and hence of Theorem \ref{thm: existence and the uniqueness of grad} modulo Lemma  \ref{lem: inequality of existence}.
\end{proof}
Let's verify lemma \ref{lem: inequality of existence} above. 
\begin{proof}
    \textbf{Case 1:} Suppose $u + v = 0$, $u = -v$. Since $df_p$ is concave we have
    \begin{equation*}
        0 = df_p(0) = df_p\brac{\frac{u + v}{2}} \geq \frac{1}{2}df_p(v) + \frac{1}{2}df_p(u).
    \end{equation*}
    \textbf{Case 2:} Suppose $u + v \neq 0$, then since $df_p$ is concave and positively 1-homogeneous we have
    \begin{equation*}
        \frac{1}{2}df_p(u + v) = df_p(\frac{u + v}{2}) \geq \frac{1}{2}df_p(u) + \frac{1}{2}df_p(v).
    \end{equation*}
    That implies
    \begin{align*}
        df_p(u) + df_p(v) \leq &df_p\brac{\abs{u + v}\frac{u + v}{\abs{u + v}}}\\
        = & \abs{u + v}df_p\brac{\frac{u + v}{\abs{u + v}}}\\
        \leq & S\abs{u + v}.
    \end{align*}
\end{proof}
This allows us to define the gradient for locally Lipschitz functions and semi-concave functions and in particular for distance functions. Let’s consider the following example. 
\begin{example}
    Take $f = d(\cdot, A)$ where $A$ is a closed subset in $M$ and let $p \notin A \subseteq M$. We want to compute $\nabla f_p$. 
    \begin{itemize}
        \item If $p$ is a critical point, then for any $v \in T_pM$, $df_p(v) \leq 0$. Geometrically, this means the angle $\alpha = \mangle(v, \Uparrow_p^A) \leq \frac{\pi}{2}$ so that $df_p(v) = -\cos{(\alpha)}\leq 0$. From the trivial case in the existence part of the theorem \ref{thm: existence and the uniqueness of grad}, this means $\nabla f_p = 0$;
        \item If $p$ is a regular point for $f$, then there exists $v \in T_pM$ such that $df_p(v) > 0$. Then by the first variation formula, for every $u \in \Uparrow_p^A$, we have $\mangle(u, v) > \frac{\pi}{2}$. By the compactness of $S^{n - 1} \subseteq T_pM$, there exists $v \in S^{n - 1}$ such that $v$ attains the largest angle with $\Uparrow_p^A$. We denote this vector $v_{\max{\empty}}$. Then this $v$ maximize $d(\cdot, A)|_{S^{n - 1}}$. 
        
        Let $\alpha = d(v_{\max{\empty}}, \Uparrow_p^A)|_{S^{n - 1}} = \mangle(v_{\max{\empty}}, \Uparrow_p^A)$. Now $S = df_p(v_{\max{\empty}}) = -\cos{\alpha} > 0$. Hence  $\nabla f_p = S\cdot v_{\max{\empty}} = -\cos{\alpha}\cdot v_{\max{\empty}}$.  Note that uniqueness of gradients implies that the vector $v_{\max}$ is unique.
    \end{itemize}
\end{example}
\begin{example}
    Let $M = \R^2, A=\{(x,y)\in\R^2\mid x\le 0$  or $y\le 0\}$. Then $f=d(\cdot, A)$ is 0 on $A$ while on the first quadrant $\{x>0,y>0\}$ it is given by the formula 
    $f(x,y)= \min{(x, y)}$ on $x, y \geq 0$. Then
    \begin{equation*}
        f(x,y) = 
        \begin{cases}
            y \quad \text{if } x>0, y>0, x>y  \implies \text{$\nabla f_p = (0, 1)$} \\
         x \quad \text{if } x>0, y>0, y>x  \implies \text{$\nabla f_p = (1, 0)$ }
        \end{cases}
    \end{equation*}
    Let us now compute the gradient of $f$ at points in the first quadrant with $x=y$.
 By our construction
    \begin{equation*}
        \Uparrow_p^A = \br{(0, -1), (-1, 0)}.
    \end{equation*} 
    By the picture, we know that $v_{\max{\empty}} = \brac{\frac{1}{\sqrt{2}}, \frac{1}{\sqrt{2}}}$.
    
    \begin{figure}[htbp]
    \centering
    \includegraphics[width=0.4\textwidth]{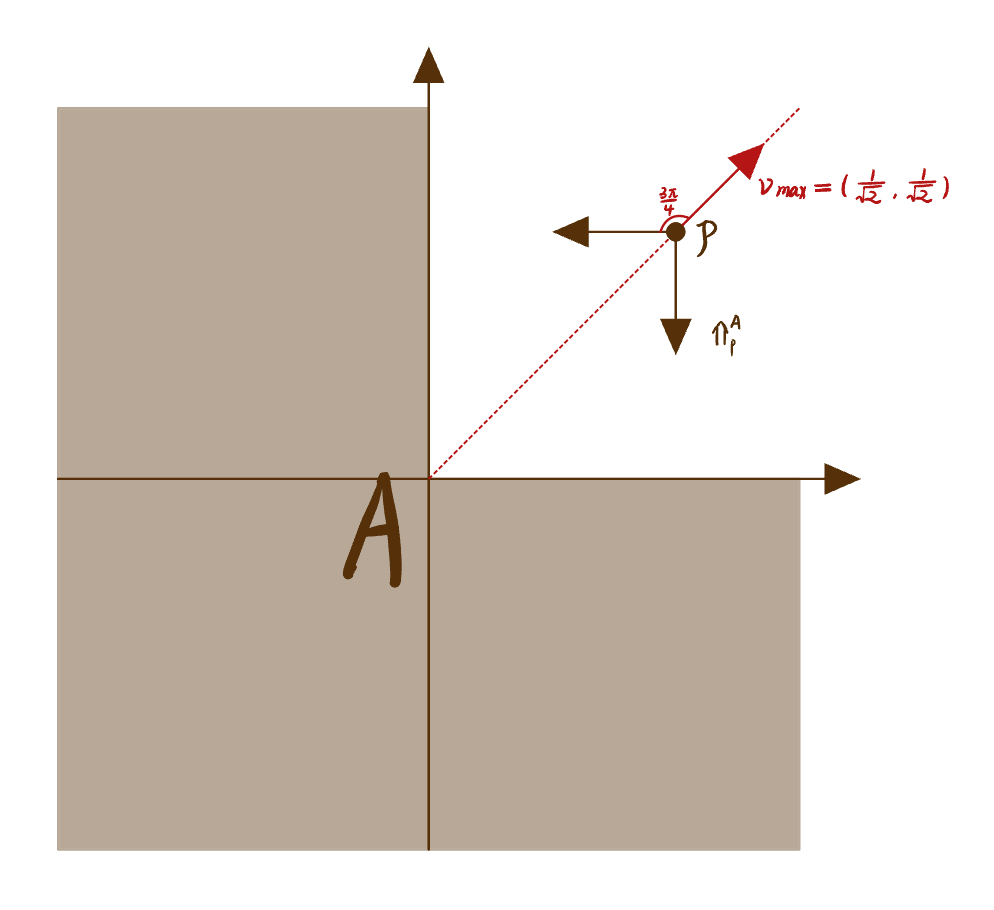}
    \end{figure}
	And $S = df_p(v_{\max{\empty}}) = -\cos{\alpha} = -\cos{\frac{3\pi}{4}} = \frac{1}{\sqrt{2}}$. Therefore
    \begin{equation*}
        \nabla f_p = S\cdot v_{\max{\empty}} = \frac{1}{\sqrt{2}}\brac{\frac{1}{\sqrt{2}}, \frac{1}{\sqrt{2}}} = \brac{\frac{1}{2}, \frac{1}{2}}
    \end{equation*}
  Note that $\abs{\nabla f_p} = \frac{1}{\sqrt{2}} < 1$.
  \begin{figure}[htbp]
    \centering
    \includegraphics[width=0.8\textwidth]{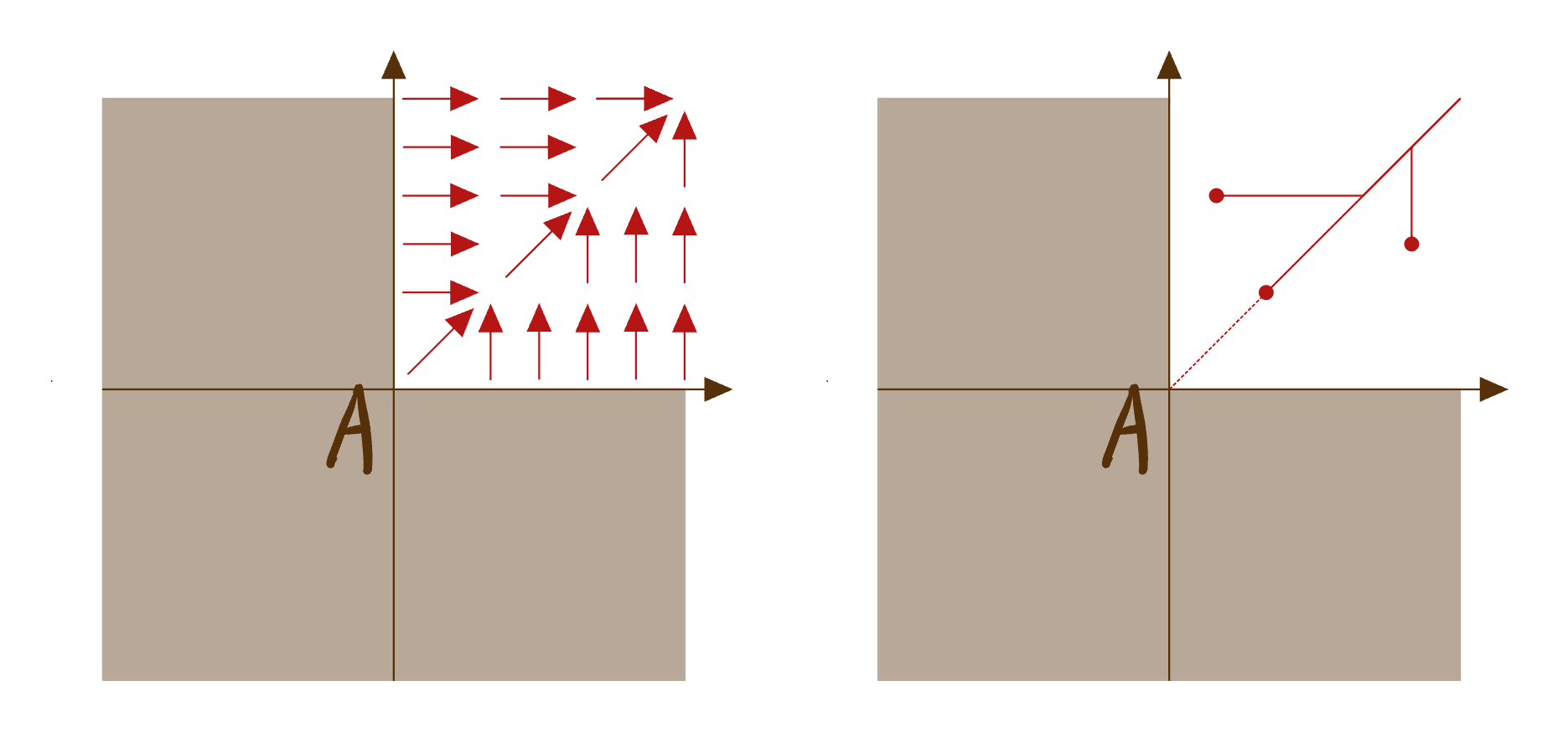}
    \end{figure}
\end{example}
\subsection{Gradient Flows of  Semi-concave Functions}
In this section, we are going to define gradient flows of semi-concave functions.
\begin{definition}
Let $M^n$ be a Riemannian manifold and $U\subset M$ be open. Let $f:M \to \R$ be a function that is semi-concave on $U$. An absolute continuous curve $\gamma: I \to U $, where $I$ is an interval, is a gradient curve of $f$ if its right derivative satisfies $\Dot{\gamma}_+(t) = \nabla f_{\gamma(t)}$ for every $t \in I$.
\end{definition}
The next theorem is the existence and uniqueness theorem of gradient curves. We are not going to present the proof of the existence part which was originally proved by Petrunin and Perelman in \cite{PP95}, and a simplified proof was given by Lytchak in \cite{Ly05}.

Uniqueness will follow from Lipschitz properties of the gradient flow and we will prove this below.
\begin{theorem}\label{thm: gradient curve exists and unique}
Let $M, U, f$ be as above. 
 Then for any $ p \in U$,  there is an $\eps > 0$ such that there exists a unique gradient curve of $f$, $\gamma: [0, \eps) \to U$, such that $\gamma(0) = p$. 
\end{theorem}

\subsection{More Comments on Gradients}
If $\gamma$ is a gradient curve, and $c$ is a constant then $\gamma(t + c)$ is also a gradient curve by the same argument that works for integral curves of smooth vector fields.  Hence we can define the positive gradient flow $\Phi(x, t) = \gamma_x(t)$ where $\gamma_x$ means the gradient curve starting at $x$. Namely,
\begin{equation*}
    \gamma_x(0) = x = \Phi(x, 0).
\end{equation*}
Denote $\phi_t(x) = \Phi(x, t)$. It also has the semigroup property, i.e.
\begin{equation*}
    \phi_{t + s} = \phi_t\circ \phi_s \quad\forall t, s \ge  0
\end{equation*}
This follows since for any fixed $x$, if $\gamma_x(t)$ is a gradient curve then $\eta(t) \gamma_x(t + s)$ is also a gradient curve starting at $\gamma_s(x)=\phi_s(x)$. If $\gamma_x(0) = x$, $\eta(t)=\gamma_x(s + t) = \phi_{s + t}(x)$. Consider $\sigma(t)$ which is also a gradient curve such that $\sigma(0) = \gamma_x(0 + s) = \gamma_x(s) = \phi_s(x)$, $\sigma(t) = \phi_t(\sigma(0)) = \phi_t\circ\phi_s(x)$. 

We see that both $\sigma(t)$ and  $\eta(t)$  are gradient curves starting at $\phi_s(x)$. By uniqueness of gradient curves this implies that $\sigma(t)=\eta(t)$ 
Since $\sigma(t) = \gamma(s + t) = \phi_{s + t}(x)$  for all $t\ge 0$ which by above means that 
\begin{equation*}
    \phi_{s + t}(x) = \phi_t\circ\phi_s(x).
\end{equation*}

Our next comment is that in general, we cannot define a flow for negative $t$. Let $f(x, y) = \min{\br{x, y}}$ on $\R_{+}^2$. Then 
\begin{equation*}
\nabla f_p = 
    \begin{cases}
        (1, 0) \quad\text{On $\br{x < y}$};\\
        (0, 1) \quad\text{On $\br{x > y}$};\\
        (\frac{1}{\sqrt{2}}, \frac{1}{\sqrt{2}}) \quad\text{On $\br{x = y}$};
    \end{cases}
\end{equation*}
we can find the gradient flow of $f$ actually converges. The flow map is 
\begin{equation*}
	\Phi(t, (x, y)) = 
	\begin{cases}
		f_1 = 
		\begin{cases}
			(x + t, y) \quad\text{for $t < y - x$}\\
			(\frac{x + t}{\sqrt{2}}, \frac{x + t}{\sqrt{2}}) \quad\text{for $t \geq y - x$}
		\end{cases}
		\quad\text{for $x < y$}\\
		f_2 = 
		\begin{cases}
			(x, y + t) \quad\text{for $t < x - y$}\\
			(\frac{y + t}{\sqrt{2}}, \frac{y + t}{\sqrt{2}}) \quad\text{for $t \geq x - y$}
		\end{cases}\quad\text{for $x > y$}\\
		f_3 = (\frac{x+t}{\sqrt{2}}, \frac{x+t}{\sqrt{2}}) \quad\text{for $x = y$}
	\end{cases}
\end{equation*}
\begin{figure}[htbp]
    \centering
    \includegraphics[width=0.8\textwidth]{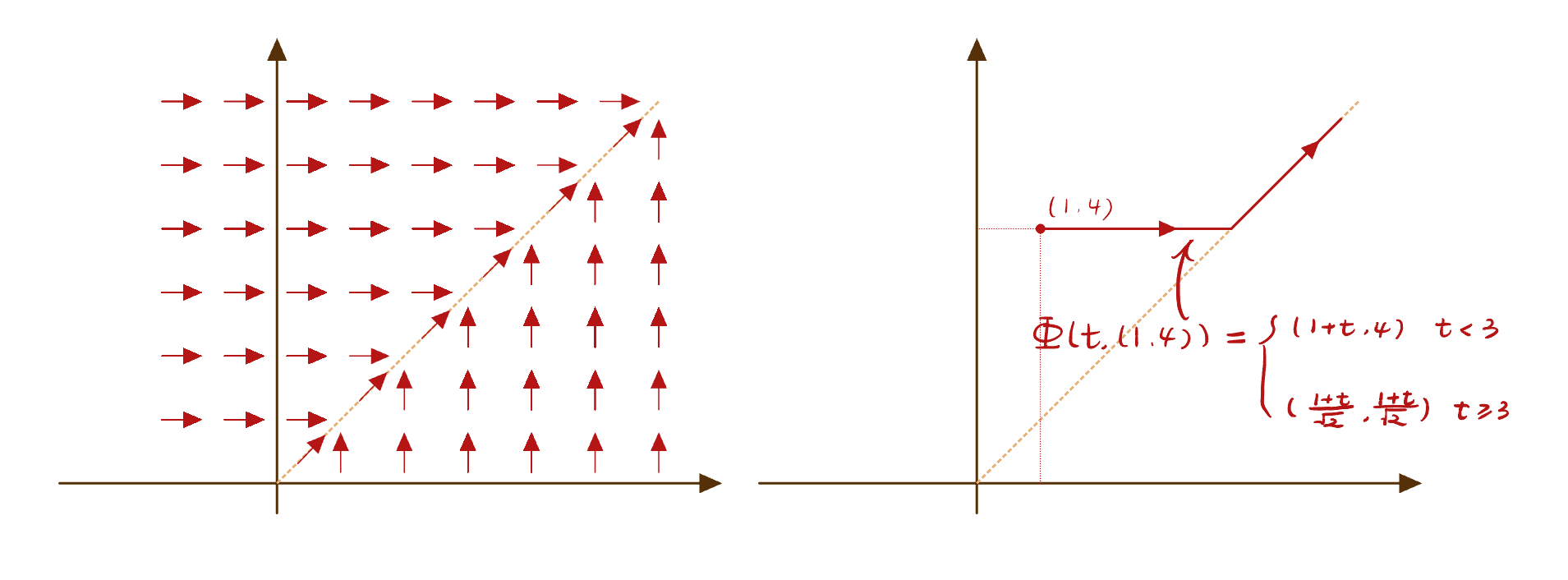}
    \end{figure}

\begin{remark}
Gradient flows of semi-concave functions can be defined on Alexandrov spaces where the theory is largely the same as on smooth manifolds.
They can also be defined on Wasserstein spaces of Alexandrov and more generally $RCD(K,N)$ spaces. 
In particular, the heat flow is constructed as the gradient flow of entropy functionals which are semi-convex (but not locally Lipschitz) in the above setting. 
However, the constructions of gradient flow in those settings are more complicated because unlike in the case of manifolds, we are considering point-wise gradients can not be defined, and measure-theoretic arguments are necessary.
 
\end{remark}
\begin{remark}
	Our next comment is that if $\gamma(t)$ is a gradient curve, then $(f(\gamma(t)))_+' = df_{\gamma(t)}(\gamma_+'(t)) = \abs{\nabla f_{\gamma(t)}}^2 \geq 0$. Hence, $f$ is non-decreasing along gradient curves. In the next lecture, we are going to prove a  distance contraction estimate and show the uniqueness of gradient curves. 
\end{remark}

\subsection{Uniqueness of Gradient Curves}
\begin{lemma}\label{lem: grad-estimates of lambda concave function}
Let $f$ be $\lambda$-concave, and $p, q \in M$. For the directional unit vector of the shortest curve from $p$ to $q$, $\uparrow_p^q \in T_pM$, we have
    \begin{equation}\label{eq: grad-estimates of lambda concave function}
        \inp{\nabla f_p, \uparrow_p^q} \geq \frac{f(p) - f(q) - \frac{\lambda l^2}{2}}{l}
    \end{equation}
\end{lemma}
\begin{proof}
    Let $\gamma(t) = \exp_p(t\uparrow_p^q)$ such that $\gamma(0) = p$ and $\gamma(l) = q$, where $l = \length(\gamma)$. By the definition of $\lambda$-concavity  $f(\gamma(t))$ is also $\lambda$-concave. Therefore,
    \begin{equation*}
        h(t) = f(\gamma(t)) - \frac{\lambda t^2}{2}
    \end{equation*}
    is concave. Thus
    \begin{equation*}
        h(t) \leq h(0) + h_+'(0)t \quad\text{for $t > 0$}
    \end{equation*}
    and we should keep in mind that $h(0) = f(p)$ and $h_+'(0) = f_+'(\gamma(t))$. Plug $t = l$ into the inequality, and we have
    \begin{align*}
        f(q) - \frac{\lambda l^2}{2} = h(l) \leq & f(p) + f_+'(\gamma(t))|_{t = 0}\cdot l\\
        = & f(p) + df_p(\dga(0))\cdot l\\
        = & f(p) + df_p(\uparrow_p^q)\\
        \leq & f(p) + \inp{\nabla f_p, \uparrow_p^q}\cdot l
    \end{align*}
    By switching the order, we have
    \begin{equation*}
        \frac{f(q) - f(p) - \frac{\lambda l^2}{2}}{l} \leq \inp{\nabla f_p, \uparrow_p^q}
    \end{equation*}
\end{proof}
\begin{corollary}\label{cor: symmetric sum concave grad}
Let $f$ be $\lambda$-concave,  $p, q \in M$.   Let  $[pq]$  a shortest geodesic, then 
    \begin{equation}\label{eq: symmetric sum concave grad}
        \inp{\nabla f_p, \uparrow_p^q} + \inp{\nabla f_q, \uparrow_q^p} \geq -\lambda l
    \end{equation}
\end{corollary}
\begin{proof}{Corollary}{\ref{cor: symmetric sum concave grad}}
    By switching the order of $p$ and $q$ in \ref{eq: grad-estimates of lambda concave function}, we have
    \begin{equation}\label{eq1: grad est}
        \frac{f(q) - f(p) - \frac{\lambda l^2}{2}}{l} \leq \inp{\nabla f_p, \uparrow_p^q}
    \end{equation}
    and
    \begin{equation}\label{eq2: grad est}
        \frac{f(p) - f(q) - \frac{\lambda l^2}{2}}{l} \leq \inp{\nabla f_q, \uparrow_q^p}
    \end{equation}
    Thus by summing \ref{eq1: grad est} and \ref{eq2: grad est}, we have
    \begin{equation*}
        \inp{\nabla f_p, \uparrow_p^q} + \inp{\nabla f_q, \uparrow_q^p} \geq -\lambda l
    \end{equation*}
\end{proof}
\begin{proposition}[Distance Estimates]\label{prop: distance estimates}
    Let $f$ is a $\lambda$-concave function, and $\alpha(t), \beta(t)$ the gradient curves of $f$ starting at $p$ and $q$. Namely $\alpha, \beta: [0, \eps) \to M$ such that $p = \alpha(0)$ and $q = \beta(0)$. Denote 
    \begin{equation*}
        l(t) = d(\alpha(t), \beta(t))
    \end{equation*}
    Then 
    \begin{equation*}
        l(t) \leq l(0)e^{\lambda t}
    \end{equation*}
\end{proposition}
\begin{proof}
    Firstly, we claim that
    \begin{equation}\label{eq 1: distance estimates}
        l_+'(t) \leq \lambda l(t) \quad \forall t \in [0, \eps)
    \end{equation}
    And if this is true, then the conclusion follows
    \begin{align*}
        & (l_+' - \lambda l)\cdot e^{-\lambda t} \leq 0 \cdot e^{-\lambda t}\\
        \implies & l_+'\cdot e^{-\lambda t} - \lambda l\cdot e^{-\lambda t} \leq 0\\
        \implies & (l \cdot e^{-\lambda t})_+' \leq 0\\
        \implies & \text{$l \cdot e^{-\lambda t}$ is non-increasing}\\
        \implies & l \cdot e^{-\lambda t}|_t \leq l\cdot e^{-\lambda t}|_0\\
        \implies & l(t)\cdot e^{-\lambda t} \leq l(0)\\
        \implies & l(t) \leq l(0)e^{\lambda t}
    \end{align*}
    Thus, the only thing that remains to be proven is inequality \eqref{eq 1: distance estimates}. Let’s check this at $t = 0$, $l = l(0) = \abs{pq}$. the proof for other $t$ is the same. 
    
     Let $m$ be the mid-point of a shortest geodesic $\gamma$ from $p$ to $q$. That is $\gamma: [0, l] \to M$, $\gamma(0) = p$, $\gamma(l) = q$ and $m = \gamma(\frac{1}{2})$.   By the first variation formula, we have
    \begin{equation}\label{eq 2: distance estimates}
        d(\alpha(t), m) \leq \frac{l(0)}{2} - \inp{\nabla f_p, \uparrow_p^q}t + o(t);
    \end{equation}
    \begin{equation}\label{eq 3: distance estimates}
        d(\beta(t), m) \leq \frac{l(0)}{2} - \inp{\nabla f_q, \uparrow_q^p}t + o(t)
    \end{equation}
  We used  that
    \begin{align*}
        \nabla f_p = \alpha_+'(0);\\
        \nabla f_q = \beta_+'(0)
    \end{align*}
    Then we add \ref{eq 2: distance estimates} and \ref{eq 3: distance estimates} together to conclude our result. By the triangule inequality, we have
    \begin{equation*}
        l(t) = d(\alpha(t), \beta(t)) \leq d(\alpha(t), m) + d(\beta(t), m),
    \end{equation*}
    which by \eqref{eq 2: distance estimates} and \eqref{eq 3: distance estimates}  gives
    \begin{align*}
        l(t) \leq d(\alpha(t), \beta(t)) \leq & l(0) - \inp{\nabla f_p, \uparrow_p^q} - \inp{\nabla f_q, \uparrow_q^p} + o(t)\\
        \leq & l(0) + (\lambda l(0))(t) + o(t) \quad\text{By the corollary \ref{cor: symmetric sum concave grad}}
    \end{align*}
    Thus, we have
    \begin{equation*}
        l_+'(0) \leq \lambda l(0)
    \end{equation*}
    and by the same argument, we have
    \begin{equation*}
        l_+'(t) \leq \lambda l(t) \quad \forall t \in [0, \eps)
    \end{equation*}
By the argument earlier this implies that
    \begin{equation*}
        l(t) \leq l(0)e^{\lambda t}
    \end{equation*}
\end{proof}

Thus, we have finished our proof of the uniqueness of the gradient curves. And we are going to discuss the gradient flow.

\begin{remark}
    The above proof works in large generality and only needs the $1$-st variation inequality
    \begin{equation*}
        d(\alpha(t), q)'|_{t = 0}^+ \leq -\inp{\alpha'(0), \uparrow_p^q},
    \end{equation*}
    which is much easier to prove compared with the first variation formula. 
    The first variation inequality holds in more general cases, i.e. infinitesimally Hilbertian spaces. 
    \begin{example}
        For example, suppose a metric measure space $(X, d, m)$ is $RCD(0,\infty)$. Consider the entropy function $E: P_2(X) \to \R$ and $\nu = \rho\cdot m$ the absolutely continuous measure. Then 
        \begin{equation*}
            E(v) = \int_X \rho\log{\rho} dm
        \end{equation*}
         is convex on $(P_2(X), W_2)$ ($-E$ is concave on $(P_2(X), W_2)$) where $W_2$ is the Wasserstein metric. And we can conclude that its gradient flow of $E$  (which is equal to the heat flow) is contracting because $(P_2(W), W_2)$ satisfies the first variation inequality. 
    \end{example}
\end{remark}
\subsection{Gradient Flow for $\lambda$-Concave Functions}\label{sect: gradient flow of lambda concave}
Let $\Phi_t$ be the gradient flow of $f$ such that 
\begin{equation*}
    \Phi_t(p) = \alpha(t)
\end{equation*}
where $\alpha(t)$ is a gradient curve starting at $p$. Then, $\Phi_t$ satisfies 
\begin{equation*}
    \begin{cases}
        \Phi_0 = \id\\
        \Phi_t\circ \Phi_s = \Phi_{t + s} \quad\text{for $t + s \geq 0$}
    \end{cases}
\end{equation*}
And by our distance estimate for $\lambda$-concave function, we know that 
\begin{equation*}
    d(\Phi_t(x), \Phi_s(y)) \leq e^{\lambda t}d(x, y)
\end{equation*}
That means the flow $\Phi_t$ is a Lipschitiz map with Lipschitz constant $e^{\lambda t}$ when $f$ is $\lambda$-concave. In particular, if $\lambda = 0$ ($f$ is concave), then $\Phi_t$ is $1$-Lipschitz for all $t \geq 0$ i.e.  $\Phi_t$ is distance non-increasing for all $t$.

\section{Sharafutdinov Retraction}
Let’s recall the definition of strong deformation retraction. 
\begin{definition}[Strong Deformation Retraction]\label{def: sdr}
	Let $X$ be a topological space and $A \subseteq X$ its topological subspace. $A$ is a \textbf{strong deformation retraction} if there is a family of maps $\Pi:[0, 1] \times X \to X$, $t \in [0, 1]$ such that
	\begin{itemize}
		\item $\Pi$ is a homotopy;
		\item $\Pi_0 = \Pi(0, \cdot) = \id_X$;
		\item $\Pi_t|_A = \Pi(t, \cdot) = \id_A$ for all $t \in [0, 1]$;
		\item $\Pi_1(x) = \Pi(1, x) \in A$ for all $x \in X$.
	\end{itemize}
	We call $\Pi$ the \textbf{strong deformation retraction map induced by $A$}. 
\end{definition}
\begin{notation}
	We will abuse the notation by writing the strong deformation retraction map induced by $A$ as $\Pi: X \to A$ for simplicity. 
\end{notation}

\begin{corollary}\label{cor: retraction exists}
    Let $b: M \to \R$ be concave and $(M^n, g)$ is a complete Riemannian manifold. Assume all level sets of $b$ are compact. Then for any $c \in b(M)$, $\br{b \geq c}$ is a strong deformation retraction, i.e. there exists a $\Pi: M \to S$ as in Definition~\ref{def: sdr}.
\end{corollary}
\begin{proof}
    Firstly, we consider the case when $c < b_{\max}$ where $b_{\max}$ is the maximum value of $b$. Let $f = \min\br{b, c}$. 
    \begin{figure}[htbp]
    \centering
    \includegraphics[width=0.8\textwidth]{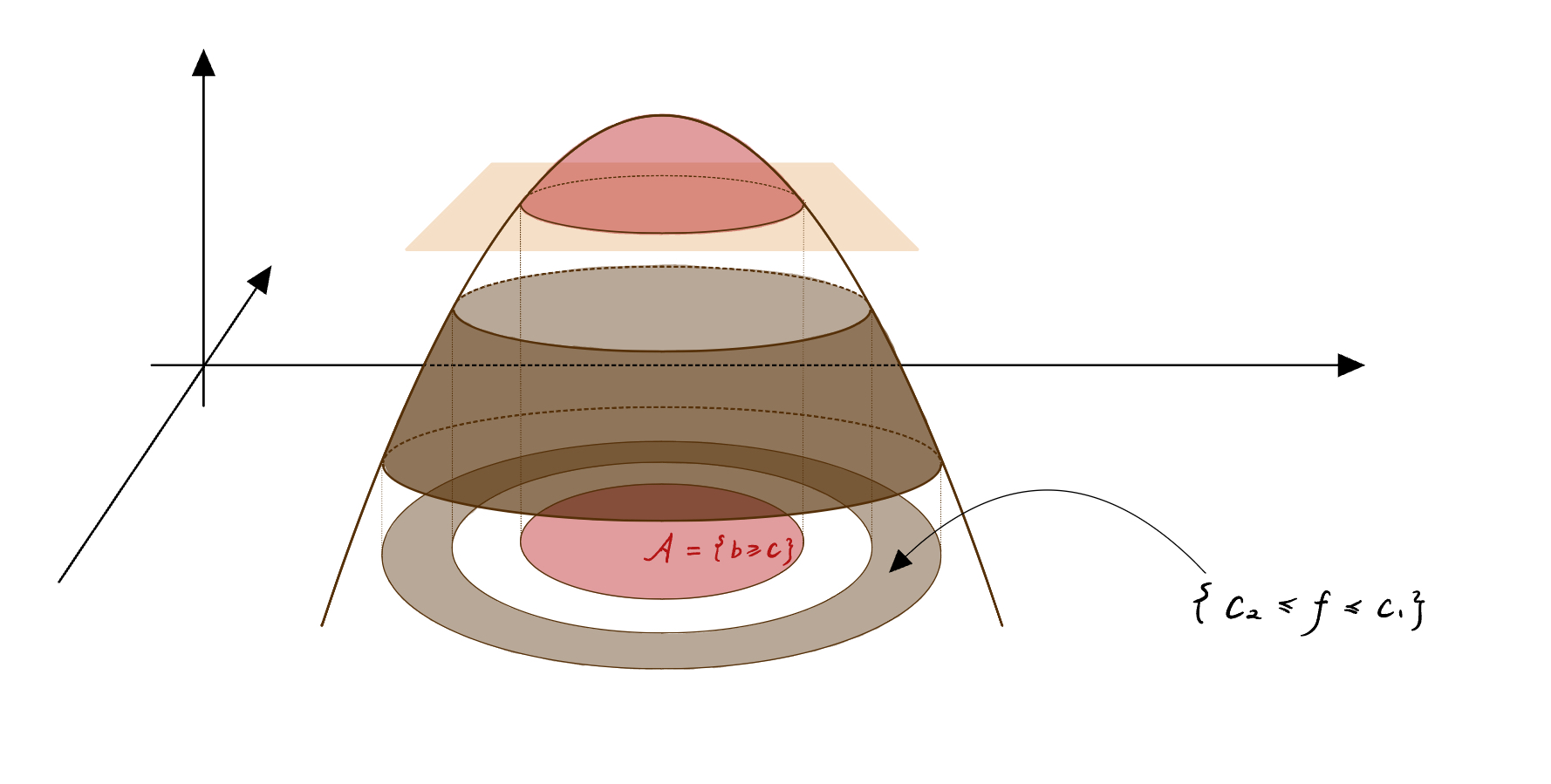}
    \end{figure}
    We denote $A = \br{b \geq c}$. Notice that $\br{c_2 \leq f \leq c_1}$ is compact for $c_1 \leq c_2 \leq c$, so we have a uniform lower bound for $\abs{\nabla f}$ over this region, i.e. for every $p \in \br{c_2 \leq f \leq c_1}$, $\abs{\nabla f|_p} > \delta > 0$ for some constant $\delta$, thus, there exists some finite time $T_x$ such that $\phi_t(x) \in A$ where $\phi_t$ is the flow map generated by the gradient vector field.
    
  	Let $\alpha(t)=\frac{\arctan(t)}{\pi/2}$. Then  $\alpha : [0, \infty) \to [0, 1)$ is a bijection. We can continuously extend it to a bijection  $ [0, \infty] \to [0, 1]$ which we still denote by $\alpha$.  then we can define the strong deformation retraction map $\Pi: M\times [0,1]\to M$   by the formula $\Pi_t= \phi_{\alpha^{-1}(t)}$.
    \begin{itemize}
    	\item $\Pi_t$ is $1$-Lipschitiz: This is because $f$ is still concave as the minimum of two concave functions. By the previous discussion, the flow map $\phi_t$ of the concave function $f$ is $1$-Lipschitz for any $t \ge 0$ (Section~\ref{sect: gradient flow of lambda concave});
    	\item $\Pi_0 = \id_M$: This is trivial;
    	\item $\Pi_t|_{A} = \id_{A}$ for $t \in [0, 1]$. This is because $\nabla f|_{A} \equiv 0$ on $A$ and thus $\phi_t$ is constant on $A$;
    	\item $\Pi_1(x) \in A$ for all $x \in M$: This is because for any $x \in M$, $T_x$ is finite, thus $\alpha^{-1}(t) > T_x$ when $t \to 1$. 
    \end{itemize}
    It is clear that $\Pi$ is a homotopy. Thus we have proved the corollary for $c < b_{\max}$. 
       
    In general, if $c = b_{\max}$, then we can take $A_\eps = \br{b \geq c - \eps}$ for some small $\eps > 0$. And repeat the above arguments producing a family of deformations $\br{\Pi^\eps}$. Then we can show that $\br{b = b_{\max}}$ is a strong deformation retraction by taking $\Pi^0 = \lim_{\eps \to 0}\Pi^\eps$ (The limits of 1-Lipshcitz maps are also 1-Lipschitz). 
\end{proof}
\begin{example}
	Here is an example that explains why we need to consider $\br{b = b_{\max}}$ as a separate case in the above proof. Consider $b: \R \to \R$ such that $b = -x^2$. In this case $0 = b_{\max}$. Thus, $f = \min{\br{b, 0}} - b$ and $\nabla f(x) = -2x$. The flow map solves $x(t)$ the following ODE:
	\begin{equation*}
		x' = 2x \implies x(t) = ce^{-2t} 
	\end{equation*}
	for constant $c$. Therefore, it takes $t = \infty$ for the flow map sending $x$ to the maximum point $0$. By uniqueness of solutions of IVP this phenomena always happens if $b$ is smooth and $c=b_{\max}$.
\end{example}
Next, we are going introduce the corollary concluded by Sharafutdinov which is an essential construction in the proof of the soul conjecture. 
\begin{corollary}[Sharafutdinov]\label{cor: Sharafutdinov}
    Let $(M^n, g)$ be complete, $\sect_M \geq 0$. $S \subseteq M$ is the soul. Then there is a $1$-Lipschitz strong deformation retraction map
    \begin{equation*}
        \sh: M \to S.
    \end{equation*} 
\end{corollary}
\begin{proof}
Recall the construction of the Soul. We first consider a concave function with compact super-level sets. 
$b: M \to \R$, $b = {{\min{\empty}}_{\gamma}}\br{b_\gamma}$, where $\gamma$ is a ray starting at $p$.

Then  $b$ is concave with compact super-level sets. 
Denote $C^0 = C_{\max} = \br{b = b_{\max}}$ the maximum level set of $b$.
Then by the above corollary \ref{cor: retraction exists}, there exists a $1$-Lipschitz strong deformation retraction $\Pi_0: M \to C^0$.
Next, if the boundary $\partial C^0 \neq \emptyset$. 
Then we take $f_0 = d(\cdot, \partial C^0)$ on $C^0$ and $f_0$ is concave.
So we can take the $C^1$,  the second maximal level set again. 
As before, there is a strong deformation $\Pi_1: C^0 \to C^1$ (by corollary \ref{cor: retraction exists}). Then we know the soul $S$ can be found in
\begin{equation*}
        C^0 \supset C^1 \supset C^2 \supset \cdots \supset C^m =: S 
    \end{equation*}
    and remember that $C^m$ has no boundary. Again, we can construct the retraction. $\Pi_j: C^{j - 1} \to C^j$ for each $j = 0, \dots, m$. We can define the Sharafutdinov retraction to be the composition of these retractions. Namely,
    \begin{equation*}
        \sh = \Pi_m\circ \Pi_{m - 1}\circ \cdots \Pi_2 \circ \Pi_1 \circ \Pi_0
    \end{equation*}
    which is clearly also a $1$-Lipschitz strong deformation retraction. 
\end{proof}
\begin{definition}
The map $\sh: M \to S$ from the above corollary is called a \textbf{Sharafutdinov retraction}.
\end{definition}
Note that at this point we don't know if a 1-Lipschitz deformation retraction onto the soul is unique, hence for now we will say a Sharafutdinov retraction and not the Sharafutdinov retraction.
\section{Perelman Theorem and the Soul Conjecture}
Now, we can study Perelman's proof of the  Cheeger-Gromoll soul conjecture. The key is the following theorem by Perelman. 
\begin{theorem}[Perelman \cite{P94}]\label{thm: perelman's theorem}
    Let $(M^n, g)$ be complete, noncompact Riemannian manifold without boundary with $\sect_M \geq 0$. By the Soul theorem, \ref{thm: soul theorem}, there exists $S \subseteq M$ which is the soul. Let $\sh: M \to S$ be the $1$-Lipschitz Sharafutdinov retraction in \ref{cor: Sharafutdinov}. Then
    \begin{itemize}
    	\item \textbf{Part 1:} For any $p \in S$, and any $v \in \nu_pS = (T_pS)^\perp$, we have 
    		\begin{equation}\label{eq: perelman pi(ex) = p}
        		\sh(\exp_p(tv)) = p \quad\text{for all $t \geq 0$.}
    		\end{equation}
    	\item \textbf{Part 2:} For any $x \in M$, there is a unique $p \in S$ closest to $x$ and  $\sh(x) = p$. In other words, $\sh$ is the nearest point projection onto the soul $S$.
    	\item \textbf{Part 3:} Moreover, if $\gamma: [0, l] \to S$ is a nonconstant geodesic starting at $p, v \in \nu_pS \subseteq T_pM$, we can extend $v$ parallelly along $\gamma$, say $\nu(\gamma(t))$, and exponentiate to a flat totally geodesicically immersed rectangle $R$ in $M$ via $\Gamma(t, s) = \exp(s\nu(\gamma(t)))$, i.e.
    	\begin{equation*}
    		R = \br{\Gamma(t, s): t \in [0, l], s \in [0, \infty)}
    	\end{equation*}
    \end{itemize}
\end{theorem}
\begin{remark}
	\textbf{Part 2} is a trivial implication from \textbf{Part 1}. Let $x \in M$, we can emanate a geodesic $\gamma$ from $x$ to $S$ at $p$. We parameterize $\gamma: [0, 1]$ by $\gamma(0) = p$ and $\gamma(1) = x$, then 
	By the first variation formula, we know that $v = \dga(0) \in \nu_pS$, thus by \textbf{Part 1}, we know that for all $t \geq 0$
	\begin{equation*}
		\sh(\exp_p(tv)) = p 
	\end{equation*}
	which gives us the uniqueness in \textbf{Part 2}. and tells us $\sh$ is the nearest point projection.  This shows that the map  $\sh$ satisfying the conditions of the theorem is unique. therefore from now on we will refer to it as \emph{the} Sharafutdinov retraction.
\end{remark}
\subsection{The Proof of the Soul Conjecture}
Using this theorem \ref{thm: perelman's theorem}, the soul conjecture is an easy corollary.
\begin{corollary}[Soul Conjecture of Cheeger and Gromoll]\label{cor: Soul Conj}
    Let $(M^n, g)$ be a complete open Riemannian manifold with $\sect_M \geq 0$. Suppose there exists $x \in M$ such that $\sect > 0$ near $x$. Then the soul $S$ is a point and 
    \begin{equation*}
        M^n \stackrel{\text{diffeo}}{\cong}\R^n
    \end{equation*}
\end{corollary}
\begin{proof}
    The proof uses the Sharafutdinov retraction $\sh: M \to S$.
    \begin{figure}[htbp]
    \centering
    \includegraphics[width=0.4\textwidth]{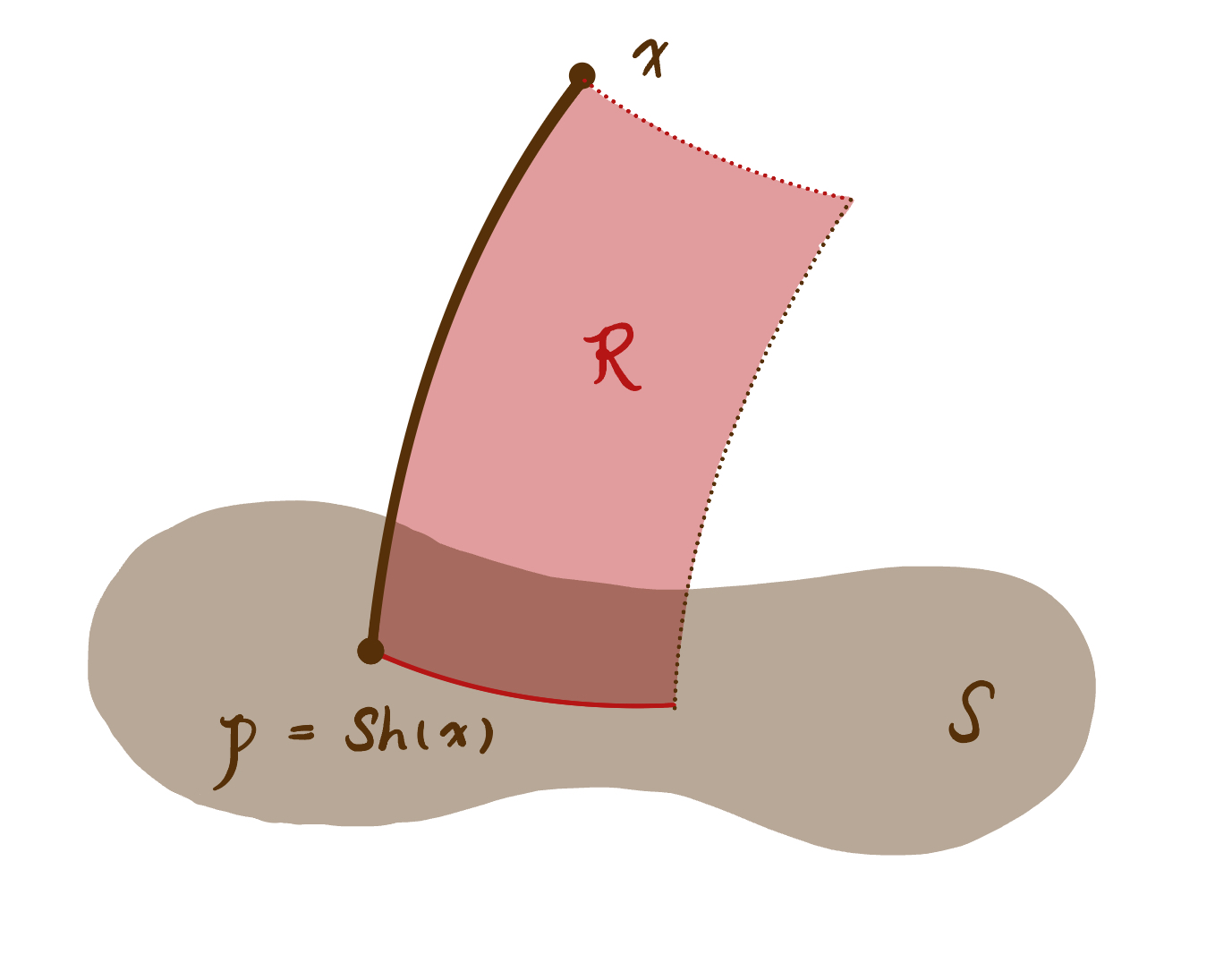}
    \end{figure}
    The existence of the Sharafutdinov retraction $\sh$ follows from the corollary~\ref{cor: Sharafutdinov}. 
    And by 2. in the Perelman theorem \ref{thm: perelman's theorem}, there exists a unique $p \in S$ closest to $x$ such that $\sh(x) = p$. 
    Suppose $S \neq \br{pt}$. Then we can take a nontrivial geodesic $\gamma$ in $S$ starting at $p$. 
    Then by 3. in Perelman's theorem \ref{thm: perelman's theorem}, we can obtain a flat totally geodesic rectangle $R$. The above picture shows that $x \in R$ at the top right corner. 
    Thus, if $\sigma$ is a $2$-dimensional plane at $x$ tangent to $R$, then $\sect(\sigma) = 0$ because  $R$ is flat. 
    However, this contradicts the fact that $
    \sect_M > 0$ near $x$. So $S = \br{pt}$ and therefore,
    \begin{equation*}
        M^n \stackrel{\text{diffeo}}{\cong}\R^n.
    \end{equation*}
\end{proof}
\section{The Proof of Perelman Theorem}
	Define a function $F_p: \nu_pS \to S$ such that $F_p(v) = \sh(\exp_p(v))$ where $\sh$ is the Sharafutdinov retraction. 
	
	In particular, when $v = 0$ and $F_p(0) = p$, we have $F_p = \id$ on all the zero sections. 
	This implies that for any $v \in \nu_pS$ such that $\abs{v} = t$, we have 
    \begin{equation}\label{eq: m(t) leq t}
    	\begin{split}
    	 	d(p, F_p(v)) 
    	 	&= d(\sh(p), \sh(\exp_p(v)))\\
    	 	&\leq d(p, \exp_p(v))\\
    	 	& \leq t 
    	\end{split}
    \end{equation}
    because $\sh$ is $1$-Lipchitz and $\sh(p) = F_p(0) = p$.
    \subsection{Perelman Theorem for Small $t$ (Part 3.)}

	We first consider the case when $t$ is small, i.e. $t < \inj(S)$. In this case, we define a map such that for each $t < \inj(S)$, 
    \begin{equation*}
        m(t) = \max{\br{d(p, F_p(v)): p \in S, v \in \nu_pS, \abs{v} = t}}.
    \end{equation*}	
    One readily checks that $m$ has the following properties
    \begin{itemize}
    	\item $m(0) = 0$;
        \item  $m(t) \leq t$ by \ref{eq: m(t) leq t};
        \item  $m \geq 0$.
    \end{itemize}
    
   	Because $d(p, F_p(v)) \leq t < \inj(S)$, we connect $F_p(v)$ to $p$ via the shortest geodesic $\gamma$ and we can extend $\gamma$ a little past $p$ to some point $q$ so that it remains shortest. Because the three points $p, q, F_p(v)$ lies on the same shortest geodesic, we have precisely 
    \begin{equation*}\label{eq: extension trivial}
        d(F(v), q) = d(F(v), p) + d(p, q). 
    \end{equation*}
    Now, we extend $v$ to a parallel vector field $V$ along $\gamma$ and denote the vector at $q$ by $w$. Then, we can project both $v$ and $w$ to $S$ by $F$. This gives us the triangle $\triangle(F(v), F(w), q)$. 
    \begin{figure}[htbp]
    \centering
    \includegraphics[width=0.8\textwidth]{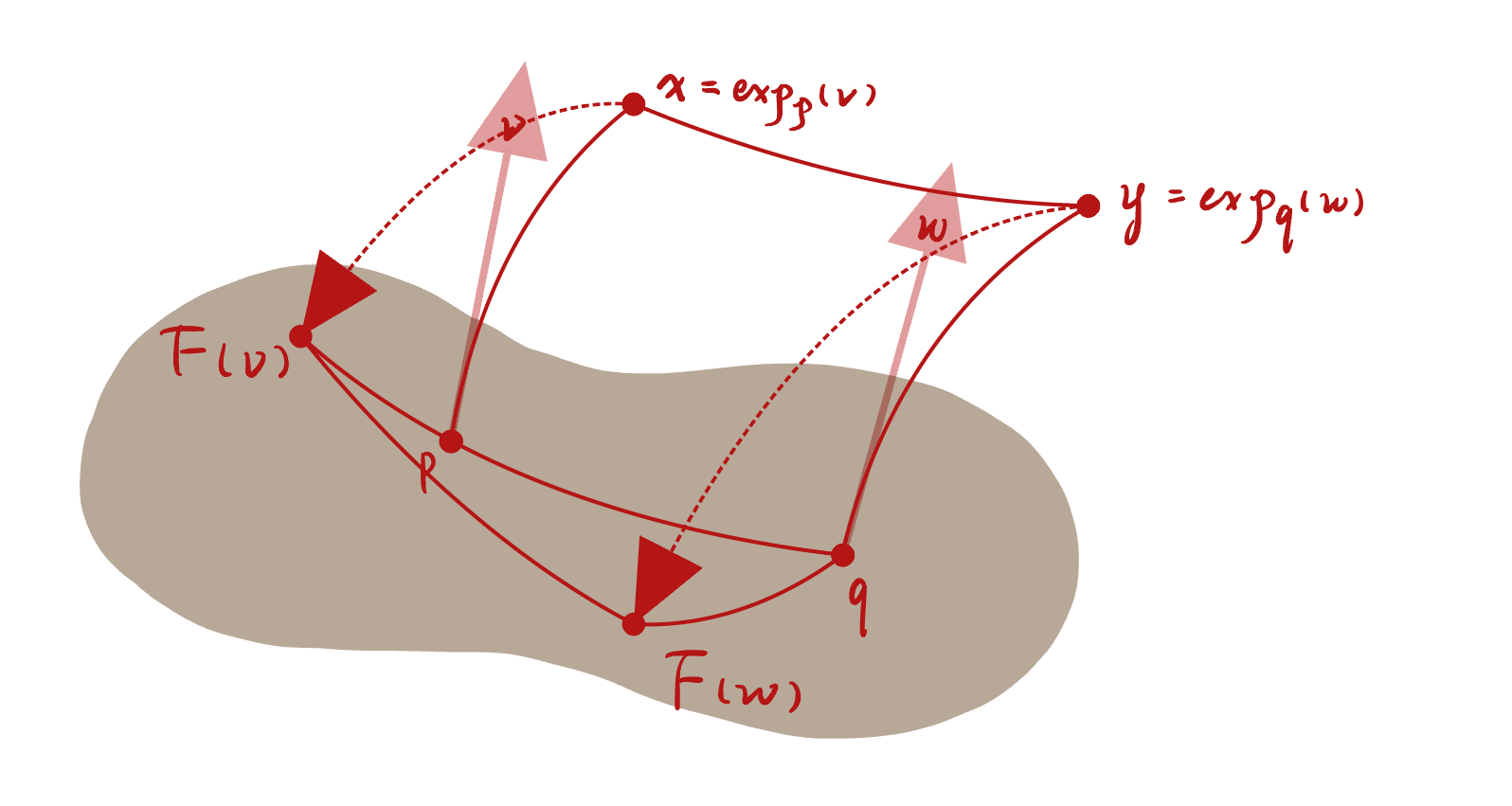}
    \caption{The construction of the triangle $\triangle(F(v), F(w), q)$}
    \end{figure}
    
    By the triangle inequality, we have
    \begin{equation}\label{eq: tri-ineq-perelman}
        d(F(v), q) \le d(F(v), F(w)) + d(F(w), q)
    \end{equation}
    By the maximality of the choice of $v$, we know
    \begin{equation}\label{eq: maximality v and w}
        d(F(w), q) \leq d(F(v), p).
    \end{equation}
    On the other hand, we can check that
    \begin{equation}\label{eq: d(F(v), F(w))}
        d(F(v), F(w))  \leq d(p, q).
    \end{equation}
    This is because 
    \begin{align*}
        d(F(v), F(w)) = 
        & d(\sh(\exp_p(v)),\sh(\exp_q(w)))\\
        \leq & d(\exp_p(v), \exp_q(w)) \quad\text{Because $\sh$ is $1$-Lipschitz}\\
        \leq & d(p, q)\quad\text{By Berger's comparison}
    \end{align*}
    Next, we can continue the triangle inequality \ref{eq: tri-ineq-perelman} by the estimates \ref{eq: maximality v and w} and \ref{eq: d(F(v), F(w))}
    \begin{align*}
        d(F(v), q) = & d(F(v), F(w)) + d(F(w), q) \\
        \leq & d(p, q) + d(F(v), p)\\
        =& d(F(v), q)\quad\text{Remember we extend $d(F(v), p)$ to $q$, thus \ref{eq: extension trivial}.}
    \end{align*}
    We find the inequalities above must be equal! In particular, we conclude that
    \begin{equation}\label{eq: perelman-rectangle}
        d(p, q) = d(\exp_p(v), \exp_q(w)).
    \end{equation}
    By the rigidity case of Berger's comparison equality \ref{eq: perelman-rectangle} tells us that $\Gamma(s, t) = \exp(sV(\gamma(t)))$ is a flat totally geodesic rectangle. 
    
	\subsection{Perelman Theorem for Small $t$ (Part 1.)}
    The following lemma implies the \ref{eq: perelman pi(ex) = p} directly. 
    \begin{lemma}\label{lem: perelman theorem lem 1}
    	$m'_{-}(t) \leq 0$. Then by the properties of $m$, we can conclude that $m \equiv 0$ for all small $t$. 
    \end{lemma}
    \begin{proof}
    	We prove this lemma via the perturbation method. 
   	Pick small some $\eps$, and define $w' := (1 - \frac{\eps}{t})w$ a vector slightly shorter than $w$, i.e. $\abs{w'} = (1 - \frac{\eps}{t})\cdot t = t - \eps$. Thus
    \begin{align*}
        d(F(v), F(w')) = & d(\sh(\exp_p(v)), \sh(\exp_q(w')))\\
        \leq & d(\exp_p(v), \exp_q(w'))\quad \text{Again because $\sh$ is $1$-Lipschitz}\\
        \leq & \sqrt{d(p, q)^2 + \eps^2}\\
        = & d(p, q)\sqrt{1 + \frac{\eps^2}{d(p, q)^2}}\\
        \leq & d(p, q)\brac{1 + \frac{2 \eps^2}{d(p, q)^2}} \quad\text{$\sqrt{(1 + \lambda) \leq 1 + 2\lambda}$}\\
        = & d(p, q) + \frac{2\eps^2}{d(p, q)} = d(p, q) + C\eps^2\quad\text{Where $C = \frac{2}{\abs{pq}}$}.
    \end{align*}
    \begin{figure}[htbp]
    \centering
    \includegraphics[width=0.6\textwidth]{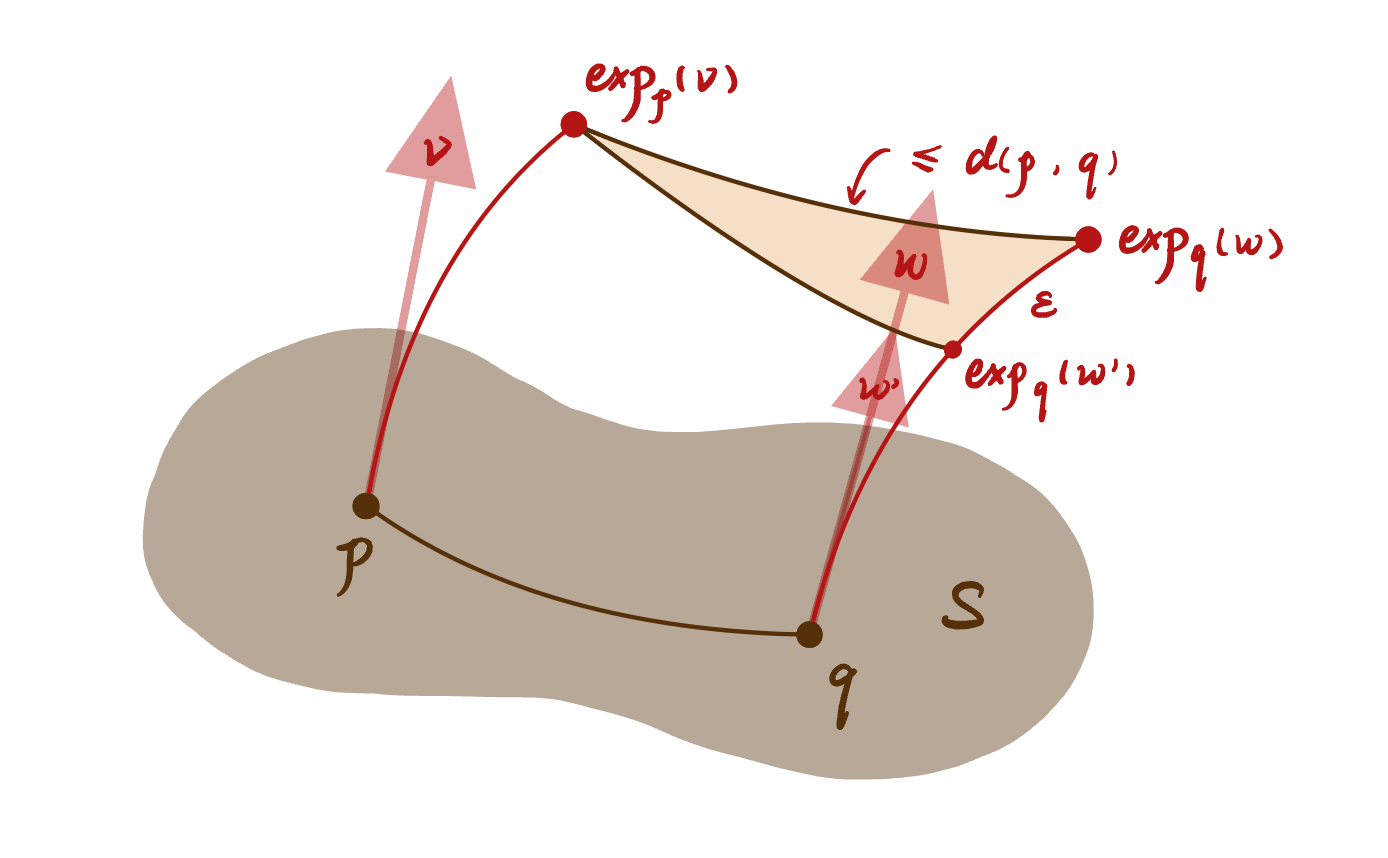}
    \end{figure}
    By the triangular inequality 
    \begin{align*}
        d(F(w'), q) \geq & d(F(v), q) - d(F(v), F(w'))\\
        \geq & d(F(v), p) - C\eps^2\\
        = & m(t) - C\eps^2
    \end{align*}
    Still, by the definition of $m$, we have
    \begin{equation*}
        m(t - \eps) \geq d(F(w'), q) \quad\text{By the maximality of $m$}
    \end{equation*}
    Thus
    \begin{align*}
        & m(t - \eps) \geq m(t) - C\eps^2\\
        \implies & m_-'(t) \leq 0\\
        \implies & \text{$m$ is nonincreasing}
    \end{align*}
Since $m\ge 0$ and  $m(0)=0$ we can conclude that $m \equiv 0$ for all small $t$.
    \end{proof}
	\begin{remark}
   		Notice that it is possible that the above inequality 
   		\begin{equation*}
   			d(\exp_p(v), \exp_q(w')) \leq \sqrt{d(p, q)^2 + \eps}
   		\end{equation*}
   		is a strict inequality. This is because the rectangle is not necessarily convex so the shortest geodesic of the rectangle is not necessarily the shortest geodesic in the manifold. 
   	\end{remark}

\subsection{Perelman's Theorem for Large $t$}

The proof of the equality \ref{eq: perelman pi(ex) = p} can be generalized to large $t$ in \cite{Per94}. 
Let 
\begin{equation*}
	T = \max{\br{t : m|_{[0, t]} = 0}}. 
\end{equation*}
We want to show that $T = \infty$. Suppose $T < \infty$. Then for small $0<\eps<\inj(S)$ we have that for any  $0<t<\eps, p \in S, v \in \nu_pS  $ a unit vector it holds that $d( F((T+t)v), p)=d( F((T+t)v), F(Tv))<\eps<\inj(S)$.

 Then Lemma~\ref{lem: perelman theorem lem 1} shows that there exist a small $\eps > 0$ such that $m \equiv 0$ on $[T, T + \eps)$. This means that $T$ is not the maximum. Therefore, $T = \infty$.



\section{Regularity of Sharafutdinov Rectraction}

The proof of the Perelman theorem implies that the Sharafutdinov retraction $\sh: M \to S$ is a \emph{submetry}, which implies some regularity of $\sh$. What we mean by a submetry is rigorously defined below.
\begin{definition}[Submetry]
    Let $(X, d^X)$ and $(Y, d^Y)$ be two metric spaces. Then $f: (X, d^X) \to (Y, d^Y)$ is a \textbf{submetry} if
    \begin{equation*}
        f(B^X_R(p)) = B^Y_R(f(p)) \quad \text{$\forall p \in X$ and $\forall R > 0$}. 
    \end{equation*}
    (Notice that submetries are $1$-Lipschitz by  definition.)
\end{definition}
\begin{example}
 Here are some examples of submetries.
    \begin{itemize}
        \item The first coordinate projection from $X = Y \times F$ to $Y$ is a submetry.
        \item Any Riemannian submersion between two complete and connected Riemannian manifolds is a submetry. Let $f: (M^m, g) \to (N^n, h)$ be a Riemannian submersion where $M$ and $N$ are  complete. By the definition of a Riemannian submersion, it is trivial that 
        \begin{equation*}
            f(B_R(p)) \subseteq B_R(f(p)).
        \end{equation*}
        Now we want to show 
        \begin{equation*}                       B_R(f(p))\subseteq f(B_R(p)) 
        \end{equation*}
     This is because we can lift the geodesics horizontally. Fix $q \in B_R(f(p))$, Denote by $\gamma$ the shortest geodesic from $f(p)$ to $q$ parameterized by its arc-length $l:= \length(\gamma) = d(f(p), q)$. We can lift $\gamma$ horizontally to $\overline{\gamma}$ starting at $p$, there is $\overline{\gamma}$ and $\overline{p}$ such that
        \begin{itemize}
            \item $f(\overline{\gamma}(t)) = \gamma(t)$ for each $t$;
            \item $f(\overline{\gamma}(0)) = f(p)$ and $f(\overline{\gamma}(l)) = f(\overline{q}) = q$;
            \item $\length(\gamma) = \length(\overline{\gamma})$
        \end{itemize}
        The last point implies that
        \begin{equation*}
          	d(p, \overline{q}) \le d(f(p), q) < R \implies \overline{q} \in f(B_R(p))
        \end{equation*} 
        therefore, we can conclude that
        \begin{equation*}
            B_R(f(p)) = f(B_R(p))
        \end{equation*}
        so that $f$ is a submetry. 
        
       In fact this shows that  $d(p, \overline{q}) = d(f(p), q)$ since the 1-Lipschitz condition implies that  $d(p, \overline{q}) \ge d(f(p), q)$.
        This means that the fibers if $f$ are equidistant and the distance between them is equal to the distance between their images. That is given $y_1,y_2\in Y$ for any $p\in f^{-1}(y_1)$ it holds that $d(p,  f^{-1}(y_1))=d(y_1,y_2)$.
        This property more generally holds for all submetries, not just Riemannian submersions.

        \item Suppose $G$ is a compact group acting on $X$  by isometries where $X$ is a complete geodesic metric space. Take $Y = X/G$. Then the orbits of this action are equidistant and $X/G$ admits a natural quotient metric
        and the projection $\pi: X \to Y$ is a submetry. 
	\end{itemize}
\end{example}
We used the following result by Guijarro and Walschap as a fact.
\begin{theorem}[Regularity of Submetries between manifolds~\cite{GW09}]\label{thm: submetry regularity}
    Suppose $f: M^m \to N^n$ is a submetry between two smooth compact Riemannian manifolds. Then $f$ is a $C^{1,1}$-Riemannian submersion, where $C^{1,1}$ means $\frac{\partial f_i}{\partial x_j}$ are locally Lipschitz in coordinates.
\end{theorem}
\begin{remark}
    The $C^{1,1}$ regularity in the above theorem is sharp and can not be improved in general.
\end{remark}
\begin{proposition}\label{cla: soul submetry}
    The Sharafutdinov retraction $\sh: M \to S$ in  theorem \ref{thm: perelman's theorem}  is a submetry
\end{proposition}
\begin{proof}[Proof of Proposition~\ref{cla: soul submetry}]
It is easy to see that $\sh(B_R(x)) \subseteq B_R(\sh(x))$ since $\sh$ is $1$-Lipschitz. Now we want to show $B_R(\sh(x)) \subseteq \sh(B_R(x))$. Firstly we denote $p = \sh(x)$ and fix $q \in B_R(p) \subseteq S$ such that $\abs{pq} < R$. To show $q \in \sh(B_R(x))$ we take $\sigma$ as the shortest geodesic from $p$ to $x$ and we know that $\sigma \perp S$. Also, we take $\gamma$ to be the shortest geodesic from $p$ to $q$.

\begin{figure}[htbp]
    \centering
    \includegraphics[width=0.6\textwidth]{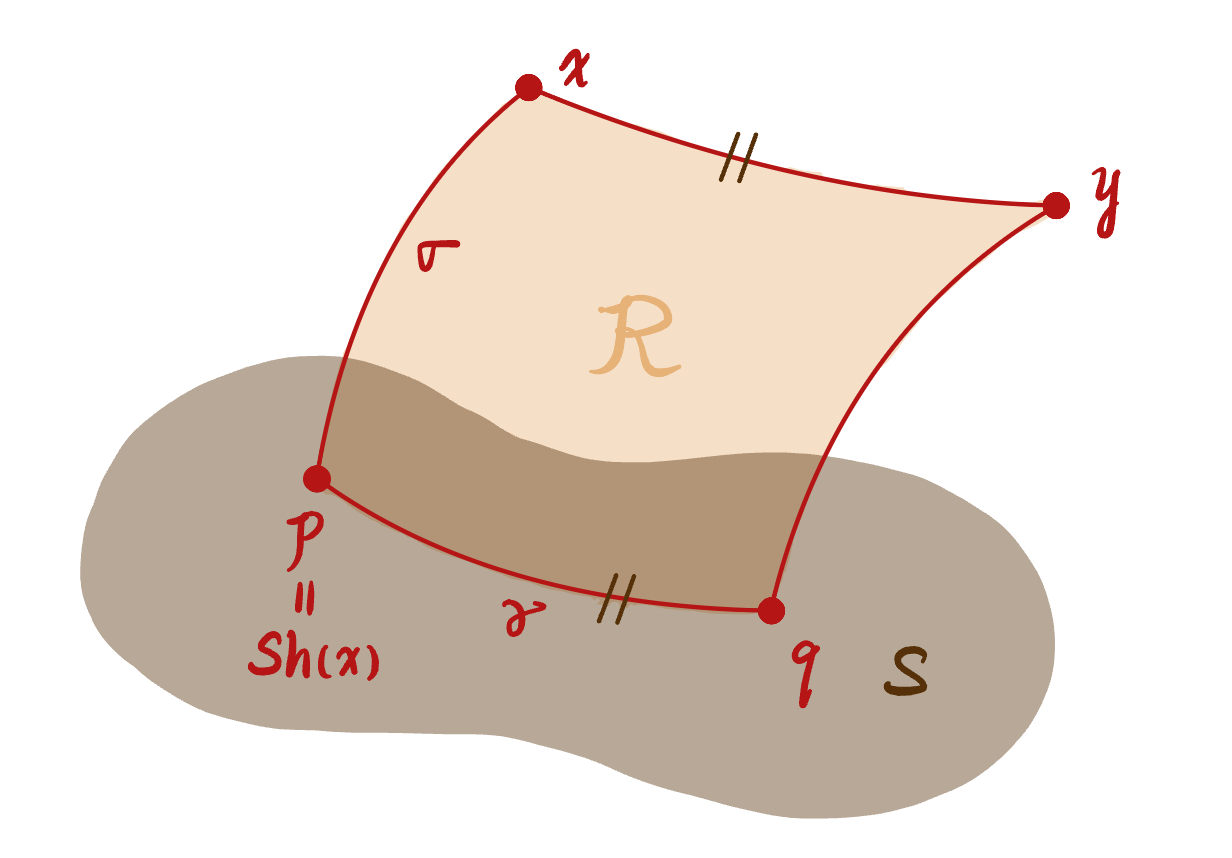}
    \end{figure}
    
 Then by Perelman's theorem \ref{thm: perelman's theorem}, we have a flat totally geodesic rectangle with sides $\gamma$ and $\sigma$. Then by the picture, we know that there is $y \in B_R(x)$ such that $\sh(y) = q$ and $d(x, y) = d(p, q) < R$.

This is equivalent to say $q = \Pi(y) \in B_R(\Pi(x))$.
\end{proof}
Combining the claim \ref{cla: soul submetry} and the regularity result of submetry \ref{thm: submetry regularity}, we can conclude the following corollary.
\begin{corollary}
    The Sharafutdinov retraction $\Pi$ in the Soul theorem is a $C^{1,1}$ Riemannian submersion.
\end{corollary}
Indeed the full regularity result that $\sh$ is $C^\infty$ is proved by Wilking in 2007 \cite{Wi07}. This is a very difficult result and we are not going to prove this in the lecture. 

In the Soul conjecture \ref{cor: Soul Conj}, we know that $\nu(S)$ is diffeomorphic to $M$. Does that mean the normal exponential map $\nu(S) \to M$ is a diffeomorphism? The answer to this is NO. The map is generally not a diffeomorphism even when $S = \br{pt}$. Consider the following example.
\begin{example}
   \corrv{ Take $M \subseteq \R^3$ to be a surface obtained by rotating the parabola $y=x^2$ around the $y$-axis.} This surface has $\sect_M > 0$. And the Soul of the surface is $S = \br{p}$. Now the normal exponential map $\exp_p: T_pM \to M$ is a diffeomorphism. 
    
    Now we deform this metric to be slightly bumpy and still keep the sectional curvature strictly positive. This is possible since the curvature of the paraboloid of revolution is positive away from the origin. Then there exists a $q \in M$ with non-unique shortest geodesic from $p$ to $q$ and that makes the exponential map no longer diffeomorphism.
\end{example}
In this example, because $S = \br{p}$, thus $T_pM = \nu(S)$ is the normal bundle. We should learn from this example that $\nu(S) \stackrel{\text{diffeo}}{\cong} M$ does not imply that the normal exponential map $\exp_p: \nu(S) \to M$ is a diffeomorphism.

\chapter{Cheeger-Gromoll Splitting Theorem}
We now going to move to our next topic - The splitting theorem of the manifolds of non-negative sectional curvature. The initial version of the splitting theorem is proved in the same paper of Cheeger-Gromoll where they proved the Soul theorem. 
\begin{theorem}[Splitting Theorem by Cheeger-Gromoll]\label{thm: splitting-cg}
Let $(M^n, g)$ be a complete open manifold with $\sect_M \geq 0$ and $M$ contains a line $\gamma: \R \to M$ (global geodesic) such that 
\begin{equation*}
    d(\gamma(t), \gamma(s)) = \abs{t - s} \quad \forall s, t \in \R
\end{equation*}
Then there is an isometric splitting on $M$
\begin{equation*}
    M^n \stackrel{\text{isom}}{\cong} N^{n - 1} \times \R
\end{equation*}
where $N^{n - 1}$ also a complete open manifold with $\sect_N \geq 0$.
\end{theorem}
\begin{example}
    Consider $M^n = S^{n - 1} \times \R^n$ where $\gamma$ in this case are just the vertical lines $t \to (p, t)$ for each $p \in S^{n - 1}$.
\end{example}
\begin{remark}
    The splitting theorem~\ref{thm: splitting-cg} also holds for Alexandrov spaces of nonnegative curvature (We will study Alexandrov spaces in Chapter~\ref{Ch13}).  The proof of the splitting theorem in the setting of nonnegative Ricci curvature manifolds can be found in \cite{EH84, CG71}. The splitting theorem for $RCD(0, N)$ space was also discussed in Gigli's lecture in Fields Institute in the Fall of  2022 ~\cite{Gig22}. 
    \end{remark}
\section{The Proof of the Splitting Theorem in the Smooth Setting}
    Let $\gamma: \R \to M$ be a line. We can define two rays $\gamma_+, \gamma_-: [0, \infty) \to M$, i.e. 
    $\gamma_+(t) := \gamma(t)$ for all $t \geq 0$ and $\gamma_{-}(t) := \gamma(-t)$ for all $t \geq 0$. Then, we can associate the corresponding Busemann functions $b_+$ and $b_-$ to them. Remember that $b_+$ and $b_-$ are concave $1$-Lipschitz. 
    Firstly, we claim that $b_+ + b_- \geq 0$ on $M$.
    \begin{lemma}\label{cla:splitting-busemann nonnegative}
        $b_+ + b_- \geq 0$ on $M$. 
    \end{lemma}
    \begin{proof}
    For any $t \geq 0$, we have 
    \begin{equation*}
    	d(\gamma_+(t), \gamma_-(t)) = \gamma(\gamma(t), \gamma(-t)) = 2t.
    \end{equation*}
    Therefore, for any $x \in M$, by the triangular inequality, we have
    \begin{equation*}
    	d(x, \gamma_+(t)) + d(x, \gamma_-(t)) \geq 2t
    \end{equation*}
	Notice that this gives us
	\begin{equation*}
		\brac{d(x, \gamma_+(t)) - t} + \brac{d(x, \gamma_-(t)) - t} \geq 0
	\end{equation*}
	By taking limits $t \to \infty$, we got 
	\begin{equation*}
		b_+(x) + b_-(x) = \lim_{t \to \infty}(d(x, \gamma(t)) - t) + \lim_{t \to \infty}(d(x, \gamma(-t)) - t) \geq 0. 
	\end{equation*}
    \end{proof}

    \begin{lemma}
    	In particular, we have 
    	\begin{equation*}
    		 b_+(\gamma(t)) + b_-(\gamma(t)) = 0 \quad \forall t \in \R
    	\end{equation*}
    \end{lemma}
    \begin{proof}
    	We first consider the case when $t \geq 0$, In this case, 
    	\begin{equation*}
    		b_+(\gamma(t)) = b_+(\gamma_+(t)) = -t
    	\end{equation*}
    	and
    	\begin{align*}
    		b_-(\gamma(t)) 
    		&= b_+(\gamma_+(t))\\
    		&= \lim_{s \to \infty} \brac{d(\gamma_+(t), \gamma_-(s)) - s}\\
    		&= \lim_{s \to \infty} \brac{d(\gamma(t), \gamma(-s)) - s}\\
    		&= \lim_{s \to \infty}\abs{t - (-s)} - s\\
    		&= t
    	\end{align*}
    	Thus, 
    	\begin{equation*}
    		b_+(\gamma(t)) + b_-(\gamma(t)) = 0. 
    	\end{equation*}
    	The proof for $t \leq 0$ is exactly the same. 
    \end{proof}

    Denote
    \begin{equation*}
    	b=b_+ + b_-.
    \end{equation*}
    We notice that $b$ is concave,  nonnegative, and $b(p)=0$ (since $p = \gamma(0)$).
    
    \begin{proposition}
    	\label{prop: b = 0}
   		$b\equiv 0$ on $M$ 
    \end{proposition}
    There are two approaches to proving this lemma, the first one relies on the extendability of a geodesic in the complete Riemannian manifold. This approach relies on the smoothness of the space. Nevertheless, the second approach doesn't, which therefore can also be applied to the setting of Alexandrov space. 
    
    \begin{proof}[Proof of Proposition~\ref{prop: b = 0} (Method 1)]
    	We only need to show the following statement:
    	\begin{center}
    		Let $M$ be a complete Riemannian manifold and $f: M\to \R$ be a concave, non-negative function. And, we require $f(p)=0$ for some $p\in M$. Then $f\equiv 0$.
    	\end{center}
    	    Let $x\in M$. Let $\sigma: \R\to M$ be an infinite geodesic such that $\sigma(0)=p$ and $\sigma(l)=x$. Then $f(\sigma(t))$ is concave, nonnegative, and vanishes at 0. Since
    	    \begin{equation*}
    	    	0 =f(\sigma(0))\geq \frac{1}{2}f(\sigma(l)) + \frac{1}{2}f(\sigma(-l)) 
    	    \end{equation*}
    	    And because $f$ is non-negative, we must have
    	    \begin{equation*}
    	    	f(\sigma(l)) = f(\sigma(-l)) = 0 
    	    \end{equation*}
    	    We therefore can conclude that $f(x) = f(\sigma(l)) = 0$. 
    \end{proof}
    \begin{proof}[Proof of Proposition~\ref{prop: b = 0} (Method 2)]
        We know that $d(x, \gamma(t))^2$ is $2$-concave since $\sect_M \geq 0$, i.e. $d^2(x, \gamma(t)) - t^2$ is concave in $t$. And we know that every concave function $\leq$ linear function. Thus, there exist $A, B \in \R$ such that for any $t \in \R$
        \begin{align*}
            &d^2(x, \gamma(t)) - t^2 \leq At + B \\
            \iff& d^2(x, \gamma(t)) \leq t^2 + At + B
        \end{align*}
        Also, we have that
        \begin{equation*}
            d^2(x, \gamma(-t)) \leq t^2 - At + B
        \end{equation*}
    Then we can add them together such that
    \begin{align*}
        d(x, \gamma(-t)) + d(x, \gamma(t)) - 2t
        \leq & \sqrt{t^2 + At + B} + \sqrt{t^2 - At + B} - 2t\\
        =& (\sqrt{t^2 + At + B} - t) + (\sqrt{t^2 - At + B} - t)\\
        =& \brac{t\sqrt{1 + \frac{A}{t} + \frac{B}{t^2}} - t} + \brac{t\sqrt{1 - \frac{A}{t} + \frac{B}{t^2}} - t}\\
    \end{align*}
    Since $\sqrt{1 + x} \leq 1 + \frac{x}{2}$, 
    \begin{align*}
    	& d(x, \gamma(-t)) + d(x, \gamma(t)) - 2\\
        \leq & \brac{t\brac{ 1 + \frac{1}{2}\brac{\frac{A}{t} + \frac{B}{t^2}}} - t} + \brac{t\brac{ 1 - \frac{1}{2}\brac{\frac{A}{t} - \frac{B}{t^2}}} - t}\\
        \leq& \frac{1}{2}\brac{A + \frac{B}{t}} - \frac{1}{2}\brac{A - \frac{B}{t}}\\
        =& \frac{B}{2t} + \frac{B}{2t} = \frac{B}{t}
    \end{align*}
    We notice that
    \begin{align*}
    	b(x) 
    	&= b_-(x) + b_+(x)\\
    	&= \lim_{t \to \infty}\brac{(d(x, \gamma(-t)) - t) + (d(x, \gamma(t)) - t)}\\
    	&= \lim_{t \to \infty}\brac{d(x, \gamma(-t)) + d(x, \gamma(t)) - 2t}\\
    	&\leq \lim_{t \to \infty}\frac{B}{t} = 0
    \end{align*} 
   	This means, $b  \leq 0$ on $M$.
	However, we already know that $b\geq 0$ by the lemma~\ref{cla:splitting-busemann nonnegative}. Thus, we can conclude $b \equiv 0$.
    \end{proof}
{ The advantage of method 2 of the proof is that unlike method 1 it does not rely on existence of bi-infinite geodesics in all directions and generalizes to Alexandrov spaces.}
    
    Since both $b_+$ and $b_-$ are concave, then $-b_+$ and $-b_-$ are convex. However, since $b_+ = - b_-$, these two Busemann functions are both concave and convex. Therefore, they are affine. That means for any geodesics $\sigma$ in $M$, 
    \begin{equation*}
        b_+(\sigma(t)) = at + c
    \end{equation*}
    for some constants $a$ and $c$. 
    \begin{lemma}\label{cla: busemann convex level set}
        The set $\br{b_+ = c}$ is totally convex in $M$ for every $c \in \R$
    \end{lemma}
    \begin{proof}
        Since
        \begin{align*}
            \text{$b_+$ is concave} \implies \text{$\br{b_+ \geq c}$ is totally convex}\\
            \text{$b_+$ is convex} \implies \text{$\br{b_+ \leq c}$ is totally  convex}
        \end{align*}
        Therefore, their intersection
        \begin{equation*}
            \br{b_+ = c} = \br{b_+ \geq c}\cap\br{b_+ \leq c}
        \end{equation*}
        is totally convex.
    \end{proof}
    \begin{remark}
    	However, the intersection of two convex subsets is not necessarily convex. In fact, the intersection is not necessarily connected. By Figure~\ref{fig: intersection of A and B}, $A$ and $B$ are both convex subsets of $S^2$. However, their intersection is just a set containing two points.
    	\begin{figure}[htbp]
    \centering
    \includegraphics[width=0.7\textwidth]{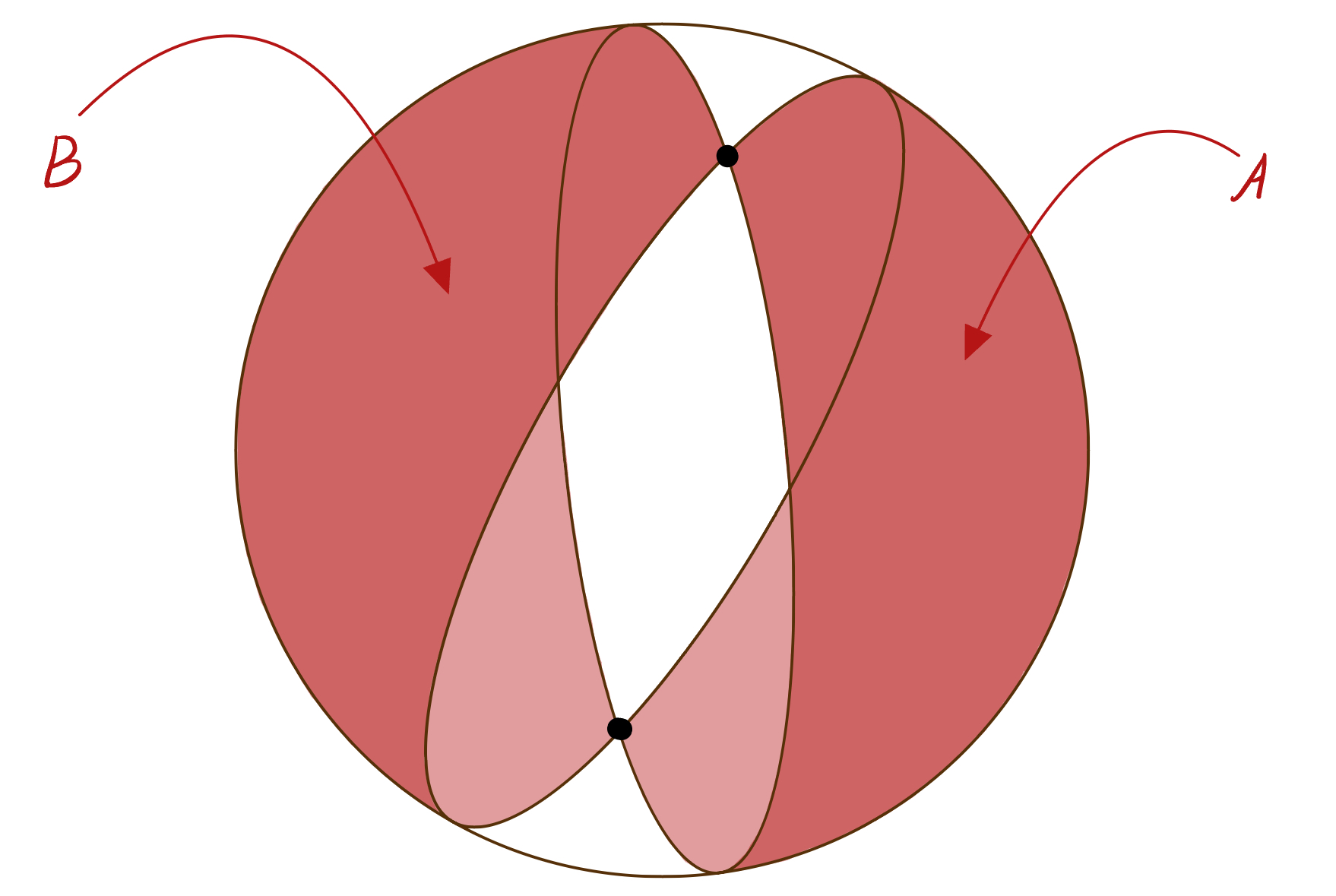}
    	\caption{The intersection of two convex subsets of $S^2$}
    	\label{fig: intersection of A and B}
    \end{figure}
    \end{remark}
\subsection{The Construction of the Gradient Flow}
Recall that both  $b_+$ and $b_-$ are affine. 
Therefore if $x, y \in \br{b_+ = c}$ and $\sigma: \R\to M$ is a geodesic passing through $x,y$ then $b_+\equiv c$ on $\sigma$,
Therefore the convex set $ \br{b_+ = c}$  is a totally geodesic submanifold without boundary.
    
Denote $\gamma_+^x$ and  $\gamma_-^x$ be a pair of negative gradient rays of $b_+$ and $b_-$ respectively starting from $x$. 
Since for every $t$, 
    \begin{equation*}
        b_+(\gamma_+^x(t)) = b_+(x) - t, \quad b_-(\gamma_-^x(t)) = b_-(x) - t
    \end{equation*}
    and $b_+ = -b_-$ (By proposition~\ref{prop: b = 0}), we have 
    \begin{align*}
        b_-(\gamma_+^x(t)) = (-b_+)(\gamma_+^x(t)) = -b_+(x) + t = b_-(x) + t. 
    \end{align*}
    and
    \begin{align*}
        b_+(\gamma_-^x(t)) = (-b_-)(\gamma_-^x(t)) = -b_-(x) + t = b_+(x) + t. 
    \end{align*}
    We summarize the results here
    \begin{itemize}
    	\item $b_+(\gamma_+^x(t)) = b_+(x) - t$
        \item $b_-(\gamma_-^x(t)) = b_-(x) - t$
        \item $b_+(\gamma_-^x(t)) = b_+(x) + t$
        \item $b_-(\gamma_+^x(t)) = b_-(x) + t$
    \end{itemize}
    Now, let's construct the joint curve
    \begin{align}
        \gamma^x(t) = 
        \begin{cases}
            \gamma_+^x(t)\quad t \geq 0\\
            \gamma_-^x(-t)\quad t < 0
        \end{cases}
    \end{align}
    So along $\gamma^x$, because $b_\pm$ are $1$-Lipschitz
    \begin{align}\label{eq-along-flow}
        & b_-(\gamma^x(t)) = 
        \begin{cases}
            b_-(\gamma_+^x(t)) = b_-(x) + t \quad t \geq 0 \\
            b_-(\gamma_-^x(-t)) = b_-(x) + t \quad t < 0 
        \end{cases}
        \\
        \implies& b_-(\gamma^x(t)) = b_-(x) + t \quad \forall t \in \R\\
        \implies& \text{$\gamma^x(t)$ is the gradient curve of $b_-$}
    \end{align}
    and
    \begin{align*}
        &b_+(\gamma^x(t)) = 
        \begin{cases}
            b_+(\gamma_+^x(t)) = b_+(x) - t \quad t \geq 0 \\
            b_+(\gamma_-^x(-t)) = b_+(x) - t \quad t < 0 
        \end{cases}
        \\
        \implies& b_+(\gamma^x(-t)) = b_+(x) + t \quad \forall t \in \R\\
        \implies& \text{$\gamma^x(-t)$ is the gradient curve of $b_+$}
    \end{align*}    
    Denote $\Phi_t^+$ and $\Phi_t^-$ the gradient flows for $b_+$ and $b_-$ respectively. The lines $\gamma^x$ are trajectories of the flows passing through $x \in M$. Because for every $\gamma^x$-line such that $\gamma^x(0) = x$. 
    \begin{equation*}
        \Phi_t^{-1}(\gamma^x(t)) = x \implies \Phi_t^+ = (\Phi_t^-)^{-1},
    \end{equation*}
    we can say that $\forall t \geq 0$,
    \begin{equation*}
        \Phi_t^+ = (\Phi_t^-)^{-1}.
    \end{equation*}
    Moreover, because both $b_+$ and $b_-$ are concave, both $\Phi_t^+$ and $\Phi_t^-$ are 1-Lipschitz. However, $\Phi_t^+ \circ \Phi_t^- = \id$. Then $\forall x, y \in M$, 
    \begin{align*}
        d(x, y) 
        =& d(\Phi_t^+ \circ \Phi_t^-(x), \Phi_t^+ \circ \Phi_t^-(y))\\
        \leq & d(\Phi_t^+(x), \Phi_t^+(y)) \quad\text{Because $\Phi_t^\pm$ are $1$-Lipschitz}\\
        \leq & d(x, y) \quad \text{Because $\Phi_t^\pm$ are $1$-Lipschitz}
    \end{align*}
    Therefore, all the inequalities are equal. Thus
    \begin{align*}
        d(\Phi_t^+(x), \Phi_t^+(y)) = d(x, y)\\
        d(\Phi_t^-(x), \Phi_t^-(y)) = d(x, y)
    \end{align*}
    Therefore, we can define an isometric global flow 
    \begin{equation*}
        \Phi_t(x) = 
        \begin{cases}
            \Phi_t^-(x) \quad t \geq 0\\
            \Phi_{-t}^+(x) \quad t < 0
        \end{cases}
    \end{equation*}
    This is clearly a flow since
    \begin{equation*}
        \Phi_s\circ\Phi_t = \Phi_{s + t}\quad \forall s, t \in \R
    \end{equation*}
    
    \subsection{The Flow Maps are Bijective}
  	We can think of it as the gradient flow of $b_-$ but defined on all of $\R$ as opposed to $\R_{\ge 0}$ only. By above we have that $b_-$ increases with speed 1 along the flow $\Phi_t$.
   	That is 
   	\[
   	b_-(\Phi_t(x))=b_-(x)+t \quad \forall  t \in \R
   	\]
   	Similarly $b_+$ increases with speed $1$ along its gradient flow $\Phi_{-t}(x)$. By Lemma~\ref{cla: gradient curve exp}, for each $x \in M$, the curve $t \mapsto \Phi_t(x)$ is a line as $\Phi_{\mp t}(x)$ are the unit speed gradient curves for $b_\pm$ respectively.

    Let's look at the map $F: N \times \R \to M$ where we define $N = \br{b_+ = 0} = \br{b_- = 0}$. For each $(x, t) \in N \times \R$
    \begin{equation*}
        F(x, t) := \Phi_t(x).
    \end{equation*}
    Because our manifold is foliated by the gradient curves of $b_-$ it easily follows that $F$ is a bijection since all of our flow maps are invertible. 

  Suppose $F(x_1,t_1)=F(x_2,t_2)$. This means  $\Phi_{t_1} (x_1)= \Phi_{t_2} (x_2)$. Applying $\Phi_{-t_2}$ to both sides gives
  $\Phi_{t_1-t_2}(x_1)=x_2$. We have $b_-(x_2)=0$ and $b_-( \Phi_{t_1-t_2}(x_1))=t_1-t_2$.  Hence $t_1-t_2=0$, $x_1=\Phi_0(x_1)=x_2$ and hence $x_1=x_2, t_1=t_2$. This shows that $F$ is 1-1. Now we want to show $F$ is onto. Let $y \in M$, we want to show there exists $x \in N$ and $t \in \R$ such that $F(x, t) = \Phi_t(x) = y$. Let $t = b_-(y)$, and consider $\gamma^y$ the line passing through $y$, by applying $b_-$, we have
  \begin{equation*}
  		b_-(\gamma^y(-t)) = b_-(y) - t = t - t = 0
  \end{equation*}
  thus we know that $x = \gamma^y(-t) \in N$ thus we can conclude that
  \begin{equation*}
  	F(x, t) = \Phi_t(\gamma^y(-t)) = y.
  \end{equation*}
  \subsection{The Flow Maps are Isometries}
  	Now we only need to verify that to $F$ is an isometry. Assuming we know that $\nabla b_-$ is $C^\infty$, we can conclude the argument as follows.  Consider the map
        \begin{equation*}
            dF_{(x, t)}: T_xN \times \R \to T_{\phi_t(x)}M
        \end{equation*}
    For any fixed $t$ look at the map  $N\times t \to \br{b_- = b_-(x) + t}$. 
    Since $\Phi_t$ is an isometry this map is an isometry as well and the same is true for its differential at any point in $N$.
    Also, we know that $\Phi_t(x) = \gamma^x(t)$ is a line orthogonal to the level sets of $b_-$.
   	In particular, its derivative is a unit vector for any $t$. Note that the tangent space $T_{(x,t)}(N\times \R)$ is a direct orthogonal sum of $T_xN$ and $\R$ and by above the differential preserves their orthogonality and is an isometry on each component. 
   	Thus  $dF_{(x,t)}$ is a block-diagonal matrix with each block an orthogonal matrix. 
   	Hence $dF_{(x,t)}$ is an isometry.
   	Since $(x,t)$ is arbitrary and $F$ is a bijection this means that $F$ is an isometry.
    And thus we finished the proof.     

    Therefore, $F$ is an isometry and $M \stackrel{\text{isom}}{\cong} N \times \R$. 

\begin{remark}
   The arguments in the last part of the proof only work if we know the flow is smooth (at least $C^2$). It can be shown that this is the case in our situation but we are going to present a different argument that does not rely on the smoothness of the flow and works on Alexandrov spaces and not just on the Riemannian manifold. 
\end{remark}
\begin{remark}
    The splitting theorem in $RCD(0, N)$ space is much harder to proof because the lower Ricci curvature bound is rather a weak assumption.
\end{remark}


\section{The Proof of the Splitting Theorem in the Non-Smooth Setting}
In the last lecture, we proved the splitting theorem for nonnegative sectional curvature Riemannian manifolds. However, the proof last time can only be made of only when the flow is $C^\infty$. Therefore, we need an alternative proof to generalize the theorem to Alexandrov space of nonnegative curvature in this lecture. Still, our goal of this section is to show $F(x, t) = \Phi_t(x)$ is an isometry. 
\subsection{Petrunin's Estimate}
Recall from the last section, we already show that $F$ is a bijection. And we know for each $t$, $\Phi_t$ is an isometry from $M$ to itself. This means if we define $N = \br{b_- = 0}$, the restriction map
    \begin{equation*}
        F|_{N \times {t_0}}: N \times \br{t_0} \to \br{b_- = t_0}
    \end{equation*}
    is an isometry. 
    
    However, this only shows that $F$ preserves distances between points in the same slice, i.e. for $p, q \in N \times \br{t_0}$, we have $d(p, q) = d(F(p), F(q))$.

    What if $p$ and $q$ are in different slices where $p = (y_1, t_1) \in N \times \br{t_1}$ and $q = (y_2, t_2) \in N \times \br{t_2}$? This is nontrivial. To estimate the distance between $\Phi_{t_1}(y_1)$ and $\Phi_{t_2}(y_2)$. We need the following proposition by Petrunin.
    \begin{proposition}[Petrunin]\label{prop: petrunin}
        Let $(M^n, g)$ be a Riemannian manifold with $\sect_M \geq 0$. Let $f$ be a $1$-Lipschitz and concave function over $M$. And $\Phi_t$ the gradient flow of $f$ for $t \geq 0$. Let $p, q \in M$ and $s, t \in M$. Then
        \begin{equation}\label{eq: flow estimate}
            d(\Phi_s(p), \Phi_t(q))^2 \leq d(p, q)^2 + 2(f(p) - f(q))(s - t) + (s - t)^2
        \end{equation}
    \end{proposition}
    The proof of Petrunin's estimate will be given in the next section. In this section, we want to apply Petrunin's estimate to show $F: N \times \R \to M$ is an isometry. Indeed, it is enough to show that 
    \begin{equation*}
    	d(F(p, s), F(q, t)) = \sqrt{d(p, q)^2 + (t - s)^2}
    \end{equation*}
    for $p, q \in N$, $s, t \in \R$. 
    
    The key point here is to apply Petrunin's estimate applies both to the flow $t\to \Phi_t$ where $t \ge 0$ which is the gradient flow of $b_-$ and to the flow $t\to \Phi_{-t}$ where $t\ge 0$ which is the gradient flow of $b_+$.
    
    Without loss of generality, suppose $s \leq t$, then 
    \begin{align*}
    	F(p, s) &= \Phi_s(p);\\
    	F(q, t) &= \Phi_t(q) = \Phi_s(\Phi_{t - s}(q))
    \end{align*}
    As we mentioned in the beginning, $\Phi_s$ is an isometry. We therefore have
    \begin{equation*}
    	d(F(p, s), F(q, t)) = d(\Phi_s(p), \Phi_t(q)) = d(\Phi_s(p), \Phi_s(\Phi_{t - s}(q))) = d(p, \Phi_{t - s}(q)). 
    \end{equation*}
 Therefore it is enough to prove the case when $s = 0$, $t \geq 0$, i.e.
    \begin{equation}\label{eq: equation of isometry}
    	d(p, \Phi_t(q)) = \sqrt{d(p, q)^2 + t^2}.
    \end{equation}
    To verify equation \ref{eq: equation of isometry}, we apply Petrunin's estimate to both $b_-$ and $b_+$. Since we assume $t \geq 0$, we have (corresponding to $b_-$ and $b_+$)
    \begin{equation*}
    	\Phi_t^-(x) = \Phi_t(x), \quad \Phi_t^+(x) = \Phi_{-t}(x) 
    \end{equation*}
    Thus, by applying Petrunin's estimate twice to both $\Phi_t(x)$ and $\Phi_{-t}(x)$, we have
    \begin{align*}
    	d(p, q)^2 
    	&= d(p, \Phi_{-t}(\Phi_t(q)))^2 = d(p, \Phi_t^+(\Phi_t^-(q)))^2\\
    	&\leq d(p, \Phi_t^-(q))^2 + 2(b_+(p) - b_{+}(\Phi_t^-(q)))(0 - t) + (0 - t)^2\\
    	&= d(p, \Phi_t^-(q))^2 - t^2\\
    	&\leq d(p, q)^2 + 2(b^-(p) - b^-(q))(0 - t)^2 + (0 - t)^2-t^2\\
    	&= d(p, q)^2
	\end{align*}
	This is because $\Phi^\pm_t$ are the gradient flows of $b_\pm$ and $b_+ = -b_-$. And remember that $b_\pm$ are zero on $N \ni p, q$. 
	
	Thus, all the inequalities above are actually equalities. 
    In particular, we have
    \begin{equation*}
    	d(p, q)^2 = d(p, \Phi_t^-(q))^2 - t^2 = d(p, \Phi_t(q))^2 - t^2
    \end{equation*}
    By rearranging the terms and the square root, we prove that the equality \ref{eq: equation of isometry} is true. 
    \begin{remark}
        The equation \ref{eq: flow estimate} is an equality for the Busemann function on $\R$ or more generally if we consider $\R \times Y$ and $\gamma(t) = (y_0, -t)$ the geodesic passing through $y_0 \in Y$. Then for $f = b_\gamma(y, t) = t$, the equation \ref{eq: flow estimate} admits the equality. Here we check that. Denote $p = (y_1, t_1)$ and $q = (y_2, t_2)$, then 
        \begin{equation*}
            d(p, q)^2 = (t_1 - t_2)^2 + \abs{y_1 - y_2}^2.
        \end{equation*}
        So $\Phi_s(p) = (y_1, t_1 + s)$ and $\Phi_t(q) = (y_2, t_2 + t)$. Then we can write
        \begin{align*}
           d(\Phi_s(p), \Phi_t(q))^2 
            =& (t_1 + s - (t_2 + t))^2 + \abs{y_1 - y_2}^2\\
            =& ((t_1 - t_2) + (s - t))^2 + \abs{y_1 - y_2}^2\\
            =& (t_1 - t_2)^2 + (s - t)^2 + 2(t_1 - t_2)(s - t) + \abs{y_1 - y_2}^2\\
            =& d(p, q)^2 + 2(b_\gamma(y_1, t_1) - b_\gamma(y_2, t_2))(s - t) + (s - t)^2\\
            =& d(p, q)^2 + 2(f(p) - f(q))(s - t) + (s - t)^2
        \end{align*}
        That means equation \ref{eq: flow estimate} can be written as
        \begin{align*}
            d(\Phi_s(p), \Phi_t(q)) \leq d(\overline{\Phi}_s(\overline{p}), \overline{\Phi}_t(\overline{q}))
       	\end{align*}
        where $\overline{p}, \overline{q} \in \R^n = \R^{n - 1} \times \R$ and $d(\overline{p}, \overline{q}) = d(p, q)$.
        So that $\overline{\Phi}_t$ is the gradient flow of $b_-$ on $\R^n$.         
    \end{remark}
          
\subsection{The Proof of Petrunin's Estimates}
Let $s \leq t$, otherwise, just flip the roles of $p$ and $q$. It is enough to prove the inequality \ref{eq: flow estimate} for $s = 0$. Since $\Phi_t = \Phi_s \circ \Phi_{t - s}$, 
\begin{align*}
    d(\Phi_t(q), \Phi_s(p)) 
    =& d((\Phi_s \circ \Phi_{t - s}))(q), \Phi_s(p))\\
    \leq& d(\Phi_{t-s}(q), p)\quad \text{Since $f$ is concave so that $\Phi_s$ is $1$-Lipschitz }
\end{align*}
Now, it is enough to prove the inequality \ref{eq: flow estimate} for $\abs{\Phi_{t-s}(q)p}$ from above. Equivalently it is enough to prove the inequality \ref{eq: flow estimate} for $\abs{\Phi_t(q)p}$ for $t \geq 0$. Namely, we want to show
\begin{equation}\label{eq: simplified flow estimate}
    d(\Phi_t(q), p)^2 \leq d(p, q)^2 + 2t(f(q) - f(p)) + t^2
\end{equation}
Notice that $\Phi_t(q)$ is the gradient flow passing through $q$, denote $l(t) = \abs{\Phi_t(q)p}$. Then the left hand side of the inequality can be denoted as $l^2(t)$. Consider
\begin{equation}\label{eq: l^2}
    (l^2(t))_+' \leq 2l(t)l_+'(t)
\end{equation}
Let $\sigma$ be the geodesic from $p$ to $q$, then $-\sigma'(l) = \uparrow_{\Phi_t(q)}^p$. Then we have the first variantion formula, which implies 
\begin{equation}\label{eq: l'(t)}
    l_+'(t) \leq -\inp{\uparrow_{\Phi_t(q)}^p, \nabla f_{\Phi_t(q)}}
\end{equation}
Recall that for the geodesic $\sigma$ between $x$ and $y$ and $f$ is concave, we proved before that along $\sigma$, we have
\begin{equation*}
    \inp{\nabla f_y, \uparrow_y^x} \geq \frac{f(x) - f(y)}{l}
\end{equation*}
Apply this to the inequality \ref{eq: l'(t)}, we have 
\begin{equation*}
    l_+'(t) \leq -\brac{\frac{f(p) - f(\Phi_t(q))}{l(t)}} = \frac{f(\Phi_t(q)) - f(p)}{l(t)}
\end{equation*}
So apply this to the inequality \ref{eq: l^2}, we have
\begin{equation*}
    (l^2(t))_+' \leq 2l(t)l_+'(t) \leq 2l(t)\frac{f(\Phi_t(q)) - f(p)}{l(t)} = 2(f(\Phi_t(q)) - f(p))
\end{equation*}
Remember that $f$ is $1$-Lipshchitz, we have $f(\Phi_t(q)) \leq f(q) + t$ because
\begin{equation*}
    \abs{\frac{\partial}{\partial t}f(\Phi_t(q))} = \abs{\nabla_{\Phi_t(q)} f}^2 \leq 1
\end{equation*}
Therefore, 
\begin{align*}
    &(l^2(t))_+' \le 2(f(q) - f(p) + t) = 2(f(q) - f(p)) + 2t\\
    \implies & l^2(t) \leq l^2(0) + 2(f(q) - f(p))t + t^2.
\end{align*}
Then we finish the proof since $l^2(0) = d(p, q)^2$


\chapter{Cheeger-Gromoll Covering Theorem}
We are going to study the space of the isometries of the non-negatively curved manifolds. 
\section{Preliminaries}

\begin{definition}
	Let $C$ be a subset of $(M, g)$ and $G$ a group acting on $C$, then we denote $G\cdot C = \br{gx: g \in G, x \in C}$ the family of the orbits of $C$. 
\end{definition}

\begin{definition}[$G$-invariant]
	Let $C$ be a subset of a Riemannian manifold, and $G$ be a group. Then $C$ is called \textbf{$G$-invariant} if for any  $x \in C$ and any $g \in G$ it holds that  $gx \in C$.
	
	\end{definition}

\begin{lemma}
	It is easy to check that $G\cdot C$ is $G$-invariant. 
	\begin{proof}
		For each $g_1 x \in G \cdot C$, $g_2(g_1 x) = (g_2g_1)x \in G\cdot C$. 
	\end{proof}
\end{lemma}

\begin{definition}[Totally Convex Hull]
	Let $C$ be a subset of a Riemannian manifold $(M, g)$, we denote $\hat{C}$ the \textbf{closed totally convex hull} of $C$ if it is the smallest closed totally convex subset of $M$ containing $C$. 
\end{definition}

\section{Introduction to the Covering Theorem}

\begin{theorem}
	Let $(M, g)$ be a complete Riemannian manifold of $\sect_M \geq 0$. If
	\begin{equation*}
		M \stackrel{isom}{\cong} N \times \R^k
	\end{equation*}
	and $N$ has no lines, then the isometry group $
	\isom(N)$ is compact. Moreover, 
	\begin{equation*}
		\isom(M) \stackrel{isom}{\cong} \isom(N)\times \isom(\R^k), 
	\end{equation*}
	i.e. any isometry $f: M \to M$ looks like $(h_1, h_2)$ where $h_1$ is an isometry of $N$ and $h_2$ an isometry of $\R^k$.
\end{theorem}
There are several statements that we need to show here. The first one is to show $\isom(N)$ is compact. Then we will use this theorem to prove the following theorem: 
\begin{theorem}
	If $(M, g)$ is a compact Riemannian manifold of $\sect_M \geq 0$, then its universal cover $(\tilde{M}, \tilde{g})$ is isomorphic to $\overline{M} \times \R^k$ where $\overline{M}$ is also of $\sect_{\overline{M}} \geq 0$. And a finite cover of $M$ is diffeomorphic to $\overline{M}\times T^k$. In fact, $\pi_1(M)$ up to finite index is just $\Z^k$
\end{theorem}
The first theorem we want to prove is 
\begin{theorem}
	If $(M, g)$ has no line and satisfies $\sect_M \geq 0$, then $\isom(M)$ is compact.
\end{theorem}
For proof, we need the following proposition

\begin{proposition}\label{prop: C-hat in neighborhood}
	Let $(M^n, g)$ be a complete open manifold with $\sect_M \geq 0$, $C \subseteq M$ be a compact subset. Let $G = \isom(M)$ and let $\hat{C}$ be the closed totally convex hull of  $G \cdot C$ i.e. the smallest closed totally convex set containing $G \cdot C$  where $G$ is $\isom(M)$. Then there exists $d > 0$ finite such that $\hat{C} \subseteq U_d(G\cdot C)$. 
\end{proposition}
\begin{proof}[Proof of Proposition~\ref{prop: C-hat in neighborhood}]

	The first thing we need to claim is that $\hat{C}$ is also $G$-invariant.  Let $g\in G$. Since $\hat{C}$ is the closed totally convex hull of $G\cdot C$ and $g(G\cdot C)=G\cdot C$ it follows that  $g(\hat{C})$ is also a closed totally convex subset containing $G\cdot C$. Since $\hat{C}$ is the smallest of such sets, then $\hat{C} \subseteq g(\hat{C})$. Applying $g^{-1}$, we have
	\begin{equation*}
		\hat{C} \subseteq g^{-1}(\hat{C}) \subseteq g^{-1}(g(\hat{C})) \subseteq \hat{C}. 
	\end{equation*}
	Therefore, we actually have
	\begin{equation*}
		\hat{C} = g(\hat{C})
	\end{equation*}
	And this means $\hat{C}$ is $G$-invariant. 

	We next claim that for any $y \in M$ and $g \in G$, 
	\begin{equation*}
		d(y, G\cdot C) = d(gy, G\cdot C).
	\end{equation*}
	Let $x \in G\cdot C$ be such that 
	\begin{equation*}
		d(y, x) = d(y, G\cdot C)
	\end{equation*}
	Such $g$ exists since. $G\cdot C$ is closed. Since $g$ is an isometry, we know that
	\begin{equation*}
		d(y, x) = d(gy, gx). 
	\end{equation*}
	Then,
	\begin{equation*}
		d(gy, G\cdot C) \leq d(gy, gx) = d(y, x) = d(y, G\cdot C).
	\end{equation*}
	To show the opposite inequality, we can apply $g^{-1}$, so that
	\begin{equation*}
		d(y, G\cdot C) = d(g^{-1}(g y), C\cdot C) \leq d(gy, G\cdot C)
	\end{equation*}
	Therefore, all the inequalities are actually equal. Hence our second claim is also true. 
	
	Suppose it is not true that there exists a finite $d$ such that $\hat{C} \subset U_d(G\cdot C)$. That means we can construct a sequence of $q_n \in \hat{C}$ such that 
	\begin{equation*}
		d(q_n, G\cdot C) \geq n
	\end{equation*}
	By our claim above, we know that 
	\begin{equation*}
		d(g_n^{-1}q_n, G\cdot C) \geq n
	\end{equation*}
	for each $n$ as well. For each $n$, denote $x_n \in G\cdot C$ the closest point to $g_n^{-1}q_n$. 

	Then by the compactness of $C$, by passing $x_n$ up to the sub-sequence, $x_n\to x \in C$ and segments $[x_n, g_n^{-1}q_n]\subseteq \hat C$ produce a ray $\sigma$ starting at $x$ such that
	\begin{equation*}
		d(\sigma(n), G\cdot C) = d( \sigma(n), x) = n
	\end{equation*}
		Note that the whole segment $[x_n, g_n^{-1}q_n]$ is contained in $\hat C$ and since $\hat C$ is closed this implies that the entire ray $\sigma$ is contained in $\hat C$ as well.

	Take $b_\sigma$ the corresponding Busemann function. We can check that 
	\begin{equation*}
		G\cdot C \subseteq \br{b_\sigma \geq 0}
	\end{equation*}
	this is because for each $y \in G\cdot C$,
	\begin{equation*}
		b_{\sigma}(y) = \lim_{n \to \infty}d(\sigma(n), y) - n \geq d(\sigma(n), x)-n = 0. 
	\end{equation*}
	Hence, take some $n \in \N$, the set $\br{b_\sigma \geq -n}$ is a closed totally convex subset containing $G \cdot C$. This means that the set $\br{b_\sigma \geq -n} \supseteq \hat{C}$. Hence, considering $\sigma(n + 1)$. On one hand, we have 
	\begin{equation*}
		b_{\sigma}(\sigma(n + 1)) = -n - 1 \leq \implies \sigma(n + 1) \notin \br{b_{\sigma} \geq -n}.
	\end{equation*}
	However, we know the entire ray $\sigma$ is contained in $\hat{C}$, thus in particular $\sigma(n + 1) \in \hat{C} \subseteq \br{b_\sigma \geq -n}$. Contradiction!
\end{proof}
\begin{remark}
	Without the assumption of $\sect_M \geq 0$, the totally convex hull $\hat{C}$ need not lie in $U_d(G\cdot C)$
\end{remark}

\begin{corollary}\label{cor: contains a line}
    If $G\cdot C$ is non-compact, then $M$ contains a line.
\end{corollary}
\begin{proof}[Proof of Corollary~\ref{cor: contains a line}]
Suppose that $G \cdot C$ is non-compact, then the closed totally convex hull $\hat{C}$ of $G \cdot C$ must be non-compact as well. By Hopf-Rinow theorem $\hat{C}$ must be unbounded since, in complete manifolds, closed bounded sets are compact.

Since $\hat{C}$ is also convex, this means it must contain a ray, i.e. for each point $p \in \hat{C}$, there must be a ray, say $\gamma$ starting from $p$, such that $\gamma \subseteq \hat{C}$. Indeed, since $\hat C$ is unbounded  we have a sequence of points $p_i\in \hat{C}$ such that $d(p, p_i) \to \infty$ and the shortest geodesic segments $[pp_i] \subseteq \hat{C}$ going to sub-converge to a ray $\gamma:[0, \infty) \to \hat{C}$ (since the directional vector start from $p$ sub-converge in the unit sphere). 

\begin{figure}[htbp]
    \centering
    \includegraphics[width=0.6\textwidth]{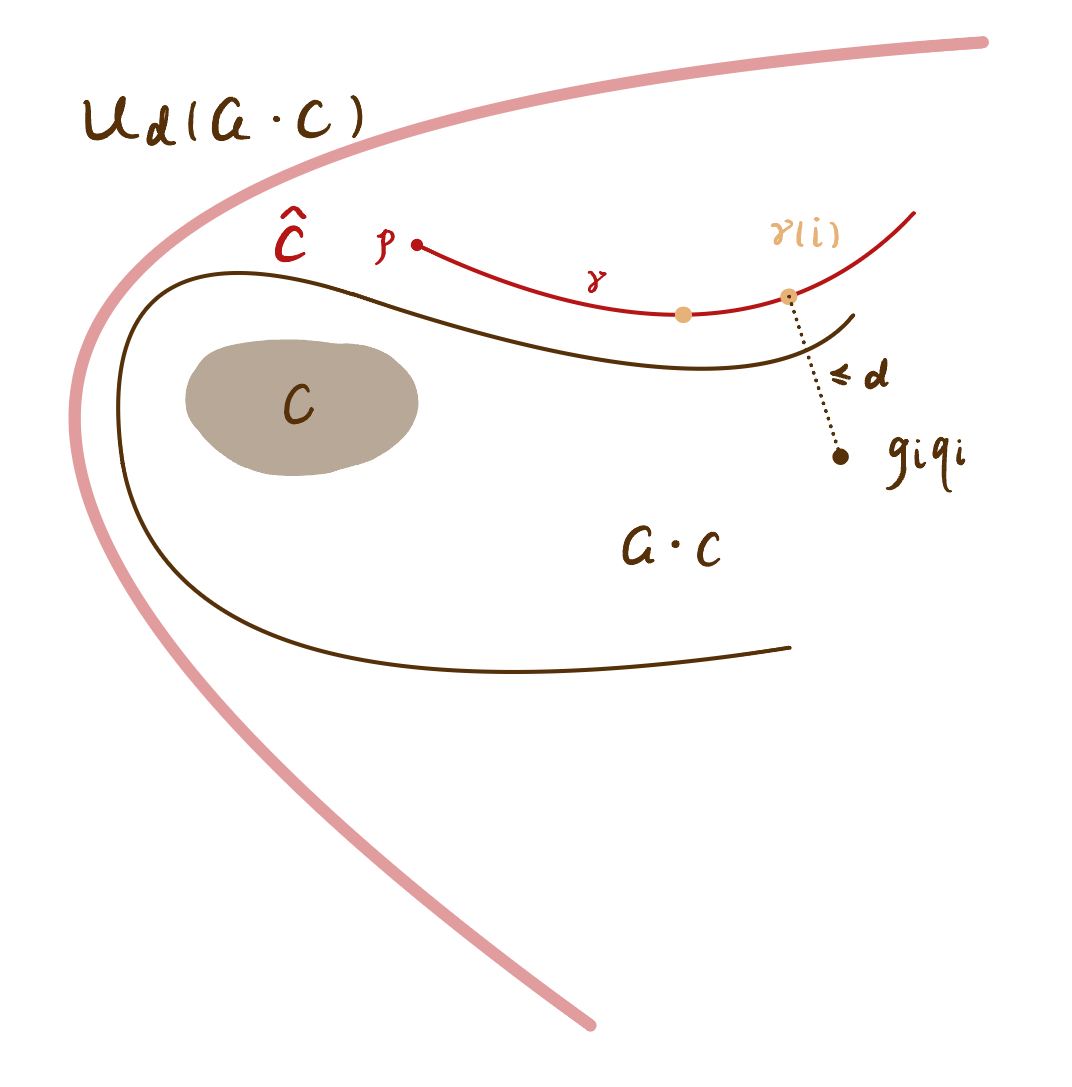}
    \end{figure}

Recall that  $\hat{C} \subseteq U_d(G\cdot C)$. Hence, for each natural number $i \in \N$, $\gamma(i) \in U_d(G \cdot C)$. This means there exists $g_i \in G, q_i \in C$ such that $d(g_iq_i, \gamma(i)) \leq d$.

    Notice that
    \begin{equation*}
        d(g_iq_i, \gamma(i)) = d(q_i, g_i^{-1}(\gamma(i))) \leq d
    \end{equation*}
    Taking $\gamma_i := g_i^{-1}\circ\gamma$,  and we can shift $\gamma_i$ to $\Tilde{\gamma}_i: [-i, \infty) \to M$ such that $\Tilde{\gamma}_i(t) = g_i^{-1}(\gamma(t + i))$. Notice that now
    \begin{equation*}
    	d \geq d(q_i, g_i^{-1}(\gamma(i))) = d(q_i, \gamma_i(i)) = d(q_i, \tilde{\gamma}_i(0))
    \end{equation*}
    
    \begin{figure}[htbp]
    \centering
    \includegraphics[width=0.8\textwidth]{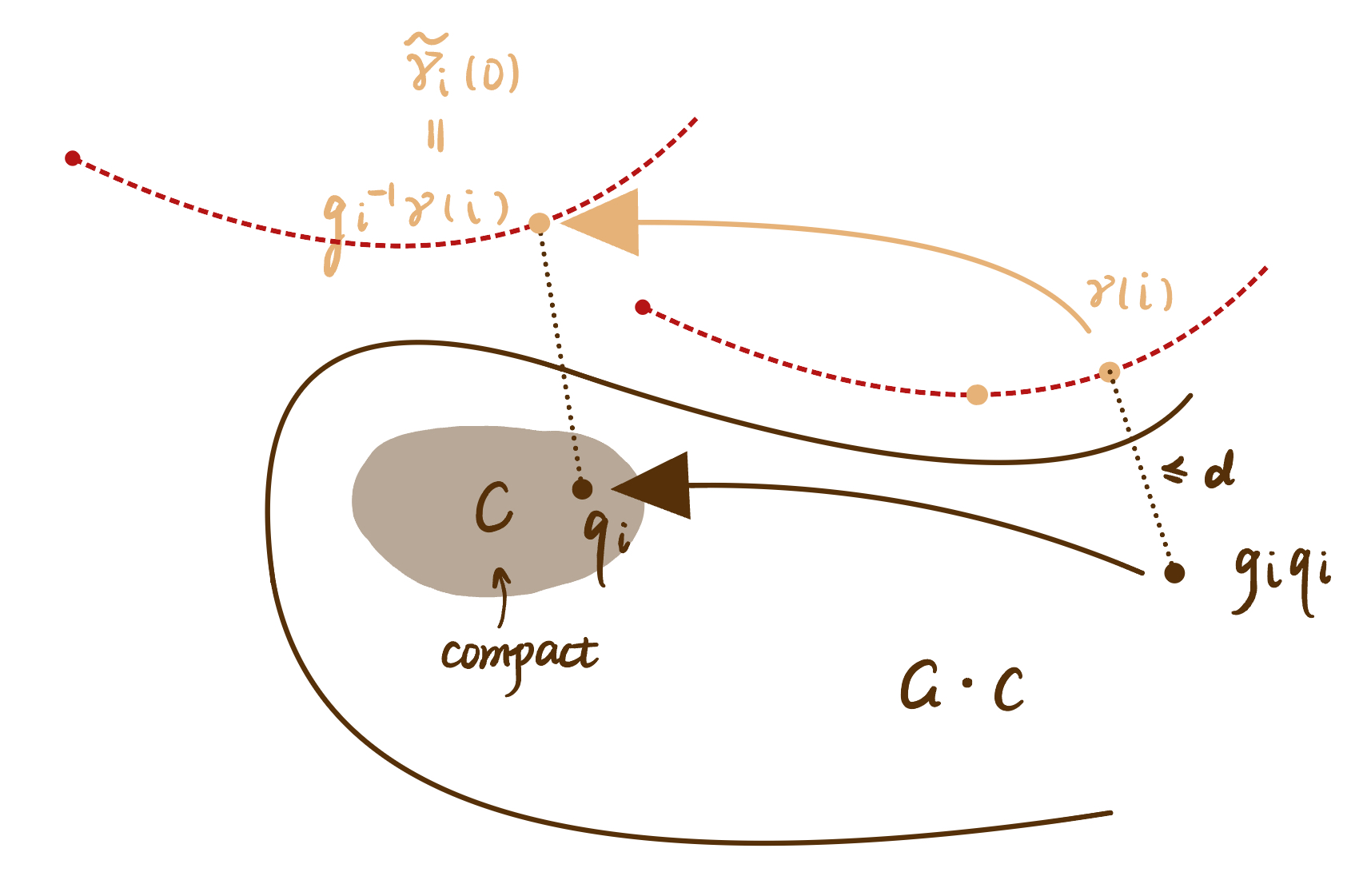}
    \end{figure}
    
    This means $\Tilde{\gamma}_i(0) \in \underbrace{U_d(C)}_{\text{compact.}}$ for any $i$.
   
    Since $C$ is compact, by passing to a subsequence we can make $\Tilde{\gamma}_i$ sub-converge to a line $\Tilde{\gamma}:(-\infty, \infty) \to M$. 
    \end{proof}
\begin{corollary}\label{cor: isom compact}
   If $M$ is complete with  $\sect_M \geq 0$ and $M$ has no lines, then $\isom(M)$ is compact.
\end{corollary}
\begin{proof}[Proof of Corollary~\ref{cor: isom compact}]
    Let $C = \br{p}$ be a point in $M$, which is certainly compact. Suppose $G = \isom(M)$ is non-compact, then $G\cdot \br{p}$ is non-compact. Otherwise, the whole group is compact because $G\cdot \br{p} \cong G/G_p$ and the isotropy group of $p$  $G_p = \br{g \in G: g(p) = p} $  is compact. Indeed, any  $g\in G$ is determined by $g(p)$ and $dg_p$. In particular for $g\in G_p$ the map $g$ is determined by  $dg_p \in O(n)$ (the orthogonal group in the tangent space.). And since $G_p$ is a closed subgroup of $O(n)$ which is compact it follows that $G_p$ is also compact. So $\underbrace{G}_{\text{non-compact}}/\underbrace{G_p}_{\text{compact}}$ is non-compact. Hence, if $G\cdot\br{p}$ is non-compact, by the previous corollary \ref{cor: contains a line} with $C = \br{p}$, $M$ has a line. So contradiction.
\end{proof}
So far, we can conclude that when $M$ has $\sect_M \geq 0$ then
\begin{equation*}
    \text{$M$ has no line} \iff \text{$\isom(M)$ is compact.}
\end{equation*}
And we know that
\begin{align*}
    &\text{If $M$ has a line}\\
    \implies & \text{$M \cong N \times \R$ (splits)}\\
    \implies & \text{This implies that translations along $\R$ factor are isometries.}\\
    \implies & \text{$\isom(M)$ is non-compact}
\end{align*}
Now, let's summarize everything in a theorem
\begin{theorem}\label{thm: theorem on complete rm}
    Let $(M^n, g)$ be a complete Riemannian manifold with $\sect_M \geq 0$. Then 
    \begin{enumerate}
        \item $M \stackrel{\text{isom}}{\cong} N \times \R^k$ where $N$ has $\sect_N \geq 0$ has no lines and $\isom(N)$ is compact. This splitting is unique in a strong sense. Namely, if $f: N \times \R^k \to M$ is an isometry, for $p \in M$, $p = f(q)$, then $q = (x, v)$ where $x \in N$, $v \in \R^k$ the factors $f(\br{x}\times \R^k)$ and $f(N \times \br{v})$ are uniquely determined and do not depends on the choice of $f$.
        \item $\isom(M) \cong \isom(N)\times \isom(\R^k)$. Namely, if $f \in \isom(N)$ and $h \in \isom(\R^k)$ then $(f, h): N\times \R^k \to N \times \R^k$ is an isometry. This result means all the isometries of $M$ are like this. 
    \end{enumerate}
\end{theorem}

Let's quickly explain \textit{1.} and then prove \textit{2.}.
The only thing that is left to be proved is the uniqueness of the splitting factors.

Let $\gamma:\R \to \subseteq N \times \R^k$ be a line. Then $\gamma=(\gamma_1, \gamma_2) $ where $\gamma_1$ and $\gamma_2$ are geodesics in $N$ and $\R^k$ respectively. Suppose $\gamma_1(t)$ is not constant. Then $\gamma_1$ cannot be globally shortest. Eventually, this means $\gamma$ is not minimizing, i.e. there are $t_1, t_2 \in \R$ such that the segment $[\gamma_1(t_1)\gamma_1(t_2)]$ can be replaced by a shorter curve, which implies $\gamma$ is also not shortest. Contradiction! This means that  $\gamma(t) = (x, \gamma_2(t))$ for some $x \in N$ and $\gamma_2$ is a line in $\R^k$. Now, since $f$ is an isometry it sends lines to lines. Hence  the above discussion tells us that $f(\gamma(t)) = (y, \tilde{\gamma_2}(t)) \subseteq \br{\text{pt}}\times \R^k$. So $f(\br{\text{pt}} \times \R^k) = \br{\text{pt}} \times \R^k$. Since isometries preserve angles between vectors this implies that geodesics orthogonal to the $\R^k$ factors must map to geodesics orthogonal to $\R^k$ factors.  Hence  $f(N \times \br{\text{pt}}) = N \times \br{\text{pt}}$. 
This finishes the proof of  \textit{1.}

\begin{proof}[Proof of 2. in theorem \ref{thm: theorem on complete rm}]

Since $M$ is isometric to  $N \times \R^k$ to understand its isometry group we can just assume that $M = N \times \R^k$.
    It is sufficient to consider $\Phi: \isom(N)\times \isom(\R^k) \to \isom(N \times \R^k)$. Remember how this map is constructed, for $h_1 \in \isom(N)$ and $h_2 \in \isom(\R^k)$, we have
    \begin{equation*}
    	\Phi(h_1, h_2) := h_1 \times h_2 = h \in \isom(N \times \R^k)
    \end{equation*}
    
    This is an obvious homomorphism. We want to show that $\Phi$ is moreover a bijection.
    \begin{itemize}
        \item \textbf{$\Phi$ is one-to-one:} This is obvious.
        \item \textbf{$\Phi$ is onto:}Let $f \in \isom(N\times \R^k)$, $f: N \times \R^k \to N \times \R^k$ maps lines to lines, $\R^k$ factors to $\R^k$ factors and $N$ factors to $N$ factors isometrically. Fix $x \in N$, $v \in \R^k$, then 
        \begin{equation*}
            \br{x} \times \R^k \stackrel{f}{\longrightarrow} N \times \R^k
        \end{equation*}
        is an isometrically embedding onto $\br{y} \times \R^k$. Also, 
        \begin{equation*}
            N \times \br{v} \stackrel{f}{\longrightarrow} N \times \R^k
        \end{equation*}
        is an isometrically embedding on to some $N \times \br{w}$. The linear map $df_{(x, v)}$ looks like $(A, B)$ where $A$ and $B$ are both isometries $A: T_xN \to T_yN$ and $B: T_v\R^k \to T_w\R^k$. Hence, we can have $h_1: N \to N$ isometry such that $d{h_1}_x = A$ and $h_2: \R^k \to \R^k$ isometry such that $d{h_2}_v = B$. If we take the isometry $h = (h_1,h_2): N \times \R^k \to N\times \R^k$ and $dh_{(x, v)} = df_{(x, v)}$. Then $h \equiv f$ since two isometries that have the same values and the same differentials at a point must be equal.

        Therefore $f=\Phi(h_1,h_2)$ which proves that $\Phi $ is onto.
    \end{itemize}
\end{proof}
\begin{theorem}[Cheeger-Gromoll Covering Theorem]\label{thm: split to torus}
    Let $(M^n, g)$ be compact and $\sect_M \geq 0$, then the universal cover $\Tilde{M} \stackrel{\text{isom}}{\cong} \overline{M} \times \R^k$ where $\bar M$ is compact and simply connected. And a finite cover of $M$ is diffeomorphic  to $\underbrace{\overline{M}}_{\text{simply connected and compact}} \times \underbrace{T^k}_{\text{torus}}$
\end{theorem}
\begin{remark}
    \textbf{Warning:} It is not true that the finite cover must be isometric to $\overline{M} \times T^k$. It is in general only diffeomorphic.
\end{remark}
\begin{example}
    Take $S^2 \times \R$. And $\Gamma$ is $\Z$-actions, i.e. $\Z$ acts on $\R$ by translations in standard way: $\overline{1}(x) = x + 1$. 
    Moreover, we take the action on $S^2$ to be irrational rotation $R_\alpha$ around the  $z$-axis, i.e. take $\frac{\alpha}{2\pi}$ irrational, so that the order of the rotation $\abs{R_\alpha} = \infty$.
    
    \begin{figure}[htbp]
    \centering
    \includegraphics[width=0.3\textwidth]{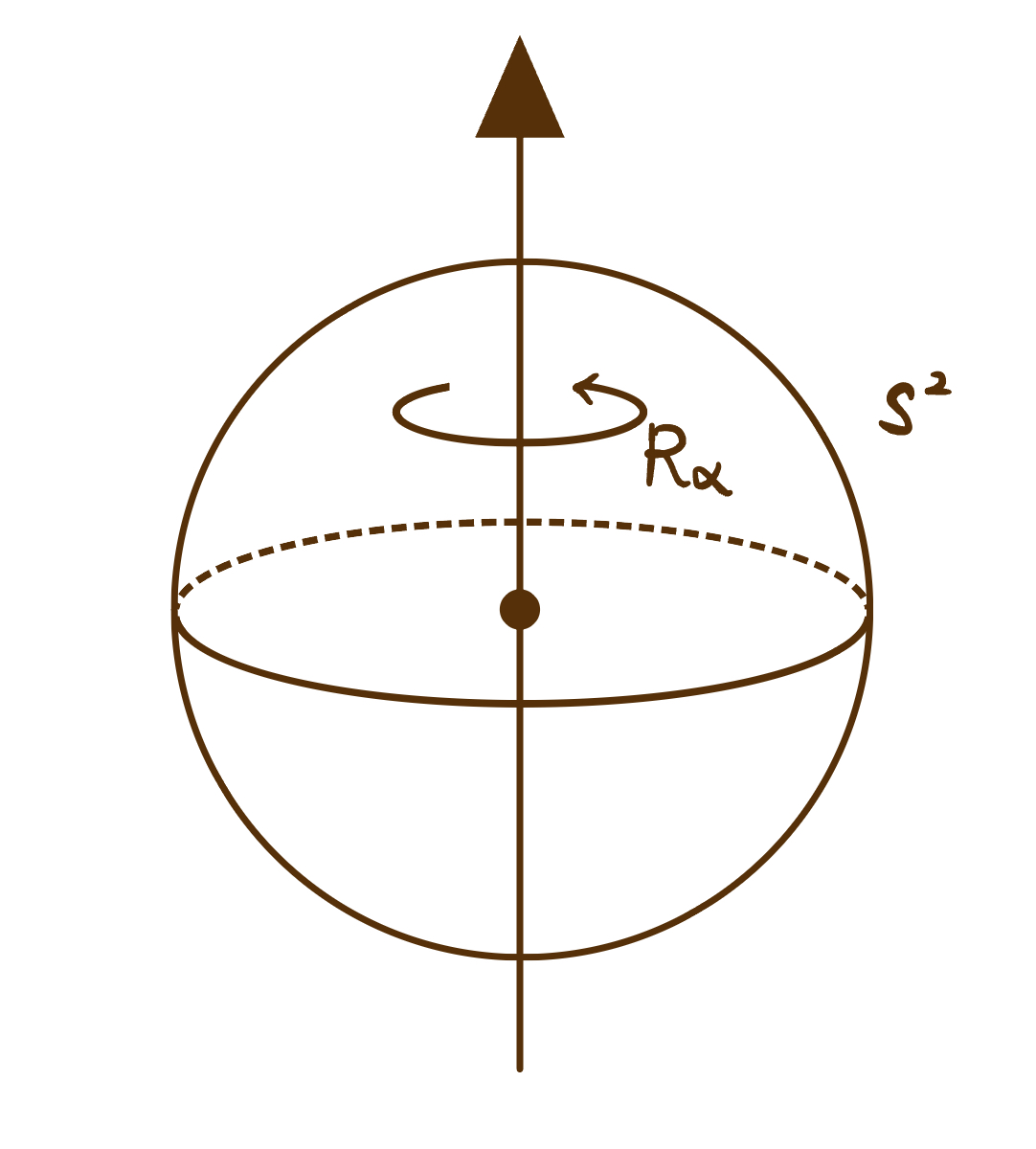}
    \end{figure}
    
    Take the diagonal action of $\Gamma$ on $S^2 \times \R$, i.e. $g \in G$, $g(\underbrace{p, x}_{\in S^2 \times \R}) = (gp, gx)$. 
    More explicitly, if $g = n$, then $n(p, x) = (R_{n\alpha}(p), x + n)$. Let $M =(S^2\times \R)/\Z$. 
    Then $M$ is compact, and the universal cover is just $S^2 \times \R$ isometrically, $M \stackrel{\text{diffeo}}{\cong} S^2 \times S^1$ but not isometric. 
    Moreover, no finite conver is isometric to $S^2 \times S^1$ since $\frac{\alpha}{\pi}$ is irrational and the rotation $R_{n\alpha}$ can never be equal to identity for $n \neq 0$. 
\end{example}
\section{The Proof of the Covering Theorem}
   let $\Tilde{M}$ be the universal cover of $M$.  Then $\Tilde{M} \cong \overline{M}\times \R^k$ where $\overline{M}$ has no lines and compact isometry group. 
    \begin{claim}
        $\overline{M}$ is compact. 
    \end{claim}
 Let $\Gamma = \pi_1(M) \acts \tilde M$ by deck transformation. $\Tilde{M}/\Gamma = M$ but $\isom(\Tilde{M}) = \isom(\overline{M}) \times \isom(\R^k)$. Then $\Gamma$ acts diagonally 
    \begin{align*}
        &\rho: \Gamma \to \isom(\overline{M});\\
        &\phi: \Gamma \to \isom(\R^k)
    \end{align*}
    For $g \in \Gamma$, then 
    \begin{equation*}
        g(x, v) = (g(x), g(v)) = (\rho(g)(x), \rho(g)(v))
    \end{equation*}
    \begin{remark}
      	$h: \R^k \to \R^k$ isometry, then $\exists A \in O(k)$, $u \in \R^k$ such that $h(v) = Av + u$.
    \end{remark}

    We have $\Gamma\acts \overline{M} \times \R^k$ diagonally.  We are given that $(\overline{M} \times \R^k)/\Gamma = M$ is compact.  For $\rho: \Gamma \to \isom(\overline{M})$ let $H = \overline{\rho(\Gamma)}$. It is a  closed subgroup of $\isom(\overline{M})$ which is compact and hence $H$ is compact too.  We have $(\bar M \times \R^k)/\Gamma \stackrel{\text{onto}}{\to} \bar M/\Gamma \to \bar M/H$. Since $M=\bar M/\Gamma$ is compact this implies that $\bar M/H$ is compact too. Since  $H$ is compact this implies that $\overline{M}$ is compact as well. So we can conclude that 
    \begin{equation*}
        \Tilde{M}\cong\underbrace{\overline{M}}_{\text{ simply connected \& no lines \& compact}}\times \R^k. 
    \end{equation*}
    Let $\hat{\Gamma} = \ker(\phi) \lhd \Gamma$, $\hat{\Gamma}$ acts trivially on $\R^k$ and $\hat \Gamma \subseteq \Gamma$ acts freely and properly discontinuously on $\Tilde{M}$.  This action must be free since $\hat{\Gamma}$ acts freely on $M$ and the action of $\hat{\Gamma}$ on $\R^k$ is trivial. But $\overline{M}$ is compact. $\overline{M}$ is closed manifold, $\hat{\Gamma}\acts \overline{M}$ freely \& properly discontinuously. Therefore, $\hat{\Gamma}$ is finite, otherwise, the action would not act freely and properly discontinuously. Indeed,  If $\hat{\Gamma}$ is not compact, $\hat{\Gamma}\cdot x \in \overline{M}$ has an accumulation point and the action is not properly discontinuous. Thus $\hat{\Gamma}\lhd \Gamma$ is a  finite subgroup. Let $\Tilde{M}_1 = \Tilde{M}/\hat{\Gamma} = \underbrace{\overline{M}/\Gamma}_{=: \overline{M}_1} \times \R^k$. Notice that $\overline{M}_1$ is compact but not simply connected unless $\hat \Gamma$ i trivial. Then $\pi_1(\overline{M}_1) \cong \hat{\Gamma}$ which is finite and $\Gamma_1 = \Gamma/\hat{\Gamma}$ acts diagonally in $\overline{M}_1\times \R^k$ and $\overline{M}_1\times \R^k/\Gamma_1=M$.
    \begin{claim}
        A finite cover of $M$ is diffeomorphic to $\overline{M_1}\times T^k$. 
    \end{claim}
    
    Note that $\overline{M_1}\times T^k$. has a finite cover diffeo to $\overline{M} \times T^k$. Hence   this claim implies our theorem.

    It is remain to prove this claim. We have $\Gamma_1$ action, and $\rho_1: \Gamma_1 \to \isom(\overline{M}_1)$ and $\phi_1: \Gamma_1 \to \isom(\R^k)$ but now this $\phi_1$ is injective, therefore, 
    \begin{equation*}
        \underbrace{\phi_1(\Gamma_1)}_{\approx \Gamma_1} < \isom(\R^k)
    \end{equation*}
    is a discrete subgroup. The quotient  $\R^k /\Gamma_1$ must be compact. Otherwise, as before, we get that $(\overline{M}_1 \times \R^k)/\Gamma_1$ is also non-compact which we know is false. Also, $\Gamma_1 \acts \R^k$ the discrete subgroup of $\isom(\R^k)$ and  $\R^k/\Gamma_1$ is compact. By Bieberboch Theorem, a finite index subgroup $\Z^k \cong \Gamma'_1 < \Gamma$ generated by translations by linearly independend $v_1, \dots, v_k$ vectors in $\R^k$. $(n_1, \dots, n_k) \in \Z^k$, $(x_1, \dots, x_k) \in \R^k$, then 
    \begin{equation*}
        (n_1, \dots, n_k)(x_1, \dots, x_k) = (x_1, \dots, x_k) + \sum_{i = 1}^kn_iv_i.
    \end{equation*}
    $\R^k/\Z^k = T^k$. By passing to a finite cover $M'$ of $M$, we have $\underbrace{\Gamma_1'}_{= \Z^k}$ acts on $\overline{M}_1\times \R^k$ diagonally action on $\R^k$ translations generated by $v_1, \dots, v_k$ basis of $\R^k$. Then $M' = (\overline{M}_1\times \R^k)/\Gamma_1'$. 

    Consider 
    \begin{align*}
        \rho: \Gamma_1' \to \isom(\overline{M}_1);\\
        \underbrace{\phi: \underbrace{\Gamma_1'}_{\Z^k} \to \isom(\R^k)}_{\text{We now know this map.}}
    \end{align*}
    Let $H = \overline{\rho(\Gamma_1')}$, so $\Gamma_1' = \Z^k$ is Abelian, then $H$ is Abelian Lie group,  $H$ is compact. Let $H_0$ be the identity component of $H$, which is compact, Abelian and connected. Therefore, $H_0 \cong T^l$ for some $l$. And $\abs{H: H_0} < \infty$ since $H$ is compact and hence has finitely many components. Take $\Gamma_1'' = \rho^{-1}(H_0)$, then $\Gamma_1'' < \Gamma_1'$ finite index subgroup. then the image of $\Gamma_1'' \subseteq H_0$ finite index subgroup of $\Z^k$ is isomorphic to $\Z^k$. 

    $\Gamma_1'' = \Z^k \acts \overline{M}_1\times \R^k$ still acts  cocompactly. Consider  $M_1=\overline{M}_1\times \R^k/ \Gamma_1'$. It is a the finite cover of $M$ and it is enough to understand 
    $M_1$. 
    \begin{claim}
        $M_1=(\overline{M_1}\times \R^k) /\Gamma_1'' \stackrel{\text{diffeo}}{\cong} \overline{M}_1\times T^k$
    \end{claim}
    to show this claim, consider
    \begin{equation*}
        \rho: \underbrace{\Gamma_1''}_{\Z^k} \to \underbrace{H_0}_{\text{$T^l$}} \underbrace{\subseteq}_{\text{dense subgroup}} \isom(\overline{M}_1)
    \end{equation*}
  Also   $\phi(\Gamma_1'')=\Z^k \subseteq \R^k = L$ is a cocompact lattice.

       Observe that  we can extend the homomorphism $\rho: \Z^k \to H_0$  to a homomorphism $\rho: \R^k \to H_0$

    To see why, it is enough to check the extension of one of the component of $\rho = (\rho_1, \dots, \rho_l)$. Namely, we only need to extend $\rho_i: \Z^k \to S^1$ to $\rho_i: \R^k \to S^1$. This is easy since everything here is Abelian. So if $\Z^k = \inp{v_1, \dots, v_k} \subseteq \R^k$, we have
    \begin{equation*}
        \rho_i(v_j) = e^{i\theta_{ij}} \in S^1, \quad\theta_{ij} \in \R
    \end{equation*}
    Set $\rho_i: \R^k \to S^1$ by the formula, $\rho_i(\sum_{i = 1}^kt_jv_j) := e^{i(\sum_{i = 1}^kt_j \theta_{ij})}$, $t_j \in \R$. It is easy to see this is a homomorphism. Now we get $\rho: \Gamma_1'' \to \isom(\overline{M}_1)$ extending to $\rho: \R^k \to \isom(\overline{M}_1)$ still homomorphism. Now we can produce a diffeomorphism $\overline{M}_1\times T^k \to( \overline{M}_1\times \R^k/)\Gamma_1''$ as follows.  

    First look at 
    \begin{equation*}
        f: \overline{M}_1 \times \R^k \to \overline{M}_1\times \R^k
    \end{equation*}
    given by 
    \begin{equation*}
        f(p, q)\stackrel{\text{def}}{:=} (\rho(a)\cdot p, a)
    \end{equation*}
    This is a diffeomorphims since we have the inverse map
    \begin{equation*}
        f^{-1}(p, a) = (\rho(a^{-1})\cdot p, a). 
    \end{equation*}
    \begin{claim}
        $f$ is $\Gamma_1''$ equivariant if we take the diagonal actions of $\Gamma_1''$ on the target $\overline{M}_1\times \R^k$ and on the domain, we just take the action on $\R^k$. $\overline{M}_1\times \R^k$

    $\Gamma_1'' = \Z^k \acts \overline{M}_1\times \R^k$ on the second factor only $g(p, a) = (p, ga)$.  In other words we claim that 
     \[f(g(p, a)) = g(f(p, q))\]
         \end{claim}

    i.e. $f$ is $\Gamma_1''$ equivariant. $g(p, a) = (p, ga)$, $f(g(p, a)) = f(p, ga) = ((ga)\cdot p, ga) = g(ap, a) = g(f(p, a))$. This means $f: \overline{M}_1 \times \R^k \to \overline{M}_1\times \R^k$ passes to a quotient by $\Gamma_1'' = \Z^k$. This we get $\overline{f}: \underbrace{\overline{M}_1\times (\R^k/\Gamma_1'')}_{= \overline{M}_1\times T^k} \to (\overline{M_1}\times \R^k)/\Gamma_1''$ diffeomorphism (diagonal action, finite cover of $M$). Therefore, $\overline{f}: \overline{M}_1 \times T^k \to M_1$ is a diffeomorphism where $M_1$ is the finite conver of $M$. 

\chapter{Bishop-Gromov Volume Comparison}
In this chapter, we are going to state and prove the classical Bishop-Gromov volume comparison theorem.
\begin{lemma}\label{lem: monotone average integral}
	Let $q: \R \to \R$ be a monotonous non-increasing function. Then the average integral
	\begin{equation*}
		f(t) = \dashint_{0}^t q(s)ds
	\end{equation*}
	must also be non-increasing. 
	
	\begin{figure}[htbp]
        \centering
       	\includegraphics[width=1.0\textwidth]{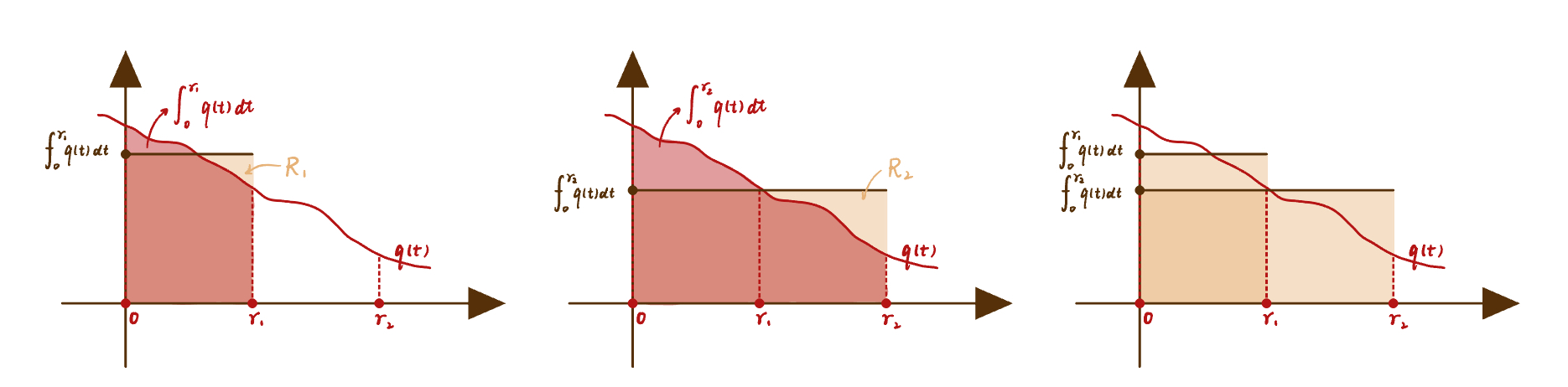}
       	\caption{Average integral of a non-increasing function is also non-increasing}
    \end{figure}
	\begin{proof}
		Indeed, for $T > t > 0$
    \begin{equation*}
    	\begin{split}
    		\dashint_0^Tq
    		&= \frac{1}{m([0, T])}\brac{\int_0^tq + \int_t^T q}\\
    		&= \underbrace{\frac{m([0, t])}{m([0, T])}}_{:= \lambda_1}\dashint_0^tq  + \underbrace{\frac{m([t, T])}{m([0, T])}}_{:= \lambda_2} \dashint_t^T q 
     	\end{split}
    \end{equation*}
    Notice that $\lambda_1 + \lambda_2 = 1$ and $\lambda_1, \lambda_2 > 0$. Also notice that since $q$ is non-increasing, then 
    \begin{equation*}
    \dashint_t^Tq \leq \dashint_0^t q 
    \end{equation*}
    Therefore, 
    \begin{equation*}
    	\begin{split}
    		\dashint_0^Tq 
    		&= \lambda_1\dashint_0^t q +\lambda_2\dashint_t^T q \\
    		&\leq \lambda_1\dashint_0^t q +\lambda_2\dashint_0^t q\\
    		&= \dashint_0^t q
     	\end{split}
    \end{equation*}
   	for $0 < t < T$. 
	\end{proof}
\end{lemma}
\begin{theorem}[Bishop-Gromov Relative Volume Comparison ]\label{thm:bishop-gromov}
Let $(M^n, g)$ be a complete manifold with $\ric_M \geq (n - 1)\kappa$. Recall that $\modsp{\kappa}$ denotes the complete simply connected $n$-dimensional space of $\sect \equiv \kappa$ (thus $\ric_{\modsp{\kappa}} \equiv (n - 1)\kappa$). Let $p \in M$ and $\overline{p} \in \modsp{\kappa}$. Then the function $f_p: \R^+ \to \R^+$, where
\begin{equation}\label{eq:bishop-gromov}
    f_p(r) = \frac{\vol_M(B_p(R))}{\vol_{\modsp{\kappa}}(B_{\overline{p}}(R))}
\end{equation}
is non-increasing with respect to $r$. 
\end{theorem}
Relative volume comparison has the following  immediate corollary
\begin{corollary}[Absolute Volume Comparison] \label{cor: bishop-gromov}
Under the assumptions of Theorem \ref{thm:bishop-gromov} for any $r \geq 0$ it holds that
\begin{equation*}
    \vol_M(B_p(r)) \leq \vol_{\modsp{\kappa}}(B_{\overline{p}}(r))
\end{equation*}
\end{corollary}
\begin{proof} [Proof of Corollary \ref{cor: bishop-gromov}] 
  	Because $M$ is asymptotically Euclidean for small $r$ ( since $d\exp_p|_0 = \id$) we have that 
    \begin{align*}
        & \lim_{r \to 0}f(p, r) = 1
    \end{align*}
    Now monotonicity of $f$ immediately implies the result.
\end{proof}
\begin{proof}[Proof of Theorem \ref{thm:bishop-gromov}]
    Denote by $\omega$ the volume form on $M$ and we want to compute the volume form in exponential coordinates. For each $p$, $T_pM\cong \R^n$, thus we can define polar coordinates on $\R^n\setminus\{0\}$. Namely, if $S^{n - 1} \subseteq \R^n$ and $r \in \R^+$, then 
    \begin{equation*}
        \psi: S^{n - 1}\times \R^+ \to \R^n; \quad (v, r) \mapsto r\cdot v
    \end{equation*}
    is the polar coordinates. Take $\phi = \exp \circ \psi$. Thus
    \begin{equation*}
        \phi: S^{n - 1}\times \R^+ \stackrel{\psi}{\to} \R^n \cong T_pM \stackrel{\exp}{\to} M
    \end{equation*}
    We can pull back the volume form by $\phi$. 
    \begin{remark}
        Let's recall the definition of the pullback: For $f: M^n_1 \to M^n_2$, $\omega \in \Omega^n(M)$, $\br{v_1, \dots, v_n} \subseteq T_pM^n_1$ is the basis. Thus we can define the pullback as
        \begin{equation*}
            (f^*\omega(p))(v_1, \dots, v_n) = \omega(f(p))(df_p(v_1), \dots, df_p(v_n))
        \end{equation*}
    \end{remark}
    Consider $\R^n$ which is parameterized by $S^{n -1}\times \R^+$ and consider $\overline{\gamma}(t) = t\cdot v$, where $v \in S^n$ is the unit vector at the origin. Then $v^\perp = T_vS^{n - 1}$. Let $v_1, \dots, v_{n-1}$ be the orthonormal basis of $T_vS^{n - 1}$. Since $S^{n-1} \times \R^+$ has the volume form $dS^{n - 1}\wedge dr$, i.e. $dS^{n - 1}\wedge dr(v_1, v_2, \dots, v_{n - 1}, \frac{\partial}{\partial r}) = 1$ for the volume form $\omega \in \Omega^n(M)$, there exists $j: S^{n - 1}\times \R^+ \to \R$ such that 
    \begin{equation*}
        \phi^*(\omega) = j(v, r)dS^{n - 1}\wedge dr
    \end{equation*}
    We want to understand  $j(v, r)$. By the definition of the pullback form
    \begin{align*}
        j(v, r) =&  j(v, r)dS^{n - 1}\wedge dr(v_1, \dots, v_{n - 1}, \frac{\partial}{\partial r})\\
        = & \phi^*(\omega(v_1, \dots, v_{n - 1}, \frac{\partial}{\partial r}))\\
        = & \omega(d\phi(v_1), \dots, d\phi(v_{n - 1}), d\phi(\frac{\partial}{\partial r}))
    \end{align*}
    By the definition of $\psi$, we know that 
    \begin{equation*}
        d\psi: (v_1, \dots, v_{n - 1}, \frac{\partial}{\partial v}) \mapsto (rv_1, \dots, rv_{n - 1}, v)
    \end{equation*}
    On $\R^n \cong T_pM$, $\overline{\gamma}_v(t) = tv = \psi(v, t)$. For vectors $tv_1, \dots, tv_{n - 1}, v$, we apply $d\exp$ to them. Look at the  Jacobi fields in the flat metric on $\R^n$ given by $\overline{J}_i(t) = tv_i$. $d\psi_{(v, t)}(v_i) = \overline{J}_i(t)$, which are normal to $\overline{\gamma}_v(t)$. Then
    \begin{equation*}
        d\exp_p|_{tv}(\overline{J}_i(t)) = J_i(t)
    \end{equation*}
    is the Jacobi field along $\gamma_v(t) = \exp(tv)$ satisfying the initial conditions $\bar J_i(0)=0,  \bar J_i'(0)=v_i, i=1,\ldots, n-1$.
    \begin{remark}
        Recall that Jacobi fields that vanish at $0$ can be  obtained by taking the variation of the exponential map. Thus we denote $\overline{J}_i(t) := \frac{\partial}{\partial s} = d\exp_{\overline{\gamma}_v(t)}(J_i(t))$. And also $d\exp_{\overline{\gamma}_v(t)}(v) = \Dot{\gamma}_v(t)$. Notice that $J_i(t)\perp \Dot{\gamma}_v(t)$ for all $t$. Recall that if $J$ is Jacobi along $\gamma(t)$. Then 
        \begin{equation*}
            \inp{J, \Dot{\gamma}} = at + b, \quad\text{where $b = \inp{J(0), \Dot{\gamma}(0)}$, $a = \inp{\Dot{J}(0), \Dot{\gamma}(0)}$}.
        \end{equation*}
        In our case, $J_i(0) = 0 \implies b = 0$, $\Dot{J}_i(0) = v_i\perp v \implies a = 0 \implies \inp{J_i(t), \Dot{\gamma}_v(t)} \equiv 0$. 
    \end{remark}
    So we have $\inp{J_i(t), \Dot{\gamma}_v(t)}\equiv 0$. Now by the definition of $\phi$, we know that 
    \begin{equation*}
        d(\exp\circ\psi)(v_1, \dots, v_n, \frac{\partial}{\partial t}) = (J_1(t), \dots, J_{n - 1}(t), \Dot{\gamma}_v(t))
    \end{equation*}
    Therefore, 
    \begin{equation*}
        j(v, t) = \omega(d\phi(v_1), \dots, d\phi(v_{n - 1}), d\phi(\frac{\partial}{\partial t})) = \omega(J_1(t), \dots, J_{n - 1}(t), \Dot{\gamma}_v(t))
    \end{equation*}
    which is just the volume of the cuboid spanned by $\br{J_1(t), \dots, J_{n - 1}(t), \Dot{\gamma}_v(t)} \subseteq T_pM$.
    \begin{remark}
        If $\omega_1, \dots, \omega_n$ in $\R^n$, $E = \br{e_1, \dots, e_n}$ are orthonormal basis of $\R^n$, we can write $[\omega_1]_E, \dots, [\omega_n]_E$ the column vectors in this basis, then 
        \begin{equation*}
            d\vol_{\R^n}(\omega_1, \dots, \omega_n) = \det{[\omega_1]_E| \dots| [\omega_n]_E}
        \end{equation*}
    \end{remark}
    Therefore, let $v_1, \dots, v_n$ be parallel vector fields along $\gamma_v(t)$, $v_i(0) = v_i$. Start with $\underbrace{v_1, \dots, v_{n - 1}}_{\text{Orthonormal Basis}} \in T_VS^{n - 1}$, and $v_n(t) = \Dot{\gamma}_v(t)$. Therefore, we can write $\br{J_1(t), \dots, J_{n - 1}(t), \Dot{\gamma}_v(t)}$ as colume vectors with respect to the orthonormal basis $\Gamma = \br{v_1(t), \dots, v_{n - 1}(t), v_n(t)}$. And in the basis $\Gamma$, we have
    \begin{equation*}
        [[J_1(t)]_\Gamma|\cdots|[J_n(t)]_{\Gamma}] = 
        \begin{bmatrix}
            \begin{matrix}
            \vdots & \empty & \vdots\\
            [J_1(t)]_\Gamma & \cdots & [J_{n - 1}(t)]_\Gamma\\
            \vdots & \empty & \vdots
            \end{matrix} & \rvline & \bigzero\\
            \hline
            \bigzero & \rvline & 1
        \end{bmatrix}
    \end{equation*}
    Then $j(v, t) = \det{[J_1(t)]_{\Gamma}| \dots| [J_{n - 1}(t)]_{\Gamma}}$, which is the determinant of the $(n - 1)\times (n - 1)$-block. Remember that for 
    \begin{equation*}
        A(t) = 
        \begin{bmatrix}
            a_{11}(t) & \dots & a_{1n}(t)\\
            \vdots & \ddots & \vdots\\
            a_{n1}(t) & \dots & a_{nn}(t)
        \end{bmatrix}
        = \begin{bmatrix}
            \vdots & \empty & \vdots\\
            a_1(t) & \cdots & a_n(t)\\
            \vdots & \empty & \vdots
        \end{bmatrix}
    \end{equation*}
    we have
    \begin{equation*}
        \frac{d}{dt}\det{A(t)} = \sum_{i = 1}^n\det{a_1(t)|\dots|a_i'(t)|\dots |a_n(t)}
    \end{equation*}. 
    Therefore,
    \begin{equation*}
        \partial_tj(v, t) = \sum_{i = 1}^{n - 1}\det{(J_1(t)| \dots| J_i'(t)|\dots|J_{n - 1}(t))}. 
    \end{equation*}
    Recall the Jacobi equation along $\gamma_v(t)$
    \begin{equation*}
      J'' + R_{\Dot{\gamma}(t)}(J) = 0 \quad J(0)=0\quad \text{where $R_{\Dot{\gamma}(t)}(J) = R(J, \Dot{\gamma}(t))\Dot{\gamma}(t)$}
    \end{equation*}
    can be split into 
   	\begin{equation*}
    \begin{cases}
    J(0)=0\\
        \Dot{J} = SJ \quad\text{$S$ is the symmetric shape operator of $t$-sphere}\\
        \Dot{S} + S^2 + R_{\Dot{\gamma}} = 0\quad\text{Riccati Equation}
    \end{cases}
    \end{equation*}
    where $S$ satisfies the initial condition $S(t)\sim \frac{1}{t}\id$ as $t\to 0^+$.
    
    Then 
    \begin{equation*}
        \partial_tj(v, t) = \sum_{i = 1}^{n - 1} \det{[J_1(t)]_\Gamma| \dots| [J_{i - 1}(t)]_\Gamma| [SJ_i(t)]_\Gamma| [J_{i + 1}(t)]_\Gamma| \dots  |[J_{n - 1}(t)]_\Gamma}.
    \end{equation*}
    We will only need $t \leq \cut(v)$, i.e. $\gamma_v(t)$ is shortest on $[0, \cut(v)]$. That implies there are no conjugate points on $[0, \cut(v))$. Therefore, $J_1(t), \dots, J_{n - 1}(t)$ are linear independent for all $t \in [0, \cut(v))$ (This is a basis of $\Dot{\gamma}(t)^\perp$ and we can write $S$ as a matrix in this basis.). 

    Since
    \begin{align*}
    &
    \begin{cases}
        & S(J_1) = \sum_{k = 1}^{n - 1} S_{k1}\cdot J_k\\
        & \vdots \\
        & S(J_i) = \sum_{k = 1}^{n - 1} S_{ki}\cdot J_k
    \end{cases}
        \\
        \implies & \det{[J_1(t)]_\Gamma| \dots| [J_{i - 1}(t)]_\Gamma| \underbrace{[SJ_i(t)]_\Gamma}_{= \sum_{k = 1}^nS_{ki}\cdot J_k}| [J_{i + 1}(t)]_\Gamma| \dots  |[J_{n - 1}(t)]_\Gamma}\\
        & =   S_{ii}\cdot\det{[J_1(t)]_\Gamma|\dots|[J_i(t)]_\Gamma|\dots|[J_{n - 1}(t)]_\Gamma}\\
        \implies & \partial_t j(v, t) = \sum_{i = 1}^{n - 1}S_{ii}\cdot\det{[J_1(t)]_\Gamma|\dots|[J_i(t)]_\Gamma|\dots|[J_{n - 1}(t)]_\Gamma}\\
        \implies &  \partial_t j(v, t) = \tr(S)j(v, t)
    \end{align*}
    Then we can take the trace on the matrix equation $S' + S^2 + R_{\Dot{\gamma}(t)} = 0$. We get
    \begin{equation}\label{eq: trace of the shape operator}
        \tr(S') + \tr(S^2) + \tr(R_{\Dot{\gamma}(t)}) = 0
    \end{equation}
    Notice that
    \begin{itemize}
        \item $\tr(S') = \tr(S)'$;
        \item $\tr(R_{\Dot{\gamma}(t)}) = \ric(\Dot{\gamma}(t), \Dot{\gamma}(t)) \geq (n - 1)\kappa\abs{\Dot{\gamma}(t)}^2 = (n - 1)\kappa$;
        \item $\tr(S^2) = \inp{S, S} = \abs{S_0}^2 + a\abs{I}^2 = \underbrace{\abs{S_0}^2}_{\geq 0} + a^2(n - 1)$. where $S_0$ is defined as $S = S_0 + aI$ and $a = \frac{\tr(S)}{n - 1}$. Then $\tr(aI) = (n - 1)a \implies \tr(S_0) = \tr(S) - \tr(S) = 0$. And the inner product on square matrices is defined as $\inp{A, B} = \tr(A\cdot B^t)$. We can check that $S_0 \perp I$ under this inner product:
        \begin{align*}
            \inp{S_0, I} = \tr(S_0\cdot I^t) = \tr(S_0) = 0
        \end{align*}
    \end{itemize}
    Therefore, we can write the equation \ref{eq: trace of the shape operator} of the trace of the shape operator as
    \begin{align*}
        & \underbrace{\tr(S')}_{= \tr(S)'} + \underbrace{\tr(S^2)}_{= \abs{S_0}^2 + a^2(n - 1)} + \underbrace{\tr(R_{\Dot{\gamma}(t)})}_{\geq \kappa(n - 1} = 0\\
        \implies & \underbrace{\tr(S)'}_{= (n - 1)a'} + \abs{S_0}^2 + a^2(n - 1) + \kappa(n - 1) \leq 0\\
        \implies & (n - 1)a' + a^2(n - 1) + \kappa(n - 1) \leq 0\\
        \implies & a' + a^2 + \kappa \leq 0\quad\text{where $a = \frac{\tr(S)}{n - 1}$}
    \end{align*}
    And now we have a Riccati equation on the scalar function $a$. Notice that we have $a(t) \to \infty$ as $t \to 0^+$ because $S(t) \sim\frac{1}{t}\id$ as $t \to 0^+$. Remember that $\overline{a}(t) := \ctk(t)$ is the solution of $\overline{a}' + \overline{a}^2 + \kappa = 0$ and $\lim_{t \to 0^+}\overline{a}(t) = +\infty$ on the interval of existence. (The interval of $\overline{a}$ is bigger.) Therefore, $a(t) \leq \overline{a}(t)$. 
    \begin{remark}
        In the case of the model space $\modsp{\kappa}$. $\sect \equiv \kappa$, $\overline{S}_0 \equiv 0$. So we get equality on $\overline{a} = \tr(\overline{S})$. 
        \begin{equation*}
            \overline{a}' + \overline{a}^2 + \kappa \equiv 0 \implies \overline{a} = \ctk(t)
        \end{equation*}
        By Riccati comparison, $a(t) \leq \overline{a}(t) = \ctk(t)$ before the first conjugate point along $\gamma_v$. Just as in Rauch's comparison, we get an inequality on $\underbrace{j(t)}_{j(t, v)}$ and $\overline{j}(t)$. Therefore, we have
        \begin{align*}
            &
            \begin{cases}
                j'(t) = (n - 1)aj(t)\\
                \overline{j}'(t) = (n - 1)\overline{a}\overline{j}(t)
            \end{cases}
             \implies
            \begin{cases}
                (n - 1)a = \frac{j'}{J}= \ln(j)' \\
                (n - 1)a = \frac{\overline{j}'}{\overline{j}}= \ln(\overline{j})'
            \end{cases}\\
            \implies & \ln(\frac{j}{\overline{j}})' = \ln(j)' - \ln(\overline{j})' = (a - \overline{a})(n - 1) \leq 0\\
            \implies & \text{$\ln(\frac{j}{\overline{j}})$ is non-increasing}\\
            & and \lim_{t \to 0^+}(j(t)) = \lim_{t \to 0^+}(\overline{j}(t)) = 1
        \end{align*}
    \end{remark}
    Therefore, $\abs{j(t)} \leq \abs{\overline{j}(t)}$ up to the first zero of $j(t)$ which must occur before the first zero of $\overline{j}$. So 
    \begin{equation*}
        j' = \underbrace{(n - 1)a}_{= \tr(S)} j
    \end{equation*}
    and
    \begin{equation*}
        \overline{j} = (n - 1)\underbrace{\overline{a}}_{=\ctk(t)}\overline{j} \implies \overline{j}(t) = (\snk(t))^{n - 1}
    \end{equation*}
    Then we can conclude the following
    \begin{itemize}
        \item If $\kappa \leq 0$, then $\overline{j}(t) > 0$ for $t > 0$;
        \item If $\kappa > 0$, then the first zero of $\snk(t)$ is $\frac{\pi}{\sqrt{\kappa}} = \diam(\mathbb{M}^2(\kappa))$. Therefore $j$ must have a zero before  $\frac{\pi}{\sqrt{\kappa}}$. Which means that the first conjugate point along $\gamma_v$ occurs before  $\frac{\pi}{\sqrt{\kappa}}$. This immediately gives. 
        	\begin{corollary}
            If $\ric_M \geq (n - 1)\kappa > 0$, then $\diam(M) \leq \frac{\pi}{\sqrt{\kappa}} = \diam(\modsp{\kappa})$. \end{corollary}
	\end{itemize}
    Let's collect what we have so far. For $v \in T_pM$, $\abs{v} = 1$, $\gamma_v(t) = \exp(tv)$. Because $\abs{v}$ is a unit vector, $\cut(v) = \max(T)$ such that $\gamma_v(t)$ is shortest on $[0, T]$. ($\cut(v)$ could be $\infty$ but for compact manifolds $\cut(v) < \infty$). 
    \begin{align*}
        T_pM & \supseteq \overline{\cut(p)}:=\bigcup_{: v \in T_pM, \abs{v} = 1}\br{\cut(v)\cdot v}\\
        M & \supseteq \cut(p) := \bigcup_{: v \in T_pM, \abs{v} = 1}\br{\exp(\cut(v)\cdot v)}\\
        C_p & = \br{tv: t < \cut(v): v \in T_pM, \abs{v} = 1} = \cut(p)^\circ \subseteq M
    \end{align*}
    We know that $\exp_p: C_p \to M \backslash \cut(p)$ is a diffeomorphism.
    \begin{fact}
        $\cut(p)$ has measure $0$, for computing volumes $\cut(p)$ is irrelevant 
        \begin{equation*}
            \vol(B(p, r)) = \vol(B(p, r)\backslash\underbrace{\cut(p)}_{\text{has measure}}) = \vol(\exp(C_p\cap B(0, r)))
        \end{equation*}
    \end{fact}
    So $\exp_p: B(0, r)\cap C_p \to B(p, r)\backslash\cut(p)$ is a diffeomorphism (This is for computing volume). For computing volume, we just want to use this part of the map. In short, 
    \begin{equation*}
        \vol(\underbrace{B(p, r)}_{\text{in polar coordinate}}) = \int_{S^n}\int_0^{\min{\br{r, \cut(v)}}}j(v, t)dtd\vol^{n - 1}
    \end{equation*}
    Let's redefine $j(v, t)$ to be $0$ for $t \geq \cut(v)$. And do the same for $\overline{j}(t)$:
    \begin{itemize}
        \item If $\kappa \leq 0$, then $\overline{j}(t) > 0$ for $t > 0$ (Don't need to redefine $\overline{j}$);
        \item  If $\kappa > 0$, then $\overline{j}(t) = (\snk(t))^{n - 1}$. We set it to be zero for $t > \frac{\pi}{\sqrt{\kappa}} = \diam(\modsp{\kappa})$
    \end{itemize}
    Recall the first zero of $J(t)$ is $\leq$ than the first zero of $\overline{J}$, then $\frac{j(t, v)}{\overline{j}(t)}$ is still non-increasing with this definition, for $\tau \geq \cut(v)$, we set it to be zero. Then we have redefined $j(t, v)$ and $\overline{j}(t)$. 
    
    Since we can express
    \begin{align*}
        \vol(B(p, r)) = & \int_0^r\brac{\int_{S^{n - 1}}j(t, v)dt}d\vol^{n - 1}\\
        = & \int_0^r\brac{\int_{S^{n - 1}}\frac{j(t, v)}{\overline{j}(t)}\overline{j}(t)d\vol^{n - 1}}dt\\
        = & \int_0^r \overline{j}(t)\cdot\brac{\int_{S^{n - 1}}\frac{j(t, v)}{\overline{j}(t)}d\vol^{n - 1}}dt
    \end{align*}
    and
    \begin{align*}
        \vol(\overline{B}(p, r)) = &\int_0^r\brac{\int_{S^{n - 1}}\underbrace{\overline{j}(t)}_{\text{doesn’t depends on $v$}}dt}d\vol^{n - 1}\\
        = & \vol(S^{n - 1})\int_0^r\overline{j}(t)dt
    \end{align*}
    Thus, we have
    \begin{equation*}
        \frac{\vol(B(p, r))}{\vol(\overline{B}(p, r))} = \frac{\int_0^r\overline{j}(t)\cdot\brac{\int_{S^{n - 1}}\frac{j(t, v)}{\overline{j}(t)}d\vol^{n - 1}}dt}{\vol(S^{n - 1})\int_0^r\overline{j}(t)dt} =: \frac{\int_0^rq(t)\overline{j}(t)dt}{\int_0^r\overline{j}(t)dt}
    \end{equation*}
    where we denote $q(t) = \frac{1}{\vol(S^{n - 1})}\int_{S^{n - 1}}\frac{j(t, v)}{\overline{j}(t)}d\vol^{n - 1}$. Since we know that $\frac{j(t, v)}{\overline{j}(t)}$ is non-increasing for all $t$ and $q(t)$ is also non-increasing concerning $t$. 
    Thus the ratio $\frac{\vol(B(p, r))}{\vol(\overline{B}(p, r))}$ is also non-increasing. This is because 
    \begin{equation}\label{eq: volume ratio}
        \frac{\vol(B(p, r))}{\vol(\overline{B}(p, r))} = \frac{\int_0^rq(t)dm}{m([0, r])} = \dashint_0^rq(t)dm
    \end{equation}
    and the integrand $q(t)$ is non-increasing (by Lemma~\ref{lem: monotone average integral}), where $m = \overline{j}(t)dt$ the measure on $\R^+$ and $m(A) = \int_A\overline{j}(t)dt$ for $A \subseteq \R^+$.
    \end{proof}

\chapter{Gromov-Hausdorff Convergence of Metric Spaces}
In this section, we are going to introduce the Gromov-Hausdorff topology. This concept is generalized from the Hausdorff topology for metric spaces. The reason that we need this generalized concept is that we want to study a family of compact Riemannian manifolds, or more generally, a family of metric spaces, sharing the same properties on curvature bound, volume bound, and diameter bound, for example, the family of Riemannian manifolds
	\begin{equation*}
        M_{ric}(n, \kappa, D) = \br{(M^n, g) : \ric_M \geq (n - 1)\kappa, \diam_M \le D}. 
    \end{equation*}
    We will see that this class is pre-compact in Gromov-Hausdorff topology. This fact has a number of interesting applications.
    
    \section{Hausdorff Distance}
Let us briefly review the Hausdorff distance here. The Hausdorff distance allows us to measure the distance of two subsets intrinsically from the point of view of metric. Intuitively, the Hausdorff distance gives us the least radius that we can shrink one subset to touch another subset. 

\begin{notation}
	For a metric space $(X, d)$, and a subset $A \subseteq X$. For any $x \in X$, we can denote the distance from $x$ to $A$ as 
	\begin{equation*}
		d(x, A) = \inf_{y \in A}\br{d(x, y)}.
	\end{equation*}
\end{notation}
\begin{definition}[$\varepsilon$-neighborhood]
	Let $(X, d)$ be a metric space and $A \subseteq X$. For $\varepsilon > 0$, we can define the \textbf{$\varepsilon$-neighborhood} of $A$ as
	\begin{equation*}
		U_\varepsilon(A) = \br{x \in (X, d): d(x, A) < \varepsilon}.
	\end{equation*}
	where $\varepsilon$ is called the \textbf{radius} of the neighborhood. 
\end{definition}
\begin{definition}[Hausdorff Distance]\label{def:haus-distance}
Let $Z$ be a metric space, and let $X, Y \subseteq Z$ be subsets. The the \textbf{Hausdorff distance of $X$ and $Y$}, denoted by $d_H(X, Y)$, is defined as 
\begin{equation}
    d_H(X, Y) = \inf\br{\varepsilon > 0: Y \subseteq U_\varepsilon(X), X \subseteq U_\varepsilon(Y)}.
\end{equation}
where $U_\varepsilon(X)$ and $U_\varepsilon(Y)$ are the $\varepsilon$-neighborhood of $X$ and $Y$.
\end{definition}

We can provide an equivalent definition of the Hausdorff distance by thinking of the smallest radius of the neighborhood of $A$ covering $B$ as the largest distance from some point of $A$ to $B$ (See Figure~\ref{fig: haus-dist-equiv}).  That is
 \begin{definition}\label{def: haus-distance 2}
 	Let $Z$ be a metric space, and $X, Y \subseteq Z$ be subsets, the \textbf{Hausdorff distance of $X$ and $Y$} can also be defined as
	\begin{equation}
    d_H(X, Y) = \max\br{\sup_{y \in Y}\br{d(y, X)}, \sup_{x \in X}\br{d(x, Y)}}.
\end{equation}
 \end{definition}
\begin{figure}[htbp]
        \centering
       	\includegraphics[width=0.8\textwidth]{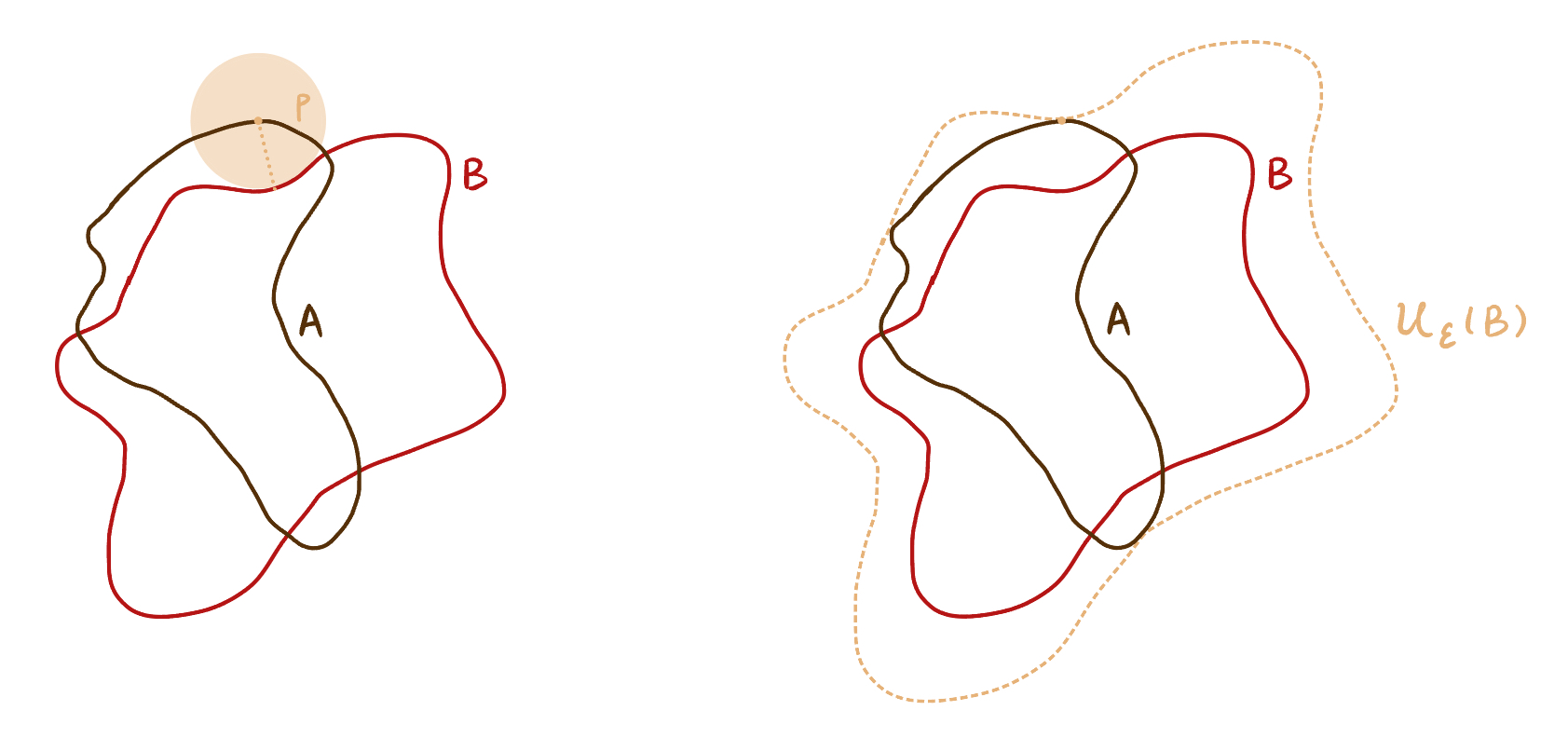}
       	\caption{The largest distance from the fixed point $p$ of $A$ to $B$ (left) is equal to the smallest radius of $A$ covers $B$ (right)}
       	\label{fig: haus-dist-equiv}
    \end{figure}

\begin{remark}
It is possible that the $d_H(X, Y) = + \infty$
\end{remark}
\begin{proposition}
	Let $(Z, d)$ be a metric space, then $d_H$ is a distance fuction on closed subsets of $Z$, i.e.
\begin{itemize}
    \item The triangle inequality holds: For  closed $X_1, X_2, X_3 \subseteq Z$,
    \begin{equation*}
        d_H(X_1, X_3) \leq d_H(X_1, X_2) + d_H(X_2, X_3). 
    \end{equation*}    
    \item For $X, Y$ closed subsets of $Z$, we have
    \begin{equation*}
        d_H(X, Y) = 0  \Longleftrightarrow X = Y
    \end{equation*}
    \item $d_H$ is symmetric, i.e  $d(X,Y)=d(Y,X)$. 
\end{itemize}
\begin{proof}
	 The idea of proving the triangule inequality is the following: For any $\varepsilon, \delta \in \R$ be two numbers such that 
    \begin{equation*}
    	\varepsilon > d_H(X_1, X_2) \quad\text{and} \quad \delta > d_H(X_2, X_3).
    \end{equation*}
    This implies that
    \begin{equation*}
    	X_3 \subseteq U_\delta(X_2) \quad\text{and}\quad X_2 \subseteq U_\varepsilon(X_1)
    \end{equation*}
    We claim that $X_3 \subseteq U_{\varepsilon + \delta}(X_2)$. Let $x_3 \in X_3$, then there exists $x_2 \in X_2$ such that $d(x_3, x_2) \leq \delta$. And moreover there exists $x_1 \in X_1$ such that $d(x_1, x_2) < \varepsilon$. Therefore, by the triangular inequality in the metric space, we have $d(x_1, x_3) \leq \varepsilon + \delta$, which means $X_3 \subseteq U_{\varepsilon + \delta}(X_1)$. Similarly, we can check that $X_1 \subseteq U_{\varepsilon + \delta}(X_3)$. 
\end{proof}
\end{proposition}

\section{Gromov-Hausdorff Distance}
In this section, we generalize the concept of Hausdorff distance. This generalization allows us to define the distance between two metric spaces that are not necessarily subsets of the same metric space. 
\begin{definition}[Gromov-Hausdorff Distance]\label{def: gh-distance}
    Let $X, Y$ be compact metric spaces. Then the \textbf{Gromov-Hausdorff distance between $X$ and $Y$}, denoted by $d_{GH}(X, Y)$, is defined as 
    \begin{equation*}
    	d_{GH}(X, Y) = \inf\br{\varepsilon > 0: \text{$d_H^Z(i(X), i(Y)) \leq \varepsilon$ for some metric space $Z$}}
    \end{equation*}
    where $i: X \hookrightarrow Z$ and $j: Y \hookrightarrow Z$ are both distance-preserving embeddings. 
\end{definition}
\begin{remark}
	The compactness of $X$ and $Y$ implies that $i(X)$, $j(Y)$ are closed subset of $Z$.
\end{remark}
\begin{fact}
We have the following observations regarding the Gromov-Hausdorff distance:
\begin{itemize}
    \item In the definitionn of $d_{GH}(X, Y)$ it is sufficient to consider $Z$ as the disjoint union of $X$ and $Y$.
    \item In the original definition, if $Z$ is arbitrary, then it is possible that 
    \begin{equation*}
        i(X)\cap j(Y) \neq \emptyset
    \end{equation*}
    In this situation, we can take $\Tilde{Z} = Z \times \R$ and take $\hat{i}:= (i, 0)$, $\hat{j} = (j, \eps)$ so that
    \begin{equation*}
        \hat{i}(X)\cap \hat{j}(Y) = \emptyset
    \end{equation*}
    This construction allows us to consider only the disjoint case. So we can always take the disjoint union of $X$ and $Y$. 

    \begin{figure}[htbp]
        \centering
        \includegraphics[width=0.4\textwidth]{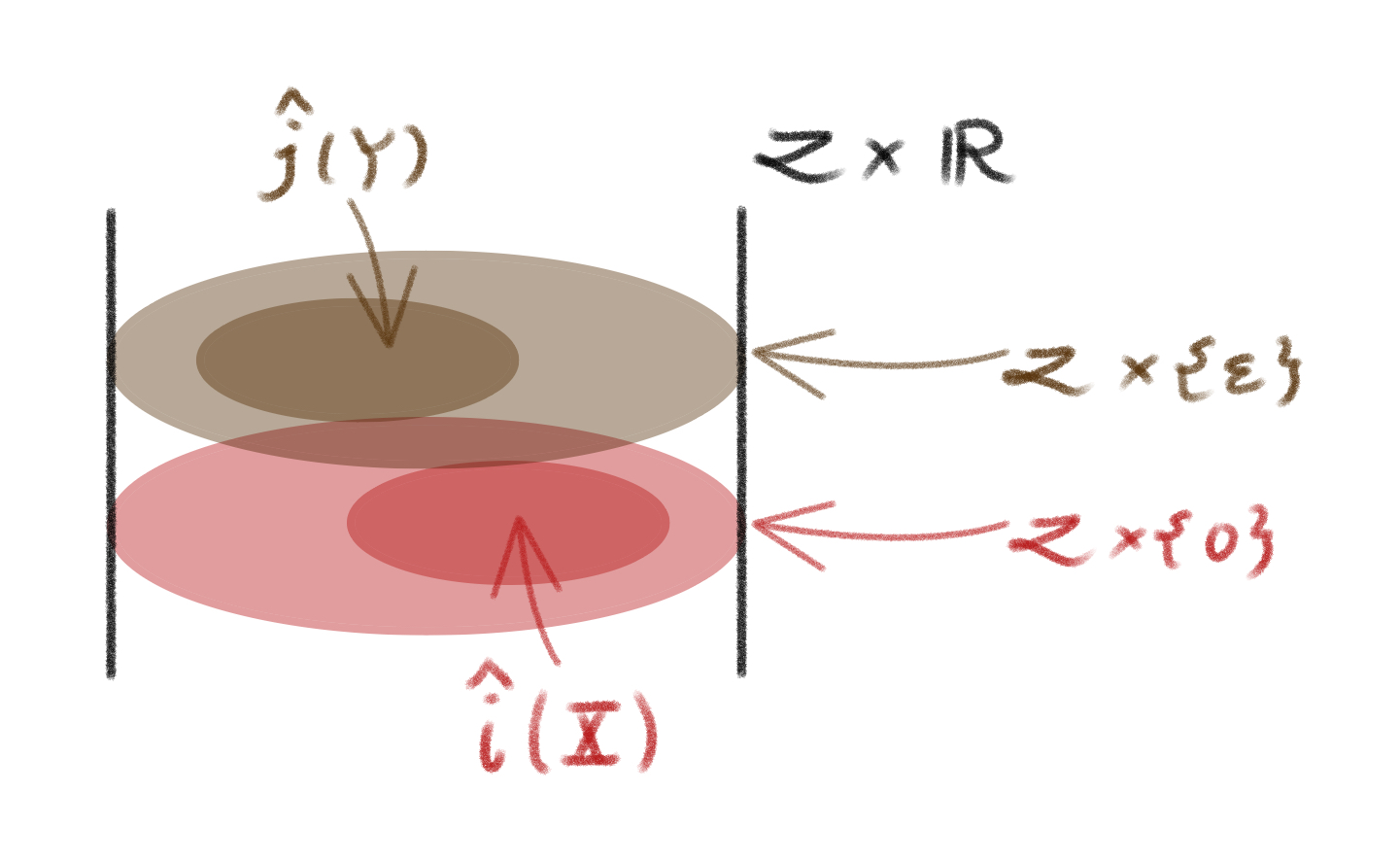}
        \caption{Construction of $\Tilde{Z}$}
    \end{figure}
    
    \item $d_{GH}$ satisfies the triangle inequality. Fix $X_1, X_2, X_3$ and we claim that 
    \begin{equation*}
        d_{GH}(X_1, X_3) \leq d_{GH}(X_1, X_2) + d_{GH}(X_2, X_3)
    \end{equation*}
    Take $Z = X_1\coprod X_2 \coprod X_3$, since we have metrics on $X_1\coprod X_2$ and $X_2\coprod X_3$, say
    \begin{equation*}
        \eps = d_{GH}(X_1, X_2), \quad \delta = d_{GH}(X_2, X_3)
    \end{equation*}
    in order to get the distance on $Z$, for $x_1 \in X_1$, $x_3 \in X_3$, we define 
    \begin{equation*}
        d^Z(x_1, x_3) = \inf_{y \in X_2}\br{d(x_1, y) + d(y, x_3)}
    \end{equation*}
  	It is easy to check that this is a metric on $Z$ and in this metric
    \begin{equation*}
        d_H^Z(X_1, X_3) \leq d_H^Z(X_1, X_2) + d_H^Z(X_2, X_3)
    \end{equation*}
    Via the same approach of triangle inequality of $d_H$, we can get the triangular inequality for $d_{GH}$ as well.
\end{itemize}
\end{fact}
\section{The Topology of Isometry Class}
Recall that we want to ensure that the Gromov-Hausdorff distance $d_{GH}$ is a distance on the isometry classes of compact metric spaces. 
We have shown that it satisfies the triangle inequality. It is also obviously symmetric. Thus to check that it is a distance it remains to 
verify the following lemma. 
\begin{lemma}\label{lem:gh_distance zero}
Let $X$ and $Y$ be compact metric spaces. Then 

\begin{equation*}
    d_{GH}(X, Y) = 0 \iff X \stackrel{isom}{\cong} Y.
\end{equation*}
(We will prove this lemma later.)
\end{lemma} 
Usually, computing the exact value of Gromov-Hausdorff distance is complicated and unnecessary. And mostly we are only interested in the topology induced by such metric, in particular in the notion of convergence. 
\begin{definition}[Gromov-Hausdorff Convergnce]\label{def:gh-convergence}
    Let $\br{X_n}_{n\in \N}$ and $Y$ be compact metric spaces, then we say $X_n$ converges to $Y$ in the sense of Gromov-Hausdorff topology or  $X_n \ghto Y $ if 
    \begin{equation*}
        \lim_{n \to \infty}d_{GH}(X_n, Y) = 0
    \end{equation*}
\end{definition}
\begin{example}
One can check that
\begin{equation*}
    X_n \ghto \br{pt} \iff \diam(X_n) \to 0
\end{equation*}
In order to verify this example, we only need to check
\begin{equation*}
    d_{GH}(X, \br{pt}) = \frac{1}{2}\diam(X)
\end{equation*}
\end{example}
\begin{example}\label{ex:gh-convergence on torus}
    Let $(M^n, g)$ be a compact Riemannian manifold. Take a sequence $\eps_n \downarrow 0$. Denote $X_n := \eps_nM := (M^n, \eps_n^2 g)$ (so that the distance in $M$ are multiplied by $\eps_n$ for each $n$). 
    Then $\diam(X_n)\to 0$ and hence $  X_n \ghto \br{pt} $.
 All the fine structures such as topology and curvature, etc. are forgotten. The only information left is the diameter. Therefore, we can say the Gromov-Hausdorff convergence is somehow a ``weak'' convergence.

    \begin{figure}[htbp]
        \centering
        \includegraphics[width=0.9\textwidth]{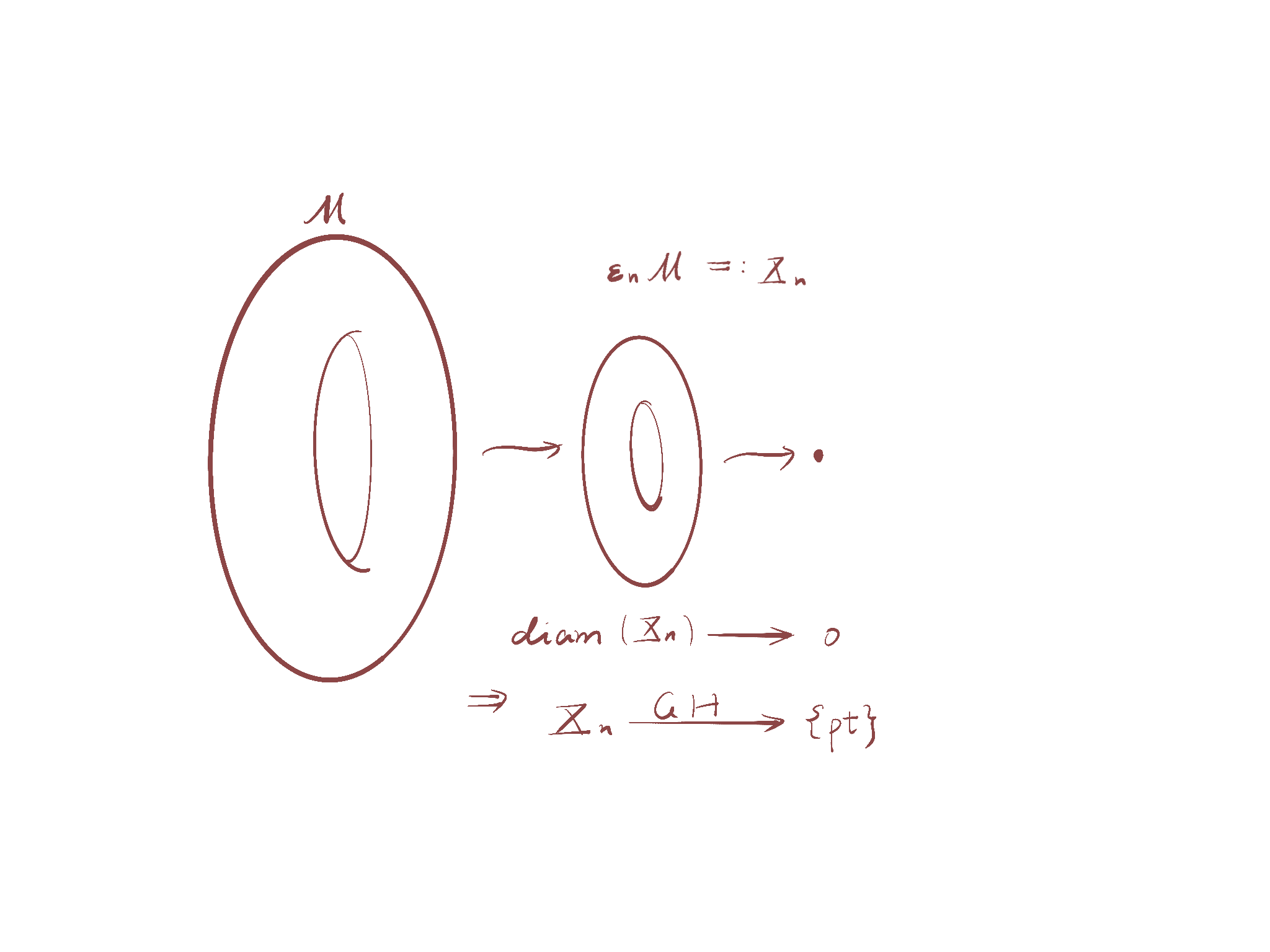}
        \caption{Gromov-Hausdorff convergence on torus}
    \end{figure}
    
    The above picture shows how the topological information is lost at the limit of the Gromov-Hausdorff convergence. 
\end{example}
\begin{example}\label{ex: GH}
	Consider $f: \R \to \R$ with $f(0) = 0$. Then as $\eps \to 0$, 
\begin{equation*}
	\frac{1}{\eps}f(t\eps) \to f'(0). 
\end{equation*}
Since 
\begin{equation*}
	\brac{\frac{1}{\eps}f(t\eps)}' = \frac{1}{\eps}f'(t\eps)\eps = f'(t\eps) 
\end{equation*}
and
\begin{equation*}
	\brac{\frac{1}{\eps}f(t\eps)}'' = \brac{f'(t\eps)}' = \eps f''(t\eps)
\end{equation*}
Thus, if $f'' \leq \lambda$, we have $\brac{\frac{1}{\eps}f(t\eps)}'' \leq \eps\lambda$. This means if $f$ is $\lambda$-concave, then $\frac{1}{\eps}f(t\eps)$ is $\eps\lambda$-concave. Thus, at the limits, we know that $f'$ is $0$-concave, which just means concave. 

Later, we will introduce the pointed-Gromov-Hausdorff convergence. And one can show that given a sequence of $\lambda_i$-concave function $f_i: (X_i, p_i) \to \R$ with $\lambda_i \to \lambda$ and $(X_i, p_i) \pghto (X, p)$. We can conclude that $f_i$ pointwise converges to a $\lambda$-concave function $f: (X, p) \to \R$.

In our case, consider $f: (M, p) \to \R$, by the previous argument, we construct the following sequence of $\eps\lambda$-concave functions $\frac{1}{\eps}f: (\frac{1}{\eps}M, p) \to \R$. Since $(\frac{1}{\eps}M, p) \pghto (T_pM, 0)$, we know that at $\eps \to 0$, $df_p: (T_pM, 0) \to \R$ must be a concave function.
\end{example}
\subsection{Gromov-Hausdorff Approximations}
As we discussed in the example \ref{ex:gh-convergence on torus}, even though the notion of the Gromov-Hausdorff convergence is ``weak'', it is still useful because of the pre-compactness as we mentioned before (It is, in particular, useful to the class of manifolds with lower Ricci curvature or sectional curvature bound). In the example \ref{ex:gh-convergence on torus}, where we start scaling down the Riemannian metric. We basically lose all the curvature bound (the curvature blows up). Thus, if we want to keep inside some class of lower curvature bound, we should be more careful and the Gromov-Hausdorff convergence will become more useful. 

Now, we are going to list some more meaningful equivalent definitions that are going to help us to prove the lemma \ref{lem:gh_distance zero}.

\begin{definition}[$\varepsilon$-net]
Let $(X, d)$ be a metric space, a subset $S = \br{x_\alpha: \alpha \in I}$ is called an \textbf{$\varepsilon$-net} in $X$ if $U_\varepsilon(S) = X$, i.e. 
	\begin{equation*}
		X = \bigcup_{\alpha \in I} B_\varepsilon(x_\alpha)
	\end{equation*}
\end{definition}
\begin{definition}[Totally Bounded]
	A metric space is totally bounded if it has a finite $\eps$-net for every $\eps > 0$. 
\end{definition}
\begin{theorem}
	A metric space is sequentially compact if and only if it is complete and totally bounded. 
\end{theorem}

\begin{definition}[$\eps$-Gromov-Hausdorff Approximation]\label{def:epsilon-GH-approx}
    Let $X$ and $Y$ be metric spaces and let $\varepsilon > 0$. Then $f: X \to Y$ is an \textbf{$\varepsilon$-Gromov-Hausdorff approximation} if
    \begin{enumerate}
        \item For any $x_1, x_2 \in X$, 
        \begin{equation*}
            \abs{d^Y(f(x_1), f(x_2)) - d^X(x_1, x_2)} \leq \eps;
        \end{equation*}
        \item $f(X)$ is \textbf{$\varepsilon$-dense} in $Y$, i.e. for any $y \in Y$, there exists $x \in X$ such that $d^Y(y, f(x)) \leq \varepsilon$. (Y = $U_\varepsilon(f(X))$ every point is not too far away from the image of $f$.)
    \end{enumerate}
\end{definition}
\begin{remark}
    In the definition \ref{def:epsilon-GH-approx} of $\varepsilon$-Gromov-Hausdorff approximation, there is no assumption on the continuity of the $f$. 
\end{remark}
We will see the definition \ref{def:epsilon-GH-approx} is related to the Gromov-Hausdorff distance. But first of all, we want to list some of its properties
\begin{proposition}[See page 258 in \cite{BBI01}]\label{prop: GH and GH-ap}
	If $d_{GH}(X, Y) \leq \varepsilon$, then there exists a $2\varepsilon$-Gromov-Hausdorff approximation $f: X \to Y$. Conversely, if $f: X \to Y$ is $\varepsilon$-Gromov-Hausdorff approximation, then $d_{GH}(X, Y) \leq 2\varepsilon$
\end{proposition}
\begin{proposition}
	If $f: X \to  Y$ is an $\varepsilon$-Gromov-Hausdorff approximation, then there exists $g: Y \to X$ which is a $7\varepsilon$-Gromov-Hausdorff approximation. 
\end{proposition}
    \begin{proof}
    Suppose $f: X \to Y$ is an $\eps$-Gromov-Hausdorff approximation. Take $\br{x_1, \dots, x_N}$ an $\eps$-net in $X$ and set $y_i = f(x_i)$ for each $i$.
    
    We claim that $\br{y_1, \dots, y_M}$ is an $3\eps$-nets in $Y$. Indeed, for any $y \in Y$ since $f(X)$ is $\eps$-dense in $Y$, there exists $x \in X$ such that $d^Y(y, f(x)) \leq \eps$ and we know that this $x \in B_\eps(x_i)$ for some $i \in \br{1, \dots, N}$. Thus, 
    \begin{align*}
    	d^Y(y, y_i) 
    	&= d^Y(y, f(x_i))\\
    	&\leq d^Y(y, f(x)) + d^Y(f(x), f(x_i)) \\
    	&\leq \eps + d^X(x, x_i) + \eps \\
    	&< 3\eps.
    \end{align*}
    Next, we can define $g: Y \to X$ as follows. For each $y \in Y$, if $y = y_i$ for some $i$, we take $g(y) = x_i$ and for other $y$, we can pick the $y_i$ with $d^Y(y, y_i) \leq 3\varepsilon$ and set $g(y) = x_i$. In the case when $f(x_i) = f(x_j) = y$ for $i < j$, we define $g(y) = x_i$ (We just pick the one with a smaller index so that $g$ is well defined.). Then it is easy to see that $g$ is a $7\varepsilon$-Gromov-Hausdorff approximation. Indeed. Let $y, y' \in Y$, then we know that there are $y_i, y_j$ such that $d^Y(y, y_i) \leq 3\eps$ and $d^Y(y', y_j) \leq 3\eps$, thus $g(y) = x_i$ and $g(y') = x_j$. Then
   \begin{equation*}
   	\begin{split}
   		&\abs{d^X(g(y), g(y')) - d^Y(y, y')}\\ 
   		\leq & \abs{d^X(x_i, x_j) - d^Y(y_i, y_j)} + \abs{d^Y(y_i, y_j) - d^Y(y, y')}\\
   		\leq & \abs{d^X(x_i, x_j) - d^Y(f(x_i), f(x_j))} + d^Y(y, y_i) + d^Y(y', y_j)\\
   		\leq & \eps + 3\eps + 3\eps = 7\eps   
   	\end{split}
   \end{equation*}
   We still need to show $g(Y)$ is $7\eps$-dense in $X$. Let $x \in X$, we can find $x_i$ such that $d(x, x_i) \leq \eps$. Take $y_i = f(x_i)$. By the definition of $g$ it is possible that $g(y_i) = x_j$ for some $j \neq i$ such that $f(x_i) = f(x_j)$. By the definition of $\eps$-Gromov-Hausdorff approximation, we know that
   \begin{equation*}
   	\abs{d^X(x_i, x_j)} = \abs{d^Y(f(x_i), f(x_j)) - d^X(x_i, x_j)} \leq \eps
   \end{equation*}
   Therefore, 
   \begin{equation*}
   		d^X(x, g(y_i)) = d^X(x, x_j) \leq d^X(x, x_i) + d^X(x, x_j) \leq 2\eps < 7\eps
   \end{equation*}
Therefore, $g(Y)$ is also $7\eps$ dense.
\end{proof}
\begin{remark}
    The above facts reflect that the definition of $\varepsilon$-Gromov-Hausdorff approximation is almost symmetric. And the function in the opposite direction is slightly worse.
\end{remark}
The two properties above allow us to conclude the following corollaries.
\begin{corollary}\label{cor:GH-con iff GH-approx}
Let $X, X_n, n\in \mathbb N$ be compact metric spaces. Then $X_n \ghto X$ as $n \to \infty$ if and only if there exists a sequence $\varepsilon_n \downarrow 0$ and there exists $f_n: X_n \to X$ are $\varepsilon_n$-Gromov-Hausdorff approximation.
\end{corollary}

\begin{corollary}\label{cor: limits of GH-con are metric spaces}
    Every compact metric space is the limit of some sequence of finite metric spaces in the sense of Gromov-Hausdorff convergence.
\end{corollary}
\begin{proof}[Proof of Corollary~\ref{cor: limits of GH-con are metric spaces}]
    Let $X$ be a compact metric space.  Consider any sequence $\varepsilon_n \downarrow 0$. Then for each $\varepsilon_n$, We can take 
    \begin{equation*}
        x_1, \dots x_{N(n)} \in X
    \end{equation*}
    to be an $\varepsilon_n$-net in $X$. Then the inclusion
    \begin{equation*}
        X_n := \bigcup_{i = 1}^{N(n)}\br{x_i}\xhookrightarrow{i_n} X
    \end{equation*}
    is an $\varepsilon_n$-Gromov-Hausdorff approximation. Thus by the corollary \ref{cor:GH-con iff GH-approx}. We can conclude that
    \begin{equation*}
        d_{GH}(X_n, X) \to 0 \quad\text{as $n \to \infty$}.
    \end{equation*}
\end{proof}
Now we can prove the lemma \ref{lem:gh_distance zero} which confirms the $d_{GH}$ distance on the isometry classes of compact metric spaces.
\begin{proof}[Proof of Lemma~\ref{lem:gh_distance zero}]
    Suppose $d_{GH}(X, Y) = 0$, for $\varepsilon_n \downarrow 0$, we know that there exists $f_n: X \to Y$ being $\varepsilon_n$-Gromov-Hausdorff approximation for some $\varepsilon_n \downarrow 0$. And we also have $\br{g_n: X \to Y}$ the corresponding $\varepsilon_n$-Gromov-Hausdorff approximation. 
    
    Since $X$ is a compact metric space, it admits a countable dense subset say $S = \br{x_1, x_2, \dots }$. We are going to use the diagonal argument to show that $f_n|S: S \to Y$ sub-converge to some $f: S \to Y$ which preserves the distances. 
    
    By the compactness of $Y$, the sequence $f_n(x_i)$ sub-converges to some $y_i$. As we define $f$ as $f(x_i) = y_i$ we can easily show that $f$ is distance preserving, i.e. 
    \begin{equation*}
    	d(y_i, y_j) = d(x_i, x_j).
    \end{equation*}
    This is because as $n \to \infty$(distance function is continuous)
    \begin{equation*}
    	\abs{d(f_n(x_i), f_n(x_j)) - d(x_i, x_j)} < \varepsilon_n \to 0
    \end{equation*}
    We can extend $f:S \to Y$ to $f: X \to Y$ which is also distance preserving. This is because for each $x \in X$, by the compactness of $X$ and the fact that $S$ is compact, we can find a subsequence $x_{ij}$ converges to $x$. Moreover, $f(x_{ij})$ is Cauchy which since
    \begin{equation*}
    	d(f(x_{ij_1}), f(x_{ij_2})) = d(x_{ij_1}, x_{ij_2}) \to 0 
    \end{equation*}
    By the completeness of $Y$, we know that there exists $y \in Y$ such that $\lim_{j \to \infty}f(x_{ij}) = y$. Thus we define $f(x) = y$. This extension is distance-preserving as well
    \begin{align*}
    	d(f(x^1), f(x^2)) 
    	&= \lim_{j, k \to \infty}d(f(x^1_{ij}), f(x^2_{ik}))	\\
    	&= \lim_{j, k \to \infty}d(x^1_{ij}, x^2_{ij})\\
    	&= d(x^1, x^2)
    \end{align*}

    Similarly, we get $g: Y \to X$ preserves the distance. However, to finish our proof, as we need to know that $f$ and $g$ are both surjective.  This is not obvious and in fact can fail for noncompact spaces 
    as the following example shows
        \begin{example}
        Consider $X = [0, \infty)$ and $f: X \to X$ such that $f(x) = x + 1$. This $f$ is distance preserving but not surjective.
    \end{example}
    The following lemma shows that this can not happen if $X$ is compact.
    
    \begin{lemma}\label{lem: h onto}
        Let $X$ be a compact metric space. If $h: X \to X$ is distance preserving, then $h$ is onto.
    \end{lemma}
    \begin{proof}[Proof of lemma~\ref{lem: h onto}]
    Suppose the lemma is not true. That is $h(X) \subsetneq X$. Since $X$ is compact and Hausdorff, we know that the image $h(X)$ is compact and closed. Let $x_0 \in X\backslash h(X)$ be an element in this open set (Notice that $h(X)$ is closed). Thus there exists $\eps > 0$ such that $B(x_0, \eps)\cap h(X) = \emptyset$. 

 Let $x_1, \dots, x_N$ be a maximal $\eps$-separated net in $X$. The Finiteness of such net  follows by the compactness of $X$ and $N$ is the maximal possible number of points in an $\eps$-separated net. Now take $h(x_1), \cdots, h(x_N)$, we know that since $h$ is distance preserving
    \begin{equation*}
        d(h(x_i), h(x_j)) = d(x_i, x_j) \geq \eps, \quad\forall i \neq j \in \br{1, \dots, N}
    \end{equation*}
    which is again an $\eps$-separated net. However, we know that 
    \begin{equation*}
        d(x_0, h(x_i)) \geq \eps, \quad\forall i \in \br{1, \dots, N}
    \end{equation*}
    So that implies $\br{x_0, x_1, \dots, x_N}$ is also an $\eps$-separated net but with $N + 1$ elements, which contradicts the maximality of $\br{x_1, \dots, x_N}$
\end{proof}

    Now, we know that both $f$ and $g$ are distance preserving, both $X$ and $Y$ are compact. Therefore $f\circ g: Y \to Y$ and $g \circ f: X \to X$ are also both distance preserving. Then by lemma \ref{lem: h onto}, 
    $f\circ g: Y \to Y$ and $g \circ f: X \to X$ are both surjective. Hence both $f$ and $g$ are onto and hence bijective then we can conclude
    \begin{equation*}
        X \stackrel{isom}{\cong} Y.
    \end{equation*}
    In particular, $f$ is a distance preserving bijective map.
\end{proof}

In the last lecture, we proved that the Gromov-Hausdorff distance $d_{GH}$ is indeed a distance on the isometry classes of compact metric spaces.  It is natural to wonder if it remains a metric on the isometry classes on complete metric spaces. It is easy to see that on this bigger class, $d_{GH}$  is symmetric and satisfies the triangle inequality. However Lemma \ref{lem:gh_distance zero} fails for the class of metric spaces as the following example shows.

\begin{example}\label{ex-complete-GH=0}[By Kapovitch]
We can construct an example that $X$ and $Y$ are both complete metric spaces of finite diameter with $d_{GH}(X, Y) = 0$ but $X$ is not isometric to $Y$. 

For each $q \in \Q\cap (1, 2)$, we construct interval an $I^X_q$  such that $\length\brac{I^X_q} = q$ and denote ${I^X_q}^+$ and ${I^X_q}^-$ the left and the right end of the interval. We construct $X$ by gluing ${I^X_q}^+$ for each $q$ together and all the ${I^X_q}^-$ for each $q$ together. It is not hard to see that $X$ is a complete metric space with a finite diameter. The only difference between the construction of $Y$ from $X$ is that it has no interval $I^Y_{\frac{3}{2}}$. Notice that $X$ is not isometric to $Y$ since an isometry must map an interval of length $\frac{3}{2}$ to an interval of length $\frac{3}{2}$, which is impossible in our case.

To show that $d_{GH}(X, Y) = 0$, we can show that $d_{GH}(X, Y) < \eps$ for each $\eps > 0$. Let $\eps > 0$, we can construct an $\eps$-Gromov-Hausdorff approximation $f:X \to Y$ in the following way. Consider the set $J_{\eps}(\frac{3}{2}) = (\frac{3}{2} - \eps, \frac{3}{2} + \eps)\cap \Q \subseteq \Q\cap (1, 2)$, for each $q \notin J_{\eps}(\frac{3}{2})$, $f$ is the identity map from $I_q^X$ to $I_q^Y$. Since $J_\eps(\frac{3}{2})$ has countably many elements, we can list them as $q_0, q_1, q_2, \dots$ and in particular we take $q_0 = \frac{3}{2}$. Then for each $q_i$ we  have that $\abs{q_i - q_{i + 1}} < \eps$. We can map $I^X_{q_i}$ to $I^Y_{q_{i + 1}}$ bijectively by an affine map sending endpoints to endpoints. Since $\abs{q_i - q_{i + 1}} < \eps$ it is easy to see that any such affine map is an $\eps$-Gromov-Hausdorff approximation. Doing this   for each $q_i \in J_\eps(\frac{3}{2})$ an $\eps$-Gromov-Hausdorff approximation $f:X\to Y$. 

Since this works for every $\eps$ we can conclude that $d_{GH}(X, Y) = 0$.
\end{example}

We mentioned that in general Gromov-Hausdorff convergence is a rather weak convergence that is unable to preserve fine properties such as topology and curvature of the space. However, it becomes useful when we impose extra geometric restrictions on the metric spaces were are considering.  In this section, we will see the class of compact spaces with an upper bound on diameter and a lower curvature bound is pre-compact in Gromov-Hausdorff topology.

\begin{definition}[Inner/Intrinsic Metric Space]
	Let $(X, d)$ be a metric space, it is called \textbf{inner} or \textbf{intrinsic} if for any $x, y \in X$, 
	\begin{equation*}
		d(x, y) = \inf\br{\length_d(\gamma): \text{$\gamma$ connects $x$ and $y$}}.
	\end{equation*}
\end{definition}

\begin{definition}[Geodesic Metric Space]
An inner metric space is called \textbf{geodesic} if for any $x, y \in X$ in the above definition the infimum is a minimum i.e. $d(x, y)$ is realized by the length of a geodesic connecting $x$ and $y$.
\end{definition}

\begin{example}
Inner metric space is not necessarily geodesic. For example, for $x = (1, 0)$ and $y = (-1, 0)$ in $\R^2\backslash\br{(0, 0)}$ (This space is not complete). The distance $d(x, y)$ cannot be realized by any geodesic connecting $x$ and $y$. However, the infimum exists and equals $2$. The space $X$ in Example \ref{ex-complete-GH=0} is complete, inner but not geodesic.
\end{example}
Given any metric space $(M, d)$, we can induce an inner metric space $(M, d^{inn})$. Since we can define 
\begin{equation*}
	d^{inn}(x, y) = \inf\br{\length_d(\gamma): \text{$\gamma$ connects $x$ and $y$}}. 
\end{equation*}

Note however that $d^{inn}(x, y) $ might be equal to $+\infty$.

For example,  we can define the inner distance on $S^n \subseteq \R^{n + 1}$ which is realized by the length of the geodesics lie in $S^n$ that connects points in $S^n$. 

\begin{example}
	Any Riemannian manifold is certainly an inner metric space. If it is complete, it is geodesic. This is followed by the Hopf-Rinow Theorem in Riemannian geometry.
\end{example}


\begin{proposition}\label{prop:mid-point geodesic}
Let $(X, d)$ be a complete metric space. Then it is geodesic if and only if it admits the \textbf{mid-point property}, i.e. for any $x, y \in X$, there exists an $m \in X$ such that
\begin{equation*}
    d(x, m) = d(m, y) = \frac{1}{2}d(x, y).
\end{equation*}
\end{proposition}
\begin{proof}[Proof of Proposition~\ref{prop:mid-point geodesic}]
    \begin{itemize}
        \item ``$\implies$'': This is trivial. If $X$ is geodesic, then we can just pick the mid-point of the geodesic $[xy]$.
        \item ``$\impliedby$'': Conversely, we want to use the mid-point property and the completeness assumption to produce a geodesics between $x$ and $y$ in $X$. If the mid-point $m$ exists for any two points in $X$,  we can keep picking the mid-points again from the pairs $x, m$ and $m, y$.
        
        \begin{figure}[htbp]
        \centering
        \includegraphics[width=0.4\textwidth]{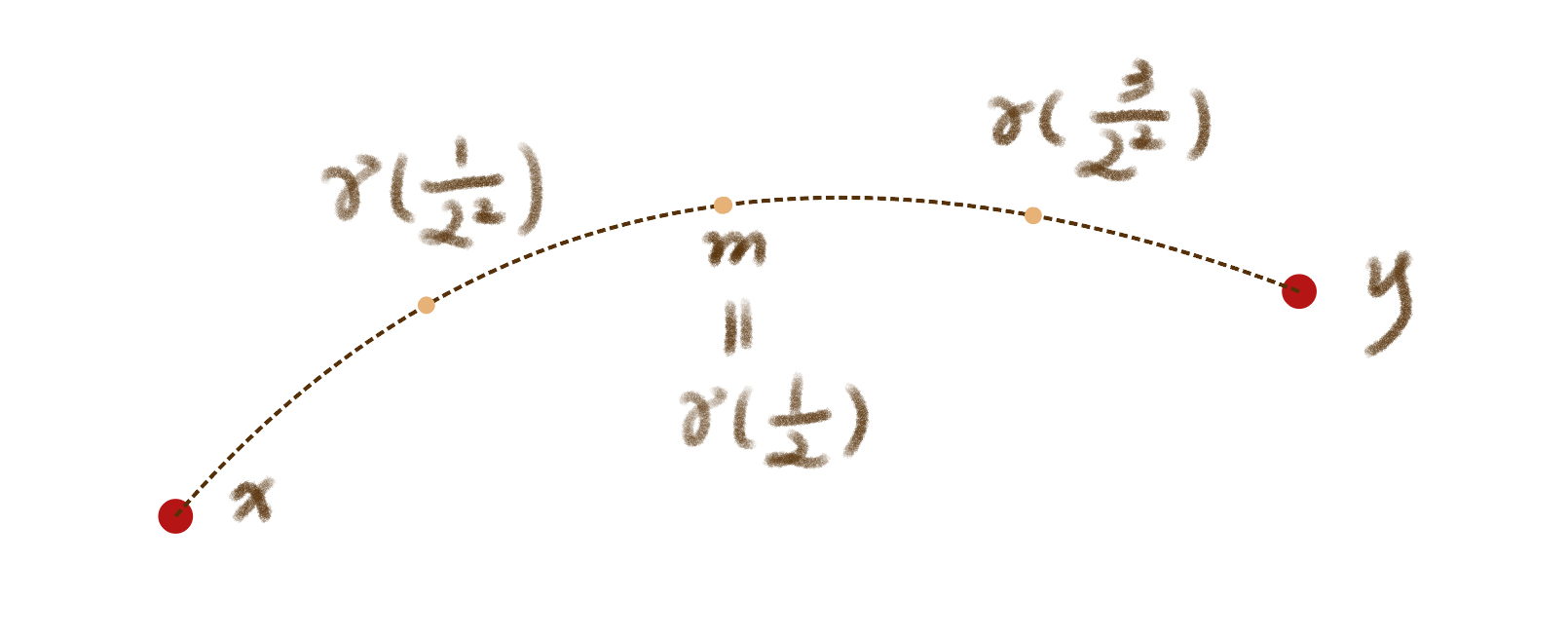}
        \end{figure}
        
        Keep this procedure, we can define geodesics for diadic rational $t$, i.e. get the position of the geodesics $\gamma$ at $\frac{m}{2^n}$ a prior for each $0 \leq \frac{m}{2^n} \leq 1$. This ensures us to define the geodesic on this dense set, say $S$ of numbers of $[0, 1]$. By the denseness of this set $S$ we can sub-sequentially approximate any $t \in [0, 1]$ using the sequence of these diadic times, i.e. 
        \begin{equation*}
        	t = 
        	\lim_{i \to \infty}\frac{m_i}{2^{n_i}}
        \end{equation*}
        And because $X$ is complete, the Cauchy sequence
        \begin{equation*}
        	d(\gamma(\frac{m_{i_1}}{2^{n_{i_1}}}), \gamma(\frac{m_{i_2}}{2^{n_{i_2}}})) = \abs{\frac{m_{i_1}}{2^{n_{i_1}}} - \frac{m_{i_2}}{2^{n_{i_2}}}}\to 0
        \end{equation*}
        ensures the existence of $z \in X$ such that $\gamma(t) = z$. Hence we extend $\gamma$ to entire $[0, 1]$. The resulting extension is easily seen to be a geodesic connecting $x$ to $y$.
    \end{itemize}
\end{proof}
\begin{remark}
	 Completeness is necessary to define $\gamma$ for all $t\in [0.1]$, not just for diadic rational $t$. For example, $\Q$ has the midpoint property since for any $x,y\in \Q$ the point $\frac{x+y}{2}$ is rational too.
	 
	 But there are no geodesics between any two distinct points (because the cardinality from $[0, 1]$ to an interval of rational numbers is different).
\end{remark}
\begin{proposition}\label{prop: gh-con implies mid-point convergence}
   	Suppose $X_n \ghto X$. Then if all $X_n$ have the mid-point property then so does $X$. Therefore if all $X_n$ are geodesic then so is $X$.
    
\end{proposition}

\begin{proof}[Proof of proposition~\ref{prop: gh-con implies mid-point convergence}]
   Let $x, y$ be points in $X$. we need to prove that there is a midpoint between them.
   
    By the corollary \ref{cor:GH-con iff GH-approx}, we can find $\eps_n \downarrow 0$ and $\eps_n$-Gromov-Hausdorff convergence $f_n: X_n \to X$. Therefore, for any $x_n, y_n \in X_n$, $d^X(f(x_n), x) \leq \eps_n$ and $d^X(f(x_n), y) \leq \eps_n$. Now we can pick mid-point $m_n$ between $x_n, y_n$. And it is easy to show $f_n(m_n)$ sub-converges to a mid-point between $x$ and $y$.
    \end{proof}
    
\subsection{The Pre-compactness Theorem}

\begin{definition}[Pre-compactness]
Given $(X, d)$ a metric space, a subspace $(A, d)$ is called \textbf{pre-compact} if its closure $\overline{A}$ is compact. Equivalently, every sequence $\br{x_i} \subseteq A$ sub-converge to some limit point $x \in \overline{A}$.
\end{definition}

\begin{definition}[Totally Bounded]
	Recall that a metric space $(X, d)$ is \textbf{totally bounded} if, for every $\varepsilon > 0$, it admits a finite $\varepsilon$-net. 
\end{definition}
We consider the following special family of metric spaces. Then, we are going to show this family of compact metric spaces is pre-compact under the Gromov-Hausdorff convergence. 

\begin{definition}[Uniformly Totally Bounded]
Let $D > 0$, a family of compact metric space is called \textbf{uniformly totally bounded by $D$}, which is denoted as $M_c(D, N)$, if it is
\begin{itemize}
	\item \textbf{Uniformly bounded diameter:} $\diam(X) \leq D$ for each $(X, d) \in M_c(D, N)$;
	\item \textbf{Uniformly totally bounded:} and for each $\varepsilon > 0$, there is a finite number $N(\varepsilon) > 0$ such that any $X \in M_c(D, N)$ admits an $\varepsilon$-net with carnality less than or equal to $N(\varepsilon)$. 
\end{itemize}
where $N: \R^+ \to \R^+$ is called a covering function.
\end{definition}

\begin{theorem}[Precompcact Theorem of the Gromov-Hausdorff Convergence]\label{thm: precompact gh-con}
    Let $N: \R^+ \to \R^+$ be a non-increasing covering function. Let $D > 0$. Then $M_c(D, N)$ is pre-compact in the Gromov-Hausdorff topology. Namely, For any $\br{X_n} \subseteq M$, there exists a convergent subsequence $X_{n_k} \ghto X$ for some $X \in M_c(D, N)$ and $X$ admits the same covering function $N$. 
\end{theorem}

\begin{proof}[Sketch of the Proof of Theorem \ref{thm: precompact gh-con}]
    Fix an integer $k > 1$ and let $\br{X_n} \subseteq M_c(D, N)$ be a sequence. By the above, for any $m \leq k$, we have find a $\frac{1}{m}$-net in $X_n$ with $\leq N(\frac{1}{m})$ elements. 

    Take a union of these for $m = 1, 2, \dots, k$, which has $N_k := \sum_{n = 1}^k N(\frac{1}{n})$ elements in $X_n$, say 
    \begin{equation*}
        X_{n, k} := \br{x_1^n, \dots, x_{N_k}^n} \subseteq X_n.
    \end{equation*} 

    Therefore, whenever we fix $k$, we can have fixed number of points whose union is a finite metric space $X_{n, k} \subseteq X_n$ (finite subset of $X_n$). Recall that $\diam(X_n) \leq D$, so all the pairwise distance for the points in $X_{n, k}$ are $\leq D$, i.e. $d^{X_n}(x_i^n, x_j^n) \leq D < \infty$ for any $i, j \leq N_k$. Then by compactness of $[0, D]$ we can conclude $\br{d^{X_n}(x_i^n, x_j^n)}_{n = 1}^\infty$ subsequentially converge to a number $d_{ij} \in [0, D]$ for any $i, j \leq N_k$. This $d_{ij}$ is a limit metric on the set of finite points with cardinality less than  or equal to $N_k$.

    Since
    \begin{equation*}
        X_{n,k} \subseteq X_{n, k + 1} \subseteq X_{n, k + 2} \subseteq \cdots,
    \end{equation*}
    by the diagonal argument, we can get a metric on a countable set $S$ by passing to a sub-sequence. Here 
    \begin{equation*}
        S_1 \subseteq S_2 \subseteq S_3 \subseteq \cdots \subseteq S
    \end{equation*}
    and $S_i$ has $N_i$ elements for each $i \in \N$. Then we take the metric completion of $S$ say $X = \overline{S}$ and check that
    \begin{equation*}
        X_n \ghto X, \quad n \to \infty
    \end{equation*}
\end{proof}

Recall that if we fix $n, \kappa, D$, and consider the following family of closed Riemannian manifolds:
\begin{equation*}
    M_{ric}(n, \kappa, D) := \br{\text{$(M^n, g)$: $\ric_M \geq (n - 1)\kappa$ and $\diam(M) \leq D$}},
\end{equation*}
then by the Bishop-Gromov volume comparison theorem, for any $\eps>0$  there exists $N(\eps) = N(\eps, n, \kappa, D)$ such that for any $M \in M(n, \kappa, D)$, $M$ admits an $\eps$-net with $\leq N(\eps)$ elements. And thus we have a natural corollary by the theorem \ref{thm: precompact gh-con}. 
\begin{corollary}\label{cor: ric-precompact}
    $M_{ric}(n, \kappa, D)$ is precompact in Gromov-Hausdorff topology. Being pre-compact means every sequence 
    \begin{equation*}
        \br{(X_i, d)} = \br{(M_i^n, g)} \subseteq M_{ric}(n, \kappa, D) 
    \end{equation*}
    admits a convergent sub-sequence in  Gromov-Hausdorff topology.  But the limit may not be smooth or even a manifold.
\end{corollary}

\begin{fact}
If $A\subset (X,d)$ is precompact then any $B\subset A$ is precompact too.
\end{fact}

An easy corollary from this fact is:
\begin{corollary}\label{cor: sec-precompact}
    The following family of closed Riemannian manifolds
    \begin{equation*}
        M_{sec}(n, \kappa, D) := \br{(M^n, g): \diam(M) \leq D, \sect_M \geq \kappa}
    \end{equation*}
    is also precompact in the Gromov-Hausdorff topology. 
    \begin{proof}
	Because $\sect_M \geq \kappa$ implies $\ric_M \geq (n - 1)\kappa$. 
\end{proof}
\end{corollary}


\chapter{Alexandrov Spaces}\label{Ch13}

\section{Introduction}
Recall that precompact means every sequence has a convergent subsequence. And it is natural to ask what the limit looks like, of sub-sequences in Corollary~\ref{cor: sec-precompact}. To answer this question, we need to study Alexandrov geometry. 
\begin{definition}[Alexandrov Space]\label{def: alex-space}
    Let $(X, d)$ be a complete geodesic metric space is called \textbf{Alexandrov of curvature $\geq \kappa$} if the following four-point condition in the Toponogov comparison holds. Namely, for any four points $p, a, b, c \in X$, we have 
    \begin{equation*}
        \modangle^\kappa(p_a^b) + \modangle^\kappa(p_a^c) + \modangle^\kappa(p_b^c) \leq 2\pi
    \end{equation*}
    if such comparison triangles exist.
\end{definition}

Therefore, by the Toponogov comparison, with this definition, every complete Riemannian manifold $(M^n, g)$ with $\sect_M \geq \kappa$ is an Alexandrov space of curvature $\geq \kappa$. 

\begin{remark}
	More generally, one can consider inner metric instead of geodesic Alexandrov spaces. But in finite dimensions the two notions are equivalent.
\end{remark}

Alexandrov's condition is preserved under the Gromov-Hausdorff convergence. 
\begin{proposition}\label{prop: alex-space and GH-convergence}
    Suppose $X_n \ghto X$ where $X_n$ are Alexandrov spaces of curvature $\geq \kappa$, then $X$ is also an Alexandrov space of curvature $\geq \kappa$. 
\end{proposition}
\begin{proof}[Proof of Proposition~\ref{prop: alex-space and GH-convergence}]
    Firstly, we just pick $p, a, b, c \in X$. Since $X_n \ghto X$, then by the corollary \ref{cor:GH-con iff GH-approx}, we have a sequence $\eps_n \downarrow 0$ and the corresponding $\eps_n$-approximation $f_n$ for each $n$. Fix $n$, since $f_n$ is $\eps_n$ dense in $X$ for each $n$, we can pick the corresponding $p_n, a_n, b_n, c_n$ in $X_n$ such that
\begin{align*}
   	& d^X(f_n(p_n), p) \leq \eps_n, \quad d^X(f_n(a_n), a) \leq \eps_n,\\
   	& d^X(f_n(b_n), b) \leq \eps_n, \quad d^X(f_n(c_n), c) \leq \eps_n.
\end{align*}
Moreover, since $(X^n, d^{X_n})$ are Alexandrov of curvature $\geq \kappa$, the four-points comparison holds for each group of $\br{p_n, a_n, b_n, c_n}$, say
\begin{equation*}
    \modangle^\kappa({p_n}_{a_n}^{b_n}) + \modangle^\kappa({p_n}_{c_n}^{b_n}) + \modangle^\kappa({p_n}_{a_n}^{c_n}) \leq 2\pi
\end{equation*}
Since the Gromov-Hausdorff convergence implies the convergence of the geodesics \ref{prop:mid-point geodesic}, say
\begin{equation*}
    \abs{a_n p_n} \to \abs{ap}, \quad \abs{a_n b_n} \to \abs{ab}, \cdots, etc
\end{equation*}
Then, by the cosine law, the sequence of angles in the model space $S_\kappa^n$ will certainly converges.
\begin{figure}[htbp]
    \centering
    \includegraphics[width=0.7\textwidth]{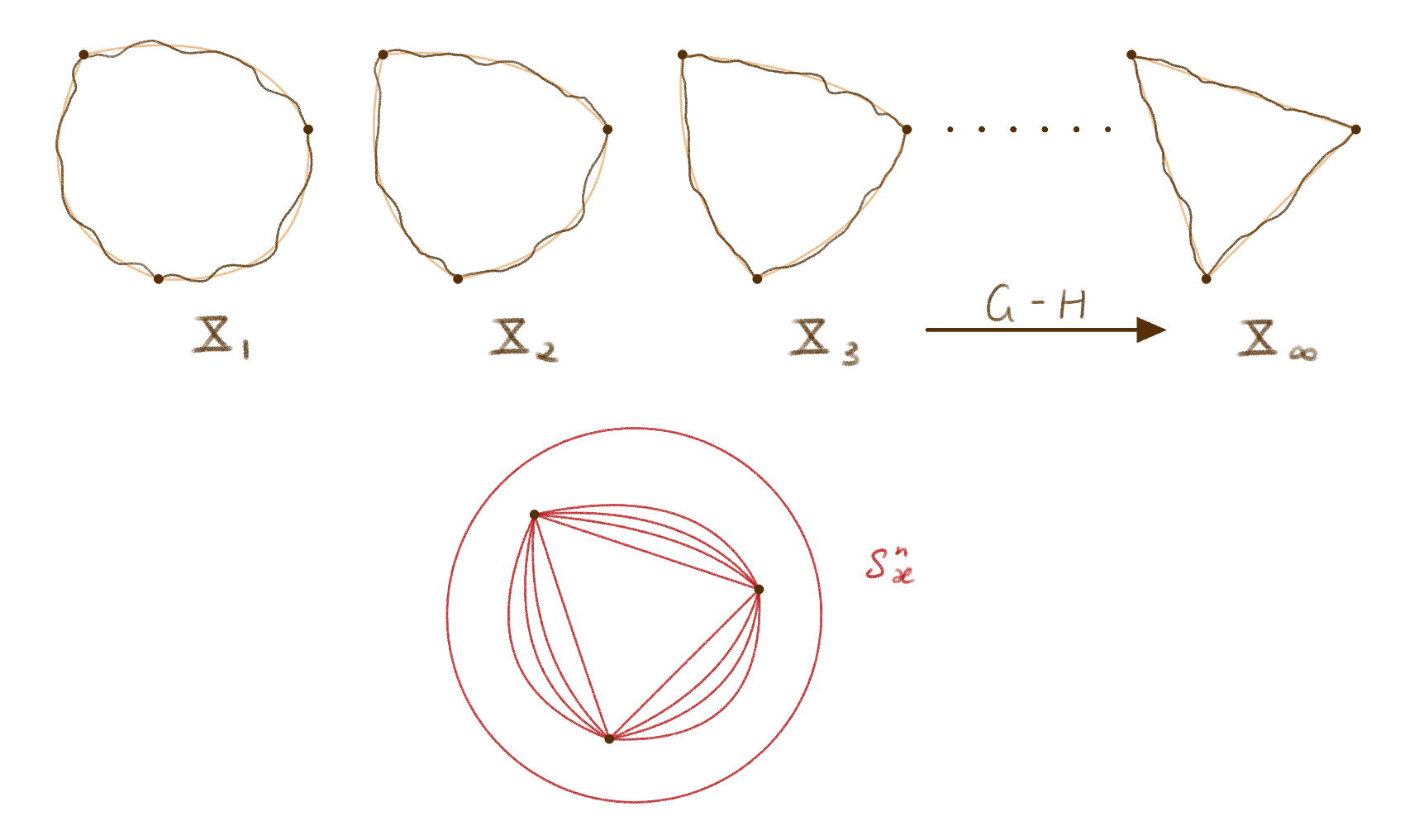}\caption{ Convergence of angles in the model space $S_\kappa^n$ (Proposition \ref{prop: alex-space and GH-convergence})}
\end{figure}

Therefore, the convergence of angles implies
\begin{equation*}
    \modangle^\kappa({p}_{a}^{b}) + \modangle^\kappa({p}_{c}^{b}) + \modangle^\kappa({p}_{a}^{c}) \leq 2\pi
\end{equation*}
Then we can conclude that $X$ is also an Alexandrov space of curvature $\geq \kappa$.
\end{proof}
Now we can answer our question rises from the corollary \ref{cor: sec-precompact}, the sub-sequential limits of the manifolds with $\sect_{M_n} \geq \kappa$ are Alexandrov spaces of curvature $\geq \kappa$.

\section{Examples}
So far we have introduced the notion and the motivition of Alexandrov spaces. In words, the study of Alexandrov geometry is the study of the Toponogov comparison theorem in the context of geodesic metric spaces. Alexandrov spaces have many properties similar to manifolds. We are going to list some of them. 
\begin{example}\label{ex:alex-mod by lie group}
    Let $(M^n, g)$ be compact manifolds of $\sect_M \geq \kappa$ where $\kappa \leq 0$. 
    Consider a connected compact Lie group acting by isometries $G\acts M$. 
    As every compact Lie group  $G$ has a bi-invariant Riemannian metric $h$ so that $\sect_G \geq 0$. 
    Take any $\eps > 0$, take $(M\times \eps G)$ where $\eps G := (G, \eps^2 h)$. The space $(M\times \eps G)$ still has $\sect_{(M\times \eps G)} \geq \kappa$. 
    And $G\acts (M\times \eps G)$ diagonally as isometries. 
    This action is free and isometric. 
    Take $M_\eps = (M \times \eps G)/G$, which is diffeomorphic to $M$ with different metrics, and orbits of $G$ are shrunk. 
    We have a Riemannian submersion $M\times \eps G\to M_\eps$. 
    By theorem~\ref{thm: Riem-submersion} it holds that  $\sect_{M_\eps}\ge \kappa$ for any $\eps > 0$. In particular, it is an Alexandrov space of $\curv\ge \kappa$.
    
    Finally, 
    \begin{equation*}
        M_\eps \ghto M/G, \quad \eps \to 0. 
    \end{equation*}
    Therefore, we can show that $M/G$ is an Alexandrov space of curvature $\geq \kappa$. 
\end{example}
The construction in the example \ref{ex:alex-mod by lie group} allows us to produce Alexandrov spaces that are not manifolds. Below is a more explicit example of such non-manifold Alexandrov space. 
\begin{example}
    Notice that $\sect_{\R^{2n + 2}}\geq 0$. So $\R^{2n + 2}$ is an Alexandrov space of curvature $\geq 0$. Consider $S^1 \acts \C^{n + 1} \cong \R^{2n + 2}$, which is an isometric action
    \begin{equation*}
        z(z_1, \dots, z_{n + 1}) = (zz_1, \dots, zz_{n + 1}). 
    \end{equation*}
    By the last example \ref{ex:alex-mod by lie group}, we know that $\R^{2n + 2}/S^1$ is an Alexandrov space of curvature $\geq 0$. Indeed $\R^{2n + 2}/S^1 \cong C(\C P)^n$ the cone of $\C P^n$ which is not a manifold. Intuitively, $\C P^n \cong S^{2n + 1}/S^1$ and the cone structure comes from $C(S^{2n + 1})\cong \R^{2n + 2}$.  And this example is non-compact.
\end{example}
\begin{example}\label{ex: convolution-con}
    Let $(M^n, g)$ be complete with $\sect_M \geq \kappa$. Suppose $f: M \to \R$ is convex. Then the level set $\br{f = c}$ with respect to the intrinsic metric is an Alexandrov space of curvature $\geq \kappa$. 
    For a smooth $f$ this follows by the Gauss formula which implies that $\{f=c\}$ is a smooth manifold of $\sect\ge \kappa$. The general case follows by approximation.

In particular, for any convex body $C \subseteq \R^n$, its boundary $\partial C$ is an Alexandrov space of curvature $\geq 0$. In particular,  the boundaries of the convex body having infinite many corners are also  Alexandrov spaces of $\curv\ge 0$.
. To see why this is true, there exits $f: \R^n \to \R$ a convex function such that $C=\{f\le c\}$. The function $f$ can be smoothed out by taking convolution with smooth kernels. This gives $C^\infty \ni f_\eps \to f$ point-wise and all $f_\eps$ are smooth and convex. We claim that $\br{f_\eps = c} \ghto \br{f = c}$ in intrinsic metrics. Therefore, $\br{f = c}$ is also an Alexandrov space of curvature $\geq 0$.

    \begin{figure}[htbp]
        \centering
        \includegraphics[width=0.8\textwidth]{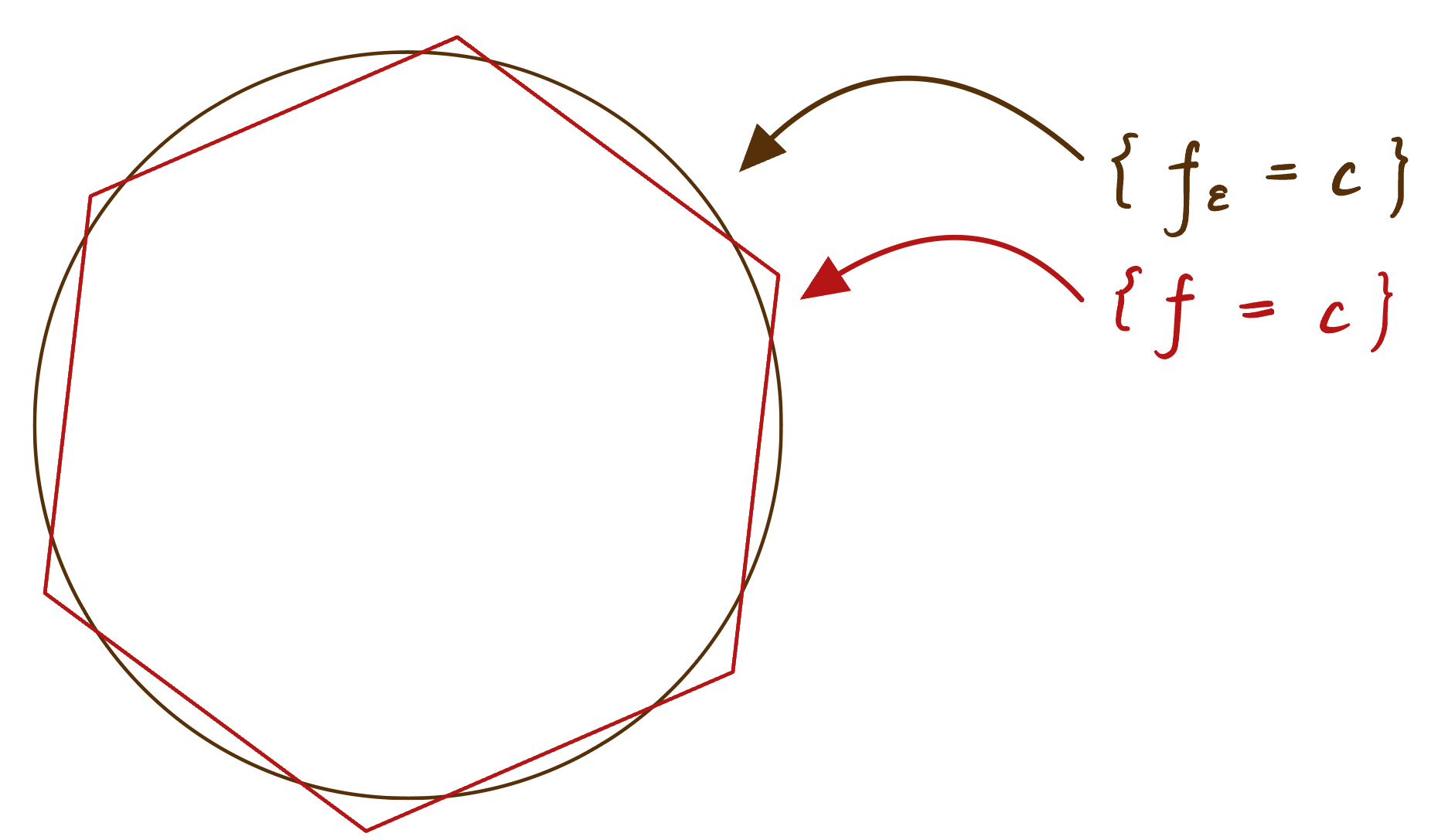}\caption{Point-wise convergence implies $\mathbf{G-H}$-convergence in intrinsic metric  (Example \ref{ex: convolution-con})}
    \end{figure}

    To show this, notice that $f_\eps \to f$-point-wise implies $\br{f_\eps = c} \to \br{f = c}$ in the Hausdorff sense as subsets of $\R^n$. However showing convergence in intrinsic metrics is less obvious, even in low dimensions. 
    For example, for convex function $f$ on $\R$ and smooth convex functions $f_\eps$ point-wise converge to $f$. Then over the closed interval $[a, b]$, 
    \begin{equation*}
        \length_{\R^2}(f_\epsilon([a, b])) \to \length_{\R^2}(f([a, b]))
    \end{equation*}
    is not an even trivial convergence. 

    \begin{figure}[htbp]
    \centering
    \includegraphics[width=0.4\textwidth]{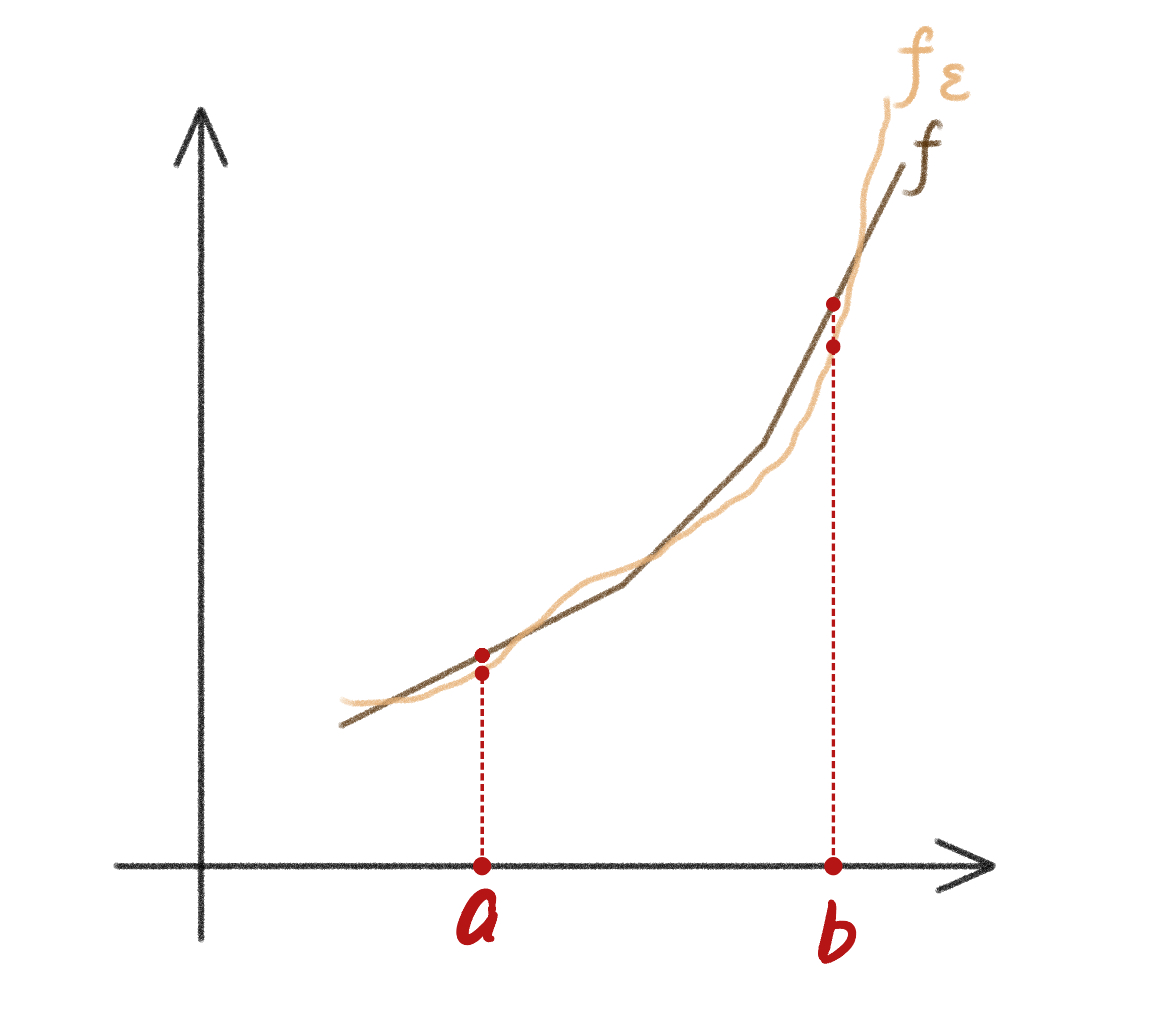}\caption{Convergence in length}
    \end{figure}
    
    To show this convergence in intrinsic metric is true, we have to use Sharafutdinov retraction. Since $\sect_{\R^n} \geq 0$,
    
    we have the Sharafutdinov retraction $\Gamma$, which is a $1$-Lipschitz retraction from $\Gamma: \R^n \to \br{f \leq c}$. 
    
    We consider the restriction $\Gamma|_{\br{f_\eps = c + \delta}}: \br{f_\eps = c + \delta} \to \br{f = c}$ which is also a surjective $1$-Lipschitz map. And similarly we get another Sharafutdinov retraction $\Phi: \R^n \to \br{f_\eps \leq c - \delta}$ which is also surjective $1$-Lipschitz and its restriction 
    ${\Phi|_{\br{f = c}} : \br{f = c} \to \br{f = c - \delta}}$
    is also surjective $1$-Lipschitz.
    Since $f$ and $f_\eps$ can be chosen to be norms on $\R^n$, different level sets are homeomorphic to each other.
    The end result is that we have
    \begin{align*}
        &\Gamma_\eps: \br{f_\eps = c + \delta} \to \br{f = c}\quad\text{is surjective $1$-Lipschitz};\\
        &\Phi_\eps: \br{f = c} \to \br{f_\eps = c + \delta} \quad\text{is surjective $(1 + o(\epsilon))$-Lipschitz.}
    \end{align*}
    Because they are all proportional to each other, we can just write
    \begin{align*}
        &\Gamma_\varepsilon: \br{f_\varepsilon = c} \to \br{f = c}\quad\text{is surjective $1$-Lipschitz};\\
        &\Phi_\varepsilon: \br{f = c} \to \br{f_\varepsilon = c} \quad\text{is surjective $(1 + o(\varepsilon))$-Lipschitz.}
    \end{align*}
    Roughly speaking, we have the almost $1$-Lipschitz functions between $\br{f_\varepsilon = c}$ and $\br{f = c}$ for both directions. 
    
    \begin{figure}[htbp]
        \centering
        \includegraphics[width=0.7\textwidth]{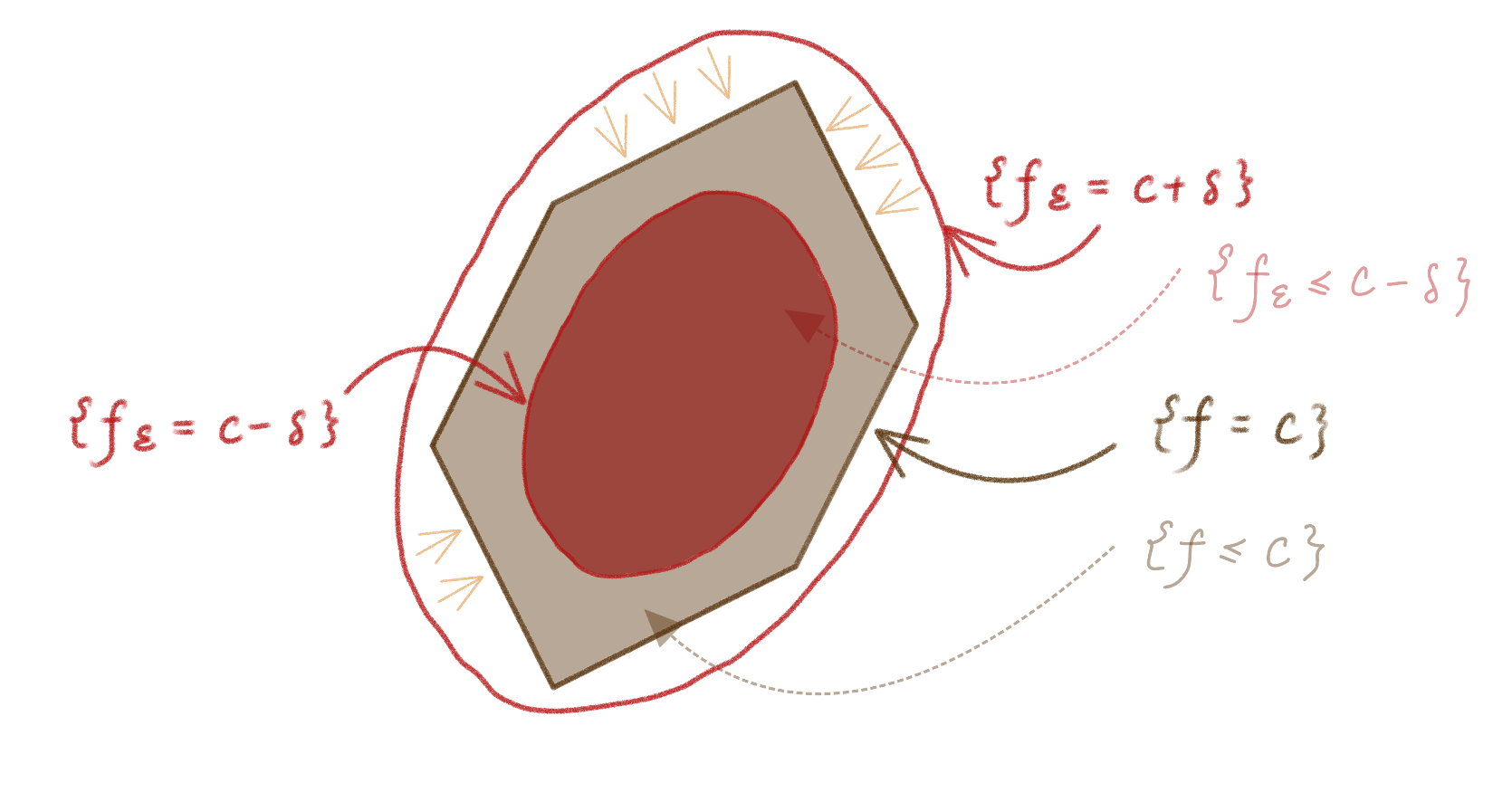}\caption{Construction of the Sharafutdinov retractions and the restrictions}
    \end{figure}
    
    Take the composition of the two maps,
    \begin{equation*}
        \Gamma_\varepsilon\circ \Phi_\varepsilon: \br{f = c} \to \br{f = c}
    \end{equation*}
    is surjective $(1 + o(\varepsilon))$-Lipschitz. And as $\varepsilon \to 0$, the composition will converge to some map $F: \br{f = c} \to \br{f = c}$ which is also surjective $1$-Lipschitz. Then by the following lemma (The proof is left as an exercise, which can also be found in \cite{BBI01}) 
    \begin{lemma}
        Let $(X, d)$ be a compact metric space. Then if $f: X \to X$ is surjective $1$-Lipschitz, then $f$ is an isometry. 
    \end{lemma}
   	By passing to a limit as $\eps\to 0$ we can conclude that $\Phi_\varepsilon$ and $\Gamma_\varepsilon$ are  $o(1)$-Gromov-Hausdorff approximations.
   \end{example}
\section{Pointed-Gromov-Hausdorff Convergence}
Let $(X_n, p_n)$, $(X, p)$ be proper metric spaces, i.e. closed balls are compact. Here $p_n$ and $p$ are points in the space. For example,  complete Riemannian manifolds are proper metric spaces. 

\begin{figure}[htbp]
        \centering
        \includegraphics[width=0.6\textwidth]{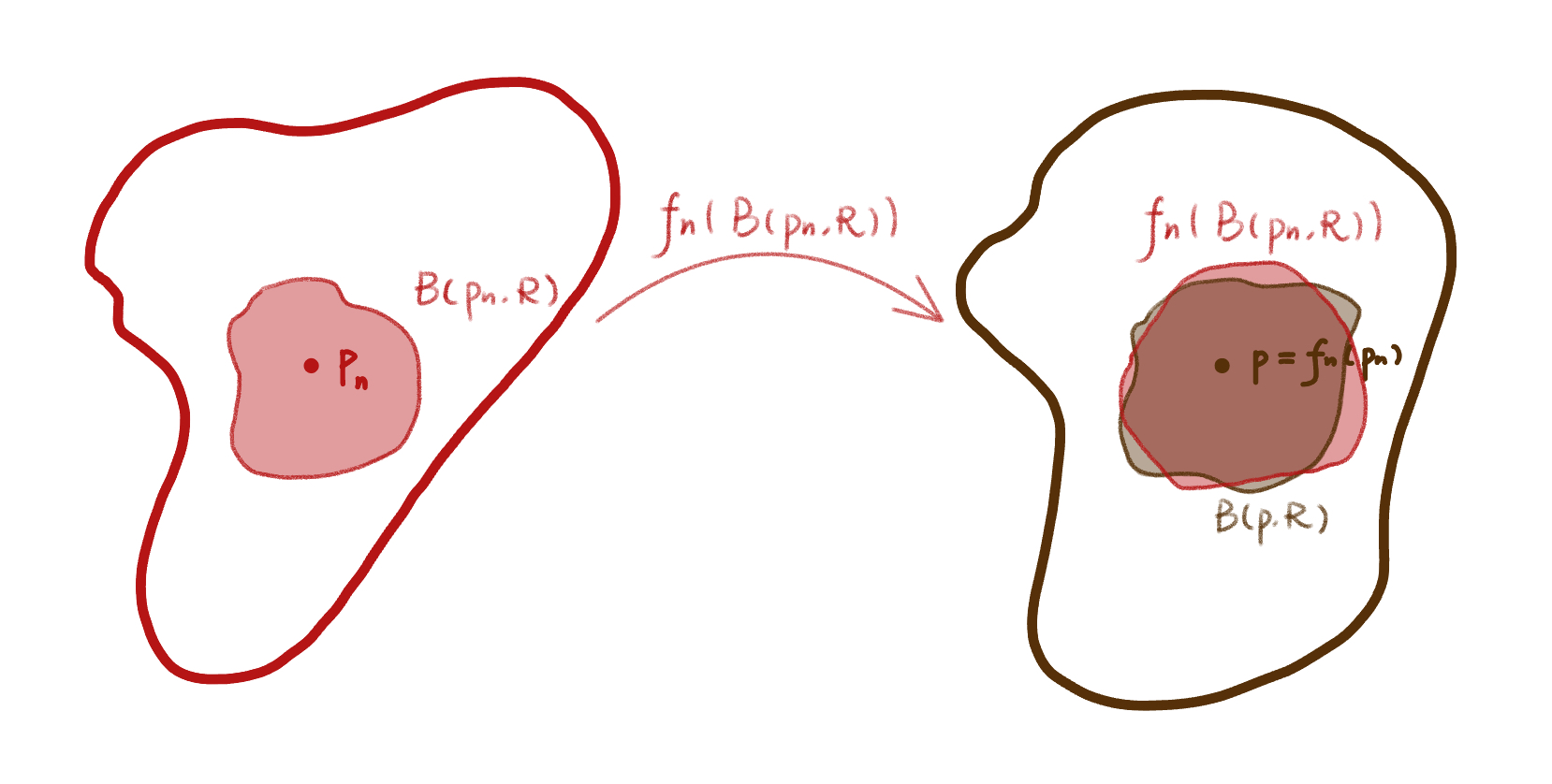}\caption{The image $f_n(B(p_n, R))$ is Hausdorff close to but not equal to $B(p, R)$ (Definition \ref{def: pgh-con})}\label{fig: pgh-con}
    \end{figure}
\begin{definition}[Pointed-Gromov-Hausdorff Convergence]\label{def: pgh-con}
    Let $(X_n, p_n)$, $(X, p)$ be proper metric spaces. The convergence
    \begin{equation*}
        (X_n, p_n) \pghto (X, p)
    \end{equation*}
    is called a \textbf{pointed-Gromov-Hausdorff convergence} if $\exists f_n: X_n \to X$ such that $f_n(p_n) = p$ and $\forall R > 0$ fixed, $f_n|_{B(p_n, R)}$ is an $\varepsilon_n(R)$-Gromov-Hausdorff approximation onto $B(p, R)$ where $\varepsilon_n(R) \downarrow 0$ as $n \to \infty$. 

    In the definition, $f_n(B(p_n, R))$ is $\varepsilon_n(R)$-Hausdorff closed to $B(p, R)$ and for every $x, y \in B(p_n, R)$, 
    \begin{equation*}
        \abs{d^X(f_n(x), f_n(y)) - d^X(x, y)} \leq \varepsilon_n(R).
    \end{equation*}
    (See figure \ref{fig: pgh-con} for visualization.)
\end{definition}

\begin{example}\label{ex: pointed-gh-con tangent space}
    If $(M^n, g)$ is a Riemannian manifold. Let $p \in M^n$ and $\lambda_i \to \infty$ a sequence. Then 
    \begin{equation*}
        (\lambda_i M^n, p_i) \pghto (T_pM, 0)
    \end{equation*}
    where $p_i$ is the point corresponding to $p$ in $\lambda_i M$. And $(\lambda_i M^n)$ means the rescaled manifolds of $M^n$ by factor $\lambda_i$. Thus any ball of radius $\frac{1}{\lambda_i}$ in $M^n$ becomes a unit ball in $\lambda_i M^n$. Finally, the nontrivial Riemannian metric $g$ will become flat at the limit of the pointed Gromov-Hausdorff convergence.
\end{example}

The phenomenon in this example \ref{ex: pointed-gh-con tangent space} can be generalized to Alexandrov spaces with a lower curvature bound. The generalization is constructed in the next lecture. 

Let $X$ be an Alexandrov space of curvature $\geq \kappa$. Moreover, we assume $X$ has a finite Hausdorff dimension. 
\begin{note}
    Remember that $\haus^n$ is denoted as the $n$-dimensional Hausdorff measure. And for some space $X$
 \begin{equation*}
        \dim_{\haus}(X) < \infty \implies \dim_{\haus}(X) = n
    \end{equation*}
    for some integer $n$. 
    Furthermore $n=\dim_{top}X$, the topological dimension of $X$. We will refer to $n$ as \emph{dimension} of $X$ and will denote it by $\dim X$. 
    
    And $\haus^n$ behaves like volume on manifolds. Here are some properties of the Hausdorff measure and the Hausdorff dimensions we would like to mention:
    \begin{itemize}
        \item Moreover, $\haus^n$ is continuous with respect to the Gromov-Hausdorff topology. Namely, suppose we have a sequence of Alexandrov spaces $\br{X_i^n}_{i \in \N}$ of curvature $\geq \kappa$ and $\dim=n$. If
        \begin{equation*}
            X_i^n \ghto X, \quad i \to \infty
        \end{equation*}
        to some Alexandrov space $X$ of curvature $\geq \kappa$. Then we can conclude that
        \begin{itemize}
            \item $\dim_{\haus}(X) \leq n$; 
            \item and if $\dim_{\haus}(X) < n$, then 
            \begin{equation*}
                \haus^n(X_i^n) \to 0 = \haus^n(X).
            \end{equation*}
        \end{itemize}
        \item The Bishop-Gromov comparison holds for $\haus^n$. Proving this is similar to the Bishop-Gromov comparison for $\sect \geq \kappa$. This theorem is generalized to the Alexandrov spaces. 
    \end{itemize}
\end{note}


\section{Properties}
We can summarize those properties as the followings. Fix $X^n$, an $n$-dimensional Alexandrov space of curvature $\geq \kappa$. 
\begin{itemize}
    \item We can define angles in $X$ using the comparison angle in the model space $S_\kappa^n$. Namely, for two geodesics $\gamma_1, \gamma_2$ starting at $p \in X$ we can define the angle between $\gamma_1$ and $\gamma_2$ at $p$ as the limits of the angle in the model space:
    \begin{equation*}
        \lim_{t, s \to 0}\modangle^\kappa\brac{p_{\gamma_1(t)}^{\gamma_2(s)}};
    \end{equation*}
    \item Other versions of Toponogov comparison theorem hold;
    \item Globalization theorem holds: if Toponogov holds locally then it holds globally.
    \item We can define tangent space for Alexandrov space of curvature $\geq \kappa$. More detailed construction of the tangent space can be found in the section \ref{sec: tangent space for alex space}.
    \item If $\dim(X)< \infty$, then $\dim_{top}X=\dim_{\haus}(X) = n$  is an integer and the $n$-th Hausdorff measure behaves like the volume on manifolds. In particular, the Bishop-Gromov volume comparison holds. Absolute volume comparison holds as well.
    \item $\haus^n$ the volume function on Alexandrov space is continuous concerning the Gromov-Hausdorff convergence. Let $\br{X_i^n}$ be a sequence of $n$-dimensional Alexandrov space of curvature $\geq \kappa$ with $\diam(X_i^n) \leq D$. Then the Gromov-Hausdorff limit of the sequence is again an Alexandrov space of curvature $\geq \kappa$,  $\diam(X) \leq D$, $\dim_{\haus^n}(X) \leq n$. And 
    \begin{equation*}
        \haus^n(X_i^n) \to \haus^n(X), \quad\text{as $i \to \infty$}
    \end{equation*}
    and in particular, if $\dim_{\haus}(X) < n$, the sequence of volumes  converges to $0$. 
    \begin{example}[Example of Volume Collapsing]\label{ex: collapsing volume}
        Notice that $\sect_{S^3} \equiv 1$. Consider the Hopf action $S^1 \acts S^3$: Since we can consider $S^2 \subseteq \C^2$. And $z = z_1 + i z_2 = (z_1, z_2) \in S^1$ as unit complex number. So that for we can write any complex number as $\lambda(z_1, z_2) = (\lambda z_1, \lambda z_2)$. This action $S^1 \acts S^3$ is free and isometric.
        Now we can construct our volume collapsing example. Let $S^1_\eps$ be the circle of length $\eps$. Consider $(S^3 \times S_\eps^1)/S^1 = (S^3, g_\eps) = (S_\eps^3)$, for each $\eps$, this space has nonnegative sectional curvature. Let $\eps \to 0$, then
        \begin{equation*}
            S_\eps^3 \ghto S^2
        \end{equation*}
        because $S^3/S^1 = S^2$. Hence we have the volume collapsing. 
    \end{example}
    \item Let $X$ be an Alexandrov space of curvature $\geq \kappa$. Let $f: X \to \R$ be a semi-concave function. Fix $p \in X$, we can have the gradient $\nabla f_p$ defined in the same way as we did for manifolds. So the gradient flow exists and has the same contraction properties as manifolds for the same reason.
    \item The first variation formula holds for Alexandrov spaces. Let $X$ be an Alexandrov space of curvature $\geq \kappa$. Let $f: X \to \R$ be an $\lambda$-concave function. The its gradient flow $\Phi_t$ is $e^{\lambda t}$-Lipschitz. Toponogov condition implies there are many semi-concave functions in Alexandrov spaces. 
    \begin{example}
        Let $X$ be an Alexandrov space of curvature $\geq 0$. And $A \subseteq X$ a closed subset. Then $d^2(\cdot, A)$ is $2$-concave just like for Riemannian manifolds of nonnegative sectional curvature. 
    \end{example}
    \item The splitting theorem holds. The same proof works. Let $X$ be an Alexandrov space of curvature $\geq 0$. If $X$ has a line then 
    \begin{equation*}
        X \stackrel{isom}{\cong} Y \times \R
    \end{equation*}
    where $Y$ is also an Alexandrov space of $\curv\ge 0$.
\end{itemize}

Remember that we want to generalize the phenomenon of example \ref{ex: pointed-gh-con tangent space} of Riemannian manifolds to Alexandrov spaces with lower curvature bound. That motivates us to define the notion of tangent spaces for Alexandrov spaces. 
\section{Tangent Spaces and Cones}\label{sec: tangent space for alex space}
We can also generalize the concept of tangent space to Alexandrov spaces. Given $p \in X$ a point in an Alexandrov space. We can look at geodesic directions starting at $p$, say $S_p$, and take a metric completion denoted as $\Sigma_p = \overline{S}_p$ with respect to angle metric. We call $\Sigma_p$ the space of directions at $p$. It is an analogue of unit sphere in $T_pM$ if $M^n$ is a Riemannian manifold. Now the tangent space of $X$ is defined as $C(\Sigma_p)$ which is the Euclidean cone over the space $\Sigma_p$. 

Recall that the Euclidean cone is constructed as the following. If $(\Sigma, d_\Sigma)$ is a metric space then its Euclidean cone $C(\Sigma)$ is defined as $(\Sigma\times [0, \infty))/[(\Sigma \times \br{0}) \sim \br{pt}]$. 

\begin{figure}[htbp]
        \centering
        \includegraphics[width=0.8\textwidth]{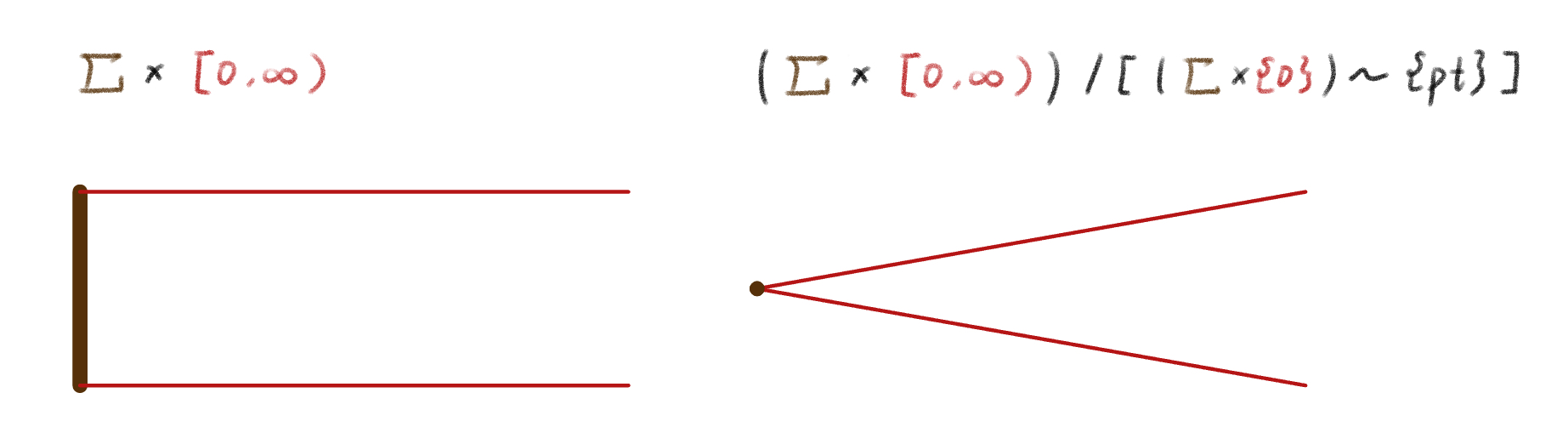}\caption{Euclidean cone for metric space $\Sigma$.}
\end{figure}

$C(\Sigma)$ is again a metric space as well and we can define the distance function $d^{C(\Sigma)}(\cdot, \cdot)$.

\begin{figure}[htbp]
        \centering
        \includegraphics[width=0.5\textwidth]{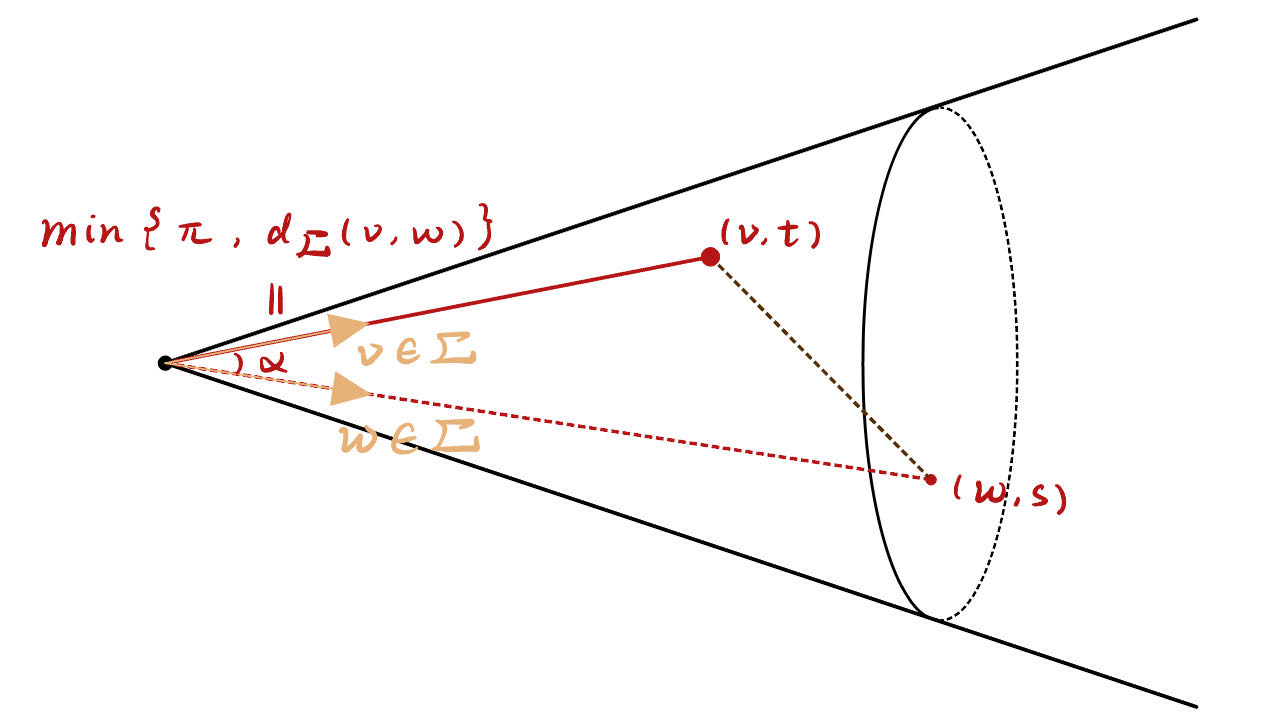}\caption{The construction of the distance function on $C(\Sigma)$.}
    \end{figure}

For each $v, w \in \Sigma$ and $t, s \geq 0$, 
\begin{equation*}
    [d^{C(\Sigma)}((v, t), (w, s))]^2 = t^2 + s^2 - 2ts\cdot \cos{\alpha}
\end{equation*}
where $\alpha = \min\br{\pi, d_\Sigma(v, w)}$. We should consider $d_\Sigma$ as the angular metric. 
\begin{definition}[Tangent Space of Alexandrov Space]
    Let $X$ be a $n$-dimensional Alexandrov space. By our construction above, the tangent space, denoted by $T_pX$, of $X$ at $p \in X$ is defined as
    \begin{equation*}
        T_pX := C(\Sigma_p).
    \end{equation*}
    $T_pX$ is again a metric space equipped with the metric $d^{C(\Sigma_p)}$, which also has been constructed already. 
\end{definition}

\begin{theorem}[See \cite{AKP22}]\label{sigma-iff-cone}
	$\Sigma$ is Alexandrov of curvature $\ge 1$ if and only if $C(\Sigma)$ is Alexandrov of curvature $ \ge 0$.
\end{theorem}

Now we can generalize the example \ref{ex: pointed-gh-con tangent space} to the Alexandrov space with a lower curvature bound. 
\begin{theorem}
	Let $X$ be a finite-dimensional Alexandrov space of curvature $\geq \kappa$, we have the following pointed Gromov-Hausdorff convergence
\begin{equation*}
    (\lambda_i X, p_i) \pghto (T_pX, o), \quad \text{as $\lambda_i \to \infty$}
\end{equation*}
where $o$ is the cone point. 
\end{theorem}
\begin{corollary}
	Moreover, if the curvature of $X$ is $\geq \kappa$, then the curvature of $\lambda X$ is $\geq \frac{\kappa}{\lambda^2}$ just like the situation for manifolds. Therefore, $T_pX$ is of the curvature $\geq 0$ at the Gromov-Hausdorff limit, which means $T_pX$ is an Alexandrov space of curvature $\geq 0$. By Theorem ~\ref{sigma-iff-cone} that implies that $\Sigma_p$ has curvature $\geq 1$.  
\end{corollary}

\section{More Examples}
\begin{example}
    In particular, if $X = M^n$ is a Riemannian manifold. Then $\forall p \in X$, 
    \begin{align*}
        &\Sigma_p \cong S^{n - 1};\\
        &T_pX \cong \R^n
    \end{align*}
\end{example}
\begin{fact}
    We would like to mention the fact that if $X^n$ is an $n$-dimensional Alexandrov space, then $T_pX \cong \R^n$ for almost all $p \in X$ with respect to $\haus^n$.
\end{fact}
\begin{example}
    If $(M^n, g)$ is a complete Riemannian manifold of $\sect_M \geq \kappa$, then it is an Alexandrov space of curvature $\geq \kappa$. 
\end{example}
\begin{example}
    Let $X$ be an Alexandrov space of curvature $\geq \kappa$. And $G$ is a compact Lie group and $G \acts X$ by isometries. Then $X/G$ gives us again an Alexandrov space of curvature $\geq \kappa$. Many singular examples can be produced in this way. More generally if $X$ is an Alexandrov space of curvature $\geq \kappa$ and $f\o X\to Y$ is a submetry then $Y$ is also an Alexandrov space of curvature $\geq \kappa$.
\end{example}
\begin{example}
    If $(X, d)$ is an Alexandrov space of curvature $\geq \kappa$ and $A \subseteq X$ a closed convex subset. Then $(A, d)$ is also an Alexandrov space of curvature $\geq \kappa$. In particular, any convex body in $\R^n$ is an Alexandrov space of curvature $\geq 0$. 
\end{example}
\begin{example}\label{ex: Alex level set of Riem mfld}
    Let $(M^n, g)$ be a Riemannian manifold of $\sect_M \geq \kappa$ and let $f: M \to \R$ be a convex function. Then $\br{f = c}$ is an Alexandrov space of curvature $\geq \kappa$.
\end{example}
\begin{conj}
    The statement of the example \ref{ex: Alex level set of Riem mfld} is unknown but widely expected to be true if we weaken the assumption that the ambient space is a Riemannian manifold to an Alexandrov space. 
    \end{conj}
\chapter{Grove-Peterson Homotopy Type Theorem}
\section{Preliminaries}
The following fact~\ref{fact: compact alex in gh} explains why the Alexandrov space of lower curvature bound is better in Gromov-Hausdorff convergence. 
\begin{notation}
	We denote the family of the Alexandrov spaces of curvature $\geq \kappa$ as $\alex_\kappa$. 
\end{notation}
\begin{notation}\label{notation: A space}
	Fix $\kappa \in \R, n \in \N, D > 0$. We denote
\begin{equation*}
    A(n, \kappa, D) := \br{\text{$(X, d) \in \alex_\kappa$: $\dim(X) \leq n$, $\diam(X) \leq D$}}
\end{equation*}
the family of Alexandrov spaces with bounded dimension and diameter. 
And in addition, we denote
\begin{equation*}
    A(n, \kappa, D, V) := \br{\text{$(X, d) \in \alex_\kappa$: $\dim(X) \leq n$, $\diam(X) \leq D$, $\vol(X) \geq V$}}
\end{equation*}
\end{notation}
\begin{fact}\label{fact: compact alex in gh}
	$A(n, \kappa, D)$ and $A(n, \kappa, D, V)$ are both compact subsets of $\alex_\kappa$ in the Gromov-Hausdorff topology.
\end{fact}
\begin{remark}
	The spaces $A(n, \kappa, D)$ and $A(n, \kappa, D, V)$ are bigger comparing with $M_{sec}(n, \kappa, D)$ and $M_{sec}(n, \kappa, D, V)$
\end{remark}
\begin{remark}
	However, by the previous corollaries \ref{cor: ric-precompact}
	 and \ref{cor: sec-precompact} the family of Riemannian manifolds with lower curvature bound and bounded diameter is only pre-compact in the Hausdorff Topology. 
\end{remark} 
\begin{remark}
    There is a similar metric measure generalization to $RCD(\kappa, n)$ spaces, i.e. the family of Riemannian manifolds of $\ric \geq \kappa$ and bounded dimension. But we are not going to discuss that in this course. 
\end{remark}
In this section, we want to prove the following theorem. 
\begin{theorem}[Grove-Peterson]\label{thm: Grove-Peterson}
    Fix $n \in \N$, $\kappa \in \R$ and $D, V > 0$. Consider the following family of $n$-dimensional Riemannian manifolds. 
    \begin{equation*}
        M_{sec}(n, \kappa, D, V) := \br{(M^n, g): \diam(M) \leq D, \vol(M) \geq V, \sect_M \geq \kappa}
    \end{equation*}
    Then $M_{sec}(n, \kappa, D, V)$ contains only finitely many homotopy types of manifolds. That is, up to homotopy equivalence, it has only finitely many elements. 
\end{theorem}
Note that  $M_{sec}(n, \kappa, D, V) \subseteq A(n, \kappa, D, V)$. And it is also true that $A(n, k, D, V)$ contains finitely many homotopy types, In fact, also homeomorphism types. This follows from the Stability Theorem of Perelman.

The scheme of the proof of the Grove-Peterson theorem \ref{thm: Grove-Peterson} is an argument by contradiction:
\begin{proof}[Sketch Proof of Theorem~\ref{thm: Grove-Peterson}]
    Suppose the theorem is not true. That is, we have a sequence of $M_i \in M_{sec}(n, \kappa, D, V)$ such that $\forall i \neq j$, $M_i$ is not homotopy equivalent to $M_j$. Then by compactness, we have a subsequence 
    \begin{equation*}
        M_{i_j}^n \ghto X, \quad j \to \infty
    \end{equation*}
    where $X \in M_{sec}(n, \kappa, D, V)$. The reason for $X$ being dimensional $n$ is that we have the lower volume bound $V$ for each $j$ so there is no volume collapse at the limit. Next, we want to try to show that for large $i$ all $M_i$ are homotopically equivalent to $X$ and hence contradiction.
    
    It requires some work to show for large $i$ all $M_i$ are homotopically equivalent to $X$. 
\end{proof}
\section{Grove-Bishop-Gromov Volume Comparison}
To prove the theorem \ref{thm: Grove-Peterson} rigorously, we need a generalization of the Bishop-Gromov volume comparison theorem.
\begin{theorem}[Grove-Bishop-Gromov Volume Comparison]\label{thm: BG-comparison by Grove}
    Let $(M^n, g)$ be a complete Riemannian manifold with $\ric_M \geq (n - 1)\kappa$. And $\modsp{\kappa}$ the simply connected model space of $\sect \equiv \kappa$. Let $A \subseteq M$ be a compact set and $B_R(A) = B(A, R)$ the $R$-neighborhood of $A$. Let $\bar{p}\in  \modsp{\kappa}$.  Then the quantity
    \begin{equation}\label{eq:GBG-comparison}
        f(R) = \frac{\vol(B_R(A))}{\vol(B_R(\overline{p}))}
    \end{equation}
    is non-increasing in $R$. The classical Bishop-Gromov volume comparison now becomes the special case when $A = p \in M$. This result is true for any compact subset $A \subseteq M$. 
\end{theorem}
The idea is that we first prove the result for any finite set $A = \br{x_1, \dots, x_m}$. And if this is true, we can take denser and denser finite subsets of points for general $A$. Then we can apply the theorem to the denser set and pass to the limit of these denser and denser sets. 

\subsection{Voronoi Cells Construction}
As we mentioned, we first consider the case when $A$ is a finite set. An important idea here is to construct Voronoi cells for $A$ and decompose the manifolds into pieces having nice geometric features. 
\begin{definition}[Voronoi Cell]
	Given a finite subset $A = \br{x_1, \dots, x_m} \subseteq (M, g)$, we can define corresponding \textbf{Voronoi cell $V(x_i)$} which is denoted as 
    \begin{equation*}
        V(x_i) = \br{x \in M: \text{$d^M(x, x_i) \leq d^M(x_i, x_j)$ for all $j \neq i$}}
    \end{equation*}
    It is clear that $M = \bigcup_{i = 1}^m 
    V(x_i)$.
\end{definition}

Here we list two important facts of the Voronoi cells. 
    \begin{fact}\label{fact: Voronoi perpendicular}
   If $A=\{x,y\}\subset \R^n$ then $V(x)$ and $V(y)$ are half-spaces bounded by the hyperplane perpendicular to $[x,y]$ and passing through its middle.
    Therefore if $A = \br{x_1, \dots, x_m}\subset \R^n$ then each Voronoi cell $V(x_i)$ is an intersection of several half spaces. In particular it is convex.

    \end{fact}
    \begin{fact}\label{fact: Voronoi disjoint upto measure 0}
    Given a finite subset $A = \br{x_1, \dots, x_m} \subseteq (M, g)$. we can show $\haus^n(V(x_i)\cap V(x_j)) = 0$. Thus up to measure $0$, we can think of $\br{V(x_i)}_i$ as a disjoint collection of subsets. 
    \end{fact}

    \begin{figure}[htbp]
        \centering
            \includegraphics[width=0.7\textwidth]{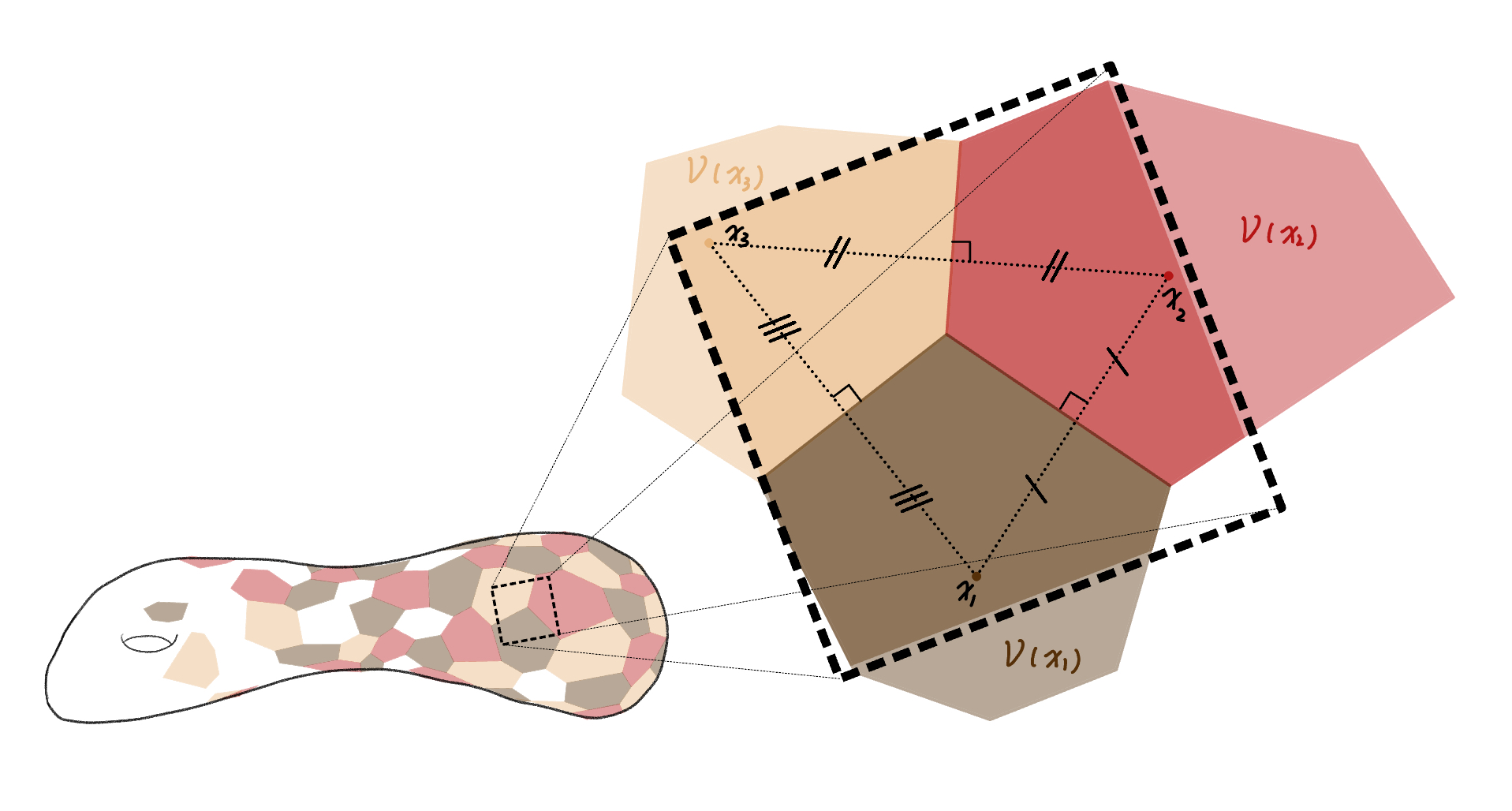}\caption{Voronoi Sets [Fact~\ref{fact: Voronoi perpendicular} and Fact~\ref{fact: Voronoi disjoint upto measure 0}]}

        \end{figure}
    In Riemannian geometry, we have the local uniqueness of geodesics if the initial condition is given. However, in a general geodesic metric space, this property may fail since two geodesics with the same initial condition may split at some point (Example~\ref{ex: splitting geodesics}). We will show that the property that geodesics do not split holds in Alexandrov spaces with lower curvature bound. 
    \begin{example}\label{ex: splitting geodesics}
    	Let $(X, d) = (\R^2, \norm{\cdot}_{1})$, i.e. for $x, y \in X$, $\norm{(x, y)}_1 = \abs{x} + \abs{y}$. Let $x = (1, 0)$, $x_2 = (0, 1)$, $x_3 = (0, -1)$. Then the geodesic segment $[xx_1] = \br{0}\times[0, -1]\cup[0, 1]\times \br{0}$ and $[xx_2] = \br{0}\times[0, 1]\cup[0, 1]\times \br{0}$ (See Figure~\ref{fig: splitting in r^2}). Notice that $[xx_1]\cap[xx_2] = [0, 1]\times\br{0}$, which is an interval. In this example, two geodesics branch at $(0, 0)$.
    	\begin{figure}[htbp]
        \centering
            \includegraphics[width=0.4\textwidth]{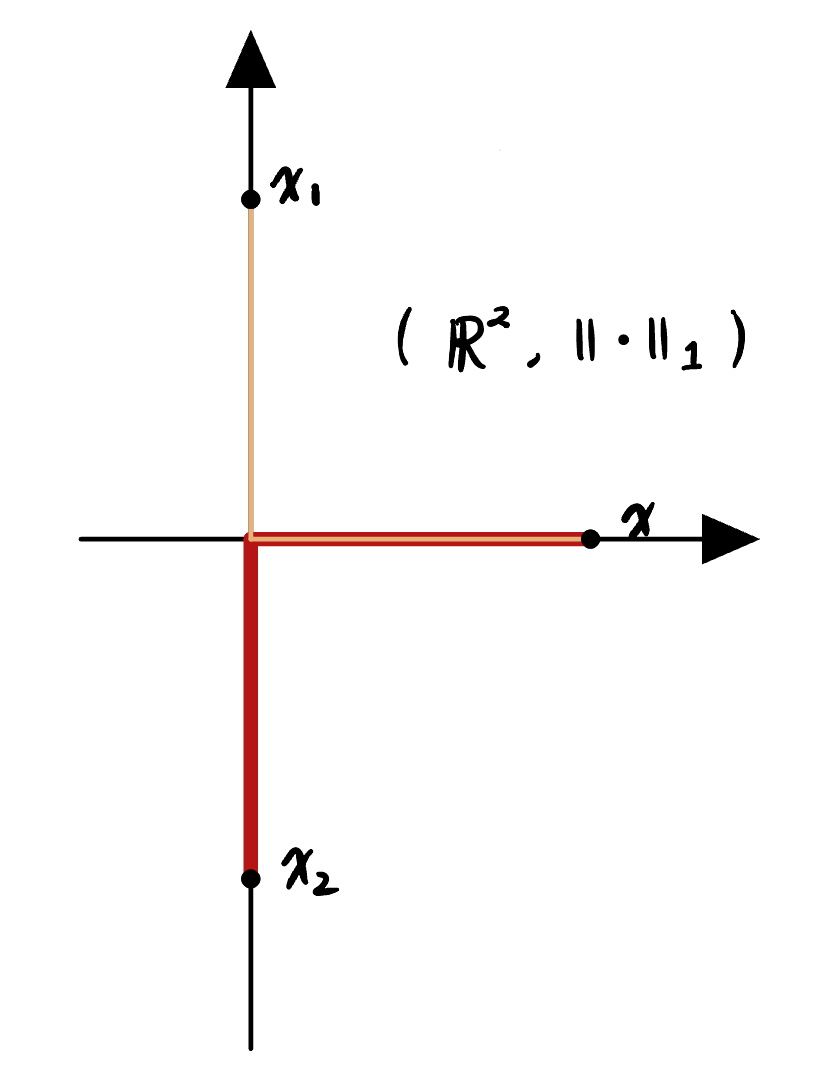}\caption{Geodesics may have branches in $(\R^2, \norm{\cdot}_{1})$}
            \label{fig: splitting in r^2}
        \end{figure}
    \end{example}
    \begin{definition}[Non-Branching Subset (From~\cite{RS12})]\label{def: non-branching}
    	Two geodesics $ \gamma_1, \gamma_2: [0,T]\to X$ are \textbf{branching} there exist $0<t_1<t_2$ such that $\gamma_1|_{[0,t_1]}= \gamma_2|_{[0,t_1]}$ but $\gamma_1(t_2)\ne \gamma_2(t_2)$. A metric space is called \textbf{non-branching} if there are no branching geodesics.
    \end{definition}
    \begin{lemma}[Non-Branching Property of Alexandrov Space]\label{lem: non-branching alex}
    	$\alex_{\kappa}$ is non-branching.
    \end{lemma}
    \begin{figure}[htbp]
        \centering
        \includegraphics[width=0.6\textwidth]{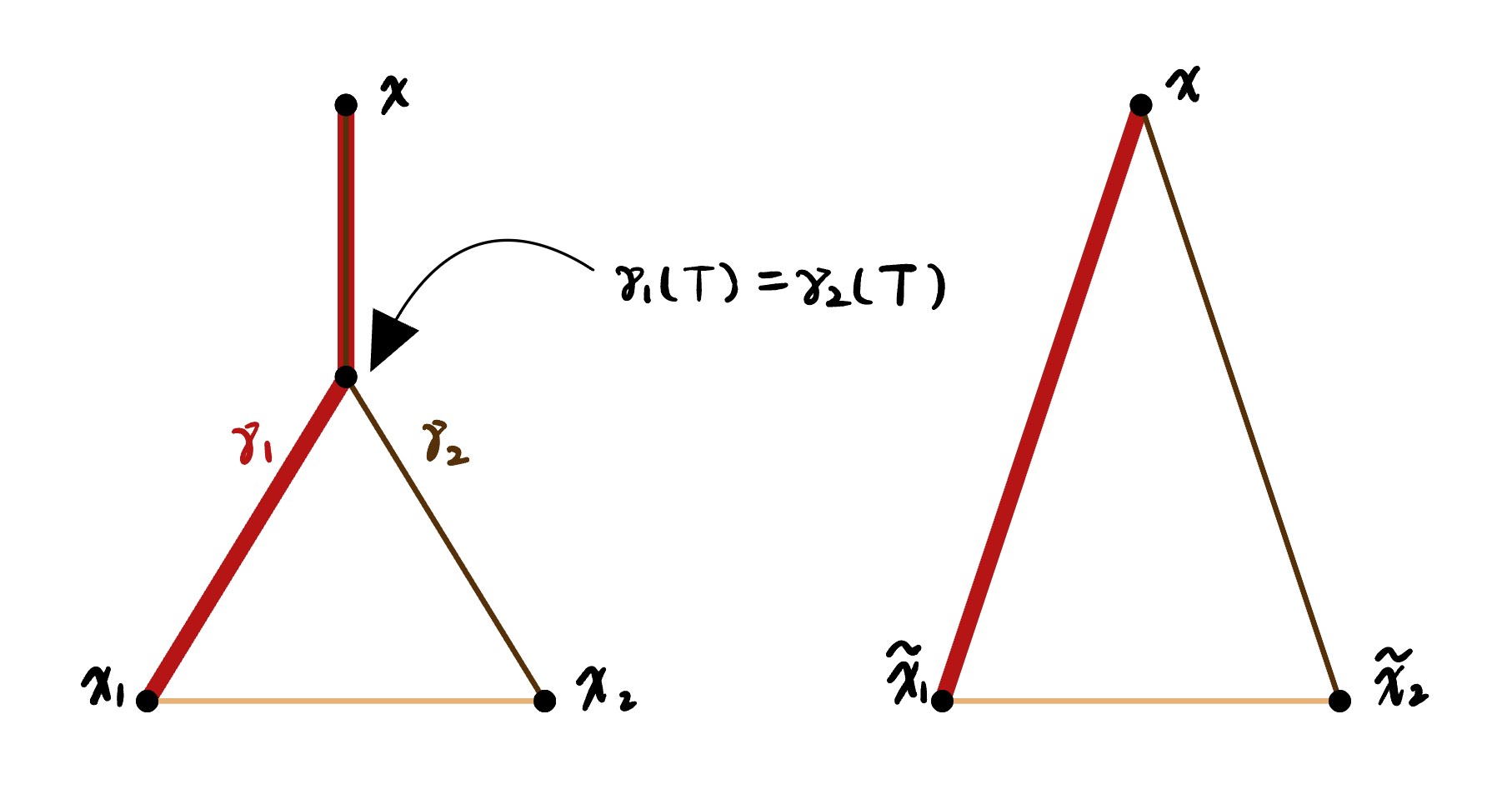}\caption{Branching geodesics Alexandrov space contradicts the Toponogov angle comparison.}
            \label{fig: alex-non branching}
        \end{figure}
    \begin{proof}
    Let $\gamma_1, \gamma_2 \in Geo(\alex_{\kappa})$ be two geodesic segments such that $\gamma_1|_{[0, t]} = \gamma_2|_{[0, t]}$ for some $0 < t < 1$,  We denote
    \begin{equation*}
    	T = \max{\br{t \in [0, 1]: \gamma_1|_{[0, t]} = \gamma_2|_{[0, t]}}}.
    \end{equation*}
    We can use a contradiction argument to show $\gamma_1 \equiv \gamma_2$. Since $\gamma_1|_{[0, T]} = \gamma_2|_{[0, T]}$, we know that $\gamma_1(0) = \gamma_2(0)$, without lost of generality, we suppose $\gamma_1(s) \neq \gamma_2(t)$ for all $s \in (T, 1]$. Especially, we denote $x = \gamma_1(0) = \gamma_2(0)$, $x_1 = \gamma_1(1)$ and $x_2 = \gamma_2(1)$. We can see from Figure~\ref{fig: alex-non branching} such that the angle $\mangle[x_{x_1}^{x_2}] = 0$ for the triangle $[xx_1x_2]$ (See Definition \ref{def: triangle} and Notation~\ref{not: angle}). By the definition of $\alex_{\kappa}$ and the equivalent statements of Toponogov comparisons. We know the Toponogov angle comparison holds in $\alex_{\kappa}$. Therefore, this means the comparison angle $\modangle^{\kappa}[x_{x_1}^{x_2}]$ in the model space is zero. However, since $x, x_1, x_2$ are three distinct points, then it is impossible to have $\abs{x_1 - x_2} = 0$. Contradiction! 
    \end{proof} 
	\begin{remark}
		The above Lemma~\ref{lem: non-branching alex} was generalized to $RCD(K, N)$ spaces in the thesis~\cite{Den21}.
	\end{remark}
    \begin{lemma}\label{lem: Voronoi star-shape}
        Given a finite subset $A = \br{x_1, \dots, x_m}$ of an Alexandrov space of curvature $\geq \kappa$, say $X$. Then $V(x_i)$ is star-shaped with respect to $x_i$, i.e. for any $x \in V(x_i)$, the geodesic segment $[xx_i] \subseteq V(x_i)$. 
        
        If we parametrize $[xx_i]$ using $\gamma$ such that $\gamma(0) = x$, if $d^M(\gamma(t), x_i) = d^M(\gamma(t), x_j)$ for some $t$, then $t = 0$ and $x \in V(x_i)\cap V(x_j)$. Furthermore, if $t > 0$, we must have $d^M(\gamma(t), x_i) < d^M(\gamma(t), x_j)$.
    \end{lemma}
    \begin{proof}[Proof of Lemma~\ref{lem: Voronoi star-shape}]
        Denote $d = \dist(x, x_i)$ and $d^X$ the distance function of the alexandrov space $X$ and consider the unit speed geodesic $\gamma: [0, d] \to V(x_i)$. In particular, we set $x = \gamma(0)$, $x_i = \gamma(d)$ and thus
        \begin{equation*}
            g_i(t) := d^X(\gamma(t), x_i) = d-t
        \end{equation*}
        decreasing with speed $1$, i.e. $g_i'(t) \equiv -1$. Let $x_j \neq x_i$. We know that $g_j(t) = d^X(\gamma(t), x_j)$ is an $1$-Lipschitz function and hence its right derivative exists a.e.  satisfies
        \begin{equation*}
            {g_j}_+'(t) \geq -1
        \end{equation*}
        the function decreases with the speed at most $1$. 
        Thus we know for any $t \in [0, d]$, we have
        \begin{equation*}
        	g_j(t) \geq g_j(0) - t \geq g_i(0) - t = g_i(t)
        \end{equation*}
        Therefore we can conclude that 
        \begin{equation*}
        	d^M(\gamma(t), x_i) \leq d^M(\gamma(t), x_j)
        \end{equation*}
        for all $j \neq i$. Thus we have
        \begin{equation*}
        	\gamma(t) \in V(x_i)
        \end{equation*}
        for all $t$.
        
        Now if $d^X(\gamma(t), x_i) = d^X(\gamma(t), x_j)$ for some $t$, this means the following inequalities
        \begin{equation*}
        	g_j(t) \geq g_j(0) - t \geq g_i(0) - t = g_i(t)
        \end{equation*}
        are all equalities. This implies $g_i(0) = g_j(0)$, which means $x \in V(x_i) \cap V(x_j)$.
        
        \begin{figure}[htbp]
        \centering
            \includegraphics[width=0.7\textwidth]{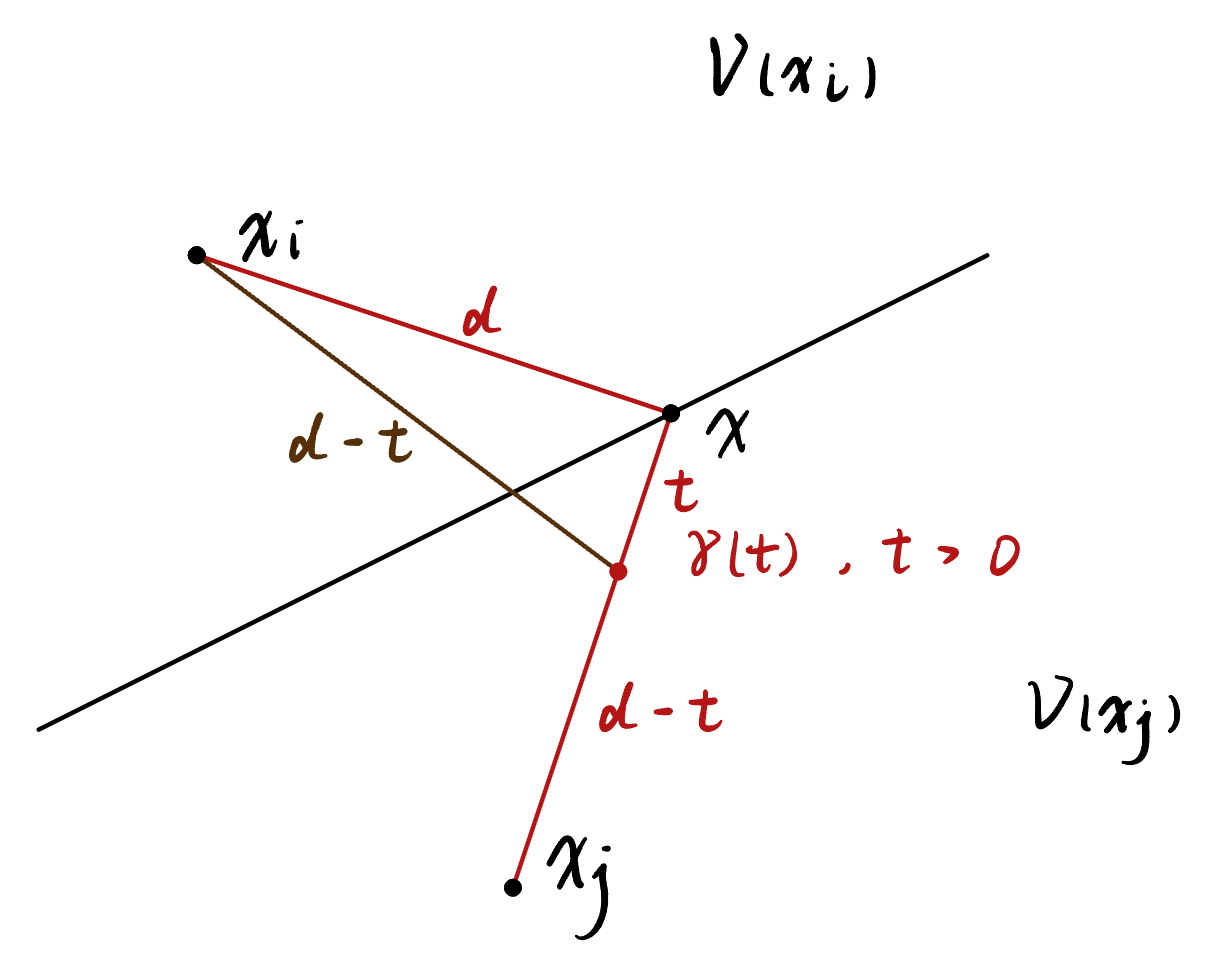}\caption{Geodesics are branching at $\gamma(t)$}
			\label{fig: branch at gamma(t)}
        \end{figure}
        Next we claim that if $g_i(t) = g_j(t)$ then $t = 0$. Suppose this is not true, i.e. $g_i(t) = g_j(t)$ for some $t > 0$. By above this implies  $g_j(0) = g_i(0)$, namely, $d^X(x, x_i) = d^X(x, x_j)$ and $x \in V(x_i)\cap V(x_j)$. By our assumption, we also know that $d^X(\gamma(t), x_i) = d^X(\gamma(t), x_j) = d - t$. This makes the length of the hinge $[\gamma(t)_{x}^{x_j}]$ equals to $d$ and thus both $[\gamma(t)_{x}^{x_j}]$ and $[xx_i]$ are geodesics from $x$ to $x_i$. However, we notice that this means the geodesics $[\gamma(t)_{x}^{x_j}]$ and $[xx_j]$ branching at $\gamma(t)$ as in Figure~\ref{fig: branch at gamma(t)}, which is impossible by the above Lemma~\ref{lem: non-branching alex}. 
    \end{proof}
\subsection{Volume Comparison for Finite Set}
In this section, we claim that the theorem~\ref{thm: BG-comparison by Grove} holds in the case when $A$ is finite. Namely, we need to show $f(R)$ in \ref{eq:GBG-comparison} is non-increasing. 

        By the fact \ref{fact: Voronoi disjoint upto measure 0}, we can essentially think of $B_R(A)$ as a  disjoint union decomposed by Voronoi sets upto measure $0$, i.e.
        \begin{equation*}
            B_R(A) = B_R(\br{x_1, \dots, x_n}) = \stackrel{\textbf{Ess}}{\bigcup}^n_{i = 1}B_R(A)\cap V(x_i). 
        \end{equation*} 
        That implies that 
        \begin{equation*}
            \vol(B_R(A)) = \sum_{i = 1}^n \vol(B_R(A)\cap V(x_i))
        \end{equation*}
        Therefore, it is enough to prove that 
        \begin{equation*}
            f_i(R) := \frac{\vol(B_R(A)\cap V(x_i))}{\vol(B_R(\overline{p}))}
        \end{equation*}
        is non-increasing for each $i$.

        Now, we need to define our cut function. Fix a unit vector $v \in T_{x_i}M$, we denote $t_{\text{cut}}(x_i, v)$ the cut-time of the pair $(x_i, v)$. Then, for the unit speed geodesic
        \begin{equation*}
            \gamma_v(t) = \exp{(tv)}, 
        \end{equation*}
        and define the function
        \begin{equation*}
            c(v) = 
            \begin{cases}
                t_{\text{cut}}(x_i, v) \quad\text{if $\gamma_v(t_{\text{cut}}(x_i, v)v) \in V(x_i)$}\\
                \max{\br{d \in \R: \gamma_v(d) \in V(x_i)}}\quad\text{if $\gamma_v(t_{\text{cut}}(x_i, v)v) \notin V(x_i)$}
            \end{cases}
        \end{equation*}
        Thus it is always true that $c(v) \leq t_{\text{cut}}(x_i, v)$. Then we proceed  as in the proof of the original Bishop-Gromov comparison. For the geodesic $\gamma_v(t)$, we can consider parallel unit orthonormal vector fields $\br{v_1, \dots, v_{n - 1}}$ along $\gamma_v(t)$ perpendicular to $\gamma_v(t)$. Thus we can get the Jacobi fields $J_1(t), \dots, J_{n-1}$ along $\gamma_v(t)$ such that for each $i = 1, \dots, n-1$,
        \begin{equation*}
            J_i(0) = 0, \quad J_i'(0) = v_i.
        \end{equation*}
        Then we can define 
        \begin{equation*}
            j_v(t) := \det{(J_1(t), \dots, J_{n - 1}(t))}
        \end{equation*}
        in $v_1, \dots, v_{n - 1}$ basis. Now we redefine $j_v(t)$ using $c(v)$. Namely, 
        \begin{equation*}
            \hat{j}_v(t) = 
            \begin{cases}
                j_v(t) \quad\text{if $t < c(v)$};\\
                0 \quad \text{if $t \geq c(v)$}
            \end{cases}
        \end{equation*}
        
        Next, we just proceed the proof as in Bishop-Gromov comparsion. Consider the quantity $\frac{\hat{j}_v(t)}{\overline{j}(t)}$, where $\overline{j}(t)$ is the corresponding function in the model space ($\overline{j}(t) = (\sn_\kappa(t))^{n - 1}$). Notice the quantity $\frac{\hat{j}_v(t)}{\overline{j}(t)}$ is non-increasing because we can apply the same proof to the star-shaped region $V(x_i)$. Therefore, we get
        \begin{equation*}
            \vol(B_R(A) \cap V(x_i)) = \int_{S^n}\brac{\int_0^{R}\hat{j}_v(t)dt}d\vol^{n - 1}.
        \end{equation*}
        Meanwhile, as in Bishop-Gromov comparison, this is divided by 
        \begin{equation*}
            \vol(B_R(\overline{p})) = \int_{S^n}\brac{\int_0^{R}\overline{j}(t)dt}d\vol_{n - 1}.
        \end{equation*}
        Thus we obtain the monotonicity by the same argument as in the original proof of Bishop-Gromov.
\subsection{Approximation Arguments by Finite Sets}
    As we mentioned in the beginning, we can pick denser and denser subsets so that the theorem \ref{thm: BG-comparison by Grove} is also true. For compact set $A$, we can pick a sequence of finite sets $A_n$ such that the Hausdorff distance
    \begin{equation*}
    	d_H(A_n, A) \to 0
    \end{equation*}
    as $n \to \infty$. This is possible since for any $\varepsilon > 0$, we can find a finite $\varepsilon$-net $A_n$ in $A$. 
    
    And thus, point-wisely, 
    \begin{equation*}
    	\vol(B_R(A_n)) \to \vol(B_R(A))
    \end{equation*}
    Hence, we have the point-wise convergence, 
    \begin{equation*}
    	f_n(R) = \frac{\vol(B_R(A_n))}{\vol(B_R(\overline{p}))} \to \frac{\vol(B_R(A))}{\vol(B_R(\overline{p}))} = f(R)
    \end{equation*}
    Thus, we can generalize the theorem~\ref{thm: BG-comparison by Grove} from finite set to compact set since  monotonicity is preserved under point-wise convergence.

\section{$\varepsilon$-Critical Points Theory}

From our new volume comparison theorem~\ref{thm: BG-comparison by Grove}, we have the following corollary which particularly estimates the volume of  small annuli in a sphere. This corollary is important for proving the theorem~\ref{thm: cannot be too closed}. And this theorem is essential for the Grove-Peterson theorem. 
\subsection{Volume Estimates for Small Annulus.}
\begin{corollary}\label{cor: annulus of closed set}
    Let $A \subseteq S^n$ be a closed subset. And denote $\an_{S^n}(A, r, R)$ the annulus in $S^n$ with respect to the set $A$, i.e. 
    \begin{equation*}
        \an_{S^n}(A, r, R) = \br{x \in S^n: r \leq d(x, A) \leq R}. 
    \end{equation*}
     Then for sufficiently small $\eps > 0$ it holds that
    \begin{equation*}
        \vol(\an_{S^n}(A, \frac{\pi}{2} - \eps, \frac{\pi}{2} + \eps)) \leq c(n)\cdot\eps
    \end{equation*}
    where $c(n)$ is some universal constant depends on $n$. 
\end{corollary}
\begin{proof}[Proof of Corollary~\ref{cor: annulus of closed set}]
    This corollary is easily followed by the Bishop-Gromov volume comparison. According to the Bishop-Gromov's conparison, if $r < R$, then 
    \begin{align*}
        & \frac{\vol(B_r(A))}{\vol(B_R(A))} \geq \frac{\vol(B_r(\overline{p}))}{\vol(B_R(\overline{p}))}\\
        \iff & 1 - \frac{\vol(B_r(A))}{\vol(B_R(A))} \leq 1 - \frac{\vol(B_r(\overline{p}))}{\vol(B_R(\overline{p}))}
    \end{align*}
    Notice the LHS of the inequality above,
    \begin{align*}
        &1 - \frac{\vol(B_r(A))}{\vol(B_R(A))}\\
        =& \frac{\vol(B_R(A)) - \vol(B_r(A))}{\vol(B_R(A))}\\
        =& \frac{\vol(\an_{S^n}(A, r, R))}{\vol(B_R(A))}
    \end{align*}
    Similar things can be applied to the RHS. Therefore,
    \begin{align*}
        &\frac{\vol(\an_{S^n}(A, r, R))}{\vol(B_R(A))} \leq \frac{\vol(\an_{S^n}(p, r, R))}{\vol(B_R(\overline{p}))}\\
        \iff & \vol(\an_{S^n}(A, r, R)) \leq \frac{\vol(B_R(A))}{\vol(B_R(\overline{p}))}\cdot\vol(\an_{S^n}(\overline{p}, r, R))
    \end{align*}
    If $r = \frac{\pi}{2} - \eps$, $R = \frac{\pi}{2} + \eps$, $A \subseteq S^n$, $p \in S^n$. Then $\an_{S^n}(\overline{p}, \frac{\pi}{2} - \eps, \frac{\pi}{2} + \eps)$ is just the $\eps$-neighborhood of the equator of $S^n$. 
    \begin{figure}[htbp]
        \centering
            \includegraphics[width=0.6\textwidth]{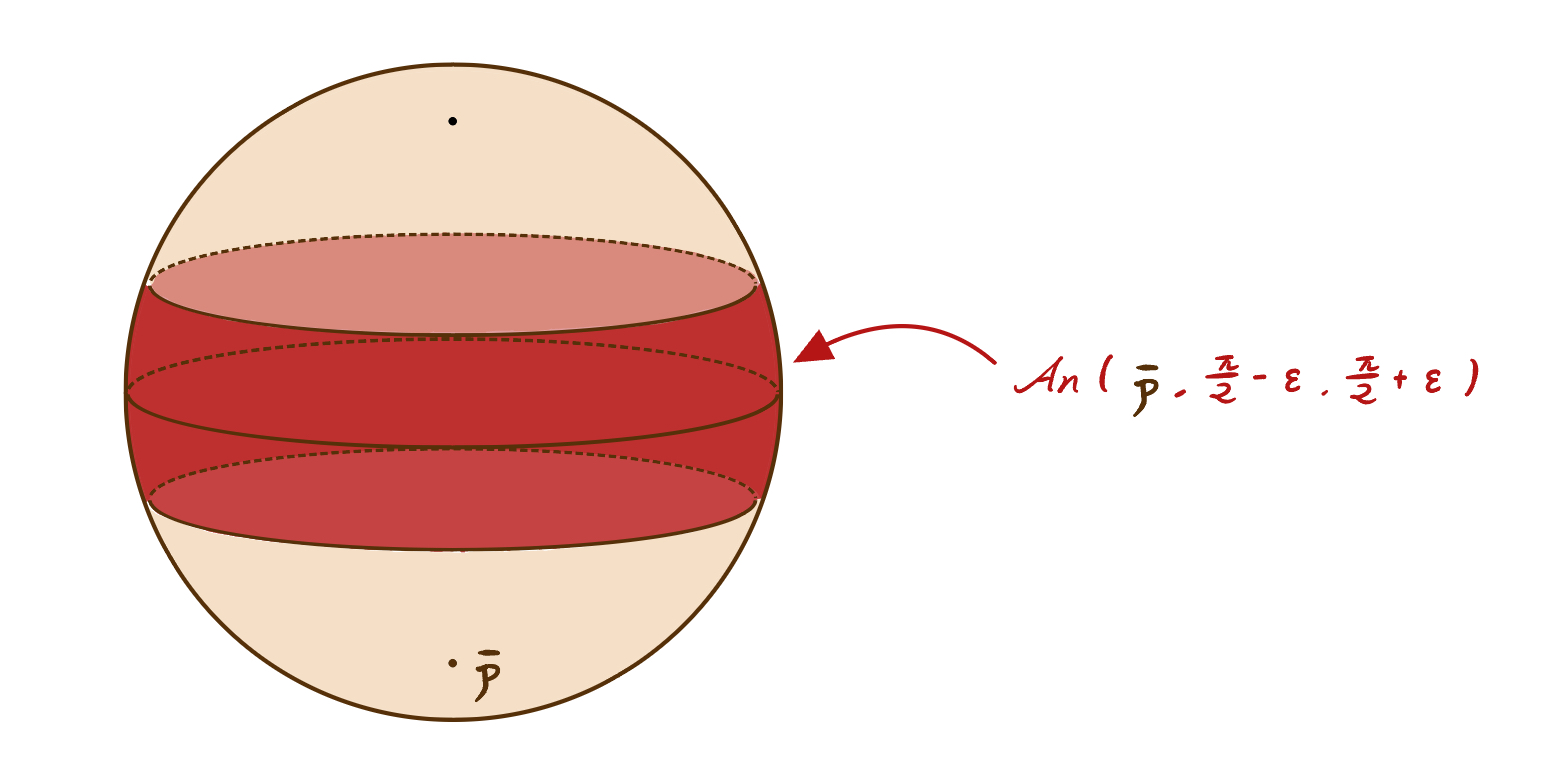}\caption{$\an_{S^n}(\overline{p}, \frac{\pi}{2}- \eps, \frac{\pi}{2}+ \eps)$ is the $\eps$-neighborhood of the equator of $S^n$.}
        \end{figure}
        Therefore, the volume of the $\eps$-neighborhood is $\leq c(n)\cdot \eps$. Therefore, it follows that  
        \begin{align*}
            &\vol(\an_{S^n}(A, \frac{\pi}{2} - \eps, \frac{\pi}{2} + \eps)) \\
            &\leq \frac{\vol(B_R(A))}{\vol(B_R(\overline{p}))}\cdot\vol(\an_{S^n}(p, \frac{\pi}{2} - \eps, \frac{\pi}{2} + \eps))\\
            &\leq c(n)\cdot\eps
        \end{align*}
\end{proof}
\subsection{Introduction to $\varepsilon$-Critical Points}

We will need the following relaxation of the notion of a critical point of  a distance function.

\begin{definition}[$\varepsilon$-Critical Point]\label{def: e-crit-pt}
    Let $(M^n, g)$ be a compact Riemannian manifold. Let $p, q \in M^n$ and let $\eps > 0$. Then $q$ is called \textbf{$\eps$-critical for $f^p := d(\cdot, p)$}, if 
    \begin{equation*}
        df^p_q(v) \leq \eps, \quad\text{for all unit $v \in T_qM$}
    \end{equation*}
    In particular, $q$ is just a critical point for $f$ when $\eps = 0$ as we commonly defined. 

    For general semiconcave function $f$ on $(M^n, g)$, let $\eps > 0$ and let $q \in M^n$, we say \textbf{$f$ is $\eps$-critical at $p$} if $df_q(v) \leq \eps$ for all unit $v \in T_qM$. Similarly, \textbf{$f$ is $\eps$-regular at $q$} is $df_q(v) > \eps$ for all unit  $v \in T_qM$.  
\end{definition}
Now by the first variation formula, consider all the shortest geodesics starting from $q$ to $p$ and fix a unit vector $v$ at $q$. Denote 
\begin{equation*}
    \alpha := \mangle(\Uparrow_q^p, v):= \min_{u \in \Uparrow_q^p}\br{\mangle(u, v)}
\end{equation*}
so $df^p_q(v) = -\cos{\alpha}$. Then we can see that $q$ is $\eps$-critical if
\begin{equation*}
    -\cos{\alpha} \leq \eps \iff \alpha \leq \frac{\pi}{2} + \eps. 
\end{equation*}

\begin{figure}[htbp]
        \centering
            \includegraphics[width=0.6\textwidth]{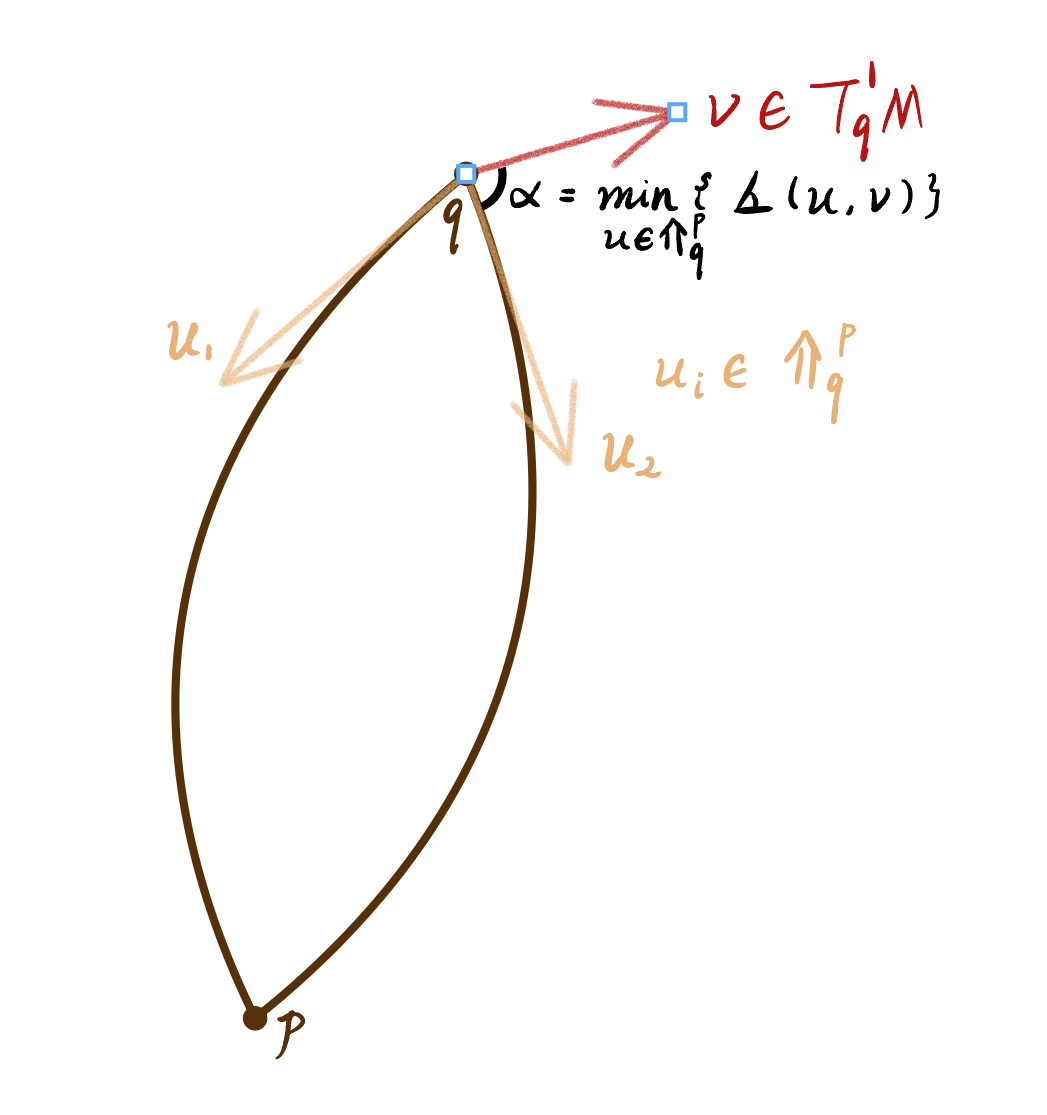}\caption{Equivalent definition of an $\eps$-critical point [Definition \ref{def: equi-e-crit-pt}].}
        \end{figure}
This gives us the following equivalent definition of an  $\eps$-critical point. 
\begin{definition}[Equivalent Definition of $\eps$-Critical Point]\label{def: equi-e-crit-pt}
    Let $(M^n, g)$ be a compact Riemannian manifold. Let $p, q \in M^n$ and let $\eps > 0$. Then $q$ is \textbf{$\eps$-critical for $f^p := d^M(\cdot, p)$} if for any $v$-unit vector $v$ at $q$, there exists $u \in \Uparrow_q^p$ such that 
    \begin{equation*}
        \mangle(u, v) \leq \frac{\pi}{2} + \eps
    \end{equation*}
\end{definition}
The following example shows that if there is no lower volume bound mutually $\eps$-critical points can be very close to each other.
\begin{example}
    Consider $S_\eps^n$ a sphere of radius $\eps$, then any pairs of points $p$ and $q$ that are opposite to each other are also $\eps$-critical to each other. 
\end{example}


After a brief discussion of $\eps$-critical points, we would like to use the following theorem to show how $\eps$-critical points are distributed in Riemannian manifolds. Indeed, we can show that if two points are mutually $\eps$-critical to each other, then these points cannot be too close to each other. 
\begin{theorem}\label{thm: cannot be too closed}
    Fix $n \in \N$, $\kappa \in \R$ and $D, V > 0$. Consider the family of Riemannian manifolds
    \begin{equation*}
        M_{sec}(n, \kappa, D, V) = \br{(M^n, g): \sect_M \geq \kappa, \diam(M) \leq D, \vol(M) \geq V}
    \end{equation*}
    Then there exists $\eps = \eps(n, \kappa, D, V) > 0$ and $\delta = \delta(n, \kappa, D, V) > 0$ such that if $p, q \in M \in M_{sec}(n, \kappa, D, V)$ are mutually $\eps$-critical to each other, then 
    \begin{equation*}
        d(p, q) \geq \delta
    \end{equation*}
\end{theorem}

\begin{proof}[Proof of Theorem~\ref{thm: cannot be too closed}]
    Assume for simplicity $\kappa = 0$ and $\eps > 0$ sufficiently small and take $\delta = \eps^3$. Suppose $p, q \in M$ are mutually $\eps$-critical to each other and $d(p, q) < \eps^3$. We want to show that this implies the volume of the manifolds is small, i.e. 
    \begin{equation*}
        \vol(M) < V
    \end{equation*}
  which leads to a  contradiction. 

    Since $d(p, q) < \eps^3$ and $\eps^3 << \eps$, $q \in B_{\eps}(p)$. By Bishop-Gromov volume comparison,
    \begin{equation}\label{eq: vol est for eps small ball}
        \vol(B_\eps(p)) \leq \omega_n\eps^n
    \end{equation}
     We want to estimate the volume of $M\backslash B_{\eps}(p)$. Fix $x \in M\backslash B_{\eps}(p)$. Then we have
     \begin{align*}
         d(x, p)& \geq \eps\\
         d(x, q)& \geq \eps - \eps^3
     \end{align*}
  
     \begin{figure}[htbp]
        \centering
            \includegraphics[width=0.9\textwidth]{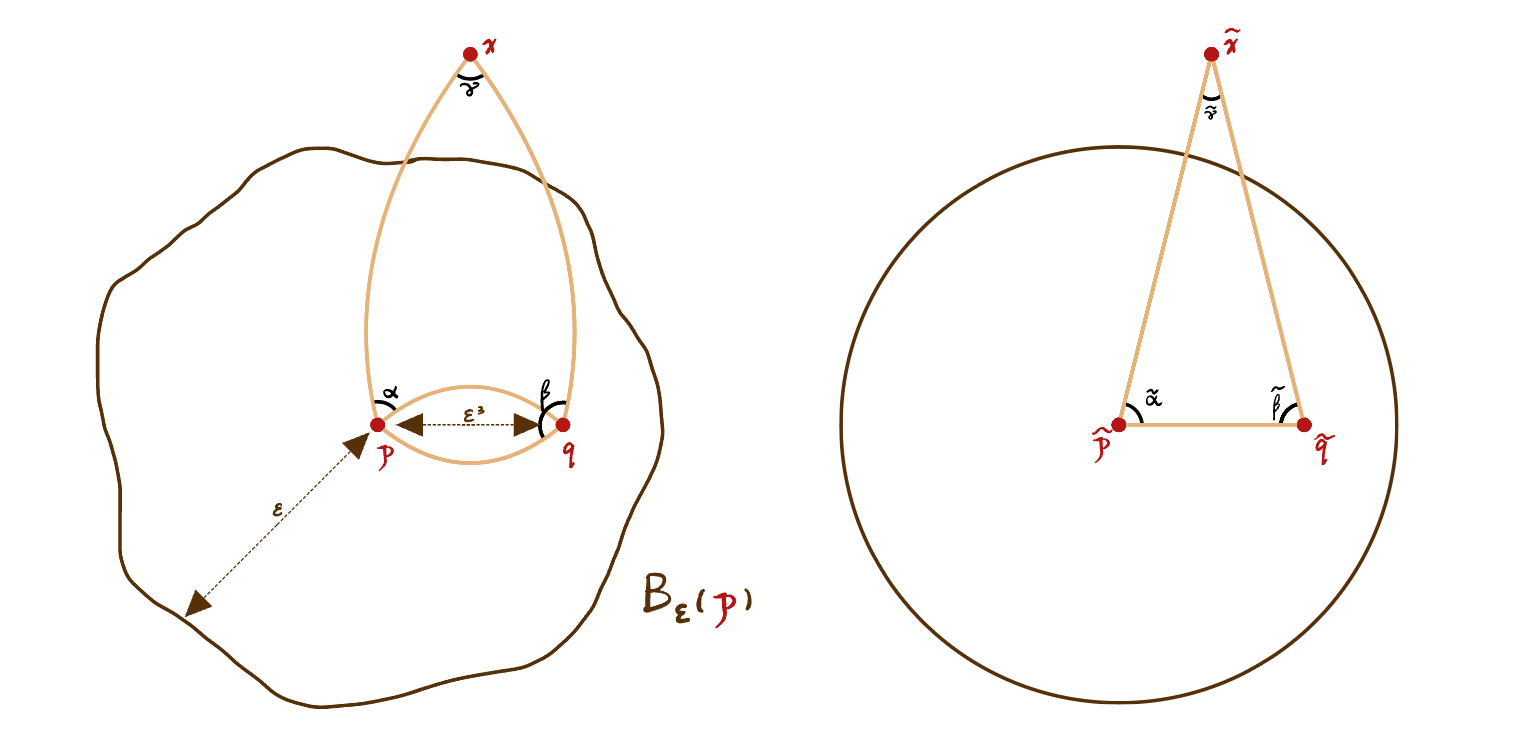}
    \end{figure}
    
    Connecting $p, q$ using shortest geodesics (May not be unique). We take $\alpha = \mangle(\Uparrow_q^p, d(p, x))$, $\beta = \mangle(\Uparrow_p^q, d(q, x))$ and $\gamma = \mangle(x_q^p)$ and denote $\Tilde{\alpha}, \Tilde{\beta}, \Tilde{\gamma}$ be the corresponding angles in the model space $\R^n$. Moreover, we denote $\Tilde{x}, \Tilde{p}, \Tilde{q}$ the corresponding points of $x, p, q$ in the model space $\R^n$ by requiring 
    \begin{align*}
         d(\Tilde{x}, \Tilde{q}) = d(x, q) > \eps - \eps^3\\
         d(\Tilde{x}, \Tilde{p}) = d(x, p) > \eps\\
         d(\Tilde{p},\Tilde{q})  = \d(p, q) \leq \eps^3
     \end{align*}
    Therefore, by the trigonometry in the Euclidean space, we have
    \begin{equation*}
        \Tilde{\gamma} \leq c\cdot \eps^2
    \end{equation*}
    for some constant $c$. By Toponogov comparison, we have 
    \begin{align*}
    &\begin{cases}
        \Tilde{\alpha} \leq \alpha \leq \frac{\pi}{2} + \eps\\
        \Tilde{\beta} \leq \beta \leq \frac{\pi}{2} + \eps
    \end{cases}\\
    \implies & \Tilde{\alpha} + \Tilde{\beta} \leq \pi + 2\eps
    \end{align*}
    On the other hand, since $\Tilde{\alpha} + \Tilde{\beta} + \Tilde{\gamma} = \pi \iff \Tilde{\alpha} + \Tilde{\beta} = \pi - \Tilde{\gamma}$, 
    \begin{equation*}
        \Tilde{\alpha} + \Tilde{\beta} \geq \pi - c\cdot\eps^2
    \end{equation*}
    Therefore, we can conclude that 
    \begin{equation*}
        \pi - c \cdot\eps^2 \leq \Tilde{\alpha} + \Tilde{\beta} \leq \pi + 2\eps
        \implies 
        \begin{cases}
            \Tilde{\alpha} \geq \frac{\pi}{2} - 2\eps\\
            \Tilde{\beta} \geq \frac{\pi}{2} - 2\eps        \end{cases}
    \end{equation*}
    then
    \begin{align*}
        \frac{\pi}{2} - 2\eps \leq \Tilde{\alpha} \leq \alpha \leq \frac{\pi}{2} + \eps\\
        \frac{\pi}{2} - 2\eps \leq \Tilde{\beta} \leq \beta\leq \frac{\pi}{2} + \eps
    \end{align*}
    Let $v := \uparrow_p^x$ be the unit vector at $p$ along the geodesic $[px]$ then by the inequalities above, we have
    \begin{equation*}
        \frac{\pi}{2} - 2\eps \leq \mangle(v, \Uparrow_p^q) \leq \frac{\pi}{2} + \eps
    \end{equation*}
    Recall the corollary \ref{cor: annulus of closed set}, there exists $c(n)$ such that for any closed subset $A \subseteq S^m$, we have
    \begin{equation*}
        \vol_m(\an(A, \frac{\pi}{2} - \eps, \frac{\pi}{2} + \eps)) \leq \eps\cdot c(m)
    \end{equation*}
    In our case, since we can think of unit vectors as elements of a unit sphere and have angular distance, we have 
    \begin{align*}
        &\frac{\pi}{2} - 2\eps \leq \mangle(v, \Uparrow_p^q) \leq \frac{\pi}{2} + \eps\\
        \implies & v = \uparrow_p^x \in \an(\Uparrow_p^q, \frac{\pi}{2} - 2\eps, \frac{\pi}{2} + \eps)
    \end{align*}
    where
    \begin{equation*}
        \vol_{n - 1}(\an(\Uparrow_p^q, \frac{\pi}{2} - 2\eps, \frac{\pi}{2} + \eps)) \leq \eps\cdot c(n - 1)
    \end{equation*}
    Recall by Rauch or Toponogov theorem, consider 
    \begin{equation*}
        \exp_p: T_pM \to M
    \end{equation*}
    and we choose $C_p \subseteq T_pM$ a set inside cut-locus. Then on $C_p$, $\exp_p$ is $1$-Lipschitz.  Therefore, it does not increase volume, i.e.
    \begin{equation*}
        \vol_n(\exp_p(U)) \leq \vol_n(U) , \quad\forall U \subseteq C_p
    \end{equation*}
    But $M^n \subseteq B(p, D)$ since $\diam(M) \leq D$. Also, since $B_D(0) \subseteq T_pM$, we can express each elements $u\in (B_D(0) \cap C_p) \backslash B_\eps(0)$ in polar coordinates, i.e. $u = (r, v)$ where $v$ is an unit vector. And $\eps \leq r \leq D$, $v \in \an(\Uparrow_p^q, \frac{\pi}{2} - 2\eps, \frac{\pi}{2}+ 2\eps) =: W$. Since $\vol_{n - 1}(W) \leq c(n) \cdot \eps$, 
    \begin{equation*}
        \vol_{n}((B_D(0)\cap C_p)\backslash B_\eps(0)) \leq \frac{D^n}{n}\cdot c(n)\cdot \eps
    \end{equation*}
    And again since $\exp_p$ is $1$-Lipschitz on this set. Then 
    \begin{equation*}
        \vol_n(B_D(p)\backslash B_\eps(p)) \leq D^n\cdot c(n)\cdot \eps
    \end{equation*}
    therefore, together with the inequality~\ref{eq: vol est for eps small ball}, we have 
    \begin{align*}
        \vol_n(M) 
        =& \vol_n(B_\eps(p)\cup B_D(p)\backslash B_\eps(p))\\
        \leq & \vol_n(B_\eps(p)) + \vol_n(B_D(p)\backslash B_\eps(p))\\
        \leq & \eps^n\cdot\omega_n + D^n\cdot c(n) \cdot \eps
    \end{align*}
    Finally, if $\eps$ is small enough so that the $RHS$ of the above inequality is $\leq V$, we can get the contradiction. Namely, if we take
    \begin{equation*}
        \eps < \frac{v}{\omega_n + D^n \cdot c(n)}
    \end{equation*}
    then $p, q$ are mutually $\eps$-critical and $d(p, q) \geq \eps^3$ which is just the contradiction.
\end{proof}
Let's use an example to understand the contradiction in this theorem \ref{thm: cannot be too closed} better. 
\begin{example}\label{ex: volume small}
    Let $M = S_\eps^n \times \R$. Let $D > 0$ be a fixed number so that $\vol(B(p, D)) \approx D \cdot \vol(S_\eps^n) = D\cdot \eps^n\cdot \omega_n$. Obviously, the volume of $B(p, D)$ is small if $\eps$ is small. Let $p, q$ be opposite points in $S^n$ which are certainly mutually $\eps$-critical to each other. The distance of $p$ and $q$ is $d(p, q) = \pi \eps$. 

    \begin{figure}[htbp]
        \centering
            \includegraphics[width=0.9\textwidth]{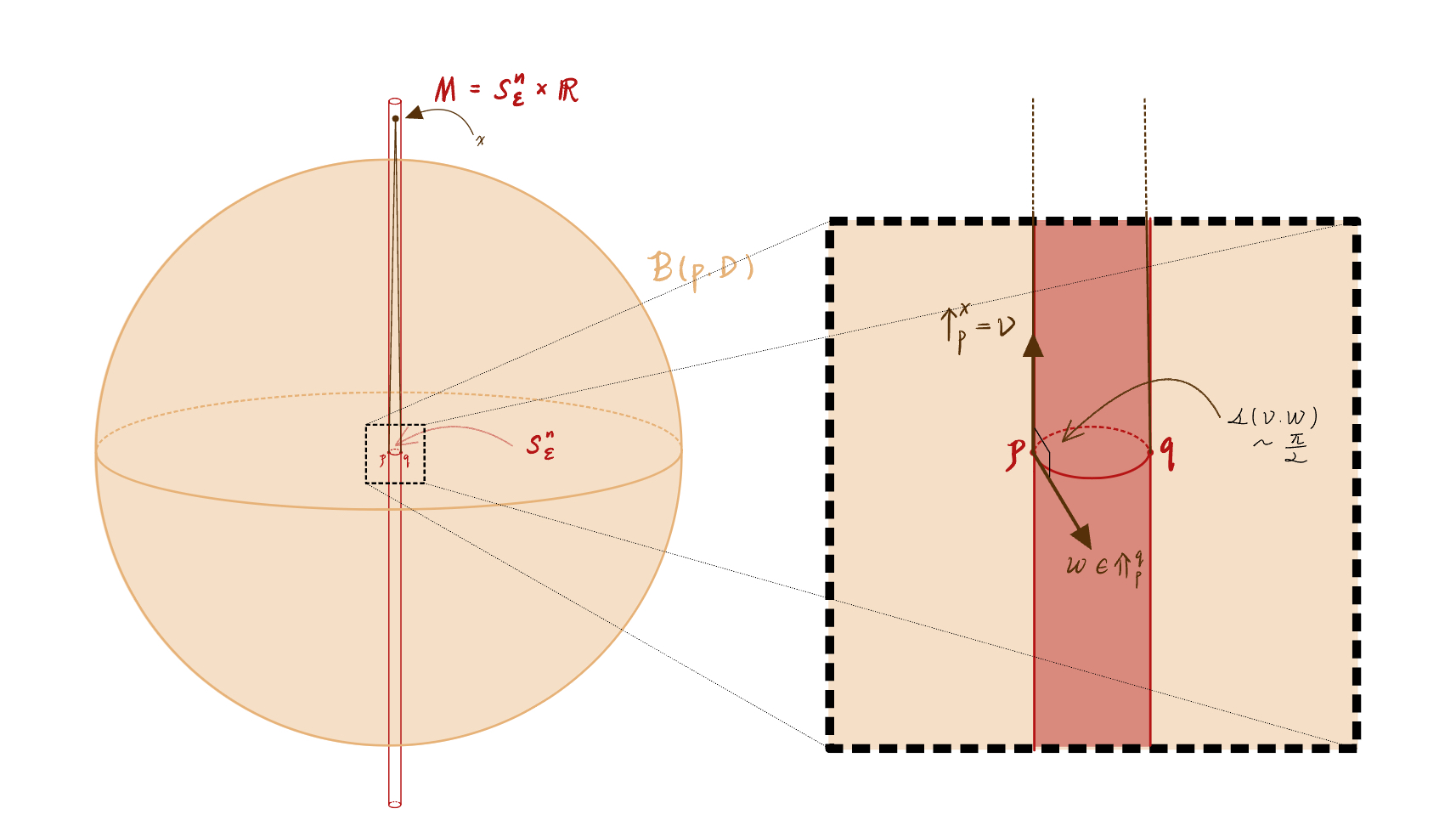}\caption{the points that are too close to each other makes the volume uniformly small [Example \ref{ex: volume small}]}
    \end{figure}
    
    Same as what we did in the proof of the theorem \ref{thm: cannot be too closed}, we pick $x \in M$ such that $\abs{xp} >> \eps$, then what happens is that for $v = \uparrow_p^x$, 
    \begin{equation*}
        \mangle (v, \Uparrow_p^q) \sim \frac{\pi}{2}
    \end{equation*}
    and this forces the volume of $B_D(p)$ to be uniformly small. 
\end{example}
\section{Homotopies and Diagonal}

\begin{notation}
	We define the notation for the distance function for the product of two spaces, i.e. for metric space $(X_1, d_1)$ and $(X_2, d_2)$, we can define and denote the distance function for $X_1\times X_2$ as
	\begin{equation*}
		d^{X_1\times X_2}((p_1, p_2), (q_1, q_2)) = \sqrt{d_1(p_1, q_1)^2 + d_2(p_2, q_2)^2} 
	\end{equation*}
\end{notation}

In case $(X,d)$ is equal to $M$ with a metric induced by a Riemannian metric $g$ on $M$ the distance $d^{X\times X}$ is induced by the Riemannian metric $g\times g$.
Also, note that if 
$(M, g)$ satisfies $\sect_M \ge \kappa$ then  $\sect_{M \times M} \geq \min\br{\kappa, 0}$
\begin{lemma}\label{lem: diagonal lemma}
	Let $(M, g)$ be a complete Riemannian manifold. Consider the product $M \times M$ with the product metric.
	
	 Take the diagonal
    \begin{equation*}
        \Delta(M) = \br{(x, x) \in M \times M: x \in M}. 
    \end{equation*}
    Then, for any $(p, q) \in M\times M$, we have
    \begin{equation*}
    	d^{M \times M}((p, q), \Delta(M)) = \frac{d^M(p, q)}{\sqrt{2}}.
    \end{equation*}
\end{lemma}
\begin{proof}
	Let $(p, q) \in M\times M$, we firstly notice that for any point $(y, y) \in \Delta(M)$, we have
	\begin{align*}
		d^{M\times M}((p, q), (y, y)) 
		=& \sqrt{d^M(p, y)^2 + d^M(q, y)^2}\\
		\geq & \frac{d^M(p, y) + d^M(q, y)}{\sqrt{2}} \quad\text{By AM-GM inequality}\\
		\geq & \frac{d^M(p, q)}{\sqrt{2}}.
	\end{align*}
	On the other hand, we can take $\gamma(t)$ the shortest geodesic from $p$ to $q$ in $M$ where $0 \leq t \leq 1$.

	And we pick the midpoint $x = \gamma(\frac{1}{2})$, so that
	\begin{align*}
		\brac{d^{M\times M}((p, q), (x, x))}^2 &= d^M(p, x)^2 + d^M(q, x)^2\\
		&= \brac{\frac{d^M(p, q)}{2}}^2 + \brac{\frac{d^M(p, q)}{2}}^2\\
		&= \frac{d^M(p, q)^2}{2}.
	\end{align*}
	Thus, 
	\begin{equation*}
		d^{M\times M}((p, q), (x, x)) = \frac{d^M(p, q)}{\sqrt{2}}.
	\end{equation*}
	This means this point admits the shortest distance from the diagonal to $(p, q)$. Indeed,	 we can construct the geodesic from $\Delta(M)$ to $(p, q)$ as $(\gamma_1(t), \gamma_2(t))$ where $t \in [0, \frac{1}{2}]$ such that 
    \begin{align*}
        \gamma_1(t) &= \gamma(t)\\
        \gamma_2(t) &= \gamma(t + \frac{1}{2}).
    \end{align*}
\end{proof}

\begin{remark}
    More generally, this corollary can be extended to any geodesic metric space $(X, d)$. 
\end{remark}

\begin{proposition}\label{prop: diagonal and e-crit}
Consider $(p, q) \in M \times M\setminus \Delta(M)$ such that $p$ is $\eps$-regular to $q$ or $q$ is $\eps$ regular to $p$ ($p$ and $q$ cannot be mutually $\eps$-critical to each other.)
\\
Let  $f = d^{M \times M}(\cdot, \Delta(M))$. 
Then
\begin{equation*}
    \abs{\nabla f_{(p, q)}} > \eps 
\end{equation*}   

\end{proposition}

\begin{proof}[Proof of Proposition\ref{prop: diagonal and e-crit}]
    Without the loss of generality, we assume $q$ being $\eps$-regular for $p$, then $\exists v \in T_qM$ such that
    \begin{equation}
        \mangle(\Uparrow_q^p, v) > \frac{\pi}{2} + \eps.
    \end{equation}
    Thus, if we take $w \in T_{(p, q)}M$, $w = (0, v)$. we have
    \begin{equation*}
        df_{(p, q)}(w) \geq \eps. 
    \end{equation*}
    Similarly for $p$ being $\eps$-regular for $q$, we take $w = (v, 0)$. 

    In either case, we can see that 
    \begin{equation*}
        \abs{\nabla f_{(p, q)}} \geq \eps
    \end{equation*}
\end{proof}
Combining the proposition \ref{prop: diagonal and e-crit} with the previous theorem \ref{thm: cannot be too closed}, we can conclude the following corollary.
\begin{corollary}\label{cor:eps-reg-diag}
    Let $M \in M_{sec}(n, \kappa, D, V)$. Let $\eps$, $\delta$ be given by Theorem~\ref{thm: cannot be too closed}, 
    \begin{align*}
        \eps &= \eps(n, \kappa, D, V);\\
        \delta &= \delta(n, \kappa, D, V)
    \end{align*}
    then $f = d^{M \times M}(\cdot, \Delta(M))$ is $\eps$-regular on the $\delta$-neighborhood of the diagonal minus the diagonal, i.e. 
    \begin{equation*}
        U_\delta(\Delta(M))\backslash \Delta(M)
    \end{equation*}
    and $\abs{\nabla f} \geq \eps$ on $U_\delta(\Delta(M))\backslash \Delta(M)$. 
\end{corollary}
This corollary allows us to construct a smooth unit gradient-like vector field $V$ on $U_\delta(\Delta(M))\backslash \Delta(M)$
such that for every $p \in U_\delta(\Delta(M))\backslash \Delta(M)$ and $v = V(p)$, we have
\begin{equation*}
    df_p(v) \geq \eps
\end{equation*}

Note that near the diagonal $\Delta(M)$, namely, inside the normal injective radius, we can just take $V = \nabla f$ which has length $1$. This is the region where the normal exponential map is diffeomorphism, we can just take unique shortest geodesics in the image of the exponential map. Then along integral curves of $v$, $f$ increases with speed $\geq \eps$.

Along integral curves of $-v$, $f$ decreases with speed at least $-\eps$, that means if we have a point $(p, q) \in U_\delta(\Delta(M))\backslash \Delta(M)$ and $\phi_t$ the backward flow of $-v$, then 
\begin{equation*}
    \phi_t(p, q) \in \Delta(M) \quad\text{for all $t \leq \frac{d^M(p, q)}{\eps}$}
\end{equation*}
which means the flow hits the diagonal. Therefore, we can conclude that $\phi_t$ is a deformation retraction from the neighborhood $U_\delta(\Delta(M))$ to $\Delta(M)$ and length of all integral curves $t \to \phi_t(p, q) = (\gamma_1(t), \gamma_2(t))$ starting at $(p, q)$ satisfies 
\begin{equation*}
    \length(\phi_t(p, q)) \leq \frac{d^M(p, q)}{\eps}.
\end{equation*}
Remember that $\gamma_1(0) = p, \gamma_1(1) = x$ and $\gamma_2(0) = q, \gamma_2(0) = x$ and $(x, x) \in \Delta(M)$. Thus we can construct a new curve $\gamma_{pq}(t)$, $t \in [0, 2]$ such that
\begin{equation*}
    \gamma_{pq}(t)
    \begin{cases}
        \gamma_1(t) \quad t \in [0, 1];\\
        \gamma_2(2 - t) \quad t \in [1, 2].
    \end{cases}
\end{equation*}
The family of curves $\br{\gamma_{pq}(t)}$, whose length is continuous with respect to the distance of $p, q$.  Given any pairs $p, q \in M$ such that $d^M(p, q) < \delta$, we have
\begin{equation*}
    \abs{\length(\gamma_{pq})} \leq \frac{2d^M(p, q)}{\eps}
\end{equation*}

\begin{corollary}\label{cor:short-curves}
Under the assumptions of Corollary~\ref{cor:eps-reg-diag} there is a continuous map $\gamma: U_\delta(\Delta M)\times[0,1]\to M$ such that for any $(p,q)\in  U_\delta(\Delta M)$ it holds that 
\begin{equation*}
    \abs{\length(\gamma_{pq})} \leq \frac{2d^M(p, q)}{\eps}
\end{equation*}
\end{corollary}

\begin{corollary}\label{cor: delta-closed to homotopy}
    For $M \in M_{sec}(n, \kappa, D, V)$ and $\eps$, $\delta$ as above, $(X, d)$ a metric space, $f, g: X \to M$ are $\delta$-close, i.e. $d(f(x), g(x)) < \delta$ for any $x \in X$. Then we can conclude that $f \stackrel{hmtp}{\sim} g$ (homotopic). 
\end{corollary}
\begin{proof}[Proof of Corollary~\ref{cor: delta-closed to homotopy}]
    Take $F(t, x) = \gamma_{f(x)g(x)}(t)$. It is a curve connecting $f(x)$ and $g(x)$ which is continuous in both $x$ and $t$. Thus it is homotopy from $f$ to $g$.
\end{proof}
\begin{remark}
    In the corollary \ref{cor: delta-closed to homotopy}, $\length(\gamma_{f(x)g(x)}(t)) \leq \frac{\delta}{\eps}$. More precisely, 
    \begin{equation*}
        \length(\gamma_{f(x)g(x)}(t))  \leq \frac{d^M(f(x), g(x))}{\eps}
    \end{equation*}
\end{remark}
We are going to use that to prove the Grove-Peterson theorem \ref{thm: Grove-Peterson}. We need the following lemma
\begin{lemma}\label{lem: GP-univeral constant}
    Let $M \in M_{sec}(n, \kappa, D, V)$ and $\eps$ be small, i.e. $(\eps^2\cdot \kappa > -1)$ \footnote{If $\kappa<0$ then $\eps$ needs to be sufficiently small in order for this inequality to be satisfied} Let $x_1, \dots, x_N$ be  a $\eps$-net in $M$. If, by maximality $\cup_{i = 1}^N B_\eps(x_i) = M$. Then the multiplicity of this cover is controlled by a universal constant $L(n)$. Namely, for any $p \in M$, by the maximality, we know there is a constant $K = K(p)$ such that 
    \begin{equation*}
        p \in \cap_{i_j = 1}^K B_\eps(x_{i_j})
    \end{equation*}
    We can always say this $K \leq L(n)$.
\end{lemma}
\begin{proof}[Proof of Lemma~\ref{lem: GP-univeral constant}]
    If $p \in \cap_{i_j}^K B_\eps(x_{i_j})$, then for each pairs of $x_{i_{j_1}}, x_{i_{j_2}} \in B_\eps(p)$ we have $d^M(x_{i_{j_1}}, x_{i_{j_2}}) \geq \eps$.     We can rescale $M$ by $\frac{1}{\eps}$, i.e. $\frac{1}{\eps}M$ so that $\sect_{\frac{1}{\eps}X} \geq \eps^2\kappa \geq -1$. Then in $\frac{1}{\eps}M$, for any $x_{i_{j_1}}, x_{i_{j_2}} \in B_1(p) \subseteq \frac{1}{\eps}M$, we have
    \begin{equation*}
        d^{\frac{1}{\eps}M}(x_{i_{j_1}}, x_{i_{j_2}}) \geq 1
    \end{equation*}
    Therefore, by Bishop-Gromov volume comparison, the number of such $x_{i_j}$s is $\leq L(n)$. Indeed, we know that there is a constant $c(n)$ such that 
    \begin{equation*}
    	\frac{\vol{B_{\frac{1}{2}}(x_{i_j})}}{\vol{B_3(x_{i_j})}} \geq \frac{\vol{B_{\frac{1}{2}}(\bar{p})}}{\vol{B_3(\bar{p})}} = c(n)
    \end{equation*}
    Notice that since $d^{\frac{1}{\eps}M}(x, x_{i_j}) \leq 1$,
    \begin{equation*}
    	B_{3}(x_{i_j}) \supseteq B_2(x)
    \end{equation*}
    Thus, 
    \begin{equation*}
    	\vol{B_{\frac{1}{2}}(x_{i_j})} \geq c(n)\vol{B_3(x_{i_j})} \geq c(n)\vol{B_2(x)}. 
    \end{equation*}
    On the other hand, because $\cup_{i_j = 1}^m B_{\frac{1}{2}}(x_{i_j}) \subset B_2(x)$. We can conclude that
    \begin{equation*}
    	\vol{B_2(x)} \geq \vol{\bigcup_{i_j = 1}^m B_{\frac{1}{2}}(x_{i_j})} \geq mc(n)\vol{B_2(x)}.
    \end{equation*}
    Therefore, we have $m \leq \frac{1}{c(n)} := L(n)$. 

\end{proof}


\section{The Proof of the Grove-Peterson Theorem}
In this section, we are going to finish the proof of the theorem by Grove-Peterson on Finite Homotopy Types~\ref{thm: Grove-Peterson}

Suppose the theorem is not true. That is, if the class $M_{sec}(n, \kappa, D, V)$ has infinitely many homotopy types of manifolds, i.e. there are a sequence of manifolds $M_i \in M_{sec}(n, \kappa, D, V)$ that are not pair-wisely homotopy equivalent. 
    
    By our conclusion on the pre-compactness \ref{cor: sec-precompact}, by passing to a subsequence we can assume that $M_i \ghto X \in A(n, \kappa, D, V)$ (See Notation~\ref{notation: A space}) by passing to a subsequence.
    
    Let $\delta, \eps$ and $L(n)$ be given by Theorem~\ref{thm: cannot be too closed} and Lemma~\ref{lem: GP-univeral constant}. Take $\delta' = \delta\cdot\eps^{10L(n)}$
    And consider $x_1, \dots, x_N$ be the maximal $\frac{1}{2}\delta'$-separated net in $X$ such that $\bigcup B_{\frac{1}{2}\delta'}(x_i) = X$. 
    
    By Corollary~\ref{cor:GH-con iff GH-approx}, we know that there is a sequence $\eps_i \downarrow 0$ and there exists $f_i: M_i \to X$ the $\eps_i$-Gromov-Hausdorff approximations. Recall that from the Definition~\ref{def:epsilon-GH-approx}, this implies $f_i(M_i)$ is $\eps_i$-dense in $X$. This means, for each $x_1, \dots, x_N$, we can have the corresponding $x_1^i, \dots, x_N^i$ in $M_i$ such that 
    \begin{equation*}
    	d^X(f(x_j^i), x_j) \leq \eps_i
    \end{equation*}
    for each $j \in \br{1, \dots, N}$. Indeed we can show that $\br{x_1^i, \dots, x_N^j}$ is an $\delta'$-nets for $M_i$. Indeed, for each $\tilde{x} \in M^i$, $f_i(\tilde{x})\in X$, thus there exists $j \in \br{1, \dots, N}$ such that $f_i(\tilde{x}) \in B_{\frac{1}{2}\delta'}(x_j)$. Since $f_i$ is an $\eps_i$-Gromov-Hausdorff approximation, we can find $x_j^i \in M_i$ such that $d^X(f(x_j^i), x_j) \leq \eps_i$. We claim that $\tilde{x} \in B_{\delta'}(x_j^i)$. To show this, remember that we have 
    \begin{equation*}
    	\abs{d^{X}(f_i(\tilde{x}), f_i(x_j^i)) - d^X(\tilde{x}, x_j^i)} \leq \eps_i
    \end{equation*}
    following from the definition of $\eps_i$-Gromov-Hausdorff approximation. This implies that 
    \begin{align*}
    	d^X(\tilde{x}, x_j^i) 
    	&\leq d^X(f_i(\tilde{x}), f_i(x_j^i)) + \eps_i\\
    	&\leq d^X(f_i(\tilde{x}), x_j) + d^X(f_i(x_j^i), x_j) + \eps_i\\
    	&= \frac{1}{2}\delta' + 2\eps_i
    \end{align*}
    Since we have $\eps_i\downarrow 0$, we can take $\eps_i$ sufficiently small such that 
    \begin{equation*}
    	d^X(\tilde{x}, x_j^i) < \delta'
    \end{equation*}
    
    We want to extend $f_i$ from the set $\br{x_1^i, \dots, x_N^i}$ to the entire $M_i$ continuously in a controlled way (we are going to explain what we mean by ``a controlled way'' later). 
    
    \subsection{Center of Mass Construction}
    
    The idea of the extension is to use the center of mass construction. 
    
    In the case when $X \subseteq \R^n$, take $\phi_j^i$ the partition of unity of $\id$ on $M_i$ subordinates to our cover $M_i = \bigcup_{j = 1}^N B_{\delta'}(x_j^i)$. That is
    \begin{itemize}
        \item For each $x \in B_{\delta'}(x_j^i)$, $0 \leq \phi_j^i(x) \leq 1$;
        \item And when $x = x_j^i$, $\phi_j^i(x_j^i) = 1$;
        \item $\supp(\phi_j^i) \subseteq B_{\delta'}(x_j^i)$;
        \item $\sum_{j = 1}^N \phi_j^i \equiv 1$ over $M_i$.
    \end{itemize}
    By the Lemma \ref{lem: GP-univeral constant}, for each $x \in M_i$, then the number of $\phi_j^i(x) \neq 0$ is at most $L(n)$ of these. And now we only look at those $L(n)$ points. By considering $\lambda_j^i = \phi_j^i(x) \in [0, 1]$ as a number, we can take the extension, still denoted by $f_i$, as a convex linear combination of $x_j$'s, i.e.
    \begin{equation*}
        f_i(x) = \sum_{j = 1}^{L(n)}\phi_j^i(x) x_j = \sum_{j = 1}^{L(n)} \lambda_j^ix_j 
    \end{equation*}
    , which provides us a point in a simplex spanned by $x_1, \dots, x_{L(n)}$. Our extension is continuous and $\delta'$-Gromov-Hausdorff approximation.

    We want to generalize this to our situation when $X$ is not necessarily $\R^n$. Let's initially assume $X \in M_{sec}(n, \kappa, D, V)$ a smooth manifold. Eventually,  instead of a map $M_i \to X$, we will do the center of mass construction of $f_{i_1i_2}: M_{i_1} \to M_{i_2}$ for $i_1, i_2$ large\footnote{The center of mass construction can be made to work in general Alexandrov spaces but that requires developing more tools and we don't present it here}
    The question is how to do a convex linear combination of points $\br{x_1, \dots, x_{L(n)}}$ in $X$.
    
   Suppose we only have to take a convex linear combination of two points $x_1,x_2\in X$ with nonnegative weights $\lambda_1,\lambda_2$ satisfying    $\lambda_1+\lambda_2=1$.
    
Note  that  $d^{X}(x_1, x_2) \leq \delta' = \delta\cdot\eps^{10 L(n)}$ which is very small. By corollary ~\ref{cor:short-curves} we can find a curve $\gamma_{x_1 x_2}(t)$ connecting $x_1$ and $x_2$. This curve depends  continuously on $x_1$ and $x_2$ and satisfies
    \begin{equation*}
        \length(\gamma_{x_1x_2}) \leq \frac{d^X(x_1, x_2)}{\eps}\le \frac{\delta'}{\eps}.
    \end{equation*}
    Then  we can simply take $x = \gamma_{x_1x_2}(\lambda_1)$ to play the role of $\lambda_1x_1+\lambda_2x_2$. call this point the center of mass  of $x_1$ and $x_2$  with weights $\lambda_1,\lambda_2$.
    \begin{equation*}
        \gamma_{x_1x_2}(\lambda_2) = 
        \begin{cases}
            x_1 \quad \text{if $\lambda_1 = 1$};\\
            x_2 \quad \text{if $\lambda_1 = 0$}.
        \end{cases}
    \end{equation*}
    Consider the case when there are three points, say $x_1, x_2, x_3 \in X$ with $d(x_i,x_j)\le  \delta'$.  Suppose $\lambda_1 + \lambda_2 + \lambda_3 = 1$ where $\lambda_i \geq 0$ so $\lambda_1+\lambda_2 + \lambda_3 = 1$. Then in $\R^n$, we can define the center of mass of $x_1,x_2,x_3$ with wights $\lambda_1,\lambda_2,\lambda_3$ by the formula
    \begin{equation*}
        \lambda_1 x_1 + \lambda_2 x_2 + \lambda_3 x_3 = \lambda_1 x_1 + \frac{\lambda_2}{1 - \lambda_2}x_2 + \frac{\lambda_3}{1 - \lambda_1}x_3
    \end{equation*}
    Denote $\lambda_2' = \frac{\lambda_2}{1 - \lambda_2}$ and $\lambda_3' = \frac{\lambda_3}{1 - \lambda_1}$, they are both $\geq 0$ and $\lambda_2' + \lambda_3' = 1$. Denote the point $x' := \lambda_2' x_2 + \lambda_3' x_3$, then we take $\lambda_1 x_1 + (1 - \lambda_1)x' = \sum_{i = 1}^3 \lambda_i x_i$. Then we can do the same thing in our case.
    \begin{figure}[htbp]
        \centering
            \includegraphics[width=0.7\textwidth]{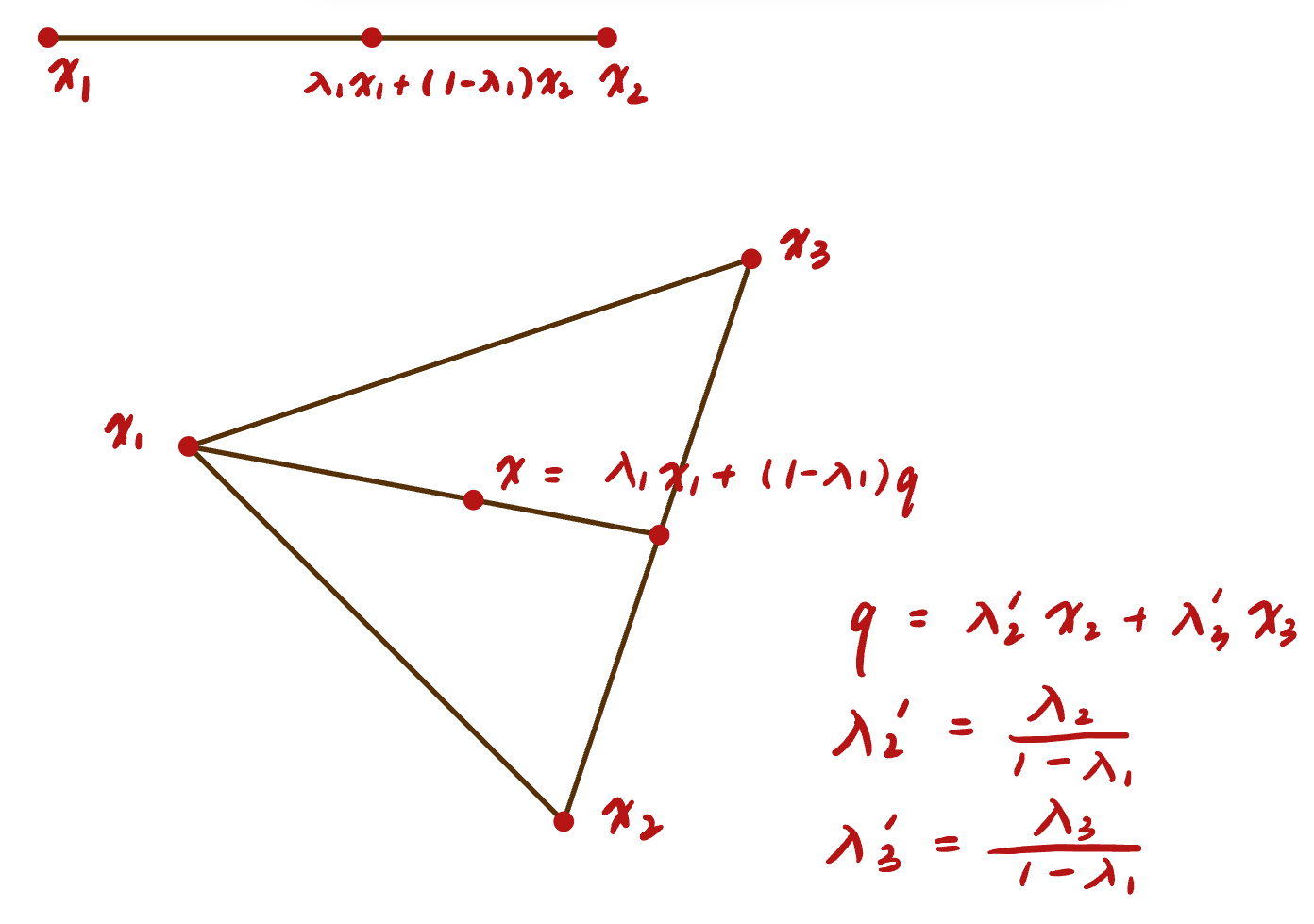}\caption{The Center of Mass Construction in $\R^n$}
    \end{figure}
    
    In our case, $X$ is not $\R^n$ so we do the following. First we take the center of mass $z$  of $x_2,x_3$ with weights $\lambda_2',\lambda_3'$.
    
    Then we take the center of mass of $x_1$ and $z$ with weights $\lambda_1,1-\lambda_1$. 
    
    More explicitly, we first take $z=\gamma_{x_2,x_3}(\lambda_2')$. Then we take  $\gamma_{x_1,z}(\lambda_1)$. We call this point the center of mass of $x_1,x_2,x_3$ with weights $\lambda_1,\lambda_2,\lambda_3$.
    
    We can iterate this procedure to get the center of mass of points $x_1,x_2,\ldots, x_{L(n)}$ with nonnegative weights $\lambda_1,\ldots, \lambda_n$ satisfying $\lambda_1+\lambda_2+\ldots \lambda_{L(n)}=1$.  This produces a singular simplex with vertices $x_1,x_2,\ldots, x_n$  in $X$. By Corollary ~\ref{cor:short-curves} the diameter of this simplex is $\le \frac{\delta'}{\eps^{L(n)}}$.

    Remember that we want to extend $f_i$ of $x \in M_i$ continuously. By Lemma~\ref {lem: GP-univeral constant} $x$ belongs to at most $L(n)$ balls $B_{\delta'}(x_j^i)$. Say these are the balls $B_{\delta'}(x_1^i), \ldots B_{\delta'}(x_{L(n)}^i)$.
    
    Let $\lambda_i = \phi_j^i(x)$ for $j = 1, \dots, N$. Then  $\lambda_1+\lambda_2+\ldots \lambda_{L(n)}=1$.  
    
    \begin{equation*}
        d^X(x_i, x_j) \leq d^X(x_i, x) + d^X(x, x_j) \leq 2 \delta'
    \end{equation*}
    
    Using the above construction we can construct the center of mass of  $x_1,x_2,\ldots, x_{L(n)}$ with weights $\lambda_1,\ldots, \lambda_n$. We set $f_i(x)$ to be equal to this center of mass.
    
    This produces a map $f_i:  M_i\to X$ which is maps $x_j^i$ to $x_j$, is continuous and is a $  2\frac{\delta'}{\eps^{L(n)}} =2 \delta\cdot\eps^{9 L(n)}$-GH approximation.
This works if $X$ is a smooth manifold. However, in general, $X$ is only an Alexandrov space and we can not carry out this construction without more work because we have only proved Corollary ~\ref{cor:eps-reg-diag} for smooth manifolds and not for Alexandrov spaces. However, we can carry out this procedure and construct maps as 
  $f_{i_1i_2}: M_{i_1} \to M_{i_2}$ for $i_1, i_2$ large enough.

Likewise, we can also do the inverse $f_{i_2i_1}: M_{i_2} \to M_{i_2}$ which is also $2\delta\cdot \eps^{9 L(n)}$-Gromov-Hausdorff approximation.

    \begin{claim}\label{cla:hmtp-equiv}
        $f_{i_1i_2}$ and $f_{i_2i_1}$ are homotopy equivalences.
    \end{claim}
    \begin{proof}{Claim}{\ref{cla:hmtp-equiv}}
        Notice that their composition
        \begin{equation*}
            f_{i_2i_1}\circ f_{i_1i_2}: M_{i_1} \to M_{i_2}
        \end{equation*}
        by construction it is $2\cdot\delta\cdot\eps^{9L(n)}$-close to identity. Since this number is smaller than $\delta$, by the result \ref{cor: delta-closed to homotopy}, we get $f_{i_2i_1}\circ f_{i_1i_2}$ is homotopic to the identity $\id_{M_{i_2}}$.  Similarly,  $f_{i_1i_2}\circ f_{i_2i_1}\sim \id_{M_{i_1}}$.
    \end{proof}
    Then we finished the proof.

\bibliographystyle{alpha} 
\bibliography{reference.bib}

\begin{thebibliography}{AKP22}

\bibitem[AKP22]{AKP22}
Stephanie Alexander, Vitali Kapovitch, and Anton Petrunin.
\newblock Alexandrov geometry: foundations, 2022.

\bibitem[BBI01]{BBI01}
D.~Burago, I.U.D. Burago, and S.~Ivanov.
\newblock {\em A Course in Metric Geometry}.
\newblock Crm Proceedings \& Lecture Notes. American Mathematical Society,
  2001.

\bibitem[Ber60]{Be60}
M.~Berger.
\newblock Les vari\'et\'es riemanniennes $(1/4)$-pinc\'ees.
\newblock {\em Annali della Scuola Normale Superiore di Pisa - Scienze Fisiche
  e Matematiche}, 3e s{\'e}rie, 14(2):161--170, 1960.

\bibitem[Ber62]{Be62}
M.~Berger.
\newblock {An extension of rauch's metric comparison theorem and some
  applications}.
\newblock {\em Illinois Journal of Mathematics}, 6(4):700 -- 712, 1962.

\bibitem[Bes07]{Bes07}
Arthur~L Besse.
\newblock {\em Einstein manifolds}.
\newblock Springer Science \& Business Media, 2007.

\bibitem[BS09]{BS09}
Simon Brendle and R.~Schoen.
\newblock Sphere theorems in geometry.
\newblock {\em Surv. Differ. Geometry}, 13, 04 2009.

\bibitem[CE75]{CE75}
J.~Cheeger and D.G. Ebin.
\newblock {\em Comparison Theorems in Riemannian Geometry}.
\newblock North-Holland mathematical library. North-Holland Publishing Company,
  1975.

\bibitem[CG71]{CG71}
Jeff Cheeger and Detlef Gromoll.
\newblock {The splitting theorem for manifolds of nonnegative Ricci curvature}.
\newblock {\em Journal of Differential Geometry}, 6(1):119 -- 128, 1971.

\bibitem[CG72]{CG72}
Jeff Cheeger and Detlef Gromoll.
\newblock On the structure of complete manifolds of nonnegative curvature.
\newblock {\em Annals of Mathematics}, 96(3):413--443, 1972.

\bibitem[dC92]{dCa}
M.P. do~Carmo.
\newblock {\em Riemannian Geometry}.
\newblock Mathematics (Boston, Mass.). Birkh{\"a}user, 1992.

\bibitem[Den21]{Den21}
Qin Deng.
\newblock H\"older continuity of tangent cones and non-branching in $rcd(k,n)$
  spaces.
\newblock 2021.

\bibitem[EH84]{EH84}
Jost Eschenburg and Ernst Heintze.
\newblock An elementary proof of the cheeger-gromoll splitting theorem.
\newblock {\em Annals of Global Analysis and Geometry}, 2(2):141--151, 1984.

\bibitem[EH90]{EH90}
J.~H. Eschenburg and E.~Heintze.
\newblock Comparison theory for riccati equations.
\newblock {\em manuscripta mathematica}, 68(1):209--214, 1990.

\bibitem[Gig22]{Gig22}
Nicola Gigli.
\newblock Introduction to the riemannian curvature dimension condition, 2022.

\bibitem[GM74]{GM74}
Detlef Gromoll and Wolfgang Meyer.
\newblock An exotic sphere with nonnegative sectional curvature.
\newblock {\em Annals of Mathematics}, 100(2):401--406, 1974.

\bibitem[GP93]{GP93}
Karsten Grove and Peter Petersen.
\newblock A radius sphere theorem.
\newblock {\em Inventiones mathematicae}, 112(1):577--583, 1993.

\bibitem[Gro82]{Gr82}
Michael Gromov.
\newblock Volume and bounded cohomology.
\newblock {\em Publications Math\'ematiques de l'IH\'ES}, 56:5--99, 1982.

\bibitem[Gro87]{GG87}
Grove~Karsten Gromoll, Detlef.
\newblock A generalization of berger's rigidity theorem for positively curved
  manifolds.
\newblock {\em Annales scientifiques de l'École Normale Supérieure},
  20(2):227--239, 1987.

\bibitem[Gro93]{Gr93}
Karsten Grove.
\newblock Critical point theory for distance functions, 1993.

\bibitem[GS77]{GS77}
Karsten Grove and Katsuhiro Shiohama.
\newblock A generalized sphere theorem.
\newblock {\em Annals of Mathematics}, 106(1):201--211, 1977.

\bibitem[GW09]{GW09}
Luis Guijarro and Gerard Walschap.
\newblock Submetries vs. submersions.
\newblock {\em Revista Matematica Iberoamericana}, 27:605--619, 2009.

\bibitem[Ham82]{H82}
Richard~S. Hamilton.
\newblock {Three-manifolds with positive Ricci curvature}.
\newblock {\em Journal of Differential Geometry}, 17(2):255 -- 306, 1982.

\bibitem[Kli61]{Kl61}
Wilhelm Klingenberg.
\newblock Über riemannsche mannigfaltigkeiten mit positiver krümmung.
\newblock {\em Commentarii mathematici Helvetici}, 35:47--54, 1961.

\bibitem[Lee03]{Lee03}
J.M. Lee.
\newblock {\em Introduction to Smooth Manifolds}.
\newblock Graduate Texts in Mathematics. Springer, 2003.

\bibitem[Lee19]{Lee19}
J.M. Lee.
\newblock {\em Introduction to Riemannian Manifolds}.
\newblock Graduate Texts in Mathematics. Springer International Publishing,
  2019.

\bibitem[Lyt05]{Ly05}
A~Lytchak.
\newblock Open map theorem for metric spaces.
\newblock {\em St Petersburg Mathematical Journal}, 17:477--491, 01 2005.

\bibitem[Mey04]{Meyer04}
Wolfgang Meyer.
\newblock Toponogov's theorem and applications.
\newblock 2004.

\bibitem[Mil56]{Mil56}
John Milnor.
\newblock On manifolds homeomorphic to the 7-sphere.
\newblock {\em Annals of Mathematics}, 64(2):399--405, 1956.

\bibitem[Mye35]{My35}
Sumner~Byron Myers.
\newblock {Riemannian manifolds in the large}.
\newblock {\em Duke Mathematical Journal}, 1(1):39 -- 49, 1935.

\bibitem[Per94a]{Per94}
G.~Perelman.
\newblock {Proof of the soul conjecture of Cheeger and Gromoll}.
\newblock {\em Journal of Differential Geometry}, 40(1):209 -- 212, 1994.

\bibitem[Per94b]{P94}
G.~Perelman.
\newblock {Proof of the soul conjecture of Cheeger and Gromoll}.
\newblock {\em Journal of Differential Geometry}, 40(1):209 -- 212, 1994.

\bibitem[PP95]{PP95}
Grigori Perelman and Anton Petrunin.
\newblock Quasigeodesics and gradient curves in alexandrov spaces.
\newblock 1995.

\bibitem[Rau51]{Ra51}
H.~E. Rauch.
\newblock A contribution to differential geometry in the large.
\newblock {\em Annals of Mathematics}, 54(1):38--55, 1951.

\bibitem[Rig78]{Rig78}
A.~Rigas.
\newblock Some bundles of non-negative curvature.
\newblock {\em Mathematische Annalen}, 232(2):187--193, 1978.

\bibitem[RS12]{RS12}
Tapio Rajala and Karl-Theodor Sturm.
\newblock Non-branching geodesics and optimal maps in strong $$cd(k,\infty )$$
  c d ( k ,$\infty$ ) -spaces.
\newblock {\em Calculus of Variations and Partial Differential Equations}, 50,
  07 2012.

\bibitem[Sha77]{Sha77}
V.~A. Sharafutdinov.
\newblock The pogorelov-klingenberg theorem for manifolds homeomorphic to rn.
\newblock {\em Siberian Mathematical Journal}, 18(4):649--657, 1977.

\bibitem[Wil02]{Wil02}
Burkhard Wilking.
\newblock Manifolds with positive sectional curvature almost everywhere.
\newblock {\em Inventiones mathematicae}, 148(1):117--141, 2002.

\bibitem[Wil07]{Wi07}
Burkhard Wilking.
\newblock A duality theorem for riemannian foliations in nonnegative sectional
  curvature.
\newblock {\em Geometric and Functional Analysis}, 17(4):1297--1320, 2007.

\end{thebibliography}
\end{document}